%% file: Rank2AmalgamsandFusionSystems.tex
\documentclass[a4paper, 12pt, final]{amsart}
\input{Rank2preamble.tex}

\title{Rank $2$ Amalgams and Fusion Systems}
\author{Martin van Beek}
\thanks{This work formed the majority of author's PhD thesis at the University of Birmingham under the supervision of Prof. Chris Parker. The author gratefully acknowledges the support received from the EPSRC (EP/N509590/1) during this period.}

\begin{document}

\maketitle

\begin{abstract}
We classify fusion systems $\fs$ in which $O_p(\fs)=\{1\}$, and there are two $\Aut_{\fs}(S)$-invariant essential subgroups whose normalizer systems generate $\fs$. We employ the amalgam method and, as a bonus, obtain $p$-local characterizations of certain rank $2$ group amalgams whose parabolic subgroups involve strongly $p$-embedded subgroups.
\end{abstract}

\input{Contents/1.Introduction.tex}
\input{Contents/2.Prelims.tex}
\input{Contents/3.FusionSystems.tex}
\input{Contents/4.Amalgams.tex}
\input{Contents/5.TheAmalgamMethod.tex}
\input{Contents/6.1.SymmetricCasei.tex}
\input{Contents/6.2.SymmetricCaseii.tex}
\input{Contents/7.1.Non-symmetricCasei.tex}
\input{Contents/7.2.Non-symmetricCaseii.tex}
\input{Contents/7.3.b=1.tex}

\printbibliography

\end{document}

%% file: Rank2preamble.tex
\usepackage{nag}
\usepackage{setspace}
\usepackage{multicol}
\usepackage{multirow}
\usepackage[indentafter]{titlesec}
\usepackage{upgreek}
\usepackage[pdftex]{graphicx}
\usepackage{amsmath, amsfonts, amssymb, amsthm, float, bm}
\usepackage{scalerel}
\usepackage{mathtools}
\usepackage{mathrsfs}
\usepackage{boldline}
\usepackage{enumitem}
\usepackage[toc,page]{appendix}
\usepackage[a4paper, bottom=4cm]{geometry}
\usepackage{array}
\usepackage{xtab}
\usepackage{lmodern}
\usepackage{tcolorbox}
\usepackage{xcolor}
\usepackage{url}
\usepackage{tikz-cd}
\usepackage{bm}
\usepackage{tabularx}
\usepackage{ltablex} 
\usepackage{etoolbox}
\usepackage{tipa}
\usepackage[backend=bibtex8, sorting=nty,style=alphabetic]{biblatex}
\usepackage[colorlinks=true,linkcolor=black,citecolor=blue!80!black]{hyperref}
\usepackage[capitalize]{cleveref}
\usepackage{relsize}
\usepackage{blindtext}
\renewbibmacro{in:}{}
\DeclareFieldFormat[article]{title}{#1}

\titleformat{name=\section}{}{\thetitle.}{0.8em}{\centering\scshape}
\titleformat{name=\subsection}{}{\thetitle.}{0.5em}{\centering\scshape}
\titleformat{name=\subsubsection}[runin]{}{\thetitle.}{0.5em}{\itshape}[.]
\titleformat{name=\paragraph,numberless}[runin]{}{}{0em}{}[.]
\titlespacing{\paragraph}{0em}{0em}{0.5em}
\titleformat{name=\subparagraph,numberless}[runin]{}{}{0em}{}[.]
\titlespacing{\subparagraph}{0em}{0em}{0.5em}

\theoremstyle{plain}
\newtheorem{theorem}{Theorem}[section]
\newtheorem{proposition}[theorem]{Proposition}
\newtheorem{lemma}[theorem]{Lemma}
\newtheorem{corollary}[theorem]{Corollary}

\newtheorem{example}[theorem]{Example}
\newtheorem*{theorem*}{Theorem}

\newtheorem*{main}{Main Theorem}

\newtheorem*{CorA}{Corollary A}
\newtheorem*{ThmC}{Theorem C}

\theoremstyle{definition}
\newtheorem{definition}[theorem]{Definition}
\newtheorem{notation}[theorem]{Notation}
\newtheorem{hypothesis}[theorem]{Hypothesis}
\newtheorem*{HypB}{Hypothesis B}

\theoremstyle{remark}
\newtheorem*{remark}{Remark}

\AtBeginEnvironment{theorem}{\setlist[enumerate,1]{label=(\roman*),font=\upshape}}
\AtBeginEnvironment{main}{\setlist[enumerate,1]{label=(\roman*),font=\upshape}}
\AtBeginEnvironment{mtheorem}{\setlist[enumerate,1]{label=(\roman*),font=\upshape}}
\AtBeginEnvironment{theorem*}{\setlist[enumerate,1]{label=(\roman*),font=\upshape}}
\AtBeginEnvironment{ftheorem}{\setlist[enumerate,1]{label=(\roman*),font=\upshape}}
\AtBeginEnvironment{example}{\setlist[enumerate,1]{label=(\roman*),font=\upshape}}
\AtBeginEnvironment{definition}{\setlist[enumerate,1]{label=(\roman*),font=\upshape}}
\AtBeginEnvironment{lemma}{\setlist[enumerate,1]{label=(\roman*),font=\upshape}}
\AtBeginEnvironment{proposition}{\setlist[enumerate,1]{label=(\roman*),font=\upshape}}
\AtBeginEnvironment{sublemma}{\setlist[enumerate,1]{label=(\roman*),font=\upshape}}
\AtBeginEnvironment{corollary}{\setlist[enumerate,1]{label=(\roman*),font=\upshape}}
\AtBeginEnvironment{theorem}{\setlist[enumerate,2]{label=(\alph*),font=\upshape}}
\AtBeginEnvironment{lemma}{\setlist[enumerate,2]{label=(\alph*),font=\upshape}}
\AtBeginEnvironment{proposition}{\setlist[enumerate,2]{label=(\alph*),font=\upshape}}
\AtBeginEnvironment{sublemma}{\setlist[enumerate,2]{label=(\alph*),font=\upshape}}
\AtBeginEnvironment{CorA}{\setlist[enumerate,1]{label=(\roman*),font=\upshape}}
\AtBeginEnvironment{ThmC}{\setlist[enumerate,1]{label=(\roman*),font=\upshape}}
\AtBeginEnvironment{HypB}{\setlist[enumerate,1]{label=(\roman*),font=\upshape}}

\renewcommand{\Gamma}{\varGamma}
\renewcommand{\epsilon}{\varepsilon}
\renewcommand{\bar}{\overline}
\renewcommand{\hat}{\widehat}
\renewcommand{\leq}{\leqslant}
\renewcommand{\geq}{\geqslant}

\newcommand{\normaleq}{\trianglelefteq}

\newcommand{\divides}{\bigm|}
\newcommand{\fs}{\mathcal{F}}

\newcommand{\N}{\mathbb{N}}

\newcommand{\SL}{\mathrm{SL}} 
\newcommand{\SU}{\mathrm{SU}} 
\newcommand{\GF}{\mathrm{GF}} 
\newcommand{\syl}{\mathrm{Syl}}
\newcommand{\GL}{\mathrm{GL}}
\newcommand{\Sp}{\mathrm{Sp}}

\newcommand{\PSL}{\mathrm{PSL}}
\newcommand{\PSp}{\mathrm{PSp}}
\newcommand{\PSU}{\mathrm{PSU}}

\newcommand{\Sz}{\mathrm{Sz}}
\newcommand{\Ree}{\mathrm{Ree}}
\newcommand{\Sym}{\mathrm{Sym}} 
\newcommand{\Alt}{\mathrm{Alt}} 
\newcommand{\Dih}{\mathrm{Dih}}

\newcommand{\Frob}{\mathrm{Frob}}

\newcommand{\Aut}{\mathrm{Aut}}
\newcommand{\Out}{\mathrm{Out}}
\newcommand{\Inn}{\mathrm{Inn}}

\newcommand{\Mor}{\mathrm{Mor}}
\newcommand{\Hom}{\mathrm{Hom}}
\newcommand{\Iso}{\mathrm{Iso}}
\newcommand{\Inj}{\mathrm{Inj}}
\newcommand{\Ob}{\mathrm{Ob}}

\def \wt {\widetilde}

%% file: Contents/1.Introduction.tex
\section{Introduction}
For a finite group $G$ and a prime $p$ dividing the order of $G$, the $p$-fusion category of $G$ provides a means to concisely express properties of the conjugacy of $p$-elements within $G$. Fusion systems may then be viewed as an abstraction of fusion categories without the need to specify a group $G$, instead focusing only on the properties of a particular $p$-group. 

The purpose of this paper is to classify certain fusion systems which are generated by automorphisms of two subgroups which satisfy certain properties. This is achieved by identifying a rank two amalgam within the fusion system, and then utilizing the amalgam method. In this way, the work in this paper may be viewed not only as a result about fusion systems, but as a result about (not necessarily finite) groups (see \hyperlink{ThmC}{Theorem C}). Furthermore, as an application, the work here will aid in classifying saturated fusion systems over maximal unipotent subgroups of finite groups of Lie type of rank $2$. Indeed, this result has already been used to give a complete description of all saturated fusion systems supported on $p$-groups isomorphic to a Sylow $p$-subgroup of $\mathrm{G}_2(p^n)$ or $\PSU_4(p^n)$, see \cite{G2pPaper}.

The methodology for proving this result breaks down as follows. Firstly, we identify the automizers of the two distinguished subgroups in our fusion system $\fs$. Here the subgroups in question are \emph{essential} in $\fs$ and, using fusion techniques and classification results concerning groups with strongly $p$-embedded subgroups, we can almost completely describe the outer automorphisms induced by $\fs$. Then, using the model theorem, we are able to able to investigate finite groups whose fusion categories are isomorphic to normalizer subsystems of the two distinguished essential subgroups. With these two groups in hand, we can then form an amalgam whose faithful completion realizes $\fs$. 

This leads to the second part of the analysis. Here, we employ the amalgam method, building on work began by Goldschmidt \cite{goldschmidt}. In our interpretation, we closely follow the techniques developed and refined by Delgado and Stellmacher \cite{Greenbook} and a large number of the amalgams we investigate are fortunately already classified there. Indeed, several of the amalgams we investigate are unique up to isomorphism and, as it turns out, this is enough to determine the fusion system up to isomorphism. However, in some cases, we do not go so far and instead aim only to bound the order of the $p$-group on which $\fs$ is supported and apply a package in MAGMA \cite{Comp1} which identifies the fusion system. In fact, in two instances there are no finite groups which realize the amalgam appropriately and we uncover two exotic fusion systems, one of which was known about previously by work of Parker and Semeraro \cite{parkersem}, and another which has been described in \cite{ExoSpo}. With that said, given the information we gather about the amalgams, it does not seem such a stretch to at least provide a characterization of these amalgams up to some weaker notion of isomorphism.

Within this work, we very often use a $\mathcal{K}$-group hypothesis when investigating automizers of essential subgroups and a $\mathcal{CK}$-system hypothesis on the fusion system $\fs$. Recall that $ \mathcal{K}$-group is a finite group in which every simple section is isomorphic to a known finite simple group. A $\mathcal{CK}$-system is then a saturated fusion system in which the induced automorphism groups on all $p$-subgroups are $\mathcal{K}$-groups. At some stage in the analysis, unfortunately, we make explicit use of the classification of finite simple groups (CFSG), specifically when $\fs$ is exotic. However, up to that point, we are still able to determine the size of the $p$-group on which $\fs$ is supported, as well the local actions, within a $\mathcal{CK}$-system hypothesis and only appeal to the classification to prove that the fusion system is exotic. Thus, we believe this result would still be suitable for use in any investigation of fusion systems in which induction via a minimal counterexample is utilized. The main theorem is as follows:

\begin{main}\hypertarget{MainThm}{}
Let $\fs$ be a local $\mathcal{CK}$-system on a $p$-group $S$. Assume that $\fs$ has two $\Aut_{\fs}(S)$-invariant essential subgroups $E_1, E_2\normaleq S$ such that for $\fs_0:=\langle N_{\fs}(E_1), N_{\fs}(E_2) \rangle_S$ the following conditions hold:
\begin{itemize}
\item $O_p(\fs_0)=\{1\}$;
\item for $G_i:=\Out_{\fs}(E_i)$, if $G_i/O_{3'}(G_i)\cong\Ree(3)$, $G_i$ is $p$-solvable or $T$ is generalized quaternion, then $N_{G_i}(T)$ is strongly $p$-embedded in $G_i$ for $T\in\syl_p(G_i)$.
\end{itemize}
Then $\fs_0$ is saturated and one of the following holds:
\begin{enumerate}
\item $\fs_0=\fs_S(G)$, where $F^*(G)$ is isomorphic to a rank $2$ simple group of Lie type in characteristic $p$;
\item $\fs_0=\fs_S(G)$, where $G\cong \mathrm{M}_{12}, \Aut(\mathrm{M}_{12}), \mathrm{J}_2, \Aut(\mathrm{J}_2), \mathrm{G}_2(3)$ or $\PSp_6(3)$ and $p=2$;
\item $\fs_0=\fs_S(G)$, where $G\cong \mathrm{Co}_2, \mathrm{Co}_3,\mathrm{McL}$, $\Aut(\mathrm{McL}), \mathrm{Suz}, \Aut(\mathrm{Suz})$ or $\mathrm{Ly}$ and $p=3$;
\item $\fs_0=\fs_S(G)$, where $G\cong\PSU_5(2), \Aut(\PSU_5(2)),  \Omega_8^+(2), \mathrm{O}_8^+(2), \Omega_{10}^-(2),\\ \Sp_{10}(2), \PSU_6(2)$ or  $\PSU_6(2).2$ and $p=3$;
\item $\fs_0$ is simple fusion system on a Sylow $3$-subgroup of $\mathrm{F}_3$ and, assuming $\mathrm{CFSG}$, $\fs_0$ is an exotic fusion system uniquely determined up to isomorphism;
\item $\fs_0=\fs_S(G)$, where $G\cong \mathrm{Ly}, \mathrm{HN}, \Aut(\mathrm{HN})$ or $\mathrm{B}$ and $p=5$; or
\item $\fs_0$ is a simple fusion system on a Sylow $7$-subgroup of $\mathrm{G}_2(7)$ and, assuming $\mathrm{CFSG}$, $\fs_0$ is an exotic fusion system uniquely determined up to isomorphism.
\end{enumerate}
\end{main}

We include $\mathrm{G}_2(2)'\cong\PSU_3(3)$, $\Sp_4(2)'\cong\Alt(6)$ and the Tits groups ${}^2\mathrm{F}_4(2)'$ as groups of Lie type in characteristic $2$. 

\sloppy{Some explanation is required for the exceptional condition when $\Out_{\fs}(E_i)/O_{3'}(\Out_{\fs}(E_i))\cong \Ree(3)$, $\Out_{\fs}(E_i)$ is $p$-solvable or $\Out_S(E_i)$ is generalized quaternion in the above theorem. This assumption ensures that the centralizers of non-central chief factors in a model of $N_{\fs}(E_i)$ are $p$-closed. This is key in our methodology and facilitates the use of various techniques, especially coprime action arguments, which are used to deduce $\Out_{\fs}(E_i)$ and its action on $E_i$. In the case where $\Out_{\fs}(E_i)/O_{5'}(\Out_{\fs}(E_i))\cong \Sz(32):5$, this could also pose a problem. However, in this situation we are able to apply a transfer argument to dispel this case (see \cref{Ree3}). It is easily seen using the Alperin--Goldschimidt theorem that if $E_i$ is not contained in any other essential subgroup of $\fs$ (which we term \emph{maximally essential}) then the exceptional condition holds. We anticipate that in in the majority of the applications of the \hyperlink{MainThm}{Main Theorem}, this hypothesis can be forced.}

In support of the \hyperlink{MainThm}{Main Theorem}, applying various results from \cite{Clelland}, \cite{OliSmall}, \cite{Sp4}, \cite{Comp1}, \cite{G2pPaper} and \cite{ExoSpo}, we can also describe $\fs$ up to isomorphism in most cases. It remains to classify the fusion systems supported on a Sylow $p$-subgroup of ${}^2\mathrm{F}_4(2^n)$, ${}^3\mathrm{D}_4(p^n)$ or $\PSU_5(p^n)$. The important cases to consider are where the set of essential subgroups is not contained in the set of unipotent radicals of the two maximal parabolic subgroups of the ${}^2\mathrm{F}_4(2^n)$, ${}^3\mathrm{D}_4(p^n)$ or $\PSU_5(p^n)$ arranged to contain $S$. We expect that this never happens in these examples and, in the language above, that $\fs=\fs_0$.

In the classification in \hyperlink{MainThm}{Main Theorem}, where $\fs_0$ is \emph{realizable} by  finite group, we provide only one example of a group which realizes the fusion system. In several instances, this example is not unique, even amongst finite simple groups. In particular, if $\fs_0$ is realized by a simple group of Lie type in characteristic coprime to $p$, then there are lots of examples which realize the fusion system, see for instance \cite{OliEquiv}. Note also that we manage to capture a large number of fusion systems at odd primes associated to sporadic simple groups. Indeed, as can be witnessed in the tables provided in \cite{AWC}, almost all of the $p$-fusion categories of the sporadic simple groups at odd primes are either constrained, supported on an extraspecial group of exponent $p$ and so are classified in \cite{RV1+2}, or satisfy the hypothesis of the \hyperlink{MainThm}{Main Theorem}.

It is surprising that in the conclusion of the \hyperlink{MainThm}{Main Theorem} there are so few exotic fusion systems. It has seemed that, at least for odd primes, exotic fusion systems were reasonably abundant. Perhaps an explanation for the apparent lack of exotic fusion system is that the setup from the \hyperlink{MainThm}{Main Theorem} somehow reflects some of the geometry present in rank $2$ groups of Lie type. Additionally, we remark that in the two exotic examples in the classification, the fusion systems are obtained by ``pruning" a particular class of essential subgroups, as defined in \cite{Comp1}. Indeed, these essential subgroups, along with their automizers, seem to resemble Aschbacher blocks, the minimal counterexamples to the Local $C(G,T)$-theorem \cite{CGT}. Most of the exotic fusion systems the author is aware of either have a set of essentials resembling blocks, or are obtained by pruning a class of essentials resembling blocks out of the fusion category of some finite group. For instance, pearls in fusion systems, investigated in \cite{grazian} and \cite{ParkerMax}, are the smallest examples of blocks in fusion systems.

The work we undertake in the proof of the \hyperlink{MainThm}{Main Theorem} may be regarded as a generalization of some of the results in \cite{OliAmal}, where only certain configurations at the prime $2$ are considered. There, the authors exhibit a situation in which a pair of subgroups of the automizers of pairs of essential subgroup generate a subsystem, and then describe the possible actions present in the subsystem, utilizing Goldschmidt's pioneering results in the amalgam method. With this in mind, we provide the following corollary along the same lines which, at least with regards to essential subgroups, may also be considered as the minimal situation in which a saturated fusion system satisfies $O_p(\fs)=\{1\}$.

\begin{CorA}\hypertarget{CorA}{}
Suppose that $\fs$ is a saturated fusion system on a $p$-group $S$ such that $O_p(\fs)=\{1\}$. Assume that $\fs$ has exactly two essential subgroups $E_1$ and $E_2$. Then $N_S(E_1)=N_S(E_2)$ and writing $\fs_0:=\langle N_{\fs}(E_1), N_{\fs}(E_2)\rangle_{N_S(E_1)}$, $\fs_0$ is a saturated normal subsystem of $\fs$ and either
\begin{enumerate}
\item $\fs=\fs_0$ is determined by the \hyperlink{MainThm}{Main Theorem};
\item $p$ is arbitrary, $\fs_0$ is isomorphic to the $p$-fusion category of $H$, where $F^*(H)\cong\PSL_3(p^n)$, and $\fs$ is isomorphic to the $p$-fusion category of $G$ where $G$ is the extension of $H$ by a graph or graph-field automorphism;
\item $p=2$, $\fs_0$ is isomorphic to the $2$-fusion category of $H$, where $F^*(H)\cong\PSp_4(2^n)$, and $\fs$ is isomorphic to the $2$-fusion category of $G$ where $G$ is the extension of $H$ by a graph or graph-field automorphism; or
\item $p=3$, $\fs_0$ is isomorphic to the $3$-fusion category of $H$, where $F^*(H)\cong\mathrm{G}_2(3^n)$, and $\fs$ is isomorphic to the $3$-fusion category of $G$ where $G$ is the extension of $H$ by a graph or graph-field automorphism.
\end{enumerate}
\end{CorA}

As intimated earlier in this introduction, we utilize the amalgam method to classify the fusion systems in the statement of the \hyperlink{MainThm}{Main Theorem}. Here, we work in a purely group theoretic setting and so, as a consequence of the work in this paper, we obtain some generic results concerning amalgams of finite groups which apply outside of fusion systems. We operate under the following hypothesis, and note that the relevant definitions are provided in \cref{AmalgamsSetup}:

\begin{HypB}\hypertarget{HypB}{}
$\mathcal{A}:=\mathcal{A}(G_1, G_2, G_{12})$ is a characteristic $p$ amalgam of rank $2$ satisfying the following:
\begin{enumerate}
\item for $S\in\syl_p(G_{12})$, $N_{G_1}(S)=N_{G_2}(S)\le G_{12}$; and
\item writing $\bar{G_i}:=G_i/O_p(G_i)$, $\bar{G_i}$ contains a strongly $p$-embedded subgroup and if $\bar{G_i}/O_{3'}(\bar{G_i})\cong \Ree(3)$, $\bar{G_i}$ is $p$-solvable or $\bar{S}$ is generalized quaternion, then $N_{\bar{G_i}}(\bar{S})$ is strongly $p$-embedded in $\bar{G_i}$ for $S\in \syl_p(G_{12})$.
\end{enumerate}
\end{HypB}

It transpires that all the amalgams satisfying \hyperlink{HypB}{Hypothesis B} are either weak BN-pairs of rank $2$; or $p\leq 7$, $|S|\leq 2^9$ when $p=2$, and $|S|\leq p^7$ when $p$ is odd. Moreover, in the latter exceptional cases we can generally describe, at least up to isomorphism, the parabolic subgroups of amalgam. It is worth mentioning that condition (ii) holds whenever $\bar{G_i}$ has a strongly $p$-embedded subgroup and $O^{p'}(\bar{G_i})$ is $p$-minimal.

What is remarkable about these results is that amalgams produced have ``critical distance" (defined in \cref{BasicAmalNot}) bounded above by $5$ . In the cases where the amalgam is not a weak BN-pair of rank $2$, the critical distance is bounded above by $2$, and when this distance is equal to $2$, the amalgam is \emph{symplectic} and was already known about by work of Parker and Rowley \cite{parkerSymp}. We present an undetailed version of the theorem summarizing the amalgam theoretic results below.

\begin{ThmC}\hypertarget{ThmC}{}
Suppose that $\mathcal{A}:=\mathcal{A}(G_1, G_2, G_{12})$ satisfies \hyperlink{HypB}{Hypothesis B}. Then one of the following occurs:
\begin{enumerate}
\item $\mathcal{A}$ is a weak BN-pair of rank $2$;
\item $p=2$, $\mathcal{A}$ is a symplectic amalgam, $G_1/O_2(G_1)\cong \Sym(3)$, $G_2/O_2(G_2)\cong (3\times 3):2$ and $|S|=2^6$;
\item $p=2$, $\Omega(Z(S))\normaleq G_2$, $\langle (\Omega(Z(S))^{G_1})^{G_2})\rangle\not\le O_2(G_1)$, $O^{2'}(G_1)/O_2(G_1)\cong\SU_3(2)'$, $O^{2'}(G_2)/O_2(G_2)\cong\Alt(5)$ and $|S|=2^9$;
\item $p=3$, $\Omega(Z(S))\normaleq G_2$, $\langle (\Omega(Z(S))^{G_1})\rangle\not\le O_3(G_2)$, $O_3(G_1)=\langle (\Omega(Z(S))^{G_1})\rangle$ is cubic $2F$-module for $O^{3'}(G_1/O_3(G_1))$ and $|S|\leq 3^7$; or
\item $p=5$ or $7$, $\mathcal{A}$ is a symplectic amalgam and $|S|=p^6$.
\end{enumerate}
\end{ThmC}

Much more information about the amalgams is provided where they arise in the proofs.

The contributions in this paper may be viewed as the ``rank $2$'' case of an attempt to classify fusion systems which contain a ``parabolic system.'' For parabolic systems in groups, work of various authors indicates that given sufficient control of the rank $2$ residues of a parabolic system, using the theory of chamber systems and buildings, one can identify a BN-pair and, if the group under investigation is finite, a result of Tits implies that the group is a group of Lie type. For a survey of results in this area, see \cite{MeixnerSurv}. A result of Onofrei \cite{Ono} translating parabolic systems in groups to their fusion theoretic counterparts, suggests that it may be possible to use the rank $2$ information present in a fusion system $\fs$ to identify a ``BN-pair'' and show that $\fs$ is actually isomorphic to the fusion category of a finite group of Lie type. Although there are some simple and gratifying corollaries to extract relating to this from the results in this article, we feel that this work merits its own investigation.

In \cref{GrpSec}, we set up the various group theoretic terminologies and results we use throughout the paper. Most importantly, we characterize groups with strongly $p$-embedded, groups with associated FF-modules and $2$F-modules, groups which contain elements which act quadratically, and exhibit situations in which these phenomena occur. The typical examples of automizers in our investigations are rank $1$ groups of Lie type in defining characteristic and, because of this, large parts of this section are devoted to the properties of such groups and their ``natural'' modules. In \cref{FusSec}, we introduce fusion systems and, for the most part, reproduce definitions and properties associated to fusion systems which may be readily found in the literature. We end this section by determining the potential automizers of maximally essential subgroups of fusion systems, which may be of independent interest. In \cref{AmalgamsSetup} we demonstrate how to identify a rank $2$ amalgam given certain hypotheses on a fusion system and begin setting up the group theoretic framework needed for the amalgam method. Here is where we deduce the \hyperlink{MainThm}{Main Theorem} from \hyperlink{ThmC}{Theorem C}. In \cref{AmalSec} we set up the hypothesis and notations needed for the amalgam method. We demonstrate how to force quadratic or cubic action in this framework, which allows the use of results from \cref{GrpSec}. For several arguments, we investigate a minimal counterexample where minimality is imposed on the order of the models of the normalizers of essential subgroups. In the amalgam method, the case division separates fairly naturally, and we follow the divisions used in \cite{Greenbook}. Then \cref{evensec} and \cref{oddsec} deal with each of the particular cases.

Our notation and terminology for groups is reasonably standard and generally follows that used in \cite{asch2}, \cite{gor}, \cite{kurz}, \cite{Huppert} and in the ATLAS \cite{atlas}. For fusion systems, we follow the notation used in \cite{ako} and when working with the amalgam method, our definitions and notation follow \cite{Greenbook}. Throughout, we adopt bar notation for quotients. That is, if $G$ is a group, $H\le G$ and $U\normaleq G$, then writing $\bar{G}:=G/U$, we recognize $\bar{H}:=HU/U$. Some clarification is probably also required on the notation used for group extensions. We use $A:B$ to denote the a semidirect product of $A$ and $B$, where $A$ is normalized by $B$. We use the notation $A.B$ to denote an arbitrary extension of $B$ by $A$. That is, $A$ is a normal subgroup of $A.B$ such that the quotient of $A.B$ by $A$ is isomorphic to $B$. We use the notation $A\cdot B$ to denote a central extension of $B$ by $A$ and the notation $A\circ B$ to denote a central product of $A$ and $B$, where the intersection of $A$ and $B$ will be clear whenever this arises.

%% file: Contents/2.Prelims.tex
\section{Group Theory Preliminaries}\label{GrpSec}

In this section, we collect various results to be used later in the paper. Several are well known or elementary, and where possible, we aim to give explicit references or rudimentary proofs. We use \cite{asch2}, \cite{gor}, \cite{GLS2}, \cite{Huppert} and \cite{kurz} as background texts and also appeal to them for proofs of several results we will use frequently throughout this work. We also make appeals to \cite{GLS3} for known facts about known finite simple groups. The remainder of the citations in this section occur as justifications of more isolated results. 

\begin{definition}
A finite group $G$ is a $\mathcal{K}$-group if every simple section of $G$ is a known finite simple group.
\end{definition}

\begin{definition}
Let $G$ be a finite group and $p$ a prime dividing $|G|$. Then $G$ is of \emph{characteristic $p$} if $C_G(O_p(G))\le O_p(G)$. Equivalently, if $F^*(G)=O_p(G)$.
\end{definition}

\begin{lemma}
Let $G$ be a finite group of characteristic $p$. If $H\normaleq\normaleq G$ or $O_p(G)\le H\le G$, then $H$ is of characteristic $p$.
\end{lemma}
\begin{proof}
This is elementary.
\end{proof}

\begin{lemma}[Coprime Action]
Suppose that a group $G$ acts on a group $A$ coprimely, and $B$ is a $G$-invariant subgroup of $A$. Then the following hold:
\begin{enumerate}
\item $C_{A/B}(G)=C_A(G)B/B$;
\item if $G$ acts trivially on $A/B$ and $B$, then $G$ acts trivially on $A$;
\item $[A, G]=[A,G,G]$;
\item $A=[A,G]C_A(G)$ and if $A$ is abelian $A=[A,G]\times C_A(G)$;
\item if $G$ acts trivially on $A/\Phi(A)$, then $G$ acts trivially on $A$;
\item if $p$ is odd, $A$ is a $p$-group and $G$ acts trivially on $\Omega(A)$, then $G$ acts trivially on $A$; and
\item for $S\in\syl_p(G)$, if $m_p(S)\geq 2$ then $A=\langle C_A(s) \mid s\in S\setminus\{1\}\rangle$.
\end{enumerate}
\end{lemma}
\begin{proof}
See, for instance, \cite[Chapter 8]{kurz}.
\end{proof}

\begin{lemma}\label{GLS2p'}
If a non-cyclic elementary abelian $p$-group $G$ acts on a $p'$-group $A$, then 
\[A=\langle C_A(B) \mid B\le G, [G: B]=p\rangle.\]
\end{lemma}
\begin{proof}
This is is \cite[Proposition 11.23]{GLS2}.
\end{proof}

\begin{lemma}[Burnside]\label{burnside}
Let $S$ be a finite $p$-group. Then $C_{\Aut(S)}(S/\Phi(S))$ is a normal $p$-subgroup of $\Aut(S)$.
\end{lemma}

The final result we describe here which still falls under the umbrella of ``coprime action'' and is essential to our analysis further on is the A$\times$B-lemma due to Thompson.

\begin{lemma}[A$\times$B-Lemma]
Let $AB$ be a finite group which acts on a $p$-group $V$. Suppose that $B$ is a $p$-group, $A=O^p(A)$ and $[A,B]=\{1\}=[A, C_V(B)]$. Then $[A,V]=\{1\}$.
\end{lemma}
\begin{proof}
See \cite[(24.2)]{asch2}.
\end{proof}

\begin{definition}
Let $G$ be a finite group and $S\in \syl_p(G)$. Then $G$ is \emph{$p$-minimal} if $S\not\normaleq G$ and $S$ is contained in a unique maximal subgroup of $G$.
\end{definition}

\begin{lemma}[McBride's Lemma]\label{McBride}
Let $G$ be a finite group, $S\in\syl_p(G)$ and $\mathcal{P}_G(S)$ denote the collection of $p$-minimal subgroups of $G$ over $S$. Then $G= \langle \mathcal{P}_G(S) \rangle N_G(S)$. Moreover, $O^{p'}(G)= \langle \mathcal{P}_G(S) \rangle.$
\end{lemma}
\begin{proof}
If $G\in \mathcal{P}_G(S)$ then the result holds trivially so assume that $G$ is counterexample to the first statement with $|G|$ minimal. Since $G$ is not $p$-minimal over $S$, there are maximal subgroups $M_1, M_2$ of $G$ which contain $S$. But then, since $G$ was a minimal counterexample, $M_i=\langle \mathcal{P}_{M_i}(S) \rangle N_{M_i}(S)$ for $i\in\{1,2\}$. Since $\mathcal{P}_{M_i}(S) \subseteq \mathcal{P}_{G}(S)$, $N_{M_i}(S)\le N_G(S)$ and $G=\langle M_1, M_2\rangle$, the result holds.

Now, let $P\in \mathcal{P}_{G}(S)$ and $x\in N_G(S)$. Then for $M$ the unique maximal subgroup of $P$ containing $S$, $M^x$ is the unique maximal subgroup of $P^x$ containing $S^x=S$, and $S\not\normaleq P^x$. It follows that $N_G(S)$ normalizes $\langle \mathcal{P}_G(S) \rangle$ and by the definition of $O^{p'}(G)$ and since $G=\langle \mathcal{P}_G(S) \rangle N_G(S)$, $O^{p'}(G)\le \langle \mathcal{P}_G(S) \rangle$. Now, suppose that there is $P\in \mathcal{P}_{G}(S)$ with $P\not\le O^{p'}(G)$. Then $O^{p'}(P)\le P\cap O^{p'}(G)<P$ and so $O^{p'}(P)$ is contained in the unique maximal subgroup of $P$ which contains $S$. Since $S$ is not normal in $P$, $N_P(S)$ is also contained in the unique maximal subgroup of $P$ containing $S$. But then, by the Frattini argument, $P=O^{p'}(P)N_P(S)<P$, a contradiction. Therefore, $\langle \mathcal{P}_G(S) \rangle\le O^{p'}(G)$ and the lemma holds.
\end{proof}

\begin{lemma}\label{p-min quot}
Suppose that $H$ is $p$-minimal over $S$ and $R$ is a normal $p$-subgroup of $H$. Then $H/R$ is $p$-minimal.
\end{lemma}
\begin{proof}
This is elementary.
\end{proof}

\begin{definition}
Let $G$ be a finite group and $H<G$. Then $H$ is \emph{strongly $p$-embedded} in $G$ if and only if $|H|_p>1$ and $N_G(Q)\le H$ for each non-trivial $p$-subgroup $Q$ with $Q\le H$. 
\end{definition}

\begin{lemma}\label{spelemma1}
Suppose that $G$ contains a strongly $p$-embedded subgroup $X$. Then the following hold:
\begin{enumerate}
\item $X$ contains a Sylow $p$-subgroup of $G$;
\item if $H\le G$ with $H\not\le X$ then provided $|H\cap X|_p>1$, $H\cap X$ is strongly $p$-embedded in $H$;
\item $O^{p'}(G)\cap X$ is strongly $p$-embedded in $O^{p'}(G)$; and
\item if $G\ne XO_{p'}(G)$, then $XO_{p'}(G)/O_{p'}(G)$ is strongly $p$-embedded in $G/O_{p'}(G)$.
\end{enumerate}
\end{lemma}
\begin{proof}
See \cite[Lemmas 3.2, 3.3]{parkerSE}.
\end{proof}

\begin{lemma}\label{CyclicSE}
If $G$ has a cyclic or generalized quaternion Sylow $p$-subgroup $T$ and $O_p(G)=1$, then $N_G(\Omega(T))$ is strongly $p$-embedded in $G$.
\begin{proof}
For $X\le T$ a non-trivial subgroup, $X$ is also cyclic or generalized quaternion and so also has a unique subgroup of order $p$. Thus, $\Omega(X)=\Omega(T)$ and since $O_p(G)\ne 1$, we have that $N_G(X)\le N_G(\Omega(X))=N_G(\Omega(T))<G$ so that $N_G(\Omega(T))$ is strongly $p$-embedded in $G$.
\end{proof}
\end{lemma}

Quite remarkably, possessing a strongly $p$-embedded subgroup is a surprisingly limiting condition. In the following two propositions, we roughly determine the structure of groups with strongly $p$-embedded subgroups. For $p=2$, we refer to work of Bender \cite{Bender}, while if $p$ is odd we make use of the classification of finite simple groups. In the application of these results, groups with strongly $p$-embedded subgroups will only ever appear in the local analysis of fusion systems. Particularly, these groups appear as automizers of certain $p$-subgroups and so would fit into the framework of any proofs utilizing a ``minimal counterexample'' hypothesis.

\begin{proposition}\label{SE1}
Suppose that $G=O^{p'}(G)$ is a group with a strongly $p$-embedded subgroup. Let $S\in\syl_p(G)$ and denote $\wt{G}:=G/O_{p'}(G)$. If $m_p(S)=1$ then one of the following holds:
\begin{enumerate}
\item $p$ is an odd prime, $S$ is cyclic, $G$ is perfect and $\wt G$ is a non-abelian finite simple group;
\item $G=SO_{p'}(G)$, $S$ is cyclic or generalized quaternion and $G$ is $p$-solvable; or
\item $p=2$, $S$ is generalized quaternion and $G/\Omega(S)O_{2'}(G)\cong K$, where either $\PSL_2(q)\le K\le \mathrm{P\Upgamma L}_2(q)$ for $q$ odd, or $K=\Alt(7)$.
\end{enumerate}
Moreover, in cases (ii) and (iii), $\langle \Omega(S)^G\rangle=\Omega(S)[\Omega(S), O_{p'}(G)]$ is the unique normal subgroup of $G$ which is divisible by $p$ and minimal with respect to this condition.
\end{proposition}
\begin{proof}
Since $m_p(S)=1$, $S$ is either cyclic or generalized quaternion by \cite[I.5.4.10 (ii)]{gor}. Suppose first that $S$ is cyclic. If $p=2$, then $G$ has a normal $2$-complement (see \cite[Theorem 7.4.3]{gor}) and (ii) holds. Hence, we may assume from now that $S$ is cyclic and $p$ is odd. Notice that $F(\wt{G})=O_p(\wt G)$ since $O_{p'}(\wt G)= \{1\}$. If $F^*(\wt G)=F(\wt G)=O_p(\wt G)$, then $O_p(\wt G)$ is self-centralizing and as $\wt S$ is abelian, we have that $O_p(\wt G)=\wt S$ and $SO_{p'}(G)\normaleq G$. In particular, $G=O^{p'}(G)\le SO_{p'}(G)\le G$, $G$ is $p$-solvable and (ii) holds.

Suppose now that $\wt G$ has a component $\wt L$. If $p\nmid |\wt L|$, then $L\le O_{p'}(E(\wt G))\le O_{p'}(\wt G)$, a contradiction. Hence, $p$ divides the order of any component of $\wt G$. Since $\wt S$ is cyclic, $\wt L$ has cyclic Sylow $p$-subgroups. By \cite[Lemma 33.14]{asch2}, $Z(\wt L)$ is a $p'$-prime group, and so $Z(\wt L)\le O_{p'}(E(\wt G))=\{1\}$ and $\wt L$ is simple. Notice also that since each component is simple, $E(\wt G)$ is a direct product of components, and since $p$ divides the order of any component, $E(\wt G)=\wt L$ is the unique component of $\wt G$, else $m_p(\wt G)=m_p(G)>1$. Since $O_p(\wt G)\cap E(\wt G)=\{1\}$, we have that $F^*(\wt G)=O_p(\wt G)\times E(\wt G)$ and since $m_p(\wt G)=1$, $O_p(\wt G)=\{1\}$. Therefore, $F^*(\wt G)$ is a non-abelian simple group. 

It remains to prove that $\wt S\le F^*(\wt G)$ to show that (i) holds. Form the group $\wt H=F^*(\wt G) \wt S$ and assume that $\wt H\ne F^*(\wt G)$. Note that by the Frattini argument, $\wt H=F^*(\wt G)N_{\wt H}(R)$ for all $R\in \syl_r(\wt F^*(\wt G))$. Moreover, for $r\ne p$ a prime, $\syl_r(F^*(\wt G))\subseteq \syl_r(\wt H)$. Then for $R\in \syl_r(F^*(\wt G))$ with $r\ne p$, let $P\in \syl_p(N_{\wt H}(R))$ and $T\in\syl_p(\wt H)$ containing $P$. Then $F^*(\wt H)\cap  T<T$ and as $T$ is cyclic and $\wt H=F^*(\wt G)N_{\wt H}(R)$, we deduce that $P=T$ and $N_{\wt H}(R)$ contains a Sylow $p$-subgroup of $\wt H$. Hence, by conjugacy, $\wt S$ normalizes a Sylow $r$-subgroup of $\wt H$, for all primes $r$. But then $\wt S$ normalizes a Sylow $r$-subgroup of $N_{\wt H}(\wt S)$ for all $r$, and so centralizes a Sylow $r$-subgroup of $N_{\wt H}(\wt S)$ for all $r$. Applying \cite[Theorem 7.4.3]{gor}, $\wt H$ has a normal $p$-complement, a contradiction since $\wt H$ contains a component of $\wt G$. Thus, $\wt S\le F^*(\wt G)$ and since $G=O^{p'}(G)$ it follows that $\wt G$ is a non-abelian simple group. Hence, $\wt{G'}=\wt G$ and so $S\le G'$. Then $G=O^{p'}(G)\le G'\le G$, $G$ is perfect and (i) holds.

Assume now that $p=2$ and $S$ is generalized quaternion. Then by the Brauer-Suzuki theorem (\cite[Theorem II.12.1.1]{gor}), we see that $\langle x\rangle O_{2'}(G)$ is a normal subgroup of $G$, and $\bar{G}:=G/\langle x\rangle O_{2'}(G)$ has dihedral Sylow $2$-subgroups. If $O_{2'}(\bar{G})\ne \{1\}$ then for $K$ the preimage in $\wt G$ of $O_{2'}(\bar{G})$, it is easily seen that $O_{2'}(K)\ne\{1\}$ and so $O_{2'}(\wt G)\ne\{1\}$, a contradiction. Thus, $\bar{G}$ is described by \cite[(III.16.3)]{gor} and the result holds.

Suppose case (ii) or (iii) occurs and let $N$ be a normal subgroup of $G$ whose order is divisible by $p$. Then, as $m_p(S)=1$, $\Omega(S)\le N$ and so $\Omega(S)[\Omega(S), O_{p'}(G)]=\Omega(S)[\Omega(S), G]=\langle \Omega(S)^G\rangle\le N$, and the result follows.
\end{proof}

\begin{remark}
Notice that if $H$ is a non-abelian finite simple with cyclic Sylow $p$-subgroups, then for $S\in\syl_p(H)$, $N_G(\Omega(S))$ is strongly $p$-embedded in $H$ by \cref{CyclicSE}. Thus, the description in case (i) is best possible up to a better understanding of $O_{p'}(G)$.  It is also worth noting that every non-abelian finite simple group has a cyclic Sylow $p$-subgroup for some odd prime $p$.
\end{remark}

\begin{proposition}\label{SE2}
Suppose that $G=O^{p'}(G)$ is a $\mathcal{K}$-group with a strongly $p$-embedded subgroup $X$. Let $S\in\syl_p(G)$ and set $\wt{G}:=G/O_{p'}(G)$. If $m_p(G)\geq 2$ then $\wt{G}$ is isomorphic to one of:
\begin{enumerate}
\item $\PSL_2(p^{a+1})$ or $\PSU_3(p^b)$ for $p$ arbitrary, $a\geq 1$ and $p^b>2$;
\item $\Sz(2^{2a+1})$ for $p=2$ and $a\geq 1$;
\item $\Alt(2p)$ for $p>3$;
\item $\mathrm{Ree}(3^{2a+1}), \PSL_3(4)$ or $\mathrm{M}_{11}$ for $p=3$ and $a\geq 0$;
\item $\Sz(32):5, {}^2\mathrm{F}_4(2)', \mathrm{McL}$ or $\mathrm{Fi}_{22}$ for $p=5$; or
\item $\mathrm{J}_4$ for $p=11$.
\end{enumerate}
\end{proposition}
\begin{proof}
If $G\ne XO_{p'}(G)$, then this follows from \cite[(2.5), (3.3)]{parkerSE} which in turn uses \cite[Theorem 7.6.1]{GLS3}. So assume that $G=XO_{p'}(G)$. By coprime action, \[O_{p'}(G)=\langle C_{O_{p'}(G)}(a) | 1\ne a\in S\rangle\] since $m_p(G)\geq 2$ and so $O_{p'}(G)\le X$ and $G=X$, a contradiction.
\end{proof}

As witnessed in the above classification, the generic examples of groups with a strongly $p$-embedded subgroup are rank $1$ simple groups of Lie type in characteristic $p$. These are the groups which will appear most often in later work, and so we take this opportunity to list some important properties of these groups (and some properties of groups which ``resemble'' rank $1$ groups). While almost all of these results are well known, we aim to provide explicit references and proofs of these results.

\begin{lemma}\label{SLGen}
Let $G\cong\PSL_2(p^n)$ or $\SL_2(p^n)$ and $S\in\syl_p(G)$. Then the following hold:
\begin{enumerate}
\item $S$ is elementary abelian of order $p^n$;
\item $\SL_2(2)\cong\Sym(3)$, $\PSL_2(3)\cong\Alt(4)$ and $\SL_2(3)$ are all solvable;
\item if $p=2$, then for $U\le S$ with $|U|=4$, there is $x\in G$ such that $G=\langle U, u^x\rangle$ for $1\ne u\in U$;
\item if $p=2$, all involutions in $S$ are conjugate and so, for $1\ne u \in S$ an involution, there is $x,y\in G$ such that $G=\langle u, u^x, u^y\rangle$;
\item if $p$ is odd, then for $1\ne u\in S$, there is $x\in G$ such that $G=\langle u, u^x\rangle$ unless $p^n=9$ in which case there is $x\in G$ such that $H:=\langle u, u^x\rangle<G$ is maximal subgroup of $G$ and $H/Z(H)\cong\PSL_2(5)$;
\item $N_G(S)$ is a solvable maximal subgroup of $G$ and for $K$ a Hall $p'$-subgroup of $N_G(S)$, $K/Z(G)$ is cyclic of order $(p^n-1)/(p^n-1,2)$ and $K$ acts fixed point freely on $S\setminus\{1\}$;
\item if $p^n\geq 4$, then $G$ is perfect and if $\wt G$ is a perfect central extension of $G$ by a group of $p'$-order, then $\wt G\cong\PSL_2(p^n)$ or $\SL_2(p^n)$; and
\item if $x$ is a non-trivial automorphism of $G$ which centralizes $S$, then $x\in\Aut_S(G)$.
\end{enumerate}
\end{lemma}
\begin{proof}
The proofs of (i)-(vi) are written out fairly explicitly in \cite[{{II.6-II.8}}]{Huppert}. Detailed information on automorphism groups and Schur multipliers is provided in \cite[Theorem 2.5.12]{GLS3} and \cite[Theorem 6.1.2]{GLS3}.
\end{proof}

\begin{lemma}\label{SU3Gen}
Let $G\cong\PSU_3(p^n)$ or $\SU_3(p^n)$ and $S\in\syl_p(G)$. Then the following hold:
\begin{enumerate}
\item $S$ is a special $p$-group of order $p^{3n}$ with $|Z(S)|=p^n$;
\item $\SU_3(2)$ is solvable, a Sylow $2$-subgroup of $\SU_3(2)$ is isomorphic to the quaternion group of order $8$ and $\SU_3(2)'\cong 3^{1+2}_+:2$ has index $4$ in $\SU_3(2)$;
\item for $p^n>2$, $N_G(S)$ is a solvable maximal subgroup of $G$ and for $K$ a Hall $p'$-subgroup of $N_G(S)$, $|K/Z(G)|=(p^{2n}-1)/(p^{2n}-1, 3)$ and $K$ acts irreducibly on $S/Z(S)$;
\item for $p^n>2$, $N_G(Z(S))=N_G(S)$ and for $K$ a Hall $p'$-subgroup of $N_G(S)$, $|C_K(Z(S))|=p^n+1$ and $C_K(Z(S))$ acts fixed point freely on $S/Z(S)$;
\item for any $x\in G\setminus N_G(S)$, $\langle Z(S), Z(S)^x\rangle\cong\SL_2(p^n)$ and $G=\langle Z(S), S^x\rangle$;
\item for $\{1\}\ne U\le Z(S)$, unless $p^n=9$ and $|U|=3$ or $p=2$ and $|U|=2$, there is $x,z\in G$ such that $G=\langle U, U^x, U^z\rangle$;
\item for $\{1\}\ne U\le Z(S)$, if $p^n=9$ and $|U|=3$ or $p=2<p^n$ and $|U|=2$, then there is $x,y,z\in G$ such that $G=\langle U, U^x, U^y, U^z\rangle$;
\item for $\{1\}\ne U\normaleq S$ with $U\not\le Z(S)$, if $p^n\ne 2$ then there is $x\in G$ such that $G=\langle U, U^x\rangle$;
\item if $p^n>2$, then $G$ is perfect and if $\wt G$ is a perfect central extension of $G$ by a group of $p'$-order, then $\wt G\cong\PSU_3(p^n)$ or $\SU_3(p^n)$; and
\item if $x$ is a non-trivial automorphism of $G$ which centralizes $S$, then $x\in\Aut_{Z(S)}(G)$.
\end{enumerate}
\end{lemma}
\begin{proof}
The proofs of (i)-(v) may be found in \cite[{{II.10}}]{Huppert}. Again, information on automorphism groups and Schur multipliers may be found in \cite[Theorem 2.5.12, Theorem 6.1.2]{GLS3}. It remains to prove (vi)-(viii).

For (vi) and (vii) suppose that $U\le Z(S)$, $p^n\ne 2$ and set $H:=\langle Z(S), Z(S)^x\rangle\cong\SL_2(p^n)$ for $x\in G\setminus N_G(S)$. By \cref{SLGen} (iv), (v), $H$ is generated by two or three conjugates of $U$, and by \cite{Mitchell2}, $H$ is contained in a unique maximal subgroup $M\cong \mathrm{GU}_2(p^n)\cong (p^n+1).\SL_2(p^n)$. Since $G=\langle U^G\rangle$, there is $z$ such that $U^z\not\le M$. It then follows from the maximality of $M$ in $G$ that $G=\langle H, U^z\rangle$ and (vi) and (vii) are proved.

Suppose now that $U\not\le Z(S)$, $U\normaleq S$ and $p^n\ne 2$. Since $U\not\le Z(S)$, $\{1\}\ne [U,S]\le Z(S)\cap U$. Set $C:=C_{N_G(S)}(Z(S))$ and observe that $C$ is irreducible on $S/Z(S)$ by (iv). Then, since $[U,S]\le Z(S)$,  $[U,S]=[U,S]^C=[\langle U^C\rangle, \langle S^C\rangle]$. By the irreducibility of $C$ on $S/Z(S)$, $(UZ(S)/Z(S))^C=S/Z(S)$ and so $[\langle U^C\rangle, \langle S^C\rangle]=Z(S)=[U,S]\le U$. Now, there is $x\in G\setminus N_G(S)$ such that $\langle Z(S), Z(S)^x\rangle\cong\SL_2(p^n)$ is contained in a unique maximal subgroup $M\cong \mathrm{GU}_2(p^n)$. Then, as $U>Z(S)$, $|U|>p^n$, $\langle Z(S), Z(S)^x\rangle<\langle U, U^x\rangle$ and (viii) follows.
\end{proof}

\begin{lemma}\label{SzGen}
Let $G\cong\Sz(2^n)$ and $S\in\syl_2(G)$. Then the following hold:
\begin{enumerate}
\item $n$ is odd and $3$ does not divide the order of $G$;
\item $\Sz(2)\cong 5:4$ is a Frobenius group, $\Phi(\Sz(2))=\{1\}$, $|\Sz(2)'|=5$ and a Sylow $2$-subgroup of $\Sz(2)$ is cyclic of order $4$;
\item if $n>1$ then $\Phi(S)=Z(S)=\Omega(S)$ and $S/\Phi(S)\cong \Phi(S)$ is elementary abelian of order $2^n$;
\item $N_G(S)$ is a solvable maximal subgroup of $G$ and for $K$ a Hall $2'$-subgroup of $N_G(S)$,$|K|=2^n-1$ and $K$ acts irreducibly on $S/\Phi(S)$ and $\Phi(S)$;
\item if $n>1$ then there is $x \in G$ such that $G=\langle Z(S), Z(S)^x\rangle$;
\item all involutions in $S$ are conjugate and if $n>1$, for $1\ne u\in Z(S)$, there is $x,y\in G$ such that $G=\langle u, u^x, u^y\rangle$;
\item for $U\le S$ with $U$ elementary abelian of order $4$, there is $x\in G$ such that $G=\langle U, U^x\rangle$;
\item if $n>1$ then $G$ is perfect and has trivial Schur multiplier; and
\item if $x$ is a non-trivial automorphism of $G$ which centralizes $S$, then $x\in\Aut_{Z(S)}(G)$.
\end{enumerate}
\end{lemma}
\begin{proof}
Most of the proofs of these facts may be found in \cite[Sections 13 - 16]{Suzuki}, except the proof of (viii) which may be gleaned from \cite[Theorem 6.1.2]{GLS3}.
\end{proof}

\begin{lemma}\label{ReeGen}
Let $G\cong\Ree(3^n)$ and $S\in\syl_3(G)$. Then the following hold:
\begin{enumerate}
\item $n$ is odd;
\item the Sylow $2$-subgroups of $G$ are abelian;
\item if $n=1$, then $G\cong\PSL_2(8):3$, $G'\cong\PSL_2(8)$, $S\cong 3^{1+2}_-$, $Z(S)=\Phi(S)$ has order $3$, $S\cap G'$ is cyclic of order $9$, $\Omega(S)$ is elementary abelian of order $9$ and $|S|=27$;
\item if $n>1$, then $S$ has order $3^{3n}$, $\Phi(S)=\Omega(S)$ has order $3^{2n}$, $Z(S)=[S, \Phi(S)]$ has order $3^n$ and $S/\Phi(S)\cong \Phi(S)/Z(S)\cong Z(S)$ is elementary abelian of order $3^n$;
\item $N_G(S)$ is a solvable maximal subgroup of $G$ and for $K$ a Hall $3'$-subgroup of $N_G(S)$, $|K|=3^n-1$ and $K$ acts irreducibly on $S/\Omega(S)$, $\Omega(S)/Z(S)$ and $Z(S)$;
\item for $\{1\}\ne U\normaleq S$, if $n>1$ then there is $x,y\in G$ such that $G=\langle U, U^x, U^y\rangle$;
\item if $n>1$ then $G$ is perfect and has trivial Schur multiplier, and $\Ree(3)'$ is perfect and has trivial Schur multiplier; and
\item if $x$ is a non-trivial automorphism of $G$ which centralizes $S$, then $x\in\Aut_{Z(S)}(G)$.
\end{enumerate}
\end{lemma}
\begin{proof}
The proofs of (i) to (v) follow from the main theorem of \cite{Ward} while (vii) and (viii) follow from \cite[Theorem 2.5.12, Theorem 6.1.2]{GLS3}. We make use of results in \cite{Ward} to prove (vi). Since the results when $n=1$ are easily verified, we assume that $n>1$ throughout. Note that if $U\not\le \Omega(S)$, then as $U\normaleq S$ and $\Omega(S)=Z_2(S)$, $Z(S)<Z(S)[U, S]\le \Omega(S)$ and $\Omega(U)\not\le Z(S)$. In particular, if $U\not\le Z(S)$, then there is $u\in \Omega(U)\setminus Z(S)$.

Suppose first that $U\le Z(S)$. Since $K$ is irreducible on $Z(S)$, there is $y\in N_G(S)$ such that for some $u\in U$, $u^y$ is not represented by elements of a subfield of $\mathrm{GF}(3^n)$. Then $\langle U, U^y\rangle$ is elementary abelian of order at least $9$ and contained in a maximal subgroup of $G$. Considering the maximal subgroup structure of $G$ (as in \cite[Theorem C]{Kleidman}) and using that the centralizer of an involution in $K$ intersects $Z(S)$ trivially, we deduce that $\langle U, U^y\rangle$ lies in a unique maximal subgroup, namely $N_G(S)$. But now, there is $x\in G\setminus N_G(S)$ such that $U^x\not\le S$. Thus, $G=\langle U, U^x, U^y\rangle$, as required.

Suppose now that $U\not\le Z(S)$ so that there is $u\in U$ such that $u\in\Omega(U)\setminus Z(S)$. Then by (v), it follows that $C_{N_G(S)}(u)=\Omega(S)\langle i\rangle$, where $i\in K$ is an involution. Then $u\in C_G(i)$ and by \cite{Ward}, $C_G(i)\cong\langle i\rangle \times L$, where $L\cong \PSL_2(3^n)$, and $C_G(i)$ is a maximal subgroup of $G$ (see also \cite[Theorem C]{Kleidman}). Since $n>1$ is odd, there is $x\in L$ such $L=\langle u, u^x\rangle$ by \cref{SLGen} (v). Further, $C_G(i)\cap Z(S)=\{1\}$ and since $U\cap Z(S)\ne\{1\}$ as $U\normaleq S$, $L< \langle U, U^x\rangle$ and since $C_G(i)$ is maximal, it follows that $G=\langle U, U^x\rangle$. 
\end{proof}

Pivotal to the analysis of local actions in the amalgam method and within a fusion system is recognizing $\SL_2(p^n)$ acting on its modules in characteristic $p$. Below, we list the most important modules for this work.

\begin{definition}
Let $X\cong\SL_2(q)$, $q=p^n$, $k=\mathrm{GF}(q)$ and $V$ a faithful $2$-dimensional $kX$-module.
\begin{itemize}
\item $V|_{\mathrm{GF}(p)X}$ is a \emph{natural $\SL_2(q)$-module} for $X$. 
\item A \emph{natural $\Omega_3(q)$-module} for $X$ is the $3$-dimensional submodule of $V\otimes_k V$ regarded as a $\mathrm{GF}(p)X$-module by restriction, and is irreducible whenever $p$ is an odd prime. 
\item If $n=2a$ for some $a\in \N$, a \emph{natural $\Omega_4^-(q^{\frac{1}{2}})$-module} for $X$ is any non-trivial irreducible submodule of $(V\otimes_k V^{\tau})|_{\mathrm{GF}(q^{\frac{1}{2}})X}$, where $\tau$ is an automorphism of $\mathrm{GF}(q)$ of order $2$, regarded as a $\mathrm{GF}(p)X$-module by restriction. 
\item If $n=3a$ for some $a\in \N$, a \emph{triality module} for $X$ is any non-trivial irreducible submodule of $(V\otimes V^{\tau}\otimes V^{\tau^2})|_{\mathrm{GF}(q^{\frac{1}{3}})X}$, where $\tau$ is an automorphism of $k$ of order $3$, regarded as a $\mathrm{GF}(p)X$-module by restriction.
\end{itemize}
\end{definition}

\begin{lemma}\label{NatMod}
Suppose $G\cong\SL_2(p^n)$, $S\in\syl_p(G)$ and $V$ is natural $\SL_2(p^n)$-module. Then the following hold:
\begin{enumerate}
\item $[V, S,S]=\{1\}$;
\item $|V|=p^{2n}$ and $|C_V(S)|=p^n$;
\item $C_V(s)=C_V(S)=[V,S]=[V,s]=[v, S]$ for all $v\in V\setminus C_V(S)$ and $1\ne s\in S$;
\item $V=C_V(S)\times C_V(S^g)$ for $g\in G\setminus N_G(S)$;
\item every $p'$-element of $G$ acts fixed point freely on $V$;
\item $V$ is self dual; and
\item $V/C_V(S)$ and $C_V(S)$ are irreducible $\mathrm{GF}(p)N_G(S)$-modules upon restriction.
\end{enumerate}
\end{lemma}
\begin{proof}
See \cite[Lemma 4.6]{ParkerBN}.
\end{proof}

\begin{lemma}\label{NatGen}
Suppose that $G\cong\SL_2(q)$, $q=p^n$, and $V$ is a direct sum of two natural $\SL_2(q)$-modules. If $U\le C_V(S)$ is $N_G(S)$-invariant and of order $q$, then $|\langle U^G\rangle|=q^2$.
\end{lemma}
\begin{proof}
By \cite[(I.3.5.6)]{gor}, the number of distinct irreducible submodules of $V$ is $q+1=(q^2-1)/q-1$. For each $W$ an irreducible submodule, $C_W(S)$ is $N_G(S)$-invariant and of order $q$. Since $C_V(S)$ may be viewed as direct sum of two irreducible modules for $N_G(S)$, again applying \cite[(I.3.5.6)]{gor} we have that there are $q+1$ $N_G(S)$-invariant subgroups of $C_V(S)$ of order $q$. Hence, each $N_G(S)$-invariant subgroup of $C_V(S)$ of order $q$ is of the form $C_W(S)$ for some irreducible module $W$ and $C_W(S)$ determines $W$ uniquely. Thus, $U$ uniquely determines a submodule $\langle U^G\rangle$ of order $q^2$.
\end{proof}

\begin{lemma}\label{Omega3}
Suppose that $G\cong\SL_2(p^n)$, $p$ an odd prime, $S\in \syl_p(G)$ and $V$ is a natural $\Omega_3(p^n)$-module for $G$. Then the following hold:
\begin{enumerate}
\item $C_G(V)=Z(G)$;
\item $[V, S,S, S]=\{1\}$;
\item $|V|=p^{3n}$ and $|V/[V,S]|=|C_V(S)|=p^n$;
\item $[V, S]=[V, s]$ and $[V,S,S]=[V,s,s]=C_V(s)=C_V(S)$ for all $1\ne s\in S$;
\item $[V.S]/C_V(S)$ is centralized by $N_G(S)$; and
\item $V/[V,S]$ and $C_V(S)$ are irreducible $\mathrm{GF}(p)N_G(S)$-modules upon restriction.
\end{enumerate}
\end{lemma}
\begin{proof}
See \cite[Lemma 4.7]{ParkerBN}.
\end{proof}

\begin{lemma}\label{Omega4}
Let $G\cong \SL_2(p^{2n})$, $S\in\syl_p(G)$ and $V$ a natural $\Omega_4^-(p^n)$-module for $G$. Then the following hold:
\begin{enumerate}
\item $C_G(V)=Z(G)$;
\item $[V, S,S, S]=\{1\}$;
\item $|V|=p^{4n}$ and $|V/[V,S]|=|C_V(S)|=p^n$;
\item $|C_V(s)|=|[V,s]|=p^{2n}$ and $[V,S]=C_V(s)\times [V,s]$ for all $1\ne s\in S$; and
\item $V/[V,S]$ and $C_V(S)$ are irreducible $\mathrm{GF}(p)N_G(S)$-modules upon restriction.
\end{enumerate}
Moreover, for $\{1\}\ne F\le S$, one of the following occurs:
\begin{enumerate}[label=(\alph*)]
\item $[V, F]=[V, S]$ and $C_{V}(F)=C_{V}(S)$;
\item $p=2$, $[V, F]=C_{V}(F)$ has order $p^{2n}$, $F$ is quadratic on $V$ and $|F|\leq p^n$; or
\item $p$ is odd, $|[V, F]|=|C_{V}(F)|=p^{2n}$, $[V, S]=[V, F]C_{V}(F)$, $C_V(S)=C_{[V, F]}(F)$ and $|F|\leq p^n$.
\end{enumerate}
\end{lemma}
\begin{proof}
See \cite[Lemma 4.8]{ParkerBN} and \cite[Lemma 3.15]{parkerSymp}.
\end{proof}

We require one miscellaneous result concerning the exceptional $1$-cohomology of $\PSL_2(9)$ on a natural $\Omega_4^-(3)$-module.

\begin{lemma}\label{A6Cohom}
Suppose that $G\cong \PSL_2(p^2)$, $p\in\{2,3\}$ and $S\in\syl_p(G)$. If $V$ is a $5$-dimensional $\mathrm{GF}(p)G$-module such that $V/C_V(G)$ is isomorphic to a natural $\Omega_4^-(p)$-module, then either $V=[V, G]\times C_V(G)$; or $p=3$ and $[V, S, S]$ is $2$-dimensional as a $\mathrm{GF}(3)S$-module.
\end{lemma}
\begin{proof}
This follows from direct computation in $\GL_5(p)$.
\end{proof}

\begin{lemma}\label{trialitydescription}
Suppose that $G\cong\mathrm{(P)SL}_2(p^{3n})$, $S\in \syl_p(G)$ and $V$ is a triality module for $G$. Then the following hold:
\begin{enumerate}
\item $[V, S, S, S, S]=\{1\}$;
\item $|V|=p^{8n}$, $|V/[V,S]|=|C_V(S)|=|[V, S, S, S]|=p^n$ and $|[V,S,S]|=p^{4n}$;
\item if $p$ is odd then $|V/C_V(s)|=p^{5n}$, while if $p=2$ then $|V/C_V(s)|=p^{4n}$, for all $1\ne s\in S$; and 
\item $V/[V,S]$ and $C_V(S)$ are irreducible $\mathrm{GF}(p)N_G(S)$-modules upon restriction.
\end{enumerate}
\end{lemma}
\begin{proof}
See \cite[Lemma 4.10]{ParkerBN}.
\end{proof}

We will also need facts concerning the natural modules for $\SU_3(p^n)$ and $\Sz(2^n)$. 

\begin{definition}
The natural modules for $\SU_3(p^n)$ and $\Sz(2^n)$ are the unique irreducible $\mathrm{GF}(p)$-modules of smallest dimension. Equivalently, they may be viewed as the restrictions of a ``natural'' $\SL_3(p^{2n})$-module and $\Sp_4(2^n)$-module respectively.
\end{definition}

\begin{lemma}\label{SUMod}
Suppose $G\cong\SU_3(p^n)$, $S\in\syl_p(G)$ and $V$ is a natural module.  Then the following hold:
\begin{enumerate}
\item $C_V(S)=[V, Z(S)]=[V, S,S]$ is of order $p^{2n}$;
\item $C_V(Z(S))=[V,S]$ is of order $p^{4n}$; and
\item $V/[V, S]$, $[V, S]/C_V(S)$ and $C_V(S)$ are irreducible $\mathrm{GF}(p)N_G(S)$-modules upon restriction.
\end{enumerate}
\end{lemma}
\begin{proof}
See \cite[Lemma 4.13]{ParkerBN}.
\end{proof}

\begin{lemma}\label{SzMod}
Suppose $G\cong\Sz(2^n)$, $S\in\syl_2(G)$ and $V$ is the natural module. Then the following hold:
\begin{enumerate}
\item $[V, S]$ has order $2^{3n}$;
\item $[V,\Omega(S)]=C_V(\Omega(S))=[V,S,S]$ has order $2^{2n}$;
\item $C_V(S)=[V,S,\Omega(S)]=[V, \Omega(S), S]=[V,S,S,S]$ has order $2^n$; and
\item $V/[V,S]$, $[V,S]/C_V(\Omega(S))$, $C_V(\Omega(S))/C_V(S)$ and $C_V(S)$ are all irreducible $\mathrm{GF}(p)N_G(S)$-modules upon restriction.
\end{enumerate}
\end{lemma}
\begin{proof}
This is an elementary calculation in $\Sp_4(2^n)$.
\end{proof}

Given the descriptions of rank $1$ Lie type groups and their modules, we now require ways to identify them. Furthermore, we would like to have ways to completely determine a group $G$ with a strongly $p$-embedded subgroup, and its actions, given reasonably general hypotheses. We achieve this through characterizations of FF-modules, quadratic action and some Hall--Higman type arguments. We first list some generic module results which will be used throughout this work.

\begin{lemma}\label{SplitMod}
Let $G$ be a group and $V$ be a faithful $\mathrm{\GF}(p)G$-module. Let $T\in\syl_p(O^p(G))$ and assume that $V = \langle C_V(T)^G\rangle$. Then $V=[V, O^p(G)]C_V(O^p(G))$.
\end{lemma}
\begin{proof}
See \cite[Lemma 1.1]{ChermakSmall}.
\end{proof}

We will require some knowledge of the minimal $\mathrm{GF}(p)$-representations of groups with a strongly $p$-embedded subgroups, especially when their $p$-rank is at least $2$. The following results achieve this.

\begin{lemma}\label{SpeMod2}
Let $G=O^{2'}(G)$ be a group with a strongly $2$-embedded subgroup and $m_2(G)>1$. Assume that $V$ is a faithful $\mathrm{GF}(2)$-module for $G$, let $S\in\syl_2(G)$ and $A\le Z(S)$ of order $4$. Then the following hold:
\begin{enumerate}
    \item if $G/O_{2'}(G)\cong \PSL_2(q)$ then $|V|\geq q^2$, $|V/C_V(A)|\geq q$ and $|V/C_V(s)|\geq q^{\frac{2}{3}}$ for any non-trivial $s\in Z(S)^\#$;
    \item if $G/O_{2'}(G)\cong \Sz(q)$ then $|V|\geq q^4$, $|V/C_V(A)|\geq q^2$ and $|V/C_V(s)|\geq q^{\frac{4}{3}}$ for any non-trivial $s\in Z(S)^\#$; and
    \item if $G/O_{2'}(G)\cong \PSU_3(q)$ then $|V|\geq q^6$, $|V/C_V(A)|\geq q^2$ and $|V/C_V(s)|\geq q^{\frac{3}{2}}$ for any non-trivial $s\in Z(S)^\#$.
\end{enumerate}
\end{lemma}
\begin{proof}
The bounds on $V$ follow from the proof of \cite[Lemma 1.7]{OliMod}. Set $d_A$ be the number of conjugates of $A$ required to generate $G/O_{2'}(G)$ and $d_s$ the number of conjugates of $s$ required to generate $G/O_{2'}(G)$. Then we form $H\le G$ from $d_A$ conjugates of $A$ or $d_s$ conjugates of $s$ with the property $G=HO_{2'}(G)$. Moreover, $H$ acts faithfully on $V$ and $|V/C_V(H)|\leq \mathrm{min}(|V/C_V(A)|^{d_A}, |V/C_V(s)|^{d_s})$. Since $H$ is also a group with a strongly $2$-embedded subgroup, $|V/C_V(H)|\geq q^2, q^4$ or $q^6$ respectively. Applying \cref{SLGen} (iii), (iv) when $G/O_{2'}(G)\cong \PSL_2(q)$, \cref{SzGen} (vi), (vii) when $G/O_{2'}(G)\cong \Sz(q)$ and \cref{SU3Gen} (vi), (vii) when $G/O_{2'}(G)\cong \PSU_3(q)$, we deduce the bounds on $|V/C_V(A)|$ and $|V/C_V(s)|$.
\end{proof}

\begin{lemma}\label{SpeModOdd}
Let $G=O^{p'}(G)$ be a $\mathcal{K}$-group with a strongly $p$-embedded subgroup, $m_p(G)>1$ and $p$ an odd prime. Assume that $V$ is a faithful $\mathrm{GF}(p)$-module for $G$, let $S\in\syl_p(G)$ and $A\le Z(S)$ of order $p^2$. Then the following hold:
\begin{enumerate}
    \item if $G/O_{p'}(G)\cong \PSL_2(q)$ then $|V|\geq q^2$, $|V/C_V(A)|\geq q$ and $|V/C_V(s)|\geq q^{\frac{2}{3}}$ for any non-trivial $s\in Z(S)^\#$;
    \item if $G/O_{3'}(G)\cong \Ree(q)$ with $q>3$ then $|V|\geq \mathrm{max}(q^6, 3^7)$ and $|V/C_V(s)|\geq \mathrm{max}(q^2, 9)$ for any non-trivial $s\in Z(S)^\#$;
    \item if $G/O_{p'}(G)\cong \PSU_3(q)$ then $|V|\geq q^6$, $|V/C_V(A)|\geq q^2$ and $|V/C_V(s)|\geq q^{\frac{3}{2}}$ for any non-trivial $s\in Z(S)^\#$; and
    \item $|V|\geq p^5$ otherwise.
\end{enumerate}
\end{lemma}
\begin{proof}
Suppose first that $G/O_{p'}(G)$ is not isomorphic to the derived subgroup of a rank $1$ finite group of Lie type, including $\Ree(3)$. Then one can check, comparing with \cref{SE2}, that $|G/O_{p'}(G)|$ does not divide the order of $\GL_4(p)$. Hence, $G/O_{p'}(G)$ is isomorphic to a rank $1$ finite group of Lie type and we appeal to the proof of \cite[Lemma 4.6]{Comp1} to deduce the bounds on $V$. For $G/O_{3'}(G)\cong \Ree(3)$, we simply calculate in $\GL_6(3)$ using MAGMA. As in \cref{SpeMod2}, using instead \cref{SLGen} (v), \cref{SU3Gen} (vi), (vii) and \cref{ReeGen} (vi) for the values of $d_A$ and $d_s$, we we deduce the bounds on $|V/C_V(A)|$ and $|V/C_V(s)|$.
\end{proof}

We require, at least when $p$ is an odd prime, a way to distinguish between $\SL_2(p^n)$ and $\PSL_2(p^n)$ from a strongly $p$-embedded hypothesis. Additionally, as can be seen from the \hyperlink{MainThm}{Main Theorem}, none of the configurations we are interested in have Ree groups as their automizers, so we will also have to dispel of this case later on. Generally, we achieve this using quadratic action.
 
\begin{definition}
Let $G$ be a finite group and $V$ an $\mathrm{GF}(p)G$-module. If $A\le G$ satisfies $[V, A,A]=\{1\}\ne [V,A]$, then $A$ acts \emph{quadratically} on $V$ and if $[V, A,A,A]=\{1\}$ and $A$ is not quadratic or trivial on $V$, then $A$ acts \emph{cubically}.
\end{definition}

\begin{lemma}\label{NotQuad}
Suppose that $V$ is an irreducible $\mathrm{GF}(p)$-module for $G\cong \Ree(3^n)$ or $G\cong \PSL_2(p^n)\not\cong\SL_2(p^n)$. If there is a non-trivial subgroup $A$ of $G$ with $[V, A,A]=\{1\}$, then $[V, A]=[V, G]=\{1\}$.
\end{lemma}
\begin{proof}
Since the Sylow $2$-subgroups of $\PSL_2(p^n)$ are either abelian or dihedral and the Sylow $2$-subgroups of $\Ree(3^n)$ are abelian, this follows from \cite[(I.3.8.4)]{gor}.
\end{proof}

For $p\geq 5$, the pairs $(G,V)$ where $G$ is a group acting faithfully on a module $V$ such that $G$ is generated by elements which act quadratically on $V$ were classified by Thompson. Thompson's results were extended to the prime $3$ by work of Ho. It seems imperative to emphasize that the these works predate the classification of finite simple groups. For convenience, the version we use here is by Chermak and utilizes the classification of finite simple groups, although as we stressed earlier, these groups will only ever appear as local subgroups in any arguments.

\begin{lemma}\label{SEQuad}
Suppose $G$ is a $\mathcal{K}$-group which has a strongly $p$-embedded subgroup for $p$ an odd prime and $V$ be a faithful $\mathrm{GF}(p)$-module for $G$. Suppose there is a $p$-subgroup $A\le G$ such that $[V,A,A]=\{1\}$ and set $L=\langle A^G\rangle$. Then one of the following occurs:
\begin{enumerate}
\item $G=L\cong \SL_2(p^n)$ where $p$ is any odd prime;
\item $G=L\cong\mathrm{(P)SU}_3(p^n)$ where $p$ is any odd prime; or
\item $G=L$, $G/C_G(U)\cong \SL_2(3), 2\cdot \Alt(5)$ or $2^{1+4}_-.\Alt(5)$ for every non-trivial irreducible composition factor $U$ of $V$, $p=3$ and $|S|=3$.
\end{enumerate}
\end{lemma}
\begin{proof}
This follows from \cite{ChermakQuadN}, \cite{ChermakQuad}, \cref{NotQuad} and a comparison with the groups listed in \cref{SE1}, \cref{SE2}.
\end{proof}

As the typical example in our case will be central extensions of a rank $1$ simple groups of Lie type, we also provide a generic result regarding these particular actions.

\begin{lemma}\label{GreenbookQuad}
Suppose that $G$ is a rank $1$ group of Lie type, $S\in\syl_p(G)$ and let $\{1\}\ne A\le S$ with $[V, A, A]=\{1\}$ for some faithful $\GF(p)G$-module $V$. Then $A\le \Omega(Z(S))$ and if $|A|\geq 3$, for any non-trivial element $t\in A$, the following hold:
\begin{enumerate}
    \item If $G\cong \SL_2(p^n)$, then $|V/C_V(t)|\geq p^n$.
    \item If $G\cong \Sz(2^n)$ or $\mathrm{(P)SU}_3(p^n)$ then $|V/C_V(t)|\geq p^{2n}$.
\end{enumerate}
\end{lemma}
\begin{proof}
See \cite[(5.9)]{Greenbook} and \cite[(5.10)]{Greenbook}.
\end{proof}

More than just a quadratic module, the natural module for $\SL_2(p^n)$ provides the minimal example of an \emph{FF-module}. FF-modules are named due to how they arise as counterexamples to \emph{Thompson factorization} (see \cite[{{32.11}}]{asch2}), which aims to factorize a group into two $p$-local subgroups. One of these $p$-local subgroups is the normalizer of the \emph{Thompson subgroup} of a fixed Sylow $p$-subgroup.

\begin{definition}
Let $S$ be a finite $p$-group. Set $\mathcal{A}(S)$ to be the set of all elementary abelian subgroups of $S$ of maximal rank. Then the \emph{Thompson subgroup} of $S$ is defined as $J(S):=\langle A \mid A\in\mathcal{A}(S)\rangle$.
\end{definition}

\begin{proposition}\label{BasicJS}
Let $S$ be a finite $p$-group. Then the following hold:
\begin{enumerate}
\item $J(S)$ is a non-trivial characteristic subgroup of $S$;
\item for $A\in\mathcal{A}(S)$, $A=\Omega(C_S(A))$;
\item $\Omega(C_S(J(S)))=\Omega(Z(J(S)))=\bigcap_{A\in\mathcal{A}(S)} A$; and
\item if $J(S)\le T\le S$, then $J(S)=J(T)$.
\end{enumerate}
\end{proposition}
\begin{proof}
See \cite[{{9.2.8}}]{kurz}.
\end{proof}

\begin{definition}
Let $G$ be a finite group and $V$ a $\mathrm{GF}(p)$-module. If there exists $A\le G$ such that
\begin{enumerate}
\item $A/C_A(V)$ is an elementary abelian $p$-group;
\item $[V,A]\ne\{1\}$; and 
\item $|V/C_V(A)|\leq |A/C_A(V)|$
\end{enumerate}
then $V$ is a \emph{failure to factorize module} (abbrev. FF-module) for $G$ and $A$ is an \emph{offender} on $V$.
\end{definition}

The following proposition describes a fairly natural situation in which one can identify an FF-module from a group failing to satisfy Thompson factorization. This result is well known and the proof is standard (see \cite[{{9.2}}]{kurz}).

\begin{proposition}\label{BasicFF}
Let $G$ be a finite group with $S\in\syl_p(G)$ and $F^*(G)=O_p(G)$. Set $V:=\langle \Omega(Z(S))^G\rangle$. Then $O_p(G)=O_p(C_G(V))$ and $O_p(G/C_G(V))=\{1\}$. Furthermore, if $\Omega(Z(S))<V$ and $J(S)\not\le C_S(V)$ then $V$ is an FF-module for $G/C_G(V)$.
\end{proposition}

As a counterpoint to the determination of groups with a strongly $p$-embedded subgroup, whenever a group with a strongly $p$-embedded subgroup has an associated FF-module, we can almost completely determine the group and its action without the need for a $\mathcal{K}$-group hypothesis. Indeed, the following lemma relies only on a specific case in the Local $C(G,T)$-theorem \cite{CGT}.

\begin{lemma}\label{SEFF}
Suppose $G=O^{p'}(G)$ has a strongly $p$-embedded subgroup and $V$ is (dual to) a faithful FF-module. Then $G\cong \SL_2(p^n)$ and $V/C_V(O^p(G))$ is the natural module.
\end{lemma}
\begin{proof}
See \cite[Theorem 5.6]{henkesl2}.
\end{proof}

Given a way to characterize a natural $\SL_2(p^n)$-module, it is a natural to ask whether we can characterize some other modules, particularly those irreducible modules associated to Lie type groups of rank $1$.

\begin{lemma}\label{DirectSum}
Let $G\cong\SL_2(p^n)$ and $S\in\syl_p(G)$. Suppose that $V$ is a module for $G$ over $\GF(p)$ such that $[V,S,S]=\{1\}$ and such that $[V, O^p(G)]\ne\{1\}$. Then $[V/C_V(O^p(G)), O^p(G)]$ is a direct sum of natural modules for $G$.
\end{lemma}
\begin{proof}
See \cite[Lemma 2.2]{ChermakQuad}.
\end{proof}

\begin{lemma}\label{SL2ModRecog}
Let $G\cong\SL_2(p^n)$, $S\in\syl_p(G)$ and $V$ an irreducible $\GF(p)G$-module. If $|V|\leq p^{3n}$ then both $C_V(S)$ and $V/[V,S]$ are irreducible as $N_G(S)$-modules, $|C_V(S)|=|V/[V, S]|$ and either
\begin{enumerate}
\item $V$ is natural $\SL_2(p^n)$-module, $|V|=p^{2n}$ and $|C_V(S)|=p^n$;
\item $V$ is natural $\Omega_4^-(p^{\frac{n}{2}})$, $n$ is even, $|V|=p^{2n}$ and $|C_V(S)|=p^{\frac{n}{2}}$;
\item $V$ is natural $\Omega_3(p^n)$, $p$ is odd, $|V|=p^{3n}$ and $|C_V(S)|=p^n$; or
\item $V$ is a triality module, $n=3r$ for some $r\in\N$, $|V|=p^{\frac{8n}{3}}$ and $|C_V(S)|=p^{\frac{n}{3}}$.
\end{enumerate}
\end{lemma}
\begin{proof}
This is \cite[Lemma 2.6]{ChermakJ}. 
\end{proof}

\begin{lemma}\label{q^3module}
Assume that $G=O^{p'}(G)$, $G/O_{p'}(G)\cong \PSL_2(q)$ and $V$ is a faithful, irreducible $\mathrm{GF}(p)$-module for $p\in\{2,3\}$, where $q=p^n$ and $n>1$. If $|V|\leq q^3$, then $O_{p'}(G)\le Z(G)$ and both $G$ and $V$ are determined in \cref{SL2ModRecog}.
\end{lemma}
\begin{proof}
Assume that $G$ is a minimal counterexample to the lemma and let $S\in\syl_p(G)$. If $E(G)\not\le O_{p'}(G)$, then $S\le E(G)$, $E(G)$ is quasisimple and as $G=O^{p'}(G)$, $G=E(G)$ is quasisimple and both $G$ and $V$ are determined in \cref{SL2ModRecog}. Hence $F^*(G)\le O_{p'}(G)$. If $p=3$ and $E(G)\ne\{1\}$, we may assume that for $K$ a component of $G$, $K\cong \Sz(l)$ for some $l=2^{2t+1}$. Moreover, we have that $G\ne N_G(K)O_{3'}(G)$, for otherwise we by minimality, $K\normaleq G$. But then, $C_S(K)\ne\{1\}$ since $\Out(K)$ is cyclic from which it follows that $[S, K]=\{1\}$ so that $K\le Z(G)$, a contradiction. But now, $N_G(K)$ has index at least $q+1$ in $G$, except when $q=9$ in which case $N_G(K)$ has index at least $6$. In the latter case, we have that $|V|\leq 3^6$, and one can check that $\Sz(8)$ has no non-trivial modules of this size, so $\Sz(l)$ does not. Let $U$ be a non-trivial constituent of $V$ restricted to $H:=\langle K^G\rangle$, viewed as a $\mathbb{F} H$-module where $\mathbb{F}$ is a splitting field for $K$. Applying \cite[Proposition 9.14]{GLS2} iteratively, we may write $U$ as a tensor product of $U=\otimes_{g\in G} W_g$ where each $W_g$ is an irreducible $\mathbb{F} K^g$-module. Hence, $3^{3n}=q^3\geq |V|\geq |U|\geq |W_g|^{q+1}=|W_g|^{3^n+1}$ for $n\in\N$, a contradiction.

Thus, there is a prime $r$ dividing $|G|$ such that $[O_r(G), x]\ne\{1\}$ for all $x\in S^\#$. Choose $K\normaleq G$ minimally by inclusion such that $K\le O_r(G)$, $[K, O_r(G)]\le Z(K)$ and $[K, x]\ne\{1\}$ for all $x\in S^\#$. Such a subgroup exists by the critical subgroup theorem (see \cref{CriticalSubgroup}). Furthermore, if $r$ is an odd prime then by minimality, we deduce that $\Omega(K)=K$.

Let $T$ be a proper, non-trivial characteristic subgroup of $K$. By minimality, we have that there is $x\in S^\#$ such that $[x, T]=\{1\}$ from which it follows that $T\le Z(G)$. In particular, $T$ is abelian. By Clifford theory (see \cite[Theorem 9.7]{GLS2}), $V|_{T}=\bigoplus\limits_{i\in I} W_i$ where $G$ acts transitively on the set $\{W_i\}$ and each $W_i$ is irreducible for $T$. Since $T$ is abelian, by Schur's lemma, we have that $T/C_{T}(W_i)$ is a cyclic $r$-group and since $C_{T}(W_i)\le Z(G)$, $C_{T}(W_i)\le C_{T}(V)=\{1\}$. Hence, $T$ is cyclic.

Assume that $\Omega(K)<K$ so that $r=2$. Then, as above, $\Omega(K)$ is cyclic and we have that $|\Omega(K)|=2$ and $K$ has a unique subgroup of order $2$. Indeed, $K$ is either cyclic or generalized quaternion. Since $n>1$ and $S$ acts faithfully on $K$, we have a contradiction. Hence, $\Omega(K)=K$ and every proper, non-trivial characteristic subgroup of $K$ is of order $p$. If $K$ is non-abelian, then $Z(K)=\Phi(K)$ is of order $r$ and $K$ is extraspecial. Set $|K|=r^{2m+1}$. Since $K$ acts faithfully on $V$, by \cite[Lemma 9.5]{GLS2}, we deduce that $|V|\geq p^{r^m}$ so that $3n\geq r^m$. Since $S$ acts faithfully on $K$ and $G/O_{p'}(G)$ is simple, by \cite{Winter}, $S$ embeds in $\Sp_{2m}(r)$ if $r$ is odd, and $S$ embeds in $\mathrm{SO}_{2m}^{\pm}(2)$ if $r=2$. Applying \cite[Corollary 13.23]{parkerSymp}, if $p=3$ then either $r=2$ and $m\leq 4$ or $r=5$ and $m=1$. Likewise, if $p=2$, then $r=3$ and $m\leq 3$ or $r\leq 11$ and $m=1$. For these small cases, we employ MAGMA to observe that $G/O_{p'}(G)\cong \PSL_2(q)$ never occurs as a section of $\Aut(K)$.

Assume now that $K$ is elementary abelian. By Clifford theory, $V|_{K}=\bigoplus\limits_{i\in I} W_i$ where $G$ acts transitively on the set $\{W_i\}$ and each $W_i$ is irreducible for $K$. Write $I(W_i)$ for the stabilizer in $G$ of $W_i$. Note that $C_K(W_i)\normaleq I(W_i)$ for all $i$. If $C_K(W_i)\normaleq G$, $C_K(W_i)\le C_K(V)=\{1\}$, $K$ acts faithfully on $W_i$ and by Schur's lemma, $|K|=r$, impossible since $|S|> p^2$ and $S$ acts faithfully on $K$. Hence, $I(W_i)\le N_G(C_K(W_i))\le M<G$ for $M$ some maximal subgroup. If $G=I(W_i)O_{p'}(G)$ then for $V_i$ an irreducible, non-trivial $O^{p'}(I(W_i))$ composition factor of $W_i$, we have by minimality that $|V_i|\geq q^2$ from which it follows that $V=W_i$, a contradiction. Thus, we can select $M$ such that $O_{p'}(G)<M$. Then we get that $q=9$ and $[G:M]=6$, or $[G:M]\geq q+1$. In the former case, we have that $|W_i|=3$, $I(W_i)=M$, $M/O_{3'}(G)\cong \Alt(5)$ and $|M/C_M(W_i)|=2$. In particular, $|O_{3'}(G)/C_{O_{3'}(G)}(W_i)|=2$ and as $O_{3'}(G)$ is faithful on $V$, $O_{3'}(G)$ is a $2$-group of order at most $2^6$. Then $|K|=2^4$ and as $G=O^{3'}(G)$, $G/K\cong \Alt(6)$ or $\SL_2(9)$. But now, $C_K(W_1)=C_K(W_j)$ for some $j\ne 1$ and by transitivity, we have that $|K|<2^4$, a contradiction. Hence, $|W_i|^{q+1}\leq |V|\leq q^3$ and since $K$ acts non-trivially on each $W_i$, we have that $|W_i|\geq 3$ and since $n>1$, we have a contradiction.
\end{proof}

We may relax the restrictions in the definition of an FF-module to allow for a greater class of module setups. An an example, the natural modules for $\SU_3(p^n)$ and $\Sz(2^n)$ are not FF-modules but satisfy the ratio $|V/C_V(A)|\leq |A/C_A(V)|^2$ for $A$ an elementary abelian $p$-group. Such modules are referred to as \emph{$2$F-modules}.

\begin{definition}
Let $G$ be a finite group and $V$ a $\mathrm{GF}(p)$-module. If there exists $A\le G$ such that
\begin{enumerate}
\item $A/C_A(V)$ is an elementary abelian $p$-group;
\item $[V,A]\ne\{1\}$; and 
\item $|V/C_V(A)|\leq |A/C_A(V)|^2$
\end{enumerate}
then $V$ is \emph{$2$F-module} for $G$.
\end{definition}

If $G$ is an almost quasisimple group with a $2$F module $V$, then both $G$ and $V$ are known by work of Guralnick, Lawther and Malle \cite{2F}, \cite{2F2}, \cite{2F3}. Importantly for applications in this work, even when $G$ is not almost quasisimple, we have good idea of the structure of groups which have a strongly $p$-embedded subgroup and a $2$F-module which admits a quadratically acting element.

First we introduce two groups that have associated $\GF(p)$-modules which exhibit $2$F-action and arise heavily in the local actions in later sections. In addition, we provide some ``characterizations'' of these groups, and some structural properties of the groups and the associated $2$F-module we are interested in.

\begin{lemma}\label{Badp2}
There is a unique group $G$ of shape $(3\times 3):2$ which has a faithful quadratic $2$F-module $V$, namely the generalized dihedral group of order $18$. Moreover, for $S\in\syl_2(G)$ and $V$ an associated faithful quadratic $2$F-module, the following hold:
\begin{enumerate}
\item $|V|=2^4$ and $G$ is unique up to conjugacy in $\GL_4(2)$;
\item $\{G, \Dih(18)\}=\{H\mid |H|=18, O_2(H)=\{1\}\,\,\text{and}\,\, H=O^{2'}(H)\}$;
\item there are exactly four overgroups of $S$ in $G$ which are isomorphic to $\Sym(3)$, any two of which generate $G$; and
\item $C_{\GL_4(2)}(G)=\{1\}$ and $|\Out_{\GL_4(2)}(G)|=4$.
\end{enumerate}
\end{lemma}
\begin{proof}
This follows directly from calculations in MAGMA, working explicitly with matrices in $\GL_4(2)$ and comparing with the Small Groups Library.
\end{proof}

Indeed, in the above lemma $G$ is also isomorphic to $\PSU_3(2)'$ and is listed in the Small Groups Library as $SmallGroup(18,4)$.

\begin{lemma}\label{Badp3}
There is a unique group $G$ of shape $(Q_8\times Q_8):3$ which has a faithful quadratic $2$F-module $V$. Moreover, for $S\in\syl_3(G)$ and $V$ an associated faithful quadratic $2$F-module, the following hold:
\begin{enumerate}
\item $|V|=3^4$ and $G$ is determined uniquely up to conjugacy in $\GL_4(3)$;
\item $G$ is the unique group of order $2^4.3$ or $2^6.3$ such that $O_3(G)=\{1\}$, $Z(G)\ne\{1\}$, $G=O^{3'}(G)$ and, if the order is $2^6.3$, there exists at least two distinct normal subgroups of $G$ of order $8$;
\item there are exactly five overgroups of $S$ in $G$ which are isomorphic to $\SL_2(3)$, any two of which generate $G$;
\item $N_{O_2(G)}(S)=Z(G)\cong 2\times 2$;
\item $\Aut(G)=\Aut_{\GL_4(3)}(G)$, $C_{\GL_4(3)}(G)=Z(G)$ and $|\Out(G)|=2^2.3$; and
\item if $U<V$ is $N_G(S)$-invariant and $|U|=3$, then $|\langle U^G\rangle|=9$.
\end{enumerate}
\end{lemma}
\begin{proof}
This follows directly from calculations in MAGMA, working explicitly with matrices in $\GL_4(3)$ and comparing with the Small Groups Library.
\end{proof}

The above group is listed in the Small Groups Library as $SmallGroup(192, 1022)$.

We now give an important characterization of certain ``small" groups which have an associated non-trivial quadratic $2$F-module. The proof of this result will be broken up over a series of lemmas.

\begin{lemma}
Assume that $G=O^{p'}(G)$ is a $\mathcal{K}$-group that has a strongly $p$-embedded subgroup, $S\in\syl_p(G)$, $V$ is a faithful $\mathrm{GF}(p)$-module with $C_V(O^p(G))=\{1\}$ and $V=\langle C_V(S)^G\rangle$. Furthermore, assume that $m_p(S)\geq 2$. If there is a $p$-element $1\ne x\in S$ such that $[V, x,x]=\{1\}$ and $|V/C_V(x)|=p^2$ then either:
\begin{enumerate}
    \item $p$ is odd, $G\cong\mathrm{(P)SU}_3(p)$ and $V$ is the natural module;
    \item $p$ is arbitrary, $G\cong\SL_2(p^2)$ and $V$ is the natural module; or
    \item $p=2$, $G\cong \PSL_2(4)$ and $V$ is a natural $\Omega_4^-(2)$-module.
\end{enumerate}
\end{lemma}
\begin{proof}
Applying the characterization in \cref{SE2} and using \cref{SEQuad} when $p$ is odd, we deduce that $G/O_{p'}(G)\cong\PSL_2(p^{n+1}), \PSU_3(p^n)$ or $\Sz(2^{2n+1})$ for $n\geq 1$. Assume that $G$ is a minimal counterexample to the lemma and let $P$ be a $p$-minimal subgroup $G$ containing $S$ with the property $G=PO_{p'}(G)$. If $G=P$ then by \cite[Theorem 1]{ChermakSmall} and \cref{SplitMod}, the lemma holds, a contradiction since $G$ was an assumed counterexample. Hence, we may assume that $P<G$. Set $V^P$ to be some non-trivial $P$-composition factor of $V$. By \cref{SEFF}, using that $m_p(S)\geq 2$, we have that $|V^P/C_{V_P}(x)|=p^2$ and $[V, x, x]$. By minimality, $P$ is described by the lemma. Since $C_V(P)\le C_V(x)$, we can arrange that $V/C_V(P)=V^P$. 

If $P\cong \SL_2(p^2)$ and $V/C_V(P)$ is a natural module, then for each $y\in S^\#$, we have that $C_V(y)=[V, y]C_V(P)=[V, S]C_V(P)$ from which it follows that $C_V(S)=[V, S]C_V(P)$ has index $p^2$ in $V$. Then \cref{SEFF} yields that $G=P$, a contradiction. If $G\cong\mathrm{(P)SU}_3(p)$ and $V/C_V(P)$ is the natural module, then $\langle x\rangle=Z(S)$. Moreover, for $A\le S$ with $|A|=p^2$ and $A$ elementary abelian, for $y,z\in A$ with $y,z\not\in Z(S)$ and $\langle y\rangle\ne \langle z\rangle$, we have that $C_V(y)=C_V(z)=C_V(A)$. Furthermore, $[V, y]C_V(y)=[V, z]C_V(z)=[V, A]C_V(A)=C_V(x)$ and by coprime action, $C_V(x)$ is normalized by $O_{p'}(G)$. Then $H:=\langle x^{O_{p'}(G)}\rangle$ centralizes $C_V(x)$ an by coprime action, $V=[V, O_{p'}(H)]\times C_V(O_{p'}(H))$ is a $\langle x\rangle$ invariant decomposition. Since $C_V(x)\le C_V(O_{p'}(H))$, we deduce that $O_{p'}(H)$ centralizes $V$ so that $O_{p'}(H)=\{1\}$ and $O_{p'}(G)$ normalizes $\langle x\rangle=Z(S)$. Hence, $[G, O_{p'}(G)]=[\langle Z(S)^G\rangle, O_{p'}(G)]=\{1\}$, $G$ is quasisimple and $G=P$, a contradiction. Finally, if $P\cong \PSL_2(4)$ and $V/C_V(P)$ is a natural $\Omega_4^-(2)$-module, then we select $x\in G$ and $s\in S^\#$ with the property $s^x\not\le P$. Then $H:=\langle P, s^x\rangle\not\cong \PSL_2(4)$ and by minimality, we have that $G=H$. Since $|V/C_V(s^x)|=4$, we deduce that $|V|=|V/C_V(G)|\leq 2^6$ and so $G$ embeds in $\SL_6(2)$. Comparing with \cite{LowMax} yields that $G=P$, a contradiction.
\end{proof}

\begin{lemma}
Assume that $G=O^{p'}(G)$ is a $\mathcal{K}$-group, $S\in\syl_p(G)$, $V$ is a faithful $\mathrm{GF}(p)$-module with $C_V(O^p(G))=\{1\}$ and $V=\langle C_V(S)^G\rangle$. Furthermore, assume that $m_p(S)=1$, $G$ is not $p$-solvable and if $p=2$ and $S$ is generalized quaternion then $N_G(S)$ is strongly $2$-embedded in $G$. If there is a $p$-element $1\ne x\in S$ such that $[V, x,x]=\{1\}$ and $|V/C_V(x)|=p^2$ then either:
\begin{enumerate}    
    \item $p=3$, $G\cong 2\cdot\Alt(5)$ or $2^{1+4}_-.\Alt(5)$ and $V$ is the unique irreducible quadratic $2$F-module of dimension $4$; or
    \item $p\geq 5$, $G\cong\SL_2(p)$ and $V$ is the direct sum of two natural $\SL_2(p)$-modules.
\end{enumerate}
\end{lemma}
\begin{proof}
Suppose first that $p=2$. Applying \cref{SE1}, we deduce that $S$ is generalized quaternion and $G=O_{2'}(G)C_G(\Omega(S))$. But now, $C_G(\Omega(S))=N_G(\Omega(Z(S)))=N_G(S)$ is solvable so that $G$ itself is solvable, a contradiction to the initial hypothesis. Hence, $p$ is odd. Applying \cref{SEQuad} and using that $G$ is not $p$-solvable, we deduce that for $L:=\langle x^G\rangle$, $L/C_L(U)\cong \SL_2(p)$ for $p\geq 5$, $2\cdot\Alt(5)$ or $2^{1+4}_-.\Alt(5)$ for $U$ some non-trivial irreducible composition factor of $V|_L$. Indeed, applying \cref{SE1}, $G=L$ and $C_G(U)$ is a $p'$-group. Now, by coprime action $V=C_V(C_G(U))\times [V, C_G(U)]$ and $U\le C_V(C_G(U))$. Applying \cref{SEFF}, if $2\cdot\Alt(5)$ or $2^{1+4}_-.\Alt(5)$ when $p=3$, we have that $|U/C_U(s)|=3^2$ so that $[V, C_G(U)]\le C_V(s)$ so that $[V, C_G(U)]\le C_V(G)=\{1\}$ and as $V$ is a faithful module, $C_G(U)=\{1\}$. Indeed, by \cref{SplitMod} and using that $C_V(G)=\{1\}$, $V=U$ is an irreducible module and outcome (i) holds.

Hence, we may assume that $G/C_G(U)\cong \SL_2(p)$ and $p\geq 5$. Then $C_V(C_G(U))$ is a quadratic module for $G/C_G(U)$ and \cref{DirectSum} and using that $C_V(G)=\{1\}$, $C_V(C_G(U))$ is a direct sum of at most two natural $\SL_2(p)$-modules. Suppose first that $C_V(C_G(U))$ is a natural $\SL_2(p)$-modules so that $U=C_V(C_G(U))$ and $|U/C_{U}(s)|=p$. Then $|[V, C_G(U)]/C_{[V, C_G(U)]}(s)|=p$ and applying \cref{SEFF}, we deduce that $G/C_G([V, C_G(U)])\cong \SL_2(p)$ and $[V, C_G(U)]$ is a natural $\SL_2(p)$-module. Since $[V, C_G(U)]$ is acted upon non-trivially by $C_G(U)$ and $C_G(U)$ is a $p'$-group, we conclude that \sloppy{$C_G([V, C_G(U)])C_G(U)/C_G(U)=Z(G/C_G(U))$, $C_G([V, C_G(U)])C_G(U)/C_G([V, C_G(U)])=Z(G/C_G([V, C_G(U)]))$ and $G/C_G([V, C_G(U)])\cap C_G(U)$ is a central extension of $\PSL_2(p)$ by a fours group.} Since the $2$-part of the Schur multiplier of $\PSL_2(p)$ has order $2$, $G$ is perfect and $G=O^{p'}(G)$, this is a contradiction. Suppose now that $C_V(C_G(U))$ is a direct sum of two natural $\SL_2(p)$-modules. Then $|C_V(C_G(U))/C_{C_V(C_G(U))}(s)|=p^2$ and we deduce that $[V, C_G(U)]\le C_V(s)$ so that $[V, C_G(U)]\le C_V(G)=\{1\}$ and as $V$ is a faithful module, $C_G(U)=\{1\}$ and outcome (ii) holds.
\end{proof}

\begin{lemma}
Assume that $G=O^{p'}(G)$, $S\in\syl_p(G)$, $V$ is a faithful $\mathrm{GF}(p)$-module with $C_V(O^p(G))=\{1\}$ and $V=\langle C_V(S)^G\rangle$. Furthermore, assume that $m_p(S)=1$, $N_G(S)$ is strongly $p$-embedded in $G$, and $G$ is $p$-solvable. If there is a $p$-element $1\ne x\in S$ such that $[V, x,x]=\{1\}$ and $|V/C_V(x)|=p^2$ then,  setting $L:=\langle x^G\rangle$, one of the following holds:
\begin{enumerate}
\item $p\leq 3$, $G=L\cong\SL_2(p)$ and $V$ is the direct sum of two natural $\SL_2(p)$-modules;
\item $p=2$, $L\cong\SU_3(2)'$, $G$ is isomorphic to a subgroup of $\SU_3(2)$ which contains $\SU_3(2)'$ and $V$ is a natural $\SU_3(2)$-module viewed as an irreducible $\mathrm{GF}(2)G$-module by restriction;
\item $p=2$, $L\cong \Dih(10)$, $G\cong\Dih(10)$ or $\Sz(2)$ and $V$ is a natural $\Sz(2)$-module viewed as an irreducible $\mathrm{GF}(2)G$-module by restriction;
\item $p=3$, $G=L\cong (Q_8\times Q_8):3$ and $V=V_1\times V_2$ where $V_i$ is a natural $\SL_2(3)$-module for $G/C_G(V_i)\cong\SL_2(3)$;
\item $p=2$, $G=L\cong(3\times 3):2$ and $V=V_1\times V_2$ where $V_i$ is a natural $\SL_2(2)$-module for $G/C_G(V_i)\cong\Sym(3)$; or
\item $p=2$, $L\cong(3\times 3):2$, $G\cong (3\times 3):4$, $V$ is irreducible as a $\mathrm{GF}(2)G$-module and $V|_L=V_1\times V_2$ where $V_i$ is a natural $\SL_2(2)$-module for $L/C_L(V_i)\cong\Sym(3)$. 
\end{enumerate}
\end{lemma}
\begin{proof}
Let $L:=\langle x^G\rangle$ so that $L=[\Omega(S), O_{p'}(G)]\Omega(S)$ by \cref{SE1}. Since $N_G(S)=N_G(\Omega(S))$, we deduce that $G=LS$ so that $O^p(L)=O^p(G)=[\Omega(S), O_{p'}(G)]$ and $C_V(O^p(L))=\{1\}$. Moreover, any element of $S$ centralizes $\Omega(Z(S))\in\syl_p(L)$ but does not centralize $L$, for otherwise, since $S$ contains a unique subgroup of order $p$, $[\Omega(Z(S)), L]=\{1\}$ and $\Omega(Z(S))\normaleq G$. Thus, $S/\Omega(S)$ embeds into $\Out(L)$. Finally, using \cref{SplitMod}, $V=[V, O^p(L)]$ and so both $L$ and $V$ are determined in \cite[Lemma 4.3]{ChermakSmall}. We examine each of the cases individually, using MAGMA for the explicit calculation in $\Out(L)$.

First, if $L\cong\SL_2(p)$ then it follows from \cref{SLGen} (viii) that $\Out_S(L)=\{1\}$, $L=G$ and $V$ is a direct sum of two natural modules. If $L\cong\Dih(10)$ then $\Aut(L)\cong\Sz(2)$ and it follows that $G=\Dih(10)$ or $\Sz(2)$, and $V$ is the restriction of a natural $\Sz(2)$-module to $G$.

Suppose that $L\cong\SU_3(2)'$. Then a Sylow $2$-subgroup of $\Aut(L)$ is isomorphic to a semidihedral group of order $16$ and since $m_p(S)=1$, $|S|\leq 8$ and $S$ is either cyclic or quaternion. Moreover, $54\leq |G|\leq 216$ and $|G|=54$ if and only if $G=L\cong\SU_3(2)'$. Suppose that $|G|=216$ and $S$ is cyclic. Utilizing the small group library in MAGMA, we identify a unique group $H$ such that $\langle \Omega(S)^H\rangle\cong\SU_3(2)'$. But in such a group, $N_H(T)<N_H(\Omega(T))$ for $T\in\syl_2(H)$, a contradiction to our hypothesis. Employing similar methods when $|G|=108$, or when $|G|=216$ and $S$ is quaternion, gives that $G$ is isomorphic to any index $2$ subgroup of $\SU_3(2)$ resp. $G\cong\SU_3(2)$. In all cases, $V$ is the restriction of a natural $\SU_3(2)$-module to $G$.

Suppose that $L\cong (Q_8\times Q_8):3$. Since $G$ acts faithfully on $V$, of order $3^4$, $G$ embeds into $\GL_4(3)$ and since the embedding of $L$ is uniquely determined up to conjugacy in $\GL_4(3)$, it follows that $G$ embeds into its normalizer in $\GL_4(3)$. For $H$ the image of $L$ in $\GL_4(3)$, we have that a Sylow $3$-subgroup of $N_{\GL_4(3)}(H)$ is elementary abelian of order $9$. Since $m_p(S)=1$, we have that $G=L$ in this case and $V$ is as described in \cite[Lemma 4.3]{ChermakSmall}.

Finally, suppose that $L\cong (3\times 3):2$. Since $G$ acts faithfully on $V$, of order $2^4$, $G$ embeds into $\GL_4(2)$ and since the embedding of $L$ is uniquely determined up to conjugacy in $\GL_4(2)$, it follows that $G$ embeds into the normalizer of its image. For $H$ the image of $L$ in $\GL_4(2)$, we have that a Sylow $2$-subgroup of $N_{\GL_4(2)}(H)$ is a dihedral group of order $8$ and there is a unique proper overgroup of $H$ in $N_{\GL_4(2)}(H)$ with a cyclic Sylow $2$-subgroup. Moreover, this group is irreducible in $\GL_4(2)$, is defined uniquely up to conjugacy in $\GL_4(2)$ and is isomorphic to any index $2$ subgroup of $\PSU_3(2)$. We denote this group $(3\times 3):4$ and it follows that either $G=L\cong(3\times 3):2$ or $G\cong (3\times 3):4$. Then $V$ is as given in \cite[Lemma 4.3]{ChermakSmall}.
\end{proof}

The following proposition is the summation of the previous three lemmas. This situation occurs frequently throughout the later sections of this work.

\begin{proposition}\label{Quad2F}
Assume that $G=O^{p'}(G)$ is a $\mathcal{K}$-group that has a strongly $p$-embedded subgroup, $S\in\syl_p(G)$, $V$ is a faithful $\mathrm{GF}(p)$-module with $C_V(O^p(G))=\{1\}$ and $V=\langle C_V(S)^G\rangle$. Furthermore, assume that if $m_p(G)=1$ and $G$ is $p$-solvable or $S$ is generalized quaternion, then $N_G(S)$ is strongly $p$-embedded in $G$. Suppose that there is a $p$-element $1\ne x\in S$ such that $[V, x,x]=\{1\}$ and $|V/C_V(x)|=p^2$. Setting $L:=\langle x^G\rangle$ one of the following holds:
\begin{enumerate}
\item $p$ is odd, $G=L\cong\mathrm{(P)SU}_3(p)$ and $V$ is the natural module;
\item $p$ is arbitrary, $G=L\cong\SL_2(p^2)$ and $V$ is the natural module;
\item $p=2$, $G=L\cong \PSL_2(4)$ and $V$ is a natural $\Omega_4^-(2)$-module;
\item $p=3$, $G=L\cong 2\cdot\Alt(5)$ or $2^{1+4}_-.\Alt(5)$ and $V$ is the unique irreducible quadratic $2$F-module of dimension $4$;
\item $p$ is arbitrary, $G=L\cong\SL_2(p)$ and $V$ is the direct sum of two natural $\SL_2(p)$-modules;
\item $p=2$, $L\cong\SU_3(2)'$, $G$ is isomorphic to a subgroup of $\SU_3(2)$ which contains $\SU_3(2)'$ and $V$ is a natural $\SU_3(2)$-module viewed as an irreducible $\mathrm{GF}(2)G$-module by restriction;
\item $p=2$, $L\cong \Dih(10)$, $G\cong\Dih(10)$ or $\Sz(2)$ and $V$ is a natural $\Sz(2)$-module viewed as an irreducible $\mathrm{GF}(2)G$-module by restriction;
\item $p=3$, $G=L\cong (Q_8\times Q_8):3$ and $V=V_1\times V_2$ where $V_i$ is a natural $\SL_2(3)$-module for $G/C_G(V_i)\cong\SL_2(3)$;
\item $p=2$, $G=L\cong(3\times 3):2$ and $V=V_1\times V_2$ where $V_i$ is a natural $\SL_2(2)$-module for $G/C_G(V_i)\cong\Sym(3)$; or
\item $p=2$, $L\cong(3\times 3):2$, $G\cong (3\times 3):4$, $V$ is irreducible as a $\mathrm{GF}(2)G$-module and $V|_L=V_1\times V_2$ where $V_i$ is a natural $\SL_2(2)$-module for $L/C_L(V_i)\cong\Sym(3)$. 
\end{enumerate}
\end{proposition}

While most of the groups and modules above have been described earlier in this section, we list some properties of the groups and modules occurring in (iv) and (x) above.

\begin{lemma}\label{AltMod}
Suppose that $G\cong 2\cdot\Alt(5)$ or $2^{1+4}_-.\Alt(5)$, $S\in\syl_3(G)$ and $V$ is an associated faithful quadratic $2$F-module. Then $C_V(S)=[V, S]$ has order $3^2$ and $V/[V,S]$ and $[V,S]$ are irreducible as $\mathrm{GF}(3)N_G(S)$-modules.
\end{lemma}
\begin{proof}
This follows directly from calculations in MAGMA, working explicitly with the matrices in $\Sp_4(3)$.
\end{proof}

\begin{lemma}\label{334}
Suppose that $G\cong (3\times 3):4$, $S\in\syl_2(G)$ and $V$ is an associated faithful quadratic $2$F-module. Then the following hold:
\begin{enumerate}
\item $[V, S]$ has order $2^{3}$;
\item $[V,\Omega(S)]=C_V(\Omega(S))=[V,S,S]$ has order $2^{2}$; and
\item $C_V(S)=[V,S,\Omega(S)]=[V, \Omega(S), S]=[V,S,S,S]$ has order $2$.
\end{enumerate}
\end{lemma}
\begin{proof}
This follows directly from calculations in MAGMA, working explicitly with the matrices in $\GL_4(2)$. 
\end{proof}

\begin{lemma}\label{pgen}
Suppose that $(G, V)$ satisfies the hypothesis of \cref{Quad2F}. In addition, assume that $V$ is generated as a $\mathrm{GF}(p)G$-module by an $N_G(S)$-invariant subspace of order $p$. Then $G\cong\PSL_2(4), \Dih(10), \Sz(2)$, $(3\times 3):2$ or $(3\times 3):4$ and $V$ is as described in \cref{Quad2F}.
\end{lemma}
\begin{proof}
We apply \cref{Quad2F} to get the list of candidates for $G$ and $V$. By \cref{SUMod} (iii), \cref{NatMod} (vi) and \cref{AltMod}, if $(G, V)$ satisfy (i), (ii), (iv) or (vi), then there are no $N_G(S)$-invariant subspaces of order $p$. By \cref{NatGen} and \cref{Badp3} (vi), if $(G,V)$ satisfy (v) or (viii) then $V$ is not generated by a subspace of order $p$. This leaves outcomes (iii), (vii), (ix) and (x), as required.
\end{proof}

Following on from \cref{SEQuad} and \cref{Quad2F}, we now recognize some groups with a strongly $2$-embedded subgroup and $2$-rank greater than $2$ which have a quadratic $2$F-module.

\begin{lemma}\label{2FRecog}
Suppose that $G$ is a finite group such that $G=O^{p'}(G)$ and $G/O_{p'}(G)\cong \PSL_2(p^n)$, $\Sz(2^n)$ or $\PSU_3(p^n)$ for $n>1$. Assume that $V$ is a faithful $\mathrm{GF}(p)$-module for $G$ such that there is $A\le S\in\syl_p(G)$ with $|A|>p^{\frac{n}{2}}$, $[V, A, A]=\{1\}$ and $|V/C_V(A)|\leq |A|^2$. Then $G\cong \SL_2(p^n)$, $\Sz(2^n)$ or $\SU_3(p^n)$ and $V/C_V(G)$ is a direct sum of at most $r$ natural modules for $G$, where $r=2$ if $G\cong \SL_2(p^n)$ and $r=1$ otherwise.
\end{lemma}
\begin{proof}
Assume that $G$ is a finite group satisfying the hypothesis of the lemma with $|G|$ minimal. Since $A$ acts quadratically on $V$ we have that $A$ is elementary abelian and by \cref{SEQuad} and \cref{GreenbookQuad} when $p$ is odd, we deduce that $A\le Z(S)$.

Let $P$ be a $p$-minimal subgroup of $G$ such that $G=PO_{p'}(G)$. If $P<G$, then by minimality, $P\cong \SL_2(p^n)$, $\Sz(2^n)$ or $\SU_3(2^n)$ and $V/C_V(P)$ is as described in the lemma. Since $[V, A]C_V(P)\le C_V(A)$ we deduce that $C_V(Z(S))=C_V(A)=C_V(z)$ for all $z\in Z(S)^\#$. By coprime action, $O_{p'}(G)$ normalizes $C_V(Z(S))$ and forming $H:=\langle Z(S)^{O_{p'}(G)Z(S)}\rangle$, we have that $H$ acts trivially on $C_V(Z(S))$. But then $V=[V O_{p'}(H)]\times C_V(O_{p'}(H))$  is a $Z(S)$-invariant decomposition and since $C_V(Z(S))\le C_V(O_{p'}(H))$, we conclude that $V=C_V(O_{p'}(H))$ and $Z(S)$ is normalized by $O_{p'}(G)$. But then $\{1\}=[Z(S), O_{p'}(G)]^G=[G, O_{p'}(G)]$, $O_{p'}(G)\le Z(G)$ and $G$ is quasisimple. A consideration of Schur multipliers contradicts that $G$ is a minimal counterexample.

Thus, we have that $G$ is $p$-minimal. We apply \cite[Theorem 1]{ChermakSmall} to $V/C_V(G)$. Since $|A|>p^{\frac{n}{2}}$ and $[V, A, A]=\{1\}$, $V/C_V(G)$ is not a natural $\Omega_4^-(2^{\frac{n}{2}})$ for $G$. Thus, $G\cong\SL_2(p^n)$, $\Sz(2^n)$ or $\SU_3(p^n)$ and $V/C_V(G)$ is a direct sum of natural modules, a contradiction since $G$ was an assumed counterexample.
\end{proof}

We now generalize even further than quadratic or cubic action by investigating the \emph{minimal polynomial} of $p$-elements in a representation, noticing that in quadratic and cubically acting elements, the minimal polynomial is of degree $2$ and $3$ respectively. We cannot hope to make such strong statements as in the earlier cases, but for larger primes and solvable groups, we have decent control due to the Hall--Higman theorem.

\begin{theorem}[Hall--Higman Theorem]
Suppose that $G$ is $p$-solvable group with $O_p(G)=\{1\}$ and $V$ a faithful $\mathrm{GF}(p)$-module for $G$. If $x\in G$ has order $p^n$ and $[V, x; r]=\{1\}$ then one of the following holds:
\begin{enumerate}
\item $r=p^n$;
\item$p$ is a Fermat prime, the Sylow $2$-subgroups of $G$ are non-abelian and $r\geq p^n-p^{n-1}$; or
\item $p=2$, the Sylow $q$-subgroups of $G$ are non-abelian for some Mersenne prime $q=2^m-1<2^n$ and $r\geq 2^n-2^{n-m}$.
\end{enumerate}
\end{theorem}
\begin{proof}
See \cite[Theorem B]{HH}.
\end{proof}

Whenever $p\geq 5$, applying the Hall--Higman theorem to the situation where the group $G$ has a strongly $p$-embedded subgroup and some associated cubic module, we can generally characterize $G$ completely. Thus, we pursue ways which in which to force cubic action. The final concept in this section is that of \emph{critical subgroups}, which first arose in the proof of the Feit--Thompson theorem. Originally in this work, critical subgroups provided a means to control the automizer of some $p$-group $Q$ whenever $p\geq 5$. In the context of the amalgam method, they force cubic action on some faithful section of $Q$ and from there, one can apply Hall--Higman type results to deduce information about $Q$ and its automizer. Where this methodology was previous employed, we now have methods to treat these cases uniformly across all primes and so critical subgroups now play a far lesser role in this work. However, we believe they still provide some interesting consequences in the amalgam method and we still include some of these consequences (see \cref{CriticalFusion}). We present the \emph{critical subgroup theorem}, due to Thompson, below.

\begin{theorem}\label{CriticalSubgroup}
Let $Q$ be a $p$-group. Then there exists $C\le Q$ such that the following hold:
\begin{enumerate}
\item $C$ is characteristic in $Q$;
\item $\Phi(C)\le Z(C)$ so that $C$ has class at most $2$;
\item $[C, Q]\le Z(C)$;
\item $C_Q(C)\le C$; and
\item $C$ is coprime automorphism faithful.
\end{enumerate}
\end{theorem}
\begin{proof}
This is \cite[(I.5.3.11)]{gor}.
\end{proof}

We call such a subgroup $C\le Q$ a \emph{critical subgroup} of $Q$.

\begin{corollary}\label{CubicAction}
Suppose that $G=O^{p'}(G)$ is a $\mathcal{K}$-group which has a strongly $p$-embedded subgroup, $S\in\syl_p(G)$ and $V$ is a faithful $\mathrm{GF}(p)$-module. Suppose that $p\geq 5$ and there is $s\in S$ of order $p^n$ such that $[V, s, s, s]=\{1\}$. Then $G\cong\mathrm{(P)SL}_2(p^n)$ or $\mathrm{(P)SU}_3(p^n)$ for any prime $p\geq 5$, or $p=5$, $G\cong 3\cdot\Alt(6)$ or $3\cdot\Alt(7)$ and for $W$ some irreducible composition factor of $V$, $|W|\geq 5^6$. 
\end{corollary}
\begin{proof}
Suppose first that $m_p(S)=1$. Then, by \cite[I.5.4.10 (ii)]{gor}, $S$ is cyclic and so we may as well assume that $[V,\Omega(S), \Omega(S), \Omega(S)]=\{1\}$. Suppose first that $G$ is $p$-solvable. Since $p^n-p^{n-1}=p^{n-1}(p-1)\geq 4$, the Hall--Higman theorem implies that $O_p(G)\ne\{1\}$, a contradiction since $G$ has a strongly $p$-embedded subgroup.

Suppose now that $m_p(S)=1$ and $G$ is not $p$-solvable. Since $G=O^{p'}(G)$, by \cref{SE1} we have that $G/O_{p'}(G)$ is a simple group with a cyclic Sylow $p$-subgroup. Form $X:=\Omega(S) O_{p'}(G)$. Then $X$ is a $p$-solvable group and $V$ is a faithful module for $X$ by restriction. Since $p\geq 5$, $p^n-p^{n-1}=p^{n-1}(p-1)\geq 4$ and by the Hall--Higman theorem $O_p(X)\ne\{1\}$. In particular, $\Omega(S)\normaleq X$ and $[O_{p'}(G), \Omega(S)]\le O_{p'}(G)\cap \Omega(S)=\{1\}$. But then, since $G/O_{p'}(G)$ is simple, $[O_{p'}(G), G]=[O_{p'}(G), \langle \Omega(S)^G\rangle]=\{1\}$ and $O_{p'}(G)\le Z(G)$. Hence, $G$ is a quasisimple group with a cyclic Sylow $p$-subgroup such that the degree of the minimal polynomial of some $p$-element is $3$. Such groups and their associated modules are determined in \cite{Zalesskii}.

Suppose that $m_p(S)\geq 2$ so that $G/O_{p'}(G)$ is determined by \cref{SE2}, and let $X=O_{p'}(G)\Omega(Z(S))$. Unless $G/O_{5'}(G)\cong \Sz(32):5$, we have that for any $1\ne s\in \Omega(S)$, $G=\langle s^G\rangle$. In this case, forming $X:=\langle s\rangle O_{p'}(G)$, we have that $X$ acts faithfully on $V$ with $s$ acting cubically, and by the Hall--Higman theorem, $\langle s\rangle\normaleq X$. But then $[s, O_{p'}(G)]\le \langle s\rangle\cap O_{p'}(G)=\{1\}$. Thus, $[G, O_{p'}(G)]=[\langle s^G\rangle, O_{p'}(G)]=[s, O_{p'}(G)]^G=\{1\}$ and $O_{p'}(G)\le Z(G)$. Since $G=O^{p'}(G)$ is perfect, $G$ is a perfect central extension of $G/O_{p'}(G)$. If $G/O_{p'}(G)$ is isomorphic to a rank $1$ simple group of Lie type in characteristic $p$, then the result follows from \cref{SLGen} (vii) and \cref{SU3Gen} (ix). If $G/O_{p'}(G)\cong \Alt(2p)$ then, as $p\geq 5$, $G$ has no faithful modules which witness cubic action by \cite{ZalAlt}. Hence, by \cref{SE2}, we are left with a finite number of perfect $p'$-central extensions of simple groups. We verify that none of these groups have a faithful module which witnesses cubic action using MAGMA, although there exists results in the literature which substantiate this claim.

So assume that $G/O_{5'}(G)\cong \Sz(32):5$. Then, for $s\in\Omega(S)$, we have that for $L:=\langle s^G\rangle$, $L/O_{5'}(L)\cong \Sz(32)$ and following the reasoning above, we have that $O_{5'}(L)\le Z(L)$. Since the Schur multiplier of $\Sz(32)$ is trivial and $\Sz(32)$ is perfect, we have that $L\cong \Sz(32)$. But $L$ acts faithfully on $V$, with $s\in S\cap L$ acting cubically, and since $\Sz(32)$ has no cubic modules, we have a contradiction. Hence, the result.
\end{proof}

As witnessed in the Hall-Higman theorem, cubic action is far less powerful tool in our application whenever $p\in\{2,3\}$. We remedy this with the following strengthening in \cite{nearquad}.

\begin{definition}
Let $G$ be a finite group and $V$ an $\mathrm{GF}(p)G$-module. If $A\le G$ acts cubically on $V$ and $[V, A]C_V(A)=[v\GF(p), A]C_V(A)$ for all $v\in V\setminus [V, A]C_V(A)$ then $A$ is said to act \emph{nearly quadratically} on $V$.
\end{definition}

\begin{theorem}\label{nearquad}
Suppose $G=O^{p'}(G)$ is a group which has a strongly $p$-embedded subgroup and $V$ be a faithful, irreducible $\mathrm{GF}(p)$-module for $G$. If $G$ is generated by elementary abelian $p$-subgroups which act nearly quadratically, but not quadratically, on $V$ then either $G$ is quasisimple, or one of the following holds:
\begin{enumerate}
    \item $p=3$ and $G\cong \Frob(39)$ and $|V|=3^3$; or
    \item $p=3$, $G\cong (C_2 \wr \Sym(n))'=O^{3'}(C_2 \wr \Sym(n))$ for $3\leq n\leq 5$ and $|V|=3^n$.
\end{enumerate}
\end{theorem}
\begin{proof}
See \cite[Theorem 2]{nearquad}.
\end{proof}

\begin{remark}
In the above, we have that $(C_2 \wr \Sym(3))'\cong \PSL_2(3)$ and $V$ is a natural $\Omega_3(3)$-module.
\end{remark}

%% file: Contents/3.FusionSystems.tex
\section{Fusion Systems}\label{FusSec}

In this section, we set up notations and terminology, and list some properties of fusion systems. The standard references for the study of fusion systems are \cite{ako} and \cite{craven} and most of what follows may be gleaned from these texts.

\begin{definition}
Let $G$ be a finite group with $S\in\syl_p(G)$. The \emph{fusion category} of $G$ over $S$, written $\fs_S(G)$, is the category with object set $\Ob(\fs_S(G)):= \{Q: Q\le S\}$ and for $P,Q\le S$, $\Mor_{\fs_S(G)}(P,Q):=\Hom_G(P,Q)$, where $\Hom_G(P,Q)$ denotes maps induced by conjugation by elements of $G$. 
\end{definition}

\begin{definition}
Let $S$ be a $p$-group. A fusion system $\fs$ over $S$ is a category with object set $\Ob(\fs):=\{Q: Q\le S\}$ and whose morphism set satisfies the following properties for $P, Q\le S$:
\begin{itemize}
\item $\Hom_S(P, Q)\subseteq \Mor_{\fs}(P,Q)\subseteq \Inj(P,Q)$; and
\item each $\phi\in\Mor_{\fs}(P,Q)$ is the composite of an $\fs$-isomorphism followed by an inclusion,
\end{itemize}
where $\Inj(P,Q)$ denotes injective homomorphisms between $P$ and $Q$. 

To motivate the group analogy, we write $\Hom_{\fs}(P,Q):=\Mor_{\fs}(P,Q)$ and $\Aut_{\fs}(P):=\Hom_{\fs}(P,P)$.

Two subgroups of $S$ are said to be \emph{$\fs$-conjugate} if they are isomorphic as objects in $\fs$. We write $Q^{\fs}$ for the set of all $\fs$-conjugates of $Q$. We say a fusion system is \emph{realizable} if there exists a finite group $G$ with $S\in\syl_p(G)$ and $\fs=\fs_S(G)$. Otherwise, the fusion system is said to be \emph{exotic}.
\end{definition}

\begin{definition}
Let $\fs$ be a fusion system on a $p$-group $S$. Then $\mathcal{H}$ is a \emph{subsystem} of $\fs$, written $\mathcal{H}\le \fs$, on a $p$-group $T$ if $T\le S$, $\mathcal{H}\subseteq \fs$ as sets and $\mathcal{H}$ is itself a fusion system. Then, for $\fs_1, \fs_2$ subsystems of $\fs$ supported on $S$, write $\langle \fs_1, \fs_2\rangle$ for the smallest subsystem of $\fs$ containing $\fs_1$ and $\fs_2$.
\end{definition}

\begin{definition}
Let $\fs$ be a fusion system over a $p$-group $S$ and let $Q\le S$. Say that $Q$ is
\begin{itemize}
\item \emph{fully $\fs$-normalized} if $|N_S(Q)|\ge |N_S(P)|$ for all $P\in Q^{\fs}$;
\item \emph{fully $\fs$-centralized} if $|C_S(Q)|\ge |C_S(P)|$ for all $P\in Q^{\fs}$;
\item \emph{fully $\fs$-automized} if $\Aut_S(Q)\in\syl_p(\Aut_{\fs}(Q))$;
\item \emph{receptive} in $\fs$ if for each $P\le S$ and each $\phi\in\Iso_{\fs}(P,Q)$, setting \[N_{\phi}=\{g\in N_S(P) : {}^{\phi}c_g\in\Aut_S(Q)\},\] there is $\bar{\phi}\in\Hom_{\fs}(N_{\phi}, S)$ such that $\bar{\phi}|_Q = \phi$;
\item \emph{$\fs$-centric} if $C_S(P)=Z(P)$ for all $P\in Q^{\fs}$;
\item \emph{$\fs$-radical} if $O_p(\Out_{\fs}(Q))=\{1\}$;
\item \emph{$\fs$-essential} if $Q$ is $\fs$-centric, fully $\fs$-normalized and $\Out_{\fs}(Q)$ contains a strongly $p$-embedded subgroup; and 
\item \emph{strongly closed} in $\fs$ if $x\alpha \le Q$ for all $\alpha\in\Hom_{\fs}(\langle x\rangle, S)$ whenever $x\in Q$.
\end{itemize}
\end{definition}

If it is clear which fusion system we are working in, we will refer to subgroups as being fully normalized (centralized, centric etc.) without the $\fs$ prefix. Note that essential subgroups of $S$ are also $\fs$-radical subgroups by definition.

\begin{definition}
Let $\fs$ be a fusion system over a $p$-group $S$. Then $\fs$ is \emph{saturated} if the following conditions hold:
\begin{enumerate}
\item Every fully $\fs$-normalized subgroup is also fully $\fs$-centralized and fully $\fs$-automized.
\item Every fully $\fs$-centralized subgroup is receptive in $\fs$.
\end{enumerate}
By a theorem of Puig \cite{Puig1}, the fusion category of a finite group $\fs_S(G)$ is a saturated fusion system.
\end{definition}

From this point on we focus more on saturated fusion systems, as they most closely parallel the group phenomenon. Throughout we will often adopt a \emph{local CK-system} hypothesis.

\begin{definition}
A \emph{local CK-system} is a saturated fusion $\fs$ such that $\Aut_{\fs}(P)$ is a $\mathcal{K}$-group for all $P\le S$.
\end{definition}

Local $\mathcal{CK}$-systems provides a means to apply the results from \cref{GrpSec} which relied on a $\mathcal{K}$-group hypothesis. This allows for minimal counterexample arguments in fusion systems and provides a link between fusion systems and the classification of finite simple groups. That is, if $G$ is a finite group which is a counterexample to the classification with $|G|$ minimal subject to these constraints, then $\fs_S(G)$ is a local $\mathcal{CK}$-system for $S\in\syl_p(G)$.

We now present arguably the most important tool in classifying saturated fusion systems. Because of this, we need only investigate the local action on a relatively small number of $p$-subgroups to obtain a global characterization of a saturated fusion system.

\begin{theorem}[Alperin -- Goldschmidt Fusion Theorem]
Let $\fs$ be a saturated fusion system over a $p$-group $S$. Then \[\fs=\langle \Aut_{\fs}(Q) \mid Q\,\, \text{is essential or}\,\, Q=S \rangle.\]
That is, every morphism in $\fs$ as a composite of restrictions of maps in $\Aut_{\fs}(Q)$ for $Q$ described above.
\end{theorem}
\begin{proof}
See \cite[Theorem I.3.5]{ako}.
\end{proof}

Along these lines, another important notion is for a $p$-subgroup to be \emph{normal} in a saturated fusion system.

\begin{definition}
Let $\fs$ be a fusion systems over a $p$-group $S$ and $Q\le S$. Say that $Q$ is \emph{normal} in $\fs$ if $Q\normaleq S$ and for all $P,R\le S$ and $\phi\in\Hom_{\fs}(P,R)$, $\phi$ extends to a morphism $\bar{\phi}\in\Hom_{\fs}(PQ,RQ)$ such that $\bar{\phi}(Q)=Q$.

It may be checked that the product of normal subgroups is itself normal. Thus, we may talk about the largest normal subgroup of $\fs$ which we denote $O_p(\fs)$ (and occasionally refer to as the $p$-core of $\fs$).  Further, it follows immediately from the saturation axioms that any subgroup normal in $S$ is fully normalized and fully centralized.
\end{definition}

\begin{definition}
Let $\fs$ be a fusion system over a $p$-group $S$ and let $Q\le S$. The \emph{normalizer fusion subsystem} of $Q$, denoted $N_{\fs}(Q)$, is the largest subsystem of $\fs$, supported over $N_S(Q)$, in which $Q$ is normal. 
\end{definition}

It is clear from the definition that if $\fs$ is the fusion category of a group $G$ i.e. $\fs=\fs_S(G)$, then $N_{\fs}(Q)=\fs_{N_S(Q)}(N_G(Q))$. The following result is originally attributed to Puig \cite{Puig2}.

\begin{theorem}
Let $\fs$ be a saturated fusion system over a $p$-group $S$. If $Q\le S$ is fully $\fs$-normalized then $N_{\fs}(Q)$ is saturated.
\end{theorem}
\begin{proof}
See \cite[Theorem I.5.5]{ako}.
\end{proof}

\begin{definition}
Let $\fs$ be a fusion systems over a $p$-group $S$ and $P\le Q\le S$. Say that $P$ is \emph{$\fs$-characteristic} in $Q$ if $\Aut_{\fs}(Q)\le N_{\Aut(Q)}(P)$.
\end{definition}

Plainly, if $Q\normaleq\fs$ and $P$ is $\fs$-characteristic in $Q$, then $P\normaleq\fs$.

We now present a link between normal subgroups of a saturated fusion system $\fs$ and its essential subgroups.

\begin{proposition}\label{normalinF}
Let $\fs$ be a saturated fusion system over a $p$-group $S$. Then $Q$ is normal in $\fs$ if and only if $Q$ is contained in each essential subgroup, $Q$ is $\Aut_{\fs}(E)$-invariant for any essential subgroup $E$ of $\fs$ and $Q$ is $\Aut_{\fs}(S)$-invariant.
\end{proposition}
\begin{proof}
See \cite[Proposition I.4.5]{ako}.
\end{proof}

As for finite groups, we desire a more global sense of normality in fusion systems, not just restricted to $p$-subgroups. That is, we are interested in subsystems of a fusion system $\fs$ which are \emph{normal}.

\begin{definition}
Let $\fs$ be a saturated fusion system over a $p$-group $S$. A fusion system $\mathcal{E}$ is \emph{weakly normal} in $\fs$ if the following conditions hold:
\begin{enumerate}
\item $\mathcal{E}$ is a saturated subsystem of $\fs$ over $T\le S$;
\item $T$ is strongly closed in $S$;
\item $^\alpha\mathcal{E}=\mathcal{E}$ for all $\alpha\in\Aut_{\fs}(T)$; and
\item for each $P\le T$ and each $\phi\in\Hom_{\fs}(P,T)$ there are $\alpha\in\Aut_{\fs}(T)$ and $\phi_0\in\Hom_{\mathcal{E}}(P,T)$ such that $\phi=\alpha\circ\phi_0$.
\end{enumerate}

A fusion system $\mathcal{E}$ is \emph{normal} in $\fs$, denoted $\mathcal{E}\normaleq \fs$, if $\mathcal{E}$ is weakly normal in $\fs$ and each $\alpha\in\Aut_{\mathcal{E}}(T)$ extends to some $\bar{\alpha}\in\Aut_{\fs}(TC_S(T))$ which fixes every coset of $Z(T)$ in $C_S(T)$.
\end{definition}

Conditions (iii) and (iv) are referred to as the invariance condition and Frattini condition respectively. As one would hope, for a $p$-subgroup $Q$, if $Q\normaleq \fs$, then $\fs_Q(Q)\normaleq \fs$. As is the case with groups, we refer to a saturated fusion system as \emph{simple} if it contains no proper non-trivial normal subsystems.

We shall describe some important subsystems associated to a saturated fusion which have a natural analogues in finite group theory. More details on the construction of such subsystems may be found in Section I.7 of \cite{ako}.

\begin{definition}
Let $\fs$ be a saturated fusion system on a $p$-group $S$. Say a subsystem $\mathcal{E}$ has \emph{index prime to $p$} in $\fs$ if $\mathcal{E}$ is a fusion system on $S$ and $\Aut_{\mathcal{E}}(P)\ge O^{p'}(\Aut_{\fs}(P))$ for all $P\le S$. 

By \cite[Theorem I.7.7]{ako}, there is a unique minimal saturated fusion system of index prime to $p$ in $\fs$ denoted by $O^{p'}(\fs)$ and $O^{p'}(\fs)$ is a normal subsystem of $\fs$.
\end{definition}

\begin{definition}
Let $\fs$ be a saturated fusion system on a $p$-group $S$. Then the \emph{hyperfocal subgroup} $\mathfrak{hyp}(\fs)$ of $\fs$ is defined as \[\mathfrak{hyp}(\fs):=\langle g^{-1}\alpha(g)\mid g\in P\le S, \alpha\in O^p(\Aut_{\fs}(P))\rangle.\] A subsystem $\mathcal{E}$ has \emph{$p$-power index} in $\fs$ if $\mathcal{E}$ is a fusion system on $T\ge \mathfrak{hyp}(\fs)$ and $\Aut_{\mathcal{E}}(P)\ge O^{p}(\Aut_{\fs}(P))$ for all $P\le S$. 

Moreover, by \cite[Theorem I.7.4]{ako}, there is a unique minimal fusion subsystem of $p$-power index in $\fs$ denoted by $O^p(\fs)$, over $\mathfrak{hyp}(\fs)$, and $O^{p}(\fs)$ is a normal subsystem of $\fs$.
\end{definition}

\begin{definition}
A saturated fusion system is \emph{reduced} if $O_p(\fs)=\{1\}$ and $\fs=O^p(\fs)=O^{p'}(\fs)$.
\end{definition}

Naturally, an important consideration in fusion systems is the notion of isomorphism. After defining what isomorphism means in the context fusion systems,  it follows readily that the ``sensible'' properties hold, which we state below.

\begin{definition}
Let $\fs$ be a fusion system on a $p$-group $S$ and $\mathcal{E}$ a fusion system on a $p$-group $T$. A \emph{morphism} $\phi:\fs\to \mathcal{E}$ is a tuple $(\phi_S, \phi_{P,Q} \mid  P,Q\le S)$ such that $\phi_S:S\to T$ is a group homomorphism and $\phi_{P,Q}:\Hom_{\fs}(P,Q)\to \Hom_{\mathcal{E}}(P\phi, Q\phi)$ is such that $\alpha\phi_S=\phi_S(\alpha\phi_{P,Q})$ for all $\alpha\in\Hom_{\fs}(P,Q)$. 

Say that $\phi$ is \emph{injective} if $\phi_S\colon S\to T$ is injective, and $\phi$ is \emph{surjective} if $\phi_S$ is surjective and, for all $P,Q\le S$, $\phi_{P_0, Q_0}\colon\Hom_{\fs}(P_0,Q_0)\to \Hom_{\mathcal{E}}(P\phi, Q\phi)$ is surjective, where $P_0,Q_0$ denote the preimages in $S$ of $P\phi,Q\phi$. Then, $\phi$ is an \emph{isomorphism} of fusion systems if $\phi:\fs\to\mathcal{E}$ is an injective, surjective morphism.
\end{definition}

\begin{lemma}\label{fusioniso}
Let $G\cong H$ be finite groups with $S\in\syl_p(G)$ and $T\in\syl_p(H)$. Then $\fs_S(G)\cong \fs_T(H)$.
\end{lemma}

\begin{lemma}\label{p'quo}
Let $\fs=\fs_S(G)$ be a saturated fusion system and set $\bar{G}=G/O_{p'}(G)$. Then $\fs_S(G)\cong \fs_{\bar{S}}(\bar{G})$.
\end{lemma}

In order to investigate the local actions in a saturated fusion systems, and in particular in its normalizer subsystems, it will often be convenient to work in a purely group theoretic context. The \emph{model theorem} guarantees that we may do this for a certain class of $p$-subgroups of a saturated fusion system $\fs$.

\begin{definition}
Let $\fs$ be a saturated fusion system on a $p$-group $S$. Then a \emph{model} for $\fs$ is a finite group $G$ with $S\in\syl_p(G)$, $F^*(G)=O_p(G)$ and $\fs=\fs_S(G)$.
\end{definition}

\begin{theorem}[Model Theorem]\label{model}
Let $\fs$ be a saturated fusion system over a $p$-group $S$. Fix $Q\le S$ which is $\fs$-centric and normal in $\fs$. Then the following hold:
\begin{enumerate} 
\item There are models for $\fs$.
\item If $G_1$ and $G_2$ are two models for $\fs$, then there is an isomorphism $\phi: G_1\to G_2$ such that $\phi|_S = \mathrm{Id}_S$.
\item For any finite group $G$ containing $S$ as a Sylow $p$-subgroup such that $Q\le G$, $C_G(Q) \le Q$ and $Aut_G(Q) = Aut_{\fs}(Q)$, there is $\beta\in\Aut(S)$ such that $\beta|_Q = \mathrm{Id}_Q$ and $\fs_S(G) = {}^{\beta}\fs$. Thus, there is a model for $\fs$ which is isomorphic to $G$.
\end{enumerate}
\end{theorem}
\begin{proof}
See \cite[Theorem I.4.9]{ako}.
\end{proof}

In order to make use of the Alperin--Goldschmidt fusion theorem, and given some restrictions on the number of essential subgroups, we must also determine the induced automorphism groups on the essential subgroups by $\fs$. The first result along these lines determines the potential automizer $\Aut_{\fs}(E)$ of an essential subgroup $E$ whenever some non-central chief factor of $E$ is an FF-module. It is important to note that this theorem does not rely on a $\mathcal{K}$-group hypothesis, and it is essentially the fusion theoretic equivalent of \cref{SEFF}.

\begin{theorem}\label{SEFFFus}
Suppose that $E$ is an essential subgroup of a saturated fusion system $\fs$ over a $p$-group $S$, and assume that there is an $\Aut_{\fs}(E)$-invariant subgroup $V\le \Omega(Z(E))$ such that $V$ is an FF-module for $G:=\Out_{\fs}(E)$. Then, writing $L:=O^{p'}(G)$, we have that $L/C_L(V)\cong \SL_2(p^n)$, $C_L(V)$ is a $p'$-group and $V/C_V(O^p(L))$ is a natural $\SL_2(q)$-module.
\end{theorem}
\begin{proof}
This is \cite[Theorem 1.2]{henkesl2}.
\end{proof}

Armed with the analysis of groups with strongly $p$-embedded subgroups from \cref{GrpSec}, we have reasonable limitations on the structure of $\Out_{\fs}(E)$ for $E$ an essential subgroup of $\fs$. We now substantiate the claim in the introduction that the restriction we impose whenever an essential automizer has quotient $\Ree(3)$, is $p$-solvable or has a generalized quaternion Sylow $2$-subgroup is implied by the maximality of the essential subgroup. We also record some generic results for \emph{maximally essential} subgroups for application elsewhere.

\begin{definition}
Suppose that $\fs$ is a saturated fusion system on a $p$-group $S$. Then $E\le S$ is \emph{maximally essential} in $\fs$ if $E$ is essential and, if $F\le S$ is essential in $\fs$ and $E\le F$, then $E=F$.
\end{definition}

Coupled with saturation arguments and the Alperin--Goldschmidt theorem, this definition drastically limits the possibilities for $\Out_{\fs}(E)$.

\begin{lemma}\label{MaxEssen}
Let $\fs$ be a saturated fusion systems on a $p$-group $S$ with $E$ a maximally essential subgroup of $\fs$. Then $N_{\Out_{\fs}(E)}(\Out_S(E))$ is strongly $p$-embedded in $\Out_{\fs}(E)$.
\end{lemma}
\begin{proof}
Let $T\le N_S(E)$ with $E<T$. Now, since $E$ is receptive, for all $\alpha\in N_{\Aut_{\fs}(E)}(\Aut_T(E))$, $\alpha$ lifts to a morphism $\hat{\alpha}\in\Hom_{\fs}(N_\alpha, S)$ with $N_\alpha>E$. Since $E$ is maximally essential, applying the Alperin--Goldschmidt theorem, $\hat{\alpha}$ is the restriction of a morphism $\bar{\alpha}\in\Aut_{\fs}(S)$. But then, $\alpha$ normalizes $\Aut_S(E)$ and so $N_{\Aut_{\fs}(E)}(\Aut_T(E))\le N_{\Aut_{\fs}(E)}(\Aut_S(E))$. This induces the inclusion $N_{\Out_{\fs}(E)}(\Out_T(E))\le N_{\Out_{\fs}(E)}(\Out_S(E)$. Since this holds for all $T\le N_S(E)$ with $E<T$, we infer that $N_{\Out_{\fs}(E)}(\Out_S(E))$ is strongly $p$-embedded in $\Out_{\fs}(E)$, as required.
\end{proof}

Hence, \cref{MaxEssen} implies that maximally essential subgroups of $\fs$ which are $\Aut_{\fs}(S)$-invariant satisfy the extra conditions in the hypothesis of the \hyperlink{MainThm}{Main Theorem}. As in the earlier analysis of groups with strongly $p$-subgroups, we divide into two cases, where $m_p(\Out_S(E))=1$ or $m_p(\Out_S(E))\geq 2$.

\begin{proposition}\label{p'-index1}
Let $\fs$ be a saturated fusion systems on a $p$-group $S$ with $E$ a maximally essential subgroup of $\fs$, and set $G=\Out_{\fs}(E)$. If $m_p(G)=1$ then either 
\begin{enumerate}
\item $\Out_S(E)$ is cyclic or generalized quaternion and 
\begin{align*}
    O^{p'}(G)&=\Out_S(E)[O_{p'}(O^{p'}(G)), \Omega(\Out_S(E))]\\
    &=\Out_S(E)\langle \Omega(\Out_S(E))^{O^{p'}(G)}\rangle
\end{align*} 
is $p$-solvable; or
\item $O^{p'}(G)/O_{p'}(O^{p'}(G))$ is a non-abelian simple group, $p$ is odd and $\Out_S(E)$ is cyclic.
\end{enumerate}
\end{proposition}
\begin{proof}
Since $G$ has a strongly $p$-embedded subgroup, so does $O^{p'}(G)$ and we apply \cref{SE1} and (ii) follows immediately. \sloppy{In the other cases of \cref{SE1}, since $\Omega(\Out_S(E))[O_{p'}(O^{p'}(G)), \Omega(\Out_S(E))]\normaleq O^{p'}(G)$, by the Frattini argument, }
\begin{align*}
    O^{p'}(G)&=N_{O^{p'}(G)}(\Omega(\Out_S(E)))[O_{p'}(O^{p'}(G)), \Omega(\Out_S(E))]\\
    &=N_{O^{p'}(G)}(\Omega(\Out_S(E)))\langle \Omega(\Out_S(E))^{O^{p'}(G)}\rangle.
\end{align*} 
Since $E$ is maximally essential, applying \cref{MaxEssen}, $N_{O^{p'}(G)}(\Omega(\Out_S(E)))\le N_{G}(\Omega(\Out_S(E)))=N_G(\Out_S(E))$. \sloppy{But then $\Out_S(E)[O_{p'}(O^{p'}(G)), \Omega(\Out_S(E))]\normaleq O^{p'}(G)$ and by the definition of $O^{p'}(G)$, we have that $O^{p'}(G)=\Out_S(E)[O_{p'}(O^{p'}(G)), \Omega(\Out_S(E))]$.}
\end{proof}

\begin{proposition}\label{p'-index2}
Let $\fs$ be a local $\mathcal{CK}$-system on a $p$-group $S$ with $E$ a maximally essential subgroup of $\fs$ and set $G=\Out_{\fs}(E)$. If $m_p(G)\geq 2$ then $O^{p'}(G)$ is isomorphic to a central extension by a group of $p'$-order of one of the following groups:
\begin{enumerate}
\item $\PSL_2(p^{a+1})$ or $\PSU_3(p^b)$ for $p$ arbitrary, $a\geq 1$ and $p^b>2$;
\item $\Sz(2^{2a+1})$ for $p=2$ and $a\geq 1$;
\item $\mathrm{Ree}(3^{2a+1}), \PSL_3(4)$ or $\mathrm{M}_{11}$ for $p=3$ and $a\geq 0$;
\item $\Sz(32):5, {}^2\mathrm{F}_4(2)'$ or $\mathrm{McL}$ for $p=5$; or
\item $\mathrm{J}_4$ for $p=11$.
\end{enumerate}
Furthermore, either $O^{p'}(G)$ is a perfect central extension, or $O^{p'}(G)\cong \Ree(3)$ resp. $\Sz(32):5$ and $p=3$ resp. $p=5$.
\end{proposition}
\begin{proof}
Set $\wt{G}=G/O_{p'}(G)$ and $K=O^{p'}(G)$. By \cref{MaxEssen}, $N_G(\Out_S(E))$ is strongly $p$-embedded in $G$. In particular, we deduce that $N_K(\Out_S(E))$ is strongly $p$-embedded in $K$. Let $A\le \Out_S(E)$ be elementary abelian of order $p^2$. By coprime action, $O_{p'}(K)=\langle C_{O_{p'}(K)}(a) \mid a \in A^\#\rangle$. Since $N_K(\Out_S(E))$ is strongly $p$-embedded in $K$, we have that $O_{p'}(K)\le N_K(\Out_S(E))$ so that $[O_{p'}(K), \Out_S(E)]=\{1\}$. Then \[[O_{p'}(K), K]=[O_{p'}(K), \langle \Out_S(E)^K\rangle]=[O_{p'}(K), \Out_S(E)]^K=\{1\}\] and $O_{p'}(K)\le Z(K)$. 

Now, $\wt K\cong K/O_{p'}(K)$ is determined as in \cref{SE2}. Moreover, $\wt{N_K(\Out_S(E))}=N_{\wt K}(\wt {\Out_S(E))})$ is strongly $p$-embedded in $\wt K$ and applying \cite[Theorem 7.6.2]{GLS3}, $\wt K\not\cong \Alt(2p)$ or $\mathrm{Fi}_{22}$. Unless $\wt K\cong \Ree(3)$ or $\Sz(32):5$, using that $\wt K$ is simple and $K=O^{p'}(K)$, $K$ is perfect central extension of $\wt K$ by a group of $p'$-order. If $\wt K\cong \Ree(3)$ or $\Sz(32):5$ then $O^{p'}(O^p(K))$ is a perfect central extension of $\Ree(3)'\cong \PSL_2(8)$ resp. $\Sz(32)$ by a $p'$-group so that $O^{p'}(O^p(K))\cong \PSL_2(8)$ resp. $\Sz(32)$. Since $O_{p'}(K)\le N_K(\Out_S(E))$ and $K=O^{p'}(O^p(K))O_{p'}(K)\Out_S(E)$, we conclude that $O_{p'}(K)=\{1\}$ and $K=O^{p'}(K)=O^{p'}(O^p(K))\Out_S(E)\cong \Ree(3)$ resp. $\Sz(32):5$. 
\end{proof}

The following elementary example gives a flavor of what can happen without the restriction on maximality of essentials.

\begin{example}\label{MaxExample}
Let $V$ be a $4$-dimensional vector space over $\GF(2)$ and let $\Dih(10)$ act irreducibly on it. In its embedding in $\GL_4(2)$, $\Dih(10)$ is centralized by a $3$-element and so we may form a subgroup of $\GL_4(2)$ of shape $\Dih(10)\times 3$. This group is normalized by an element $t$ of order $4$ such that $\langle \Dih(10), t\rangle\cong\Sz(2)$, $t^2\in\Dih(10)$ and $t$ inverts the $3$-element which centralizes $\Dih(10)$. Thus, we may construct a group $H$ of shape $\Dih(10).\Sym(3)$ in $\GL_4(2)$. Form the semidirect product $G:=V\rtimes H$ and consider the $2$-fusion category of $G$ over some Sylow $2$-subgroup $S$. Since $H$ has cyclic Sylow $2$-subgroups and $O_2(H)=\{1\}$, we have that $V$ is essential in the $2$-fusion category of $G$. Moreover, for $s$ the unique involution in $H\cap S$, we have that $E:=V\langle s\rangle$ has order $2^5$ and $N_G(E)/E\cong \Sym(3)$. Therefore, $E$ is also an essential subgroup which properly contains another essential subgroup $V$. 
\end{example}

It is easy to construct other examples in which smaller essentials are contained in some larger essential, even when imposing the condition that the essential subgroups are $\Aut_{\fs}(S)$-invariant. But it is reasonable to ask whether such examples actually occur in an amalgam setting motivated by the hypothesis of the \hyperlink{MainThm}{Main Theorem}. To this end, let $E$ be an $\Aut_{\fs}(S)$-invariant essential subgroup of a saturated fusion system $\fs$ on a $p$-group $S$, let $G$ be a model for $N_{\fs}(E)$ and suppose that $\Omega(Z(S))\not\normaleq G$. In the midst of the amalgam method, to determine $\Out_{\fs}(E)$ and its actions, we work ``from the bottom up" by determining $\Out_{\fs}(E)$-chief factors of $E$, starting with those in $\langle \Omega(Z(S))^G \rangle$ and taking progressively larger subgroups of $E$, so working ``up." Taking \cref{MaxExample} as inspiration, one might imagine a situation in which $\Out_{\fs}(E)$ induces a $\Sym(3)$-action on almost all $\Out_{\fs}(E)$-chief factors in $E$. Without examining an ever increasing sequence of subgroups and chief factors, it may be hard to eventually uncover a chief factor which witnesses non-trivial action by a $5$-element (although this would probably only happen for amalgams with large ``critical distance", see \cref{BasicAmalNot}, and even then it seems unlikely). It seems some additional tricks and techniques (or perhaps an even more granular case division) are required to treat these types of examples uniformly.

%% file: Contents/4.Amalgams.tex
\section{Amalgams in Fusion Systems}\label{AmalgamsSetup}

In this section, we introduce amalgams and demonstrate their connections with and applications to saturated fusion systems. We will only make use of elementary definitions and facts regarding amalgams as can be found in \cite[Chapter 2]{Greenbook}.

\begin{definition}
An \emph{amalgam} of rank $n$ is a tuple $\mathcal{A}=\mathcal{A}(G_1,\dots, G_n, B, \phi_1,\cdots, \phi_n)$ where $B$ is a group, each $G_i$ is a group and $\phi_i:B\to G_i$ is an injective group homomorphism. A group $G$ is a \emph{faithful completion} of $\mathcal{A}$ if there exists injective group homomorphisms $\psi_i:G_i\to G$ such that for all $i,j\in\{1,\dots, n\}$, $\phi_i\psi_i=\phi_j\psi_j$, $G=\langle \mathrm{Im}(\psi_i)\rangle$ and no non-trivial subgroup of $B\phi_i\psi_i$ is normal in $G$. Under these circumstances, we identify $G_1,\dots, G_n, B$ with their images in $G$ and opt for the notation $\mathcal{A}=\mathcal{A}(G_1, \dots, G_n, B)$.
\end{definition}

For almost all of the work towards the \hyperlink{MainThm}{Main Theorem}, we reduce to the case where the amalgam is of rank $2$ and the groups $G_1$ and $G_2$ are finite groups. In this setting, we may always realize $\mathcal{A}$ in a faithful completion, namely the free amalgamated product of $G_1$ and $G_2$ over $B$, denoted $G_1\ast_B G_2$. This completion is universal in that every faithful completion occurs as some quotient of this free amalgamated product. Generally, whenever we work in the setting of rank $2$ amalgams we will opt to work in this free amalgamated product which we will often denote $G$ and, in an abuse of terminology, refer to $G$ as an amalgam. In particular, we may as well assume the following:

\begin{enumerate}
\item $G=\langle G_1, G_2\rangle$, $G_i$ is a finite group and $G_i<G$ for $i\in\{1,2\}$;
\item no non-trivial subgroup of $B$ is normal in $G$; and
\item $B=G_1\cap G_2$.
\end{enumerate}

\begin{definition}
Let $\mathcal{A}=\mathcal{A}(G_1, G_2, B, \phi_1, \phi_2)$ and $\mathcal{B}=\mathcal{B}(H_1, H_2, C, \psi_1, \psi_2)$ be two rank $2$ amalgams. Then $\mathcal{A}$ and $\mathcal{B}$ are \emph{isomorphic} if, up to permuting indices, there are isomorphisms $\theta_i: G_i\to H_i$ and $\xi:B\to C$ such that the following diagram commutes for $i\in\{1,2\}$:

\begin{center}
\begin{tikzcd}
G_1 \arrow[dd, "\theta_1"'] &  & B \arrow[ll, "\phi_1"'] \arrow[rr, "\phi_2"] \arrow[dd, "\xi"'] &  & G_2 \arrow[dd, "\theta_2"] &                                                \\
                            &  & {} \arrow[phantom, loop, distance=2em, in=305, out=235]         &  &                            &                                                \\
H_1                         &  & C \arrow[ll, "\psi_1"] \arrow[rr, "\psi_2"']                    &  & H_2                        & 
\end{tikzcd}
\end{center}

Often, for some finite group $H$ arising as a faithful completion of some rank $2$ amalgam $\mathcal{B}$, we will often say a completion $G$ of $\mathcal{A}$ is \emph{locally isomorphic} to $H$, by which we mean $\mathcal{A}$ is isomorphic to $\mathcal{B}$.

An important observation in this definition is that the faithful completions of two isomorphic amalgams coincide. In fact, two amalgams being isomorphic is equivalent to demanding that $G_1\ast_B G_2\cong H_1\ast_C H_2$.

Say that $\mathcal{A}=\mathcal{A}(G_1, G_2, B)$ and $\mathcal{B}=\mathcal{B}(H_1, H_2, C)$ are \emph{parabolic isomorphic} if, up to permuting indices, $G_i\cong H_i$ and $B\cong C$ as abstract groups.
\end{definition}

We provide the following elementary example with regard to isomorphisms of amalgams.

\begin{example}\label{J2Example}
For $G=\mathrm{J}_2$, there are two maximal subgroups $M_1, M_2$ containing $N_G(S)$ for $S\in\syl_2(G)$. Furthermore, $M_1/O_2(M_1)\cong \SL_2(4)$, $M_2/O_2(M_2)\cong \Sym(3)\times 3$ and $|N_G(S)/S|=3$. Thus, $G$ gives rise to the amalgam $\mathcal{A}:=\mathcal{A}(M_1, M_2, N_G(S))$.

For $H=\mathrm{J}_3$ and $T\in\syl_2(H)$, $S\cong T$ and $H$ contains two maximal subgroups $N_1, N_2$ containing $N_G(T)$ such that $N_i\cong M_i$ for $i\in\{1,2\}$. Thus, $H$ gives rise to the amalgam $\mathcal{B}:=\mathcal{B}(N_1, N_2, N_H(T))$.

Then $\mathcal{A}$ is isomorphic to $\mathcal{B}$.
\end{example}

\begin{definition}
Let $\mathcal{A}=\mathcal{A}(G_1, G_2, B)$ be an amalgam of rank $2$. Then $\mathcal{A}$ is a \emph{characteristic $p$ amalgam of rank $2$} if the following hold for $i\in\{1,2\}$:
\begin{enumerate}
\item $G_i$ is a finite group;
\item $\syl_p(B)\subseteq \syl_p(G_1)\cap \syl_p(G_2)$; and
\item $G_i$ is of characteristic $p$.
\end{enumerate}
\end{definition}

For any faithful completion $G$ of an amalgam $\mathcal{A}$, $G$ is not necessarily a finite group and so we must define generally what a Sylow $p$-subgroup is. We say that $P$ is a Sylow $p$-subgroup of a group $G$ if every finite $p$-subgroup of $G$ is conjugate in $G$ to some subgroup of $P$.

The following theorem provides the connection between amalgams and fusion systems. Indeed, the original application of this theorem demonstrates that any saturated fusion system may be realized by a (possibly infinite) group. 

\begin{theorem}\label{robinson}
Let $p$ be a prime, $G_1$, $G_2$ and $G_{12}$ be groups with $G_{12}\le G_1\cap G_2$. Assume that $S_1\in\syl_p(G_1)$ and $S_2\in\syl_p(G_{12})\cap \syl_p(G_2)$ with $S_2\le S_1$. Set 
\[ G=G_1\ast_{G_{12}}G_2\] to be the free amalgamated product of $G_1$ and $G_2$ over $G_{12}$.
Then $S_1\in\syl_p(G)$ and
\[\fs_{S_1}(G)=\langle \fs_{S_1}(G_1), \fs_{S_2}(G_2)\rangle.\]
\end{theorem}
\begin{proof}
This is \cite[Theorem 1]{Robinson}.
\end{proof}

In other words, the above theorem implies that given two fusion systems which give rise to two rank $2$ amalgams, and the data from these amalgams ``generate" the fusion system, then provided that the amalgams are isomorphic, the fusion systems are isomorphic.

However, there are some key differences in the group theoretic applications of amalgams, and the fusion theoretic applications. Consider the configurations from \cref{J2Example}. The two amalgams there, $\mathcal{A}$ and $\mathcal{B}$, are isomorphic. In this way, we can actually embed a copy of the $2$-fusion system of $\mathrm{J}_2$ inside the $2$-fusion system of $\mathrm{J}_3$, but $\mathrm{J}_2$ is certainly not a subgroup of $\mathrm{J}_3$. Indeed, the $2$-fusion system of $\mathrm{J}_3$ contains an additional class of essential subgroups arising from different maximal subgroups of $\mathrm{J}_3$ of shape $2^4: (3\times \SL_2(4))$ not involved in the amalgams.

Thus, there are some important considerations demonstrated in \cref{J2Example} that one should be aware of. One is that for a group $G$ with two maximal subgroups $M_1$ and $M_2$ containing a Sylow $p$-subgroup of $G$, even though $G=\langle M_1, M_2\rangle$ there are situations in which $\fs_S(G)\ne \langle \fs_S(M_1), \fs_S(M_2)\rangle$. The second is that one must be very careful in choosing the ``correct" completion when working with amalgams in the context of fusion systems. Indeed, most of the time, this often requires knowledge of the fusion systems, and in particular the essential subgroups, of the completions of the amalgam.

We now collect some results using the \emph{amalgam method} which are relevant to this work. With the application to fusion systems in mind, we are particular interested in the case where the local action involves strongly $p$-embedded subgroups.

\begin{definition}\label{WBDef}
Let $\mathcal{A}:=\mathcal{A}(G_1, G_2, G_{12})$ be a characteristic $p$ amalgam of rank $2$ such that there is $G_i^*\normaleq G_i$ satisfying the following for $i\in\{1,2\}$:
\begin{enumerate}
\item $O_p(G_i)\le G_i^*$ and $G_i=G_i^*G_{12}$; 
\item $G_i^*\cap G_{12}$ is the normalizer of a Sylow $p$-subgroup of $G_i^*$; and
\item $G_i^*/O_p(G_i)\cong\PSL_2(p^n), \SL_2(p^n), \PSU_3(p^n), \SU_3(p^n), \Sz(2^n), \Dih(10), \Ree(3^n)$ or $\Ree(3)'$.
\end{enumerate}
Then $\mathcal{A}$ is a \emph{weak BN-pair of rank $2$}. For $G$ a faithful completion of $\mathcal{A}$, we say that $G$ is a group with a weak BN-pair of rank $2$.
\end{definition}

We define the set of groups
\begin{align*}
\bigwedge=\{&\PSL_3(q), \PSp_4(q), \PSU_4(q), \PSU_5(q), \mathrm{G}_2(q), {}^3\mathrm{D}_4(q), {}^2\mathrm{F}_4(2^n), \\
&\mathrm{G}_2(2)', {}^2\mathrm{F}_4(2)', \mathrm{M}_{12}, \mathrm{J}_2, \mathrm{F}_3 \mid q=p^n, \text{$p$ a prime}\}
\end{align*} 
and associate a distinguished prime in each case. For ${}^2\mathrm{F}_4(2^n), \mathrm{G}_2(2)', {}^2\mathrm{F}_4(2)', \mathrm{M}_{12}, \mathrm{J}_2$ the prime is $2$, for $\mathrm{F}_3$ the prime is $3$ and for the other cases, the prime is $p$ where $q=p^n$. 

For $X\in \bigwedge$, let $\Aut^0(X)=\Aut(X)$ unless $X=\PSL_3(q), \PSp_4(2^n), \mathrm{G}_2(3^n)$ in which case $\Aut^0(X)$ is group generated by all inner, diagonal and field automorphisms of $X$ so that $\Aut^0(X)$ is of index $2$ and $\Aut(X)=\langle \Aut^0(X), \phi\rangle$ where $\phi$ is a graph automorphism. Finally, define \[\bigwedge\nolimits^0=\{Y\mid \Inn(X)\le Y\le \Aut^0(X), X\in\bigwedge\}.\]

For the remainder of this work, whenever we describe a group as being locally isomorphic to $Y\in\bigwedge^0$, we will always mean that $Y$ is a faithful completion of the rank $2$ amalgam given by amalgamating two non-conjugate maximal parabolic subgroups of $Y$ which share a common Borel subgroup, sometimes referred to as the \emph{Lie amalgam} of $Y$. It is straightforward to check that this amalgam is a weak BN-pair of rank $2$.

\begin{theorem}\label{greenbook}     
Suppose that $G$ is a group with a weak BN-pair of rank $2$. Then one of the following holds:
\begin{enumerate}
\item $G$ is locally isomorphic to $Y$ for some $Y\in\bigwedge^0$;
\item $G$ is parabolic isomorphic to $\mathrm{G}_2(2)'$, $\mathrm{J}_2$, $\Aut(\mathrm{J}_2)$, $\mathrm{M}_{12}$, $\Aut(\mathrm{M}_{12})$ or $\mathrm{F}_3$.
\end{enumerate}
\end{theorem}
\begin{proof}
This follows from \cite[Theorem A]{Greenbook}, \cite{F3} and \cite{Fan}.
\end{proof}

In the above theorem, using a lemma of Goldschmidt \cite[(2.7)]{goldschmidt}, it may be checked (e.g. using MAGMA) that the groups with a weak BN-pair of rank $2$ are actually determined up to local isomorphism in all cases. However, we will not require this fact here.

Notice that all the candidates for $G_i^*/O_p(G_i)$ in the definition of a weak BN-pair of rank $2$ have strongly $p$-embedded subgroups. Indeed, the fusion categories of groups which possess a weak BN-pair of rank $2$ form the majority of the examples stemming from the hypothesis in the \hyperlink{MainThm}{Main Theorem}. Another important class of amalgams which provide examples in the \hyperlink{MainThm}{Main Theorem} and \hyperlink{ThmC}{Theorem C} are \emph{symplectic amalgams}.

\begin{definition}
Let $\mathcal{A}:=\mathcal{A}(G_1, G_2, G_{12})$ be a characteristic $p$ amalgam of rank $2$. Then $\mathcal{A}$ is a \emph{symplectic amalgam} if, up to interchanging $G_1$ and $G_2$, the following hold:
\begin{enumerate}
\item $O^{p'}(G_1)/O_p(G_1)\cong\SL_2(p^n)$;
\item for $W:=\langle ((O_p(G_1)\cap O_p(G_2))^{G_1})^{G_2}\rangle$, $G_2=G_{12}W$ and $O^p(O^{p'}(G_2))\le W$;
\item for $S\in\syl_p(G_{12})$, $G_{12}=N_{G_1}(S)$;
\item $\Omega(Z(S))=\Omega(Z(O^{p'}(G_2)))$ for $S\in\syl_p(G_{12})$; and
\item for $Z_1:=\langle \Omega(Z(S))^{G_1}\rangle$, $Z_1\le O_p(G_2)$ and there is $x\in G_2$ such that $Z_1^x\not\le O_p(G_1)$.
\end{enumerate}
\end{definition}

\begin{theorem}\label{Symp}
Suppose that $\mathcal{A}:=\mathcal{A}(G_1, G_2, G_{12})$ is a symplectic amalgam such that $G_2/O_p(G_2)$ has a strongly $p$-embedded subgroup and for $S\in\syl_p(G_{12})$, $G_{12}=N_{G_1}(S)=N_{G_2}(S)$. Assume further than $G_i$ is a $\mathcal{K}$-group for $i\in\{1,2\}$. Then one of the following holds, where $\mathcal{A}_k$ corresponds to the listing given in \cite[Table 1.8]{parkerSymp}:
\begin{enumerate}
\item $\mathcal{A}$ has a weak BN-pair of rank $2$ of type ${}^3\mathrm{D}_4(p^n)$ $(\mathcal{A}_{27})$, $\mathrm{G}_2(p^n)$ $(\mathcal{A}_2$, $\mathcal{A}_6$ and $\mathcal{A}_{26}$ when $p\ne 3)$, $\mathrm{G}_2(2)'$ $(\mathcal{A}_{1})$, $\mathrm{J}_2$ $(\mathcal{A}_{41})$ or $\Aut(\mathrm{J}_2)$ $(\mathcal{A}_{41}^1)$;
\item $p=2$, $\mathcal{A}=\mathcal{A}_{4}$, $|S|=2^6$, $O_2(L_2)\cong 2^{1+4}_+$ and $L_2/O_2(L_2)\cong (3\times 3):2$;
\item $p=5$, $\mathcal{A}=\mathcal{A}_{20}$, $|S|=5^6$, $O_5(L_2)\cong 5^{1+4}_+$ and $L_2/O_5(L_2)\cong 2^{1+4}_-.5$;
\item $p=5$, $\mathcal{A}=\mathcal{A}_{21}$, $|S|=5^6$, $O_5(L_2)\cong 5^{1+4}_+$ and $L_2/O_5(L_2)\cong 2^{1+4}_-.\Alt(5)$;
\item $p=5$, $\mathcal{A}=\mathcal{A}_{46}$, $|S|=5^6$, $O_5(L_2)\cong 5^{1+4}_+$ and $L_2/O_5(L_2)\cong 2\cdot\Alt(6)$; or
\item $p=7$, $\mathcal{A}=\mathcal{A}_{48}$, $|S|=7^6$, $O_7(L_2)\cong 7^{1+4}_+$ and $L_2/O_7(L_2)\cong 2\cdot\Alt(7)$.
\end{enumerate}
\end{theorem}
\begin{proof}
We apply the classification in \cite{parkerSymp} and upon inspection of the tables there, we need only rule out $\mathcal{A}_3$, $\mathcal{A}_5$ and $\mathcal{A}_{45}$ when $p=3$; and $\mathcal{A}_{42}$ when $p=2$. Set $Q_i:=O_3(G_i)$ and $L_i:=O^{3'}(L_i)$. With regards to $\mathcal{A}_{45}$, it is proved in \cite[Theorem 11.4]{parkerSymp} that $N_{G_2}(S)\not\le G_1$. In $\mathcal{A}_3$, we have that $L_2/O_3(L_2)\cong \SL_2(3)$ and $|S|=3^6$. In particular, if $G_{12}=N_{G_1}(S)=N_{G_2}(S)$ then $G$ has a weak BN-pair but comparing with the configurations in \cite{Greenbook}, we have a contradiction. 

Suppose that we are in the situation of $\mathcal{A}_5$ so that $L_2$ is of shape $3.((3^2:Q_8)\times (3^2:Q_8)):3$. Furthermore, by \cite[Lemma 6.21]{parkerSymp}, we have that $Q_2=\langle \Omega(Z(Q_1))^{G_2}\rangle$. Let $K_2$ be a Hall $2'$-subgroup of $L_2\cap N_G(S)$. Then $K_2$ is elementary abelian of order $4$. By hypothesis, $K_2$ normalizes $Q_1$ and so $K_2$ normalizes $\Omega(Z(Q_1))$. Moreover, $K_2$ centralizes $\Omega(Z(S))=\Omega(Z(L_2))=\Phi(Q_2)$ and since $|\Omega(Z(Q_1))/\Omega(Z(L_2))|=3$ by \cite[Lemma 6.21]{parkerSymp}, it follows that there is $k\in K$ an involution which centralizes $\Omega(Z(Q_1))$. Since $\langle kQ_2\rangle\normaleq G_2$, we infer that $\Omega(Z(S))=[\langle kQ_2\rangle, \Omega(Z(Q_1))]^{G_2}=[\langle kQ_2\rangle, \langle \Omega(Z(Q_1))^{G_2}\rangle]$. But  $Q_2=\langle \Omega(Z(Q_1))^{G_2}\rangle$ by \cite[Lemma 6.21]{parkerSymp} so that $k$ centralizes $Q_2/\Phi(Q_2)$, a contradiction since $G_2$ is of characteristic $3$.

In the situation of $\mathcal{A}_{42}$ when $p=2$, we have that $L_2/Q_2\cong\Alt(5)\cong\SL_2(4)$ so that $G$ has a weak BN-pair of rank $2$. Since $|S|=2^9$ in this case, comparing with \cite{Greenbook}, we have a contradiction.
\end{proof}

\begin{remark}
The symplectic amalgams $\mathcal{A}_3$, $\mathcal{A}_5$ and $\mathcal{A}_{45}$ where $G_2/O_3(G_2)$ has a strongly $p$-embedded subgroup have as example completions $\Omega_8^+(2): \Sym(3)$, $\mathrm{F}_4(2)$ and $\mathrm{HN}$. Indeed, in these configurations $|S|$ is bounded and one can employ \cite{Comp1} to get a list of candidate fusion systems supported on $S$. It transpires that the only appropriate fusion systems supported on $S$ are exactly the fusion categories of the above examples, but in each case there are three essentials, all normal in $S$, one of which is $\Aut_{\fs}(S)$-invariant while the other two are fused under the action of $\Aut_{\fs}(S)$.
\end{remark}

\begin{remark}
In a later section, we come across an amalgam which satisfies almost all of the properties of $\mathcal{A}_{42}$. Indeed, this amalgam contains $\mathcal{A}_{42}$ as a subamalgam and we show that the fusion system supported from this configuration is the $2$-fusion system of $\PSp_6(3)$. Indeed, $\PSp_6(3)$ is listed as an example completion of $\mathcal{A}_{42}$ in \cite{parkerSymp} and in $\PSp_6(3)$ itself, there is a choice of generating subgroups $G_1, G_2$ such that $(G_1, G_2, G_1\cap G_2)$ is a symplectic amalgam. However, the fusion subsystem generated by the fusion systems of the groups $G_1$ and $G_2$ fails to generate the fusion system of $\PSp_6(3)$. In fact, such a subsystem fails to be saturated.
\end{remark}

We now state the main hypothesis of this paper with regard to fusion systems.

\begin{hypothesis}\label{HypFus}
$\fs$ is a local $\mathcal{CK}$-system on a $p$-group $S$ with two $\Aut_{\fs}(S)$-invariant essential subgroups $E_1, E_2\normaleq S$ such that for $\fs_0:=\langle N_{\fs}(E_1), N_{\fs}(E_2) \rangle$, we have that $O_p(\fs_0)=\{1\}$. Furthermore, for $G_i:=\Out_{\fs}(E_i)$, if $G_i/O_{3'}(G_i)\cong \Ree(3)$, $G_i$ is $p$-solvable or $T$ is generalized quaternion, then $N_{G_i}(T)$ is a strongly $p$-embedded in $G_i$ for $T\in\syl_p(G_i)$.
\end{hypothesis}

We now recognize a characteristic $p$ amalgam of rank $2$ in $\fs_0$. Namely, we take the models $G_1$, $G_2$ and $G_{12}$ of $N_{\fs}(E_1)$, $N_{\fs}(E_2)$ and $N_{\fs}(S)$ and by \cref{robinson}, we have that $\fs_0=\fs_S(G)$ where $G=G_1*_{G_{12}}G_2$, and we take the liberty of recognizing $G_1, G_2$ and $G_{12}$ as subgroups of $G$. Moreover, by the uniqueness of the models in \cref{model}, it follows quickly that we may assume that $N_{G_1}(E_2)=N_{G_2}(E_1)=G_{12}$, and since $E_i$ is $\Aut_{\fs}(S)$-invariant, we deduce that $N_{G_i}(S)\le N_{G_i}(E_{3-i})=G_{12}$ for $i\in\{1,2\}$.

We now have a hypothesis in purely amalgam theoretic terms. Indeed, $G$ is a characteristic $p$ amalgam of rank $2$ such that the following hold:

\begin{hypothesis}\label{MainHypA}
$\mathcal{A}:=\mathcal{A}(G_1, G_2, G_{12})$ is a characteristic $p$ amalgam of rank $2$ with faithful completion $G$ satisfying the following:
\begin{enumerate}
\item for $S\in\syl_p(G_{12})$, $N_{G_1}(S)=N_{G_2}(S)\le G_{12}$; and
\item writing $\bar{G_i}:=G_i/O_p(G_i)$, $\bar{G_i}$ contains a strongly $p$-embedded subgroup and if $\bar{G_i}/O_{3'}(\bar{G_i})\cong \Ree(3)$, $\bar{G_i}$ is $p$-solvable or $\bar{S}$ is generalized quaternion, then $N_{\bar{G_i}}(\bar{S})$ is strongly $p$-embedded in $\bar{G_i}$.
\end{enumerate}
\end{hypothesis}

By the strongly $p$-embedded hypothesis, whenever $\fs_0$ is the $p$-fusion category of a finite group $G$ with $F^*(G)$ one of the groups in $\bigwedge$, we are almost always able to deduce that $F^*(G)=O^{p'}(G)$.

\begin{lemma}\label{noppower}
Suppose that $\fs$ satisfies \cref{HypFus}. Let $G_i$ be a model for $N_{\fs}(E_i)$ such that $S\in\syl_p(G_i)$ with $i\in\{1,2\}$. If the amalgam $\mathcal{A}:=\mathcal{A}(G_1, G_2, G_{12})$ extracted from $\fs_0$ is a weak BN-pair of rank $2$ and is locally isomorphic to $Y\in\bigwedge^0$ with $F^*(Y)\not\cong \mathrm{G}_2(2)', \PSp_4(2)', {}^2\mathrm{F}_4(2)', \mathrm{M}_{12}$ or $\mathrm{J}_2$, then $F^*(Y)=O^{p'}(Y)$.
\end{lemma}
\begin{proof}
Suppose that $\mathcal{A}$ is locally isomorphic to $Y\in\bigwedge^0$. If $m_p(S/O_p(G_i))>1$, then $O^{p'}(G_i/O_p(G_i))$ is determined by \cref{SE2} and comparing with the structure of the parabolic subgroups of $Y$, we deduce that $F^*(Y)=O_{p'}(G)$. If $m_p(S/O_p(G_i))=1$ then $Y$ is generated by only inner and diagonal automorphisms so that $F^*(Y)=Y$ as required.
\end{proof}

\begin{proposition}\label{WBNFS}
Suppose that $\fs$ satisfies \cref{HypFus}. Let $G_i$ be a model for $N_{\fs}(E_i)$ such that $S\in\syl_p(G_i)$ with $i\in\{1,2\}$. If the amalgam $\mathcal{A}:=\mathcal{A}(G_1, G_2, G_{12})$ extracted from $\fs_0$ is a weak BN-pair of rank $2$ then either:
\begin{enumerate}
    \item $\fs_0$ is the $p$-fusion category of $Y$ for some $Y\in\bigwedge^0$; or
    \item $\mathcal{A}$ is parabolic isomorphic $\mathrm{F}_3$.
\end{enumerate}
\end{proposition}
\begin{proof}
By \cref{greenbook} we have that $\mathcal{A}$ is locally isomorphic to $Y$ for some $Y\in\bigwedge^0$; or parabolic isomorphic to $\mathrm{G}_2(2)'$, ${}^2\mathrm{F}_4(2)'$, $\mathrm{J}_2$, $\Aut(\mathrm{J}_2)$, $\mathrm{M}_{12}$, $\Aut(\mathrm{M}_{12})$ or $\mathrm{F}_3$. Assume first that $\mathcal{A}$ is parabolic isomorphic to $\mathrm{G}_2(2)'$, ${}^2\mathrm{F}_4(2)'$, $\mathrm{J}_2$, $\Aut(\mathrm{J}_2)$, $\mathrm{M}_{12}$, $\Aut(\mathrm{M}_{12})$. Then the possible fusion systems on $S$ with trivial $2$-core may be enumerated using the fusion systems package in MAGMA \cite{Comp1}. Indeed, $\fs=\fs_0$ or $\fs$ is the $2$-fusion category of $J_3$ or $\Aut(J_3)$ in which case $\fs_0$ is isomorphic to the $2$-fusion category of $J_2$ resp. $\Aut(J_2)$. Alternatively, except in the case of ${}^2\mathrm{F}_4(2)'$, one could apply \cite{OliSmall} to the reduction of $\fs$ and using the \emph{tameness} of the candidate systems, both $\fs$ and $\fs_0$ can be calculated from the automorphisms groups of certain small simple groups.

Assume now that $\mathcal{A}$ is locally isomorphic to $Y$ where $F^*(Y)=O^{p'}(Y)$ is a simple group of Lie type. By \cite[Corollary 3.1.6]{GLS3}, the $p$-fusion category of $Y$ has only two essential subgroups, namely the unipotent radical subgroups of the maximal parabolic subgroups in $F^*(Y)$, which we label $P_1$ and $P_2$. Hence, by the Alperin--Goldschmidt theorem and  we have that $\fs_S(Y)=\langle N_{\fs_S(Y)}(P_1), N_{\fs_S(Y)}(P_2)\rangle$. By Robinson's result, $\fs_S(Y)$ is isomorphic to the fusion system of the free amalgamated product and by the uniqueness of the amalgam, $\fs_S(Y)=\fs_0$, as desired.
\end{proof}

The case where $S$ is isomorphic to a Sylow $3$-subgroup of $\mathrm{F}_3$ has been treated in \cite{ExoSpo}. Indeed, this is another instance where $\fs_0\subset \fs$. We also mention that for $\fs$ the $5$-fusion category of $\mathrm{Co}_1$, we have that $\fs_0$ is isomorphic to the $5$-fusion category of $\PSp_4(5)$, yet another example of this phenomena.

The main work in this paper is in proving \hyperlink{ThmC}{Theorem C}, which we list below for convenience.

\begin{ThmC}\hypertarget{MainGrpThm}{}
Suppose that $\mathcal{A}=\mathcal{A}(G_1, G_2, G_{12})$ satisfies \cref{MainHypA}. Then one of the following occurs:
\begin{enumerate}
\item $\mathcal{A}$ is a weak BN-pair of rank $2$;
\item $p=2$, $\mathcal{A}$ is a symplectic amalgam, $|S|=2^6$ , $G_1/O_2(G_1)\cong \Sym(3)$ and $G_2/O_2(G_2)\cong (3\times 3):2$;
\item $p=2$, $\Omega(Z(S))\normaleq G_2$, $\langle (\Omega(Z(S))^{G_1})^{G_2})\rangle\not\le O_2(G_1)$, $|S|=2^9$ , $O^{2'}(G_1)/O_2(G_1)\cong \SU_3(2)'$ and $O^{2'}(G_2)/O_2(G_2)\cong\Alt(5)$;
\item $p=3$, $\Omega(Z(S))\normaleq G_2$, $\langle (\Omega(Z(S))^{G_1})\rangle\not\le O_2(G_2)$, $|S|\leq 3^7$ and $O_3(G_1)=\langle (\Omega(Z(S))^{G_1})\rangle$ is cubic $2F$-module for $G_1/O_3(G_1)$; or
\item $p=5$ or $7$, $\mathcal{A}$ is a symplectic amalgam and $|S|=p^6$.
\end{enumerate}
\end{ThmC}

It seems that more information may be extracted than what we have provided here, but with the application of fusion systems in mind and the available results classifying fusion systems supported on $p$-groups of small order, we stop short of completely describing $G_1$ and $G_2$ up to isomorphism, although this seems possible in most cases.

We now prove the \hyperlink{MainThm}{Main Theorem} assuming the validity of \hyperlink{MainGrpThm}{Theorem C}. The necessary structural details for some of the cases we treat are documented later in \cref{b=1a}, \cref{b=1b} and \cref{b=1c}. To verify that two of the fusion systems uncovered are exotic, the classification of the finite simple groups is invoked (see \cite{ExoSpo} and \cite{parkersem}). This is the only occasion in this work where we apply the classification in its full strength and not in an inductive context. Without the classification, outcome (v) below would instead read ``$\fs$ is a simple fusion system on a Sylow $3$-subgroup of $\mathrm{F}_3$ which is not isomorphic to the $3$-fusion category of $\mathrm{F}_3$'' and outcome (vii) would read ``$\fs$ is a simple fusion system on a Sylow $7$-subgroup of $\mathrm{G}_2(7)$ which is not isomorphic to $7$-fusion category of $\mathrm{G}_2(7)$ or $\mathrm{M}$.''

\begin{main}\hypertarget{MainThm}{}
Let $\fs$ be a local $\mathcal{CK}$-system on a $p$-group $S$. Assume that $\fs$ has two $\Aut_{\fs}(S)$-invariant essential subgroups $E_1, E_2\normaleq S$ such that for $\fs_0:=\langle N_{\fs}(E_1), N_{\fs}(E_2) \rangle_S$ the following conditions hold:
\begin{itemize}
\item $O_p(\fs_0)=\{1\}$;
\item for $G_i:=\Out_{\fs}(E_i)$, if $G_i/O_{3'}(G_i)\cong\Ree(3)$, $G_i$ is $p$-solvable or $T$ is generalized quaternion, then $N_{G_i}(T)$ is strongly $p$-embedded in $G_i$ for $T\in\syl_p(G_i)$.
\end{itemize}
Then $\fs_0$ is saturated and one of the following holds:
\begin{enumerate}
\item $\fs_0=\fs_S(G)$, where $F^*(G)$ is isomorphic to a rank $2$ simple group of Lie type in characteristic $p$;
\item $\fs_0=\fs_S(G)$, where $G\cong \mathrm{M}_{12}, \Aut(\mathrm{M}_{12}), \mathrm{J}_2, \Aut(\mathrm{J}_2), \mathrm{G}_2(3)$ or $\PSp_6(3)$ and $p=2$;
\item $\fs_0=\fs_S(G)$, where $G\cong \mathrm{Co}_2, \mathrm{Co}_3,\mathrm{McL}$, $\Aut(\mathrm{McL}), \mathrm{Suz}, \Aut(\mathrm{Suz})$ or $\mathrm{Ly}$ and $p=3$;
\item $\fs_0=\fs_S(G)$, where $G\cong\PSU_5(2), \Aut(\PSU_5(2)),  \Omega_8^+(2), \mathrm{O}_8^+(2), \Omega_{10}^-(2),\\ \Sp_{10}(2), \PSU_6(2)$ or  $\PSU_6(2).2$ and $p=3$;
\item $\fs_0$ is simple fusion system on a Sylow $3$-subgroup of $\mathrm{F}_3$ and, assuming $\mathrm{CFSG}$, $\fs_0$ is an exotic fusion system uniquely determined up to isomorphism;
\item $\fs_0=\fs_S(G)$, where $G\cong \mathrm{Ly}, \mathrm{HN}, \Aut(\mathrm{HN})$ or $\mathrm{B}$ and $p=5$; or
\item $\fs_0$ is a simple fusion system on a Sylow $7$-subgroup of $\mathrm{G}_2(7)$ and, assuming $\mathrm{CFSG}$, $\fs_0$ is an exotic fusion system uniquely determined up to isomorphism.
\end{enumerate}
\end{main}
\begin{proof}
Let $\mathcal{A}(G_1, G_2, G_{12})$ be the amalgam determined by $\fs_0$. If $\mathcal{A}$ is a weak BN-pair of rank $2$ then by \cref{WBNFS}, we either have that $\fs_0$ satisfies one of the conclusions of the theorem; or $\mathcal{A}$ is parabolic isomorphic to $\mathrm{F}_3$ and $S$ is isomorphic to a Sylow $3$-subgroup of $\mathrm{F}_3$. In the latter case, we refer to \cite{ExoSpo} where both $\fs$ and $\fs_0$ are determined. 

Suppose now that $\mathcal{A}$ is a symplectic amalgam but not a weak BN-pair of rank $2$. If $p=2$ and $|S|=2^6$ as in case (ii) of \cref{Symp},  then it follows from \cite[Lemma 6.21]{parkerSymp} that $S=(O^2(G_1)\cap S)(O^2(G_2)\cap S)$ so that $O^2(\fs)=\fs$ by \cite[Theorem I.7.4]{ako}. Checking against the lists provided in \cite[Theorem 4.1]{OliSmall}, $\fs$ is isomorphic to the $2$-fusion category of $\mathrm{G}_2(3)$, which has two essential subgroups both of which are $\Aut_{\fs}(S)$-invariant. By the Alperin--Goldschmidt theorem, we have that $\fs_0=\fs$ so that $\fs_0$ is saturated. If $p\in\{5,7\}$ and $\mathcal{A}$ satisfies (iii)-(vi) of \cref{Symp}, then $|S|\leq p^6$ and again we deduce that $O^p(\fs)=\fs$. Then using the results follows from the tables provided in \cite{Comp1}, we deduce that if $p=5$ then $\fs$ has only two essential subgroups, both of which are $\Aut_{\fs}(S)$-invariant, and $\fs=\fs_0$ is as described. If $p=7$ then $S$ is isomorphic to a Sylow $7$-subgroup of $\mathrm{G}_2(7)$ and $\fs$ is determined in \cite{parkersem}. Indeed, it can be gleaned from \cite[Table 1]{parkersem} that $\fs_0$ is either isomorphic to the $7$-fusion category of $\mathrm{G}_2(7)$ or $\fs_0$ is a simple exotic fusion system on $S$. In both cases, there exist fusion systems which satisfy the hypothesis of $\fs$ with $\fs_0\subset \fs$.

Suppose that outcome (iii) of \hyperlink{MainGrpThm}{Theorem C} holds. Comparing with the structure provided in \cref{b2ii}, it follows that $S=(O^2(G_1)\cap S)(O^2(G_2)\cap S)$ from which it follows that $O^2(O^{2'}(\fs))=O^{2'}(\fs)$ so $O^{2'}(\fs)$ is reduced. Comparing with the lists in \cite{OliSmall}, $O^{2'}(\fs)$ is isomorphic to the $2$-fusion system of $\PSp_6(3)$. Furthermore, by \cite[Proposition 6.4]{OliSmall}, the only saturated fusion system supported on a Sylow $2$-subgroup of $\PSp_6(3)$ with $O_2(\fs)=\{1\}$ is the fusion category of $\PSp_6(3)$, which has exactly two essential subgroups, both of which are $\Aut_{\fs}(S)$-invariant. By the Alperin--Goldschmidt theorem, we have that $\fs_0=\fs=O^{2'}(\fs)$ and the result holds.

Finally, assume that outcome (iv) of \hyperlink{MainGrpThm}{Theorem C} holds and $\mathcal{A}$ is not a weak BN-pair of rank $2$. Applying the structural results provided in \cref{b=1a}, \cref{b=1b} and \cref{b=1c}, we deduce that $S=(S\cap O^3(G_1))(S\cap O^3(G_2))$ and so $\fs=O^3(\fs)$ and $O_3(\fs)=\{1\}$. Since $|S|\leq 3^7$, $\fs$ is determined by the work in \cite{Comp1} and in all cases, there are only two essential subgroups of $\fs$, both of which are $\Aut_{\fs}(S)$-invariant so that $\fs_0=\fs$ is saturated.
\end{proof}

In the above theorem, whenever $\fs\ne \fs_0$, $\fs_0$ is obtained from $\fs$ by pruning (\cite[23]{Comp1}) a class of \emph{pearls} (see \cite{grazian}). Indeed, $\fs_0$ occurs as the fusion theoretic equivalent of the $C(G, S)$ subgroup \cite{CGT}, and $\fs$ is an odd prime equivalent to an \emph{obstruction to pushing up} in the sense of Aschbacher \cite{asch1}.

Given the conclusions of the main theorem, it is not hard to prove \hyperlink{CorA}{Corollary A}.

\begin{CorA}\hypertarget{CorA}{}
Suppose that $\fs$ is a saturated fusion system on a $p$-group $S$ such that $O_p(\fs)=\{1\}$. Assume that $\fs$ has exactly two essential subgroups $E_1$ and $E_2$. Then $N_S(E_1)=N_S(E_2)$ and writing $\fs_0:=\langle N_{\fs}(E_1), N_{\fs}(E_2)\rangle_{N_S(E_1)}$, $\fs_0$ is a saturated normal subsystem of $\fs$ and either
\begin{enumerate}
\item $\fs=\fs_0$ is determined by the \hyperlink{MainThm}{Main Theorem};
\item $p$ is arbitrary, $\fs_0$ is isomorphic to the $p$-fusion category of $H$, where $F^*(H)\cong\PSL_3(p^n)$, and $\fs$ is isomorphic to the $p$-fusion category of $G$ where $G$ is the extension of $H$ by a graph or graph-field automorphism;
\item $p=2$, $\fs_0$ is isomorphic to the $2$-fusion category of $H$, where $F^*(H)\cong\PSp_4(2^n)$, and $\fs$ is isomorphic to the $2$-fusion category of $G$ where $G$ is the extension of $H$ by a graph or graph-field automorphism; or
\item $p=3$, $\fs_0$ is isomorphic to the $3$-fusion category of $H$, where $F^*(H)\cong\mathrm{G}_2(3^n)$, and $\fs$ is isomorphic to the $3$-fusion category of $G$ where $G$ is the extension of $H$ by a graph or graph-field automorphism.
\end{enumerate}
\end{CorA}
\begin{proof}
Note that if both $E_1$ and $E_2$ are $\Aut_{\fs}(S)$-invariant then (appealing to \cref{MaxEssenAmal} to verify that $E_1$ and $E_2$ are maximally essential) $\fs=\fs_0$ is determined by the \hyperlink{MainThm}{Main Theorem}. Assume throughout that at least one of $E_1$ and $E_2$ is not $\Aut_{\fs}(S)$-invariant, and without loss of generality, $E_1$ is not $\Aut_{\fs}(S)$-invariant. Then $N_S(E_1)\alpha\le N_S(E_1\alpha)$ and since $E_1$ is fully $\fs$-normalized, it follows that $N_S(E_1)\alpha=N_S(E_1\alpha)$. Moreover, $E_1\alpha$ is also essential in $\fs$ and so $E_1\alpha=E_2$. By a similar reasoning, $E_2\alpha=E_1$, $\alpha^2\in N_{\fs}(E_1)\cap N_{\fs}(E_2)$ and both $E_1$ and $E_2$ are maximally essential. Suppose first that $p$ is odd. Then $S=N_S(E_1)=N_S(E_2)$ and by \cite[Lemma I.7.6(b)]{ako} and the Alperin--Goldschmidt theorem, $\fs_0$ is a saturated subsystem of $\fs$ of index $2$ and by \cite[Theorem I.7.7]{ako}, $\fs_0$ is normal in $\fs$. Hence, $O_p(\fs_0)$ is normalized by $\fs$ and as $O_p(\fs)=\{1\}$, $O_p(\fs_0)=\{1\}$ and $\fs_0$ is determined by the \hyperlink{MainThm}{Main Theorem}. 

Since there is $\alpha\in\Aut_{\fs}(S)$ such that $E_1\alpha=E_2$, we must have that $E_1\cong E_2$ as abstract $p$-groups. Thus, comparing with the \hyperlink{MainThm}{Main Theorem}, $\fs_0$ is isomorphic to the $p$-fusion category of $H$ where $F^*(H)$ is one of $\PSL_3(p^n)$ or $\mathrm{G}_2(3^n)$ (where $p>2$ is arbitrary or $p=3$ respectively). Indeed, since $\fs_0\normaleq \fs$, there is $\fs^0\normaleq \fs$ with $\fs^0$ isomorphic to the $p$-fusion category of $F^*(H)$ and supported on $S$. At this point, we can either apply \cite[Theorem A]{FSLie}; or recognize that the possible fusion systems correspond exactly to certain subgroups of $\Aut(F^*(H))$ containing $\Inn(F^*(H))$ by applying \cite[Theorem I.7.7]{ako}.

Suppose now that $p=2$. Then $N_S(E_1)=N_S(E_2)=E_1E_2$ has index $2$ in $S$. Let $G_i$ be a model for $N_{\fs}(E_i)$ for $i\in\{1,2\}$. Note that if there is $Q\le N_S(E_1)$ with $Q$ normal in both $N_{\fs}(E_1)$ and $N_{\fs}(E_2)$, then $Q\alpha=Q$ is normal in $\fs$. Since $O_2(\fs)=\{1\}$, we deduce that $Q$ is trivial. Moreover, applying \cite[(2.2.4)]{asch1}, $N_{G_1}(E_2)=N_{G_1}(N_S(E_2))$ is isomorphic to $N_{G_2}(E_1)=N_{G_2}(N_S(E_2)$ by an isomorphism fixing $N_S(E_1)$.

Hence, suppressing the necessary inclusion maps, we form the rank $2$ amalgam $\mathcal{A}:=\mathcal{A}(G_1, G_2, G_{12}^*)$ writing $G_{12}^*$ for the group gained by identifying $N_{G_1}(N_S(E_1))$ with $N_{G_2}(N_S(E_2))$ in the previously described isomorphism. Then \sloppy{$\fs_0=\langle \fs_{N_S(E_1)}(G_1), \fs_{N_S(E_2)}(G_2)\rangle=\fs_{N_S(E_1)}(G_1\ast_{G_{12}^*} G_2)$ by \cref{robinson}, $N_S(E_1)$ is a Sylow of $\fs_0$ and $O_2(\fs_0)=\{1\}$.} Moreover, $\mathcal{A}$ satisfies \cref{MainHyp} and since $E_2=E_1\alpha$, $E_1$ and $E_2$ are isomorphic as abstract $2$-groups. Then $G_1\ast_{G_{12}^*} G_2$ is locally isomorphic to $H$ where $H\in\bigwedge^0$ is as described after \cref{WBDef}, and $F^*(H)\cong \PSL_3(2^n)$ or $\PSp_4(2^n)$. Then by \cref{WBNFS}, $\fs_0$ is isomorphic to the $2$-fusion category of $Y$ and so $\fs_0$ is saturated. Moreover, applying \cite[Theorem I.7.4]{ako} and the Alperin--Goldschmidt theorem, $\fs_0$ is a normal subsystem of index $2$ in $\fs$. Again, there is $\fs^0\normaleq \fs$ with $\fs^0$ isomorphic to the $p$-fusion category of $F^*(H)$ and supported on $N_S(E_1)$ and we can either apply \cite[Theorem A]{FSLie}; or recognize that the possible fusion systems correspond exactly to certain subgroups of $\Aut(F^*(H))$ containing $\Inn(F^*(H))$ by applying \cite[Theorem I.7.7]{ako}.
\end{proof}

%% file: Contents/5.TheAmalgamMethod.tex
\section{The Amalgam Method}\label{AmalSec}

By the previous section, it suffices to prove \hyperlink{MainGrpThm}{Theorem C} and so we classify configurations under the following hypothesis which we fix for the remainder of this work. 

\begin{hypothesis}\label{MainHyp}
$\mathcal{A}:=\mathcal{A}(G_1, G_2, G_{12})$ is a characteristic $p$ amalgam of rank $2$ with faithful completion $G$ satisfying the following:
\begin{enumerate}
\item for $S\in\syl_p(G_{12})$, $N_{G_1}(S)=N_{G_2}(S)\le G_{12}$; and
\item writing $\bar{G_i}:=G_i/O_p(G_i)$, $\bar{G_i}$ contains a strongly $p$-embedded subgroup and if $\bar{G_i}/O_{3'}(\bar{G_i})\cong \Ree(3)$, $\bar{G_i}$ is $p$-solvable or $\bar{S}$ is generalized quaternion, then $N_{\bar{G_i}}(\bar{S})$ is strongly $p$-embedded in $\bar{G_i}$.
\end{enumerate}
\end{hypothesis}

From this point, our methodology is completely based in group theory. At various stages of the analysis, we refer to $\mathcal{A}$ or $G$ as being a minimal counterexample to the \hyperlink{ThmC}{Theorem C}. By this, we mean a counterexample in each case chosen such that $|G_1|+|G_2|$ is as small as possible. 

We assume \cref{MainHyp} and fix the following notation for the remainder of this work. We let $G=G_1\ast_{G_{12}} G_2$ and $\Gamma$ be the (right) coset graph of $G$ with respect to $G_1$ and $G_2$, with vertex set $V(\Gamma)=\{G_ig \mid g\in G, i\in\{1,2\}\}$ and $(G_ig, G_jh)$ an edge if $G_ig\ne G_jh$ and $G_ig\cap G_jh\ne\emptyset$ for $\{i,j\}=\{1,2\}$. It is clear that $G$ operates on $\Gamma$ by right multiplication. Throughout, we identify $\Gamma$ with its set of vertices, let $d(\cdot, \cdot)$ to be the usual distance on $\Gamma$ and observe the following notations.

\begin{notation}
\begin{itemize}
\item For $\delta\in\Gamma$, $\Delta^{(n)}(\delta)=\{\lambda\in\Gamma\mid d(\delta,\lambda)\leq n\}$. In particular, we have that $\Delta^{(0)}(\delta)=\{\delta\}$ and we write $\Delta(\delta):=\Delta^{(1)}(\delta)$.
\item For $\delta\in\Gamma$ and $\lambda\in\Delta(\delta)$, we let $G_{\delta}$ be the stabilizer in $G$ of $\delta$ and $G_{\delta, \lambda}$ be the stabilizer in $G$ of the edge $\{\delta, \lambda\}$.
\item For $\delta\in\Gamma$, $G_{\delta}^{(n)}$ is the largest normal subgroup of $G_{\delta}$ which fixes $\Delta^{(n)}(\delta)$ element-wise. In particular, $G_\delta=G_{\delta}^{(0)}$.
\end{itemize}
\end{notation}

The following proposition is elementary and its proof may be found in \cite[Chapter 3]{Greenbook}.  

\begin{proposition}
The following facts hold:
\begin{enumerate}
\item $G_{G_ig}=G_i^g$ so that every vertex stabilizer is conjugate in $G$ to either $G_1$ or $G_2$. In particular, $G$ has finite vertex stabilizers.
\item Each edge stabilizer of $\Gamma$ is conjugate in $G$ to $G_{12}$ in its action on $\Gamma$.
\item $\Gamma$ is a tree.
\item $G$ acts faithfully and edge transitively on $\Gamma$, but does not act vertex transitively.
\item For each edge $\{\lambda_1, \lambda_2\}$, $G=\langle G_{\lambda_1}, G_{\lambda_2}\rangle$.
\item For $\delta\in\Gamma$ such that $G_{\delta}=G_i^g$, we have that $\Delta(\delta)$ and $G_{\delta}/G_{12}^g$ are equivalent as $G_\delta$-sets. In particular, $G_\delta$ is transitive on $\Delta(\delta)\setminus \{\delta\}$.
\item $G_\delta$ is of characteristic $p$ for all $\delta\in\Gamma$.
\item If $\delta$ and $\lambda$ are adjacent vertices, then $\syl_p(G_{\delta,\lambda})\subseteq \syl_p(G_{\delta})\cap \syl_p(G_{\lambda})$.
\item If $\delta$ and $\lambda$ are adjacent vertices, then for $S\in\syl_p(G_{\delta,\lambda})$, $N_{G_{\delta}}(S)=N_{G_{\lambda}}(S)\le G_{\delta, \lambda}$.
\end{enumerate}
\end{proposition}

The above proposition, along with \cref{BasicAmal} and \cref{crit pair}, holds in far greater generality than the situation enforced by \cref{MainHyp}.

The following notations will be used extensively throughout the rest of this work.

\begin{notation}\label{BasicAmalNot}
Set $\delta\in\Gamma$ to be an arbitrary vertex and $S\in\syl_p(G_{\delta})$.
\begin{itemize}
\item $L_{\delta}:=O^{p'}(G_{\delta})$.
\item $Q_\delta:=O_p(G_{\delta})=O_p(L_{\delta})$.
\item $\bar{L_{\delta}}:=L_{\delta}/Q_{\delta}$.
\item $Z_{\delta}:=\langle \Omega(Z(S))^{G_{\delta}}\rangle$.
\item For $n\in\N$, $V_{\delta}^{(n)}:=\langle Z_{\lambda}\mid d(\lambda, \delta)\leq n\rangle\normaleq G_{\delta}$, with the additional conventions $V_{\delta}^{(0)}=Z_{\delta}$ and $V_{\delta}:=V_{\delta}^{(1)}$.
\item $b_\delta:=\mathrm{min}_{\lambda\in\Gamma}\{d(\delta, \lambda)\mid Z_{\delta}\not\le G_{\lambda}^{(1)}\}$.
\item $b:=\mathrm{min}_{\delta\in\Gamma}\{b_\delta\}$.
\end{itemize}
\end{notation}

We refer to $b$ as the \emph{critical distance} of the amalgam. Indeed, as $G$ acts edge transitively on $\Gamma$ it follows that $b=\mathrm{min}\{b_\delta, b_\lambda\}$ where $\delta$ and $\lambda$ are any adjacent vertices in $\Gamma$. A \emph{critical pair} is any pair $(\delta, \lambda)$ such that $Z_\delta\not\le G_{\lambda}^{(1)}$ and $d(\delta, \lambda)=b$. This definition is not symmetric and so $(\lambda, \delta)$ is not necessarily a critical pair in this case.

It is clear from the definition that symplectic amalgams have critical distance $2$. It is remarkable that in all the examples we uncover, $b\leq 5$ and if $G$ does not have a weak BN-pair, then $b\leq 2$.

\begin{proposition}\label{BasicAmal}
The following facts hold:
\begin{enumerate}
\item $b\geq 1$ is finite.
\item We may choose $\{\alpha,\beta\}$ such that $\{G_{\alpha},G_\beta\}=\{G_1, G_2\}$ and $G_{\alpha,\beta}=G_{12}\ge N_{G_i}(S)$ for $i\in\{1,2\}$.
\item If $N\le G_{\alpha,\beta}$, $N_{G_{\alpha}}(N)$ operates transitively on $\Delta(\alpha)$ and $N_{G_{\beta}}(N)$ operates transitively on $\Delta(\beta)$, then $N=\{1\}$.
\item For $\delta\in\Gamma$, $\lambda\in\Delta(\delta)$ and $T\in\syl_p(G_{\delta, \lambda})$, no subgroup of $T$ is normal in $\langle L_{\delta}, L_{\lambda}\rangle$.
\item For $\delta\in\Gamma$ and $\lambda\in\Delta(\delta)$, there does not exist a non-trivial element $g\in G_{\delta, \lambda}$ with $gQ_{\delta}/Q_{\delta}\in Z(L_{\delta}/Q_{\delta})$ and $gQ_{\lambda}/Q_{\lambda}\in Z(L_{\lambda}/Q_{\lambda})$.
\item For $\delta\in\Gamma$ and $\lambda\in\Delta(\delta)$, $V_{\lambda}^{(i)}=\langle (V_{\delta}^{(i-1)})^{G_{\lambda}}\rangle$.
\end{enumerate}
\end{proposition}

For the remainder of this work, we will often fix a critical pair $(\alpha, \alpha')$. As $\Gamma$ is a tree, we may set $\beta$ to be the unique neighbour of $\alpha$ with $d(\beta, \alpha')=b-1$. Then we label each vertex along the path from $\alpha$ to $\alpha'$ additively e.g. $\beta=\alpha+1$, $\alpha'=\alpha+b$. In this way we also see that $\beta$ may be written as $\alpha'-b+1$ and so we will often write vertices on the path from $\alpha'$ to $\alpha$ subtractively with respect to $\alpha'$. The following diagram better explains the situation.

\begin{center}
\begin{tikzpicture}
\draw[black, thick] (-6.5,0) -- (-0.5,0);
\draw[black, thick] (0.5,0) -- (6.5,0);
\filldraw[black] (7.25,0) circle (0.5pt) node[anchor=south] {};
\filldraw[black] (7,0) circle (0.5pt) node[anchor=south] {};
\filldraw[black] (6.75,0) circle (0.5pt) node[anchor=south] {};
\filldraw[black] (6,0) circle (2pt) node[anchor=south] {$\alpha-1$};
\filldraw[black] (4.5,0) circle (2pt) node[anchor=south] {$\alpha$};
\filldraw[black] (3,0) circle (2pt) node[anchor=south] {$\beta$};
\filldraw[black] (1.5,0) circle (2pt) node[anchor=south] {$\alpha+2$};
\filldraw[black] (0.25,0) circle (0.5pt) node[anchor=south] {};
\filldraw[black] (0,0) circle (0.5pt) node[anchor=south] {};
\filldraw[black] (-0.25,0) circle (0.5pt) node[anchor=south] {};
\filldraw[black] (-1.5,0) circle (2pt) node[anchor=south] {$\alpha'-2$};
\filldraw[black] (-3,0) circle (2pt) node[anchor=south] {$\alpha'-1$};
\filldraw[black] (-4.5,0) circle (2pt) node[anchor=south] {$\alpha'$};
\filldraw[black] (-6.0,0) circle (2pt) node[anchor=south] {$\alpha'+1$};
\filldraw[black] (-7.25,0) circle (0.5pt) node[anchor=south] {};
\filldraw[black] (-7,0) circle (0.5pt) node[anchor=south] {};
\filldraw[black] (-6.75,0) circle (0.5pt) node[anchor=south] {};
\end{tikzpicture}
\end{center}

\begin{lemma}\label{crit pair}
Let $\delta\in\Gamma$, $(\alpha, \alpha')$ be a critical pair, $T\in\syl_p(G_\alpha)$ and $S\in\syl_p(G_{\alpha,\beta})$. Then
\begin{enumerate}
\item $Q_\delta\le G_{\delta}^{(1)}$;
\item $Z_{\alpha'}\le G_\alpha$, $Z_{\alpha}\le G_{\alpha'}$ and $[Z_{\alpha}, Z_{\alpha'}]\le Z_{\alpha}\cap Z_{\alpha'}$;
\item $Z_\alpha\ne \Omega(Z(T))$; and
\item if $\Omega(Z(S))$ is centralized by $L\le G_\beta$ such that $L$ acts transitively on $\Delta(\beta)$, then $Z(L_\alpha)=\{1\}$.
\end{enumerate}
\end{lemma}
\begin{proof}
For all $\lambda\in\Delta(\delta)$, we have that $Q_\delta\le T_\lambda\in\syl_p(G_\lambda\cap G_\delta)$ and $Q_\delta\le G_\lambda$. Since $Q_\delta\normaleq G_\delta$, it follows immediately that $Q_\delta\le G_{\delta}^{(1)}$. By the minimality of $b$, we have that $Z_{\alpha'}\le G_{\beta}^{(1)}\le G_\alpha$ and similarly $Z_{\alpha}\le G_{\alpha'-1}^{(1)}\le G_{\alpha'}$. In particular, $Z_\alpha$ normalizes $Z_{\alpha'}$ and vice versa, so that $[Z_{\alpha}, Z_{\alpha'}]\le Z_{\alpha}\cap Z_{\alpha'}$.

Suppose that $Z_\alpha=\Omega(Z(T))$. Then $Z_{\alpha}=\Omega(Z(S))$ by the transitivity of $G_{\alpha}$. By definition and minimality of $b$, $Z_\alpha\le Z_\beta\le G_{\alpha'}^{(1)}$, a contradiction. Finally, suppose that $\Omega(Z(S))$ is centralized by $L\le G_\beta$ such that $L$ acts transitively on $\Delta(\beta)$. Since $Q_{\alpha}$ is self-centralizing, it follows that $Z(L_\alpha)$ is a $p$-group and so $\Omega(Z(L_\alpha))\le \Omega(Z(S))$ and $L$ centralizes $\Omega(Z(L_\alpha))$. Then \cref{BasicAmal} (iii) implies that $\Omega(Z(L_\alpha))=\{1\}$, and so $Z(L_\alpha)=\{1\}$.
\end{proof}

As described at the end of  \cref{GrpSec}, we can guarantee cubic action on a faithful module for $\bar{L_{\delta}}$ for $\delta$ at least one of $\alpha,\beta$. We use critical subgroups to achieve this and refer to \cref{CriticalSubgroup} for their properties. We describe this in the following proposition.

\begin{proposition}\label{CriticalAmalgam}
There is $\lambda\in\Gamma$ such that there is a $\bar{G_\lambda}$-module $V$ on which $p'$-elements of $\bar{G_\lambda}$ act faithfully and a $p$-subgroup $C$ of $\bar{G_\lambda}$ such that $[V, C, C, C]=\{1\}$.
\end{proposition}
\begin{proof}
Let $(\alpha, \dots,\alpha')$ be a path in $\Gamma$ with $(\alpha, \alpha')$ a critical pair. For each $\lambda\in (\alpha,\dots, \alpha')$, set $K_{\lambda}$ to be a critical subgroup of $Q_{\lambda}$. Since $Z_{\alpha}\le K_{\alpha}$, we must have that $K_{\alpha}\not\le Q_{\alpha'+1}$. Set $c:=\{\mathrm{min}(d(\mu, \lambda)) \mid K_{\mu}\not\le Q_{\lambda}, \mu,\lambda\in(\alpha,\dots,\alpha'+1)\}$. Choose a pair $(\mu, \lambda)$ such that $K_{\mu}\not\le Q_{\lambda}$ and $d(\mu, \lambda)=c$. Then, by minimality of $c$, $K_{\mu}\le G_{\lambda}$ but $K_{\mu}\not\le Q_{\lambda}$ and from the definition of a critical subgroup, $p'$-elements of $\bar{G_{\lambda}}$ act faithfully on the $\bar{G_\lambda}$-module $K_{\lambda}/\Phi(K_{\lambda})$. Moreover, again by minimality, $K_{\lambda}$ normalizes $K_{\mu}$ so that $[K_{\lambda}, K_{\mu}, K_{\mu}, K_{\mu}]\le [K_{\mu}, K_{\mu}, K_{\mu}]=\{1\}$, as required.
\end{proof}

\begin{remark}
It seems likely that the above observation could also be proved using the results on Glauberman's $K$-infinity groups \cite[Theorem A]{Kinf}. Furthermore, following the above proof, we see that either $K_{\mu}$ acts quadratically on $K_{\lambda}$; or $K_{\lambda}\not\le Q_{\mu}$ and $K_{\lambda}$ acts cubically on $K_{\mu}$.
\end{remark}

We are now in a good position to apply Hall--Higman style arguments whenever $p\geq 5$. We get the following fact almost immediately from \cref{CubicAction}.

\begin{corollary}\label{CriticalFusion}
Suppose that $p\geq 5$, and $\bar{L_{\alpha}}$ and $\bar{L_{\beta}}$ have strongly $p$-embedded subgroups. Then, for some $\lambda\in\{\alpha, \beta\}$, one of the following holds:
\begin{enumerate}
    \item $p\geq 5$ is arbitrary and $\bar{L_{\lambda}}\cong \PSL_2(p^n), \SL_2(p^n), \PSU_3(p^n)$ or $\SU_3(p^n)$ for $n\in\N$; or
    \item $p=5$ and $\bar{L_{\lambda}}\cong 3\cdot\Alt(6)$ or $3\cdot\Alt(7)$.
\end{enumerate} 
\end{corollary}
\begin{proof}
By \cref{CriticalAmalgam}, there is a $p$-element $x\in\bar{L_{\lambda}}$ which acts cubically on $K_{\lambda}/\Phi(K_{\lambda})$. Suppose there is $y\in L_{\lambda}$ such that $[y, K_{\lambda}]\le \Phi(K_{\lambda})$. Since $K_{\lambda}$ is a critical subgroup, by coprime action, $y$ is a $p$-element so that $C_{L_{\lambda}}(K_{\lambda}/\Phi(K_{\lambda}))$ is a normal $p$-subgroup. In particular, $\bar{L_{\lambda}}$ acts faithfully on $K_{\lambda}/\Phi(K_{\lambda})$ and so we may apply \cref{CubicAction} and the result holds.
\end{proof}

In the following proposition, we essentially rid ourselves of the possibility that one of $G_\alpha$ or $G_\beta$ involves $\Sz(32):5$. This group is problematic in much the same way as the $\Ree(3)$ and $p$-solvable cases we exclude in the hypothesis of the \hyperlink{MainThm}{Main Theorem} in that there may arise situations in which the centralizer of a non-central chief factor is not $p$-closed. The following proposition asserts that this is not the case. We also believe that this result could also be extended to get rid of the $p$-solvable case whenever $p\geq 5$. The $\Ree(3)$ case when $p=3$ would rely on some control of transfer results at the prime $3$ which the author is unaware of.

\begin{proposition}\label{Ree3}
For any $\lambda\in\Gamma$ and $S\in\syl_5(G_\lambda)$, if $\bar{L_{\lambda}}/O_{5'}(\bar{L_{\lambda}})\cong \Sz(32):5$, then $O^5(L_\lambda)\le \langle x^{G_{\lambda}}\rangle$ for any $x\in S\setminus Q_\lambda$.
\end{proposition}
\begin{proof}
Assume that $G$ is a minimal counterexample to the above statement. Set $G_1=G_\lambda$, $G_2=G_\mu$ for some $\mu\in\Delta(\lambda)$, $Q_i=O_5(G_i)$ for $i\in\{1,2\}$ and $S\in \syl_5(G_{\lambda, \mu}$). We may fix $j\in\{1,2\}$ such that $K^\infty(S)\not\normaleq G_j$, where $K^\infty(S)$ is defined in \cite{Kinf}. Writing $L_j=O^{5'}(G_j)$, we have that $K^\infty(S)\not\normaleq L_j$. Then, by \cite[Theorem 4]{Kinf}, there is $x\in S\setminus Q_i$ and a chief factor $X/Y$ of $L_j$ contained in $Q_i$ such that $[X,x,x]\le Y$ but $[X, x]\not\le Y$. 

Set $H_j:=\langle x^{G_j}\rangle Q_j$ so that by \cref{SE1} and \cref{SE2}, either $H_j=L_j$, $G_j$ is $5$-solvable or for $\bar{G_j}:=G_j/Q_j$, $\bar{L_j}/O_{p'}(\bar{L_j})\cong \Sz(32):5$ in which case, perhaps $\bar{H_j}/O_{5'}(\bar{H_j})\cong \Sz(32)$. Note that in all cases, $Q_j=O_{5}(H_j)$. Let $Y\le V\le U\le X$ with $U/V$ a chief factor for $H_j$ in $Q_j$. Then, unless, $G_j$ is $5$-solvable we may assume that $[U, x, x]\le V$ and $[U, x]\not\le V$. Even if $G_j$ is $5$-solvable, then this holds unless $H_j=C_{H_j}(X/Y)\langle x\rangle$ and as $G_j=H_j N_{G_j}(\langle x\rangle Q_j)$, we deduce that $[X, x]\normaleq G_j$ so that $[X, x]\le Y$, and we obtain a contradiction. Now, applying \cref{SEQuad}, we deduce that $H_j/C_{H_j}(U/V)\cong \SL_2(5^n)$ or $\SU_3(5^n)$. Then \cref{SE1} and \cref{SE2} reveal that $H_j=L_j$.

Applying \cref{CriticalFusion} we have that $\bar{L_j}\cong \SL_2(5^n)$ or $\SU_3(5^n)$ so that $G_j=G_{\mu}$. Then $G_{\lambda,\mu}=N_{G_\mu}(S)=N_{G_\mu}(K^\infty(S))$. Applying \cite[Theorem B]{Kinf}, $G_{\mu}/O^5(G_\mu)\cong G_{\lambda,\mu}/O^5(G_{\lambda,\mu})$ from which it follows that $Q_\mu/Q_\mu\cap O^5(G_\mu)\cong Q_\mu/ Q_\mu\cap O^5(G_{\lambda,\mu})$ so that $Q_\mu\cap O^5(G_\mu)=Q_\mu\cap O^5(G_{\lambda,\mu})$. But now, since $\Sz(32):5$ has no outer automorphisms, we deduce that $\bar{G_{\lambda}}/O_{5'}(\bar{G_\lambda})\cong \Sz(32):5$ so that $(O^5(G_{\lambda,\mu})\cap S)Q_\lambda<S$. In particular, writing $S^*:=(Q_\mu\cap O^5(G_\mu))Q_\lambda$, $L_\delta^*:=\langle (S^*)^{G_\delta})\rangle$ and $G_\delta^*:=L_\delta^*K$ for $\delta\in\{\lambda, \mu\}$ and $K$ a Hall $5'$-subgroup of $G_{\lambda,\mu}$, we have that $L_\delta^*=O^{5'}(G_\delta^*)$, $G_{\delta}^*\normaleq G_\delta$ and $|S/S^*|=|L_\delta/L_\delta^*|=|G_\delta/G_\delta^*|=5$ for $\delta\in\{\lambda. \mu\}$. Hence, any subgroup of $G_\lambda^*\cap G_\mu^*$ which is normal in $\langle G_\lambda^*, G_\mu^*\rangle$ is also normalized by $G$, so is trivial. Thus, the triple $(G_\lambda^*, G_\mu^*, G_\lambda^*\cap G_\mu^*)$ satisfies \cref{MainHyp}, and by minimality, and since $\bar{L_\lambda^*}/O_{5'}(\bar{L_{\lambda}^*})\cong \Sz(32)$, we have a contradiction.
\end{proof}

\begin{lemma}\label{p-closure}
Suppose that $N\normaleq G_\delta$ with $N$ not $p$-closed and set $S\in\syl_p(G_\delta)$. Then the following holds:
\begin{enumerate}
\item If $L_\delta$ is not $p$-solvable, then $O^p(L_\delta)\le N$.
\item If $L_\delta$ is $p$-solvable, then $K\le NQ_\delta$, where $\bar{K}$ is the unique normal subgroup of $\bar{L_\delta}$ which is divisible by $p$ and minimal with respect to this constraint.
\item $G_\delta=NN_{G_\delta}(S)$ and $N$ is transitive on $\Delta(\delta)$.
\item For $U/V$ any non-central chief factor for $L_{\delta}$ inside of $Q_{\delta}$, we have that $Q_{\delta}\in\syl_p(C_{L_{\delta}}(U/V))$.
\end{enumerate}
\end{lemma}
\begin{proof}
Suppose $L_\delta$ is not $p$-solvable and let $A\in\syl_p(N)$. Notice that as $N$ is not $p$-closed, $A\not\le Q_\delta$ and since $\bar{L_\delta}$ has a strongly $p$-embedded subgroup, by \cref{MainHyp} we have that $\wt L_\delta:=\bar{L_\delta}/O_{p'}(\bar{L_\delta})$ is isomorphic to a non-abelian simple group; $\Sz(32):5$ or $\Ree(3)$. If $\wt{L_\delta}$ is a non-abelian simple group then $\wt L_\delta=\langle \wt A^{\wt L_\delta}\rangle$. In particular, $\bar{S}\le \bar{\langle A^{L_\delta}\rangle}$ and so $S\le \langle A^{L_\delta} \rangle Q_\delta\le L_\delta$ and since $L_\delta=O^{p'}(L_\delta)$, $L_\delta=\langle A^{L_\delta} \rangle Q_\delta$. It then follows that $O^p(L_\delta)\le\langle A^{L_\delta} \rangle\le N$. If $\wt L_\delta\cong \Ree(3)$, then as $N_{\bar{L_\delta}}(\bar{S})$ is strongly $3$-embedded in $\bar{L_\delta}$ and $m_3(\bar{S})=2$, by coprime action we have that $O_{3'}(\bar{L_{\delta}})$ normalizes $\bar{S}$. Hence, $\{1\}=[O_{3'}(\bar{L_\delta}), \bar{S}]=[O_{3'}(\bar{L_\delta}), \bar{L_\delta}]$. Then $O^3(\bar{L_\delta})$ is isomorphic to an extension of $\PSL_2(8)$ by a $3'$-group so that $O^{3'}(O^3(\bar{L_\delta}))\cong \PSL_2(8)$. But now, since $O_{3'}(\bar{L_\delta})$ normalizes $\bar{S}$, $O^{3'}(O^3(\bar{L_{\delta}}))\bar{S}\normaleq \bar{L_\delta}$ and $\bar{L_\delta}\cong \Ree(3)$ and so (i) holds in this case. By \cref{Ree3}, (i) holds when $\wt L_\delta\cong \Sz(32):5$. Thus, (i) holds in all cases.

By the Frattini argument $G_\delta=L_\delta N_{G_\delta}(S)=O^{p}(L_\delta) N_{G_\delta}(S)=\langle A^{G_\delta}\rangle N_{G_\delta}(S)$. Since $\langle A^{G_\delta} \rangle\le N$, (iii) follows whenever $L_{\delta}$ is not $p$-solvable.

Suppose now that $L_\delta$ is $p$-solvable and let $\bar{K}$ be the unique minimal normal subgroup of $\bar{L_\delta}$ divisible by $p$. Again, we let $A\in\syl_p(N)$ and remark that since $N$ is not $p$-closed $A\not\le Q_\delta$. Hence, $p\divides |\bar{N}|$ so that $\bar{K}\le \bar{N}$ and $K\le NQ_\delta$, completing the proof of (ii). By \cref{SE1}, $\bar{L_\delta}=N_{\bar{L_\delta}}(\bar{S})\bar{K}\le \bar{N_{G_\delta}(S)}\bar{N}$ so that $G_\delta=L_\delta N_{G_\delta}(S)\le NN_{G_\delta}(S)\le G_\delta$, completing the proof of (iii).

For (iv), choose any non-central chief factor $U/V$ for $L_{\delta}$ inside $Q_{\delta}$. Then $U/V$ is a faithful, irreducible module for $L_{\delta}/C_{L_{\delta}}(U/V)$. Since $[Q_{\delta}, U]\normaleq L_{\delta}$ and $[Q_{\delta}, U]<U$, $Q_{\delta}\le C_{L_{\delta}}(U/V)$. Moreover, as $C_{L_{\delta}}(U/V)$ is normal in $L_{\delta}$, we deduce that $O_p(C_{L_{\delta}}(U/V))=Q_{\delta}$. If $C_{L_{\delta}}(U/V)$ is not $p$-closed, then $L_{\delta}=C_{L_{\delta}}(U/V) N_{L_{\delta}}(S)$ and it follows that $U/V$ is irreducible for $N_{L_{\delta}}(S)$. But then $[U/V, S]=\{1\}$ from which it follows that $\{1\}=[U/V, \langle S^{L_{\delta}} \rangle]=[U/V, L_{\delta}]$, a contradiction. Hence, (iv). 
\end{proof}

Of course, the following lemma is superfluous if we initially assumed that the essentials $E_1, E_2$ of $\fs$ were maximally essential.

\begin{proposition}\label{MaxEssenAmal}
For all $\delta\in\Gamma$ and $\lambda\in\Delta(\delta)$, $Q_{\delta}\not\le Q_{\lambda}$.
\end{proposition}
\begin{proof}
Suppose that there is $\delta\in\Gamma$ and $\lambda\in\Delta(\delta)$ with $Q_{\delta}\le Q_{\lambda}$ and let $S\in\syl_p(G_{\delta,\lambda})$. Then $J(Q_{\lambda})\not\le Q_{\delta}$ for otherwise, by \cref{BasicJS} (v), $J(Q_{\lambda})=J(Q_{\delta})\normaleq \langle G_{\lambda}, G_{\delta}\rangle$. Furthermore, since $C_S(Q_{\delta})\le Q_{\delta}$, $\Omega(Z(Q_{\lambda}))<\Omega(Z(Q_{\delta}))$. Let $V:=\langle \Omega(Z(Q_{\lambda}))^{G_{\delta}}\rangle\le \Omega(Z(Q_{\delta}))$ and choose $A\in\mathcal{A}(Q_{\lambda})\setminus \mathcal{A}(Q_{\delta})$. If $Q_{\delta}<C_S(V)$, then by \cref{p-closure} (iii), $G_{\delta}=\langle C_S(V)^{G_{\delta}}\rangle N_{G_\delta}(S)=C_{G_{\delta}}(V)N_{G_\delta}(S)$ normalizes $\Omega(Z(Q_{\lambda}))$, a contradiction. Hence, $Q_{\delta}=C_S(V)$.

By the choice of $A$, $|A|\geq |C_A(V)V|=|C_A(V)||V|/|V\cap C_A(V)|=|C_A(V)||V|/|V\cap A|$. Since $A=\Omega(C_S(A))$, we have that $A\cap V=C_V(A)$ and rearranging we conclude that $|A|/|C_A(V)|\geq |V|/|C_V(A)|$ and $A/C_A(V)\cong AQ_{\delta}/Q_{\delta}$ is an offender on the FF-module $V$. By \cref{SEFF},  $L_{\delta}/C_{L_{\delta}}(V)\cong \SL_2(p^n)$ and $V/C_V(O^p(L_{\delta}))$ is a natural $\SL_2(q)$-module. But $Q_{\lambda}/Q_{\delta}< S/Q_{\delta}$ is a $G_{\lambda, \delta}$-invariant subgroup of $S/Q_{\delta}$, a contradiction by \cref{SLGen} (vi). 
\end{proof}

\begin{lemma}\label{p-closure2}
Let $\delta\in\Gamma$, $(\alpha, \alpha')$ be a critical pair and $S\in\syl_p(G_{\alpha,\beta})$. Then
\begin{enumerate}
\item $Q_\delta\in\syl_p(G_\delta^{(1)})$ and $G_\delta^{(1)}/Q_{\delta}$ is centralized by $L_{\delta}/Q_{\delta}$;
\item either $Q_\delta\in\syl_p(C_{L_\delta}(Z_\delta))$ or $Z_\delta=\Omega(Z(L_\delta))$;
\item $Z_\alpha\not\le Q_{\alpha'}$; and
\item $C_S(Z_\alpha)=Q_\alpha$, and $C_{G_{\alpha}}(Z_{\alpha})$ is $p$-closed and $p$-solvable.
\end{enumerate}
\end{lemma}
\begin{proof}
By \cref{crit pair} (i), we assume that $Q_\delta<T$ for $T\in\syl_p(G_\delta^{(1)})$. Since $G_\delta^{(1)}\normaleq G_\delta$ it follows that $O_p(G_\delta^{(1)})=Q_\delta$ and so $G_\delta^{(1)}$ is not $p$-closed. But by \cref{p-closure} (iii), then $G_\delta^{(1)}$ would be transitive on $\Delta(\delta)$, a clear contradiction. Thus, $Q_\delta\in\syl_p(G_\delta^{(1)})$. Letting $P\in\syl_p(G_{\delta})$, $[P, G_{\delta}^{(1)}]\le P\cap G_{\delta}^{(1)}=Q_{\delta}$ so that $[L_{\delta}, G_{\delta}^{(1)}]\le Q_{\delta}$, and so (i) holds.

If $Q_\delta\not\in\syl_p(C_{L_\delta}(Z_\delta))$ then by \cref{p-closure} (iii), $G_\delta=C_{L_\delta}(Z_\delta)N_{G_\delta}(S)$ and so $Z_\delta=\langle \Omega(Z(S))^{G_\delta}\rangle=\Omega(Z(S))$. But then $\{1\}=[Z_\delta, S]^{G_\delta}=[Z_\delta, L_\delta]$ and so $Z_\delta\le Z(L_\delta)$. Since $Q_{\delta}$ is self-centralizing, $Z(L_\delta)$ is a $p$-group and $Z_\delta=\Omega(Z(S))=\Omega(Z(L_\delta))$, so that (ii) holds.

If $Z_\alpha\le Q_{\alpha'}$ then $Z_\alpha\le G_{\alpha'}^{(1)}$ a contradiction and so (iii) holds. Since $Z_{\alpha}\ne \Omega(Z(S))$ by \cref{crit pair} (iii), $C_S(Z_{\alpha})=Q_{\alpha}\normaleq C_{G_\alpha}(Z_{\alpha})$ so that $C_{G_{\alpha}}(Z_{\alpha})$ is $p$-closed and $p$-solvable.
\end{proof}

By the above lemma, we can reinterpret the minimal distance $b$ as $b=\mathrm{min}_{\delta\in\Gamma}\{b_\delta\}$ where $b_\delta:=\mathrm{min}_{\lambda\in\Gamma}\{d(\delta, \lambda)\mid Z_{\delta}\not\le Q_{\lambda}\}$.

\begin{lemma}\label{critpair2}
Let $(\alpha, \alpha')$ be a critical pair. Then
\begin{enumerate}
\item if $Z_{\alpha'}\le Z(L_{\alpha'})$ then $\alpha$ is not conjugate to $\alpha'$; and
\item $C_{Z_{\alpha}}(Z_{\alpha'})\ne Z_\alpha\cap Q_{\alpha'}$ if and only if $Z_{\alpha'}=\Omega(Z(L_{\alpha'}))$ and $(\alpha', \alpha)$ is not a critical pair.
\end{enumerate}
\end{lemma}
\begin{proof}
Suppose $Z_{\alpha'}\le Z(L_{\alpha'})$. By \cref{p-closure2} (ii), $Z_{\alpha'}=\Omega(Z(L_{\alpha'}))$. If $\alpha$ and $\alpha'$ were conjugate, then $Z_\alpha=\Omega(Z(L_\alpha))$, a contradiction to \cref{crit pair} (iii).

Suppose that $Z_{\alpha'}=\Omega(Z(L_{\alpha'}))$. Since $Z_\alpha\not\le Q_{\alpha'}$ but $Z_\alpha\le L_{\alpha'}$, we infer that $Z_\alpha=C_{Z_\alpha}(Z_{\alpha'})\ne Z_{\alpha}\cap Q_{\alpha'}$. Suppose conversely that $C_{Z_\alpha}(Z_{\alpha'})\ne Z_{\alpha}\cap Q_{\alpha'}$. Then $C_{L_{\alpha'}}(Z_{\alpha'})$ is not $p$-closed and by \cref{p-closure2} (ii), we have that $Z_{\alpha'}=\Omega(Z(L_{\alpha'}))$.
\end{proof}

\begin{lemma}\label{VAbelian}
Suppose that $b>2n$. Then $V_\delta^{(n)}$ is abelian for all $\delta\in\Gamma$.
\end{lemma}
\begin{proof}
Since $b>2n$, for all $\lambda, \mu\in\Delta^{(n)}(\delta)$ we have that $Z_\lambda\le G_{\mu}^{(1)}$ by the minimality of $b$. Thus, $Z_\lambda\le Q_\mu$, $Z_\lambda$ centralizes $Z_\mu$ and since $V_\delta^{(n)}=\langle Z_\mu \mid \mu\in\Delta^{(n)}(\delta) \rangle$, it follows that $V_\delta^{(n)}$ is abelian.
\end{proof}

\begin{lemma}\label{CommCF}
$V_{\lambda}^{(n)}/[V_{\lambda}^{(n)}, Q_{\lambda}]$ contains a non-central chief factor for $L_{\lambda}$ for all $n\geq 1$ such that $V_{\lambda}^{(n)}\le Q_{\lambda}$.
\end{lemma}
\begin{proof}
Set $V_{\mu}^{(0)}=Z_{\mu}$ for all $\mu\in\Gamma$ and suppose that $O^p(L_{\lambda})$ centralizes $V_{\lambda}^{(n)}/[V_{\lambda}^{(n)}, Q_{\lambda}]$. Observe that $V_{\lambda}^{(n)}=\langle (V_{\mu}^{(n-1)})^{L_{\lambda}}\rangle$ for $\mu\in\Delta(\lambda)$ so that $V_{\mu}^{(n-1)}\not\le [V_{\lambda}^{(n)}, Q_{\lambda}]< V_{\lambda}^{(n)}$. Moreover, $V_{\mu}^{(n-1)}[V_{\lambda}^{(n)}, Q_{\lambda}]\normaleq L_{\lambda}$ so that $V_{\lambda}^{(n)}=V_{\mu}^{(n-1)}[V_{\lambda}^{(n)}, Q_{\lambda}]$. Set $V_i:=[V_{\lambda}^{(n)}, Q_{\lambda}; i]$. In particular, $V_0=V_{\lambda}^{(n)}$ and $V_1=[V_0, Q_{\lambda}]=[V_{\mu}^{(n-1)}, Q_{\lambda}]V_2$. Notice that $V_{\lambda}^{(n)}\ne V_{\mu}^{(n-1)}$ and let $k$ be maximal such that $V_{\lambda}^{(n)}=V_{\mu}^{(n-1)}V_k$. Then $V_1=[V_{\mu}^{(n-1)}, Q_{\mu}]V_{k+1}\le V_{\mu}^{(n-1)}V_{k+1}$. But $V_{\lambda}^{(n)}=V_{\mu}^{(n-1)}V_1=V_{\mu}^{(n-1)}V_{k+1}$, contradicting the maximal choice of $k$. Thus, $O^p(L_{\lambda})$ does not centralize $V_{\lambda}^{(n)}/[V_{\lambda}^{(n)}, Q_{\lambda}]$, as required.
\end{proof}

We will use the following lemma often in the amalgam method and without reference. Recall also that if $U, V\normaleq G$ with $V<U$ then, in our setup and using coprime action, $U/V$ does not contain a non-central chief factor for $G$ if and only if $O^p(G)$ centralizes $U/V$.

\begin{lemma}\label{nccf}
For any $\lambda\in\Gamma$, $V_{\lambda}^{(n)}/V_{\lambda}^{(n-2)}$ contains a non-central chief factor for $L_{\lambda}$ for all $n\geq 2$ such that $V_{\lambda}^{(n)}\le Q_{\lambda}$.
\end{lemma}
\begin{proof}
Assume that $V_{\lambda}^{(n)}/V_{\lambda}^{(n-2)}$ contains only central chief factors for $L_{\lambda}$ so that $O^p(L_{\lambda})$ centralizes $V_{\lambda}^{(n)}/V_{\lambda}^{(n-2)}$. Since $V_{\lambda}^{(n-2)}<V_{\mu}^{(n-1)}<V_{\lambda}^{(n)}$ for all $\mu\in\Delta(\lambda)$, we have that $V_{\mu}^{(n-1)}\normaleq O^p(L_{\lambda})G_{\lambda, \mu}=G_{\lambda}$ by a Frattini argument. But then $V_{\mu}^{(n-1)}\normaleq \langle G_{\mu}, G_{\lambda}\rangle$, a contradiction. Thus, $V_{\lambda}^{(n)}/V_{\lambda}^{(n-2)}$ contains a non-central chief factor, as required.
\end{proof}

We now introduce some notation which is non-standard in the amalgam method and is tailored for our purposes. 

\begin{notation}
\begin{itemize}
\item If $Z_\delta\ne \Omega(Z(L_{\delta}))$, then $R_\delta=C_{L_{\delta}}(Z_{\delta})$.
\item If $Z_\delta=\Omega(Z(L_{\delta}))$ and $b>1$, then $R_\delta=C_{L_{\delta}}(V_\delta/C_{V_{\delta}}(O^p(L_{\delta})))$.
\item If $Z_\delta=\Omega(Z(L_{\delta}))$ and $b>1$, then $C_\delta=C_{Q_{\delta}}(V_\delta)$.
\end{itemize}
\end{notation}

\begin{lemma}\label{BasicVB}
Suppose that $Z_\delta=\Omega(Z(L_{\delta}))$, $b>1$ and let $T\in\syl_p(G_{\delta})$. Then $R_{\delta}\cap T\le Q_{\delta}$ and $C_T(V_{\delta})=C_\delta$. 
\end{lemma}
\begin{proof}
Suppose for a contradiction that $R_{\delta}\cap T\not\le Q_{\delta}$. Then $R_{\delta}$ is not $p$-closed and by \cref{p-closure} (iii), $G_{\delta}=R_{\delta}N_{G_{\delta}}(T)$. Let $\mu\in\Delta(\delta)$ with $T\in\syl_p(G_{\delta,\mu})$. Then $Z_{\mu}\le V_{\delta}$ so that $Z_{\mu}\normaleq \langle G_{\delta}, G_{\mu}\rangle$, a contradiction.

Suppose now that $C_T(V_{\delta})>Q_{\delta}$ so that $C_{G_{\delta}}(V_{\delta})$ is not $p$-closed and is normal in $G_{\delta}$. As above, by \cref{p-closure} (iii), we quickly get that $G_{\delta}=C_{G_{\delta}}(V_{\delta})G_{\delta,\mu}$ normalizes $Z_{\mu}$ for $\mu\in\Delta(\delta)$ with $T\in\syl_p(G_{\delta,\mu})$. Hence, the result.
\end{proof}

\begin{lemma}\label{spelemma2}
Suppose that there is $X\le L_{\delta}$ such that $\bar{X}$ is strongly $p$-embedded in $\bar{L_{\delta}}$ for $\delta\in\Gamma$. Then $XR_{\delta}/R_{\delta}$ is strongly $p$-embedded in $L_{\delta}/R_{\delta}$.
\end{lemma}
\begin{proof}
By \cref{p-closure2} and \cref{BasicVB}, we have that $\bar{R_{\delta}}$ is a $p'$-group. Then, if $XR_{\delta}/R_{\delta}$ is not strongly $p$-embedded in $L_{\delta}/R_{\delta}$, \cref{spelemma1} (iv) implies that $L_{\delta}=SR_{\delta}$. Thus, $Z_\delta=\Omega(Z(L_\delta))$ and it quickly follows that for any $\lambda\in\Delta(\delta)$, $Z_{\lambda}\normaleq L_{\delta}$, a contradiction.
\end{proof}

\begin{lemma}\label{NiceSL2}
Suppose that $L_{\delta}/R_{\delta}\cong\SL_2(p^n)$, $Q_{\delta}\in\syl_p(R_{\delta})$ and $R_{\delta}\le G_{\delta, \lambda}$ for some $\lambda\in\Delta(\delta)$. Then $\bar{L_{\delta}}\cong\SL_2(p^n)$.
\end{lemma}
\begin{proof}
Since $R_\delta\le G_{\delta, \lambda}$ and $R_\delta\normaleq G_\delta$, $\bar{R_\delta}\le Z(\bar{L_{\delta}})$ is a $p'$-group by \cref{p-closure2}. If $p^n>3$, then as $L_{\delta}=O^{p'}(L_{\delta})$, it follows from \cref{SLGen} (vii) that $\bar{L_{\delta}}\cong\SL_2(p^n)$.

If $L_{\delta}/R_{\delta}\cong\Sym(3)$ and $R_{\delta}\ne Q_{\delta}$, then $\bar{R_\delta}$ is a non-trivial $3$-group since $L_{\delta}=O^{2'}(L_{\delta})$ and for any prime $r\ne 2,3$, $O_r(\bar{R_{\delta}})$ is complemented in $\bar{L_{\delta}}$. But now, since $\bar{R_{\delta}}$ is maximal and central in $O_3(\bar{L_{\delta}})$, $O_3(\bar{L_{\delta}})$ is abelian. By coprime action, $O_3(\bar{L_{\delta}})=[O_3(\bar{L_{\delta}}), \bar{S}]\times C_{O_3(\bar{L_{\delta}})}(\bar{S})$ and $\bar{R_{\delta}}$ is complemented in $\bar{L_{\delta}}$ by $[O_3(\bar{L_{\delta}}), \bar{S}]\bar{S}\cong\Sym(3)$. Since $L_{\delta}=O^{2'}(L_{\delta})$ the result follows.

If $L_{\delta}/R_{\delta}\cong\SL_2(3)$ then $\bar{R_\delta}$ is a non-trivial $2$-group since $L_{\delta}=O^{3'}(L_{\delta})$ and for any prime $r\ne 2,3$, $O_r(\bar{R_{\delta}})$ is complemented in $\bar{L_{\delta}}$. Let $A$ be a maximal subgroup of $\bar{R_{\delta}}$. Then $|O_2(\bar{L_{\delta}})/A|=16$. By Gaschutz' theorem \cite[(3.3.2)]{kurz}, we may assume that $\bar{R_{\delta}}/A$ is not complemented in $O_2(\bar{L_{\delta}})/A$. We see that $O_2(\bar{L_{\delta}})/A$ is a non-abelian group of order $16$ with center of order at most $4$. Checking the Small Groups Library in MAGMA for groups of order $48$ with a quotient by a central involution isomorphic to $\SL_2(3)$ and a Sylow $2$-subgroup satisfying the required properties, we have a contradiction.
\end{proof}

\begin{lemma}\label{SimExt}
Suppose that $\delta\in\Gamma$, $Z_{\delta-1}=Z_{\delta+1}$, $Q_{\delta}\in\syl_p(R_\delta)$ and $i\in\N$. If $Q_{\delta-1}Q_{\delta}\in\syl_p(L_{\delta})$, $L_{\delta}/R_{\delta}$ is generated by any two distinct Sylow $p$-subgroups and $O^p(R_\delta)$ normalizes $V_{\delta-1}^{(i-1)}$, then $V_{\delta-1}^{(i-1)}=V_{\delta+1}^{(i-1)}$.
\end{lemma}
\begin{proof}
Since $Q_{\delta-1}Q_{\delta}\in\syl_p(L_{\delta})$, if $Q_{\delta-1}R_{\delta}\ne Q_{\delta+1}R_{\delta}$, then $Z_{\delta+1}=Z_{\delta-1}\normaleq L_{\delta}=\langle R_{\delta}, Q_{\delta-1}, Q_{\delta+1}\rangle$, a contradiction. Thus, $Q_{\delta-1}R_{\delta}=Q_{\delta+1}R_{\delta}$. As $Q_{\delta-1}Q_{\delta}\in\syl_p(Q_{\delta-1}R_{\delta})$, there is $r\in R_{\delta}$ such that $Q_{\delta-1}^rQ_{\delta}=(Q_{\delta-1}Q_{\delta})^r=(Q_{\delta+1}Q_{\delta})=Q_{\delta+1}Q_{\delta}$. Since $Q_{\delta-1}Q_{\delta}$ is the unique Sylow $p$-subgroup of $G_{\delta-1, \delta}$, it follows that $G_{\delta, \delta-1}^r=G_{\delta, \delta+1}=N_{G_\delta}(Q_{\delta}Q_{\delta+1})$. Set $\theta=(\delta-1)\cdot r\in\Delta(\delta)$. Then by properties of the graph, $G_{\delta, \delta+1}=G_{\delta, \delta-1}^r=G_{\delta, \delta-1\cdot r}=G_{\delta, \theta}$ and so $(\delta-1)\cdot r=\delta+1$. Since $r$ acts as a graph automorphism on $\Gamma$, $r$ preserves $i$ neighbourhoods of vertices in the graph and it follows immediately that $V_{\delta-1\cdot r}^{(i-1)}=(V_{\delta-1}^{(i-1)})^r$ so that, as $V_{\delta-1}^{(i-1)}$ is normalized by $R_{\delta}=O^p(R_\delta)Q_{\delta}$, $V_{\delta+1}^{(i-1)}=V_{\delta-1}^{(i-1)}$, completing the proof.
\end{proof}

We record one further generic lemma concerning the action of $R_{\gamma}$ for $\gamma\in\Gamma$.

\begin{lemma}\label{UWNormal}
Let $\gamma\in\Gamma$ and fix $\delta\in\Delta(\gamma)$. Then for $n<b$, $\langle V_{\mu}^{(n)} \mid Z_{\mu}=Z_{\delta}, \mu\in\Delta(\gamma)\rangle\normaleq R_{\gamma}Q_{\delta}$. 
\end{lemma}
\begin{proof}
Set $U^\gamma:=\langle V_{\mu}^{(n)} \mid Z_{\mu}=Z_{\delta}, \mu\in\Delta(\gamma)\rangle$ and let $r\in R_{\gamma}Q_{\delta}$. Since $r$ is a graph automorphism, for $\mu\in\Delta(\gamma)$ such that $Z_{\mu}=Z_{\delta}$, $(V_{\mu}^{(n)})^r=V_{\mu \cdot r}^{(n)}$. But now, $Z_{\mu\cdot r}=Z_{\mu}^r=Z_{\delta}^r=Z_{\delta}$ and so $(V_{\mu}^{(n)})^r\le U^{\gamma}$. Thus, $U^\gamma\normaleq  R_{\gamma}Q_{\delta}$, as required.
\end{proof}

We now deal with the so called ``pushing up'' case of the amalgam method. The proof breaks up over a series of lemmas, culminating in \cref{push}. This is a generalization of \cite[(6.5-6.6)]{Greenbook} and as the result there underpinned the generic study of higher rank amalgams in characteristic $p$ (see for instance \cite{Rank3Amalgams}), we believe \cref{push} will be applicable in future works on fusion systems and amalgams. Throughout, let $\lambda\in\Gamma$, $\mu\in\Delta(\lambda)$ and $S\in\syl_p(G_{\lambda,\mu})$.

\begin{lemma}
Suppose that $Q_\lambda\cap Q_{\mu}\normaleq G_\lambda$. Then, writing $L:=\langle Q_{\mu}^{G_{\lambda}} \rangle$, we have that $Q_{\mu}\in\syl_p(L)$, $O_p(L)=Q_{\mu}\cap Q_{\lambda}$, $Z_{\lambda}/Z(L_{\lambda})$ is a natural $\SL_2(q)$-module for $L_{\lambda}/R_{\lambda}$ and no non-trivial characteristic subgroup of $Q_{\mu}$ is normal in $L$.
\end{lemma}
\begin{proof}
Set $L:=\langle Q_\mu^{G_\lambda} \rangle\normaleq L_\lambda$ and let $V:=Z_{\lambda}$ if $Z_{\lambda}\ne \Omega(Z(S))$, and $V:=V_{\lambda}/C_{V_\lambda}(O^p(L_\lambda))$ if $Z_{\lambda}=\Omega(Z(S))$ and $b>1$. Since $L\normaleq L_\lambda$, we have that $C_L(O_p(L))\le O_p(L)$ and since $Q_\mu\not\le Q_\lambda$, it follows by \cref{p-closure} that $L/O_p(L)$ has a strongly $p$-embedded subgroup and $L_{\lambda}=LS$. Note that if $J(Q_{\mu})\le O_p(L)$, then $J(Q_{\mu})\le Q_{\mu}\cap Q_{\lambda}\le Q_{\mu}$ and so, by \cref{BasicJS} (v), $J(Q_{\mu})=J(Q_{\mu}\cap Q_{\lambda})\normaleq L_{\lambda}$, a contradiction. 

Suppose first that $b=1$ and $Z_{\lambda}=\Omega(Z(S))$. Then $Z_{\lambda}\le Q_{\mu}$ and we may as well assume that $Z_{\mu}\not\le Q_{\lambda}$. But then $Z_{\mu}$ centralizes $Q_{\lambda}/Q_{\lambda}\cap Q_{\mu}$ and $Q_\lambda\cap Q_\mu$. Since $\langle Z_{\mu}^{G_{\lambda}}\rangle$ contains elements of $p'$-order, using coprime action and that $G_{\lambda}$ is of characteristic $p$, we have a contradiction. 

Now, if $V:=Z_{\lambda}$, then $O_p(L)=C_{S\cap L}(V)$ and by \cref{BasicFF} and \cref{SEFF}, $L/C_L(V)\cong\SL_2(q)$. If $Z_{\lambda}=\Omega(Z(S))$, then $Q_{\lambda}\cap Q_{\mu}=C_{\lambda}$ and we may assume that $\mu$ belongs to a critical pair $(\mu, \mu')$ with $d(\lambda, \mu)=b-1$. Then $b$ is odd, otherwise $\mu'-1\in\lambda^G$ and $Z_{\mu}\le Q_{\mu'-1}\cap Q_{\mu'-2}= Q_{\mu'-1}\cap Q_{\mu'}\le Q_{\mu'}$. Thus, $V_{\mu'}\cap Q_{\lambda}\le C_{\lambda}$ and $V_{\lambda}\cap Q_{\mu'}\le C_{\mu'}$. Without loss of generality, assume that $|V_{\mu'}/(V_{\mu'}\cap Q_{\lambda})|\leq |V_{\lambda}/(V_{\lambda}\cap Q_{\mu'})|$. A straightforward calculation ensures that $V_{\lambda}Q_{\mu'}/Q_{\mu'}$ is an offender on $V_{\mu'}/[V_{\mu'}, Q_{\mu'}]$, $[V_{\mu'}, Q_{\mu'}]\le C_{V_{\mu'}}(O^p(L_{\mu'}))$ and by \cref{SEFF}, $L_{\mu'}/C_{L_{\mu'}}(V_{\mu'}/C_{V_{\mu'}}(O^p(L_{\mu'})))\cong\SL_2(q)$. 

Either way, it follows from \cref{SLGen} (vi) that $L_{\lambda}/C_{L_{\lambda}}(V)\cong L/C_L(V)\cong\SL_2(q)$, $S=Q_\lambda Q_\mu$ and \[Q_\mu O_p(L)=Q_\mu(Q_\lambda\cap L)=Q_\lambda Q_\mu\cap L=S\cap L\in\syl_p(L).\] Since $[O_p(L), Q_\mu]\le [Q_\lambda, Q_\mu]\le Q_\lambda\cap Q_\mu\le O_p(L)$ it follows that $[O_p(L), L]=[O_p(L), \langle Q_\mu^{L_\lambda}\rangle]=[O_p(L), Q_\mu]^{L_\lambda}\le Q_\lambda \cap Q_\mu$ and so $\hat{L}:=L/(Q_\lambda \cap Q_\mu)$ is a central extension of $L/O_p(L)$ by $\hat{O_p(L)}$. But $Q_\mu\cap O_p(L)=Q_\mu\cap Q_\lambda$ and so $\hat{Q_\mu}$ is complement to $\hat{O_p(L)}$ in $\hat{S\cap L}$. It follows by Gaschutz' theorem \cite[(3.3.2)]{kurz} that there is a complement in $\hat{L}$ to $\hat{O_p(L)}$. Now, letting $K_{\lambda}$ be a Hall $p$'-subgroup of $N_L(S\cap L)$, unless $q\in\{2,3\}$, we deduce that $\hat{Q_{\mu}}\le [\hat{S\cap L}, K_{\lambda}]$ is contained in a complement to $\hat{O_p(L)}$ and since $L=\langle Q_{\mu}^{G_{\lambda}}\rangle$, it follows that $\hat{O_p(L)}=\{1\}$ and $Q_{\mu}\in\syl_p(L)$. If $q\in\{2,3\}$, then $\hat{L}\cong p\times \SL_2(p)$, $|\hat{Q_{\mu}}|=p$ and one can check that $\langle \hat{Q_{\mu}}^{\hat{L}}\rangle\cong \SL_2(p)$, contradicting the initial definition of $L$. Thus $Q_{\mu}\in\syl_p(L)$ and $O_p(L)=Q_{\mu}\cap Q_{\lambda}$. Since $L_{\lambda}=LQ_{\lambda}$, there is no non-trivial characteristic subgroup of $Q_{\mu}$ which is normal in $L$, for such a subgroup would then be normal in $\langle G_{\lambda}, G_{\mu}\rangle$.

It remains to show that $V:=Z_{\lambda}$ so suppose that $Z_{\lambda}=\Omega(Z(S))$ and $V=V_{\lambda}/C_{V_\lambda}(O^p(L_\lambda))$. Moreover, $Z(L_{\mu})=\{1\}$ by \cref{crit pair} (iv), $O_p(L)=C_\lambda$, $b>1$ is odd and $V_{\lambda}$ is abelian. Let $R_L$ be the preimage in $L$ of $O_{p'}(L/O_p(L))$ and suppose that $R_L$ is not a $p$-group. Then $V_{\lambda}=[V_{\lambda}, R_L]\times C_{V_{\lambda}}(R_L)$ is an $S$-invariant decomposition, and since $Z_{\lambda}=\Omega(Z(S))\le C_{V_{\lambda}}(R_L)$, $V_{\lambda}$ is centralized by $R_L$. Since $V_{\lambda}$ is an FF-module for $L/O_p(L)$, unless $q=2^n>2$, using coprime action and \cref{NatMod} (v) we infer that $C_{V_{\lambda}}(R_L)=C_{V_{\lambda}}(O^p(L))$ so that $Z_{\mu}$ is centralized by $L$ and normalized by $\langle L, G_{\mu}\rangle$, a contradiction. 

Suppose now that $q=2^n>2$ and let $K'\in (G_{\lambda, \mu}\cap L)/O_2(L)$ complementing $Q_\mu/O_2(L)$. Then for $K$ the preimage if $K'$ in $L$, we have that $[K, Q_{\lambda}]\le O_2(L)$. In particular, $[V_\lambda, K]$ is normalized by $Q_\lambda$ and as $V_\lambda=[V_\lambda, K]\times C_{V_\lambda}(K)$ by coprime action, $[V_\lambda, K, Q_\lambda]=\{1\}$ since $C_{V_{\lambda}}(K)=C_{V_{\lambda}}(L)$ and $V_\lambda/C_{V_\lambda}(L)$ is a natural $\SL_2(2^n)$-module for $L/R_L\cong\PSL_2(2^n)$. But then $[Z_\mu, K, Q_{\lambda}]=\{1\}$ so that $\{1\}\ne [Z_\mu, K]$ is centralized by $S=Q_\lambda Q_\mu$ and $[Z_\mu, K]\le Z_\lambda$. Hence, by coprime action $[Z_\mu, K, K]=[Z_\lambda, K]=\{1\}$, a contradiction.
\end{proof}

\begin{lemma}
Suppose that $Q_\lambda\cap Q_{\mu}\normaleq L_\lambda$. Then $b>1$ and, writing $L:=\langle Q_{\mu}^{L_{\lambda}} \rangle$, $L/O_p(L)\cong L_{\lambda}/Q_{\lambda}\cong \SL_2(q)$, $b=2$ and $O_p(L)$ contains a unique non-central chief factor for $L$. Moreover, there is $\lambda'\in\Delta(\mu)$ such that both $(\lambda, \lambda')$ and $(\lambda', \lambda)$ are critical pairs.
\end{lemma}
\begin{proof}
Suppose that $b=1$. Then $\Omega(Z(S))\le Q_{\lambda}\cap Q_{\mu}=O_p(L)\normaleq G_{\lambda}$ and it follows from the definition of $Z_{\lambda}$ that $Z_{\lambda}\le O_p(L)\le Q_{\mu}$. Thus, we may as well assume that $Z_{\mu}\not\le Q_{\lambda}$. But then $Z_{\mu}$ centralizes $O_p(L)$ and so $O^p(L)$ centralizes $O_p(L)$, a contradiction since $L$ is of characteristic $p$. Thus, we conclude that $b>1$.

Suppose that $(\lambda, \delta)$ is not a critical pair for any $\delta\in\Gamma$. Then there is some $\mu'$ such that $(\mu, \mu')$ is a critical pair and $d(\lambda, \mu')=b-1$. Then $Z_{\mu}\ne \Omega(Z(S))\ne Z_{\lambda}$, $C_{G_{\mu'}}(Z_{\mu'})$ is $p$-closed and $Z_{\mu'}\le Q_{\mu+2}\cap Q_{\lambda}=Q_{\lambda}\cap Q_{\mu}$. But then, $[Z_{\mu}, Z_{\mu'}]=\{1\}$, a contradiction for then $Z_{\mu}\le Q_{\mu'}$. Thus, we may assume $\lambda$ belongs to a critical pair $(\lambda, \lambda')$ with $d(\mu, \lambda')=d(\lambda, \lambda')-1$. Suppose that $b$ is odd. Then $Z_{\lambda}\le Q_{\lambda'-1}$ and $\lambda'-1\in\lambda^G$. But then $Z_{\lambda}\le Q_{\lambda'-1}\cap Q_{\lambda'-2}=Q_{\lambda'-1}\cap Q_{\lambda'}\le Q_{\lambda'}$, a contradiction. Thus, $b$ is even. Moreover, since $C_S(Z_{\lambda})=Q_{\lambda}\in\syl_p(G_{\lambda}^{(1)})$ and $[Z_{\lambda}, Z_{\lambda'}]\ne\{1\}$, $(\lambda', \lambda)$ is also a critical pair. Suppose that $b\geq 4$. Then $V_{\lambda}^{(2)}\le O_p(L)$ and $V_{\lambda}^{(2)}/Z_{\lambda}$ contains a non-central chief factor. Hence, if $O_p(L)$ contains a unique non-central chief factor for $L$ then $b=2$.

Suppose that $O_p(L)$ contains more than one non-central chief factor for $L$ and assume that $p$ is odd. If $b=2$, then $O_2(L)=Q_{\lambda}\cap Q_{\mu}=Z_{\lambda}(Q_{\lambda}\cap Q_{\mu}\cap Q_{\lambda'})$, a contradiction since $O_p(L)$ contains more than one non-central chief factor. Thus, we may assume that $b\geq 4$ and $b$ is even. Set $T_{\lambda}$ to be a Hall $p'$-subgroup of the preimage in $L_{\lambda}$ of $Z(L_{\lambda}/R_{\lambda})$. Note also that since $p$ is odd, we may apply coprime action along with \cref{NatMod} (v) so that $Z_{\lambda}=[Z_{\lambda}, T_{\lambda}]\times C_{Z_{\lambda}}(T_{\lambda})=[Z_\lambda, L_\lambda]\times Z(L_\lambda)$.

Choose $\lambda-1\in\Delta(\lambda)$ such that $\Omega(Z(L_{\lambda-1}))\ne \Omega(Z(L_\mu))$ and set $U=\langle V_\gamma \mid \Omega(Z(L_{\lambda-1}))=\Omega(Z(L_\gamma)), \gamma\in\Delta(\lambda)\rangle$. Let $r\in R_{\lambda}Q_{\lambda-1}\le C_{L_{\lambda}}(\Omega(Z(L_{\lambda-1})))$. Since $r$ is an automorphism of the graph, it follows that for $V_{\gamma}\le U$, $V_{\gamma}^r=V_{\gamma\cdot r}$. But $\Omega(Z(L_{\gamma\cdot r}))=\Omega(Z(L_{\gamma}))^r=\Omega(Z(L_{\lambda-1}))^r=\Omega(Z(L_{\lambda-1}))$ and so $V_{\gamma}^r\le U$ and $U\normaleq R_{\lambda}Q_{\lambda-1}$. Note that if $U\le Q_{\lambda'-2}$ then $U\le Q_{\lambda'-2}\cap Q_{\lambda'-3}=Q_{\lambda'-2}\cap Q_{\lambda'-1}\le Q_{\lambda'-1}$ and so, $U=Z_{\lambda}(U\cap Q_{\lambda'})$. Thus, $Z_{\lambda'}$ centralizes $U/Z_{\lambda}$ and since $L_{\lambda}=\langle R_{\lambda}, Z_{\lambda'}, Q_{\lambda-1}\rangle$, it follows that $O^p(L_{\lambda})$ centralizes $U/Z_{\lambda}$ and so normalizes $V_{\lambda-1}$, a contradiction.

Therefore, $U\not\le Q_{\lambda'-2}$ so that there is some $\lambda-2\in\Delta^{(2)}(\lambda)$ such that $(\lambda-2, \lambda'-2)$ is also a critical pair. Since $Z_\lambda=[Z_\lambda, L_\lambda]\times \Omega(Z(L_\lambda))$, it suffices to prove that $[Z_{\lambda}, Z_{\lambda'}]=\Omega(Z(L_{\mu}))=\Omega(Z(L_{\lambda'-1}))$ and that this holds for any critical pair, since then, as there $\lambda-2\in\Delta(\lambda-1)$ with $(\lambda-2, \lambda'-2)$ a critical pair, $Z_{\lambda}=\Omega(Z(L_{\lambda'-1}))\times \Omega(Z(L_{\lambda'-3}))\times \Omega(Z(L_\lambda))$ which is contained in $Q_{\lambda'}$ since $b>2$. 

Suppose that $Z_{\mu}=\Omega(Z(S))=\Omega(Z(L_{\mu}))$. In particular, $Z(L_{\lambda})=\{1\}$ and $Z_{\lambda}$ is irreducible. Since $Z_{\lambda}$ is a natural $\SL_2(q)$-module, $Z_{\lambda'-1}=[Z_{\lambda}, Z_{\lambda'}]=Z_{\mu}$, as required.

Assume now that $Z_{\mu}\ne \Omega(Z(S))$. Then $Z_{\lambda}=[Z_{\lambda}, T_\lambda]\times C_{Z_{\lambda}}(T_\lambda)$, $[Z_{\lambda}, T_{\lambda}]=[Z_{\lambda}, L_{\lambda}]$ and $C_{Z_{\lambda}}(T_\lambda)=\Omega(Z(L_{\lambda}))$. Moreover, $[Z_{\lambda}, Z_{\lambda'}]=C_{[Z_{\lambda}, L_{\lambda}]}(S)=\Omega(Z(S))\cap [Z_{\lambda}, L_{\lambda}]$. Since $\Omega(Z(S))=\Omega(Z(L_{\lambda}))\times \Omega(Z(L_{\mu}))$ and $T_\lambda$ normalizes $\Omega(Z(L_{\mu}))$, we have that $\Omega(Z(L_{\mu}))\geq[\Omega(Z(L_{\mu})), T_\lambda]=[\Omega(Z(S)), T_\lambda]=\Omega(Z(S))\cap [Z_{\lambda}, L_{\lambda}]$. Comparing orders, we conclude that $\Omega(Z(L_{\mu}))=[\Omega(Z(S)), T_\lambda]=[Z_{\lambda}, Z_{\lambda'}]$. By symmetry, we have that $Z(L_{\lambda'-1})=[Z_{\lambda}, Z_{\lambda'}]$, as required.

Suppose now that $p=2$ and $O_2(L)$ contains more than one non-central chief factor within $O_2(L)$. Choose $1<m<b/2$ minimal such that $V_{\lambda}^{(2m)}\le Q_{\lambda'-2m}$. Notice by the minimal choice of $m$ that $V_{\lambda}^{(2(m-k))}Q_{\lambda'-2(m-k)}\in\syl_p(L_{\lambda'-2(m-k)})$ for all $1\leq k\leq m$. Then $V_{\lambda}^{(2m)}\le Q_{\lambda'-2m}\cap Q_{\lambda'-2m-1}\le Q_{\lambda'-2m+1}$ and, extending further, $V_{\lambda}^{(2m)}=V_{\lambda}^{(2m-2)}(V_{\lambda}^{(2m)}\cap Q_{\lambda'})$. But then, $O^p(L_{\lambda})$ centralizes $V_{\lambda}^{(2m)}/V_{\lambda}^{(2m-2)}$, a contradiction. Thus, no such $m$ exists. Even still an index $q$ subgroup of $V_{\lambda}^{(2k)}/V_{\lambda}^{(2k-2)}$ is centralized by $Z_{\lambda'}$ for all $k<b/2$ and it follows that for all $1<m<b/2$, $V_{\lambda}^{(2m)}/V_{\lambda}^{(2m-2)}$ contains a unique non-central chief factor and this factor is an FF-module for $L_{\lambda}/Q_{\lambda}$. Note that for $R_1, R_2$ the centralizers in $L/O_2(L)$ of distinct non-central chief factors in $V_{\lambda}^{(2m)}$ for $1<m<b/2$, we deduce that $R_1R_2/R_i$ is an odd order normal subgroup of $L_i/R_i\cong \SL_2(q)$ for $i\in\{1,2\}$. Thus, unless $q=2$, we have that $L/O_2(L)C_{L}(V_{\lambda}^{(2m)})\cong\SL_2(q)$ and an application of the three subgroup lemma ensures that $L/O_2(L)\cong\SL_2(q)$. 

Since no non-trivial characteristic subgroup of $Q_{\beta}$ is normal in $L$, we may apply pushing up arguments from \cite[Theorem B]{nil} when $L/O_2(L)\cong\SL_2(q)$. Thus, $Q_{\mu}$ has class $2$ and there is a unique non-central chief factor for $L$ within $O_2(L)$. It is clear that $Z_{\lambda}/Z(L_{\lambda})$ is the unique non-central chief factor for $L$ inside $O_2(L)$ and is isomorphic to the natural module for $L/O_2(L)\cong\SL_2(q)$. Thus, $q=p=2$ and since no non-trivial characteristic subgroup of $Q_{\beta}$ is normal in $L$, we may apply \cite[Theorem 4.3]{GlaubermanIso} to see that $Q_{\mu}$ has nilpotency class $2$ and exponent $4$. Notice that if $b\geq 4$, then $V_{\lambda}^{(2)}$ is contained in $Q_{\mu}$ and $[V_{\lambda}^{(2)},Q_{\mu}]\le \Omega(Z(Q_{\mu}))$. But $\langle (\Omega(Z(Q_{\mu}))^L)\rangle$ is an FF-module for $L/O_2(L)$ by \cref{BasicFF}, and contains $[Z_{\lambda}, L_{\lambda}]$ as its unique non-central chief factor. Thus, it follows that $[V_{\lambda}^{(2)},L]\le Z_{\lambda}$ and $V_{\mu}\normaleq \langle L, G_{\mu}\rangle$, a contradiction. Hence, we conclude that $b=2$ so that $O_2(L)$ contains a unique non-central chief factor, as required.
\end{proof}

\begin{lemma}
Suppose that $Q_\lambda\cap Q_{\mu}\normaleq L_\lambda$. Then $Z_{\mu}\ne \Omega(Z(S))$.
\end{lemma}
\begin{proof}
We suppose throughout that there is a unique non-central chief factor for $L_{\lambda}$ contained in $Q_{\mu}\cap Q_{\lambda}$ and, as a consequence, that $L/O_p(L)\cong L_{\lambda}/Q_{\lambda}\cong \SL_2(q)$. Additionally, assume that $Z_{\mu}=\Omega(Z(S))=\Omega(Z(L_{\mu}))$. Then $Z(L_{\lambda})=\{1\}$ by \cref{crit pair} (iv). Hence, $Z_{\lambda}$ is the unique non-central chief factor within $Q_{\lambda}\cap Q_{\mu}$. In particular, $Z_\lambda$ is isomorphic to a natural $\SL_2(q)$-module and $[O^p(L_\lambda), Q_{\lambda}]=Z_{\lambda}$.

If $\Phi(Q_{\lambda})\ne \{1\}$, then the irreducibility of $Z_{\lambda}$ implies that $Z_{\lambda}\le \langle (\Phi(Q_{\lambda})\cap \Omega(Z(S)))^{L_{\lambda}}\rangle\le \Phi(Q_{\lambda})$. But then $O^p(L)$ acts trivially on $Q_{\lambda}/\Phi(Q_{\lambda})$, a contradiction by coprime action. Thus, $\Phi(Q_{\lambda})=\{1\}$ and $Q_{\lambda}$ is elementary abelian. If $p$ is odd or $q=2$, then for $T_\lambda$ the preimage in $L_{\lambda}$ of $O_{p'}(\bar{L_{\lambda}})$, we have that $Q_{\lambda}=[Q_{\lambda}, T_\lambda]\times C_{Q_{\lambda}}(T_\lambda)=Z_{\lambda}\times C_{Q_{\lambda}}(T_\lambda)$ is an $S$-invariant decomposition and since $\Omega(Z(S))\le Z_{\lambda}$, we have that $C_{Q_{\lambda}}(T_\lambda)=\{1\}$ and $Q_{\lambda}=Z_{\lambda}$. But then $Z_{\mu}=Z_{\lambda}\cap Q_{\mu}=Q_{\lambda}\cap Q_{\mu}\normaleq L_{\lambda}$, a contradiction.

If $q>2$ is even, then since $S\le N_{G_\mu}(O_2(L))$, we have that $[G:N_{G_\mu}(O_2(L))]$ is odd and applying \cite[Theorem 3]{StellmacherPush}, $V_{\mu}\normaleq G=\langle L, G_{\mu}\rangle$, a contradiction.
\end{proof}

\begin{proposition}\label{push}
Let $S\in\syl_p(G_\lambda\cap G_\mu)$ for $\lambda\in\Gamma$ and $\mu\in\Delta(\lambda)$. Then $Q_\lambda\cap Q_\mu$ is not normal in $L_\lambda$. Moreover, if $Z_\lambda Z_\mu\normaleq L_\lambda$ then $Z_\mu=\Omega(Z(S))\le Z_\lambda$.
\end{proposition}
\begin{proof}
Suppose that $Z_\lambda Z_\mu\normaleq L_\lambda$ but $Z_\mu\ne\Omega(Z(S))$. By \cref{p-closure2} (ii), we have that $C_S(Z_\mu)=Q_\mu$ and so $C_{Q_\lambda}(Z_\lambda Z_\mu)=Q_\lambda\cap C_S(Z_\mu)=Q_\lambda\cap Q_\mu$ and it follows that $Q_\lambda\cap Q_\mu\normaleq L_\lambda$. Thus, we may suppose that $Q_\lambda\cap Q_\mu\normaleq L_\lambda$, and derive a contradiction to complete the proof.

Under this assumption, $Z_{\lambda}$ contains the unique non-central chief factor for $L$ inside $Q_{\mu}\cap Q_{\lambda}$ and $Z_{\mu}\ne \Omega(Z(S))$. Moreover, $b=2$ and there is $\lambda'\in\Delta(\mu)$ such that $Z_{\lambda}\not\le Q_{\lambda'}$ and $Z_{\lambda'}\not\le Q_{\lambda}$. Since $L_{\lambda}/Q_{\lambda}\cong\SL_2(q_\lambda)$ and $Z_{\lambda}/Z(L_{\lambda})$ is a natural module, we get that $Q_{\mu}=(Q_{\lambda'}\cap Q_{\mu}\cap Q_{\lambda})Z_{\lambda}Z_{\lambda'}$ and $Q_{\lambda}\cap Q_{\mu}=(Q_{\lambda'}\cap Q_{\mu}\cap Q_{\lambda})Z_{\lambda}$. Then $(Q_{\lambda}\cap Q_{\mu})/\Phi(Q_{\lambda'}\cap Q_{\mu}\cap Q_{\lambda})$ is elementary abelian and it follows that $\Phi(Q_{\lambda}\cap Q_{\mu})=\Phi(Q_{\lambda'}\cap Q_{\mu})=\Phi(Q_{\lambda'}\cap Q_{\mu}\cap Q_{\lambda})$. Set $F:=\Phi(Q_{\lambda}\cap Q_{\mu})$. Since $Q_{\lambda}$ contains a unique non-central chief factor for $L_{\lambda}$, we infer that $F$ is centralized by $O^p(L)$ and as $Q_{\mu}$ has class $2$, $F\le Z(L)$. Let $Z_{\mu}^*$ be the preimage in $Q_{\mu}$ of $Z(Q_{\mu}/F)$. Since $F$ is normal in both $G_{\lambda}$ and $G_{\lambda'}$, we have that $Z_{\mu}^*\normaleq \langle G_{\lambda, \mu}, G_{\mu, \lambda'}\rangle$. Moreover, since $Q_{\mu}=(Q_{\lambda'}\cap Q_{\mu}\cap Q_{\lambda})Z_{\lambda}Z_{\lambda'}$, we have that $Q_{\mu}\cap Q_{\lambda}\cap Q_{\lambda'}\le Z_{\mu}^*$. Since $[Z_{\mu}^*,Z_{\lambda}]\le F\le Z(L)$, we have that $Z_{\mu}^*\le Q_{\lambda}$ and by symmetry, $Z_{\mu}^*=Q_{\mu}\cap Q_{\lambda}\cap Q_{\lambda'}$.

Suppose that $p$ is odd and let $H_{\lambda,\mu}$ be a Hall $p'$-subgroup of $G_{\lambda,\mu}\cap L_{\lambda}$. By \cref{SLGen} (vi), $H_{\lambda,\mu}$ is cyclic of order $q_{\lambda}-1$. Furthermore, $H_{\lambda,\mu}$ normalizes $Q_{\mu}, F$ and $Z_{\mu}^*$ and acts non-trivially on $Q_{\mu}/Z_{\mu}^*$. Now, for $t_\lambda$ the unique involution in $H_{\lambda,\mu}$, $t_{\lambda}$ centralizes $Q_{\mu}/Q_{\lambda}\cap Q_{\mu}$ and inverts $Q_{\lambda}\cap Q_{\mu}/Z_{\mu}^*=Z_{\lambda}Z_{\mu}^*/Z_{\mu}^*$. By coprime action, $Q_{\mu}/Z_{\mu}^*=Z_{\lambda}Z_{\mu}^*/Z_{\mu}^*\times C_{Q_{\mu}/Z_{\mu}^*}(t_{\lambda})$ is a $Q_{\mu}$-invariant decomposition. Since $[S, t_{\lambda}]\le Q_{\lambda}\cap Q_{\mu}$ the previous decomposition is $S$-invariant. But then $[Q_{\lambda}, C_{Q_{\mu}/Z_{\mu}^*}(t_{\lambda})]\le (Q_{\mu}\cap Q_{\lambda})/Z_{\mu}^*=Z_{\lambda}Z_{\mu}^*/Z_{\mu}^*$ and we deduce that $Q_{\lambda}$ centralizes $Q_{\mu}/Z_{\mu}^*$. Hence, $Q_{\lambda}$ normalizes $Q_{\lambda'}\cap Q_{\mu}$. Let $M=\langle Q_{\lambda}, Q_{\lambda'}, Q_{\mu}\rangle\le G_{\mu}$. Then there is an $m\in M$ such that $(Q_{\lambda}Q_{\mu})^m=Q_{\lambda'}Q_{\mu}$ and since $Q_{\lambda'}Q_{\mu}$ is the unique Sylow $p$-subgroup of $G_{\mu, \lambda'}$, it follows that $\lambda\cdot m=\lambda'$. But then $(Q_{\lambda}\cap Q_{\mu})^m=Q_{\lambda'}\cap Q_{\mu}$ and as $M$ normalizes $Q_{\lambda'}\cap Q_{\mu}$, we have that $Q_{\mu}\cap Q_{\lambda}=Q_{\lambda}\cap Q_{\mu}$, absurd since $Z_{\lambda}\le Q_{\lambda}\cap Q_{\mu}$. 

Suppose that $p=2$. Since $(Q_{\lambda}\cap Q_{\mu})/F$ and $(Q_{\mu}\cap Q_{\lambda'})/F$ are elementary abelian, by \cite[Lemma 2.29]{parkerSymp}, every involution in $Q_{\mu}/F$ is contained in $(Q_{\lambda}\cap Q_{\mu})/F$ or $(Q_{\mu}\cap Q_{\lambda'})/F$. Indeed, for $A$ any other elementary abelian subgroup of $Q_{\mu}/F$ and $B$ the preimage of $A$ in $Q_{\mu}$, we must have that $B=(B\cap Q_{\lambda})\cup (B\cap Q_{\lambda'})$. If $B\not\le Q_{\lambda}$, then $F\cap Z_{\lambda}=C_{Z_{\lambda}}(B)=Z_{\lambda}\cap B$ and it follows that $B\cap Q_{\lambda}=F$. By symmetry, we have shown that $\mathcal{A}(Q_{\mu}/F)=\{(Q_{\lambda}\cap Q_{\mu})/F, (Q_{\mu}\cap Q_{\lambda'})/F\}$.

Set $M=\langle Q_{\lambda}, Q_{\lambda'}, Q_{\mu}\rangle\le G_{\mu}$ so that $M$ normalizes $Q_{\mu}$, $Z_{\mu}^*$ and $F$. Thus, all elements of $M$ which do not normalize $Q_{\mu}\cap Q_{\lambda}$, conjugate $Q_{\mu}\cap Q_{\lambda}$ to $Q_{\mu}\cap Q_{\lambda}$, and vice versa. Thus all odd order elements normalize $Q_{\mu}\cap Q_{\lambda}$. There is an $m\in M$ such that $(Q_{\lambda}Q_{\mu})^m=Q_{\lambda'}Q_{\mu}$ and since $Q_{\lambda'}Q_{\mu}$ is the unique Sylow $2$-subgroup of $G_{\mu, \lambda'}$, it follows that $\lambda\cdot m=\lambda'$. Since $M=O^p(M)Q_{\lambda}Q_{\mu}$, we may as well choose $m$ of order coprime to $p$. But then $(Q_{\lambda}\cap Q_{\mu})^m=Q_{\lambda'}\cap Q_{\mu}$ and as $m$ normalizes $Q_{\lambda'}\cap Q_{\mu}$, we conclude that $Q_{\mu}\cap Q_{\lambda}=Q_{\lambda}\cap Q_{\mu}$, a final contradiction since $Z_{\lambda}\le Q_{\lambda}\cap Q_{\mu}$. 
\end{proof}

We can now prove a result analogous to \cref{nccf}, instead working ``down" through chief factors. Again, we will apply this lemma often and without reference throughout this work.

\begin{lemma}
Let $\lambda\in\Gamma$ and $\mu\in\Delta(\lambda)$, $b>1$ and $n\geq 2$. If $V_{\lambda}^{(n)}\le Q_{\lambda}$, then $C_{Q_{\lambda}}(V_{\lambda}^{(n-2)})/C_{Q_{\lambda}}(V_{\lambda}^{(n)})$ contains a non-central chief factor for $L_{\lambda}$.
\end{lemma}
\begin{proof}
Observe that as $V_{\lambda}^{(n)}\le Q_{\lambda}$, we have that $Z(Q_{\lambda})\le  C_{Q_{\lambda}}(V_{\lambda}^{(n)})\le C_{Q_{\lambda}}(V_{\mu}^{(n-1)})\le C_{Q_{\lambda}}(V_{\lambda}^{(n-2)})$. In particular, $C_{Q_{\lambda}}(V_{\mu}^{(n-1)})$ is non-trivial. If $C_{Q_{\lambda}}(V_{\lambda}^{(n-2)})/C_{Q_{\lambda}}(V_{\lambda}^{(n)})$ contains only central chief factors for $L_{\lambda}$, $O^p(L_{\lambda})$ centralizes $C_{Q_{\lambda}}(V_{\lambda}^{(n-2)})/C_{Q_{\lambda}}(V_{\lambda}^{(n)})$ and normalizes $C_{Q_{\lambda}}(V_{\mu}^{(n-1)})$. Thus, $C_{Q_{\lambda}}(V_{\mu}^{(n-1)})\normaleq O^p(L_{\lambda})G_{\lambda,\mu}=G_{\lambda}$. In order to force a contradiction, we need only show that $C_{Q_{\lambda}}(V_{\mu}^{(n-1)})=C_{Q_{\mu}}(V_{\mu}^{(n-1)})$. 

Let $S\in\syl_p(G_{\lambda,\mu})$. Since $n\geq 2$, $Z_{\lambda}\le V_{\lambda}^{(n-2)}$ is centralized by $C_S(V_{\mu}^{(n-1)})$ and unless $n=2$ and $V_{\lambda}^{(n-2)}=Z_{\lambda}=\Omega(Z(S))$, applying \cref{p-closure2} (ii) and \cref{BasicVB}, we have that $C_S(V_{\mu}^{(n-1)})\le Q_{\lambda}\cap Q_{\mu}$ and $C_{Q_{\lambda}}(V_{\mu}^{(n-1)})=C_{Q_{\mu}}(V_{\mu}^{(n-1)})$, as desired. If $V_{\lambda}^{(n-2)}=\Omega(Z(S))$, then $V_{\mu}^{(n-1)}=Z_{\mu}$ and $C_S(Z_{\mu})=Q_{\mu}$. But then, $C_{Q_{\lambda}}(Z_{\mu})=Q_{\lambda}\cap Q_{\mu}\normaleq G_{\lambda}$, a contradiction by \cref{push}.
\end{proof}

We will also makes use of the qrc-lemma, although where it is applied there are certainly more elementary arguments which would suffice. In this way, we do not use the lemma in its full capacity and instead, it serves as a way to reduce the length of some of our arguments. This lemma first appeared in \cite{Stellmacherqrc} but only for the prime $2$. We use the extension to all primes presented in \cite[Theorem 3]{Strothqrc}.

\begin{theorem}[qrc Lemma]
Let $(H,M)$ be an amalgam such that both $H, M$ are of characteristic $p$ and contain a common Sylow $p$-subgroup. Set $Q_X:=O_p(X)$ for $X\in\{H,M\}$, $Z=\langle \Omega(Z(S))^H\rangle$ and $V:=\langle Z^M\rangle$. Suppose that $M$ is $p$-minimal and $Q_H=C_S(Z)$. Then one of the following occurs:
\begin{enumerate}
\item $Z\not\le Q_M$;
\item $Z$ is an FF-module for $H/C_H(Z)$;
\item the dual of $Z$ is an FF-module for $H/C_H(Z)$;
\item $Z$ is a 2F-module with quadratic offender and $V$ contains more than one non-central chief factor for $M$; or
\item $M$ has exactly one non-central chief factor in $V$, $Q_H\cap Q_M\normaleq M$, $[V, O^p(M)]\le Z(Q_M)$ and contains some non-trivial $p$-reduced module.
\end{enumerate}
\end{theorem}

Notice that case (v) of the qrc-lemma is ruled out in our analysis by \cref{push} and in cases (ii) and (iii), \cref{SEFF} implies that $H/C_H(Z)\cong \SL_2(q)$, for $q$ some power of $p$.

We will require some results on FF-modules for weak BN-pairs and other pushing up configurations in subamalgams.

\begin{theorem}\label{JMod}
Suppose that $G$ is an outcome of \hyperlink{MainGrpThm}{Theorem C} where $L_{\alpha}$ and $L_{\beta}$ are $p$-solvable and let $S\in\syl_p(L_{\alpha})\cap \syl_p(L_{\beta})$. Assume that $G=\langle S^G\rangle$ and $V$ is an FF-module for $G$ such that $C_S(V)=\{1\}$. Then $G$ has a weak BN-pair of rank $2$ and is locally isomorphic to one of $\SL_3(p)$, $\Sp_4(p)$, or $\mathrm{G}_2(2)$. Moreover, if $G$ is locally isomorphic to $G_2(2)$, then $G/C_G(V)\cong \mathrm{G}_2(2)$.
\end{theorem}
\begin{proof}
If $G$ has a weak BN-pair of rank $2$ then this follows from \cite[Theorem A, Theorem B, Corollary 1]{ChermakJ}. If $G$ does not have a weak BN-pair of rank $2$, comparing with \hyperlink{MainGrpThm}{Theorem C}, we see that $p=b=2$, $L_{\alpha}/Q_{\alpha}\cong \Sym(3)$ and $L_{\beta}/Q_{\beta}\cong (3\times 3):2$. Moreover, there is $P_{\beta}\le L_{\beta}$ such that $P_{\beta}$ contains $S$, $P_{\beta}/Q_{\beta}\cong \Sym(3)$ and $Q_{\beta}$ contains two non-central chief factors for $P_{\beta}$. Indeed, no non-trivial subgroup of $S$ is normalized by both $L_{\alpha}$ and $P_{\beta}$ and by \cite{Fan}, $(L_{\alpha}, P_{\beta}, S)$ is locally isomorphic to $\mathrm{M}_{12}$. Setting $X:=\langle L_{\alpha}, P_{\beta}\rangle$ and applying \cite{ChermakJ}, $V$ is an FF-module for $X$ upon restriction and applying \cite[Lemma 3.12]{ChermakJ}, we have a contradiction.
\end{proof}

\begin{lemma}\label{SubAmal}
Suppose that $G$ is a minimal counterexample to \hyperlink{MainGrpThm}{Theorem C}, $\{\lambda, \delta\}=\{\alpha,\beta\}$ and the following conditions hold:
\begin{enumerate}
\item $L_{\alpha}/R_{\alpha}\cong\SL_2(q)\cong L_{\beta}/R_{\beta}$, and $Z_{\alpha}$ and $V_{\beta}/C_{V_{\beta}}(O^p(L_{\beta}))$ are natural $\SL_2(q)$-modules;
\item there is a non-central chief factor $U/W$ for $G_{\lambda}$ such that, as an $\bar{L_{\lambda}}$-module, $U/W$ is an FF-module, $C_{L_{\lambda}}(U/W)\ne R_\lambda$, and $C_{L_{\lambda}}(U/W)\cap R_\lambda$ normalizes $Q_{\alpha}\cap Q_{\beta}$; and
\item if $q=p$ then $Z(Q_{\alpha})=Z_{\alpha}$ is of order $p^2$ and $Z(Q_{\beta})=Z_{\beta}=\Omega(Z(S))$ is of order $p$.
\end{enumerate}
Then $q\in\{2,3\}$ and one of the following holds:
\begin{enumerate}[label=(\alph*)]
\item there is $H_{\lambda}\le G_{\lambda}$ containing $G_{\alpha,\beta}$ such that $(H_{\lambda}, G_{\delta}, G_{\alpha,\beta})$ is a weak BN-pair of rank $2$, $b\leq 5$ and if $b>3$, then $(H_{\lambda}, G_{\delta}, G_{\alpha,\beta})$ is parabolic isomorphic to $\mathrm{F}_3$ and $V_{\alpha}^{(2)}/Z_{\alpha}$ is not acted on quadratically by $S$;
\item $p=3$, $\lambda=\alpha$, neither $C_{L_{\alpha}}(U/W)$ nor $R_\alpha$ normalizes $Q_{\alpha}\cap Q_{\beta}$ and there does not exist $P_{\alpha}\le L_{\alpha}$ such that $S(C_{L_{\alpha}}(U/W)\cap R_\alpha)\le P_{\alpha}$, $P_{\alpha}$ is $G_{\alpha,\beta}$-invariant, $P_{\alpha}/C_{L_{\alpha}}(U/W)\cap R_\alpha\cong \SL_2(p)$,  $L_{\alpha}=P_{\alpha}R_\alpha=P_{\alpha}C_{L_{\alpha}}(U/W)$ and $Q_{\alpha}\cap Q_{\beta}\not\normaleq P_{\alpha}$;
\item $\lambda=\beta$ and neither $R_{\beta}$ nor $C_{L_{\beta}}(U/W)$ normalizes $V_{\alpha}^{(2)}$; or
\item there is $H_{\lambda}\le G_{\lambda}$ containing $G_{\alpha,\beta}$ such that for $X:=\langle H_{\lambda}, G_{\delta}\rangle$ and $V:=\langle Z_{\beta}^X\rangle$, we have that $V_{\beta}\le V\le S$, $C_S(V)\normaleq X$ and for $\wt X:=X/C_X(V)$, either $\wt X$ is locally isomorphic to $\SL_3(p), \Sp_4(p)$ or $\mathrm{G}_2(2)$; or $p=3$ and there is an involution $x$ in $G_{\alpha,\beta}$ such that $\wt X/\wt{\langle x\rangle}$ is locally isomorphic to $\PSp_4(3)$. Moreover, if $\wt Q_{\mu}$ contains more than one non-central chief factor for $\wt L_{\mu}$ where $\mu\in\{\alpha,\beta\}$, then $\wt Q_{\mu}$ contains two non-central chief factors and $\wt Q_{\nu}$ contains a unique non-central chief factor for $\wt L_{\nu}$ where $\mu\ne \nu\in\{\alpha,\beta\}$, and $\wt X\cong \mathrm{G}_2(2)$.
\end{enumerate}
\end{lemma}
\begin{proof}
It follows from (i), (ii) and \cref{SEFF} that $L_{\lambda}/C_{L_{\lambda}}(U/W)\cong L_{\lambda}/R_\lambda\cong\SL_2(q)$ and $\syl_p(C_{L_{\lambda}}(U/W))=\syl_p(R_\lambda)=\{Q_\lambda\}$. Thus, $C_{L_{\lambda}}(U/W)R_\lambda/Q_{\lambda}$ is a non-trivial normal $p'$-subgroup of $L_{\lambda}/Q_{\lambda}$. Assume that that $q\geq 4$ and $C_{L_{\lambda}}(U/W)\ne R_\lambda$. Then $C_{L_{\lambda}}(U/W)R_\lambda/C_{L_{\lambda}}(U/W)=Z(L_{\lambda}/C_{L_{\lambda}}(U/W))$ and $C_{L_{\lambda}}(U/W)R_\lambda/R_{\lambda}=Z(L_{\lambda}/R_\lambda)$. In particular, $p$ is odd and $L_{\lambda}/C_{L_{\lambda}}(U/W)\cap R_\lambda$ is isomorphic to a central extension of $\PSL_2(q)$ by an elementary abelian group of order $4$. Since $O^{p'}(L_{\lambda})=L_{\lambda}$ and the $p'$-part of the Schur multiplier of $\PSL_2(q)$ is of order $2$ by \cref{SLGen} (vii), we have a contradiction. Thus, we may assume that $q\in\{2,3\}$ throughout so that $G_{\alpha}$ and $G_{\beta}$ are $p$-solvable and by condition (iii), $Z(Q_{\alpha})=Z_{\alpha}$ is of order $p^2$ and $Z(Q_{\beta})=Z_{\beta}=\Omega(Z(S))$ is of order $p$. By \cref{Badp2} (ii) and \cref{Badp3} (ii), $L_{\lambda}/(C_{L_{\lambda}}(U/W)\cap R_\lambda)\cong (3\times 3):2$ if $p=2$, or $(Q_8\times Q_8):3$ if $p=3$.

Suppose that $p=2$. By \cref{Badp2} (iii), there are $P_1, \dots, P_4\le L_{\lambda}$ such that $S(C_{L_{\lambda}}(U/W)\cap R_\lambda)\le P_i$ and $P_{i}/(C_{L_{\lambda}}(U/W)\cap R_\lambda)\cong\Sym(3)$. Indeed, $C_{L_{\lambda}}(U/W)S$ and $R_\lambda S$ are non-equal and satisfy this condition. Moreover, $P_i$ is $G_{\alpha,\beta}$-invariant for all $i$. Since any two $P_i$ generate $L_{\lambda}$, we may choose $P_{\lambda}=P_j\ne R_{\lambda}S$ such that $Q_{\alpha}\cap Q_{\beta}\not\normaleq P_{\lambda}$ and $O^2(P_\lambda)$ does not centralize $U/W$. Set $H_\lambda:=P_\lambda G_{\alpha,\beta}$, $X:=\langle H_\lambda, G_\delta\rangle$ and $V:=\langle Z_{\beta}^X\rangle$. By (ii) and (iii), we have that $V_{\beta}\le V$.

Suppose that $p=3$. By \cref{Badp3} (iii), there is $P_1,\dots, P_5\le L_{\lambda}$ such that $S(C_{L_{\lambda}}(U/W)\cap R_\lambda)\le P_i$ and $P_{i}/(C_{L_{\lambda}}(U/W)\cap R_\lambda)\cong\SL_2(3)$. Again, $C_{L_{\lambda}}(U/W)S$ and $R_{\lambda}S$ are non-equal and satisfy this condition, and for any $i\ne j$, $L_{\lambda}=\langle P_i, P_j\rangle$. Since $C_{L_{\lambda}}(U/W)S$ and $R_\lambda S$ are $G_{\alpha,\beta}$-invariant there is at least one other $P_i$ which is $G_{\alpha,\beta}$-invariant. Notice that $R_{\beta}S$ normalizes $Q_{\alpha}\cap Q_{\beta}$ and as any two $P_i$ generate, by \cref{push} if $\lambda=\beta$ there is a choice of $P_{\lambda}=P_i$ such that $Q_{\alpha}\cap Q_{\beta}\not\normaleq P_{\lambda}$, $P_{\lambda}$ is $G_{\alpha,\beta}$-invariant and $O^3(P_\lambda)$ does not centralize $U/W$ or $V_{\beta}$. If $\lambda=\alpha$, then unless outcome (b) holds, we may choose $P_{\lambda}=P_j\ne R_\lambda S$ such that $Q_{\alpha}\cap Q_{\beta}\not\normaleq P_{\lambda}$ and $O^3(P_\lambda)$ does not centralize $U/W$. Again, we set $H_\lambda:=P_\lambda G_{\alpha,\beta}$, $X:=\langle H_\lambda, G_\delta\rangle$ and $V:=\langle Z_{\beta}^X\rangle$, remarking that $V_{\beta}\le V$.

For $p=2$ or $3$, $O_p(P_\lambda)=Q_{\lambda}$ and $P_{\lambda}/Q_{\lambda}$ has a strongly $p$-embedded subgroup. Moreover, $P_\lambda$ is of characteristic $p$, $C_S(V)\le C_{\beta}\le Q_{\alpha}\cap Q_{\beta}$ so that $C_S(V)=C_{Q_{\alpha}}(V)=C_{Q_{\beta}}(V)\normaleq X$. If no non-trivial subgroup of $G_{\alpha,\beta}$ is normal in $X$, then $X$ satisfies \cref{MainHyp} and since both $H_{\lambda}$ and $G_\delta$ are $p$-solvable, by minimality, $(H_\lambda, G_{\delta}, G_{\alpha,\beta})$ is a weak BN-pair of rank $2$; or $p=2$, $X$ is a symplectic amalgam, $|S|=2^6$ and exactly one of $\bar{H_{\lambda}}$ or $\bar{G_{\delta}}$ is isomorphic to $(3\times 3):2$. In the latter case, we get that $Q_{\lambda}$ and $Q_{\delta}$ are non-abelian subgroups of order $2^5$ and $\bar{G_{\delta}}$ and $\bar{G_{\lambda}}$ are isomorphic to subgroups of $\GL_4(2)$. Moreover, for some $\gamma\in\{\lambda, \delta\}$, $|Q_{\gamma}/\Phi(Q_{\gamma})|=2^3$ so that $\bar{G_{\gamma}}$ is isomorphic to a subgroup of $\GL_3(2)$. One can check that this implies that $G=X$, a contradiction. If $(H_\lambda, G_{\delta}, G_{\alpha,\beta})$ is a weak BN-pair then we may associate a critical distance to it. Since $\langle (V_{\delta}^{(n)})^{H_{\lambda}}\rangle\le \langle (V_{\delta}^{(n)})^{G_{\lambda}}\rangle$, it follows that the critical distance associated to $(H_\lambda, G_{\delta}, G_{\alpha,\beta})$ is greater than or equal to $b$. Comparing with the results in \cite{Greenbook}, we have that $b\leq 5$ and $b\leq 3$ unless $b=5$, $b$ is equal to the critical distance associated to $(H_\lambda, G_{\delta}, G_{\alpha,\beta})$ and $(H_{\lambda}, G_{\delta}, G_{\alpha,\beta})$ is parabolic isomorphic to $\mathrm{F}_3$. That $V_{\alpha}^{(2)}/Z_{\alpha}$ is not acted on quadratically by $S$ is a consequence of the structure of an $\mathrm{F_3}$-type amalgam.

Hence, we may assume that some non-trivial subgroup of $G_{\alpha,\beta}$ is normal in $X$. Let $K$ be the largest subgroup by inclusion satisfying this condition. Since $S$ is the unique Sylow $p$-subgroup of $G_{\alpha,\beta}$, $K$ normalizes $S$ so that $O_p(K)=S\cap K\normaleq X$. If $O_p(K)=\{1\}$, then $K$ is a $p'$-group which is normal in $G_{\delta}$, impossible since $F^*(G_{\delta})=Q_{\delta}$ is self-centralizing in $G_{\delta}$. Thus, there is a finite $p$-group which is normal in $X$. Since $O_p(K)\normaleq S$, $Z_{\beta}\le O_p(K)$. Then, by definition, $V\le O_p(K)$. Indeed, as $[O_p(K), V]=[O_p(K), \langle Z_{\beta}^X\rangle]=\{1\}$, we conclude that $V\le \Omega(Z(O_p(K)))$ and $O_p(K)\le C_S(V)$. By an earlier observation, $C_S(V)\normaleq X$ so that $C_S(V)=O_p(K)$. Set $\wt X:=X/C_X(V)$ so that $\wt X=\langle \wt H_{\lambda}, \wt G_{\delta}\rangle$ and $\wt H_{\lambda}\cong H_{\lambda}/C_{H_{\lambda}}(V)$ is a finite group. Additionally, $\wt G_{\delta}\cong G_{\delta}/C_{G_{\delta}}(V)$ is a finite group. Since $C_S(V)\in\syl_p(C_{H_{\lambda}}(V)\cap C_{G_{\delta}}(V))$, $C_S(V)\le C_{\beta}$ and $H_{\lambda}$ does not normalize $Q_{\alpha}\cap Q_{\beta}$, we deduce that $\wt Q_{\lambda}=O_p(\wt H_{\lambda})$ and $\wt H_{\lambda}/\wt Q_{\lambda}$ has a strongly $p$-embedded subgroup. Similarly, $\wt Q_{\delta}=O_p(\wt G_{\delta})$ and $\wt G_{\delta}/\wt Q_{\delta}$ has a strongly $p$-embedded subgroup.

In order to show that the triple $(\wt H_{\lambda}, \wt G_{\delta}, \wt {G_{\alpha,\beta}})$ satisfies \cref{MainHyp}, we need to show that $\wt H_{\lambda}$ and $\wt G_{\delta}$ are of characteristic $p$, $\wt {G_{\alpha,\beta}}=\wt H_{\lambda} \cap \wt G_{\delta}=N_{\wt H_{\lambda}}(\wt S)=N_{\wt G_{\delta}}(\wt S)$ and no non-trivial subgroup of $\wt{G_{\alpha,\beta}}$ is normal in both $\wt H_{\lambda}$ and $\wt G_{\delta}$. In the following, we often examine the ``preimage in $H_{\lambda}$" of some subgroup of $\wt H_{\lambda}$, by which we mean the preimage in $H_{\lambda}$ of the isomorphic image in $H_{\lambda}/C_{H_{\lambda}}(V)$.

Notice that if $\wt H_\lambda$ is not of characteristic $p$ then $F^*(\wt H_\lambda)\ne \wt Q_{\lambda}$. Then, as $\wt H_\lambda$ is $p$-solvable, $O_{p'}(\wt H_\lambda)\ne \{1\}$ so that for $\mathcal{C}_\lambda$ the preimage in $H_\lambda$ of $O_{p'}(\wt H_\lambda)$, $[\mathcal{C}_\lambda, Q_{\lambda}, V]=\{1\}$. For $r\in \mathcal{C}_\lambda$ of order coprime to $p$, it follows from the A$\times$B-lemma that if $r$ centralizes $C_{V}(Q_{\lambda})$, then $\wt r=1$. Since $Q_{\lambda}$ is self-centralizing in $S$, we have that $C_{V}(Q_{\lambda})\le Z(Q_{\lambda})$. Similarly, if $\wt G_{\delta}$ is not of characteristic $p$, defining $\mathcal{C}_\delta$ analogously, by the A$\times$B-lemma we need only show $\mathcal{C}_\delta$ centralizes $C_V(Q_\delta)\le Z(Q_{\delta})$.

Suppose that $\lambda=\beta$. Then $|Z(Q_{\beta})|=p$ and so, either $\wt H_{\beta}$ is of characteristic $p$; or $p=3$, $|\wt {\mathcal{C}_\beta}|=2$ and $\mathcal{C}_{\beta}$ acts non-trivially on $Z_{\beta}$. In the latter case, $\wt {\mathcal{C}_\beta}\le Z(\wt{H_\beta})$ so that $[\mathcal{C}_{\beta}, S]\le C_{H_{\beta}}(V)$. Moreover, by coprime action, we have that $V=[V, \mathcal{C}_\beta]\times C_{V}(\mathcal{C}_\beta)$ is an $S$-invariant decomposition and as $\wt{\mathcal{C}_\beta}$ acts non-trivially on $Z_{\beta}$, it follows that $V=[V, \mathcal{C}_\beta]$ is inverted by $\wt{\mathcal{C}_\beta}$. By the Frattini argument, $\mathcal{C}_\beta S=C_{H_{\beta}}(V)S(G_{\alpha,\beta}\cap \mathcal{C}_\beta)$ and we may as well assume that there is $x \in G_{\alpha,\beta}\cap \mathcal{C}_{\beta}$ such that $\wt {\langle x\rangle}=\wt {\mathcal{C}_\beta}$. But then $[x, Q_{\alpha}]\le [x, S]\le C_S(V)$ and as $x\in G_{\alpha,\beta}\le G_{\alpha}$, $\wt G_{\alpha}$ is not of characteristic $p$.

Consider $\mathcal{C}_\alpha$, the preimage in $G_\alpha$ of $O_{p'}(\wt G_\alpha)$. If $\wt{G_\alpha}$ is not of characteristic $p$, then applying the A$\times$B-lemma,  $\mathcal{C}_\alpha\cap C_{G_{\alpha}}(Z_{\alpha})\le C_{G_{\alpha}}(V)$ and $\wt {\mathcal{C}_\alpha}$ is isomorphic to a normal $p'$-subgroup of $\GL_2(p)$.

Suppose that $|\wt {\mathcal{C}_\alpha}|=3$ if $p=2$, or $\wt {\mathcal{C}_\alpha}\cong Q_8$ if $p=3$. Noticing that $[S, C_G(Z_{\alpha})]\le [L_{\alpha}, C_{G_{\alpha}}(Z_{\alpha})]\le R_{\alpha}$, by the Frattini argument, $C_{G_{\alpha}}(Z_{\alpha}) G_{\alpha,\beta}=R_{\alpha} G_{\alpha,\beta}$ and $G_{\alpha}=R_{\alpha} G_{\alpha,\beta} \mathcal{C}_\alpha$. By \cref{push}, since $\mathcal{C}_\alpha G_{\alpha,\beta}$ normalizes $Q_{\alpha}\cap Q_{\beta}$, it remains to prove that $R_{\alpha}$ normalizes $Q_{\alpha}\cap Q_{\beta}$ to get a contradiction. 

Assume that $R_{\alpha}$ does not normalize $Q_{\alpha}\cap Q_{\beta}$ and let $M_{\alpha}:=C_{G_{\alpha}}(Z_{\alpha}) G_{\alpha,\beta}$. Then, $C_{G_{\alpha}}(Z_{\alpha})\not\le G_{\alpha,\beta}$ so that $Q_{\alpha}=O_p(M_{\alpha})$. Reapplying the A$\times$B-lemma yields $\wt {M_{\alpha}\cap \mathcal{C}_{\alpha}}=\{1\}$ if $p=2$ and $|\wt {M_{\alpha}\cap \mathcal{C}_{\alpha}}|\leq 2$ if $p=3$. In the latter case, suppose that $\wt {M_{\alpha}\cap \mathcal{C}_{\alpha}}$ is non-trivial and choose $y\in M_{\alpha}\cap \mathcal{C}_{\alpha}$ with $[y, V]\ne\{1\}$. Indeed, $\wt{\langle  y \rangle}=\wt {M_{\alpha}\cap \mathcal{C}_{\alpha}}$ is central in $\wt {M_{\alpha}}$. It follows that $[y, S]\le C_{M_{\alpha}}(V)$. Now, by the Frattini argument, $(\mathcal{C}_\alpha \cap M_{\alpha})S=C_{M_{\alpha}}(V)S(G_{\alpha,\beta}\cap \mathcal{C}_\alpha)$ and we may as well assume that $y\in G_{\alpha,\beta}$ so that $[y, S]\le C_S(V)$. But then $[y, Q_{\beta}]\le C_X(V)$ and so $\wt {H_{\beta}}$ is not of characteristic $3$. Indeed, we can arrange that $\langle y\rangle C_{H_{\beta}}(V)=\mathcal{C}_\beta$. 

Now, we may form $M_{\alpha}^*:=C_{G_{\alpha}}(Z_{\alpha})(L_{\beta}\cap G_{\alpha,\beta})$ and $H_{\beta}^*:=(H_{\beta}\cap L_{\beta})(M_{\alpha}^*\cap G_{\alpha,\beta})$ and arguing as above, we infer that $\wt {M_{\alpha}^*}$ and $\wt {H_{\beta}^*}$ are both of characteristic $p$. Moreover, by construction and since $R_{\alpha}$ does not normalize $Q_{\alpha}\cap Q_{\beta}$, we deduce that $\wt Q_{\alpha}=O_p(\wt {M_{\alpha}^*})$ and $\wt {M_{\alpha}^*}/\wt Q_{\alpha}$ has a strongly $p$-embedded subgroup. Similarly, $\wt Q_{\beta}=O_p(\wt {H_{\beta}^*})$ and $\wt {H_{\beta}^*}/\wt Q_{\beta}$ also has a strongly $p$-embedded subgroup. Set $Y:=\langle M_{\alpha}^*, H_{\beta}^*\rangle$ and write $G_{\alpha,\beta}^*:=M_{\alpha}^*\cap G_{\alpha,\beta}$

Since $\wt S=\wt Q_{\alpha} \wt Q_{\beta}$, it is easily checked that $\wt {G_{\alpha,\beta}^*}=N_{\wt {M_{\alpha}^*}}(\wt S)=N_{\wt{ H_{\beta}^*}}(\wt S)=\wt {M_{\alpha}^*}\cap \wt {H_{\beta}^*}$. Suppose there exists $K^*\le \wt {G_{\alpha,\beta}^*}$ such that $K^*\normaleq \langle \wt {M_{\alpha}^*}, \wt {H_{\beta}^*}\rangle=\wt Y$. Since $\wt {M_{\alpha}^*}$ and $\wt {H_{\beta}^*}$ are both of characteristic $p$, we may assume that $K^*$ is not a $p'$-group, and since $K^*\le \wt {G_{\alpha,\beta}^*}$, $O_p(K^*)=K^*\cap \wt S\ne\{1\}$. Let $K_{\alpha}$ denote the preimage of $O_p(K^*)$ in $M_{\alpha}^*$ and $K_{\beta}$ denote the preimage of $O_p(K^*)$ in $H_{\beta}^*$. Then, $T_{\alpha}:=Q_{\alpha}\cap K_{\alpha}$ is a normal $p$-subgroup of $M_{\alpha}^*$ and, likewise, $T_{\beta}:=Q_{\beta}\cap K_{\beta}$ is a normal $p$-subgroup of $H_{\beta}^*$. Since $\wt {T_{\alpha}T_{\beta}}=\wt {T_{\alpha}}=\wt {T_{\beta}}$, a comparison of orders yields $T_{\alpha}T_{\beta}=T_{\alpha}=T_{\beta}\normaleq Y$. Moreover, $T_{\alpha}>C_S(V)$ and as $Y$ is normalized by $G_{\alpha,\beta}$, $T_{\alpha}$ is normalized by $G_{\alpha,\beta}$. But now, $G_{\alpha}=G_{\alpha,\beta}\mathcal{C}_{\alpha}M_{\alpha}^*$ and as $\mathcal{C}_{\alpha}$ centralizes $Q_{\alpha}/C_S(V)$, $T_{\alpha}\normaleq \langle G_{\alpha}, H_{\beta}\rangle=X$, a contradiction since $C_S(V)$ is the largest $p$-subgroup of $G_{\alpha,\beta}$ which is normalized by $X$. Hence, the triple $(\wt {M_{\alpha}^*}, \wt {H_{\beta}^*}, \wt {G_{\alpha\beta}^*})$ satisfies \cref{MainHyp}.

Since $C_S(V)\le Q_{\alpha}\cap Q_{\beta}$ and $C_S(V)$ is the largest subgroup of $S$ which is normal in $Y$, we have that $J(S)\not\le C_S(V)$ and a elementary calculation yields that $\Omega(Z(C_S(V)))$ is an FF-module for $\wt Y$. Moreover, by construction, $Y=\langle S^Y\rangle$ and, by minimality and since $\wt {M_{\alpha}^*}$ and $\wt {H_{\beta}^*}$ are $p$-solvable, $\wt Y$ is locally isomorphic to one of $\SL_3(p)$, $\Sp_4(p)$ or $\mathrm{G}_2(2)$. Moreover, $V_{\alpha}^{(2)}\le V$ so that $C_S(V)\le C_{Q_{\alpha}}(V_{\alpha}^{(2)})$. If $\wt Y$ is locally isomorphic to $\SL_3(p)$, then $C_{\beta}$ is the largest normal subgroup of $H_{\beta}$ contained in $Q_{\alpha}\cap Q_{\beta}$, it follows that $C_{\beta}\le C_S(V)\le C_{Q_{\alpha}}(V_{\alpha}^{(2)})$, a contradiction for then $C_{\beta}\normaleq \langle G_{\alpha}, G_{\beta}\rangle$. 

If $\wt Y$ is locally isomorphic to $\Sp_4(p)$, then it follows that $|\wt {C_{\beta}}|\leq p$. We may as well assume that $C_S(V)=C_{Q_{\alpha}}(V_{\alpha}^{(2)})$ has index $p$ in $C_{\beta}$, else we obtain a contradiction as before. Since $C_S(V)\normaleq X$ and $G_{\beta}=\langle H_{\beta}, R_{\beta}\rangle=\langle H_{\beta}, C_{L_{\beta}}(U/W)\rangle$, it follows that neither $R_{\beta}$ nor $C_{L_{\beta}}(U/W)$ normalizes $V_{\alpha}^{(2)}$ and conclusion (c) holds. If $\wt Y\cong \mathrm{G}_2(2)$, then one can calculate in a similar manner that $C_S(V)=C_{Q_{\alpha}}(V_{\alpha}^{(2)})$ and again we retrieve outcome (c). 

Therefore, if $\lambda=\beta$ and $\wt{G_\alpha}$ is not of characteristic $p$, then $p=3$ and $|\wt{\mathcal{C}_{\alpha}}|=2$. Then $[\wt S, \wt{\mathcal{C}_{\alpha}}]=\{1\}$ and, again applying the Frattini argument, we have that $\mathcal{C}_{\alpha}S=C_{G_{\alpha}}(V)S(G_{\alpha,\beta}\cap \mathcal{C}_{\alpha})$. Choose $y\in G_{\alpha,\beta}\cap \mathcal{C}_{\alpha}$ with $[y, V]\ne\{1\}$ so that $\wt{\langle y\rangle}=\wt{\mathcal{C}_{\alpha}}$. Indeed, $[y, S]\le C_S(V)$ and it follows that $\wt{H_{\beta}}$ is not of characteristic $3$. Hence, we have that $\wt{H_{\beta}}$ is not of characteristic $3$ if and only if $\wt {G_{\alpha}}$ is not of characteristic $3$. Moreover, there is $x\in G_{\alpha,\beta}$ such that $\wt{\langle x\rangle}=\wt{\mathcal{C}_{\alpha}}=\wt{\mathcal{C}_{\beta}}$.

If $\wt{G_\alpha}$ is not of characteristic $p$, then set $\widehat X:=\wt X/\wt {\langle x\rangle}$ so that both $\widehat H_{\beta}$ and $\widehat{G_{\alpha}}$ are of characteristic $3$. Moreover, $\widehat L_{\alpha}/ \widehat R_{\alpha}\cong \PSL_2(3)$ and $\widehat{O^{p'}(H_{\beta})}/\widehat{(R_{\beta}\cap O^{p'}(H_{\beta}))}\cong \SL_2(3)$. As in the construction of $\wt Y$ above, it is easily checked that $\hat{G_{\alpha,\beta}}=N_{\hat{G_{\alpha}}}(\hat{S})=N_{\hat{H_{\beta}}}(\hat{S})=\hat{G_{\alpha}}\cap \hat{H_{\beta}}$ and no non-trivial subgroup of $\hat{G_{\alpha,\beta}}$ is normal in $\hat{X}$. Thus, by minimality, the triple $(\hat{G_{\alpha}}, \hat{H_{\beta}}, \hat{G_{\alpha,\beta}})$ is a weak BN-pair. Indeed, $\hat{L_{\alpha}}=O^{3'}(\hat{G_{\alpha}})$ and $\hat{L_{\alpha}}\cong \PSL_2(3)$ or $\SL_2(3)$. If $\hat{L_{\alpha}}\cong \SL_2(3)$, then a Sylow $2$-subgroup of $\wt{L_{\alpha}}$ is of order $16$, and arguing as in \cref{NiceSL2}, we force a contradiction. Thus, $\hat{L_{\alpha}}\cong \PSL_2(3)$ and $\hat{X}$ is locally isomorphic to $\PSp_4(3)$. Then, using that $C_{\beta}$ is the largest normal subgroup of $H_{\beta}$ which is contained in $Q_{\alpha}\cap Q_{\beta}$ and $C_{Q_{\alpha}}(V_{\alpha}^{(2)})$ is the largest subgroup of $C_{\beta}$ normal in $G_{\alpha}$, it follows that $C_S(V)=C_{Q_{\alpha}}(V_{\alpha}^{(2)})\normaleq X$. Since $G_{\beta}=\langle H_{\beta}, R_{\beta}\rangle=\langle H_{\beta}, C_{L_{\beta}}(U/W)\rangle$, it follows that neither $R_{\beta}$ nor $C_{L_{\beta}}(U/W)$ normalizes $V_{\alpha}^{(2)}$ and conclusion (c) holds. Thus, we may as well assume that whenever $\lambda=\beta$, $\wt X$ satisfies \cref{MainHyp} and acts faithfully on $V$.
 
Suppose now that $\lambda=\alpha$ so that $H_\alpha/C_{H_\alpha}(Z_\alpha)$ is isomorphic to a subgroup of $\GL_2(p)$. If $\wt H_{\alpha}$ is not of characteristic $p$ then, by the A$\times$B-lemma, $\mathcal{C}_\alpha\not\le C_{H_\alpha}(Z_\alpha)$ and so $\mathcal{C}_\alpha C_{H_\alpha}(Z_\alpha)/C_{H_\alpha}(Z_\alpha)$ is isomorphic to a normal $p'$-subgroup of $\GL_2(p)$. If $p=2$ or $|\mathcal{C}_\alpha C_{H_\alpha}(Z_\alpha)/C_{H_\alpha}(Z_\alpha)|>2$ and $p=3$, using the Frattini argument it follows that $H_\alpha=C_{P_\alpha}(Z_{\alpha})\mathcal{C}_\alpha G_{\alpha,\beta}=(R_{\alpha}\cap C_{L_{\alpha}}(U/W))\mathcal{C}_\alpha G_{\alpha,\beta}$ which normalizes $Q_{\alpha}\cap Q_{\beta}$, a contradiction. Thus, $p=3$ and $|\wt {\mathcal{C}_\alpha}|=2$ so that $[\mathcal{C}_\alpha, S]\le C_X(V)$. Additionally,  by coprime action, $V=[V, \mathcal{C}_\alpha]\times C_{V}(\mathcal{C}_\alpha)$ and as $\wt {\mathcal{C}_\alpha}$ does not centralize $Z_{\beta}$ we deduce that $V=[V, \mathcal{C}_\alpha]$ is inverted by $\wt {\mathcal{C}_{\alpha}}$. Then, by the Frattini argument, $S\mathcal{C}_\alpha=SC_{H_{\alpha}}(V)(G_{\alpha,\beta}\cap \mathcal{C}_\alpha)$ and we may choose $x\in G_{\alpha,\beta}\cap \mathcal{C}_\alpha$ with $[x, V]\ne\{1\}$ so that $\wt{\langle x\rangle}=\wt {\mathcal{C}_{\alpha}}$ and $[x, S]\le C_S(V)$. It follows that $\wt G_{\beta}$ is not of characteristic $3$.

If $\wt G_{\beta}$ is not of characteristic $p$ then, by the A$\times$B-lemma, $\mathcal{C}_\beta$ does not centralize $Z_{\beta}$. In particular, $p=3$ and $|\wt{\mathcal{C}_\beta}|=2$. Then applying coprime action, $\wt{\mathcal{C}_\beta}$ inverts $V$ and we see that there is $x\in G_{\alpha,\beta}$ with $\wt{\langle x\rangle}=\wt{\mathcal{C}_\alpha}=\wt{\mathcal{C}_{\beta}}$. Hence, $\wt H_{\alpha}$ is of characteristic $p$ if and only if $\wt G_{\beta}$ is of characteristic $p$.

If $\wt{G_{\beta}}$ is not of characteristic $p$, then set $\widehat{X}:=\wt X/\wt {\langle x\rangle}$ so that $\widehat{H_{\alpha}}$ and $\widehat{G_{\beta}}$ are both of characteristic $3$, $\widehat{O^{p'}(H_{\alpha})}/\widehat{O^{p'}(H_{\alpha})\cap R_{\alpha}}\cong \PSL_2(3)$ and $\widehat L_{\beta}/\widehat R_{\beta}\cong \SL_2(3)$. As in the above, it quickly follows that $\hat{X}$ satisfies \cref{MainHyp} and by minimality, the triple $(\hat{H_{\alpha}}, \hat{G_{\beta}}, \hat{G_{\alpha,\beta}})$ is a weak BN-pair of rank $2$. Indeed, $\widehat{O^{p'}(H_{\alpha})}\cong \PSL_2(3)$ and $\widehat{X}$ is locally isomorphic to $\PSp_4(3)$, and the outstanding case in (d) is satisfied. We may as well assume that whenever $\lambda=\alpha$, $\wt X$ has satisfies \cref{MainHyp} and acts faithfully on $V$.

Finally, for either $\lambda=\alpha$ or $\lambda=\beta$, $\wt X$ satisfies \cref{MainHyp} and acts faithfully on $V$. Moreover, since $J(S)\not\le C_S(V)$ an elementary argument (as in the proof of \cref{BasicFF}) implies that $V$ is an FF-module for $\wt X$. By minimality, $\wt X$ satisfies \cref{MainHyp} and since both $\wt H_\lambda$ and $\wt G_\delta$ are $p$-solvable, $\wt{X}$ is determined by \cref{JMod}. Counting the number of non-central chief factors in amalgams locally isomorphic to $\SL_3(p), \Sp_4(p)$ or $\mathrm{G}_2(2)$ (as can be gleaned from \cite{Greenbook}), outcome (d) is satisfied. 
\end{proof}

The hypothesis of \cref{SubAmal} exhibit a common situation we encounter in the work ahead: where $Z_{\beta}=Z(Q_{\beta})$ is of order $p$, and both $Z(Q_{\alpha})=Z_{\alpha}$ and $V_{\beta}/C_{V_{\beta}}(O^p(L_{\beta}))$ are natural $\SL_2(p)$-modules for $L_{\alpha}/R_{\alpha}\cong\SL_2(p)\cong L_{\beta}/R_{\beta}$. Upon first glance, it seems that we have very little control over the action of $R_{\lambda}$ for $\lambda\in\{\alpha,\beta\}$. Throughout this section we strive to force situations in which the full hypotheses of \cref{SubAmal} are satisfied. In applying \cref{SubAmal}, the outcomes there will often force contradictions and the conclusion we draw is that $O^p(R_\lambda)$ centralizes $U/W$. In this situation, \cref{SimExt} becomes a powerful tool in dispelling a large number of cases. Motivated by this, we make the following hypothesis and record a number of lemmas controlling the actions of $R_{\lambda}$ for $\lambda\in\Gamma$.

\begin{hypothesis}\label{CommonHyp}
The following conditions hold:
\begin{enumerate}[label=(\roman*)]
\item $L_{\alpha}/R_{\alpha}\cong\SL_2(q)\cong L_{\beta}/R_{\beta}$, and $Z_{\alpha}$ and $V_{\beta}/C_{V_{\beta}}(O^p(L_{\beta}))$ are natural $\SL_2(q)$-modules; and
\item if $q=p$ then $Z(Q_{\alpha})=Z_{\alpha}$ is of order $p^2$ and $Z(Q_{\beta})=Z_{\beta}=\Omega(Z(S))$ is of order $p$.
\end{enumerate}
\end{hypothesis}

As a first consequence of this hypothesis, we make the following observation, gaining some control over the order of $V_{\beta}$ and the number of non-central chief factors in $V_{\alpha}^{(2)}$.

\begin{lemma}\label{VBGood}
Suppose that $b>2$ and \cref{CommonHyp} is satisfied. Then, for $\lambda\in\alpha^G$ and $\delta\in\Delta(\lambda)$, exactly one of the following occurs:
\begin{enumerate}
\item $|V_{\delta}|=q^3$ and $[V_{\lambda}^{(2)}, Q_{\lambda}]=Z_{\lambda}$; or
\item $C_{V_{\delta}}(O^p(L_{\delta}))\ne Z_{\delta}$ and for $V^\lambda:=\langle C_{V_{\delta}}(O^p(L_{\delta}))^{G_{\lambda}} \rangle$, $[V^\lambda, Q_{\lambda}]=Z_{\lambda}$, both $V_{\lambda}^{(2)}/V^\lambda$ and $V^\lambda/Z_{\lambda}$ contain a non-central chief factor for $L_{\lambda}$, and $V^\lambda V_{\delta}\not\normaleq L_{\delta}$.
\end{enumerate}
Moreover, whenever $Z_{\delta+1}C_{V_{\delta}}(O^p(L_{\delta}))=Z_{\delta-1}C_{V_{\delta}}(O^p(L_{\delta}))$ for $\delta\in \Gamma^G$, we have that $Z_{\delta+1}=Z_{\delta-1}$.
\end{lemma}
\begin{proof}
Suppose first that $|V_{\delta}|=q^3$. Then $[Q_{\lambda}, V_{\lambda}^{(2)}]=[Q_{\lambda}, V_{\delta}]^{G_{\lambda}}=Z_{\lambda}$ and (i) holds. Hence, we assume for the remainder of the proof that $C_{V_{\delta}}(O^p(L_{\delta}))\ne Z_{\delta}$. 

Suppose that $q=p\in\{2,3\}$. Then by coprime action, $V_{\delta}/Z_{\delta}=[V_{\delta}/Z_{\delta}, O^p(L_{\delta})]\times C_{V_{\delta}/Z_{\delta}}(O^p(L_{\delta}))$ and $C_{V_{\delta}/Z_{\delta}}(O^p(L_{\delta}))=C_{V_{\delta}}(O^p(L_{\delta}))/Z_{\delta}$. We infer from the definition of $V_\delta$ that $V_\delta=[V_\delta,  O^p(L_{\delta})]Z_\lambda$. In particular, $[Q_{\lambda}, C_{V_{\delta}}(O^p(L_{\delta}))]\le [V_\delta,  O^p(L_{\delta})]\cap C_{V_{\delta}}(O^p(L_{\delta}))=Z_\delta$. In all other cases of $q$ and $p$, we have that $L_\delta=O^p(L_{\delta})Q_\delta$, so that $[Q_{\lambda}, C_{V_{\delta}}(O^p(L_{\delta}))]\le [L_\delta, C_{V_{\delta}}(O^p(L_{\delta}))]\le Z_\delta$.

Note that $[Q_{\lambda}, C_{V_{\delta}}(O^p(L_{\delta}))]\ne \{1\}$, else as $Q_{\delta}/(Q_{\delta}\cap Q_{\lambda})$ is $G_{\delta, \lambda}$-irreducible, we would either have that $Q_{\lambda}\cap Q_{\delta}=C_{Q_{\delta}}(C_{V_{\delta}}(O^p(L_{\delta})))\normaleq L_{\delta}$, a contradiction by \cref{push}; or that $C_{V_{\delta}}(O^p(L_{\delta}))\le Z(Q_{\delta})$. In the latter case, since $Z(Q_{\delta})=Z_\delta$ when $q=p$ and since $L_\delta=O^p(L_{\delta})Q_\delta$ otherwise, we deduce that $C_{V_{\delta}}(O^p(L_{\delta}))=Z_\delta$, another contradiction.

Let $V^\lambda:=\langle C_{V_{\delta}}(O^p(L_{\delta}))^{G_{\lambda}}\rangle$. Since $\{1\}\ne [Q_\lambda, C_{V_{\delta}}(O^p(L_{\delta})]\le Z_{\lambda}$, we deduce that $[Q_{\lambda}, V^\lambda]=Z_{\lambda}<V^\lambda\le V_{\lambda}^{(2)}$. Suppose that $V^\lambda/Z_{\lambda}$ does not contain a non-central chief factor for $L_{\lambda}$. Then $L_{\lambda}$ normalizes $Z_{\lambda}C_{V_{\delta}}(O^p(L_{\delta}))$ so that $V^\lambda=Z_{\lambda}C_{V_{\delta}}(O^p(L_{\delta}))$ and $[Q_{\lambda}, Z_{\lambda}C_{V_{\delta}}(O^p(L_{\delta}))]\normaleq L_{\lambda}$. But $[Q_{\lambda}, Z_{\lambda}C_{V_{\delta}}(O^p(L_{\delta}))]\le Z_{\delta}$ and we deduce that $Q_{\lambda}$ centralizes $C_{V_{\delta}}(O^p(L_{\delta}))$, a contradiction. Thus, $V^\lambda/Z_{\lambda}$ contains a non-central chief factor for $L_{\lambda}$. 

Assume that $V_{\lambda}^{(2)}/V^\lambda$ does not contain a non-central chief factor for $L_\lambda$. Similarly to above, we deduce in this case that $V_{\lambda}^{(2)}=V^\lambda V_\delta$ so that $Z_{\lambda}\le [V_{\lambda}^{(2)}, Q_\lambda]\le [V^\lambda, Q_\lambda][V_\delta, Q_\lambda]\le Z_{\lambda}C_{V_{\delta}}(O^p(L_{\delta}))$. Indeed, $[V_{\lambda}^{(2)}, Q_\lambda]=Z_{\lambda}([V_{\lambda}^{(2)}, Q_\lambda]\cap C_{V_{\delta}}(O^p(L_{\delta})))$. But now, $[Q_\lambda, [V_{\lambda}^{(2)}, Q_\lambda]\cap C_{V_{\delta}}(O^p(L_{\delta}))]$ is a normal subgroup of $L_\delta$ contained in $[Q_\lambda, C_{V_{\delta}}(O^p(L_{\delta}))]\le Z_\delta$. Hence, $[V_{\lambda}^{(2)}, Q_\lambda]\cap C_{V_{\delta}}(O^p(L_{\delta}))$ is centralized by $Q_\lambda$. But $[V_{\lambda}^{(2)}, Q_\lambda]\cap C_{V_{\delta}}(O^p(L_{\delta}))\normaleq L_\delta=SO^p(L_\delta)$ so that $[L_\delta, [V_{\lambda}^{(2)}, Q_\lambda]\cap C_{V_{\delta}}(O^p(L_{\delta}))]=[\langle {Q_\lambda}^{L_{\delta}}\rangle, [V_{\lambda}^{(2)}, Q_\lambda]\cap C_{V_{\delta}}(O^p(L_{\delta}))]=\{1\}$ and $[V_{\lambda}^{(2)}, Q_\lambda]\cap C_{V_{\delta}}(O^p(L_{\delta}))=Z_\delta$. Hence, $[V_{\lambda}^{(2)}, Q_\lambda]=Z_\lambda$ so that $[V_\delta, Q_\lambda]\le Z_\lambda$. Hence $Z_\lambda Z_\mu\normaleq L_\delta=\langle Q_{\lambda}, Q_\mu, R_\delta\rangle$ for any $\mu\in\Delta(\delta)$ with $Q_\lambda R_\delta\ne Q_\mu R_\delta$. By the definition of $V_\delta$, $V_\delta=Z_\lambda Z_\mu$ is of order $q^3$ and $C_{V_{\delta}}(O^p(L_{\delta}))=Z_\delta$, a contradiction. Thus, $V_{\lambda}^{(2)}/V^\lambda$ contains a non-central chief factor for $L_\lambda$.

Suppose that $Z_{\lambda}C_{V_{\delta}}(O^p(L_{\delta}))=Z_{\mu}C_{V_{\delta}}(O^p(L_{\delta}))$ for some $\mu\in\Delta(\delta)$. In particular, $Z_{\lambda}C_{V_{\delta}}(O^p(L_{\delta}))$ is normalized by $Q_{\lambda}R_{\delta}$ and $Q_{\mu}R_{\delta}$. If $Q_{\lambda}R_{\delta}\ne Q_{\mu}R_{\delta}$ then $Z_{\lambda}C_{V_{\delta}}(O^p(L_{\delta}))\normaleq L_{\delta}=\langle Q_{\lambda}, Q_{\mu}, R_{\delta}\rangle$, and from the definition of $V_{\delta}$, $V_{\delta}=Z_{\lambda}C_{V_{\delta}}(O^p(L_{\delta}))$ is centralized, modulo $Z_{\delta}$, by $Q_{\lambda}$, a contradiction by \cref{BasicVB}. Thus, $Q_{\lambda}R_{\delta}=Q_{\mu}R_{\delta}$. Then, there is $r\in R_{\delta}$ such that $Q_{\lambda}^rQ_{\delta}=(Q_{\lambda}Q_{\delta})^r=Q_{\mu}Q_{\delta}=Q_{\mu}Q_{\delta}$ and we may as well pick $r$ of order coprime to $p$. Moreover, since $O^p(R_{\delta})$ centralizes $Q_{\delta}/C_{\delta}$, it follows that $Q_{\lambda}\in\syl_p(Q_{\lambda}O^p(R_{\delta}))$. But then $Q_{\mu}\in\syl_p(Q_{\lambda}O^p(R_{\delta}))$. Since $r$ centralizes $Q_{\delta}/C_{\delta}$ we conclude that $Q_{\lambda}\cap Q_{\delta}=Q_{\mu}\cap Q_{\delta}$. Therefore, $Z_\lambda=Z(Q_{\lambda}\cap Q_{\delta})\cap V_\delta=Z(Q_{\mu}\cap Q_{\delta})\cap V_\delta=Z_\mu$. 

Suppose for a contradiction that $V^\lambda V_{\delta}\normaleq L_{\delta}$ and choose $\mu\in\Delta(\mu)$ such that $Z_\lambda\ne Z_\mu$ and $Q_{\delta}=(Q_{\lambda}\cap Q_{\delta})(Q_{\mu}\cap Q_{\delta})$. In particular, $Z_\lambda C_{V_{\delta}}(O^p(L_{\delta}))\ne Z_\mu C_{V_{\delta}}(O^p(L_{\delta}))$. Moreover, as $V^\lambda V_{\delta}\normaleq L_{\delta}$, $V^\lambda V_\delta=V^{\mu}V^\delta$. Now, \[Z_{\delta}\le [Q_{\delta}, V^{\lambda} V_{\delta}]=[Q_{\lambda}\cap Q_{\delta}, V^{\lambda}V^\delta][Q_{\mu}\cap Q_{\delta}, V^{\mu}V^\delta]\le Z_{\lambda}Z_{\mu}\] and $[Q_{\delta}, V^{\lambda} V_{\delta}]\normaleq L_{\delta}$. Set $W:=[Q_{\delta}, V^{\lambda} V_{\delta}]$. If $W\not\le C_{V_{\delta}}(O^p(L_{\delta}))$, then $V_\delta=WC_{V_{\delta}}(O^p(L_{\delta}))$ and we deduce that $|W|\geq q^3$. Since $W\le Z_\lambda Z_\mu$ we conclude that $W=Z_\lambda Z_\mu\normaleq L_\delta$ and has order $q^3$. By definition, $V_\delta=W$, a contradiction since $C_{V_{\delta}}(O^p(L_{\delta}))\ne Z_\delta$. Hence, $W\le C_{V_{\delta}}(O^p(L_{\delta}))$. Then $W\le (Z_\lambda Z_\mu)\cap C_{V_{\delta}}(O^p(L_{\delta}))$ and since $Z_\lambda C_{V_{\delta}}(O^p(L_{\delta}))\ne Z_\mu C_{V_{\delta}}(O^p(L_{\delta}))$ we deduce that $Z_\delta\le W\le (Z_\lambda Z_\mu)\cap C_{V_{\delta}}(O^p(L_{\delta}))=Z_\delta$ so that $W=Z_\delta$. But then $Q_\delta$ centralizes $V^\lambda/Z_\lambda$, a contradiction since $V^\lambda/Z_\lambda$ contains a non-central chief factor for $L_\lambda$.
\end{proof}

\begin{lemma}\label{GoodAction1} 
Suppose that $b>3$ and \cref{CommonHyp} is satisfied. If $Z_{\alpha}$, $V^\alpha/Z_{\alpha}$ and $V_{\alpha}^{(2)}/V^{\alpha}$ are FF-modules or trivial modules for $\bar{L_{\alpha}}$, then $R_{\alpha}=C_{L_{\alpha}}(V_{\alpha}^{(2)})Q_{\alpha}$. 
\end{lemma}
\begin{proof}
Of the configurations described in \hyperlink{MainGrpThm}{Theorem C} which satisfy $b>2$, all satisfy $R_{\alpha}=Q_{\alpha}$ and so we may assume throughout that $G$ is a minimal counterexample to \hyperlink{MainGrpThm}{Theorem C} such that $R_{\alpha}\ne C_{L_{\alpha}}(V_{\alpha}^{(2)})Q_{\alpha}$. 

Suppose first that $|V_{\beta}|\ne q^3$ so that $V^{\alpha}/Z_{\alpha}$ contains a non-central chief factor for $L_{\alpha}$. Since $L_{\alpha}/R_{\alpha}\cong \SL_2(q)$ and $Q_{\alpha}\in\syl_p(R_{\alpha})$, $|S/Q_{\alpha}|=q$ and by \cref{SEFF}, $L_{\alpha}/C_{L_{\alpha}}(V^\alpha/Z_{\alpha})\cong L_{\alpha}/C_{L_{\alpha}}(V_{\alpha}^{(2)}/V^\alpha)\cong \SL_2(q)$. Thus, if $C_{L_{\alpha}}(V^\alpha/Z_{\alpha})\ne R_{\alpha}$, a standard calculation yields that $q\in\{2,3\}$. Moreover, if $p=3$ and $R_{\alpha}C_{L_{\alpha}}(V^\alpha/Z_{\alpha})S< L_{\alpha}$, then $|R_{\alpha}C_{L_{\alpha}}(V^\alpha/Z_{\alpha})/R_{\alpha}|=2$, $|L_{\alpha}/C_{L_{\alpha}}(V^\alpha)Q_{\alpha}|=2^4.3$ and \cref{Badp3} (ii) gives a contradiction. Hence, if $C_{L_{\alpha}}(V^\alpha/Z_{\alpha})\ne R_{\alpha}$ then $L_{\alpha}=R_{\alpha}C_{L_{\alpha}}(V^\alpha/Z_{\alpha})S$. But now, $C_{L_{\alpha}}(V^\alpha/Z_{\alpha})$ normalizes $Z_{\alpha}C_{V_{\beta}}(O^p(L_{\beta}))$ and so normalizes $[Z_{\alpha}C_{V_{\beta}}(O^p(L_{\beta})), Q_{\alpha}]=Z_{\beta}$, a contradiction for then $Z_{\beta}\normaleq L_{\alpha}$. Thus, $C_{L_{\alpha}}(V^\alpha/Z_{\alpha})=R_{\alpha}$. Similarly, considering $C_{L_{\alpha}}(V_{\alpha}^{(2)}/V^\alpha)$, we have that $V_{\beta}V^\alpha\normaleq C_{L_{\alpha}}(V_{\alpha}^{(2)}/V^\alpha)$ and so $Z_{\alpha}C_{V_{\beta}}(O^p(L_{\beta}))=Z_{\alpha}[V_{\beta},Q_{\alpha}]=[V_{\beta}V^\alpha, Q_{\alpha}]\normaleq C_{L_{\alpha}}(V_{\alpha}^{(2)}/V^\alpha)$. Then $[Z_{\alpha}C_{V_{\beta}}(O^p(L_{\beta})), Q_{\alpha}]=Z_{\beta}$ is normalized by $C_{L_{\alpha}}(V_{\alpha}^{(2)}/V^\alpha)$ and, as above, we conclude that $C_{L_{\alpha}}(V_{\alpha}^{(2)}/V^\alpha)=R_{\alpha}$ and the result holds.

Hence, we may assume that $|V_{\beta}|=q^3$. Since \cref{CommonHyp} is satisfied, $V_{\alpha}^{(2)}/Z_{\alpha}$ is an FF-module and $C_{L_{\alpha}}(V_{\alpha}^{(2)}/Z_{\alpha})\cap R_{\alpha}=C_{L_{\alpha}}(V_{\alpha}^{(2)})Q_\alpha$. Then, $O^p(C_{L_{\alpha}}(V_{\alpha}^{(2)})$ centralizes $Q_{\alpha}/C_{Q_{\alpha}}(V_{\alpha}^{(2)})$ and so normalizes $Q_{\alpha}\cap Q_{\beta}>C_{Q_{\alpha}}(V_{\alpha}^{(2)}$. We apply \cref{SubAmal}, taking $\lambda=\alpha$. As $b>3$ and $V_{\alpha}^{(2)}/Z_{\alpha}$ is an FF-module (so admits quadratic action), outcome (a) does not hold. Since $\lambda=\alpha$ outcome (c) does not hold.

Suppose (d) holds. Then, by construction, $\langle V_{\beta}^{H_{\alpha}}\rangle=\langle V_{\beta}^{G_{\alpha}}\rangle=V_{\alpha}^{(2)}$ from which it follows that $V_{\beta}^{(3)}\le V:=\langle Z_{\beta}^X \rangle$ and the images of both $Q_{\beta}/C_{\beta}$ and $C_{\beta}/C_{Q_{\beta}}(V_{\beta}^{(3)})$ in $\wt L_{\beta}$ contain a non-central chief factor for $\wt L_{\beta}$. By \cref{SubAmal}, $\wt X\cong \mathrm{G}_2(2)$. It follows from the structure of $\mathrm{G}_2(2)$ that $|Q_{\alpha}/C_{\beta}|=2^2$, $|Q_{\alpha}/C_{Q_{\alpha}}(V_{\alpha}^{(2)})|=2^4$ and $|\wt{C_{Q_{\alpha}}(V_{\alpha}^{(2)})}|=2$. Then, $C_S(V)=C_{Q_{\beta}}(V_{\beta}^{(3)})\normaleq X$. By \cref{Badp2} (iii), there are four non-equal subgroups of $L_{\alpha}/C_{L_{\alpha}}(V_{\alpha}^{(2)})Q_{\alpha}\cong(3\times 3):2$ isomorphic to $\Sym(3)$, and so there is $H_{\alpha}^*\ne H_{\alpha}$ such that $S\in H_{\alpha}^*$, $O^2(H_{\alpha}^*)$ acts non-trivially on $V_{\alpha}^{(2)}/Z_{\alpha}$ and $Z_{\alpha}$ and $G_{\alpha}=\langle H_{\alpha}, H_{\alpha}^*\rangle$. If $H_{\alpha}^*$ does not normalize $Q_{\alpha}\cap Q_{\beta}$, then setting $X^*$ for the subgroup of $G$ obtained from employing the method in \cref{SubAmal} with $H_{\alpha}^*$ instead of $H_{\alpha}$, it follows from the work above that $X^*$ also satisfies outcome (d) and for $V^*:=\langle Z_{\beta}^{X^*}\rangle$, $C_S(V)=C_S(V^*)=C_{Q_{\beta}}(V_{\beta}^{(3)})\normaleq \langle H_{\alpha}, H_{\alpha}^*\rangle=G_{\alpha}$, a contradiction. Hence, $H_{\alpha}^*$ normalizes $Q_{\alpha}\cap Q_{\beta}$. Choose $\tau$ in $C_{L_{\alpha}}(V_{\alpha}^{(2)}/Z_{\alpha})\setminus  C_{L_{\alpha}}(V_{\alpha}^{(2)})$. Then $\tau$ normalizes $V_{\beta}$ so normalizes $C_{\beta}=C_{Q_{\alpha}}(V_{\beta})$, and $G_{\alpha}=\langle \tau, H_{\alpha}^*\rangle$. If $\tau$ centralizes $Q_{\alpha}/C_{\beta}$, then $\tau$ normalizes $Q_{\alpha}\cap Q_{\beta}$ so that $G_{\alpha}$ normalizes $Q_{\alpha}\cap Q_{\beta}$, a contradiction by \cref{push}. Thus, $\tau$ acts non-trivially on $Q_{\alpha}/C_{\beta}$ and as $|Q_{\alpha}/C_{\beta}|=4$, $C_{\beta}=(Q_{\alpha}\cap Q_{\beta})\cap (Q_{\alpha}\cap Q_{\beta})^\tau$. Now, $[O^2(H_{\alpha}^*), \tau]\le C_{G_{\alpha}}(V_{\alpha}^{(2)})$ and as $O^2(H_{\alpha}^*)$ normalizes $Q_{\alpha}\cap Q_{\beta}$, $O^2(H_{\alpha}^*)$ normalizes $(Q_{\alpha}\cap Q_{\beta})^\tau$. But then $H_{\alpha}^*$ normalizes $C_{\beta}=Q_{\alpha}\cap Q_{\beta}\cap Q_{\beta}^\tau$ and so $G_{\alpha}=\langle \tau, H_{\alpha}^*\rangle$ normalizes $C_{\beta}$, another contradiction. 

Thus, we may assume that outcome (b) of \cref{SubAmal} holds so that $p=3$ and neither $R_{\alpha}$ nor $C_{L_{\alpha}}(V_{\alpha}^{(2)}/Z_{\alpha})$ normalizes $Q_{\alpha}\cap Q_{\beta}$. Indeed, for the subgroup $H_{\alpha}$ as constructed in \cref{SubAmal}, we have that $Q_{\alpha}\cap Q_{\beta}\normaleq H_{\alpha}$. Now, $C_{L_{\alpha}}(V_{\alpha}^{(2)}/Z_{\alpha})$ normalizes $C_{\beta}$ and we may assume that it acts non-trivially on $Q_{\alpha}/C_{\beta}$ for otherwise $Q_{\alpha}\cap Q_{\beta}\normaleq G_{\alpha}=\langle H_{\alpha}, C_{L_{\alpha}}(V_{\alpha}^{(2)}/Z_{\alpha})\rangle$, a contradiction by \cref{push}. Furthermore, $[O^3(O^{3'}(H_{\alpha})), C_{L_{\alpha}}(V_{\alpha}^{(2)}/Z_{\alpha})]\le C_{L_{\alpha}}(V_{\alpha}^{(2)})G_{\alpha}^{(1)}$ and as $H_{\alpha}$ normalizes $Q_{\alpha}\cap Q_{\beta}$ and $O^3(C_{L_{\alpha}}(V_{\alpha}^{(2)}))$ centralizes $Q_{\alpha}/C_{\beta}$, it follows that for any $r\in C_{L_{\alpha}}(V_{\alpha}^{(2)}/Z_{\alpha})$ of order coprime to $p$ which does not normalize $Q_{\alpha}\cap Q_{\beta}$, $O^3(O^{3'}(H_{\alpha}))$ normalizes $(Q_{\alpha}\cap Q_{\beta})^r$ and $H_{\alpha}$ normalizes $C_{\beta}=Q_{\alpha}\cap Q_{\beta}\cap Q_{\beta}^r$, a final contradiction for then $C_{\beta}\normaleq G_{\alpha}=\langle H_{\alpha}, C_{L_{\alpha}}(V_{\alpha}^{(2)}/Z_{\alpha})\rangle$.
\end{proof}

\begin{lemma}\label{GoodAction2}
Suppose that $b>5$ and \cref{CommonHyp} is satisfied. If $O^p(R_{\alpha})$ centralizes $V_{\alpha}^{(2)}$ and $V_{\alpha}^{(4)}/V_{\alpha}^{(2)}$ contains a unique non-central chief factor which, as a $\mathrm{GF}(p)\bar{L_{\alpha}}$-module, is an FF-module then $O^p(R_{\alpha})$ centralizes $V_{\alpha}^{(4)}$.
\end{lemma}
\begin{proof}
Since none of the configurations described in \hyperlink{MainGrpThm}{Theorem C} have $b>5$, we may assume that $G$ is a minimal counterexample such that $O^p(R_{\alpha})$ does not centralize $V_{\alpha}^{(4)}/V_{\alpha}^{(2)}$, $V_{\alpha}^{(4)}/V_{\alpha}^{(2)}$ contains a unique non-central chief factor and $O^p(R_{\alpha})$ centralizes $V_{\alpha}^{(2)}$. Since $O^p(R_{\alpha})$ centralizes $V_{\alpha}^{(2)}$, an application of the three subgroup lemma implies that $O^p(R_{\alpha})$ centralizes $Q_{\alpha}/C_{Q_{\alpha}}(V_{\alpha}^{(2)})$ and $C_{Q_{\alpha}}(V_{\alpha}^{(2)})\le Q_{\alpha}\cap Q_{\beta}$, $Q_{\alpha}\cap Q_{\beta}\normaleq R_{\alpha}$. 

We may apply \cref{SubAmal} with $\lambda=\alpha$. Since $b>5$, (a) is not satisfied. Indeed, as $\lambda=\alpha$ and $R_{\alpha}$ normalizes $Q_{\alpha}\cap Q_{\beta}$, we suppose that conclusion (d) is satisfied. For $X$ as constructed in \cref{SubAmal}, we have that $V_{\beta}^{(5)}\le V:=\langle Z_{\beta}^X\rangle$ and the images in $\wt L_{\beta}$ of $Q_{\beta}/C_{\beta}$, $C_{\beta}/C_{Q_{\beta}}(V_{\beta}^{(3)})$ and $C_{Q_{\beta}}(V_{\beta}^{(3)})/C_{Q_{\beta}}(V_{\beta}^{(5)})$ all contain a non-central chief factor for $\wt L_{\beta}$, a contradiction by \cref{SubAmal}.
\end{proof}

\begin{lemma}\label{GoodAction3}
Suppose that $b>3$ and \cref{CommonHyp} is satisfied. If $V_{\beta}^{(3)}/V_{\beta}$ contains a unique non-central chief factor which, as a $\mathrm{GF}(p)\bar{L_{\beta}}$-module, is an FF-module, then $O^p(R_{\beta})$ centralizes $V_{\beta}^{(3)}$.
\end{lemma}
\begin{proof}
Since the only configuration in \hyperlink{MainGrpThm}{Theorem C} which satisfies $b>3$ (where $G$ is parabolic isomorphic to $\mathrm{F}_3$) satisfies $[O^p(R_{\beta}), V_{\beta}^{(3)}]=\{1\}$, we may assume that $G$ is a minimal counterexample such that $O^p(R_{\beta})$ does not centralize $V_{\beta}^{(3)}$. Since $O^p(R_{\beta})$ centralizes $V_{\beta}$, the three subgroup lemma implies that $O^p(R_{\beta})$ centralizes $Q_{\beta}/C_{\beta}$ so that $R_{\beta}$ normalizes $Q_{\alpha}\cap Q_{\beta}$. Thus, the hypotheses of \cref{SubAmal} are satisfied with $\lambda=\beta$. Since $C_{L_{\beta}}(V_{\beta}^{(3)}/V_{\beta})$ normalizes $V_{\alpha}^{(2)}$ and $\lambda=\beta$, conclusions (b) and (c) are not satisfied. As $b>3$, if outcome (a) is satisfied then $b=5$ and $(G_{\alpha}, H_{\beta}, G_{\alpha,\beta})$ is parabolic isomorphic to $\mathrm{F}_3$ and $H_{\beta}/Q_{\beta}\cong \GL_2(3)$. Then $S$ is determined up to isomorphism. Indeed, as $V_{\beta}=\langle Z_{\alpha}^{G_{\beta}}\rangle=\langle Z_{\alpha}^{H_{\beta}}\rangle=Z_3(S)$, $Q_{\beta}=C_S(Z_3(S)/Z(S))$ is uniquely determined in $S$, and so is uniquely determined up to isomorphism. But then one can check (e.g. employing MAGMA) that $\Phi(Q_{\beta})=C_{\beta}$ has index $9$ in $Q_{\beta}$, and as $\bar{G_{\beta}}$ acts faithfully on $Q_{\beta}/\Phi(Q_{\beta})$, $\bar{G_{\beta}}=\bar{H_{\beta}}\cong \GL_2(3)$ and $G_{\beta}=H_{\beta}$, a contradiction.

Hence, we are left with conclusion (d). But then $V_{\alpha}^{(4)}\le V:=\langle Z_{\beta}^X\rangle$ and the images of $Q_{\alpha}/C_{Q_{\alpha}}(V_{\alpha}^{(2)})$ and $C_{Q_{\alpha}}(V_{\alpha}^{(2)})/ C_{Q_{\alpha}}(V_{\alpha}^{(4)})$ in $\wt L_{\alpha}$ both contain a non-central chief factor for $\wt L_{\alpha}$. Moreover, the images of $Q_{\beta}/C_{\beta}$ and $C_{\beta}/C_{Q_{\beta}}(V_{\beta}^{(3)})$ also a contain non-central chief factor for $\wt L_{\beta}$, and we have a contradiction.
\end{proof}

\begin{lemma}\label{GoodAction4}
Suppose that $b>5$ and \cref{CommonHyp} is satisfied. If $V_{\beta}^{(5)}/V_{\beta}^{(3)}$ contains a unique non-central chief factor which, as a $\bar{L_{\beta}}$-module, is an FF-module and $O^p(R_{\beta})$ centralizes $V_{\beta}^{(3)}$, then $[O^p(R_{\beta}), V_{\beta}^{(5)}]=\{1\}$. 
\end{lemma}
\begin{proof}
Since none of the configurations in \hyperlink{MainGrpThm}{Theorem C} satisfy $b>5$, we may assume the $G$ is a minimal counterexample to \hyperlink{MainGrpThm}{Theorem C} with $[O^p(R_{\beta}), V_{\beta}^{(3)}]=\{1\}$ and $[O^p(R_{\beta}), V_{\beta}^{(5)}]\ne \{1\}$. Since $O^p(R_{\beta})$ centralizes $V_{\beta}$, $O^p(R_{\beta})$ centralizes $Q_{\beta}/C_{\beta}$ so that $R_{\beta}$ normalizes $Q_{\alpha}\cap Q_{\beta}$ we may apply \cref{SubAmal} with $\lambda=\beta$. Since $O^p(R_{\beta})$ normalizes $V_{\alpha}^{(2)}$ and $b>5$, we are in case (d) of \cref{SubAmal}. Then, $V_{\alpha}^{(6)}\le V:=\langle Z_{\beta}^X\rangle$ and the image of $Q_{\alpha}/C_{Q_{\alpha}}(V_{\alpha}^{(6)})$ in $\wt L_{\alpha}$ contains at least three non-central chief factors for $\wt L_{\alpha}$, a contradiction.
\end{proof}

%% file: Contents/6.1.SymmetricCasei.tex
\section{$Z_{\alpha'}\not\le Q_\alpha$}\label{evensec}

Throughout this section, we assume \cref{MainHyp}. In addition, within this section we suppose that $Z_{\alpha'}\not\le Q_\alpha$ for a chosen critical pair $(\alpha, \alpha')$. By \cref{p-closure2} (iv), this condition is equivalent to $[Z_\alpha, Z_{\alpha'}]\ne \{1\}$. Throughout, we set $S\in\syl_p(G_{\alpha,\beta})$ and $q_{\lambda}=|\Omega(Z(S/Q_{\lambda}))|$ for any $\lambda\in\Gamma$.

\begin{lemma}\label{beven}
$(\alpha', \alpha)$ is also a critical pair, $Q_{\alpha'}\in\syl_p(R_{\alpha'})$, $C_{Z_{\alpha}}(Z_{\alpha'})=Z_\alpha\cap Q_{\alpha'}$ and $C_{Z_{\alpha'}}(Z_{\alpha})=Z_{\alpha'}\cap Q_{\alpha}$.
\end{lemma}
\begin{proof}
Since $Z_{\alpha'}\not\le Q_\alpha$ we have that both $(\alpha, \alpha')$ and $(\alpha', \alpha)$ are critical pairs. In particular, all the results we prove in this section hold upon interchanging $\alpha$ and $\alpha'$. By \cref{critpair2}, $C_{Z_{\alpha}}(Z_{\alpha'})=Z_\alpha\cap Q_{\alpha'}$.
\end{proof}

\begin{lemma}\label{SL2}
For $\lambda\in\{\alpha, \alpha'\}$, the following hold:
\begin{enumerate}
    \item $S=Q_{\alpha}Q_{\beta}=Z_{\alpha'}Q_{\alpha}\in\syl_p(G_{\alpha,\beta})$ and $N_{G_{\alpha}}(S)=N_{G_{\beta}}(S)=G_{\alpha,\beta}$;
    \item $Z_\alpha Q_{\alpha'}\in\syl_p(G_{\alpha',\alpha'-1})$; and
    \item $Z_\lambda/\Omega(Z(L_\lambda))$ is natural $\SL_2(q_\lambda)$-module for $L_\lambda/R_\lambda\cong\SL_2(q_\lambda)$, and $q_\alpha=q_{\alpha'}$.
\end{enumerate}
\end{lemma}
\begin{proof}
Without loss of generality, assume that $|Z_{\alpha}Q_{\alpha'}/Q_{\alpha'}|\leq |Z_{\alpha'}Q_{\alpha}/Q_\alpha|$. By \cref{beven}, we have that 
\begin{align*}
|Z_\alpha/C_{Z_\alpha}(Z_{\alpha'})|&=|Z_\alpha/Z_\alpha\cap Q_{\alpha'}|=|Z_{\alpha}Q_{\alpha'}/Q_{\alpha'}|\\
&\leq|Z_{\alpha'}Q_{\alpha}/Q_\alpha|=|Z_{\alpha'}/Z_{\alpha'}\cap Q_{\alpha}|=|Z_{\alpha'}/C_{Z_{\alpha'}}(Z_{\alpha})|.
\end{align*}
Thus, $Z_{\alpha'}$ is a non-trivial offender on $Z_\alpha$, and $Z_\alpha$ is an FF-module for $L_{\alpha}/R_{\alpha}$. Since $L_{\alpha}/R_{\alpha}$ has a strongly $p$-embedded subgroup by \cref{spelemma2}, using \cref{SEFF} we conclude that $L_\alpha/R_{\alpha}\cong\SL_2(q)$ and $Z_\alpha/\Omega(Z(L_{\alpha}))$ is a natural $\SL_2(q)$-module.

Since $L_\alpha/R_\alpha\cong\SL_2(q)$ and $Z_{\alpha}/\Omega(Z(L_{\alpha}))$ is a natural $\SL_2(q)$-module, we infer that $q=|Z_\alpha/C_{Z_\alpha}(Z_{\alpha'})|\leq|Z_{\alpha'}/C_{Z_{\alpha'}}(Z_{\alpha})|=|Z_{\alpha'}Q_{\alpha}/Q_{\alpha}|\leq q$. In particular, by a symmetric argument, $Z_{\alpha'}/\Omega(Z(L_{\alpha'}))$ is also a natural module for $L_{\alpha'}/R_{\alpha'}\cong\SL_2(q)$. It follows immediately that $Z_{\alpha'}Q_{\alpha}\in\syl_p(G_{\alpha,\beta})$ and $Z_\alpha Q_{\alpha'}\in\syl_p(G_{\alpha', \alpha'-1})$. Since $Z_{\alpha'}\le Q_{\beta}$, we have that $S=Q_{\alpha}Q_{\beta}$ and since $G_{\alpha, \beta}=G_{\alpha}\cap G_{\beta}$ normalizes $Q_{\alpha}Q_{\beta}$, the result holds.
\end{proof}

In the following proposition, we divide the analysis of the case $[Z_{\alpha}, Z_{\alpha'}]\ne\{1\}$ into two subcases. The remainder of this section is split into two subsections dealing with each of these subcases individually.

\begin{proposition}\label{BetaCenterIII}
One of the following holds:
\begin{enumerate}
    \item $b$ is even and $Z_\beta=\Omega(Z(S))=\Omega(Z(L_\beta))$; or
    \item $Z_{\beta}\ne \Omega(Z(S))$ and for all $\lambda\in\Gamma$, $Z_\lambda/\Omega(Z(L_\lambda))$ is a natural $\SL_2(q_\lambda)$-module for $L_\lambda/R_\lambda$.
    \end{enumerate}
\end{proposition}
\begin{proof}
Notice that if $Z_\beta=\Omega(Z(S))$ then $\{1\}=[Z_\beta, S]^{G_\beta}=[Z_\beta, \langle S^{G_\beta}\rangle]=[Z_\beta, L_\beta]$ so that $Z_\beta=\Omega(Z(L_\beta))$. Since $Z_{\alpha'}$ is not centralized by $Z_{\alpha}\le L_{\alpha'}$, it follows immediately in this case that $b$ is even.

Suppose that $Z_{\beta}\ne \Omega(Z(S))$. If $b=1$, the result follows immediately from \cref{SL2} replacing $\alpha'$ by $\beta$ and so we may assume that $b>1$. Assume that $V_{\alpha}\le Q_{\alpha'-1}$. In particular, $V_{\alpha}\le Z_{\alpha}Q_{\alpha'}\in\syl_p(L_{\alpha'})$ by \cref{SL2}. Thus, $[V_{\alpha}, Z_{\alpha'}]\le [Z_\alpha, Z_{\alpha'}]\le Z_\alpha$ so that $[V_{\alpha}, O^p(L_\alpha)]\le Z_{\alpha}$ and $Z_{\alpha}Z_{\beta}\normaleq L_{\alpha}$, a contradiction by \cref{push}. Hence, there is $\alpha-1\in\Delta(\alpha)$ with $Z_{\alpha-1}\not\le Q_{\alpha'-1}$. Then $(\alpha-1, \alpha'-1)$ is a critical pair and since $Z_{\alpha}\ne \Omega(Z(S))\ne Z_{\beta}$, by \cref{p-closure2} (ii), we conclude that $[Z_{\alpha-1}, Z_{\alpha'-1}]\ne\{1\}$ and \cref{SL2} gives the result.
\end{proof}

\subsection{$Z_{\beta}\ne \Omega(Z(S))$}

We first consider the case where $[Z_{\alpha}, Z_{\alpha'}]\ne\{1\}$ and $Z_{\beta}\ne \Omega(Z(S))$. Under these hypotheses, and using the symmetry in $\alpha$ and $\alpha'$, it is not hard to show that every $\gamma\in\Gamma$ belongs to some critical pair. The main work in this subsection is then to show that $R_\gamma=Q_\gamma$ and $\bar{L_{\gamma}}\cong\SL_2(q)$, for then, all examples we obtain arise from weak BN-pairs of rank $2$ and $G$ is determined by \cite{Greenbook}.

As was made apparent in \cref{SubAmal}, there is a clear distinction between the cases where $p\in\{2,3\}$ and $p\geq 5$ due to the solvability of $\SL_2(p)$ when $p\in\{2,3\}$. Throughout this subsection, and the subsections to come, this dichotomy will become a prominent theme.

\begin{lemma}\label{BetaCenter}
Suppose that $Z_{\beta}\ne \Omega(Z(S))$, $b>1$ and for $\lambda\in\{\alpha, \beta\}$, $Z_\lambda/\Omega(Z(L_\lambda))$ is a natural $\SL_2(q_\lambda)$-module for $L_\lambda/R_\lambda$. Then the following hold:
\begin{enumerate}
    \item $V_{\alpha}\not\le Q_{\alpha'-1}$ and there is a critical pair $(\alpha-1, \alpha'-1)$ with $[Z_{\alpha-1}, Z_{\alpha'-1}]\ne\{1\}$ and $V_{\alpha-1}\not\le Q_{\alpha'-2}$;
    \item $V_{\lambda}/Z_{\lambda}$ and $Z_{\lambda}$ are FF-modules for $\bar{L_{\lambda}}$;
    \item $q_{\alpha}=q_{\beta}$; and
    \item unless $q_{\lambda}\in\{2,3\}$, $R_{\lambda}=C_{L_{\lambda}}(V_{\lambda}/Z_{\lambda})$ and $L_{\lambda}/C_{L_{\lambda}}(V_{\lambda})Q_{\lambda}\cong \SL_2(q_\lambda)$. 
\end{enumerate}
\end{lemma}
\begin{proof}
By the proof of \cref{BetaCenterIII}, we have that $V_\alpha\not\le Q_{\alpha'-1}$. In particular, there is some $\alpha-1\in\Delta(\alpha)$ such that $(\alpha-1, \alpha'-1)$ is a critical pair with $[Z_{\alpha-1}, Z_{\alpha'-1}]\ne\{1\}$. By a similar reasoning, $V_{\alpha-1}\not\le Q_{\alpha'-2}$ else we arrive at a similar contradiction to the above. Hence (i) holds.

Suppose first that $b$ is odd. Then, by \cref{SL2}, \cref{BetaCenterIII} and as $\alpha'$ is conjugate to $\beta$, $L_{\beta}/R_{\beta}\cong \SL_2(q_\beta)$ and $q_{\beta}=q_{\alpha'}=q_{\alpha}$ and (iii) holds in this case. Now suppose that $b$ is even so $\alpha'-1$ is conjugate to $\beta$. Then, we observe that $V_{\alpha}\cap Q_{\alpha'-1}=Z_{\alpha}(V_{\alpha}\cap Q_{\alpha'})$ has index at most $q_{\beta}$ in $V_{\alpha}$ and is centralized, modulo $Z_{\alpha}$, by $Z_{\alpha'}$. Furthermore, since $Z_{\alpha}Z_{\beta}\not\normaleq L_{\alpha}$, it follows from \cref{p-closure} (iii) that $Q_{\alpha}\in\syl_p(C_{L_{\alpha}}(V_{\alpha}/Z_{\alpha}))$ and by \cref{SEFF}, we have that $q_{\alpha}\leq q_{\beta}$. But then $(\alpha-1, \alpha'-1)$ is also a critical pair with $V_{\alpha-1}\cap Q_{\alpha'-2}=Z_{\alpha-1}(V_{\alpha-1}\cap Z_{\alpha'-1})$ a subgroup of $V_{\alpha-1}$ of index at most $q_{\alpha}$ and applying the same reasoning as before alongside \cref{p-closure} (iii), we deduce that $Q_{\beta}\in\syl_p(C_{L_{\beta}}(V_{\beta}/Z_{\beta}))$ and using \cref{SEFF} we see that $q_{\alpha-1}=q_{\beta}\leq q_{\alpha}$. Thus, $q_{\alpha}=q_{\beta}$ and in either casem we have that $V_{\lambda}/Z_{\lambda}$ is an FF-module for $\bar{L_{\lambda}}$ for all $\lambda\in\Gamma$, and (ii) and (iii) hold.

It remains to prove (iv). By \cref{SEFF}, for all $\lambda\in\Gamma$, $L_{\lambda}/C_{L_{\lambda}}(V_{\lambda}/Z_{\lambda})\cong L_{\lambda}/R_{\lambda}\cong \SL_2(q_\lambda)$. Suppose that $q_{\lambda}\not\in\{2,3\}$ and assume that $C_{L_{\lambda}}(V_{\lambda}/Z_{\lambda})\ne R_{\lambda}$. Since $\{Q_{\lambda}\}=\syl_p(C_{L_{\lambda}}(V_{\lambda}/Z_{\lambda}))=\syl_p(R_{\lambda})$, we infer that $\bar{R_{\lambda}C_{L_{\lambda}}(V_{\lambda}/Z_{\lambda})}$ is a group of order coprime to $p$ and we see immediately that $p$ is odd, $C_{L_{\lambda}}(V_{\lambda}/Z_{\lambda})R_{\lambda}/R_{\lambda}=Z(L_{\lambda}/R_{\lambda})$ and $C_{L_{\lambda}}(V_{\lambda}/Z_{\lambda})R_{\lambda}/C_{L_{\lambda}}(V_{\lambda}/Z_{\lambda})=Z(L_{\lambda}/C_{L_{\lambda}}(V_{\lambda}/Z_{\lambda}))$. Thus, $L_{\lambda}/(C_{L_{\lambda}}(V_{\lambda}/Z_{\lambda})\cap R_{\lambda})$ is isomorphic to a central extension of $\PSL_2(q_\lambda)$ by an elementary abelian group of order $4$. Since $L_{\lambda}=O^{p'}(L_{\lambda})$ and the $2$-part of the Schur multiplier of $\PSL_2(q)$ is of order $2$ by \cref{SLGen} (vii) when $p$ is odd, we have a contradiction. Thus, we shown that, unless $q_{\lambda}\in\{2,3\}$, $C_{L_{\lambda}}(V_{\lambda}/Z_{\lambda})= R_{\lambda}$ and (iv) is proved.
\end{proof}

\begin{lemma}\label{BetaCenterI}
Suppose that for $Z_{\beta}\ne \Omega(Z(S))$ and for $\lambda\in\{\alpha, \beta\}$, $Z_\lambda/\Omega(Z(L_\lambda))$ is a natural $\SL_2(q_\lambda)$-module for $L_\lambda/R_\lambda$. Then $b\leq 2$.
\end{lemma}
\begin{proof}
Assume throughout that $b>2$ so that $V_{\lambda}$ is abelian for all $\lambda\in\Gamma$. For $\delta\in \Gamma$ and $\nu\in\Delta(\delta)$, set $S_{\delta,\nu}\in\syl_p(G_{\delta, \nu})$ and $Z_{\delta, \nu}:=\Omega(Z(S_{\delta,\nu}))$. Choose $\mu\in\Delta(\alpha'-1)$ such that $Z_{\mu, \alpha'-1}\ne Z_{\alpha'-1, \alpha'-2}$. Thus, $Z_{\alpha'-1}=Z_{\mu, \alpha'-1}Z_{\alpha'-1, \alpha'-2}$. Then, using \cref{BetaCenter} (i), as $V_{\alpha}\not\le Q_{\alpha'-1}$ and $V_{\alpha}$ centralizes $Z_{\alpha'-1, \alpha'-2}$, we have that $L_{\alpha'-1}=\langle Q_{\mu}, R_{\alpha'-1}, V_{\alpha}\rangle$. 

Set $U_{\alpha'-1, \mu}:=\langle Z_{\delta} \mid Z_{\mu, \alpha'-1}=Z_{\delta, \alpha'-1}, \delta\in\Delta(\alpha'-1)\rangle$. Let $r\in R_{\alpha'-1}Q_{\mu}$. Since $r$ is an automorphism of the graph, it follows that for $Z_{\delta}$ with $Z_{\mu, \alpha'-1}=Z_{\delta, \alpha'-1}$ and $\delta\in\Delta(\alpha'-1)$, we have that $Z_{\delta}^r=Z_{\delta\cdot r}$ and $\{\delta, \alpha'-1\}\cdot r=\{\delta\cdot r, \alpha'-1\}$. Since $S_{\delta, \alpha'-1}$ is the unique Sylow $p$-subgroup of $G_{\delta, \alpha'-1}$, it follows that $Z_{\delta, \alpha'-1}^r=Z_{\delta\cdot r, \alpha'-1}$. Since $R_{\alpha'-1}Q_{\mu}$ normalizes $Z_{\delta, \alpha'-1}$, we have that $Z_{\delta\cdot r, \alpha'-1}=Z_{\mu, \alpha'-1}$ so that $Z_{\delta\cdot r}\le U_{\alpha'-1,\mu}$. Thus, $U_{\alpha'-1,\mu}\normaleq R_{\alpha'-1}Q_{\mu}$.

Suppose that $U_{\alpha'-1, \mu}\le Q_{\alpha}$. By \cref{BetaCenter} (i), there is $\alpha-1\in\Delta(\alpha)$ such that $Z_{\alpha-1}\not\le Q_{\alpha'-1}$ and $Z_{\alpha'-1}\not\le Q_{\alpha-1}$. Moreover, we have that $L_{\alpha'-1}=\langle Q_{\mu}, R_{\alpha'-1}, Z_{\alpha-1}\rangle$. Then, $U_{\alpha'-1,\mu}=Z_{\alpha'-1}(U_{\alpha'-1,\mu}\cap Q_{\alpha-1})$ is centralized, modulo $Z_{\alpha'-1}$, by $Z_{\alpha-1}$ so that $U_{\alpha'-1,\mu}\normaleq L_{\alpha'-1}=\langle Q_{\mu}, R_{\alpha'-1}, Z_{\alpha-1}\rangle$. Since $Z_{\alpha-1}$ centralizes $U_{\alpha'-1,\mu}/Z_{\alpha'-1}$, $O^p(L_{\alpha'-1})$ centralizes $U_{\alpha'-1,\mu}/Z_{\alpha'-1}$ and $Z_{\mu}Z_{\alpha'-1}\normaleq L_{\alpha'-1}$, a contradiction by \cref{push}. Thus, $U_{\alpha'-1,\mu}\not\le Q_{\alpha}$.

Hence, there is $\delta\in\Delta(\alpha'-1)$ with $Z_{\delta, \alpha'-1}=Z_{\mu, \alpha'-1}\ne Z_{\alpha'-1, \alpha'-2}$, $L_{\alpha'-1}=\langle Q_{\delta}, R_{\alpha'-1}, V_{\alpha}\rangle$ and $(\alpha,\delta)$ a critical pair. We may as well assume that $\delta=\alpha'$ and $Z_{\alpha', \alpha'-1}\ne Z_{\alpha'-1, \alpha'-2}$. By \cref{beven}, \cref{BetaCenter} applies to $\alpha'$ in place of $\alpha$. Then $V_{\alpha'}\not\le Q_{\beta}$ and there is $\alpha'+1\in\Delta(\alpha')$ with $(\alpha'+1, \beta)$ a critical pair satisfying $Z_{\alpha'+1}\not\le Q_{\beta}$ and $Z_{\beta}\not\le Q_{\alpha'+1}$. Choose $\mu^*\in\Delta(\alpha')$ such that $Z_{\mu^*, \alpha'}\ne Z_{\alpha', \alpha'-1}$ so that $Z_{\alpha'}=Z_{\mu^*, \alpha'}Z_{\alpha', \alpha'-1}$. Then, as $Z_{\alpha}\not\le Q_{\alpha'}$ and $Z_{\alpha}$ centralizes $Z_{\alpha', \alpha'-1}$, we have that $L_{\alpha'}=\langle Z_{\alpha}, Q_{\mu^*}, R_{\alpha'}\rangle$. Forming $U_{\alpha', \mu^*}$ in an analogous way to $U_{\alpha'-1,\mu}$, we see that $U_{\alpha',\mu^*}\normaleq R_{\alpha'}Q_{\mu^*}$ and $U_{\alpha',\mu^*}\not\le Q_{\beta}$. Thus, there is some $\delta^*$ with $Z_{\delta^*, \alpha'}\ne Z_{\alpha', \alpha'-1}$, $L_{\alpha'}=\langle Q_{\delta^*}, R_{\alpha'}, Z_{\alpha}\rangle$ and $(\beta,\delta^*)$ a critical pair. We may as well label $\delta^*=\alpha'+1$ so that $L_{\alpha'}=\langle Z_{\alpha}, Q_{\alpha'+1}, R_{\alpha'}\rangle$ and $Z_{\alpha'+1, \alpha'}\ne Z_{\alpha', \alpha'-1}$ 

Now, let $R:=[Z_{\beta}, Z_{\alpha'+1}]\le Z_{\beta}\cap Z_{\alpha'+1}$. Then $R$ is centralized by $Z_{\beta}Q_{\alpha'+1}\in\syl_p(G_{\alpha'+1,\alpha'})$ so that $R\le Z_{\alpha'+1, \alpha'}$. Since $b>1$, $Z_{\alpha}$ centralizes $R\le Z_{\beta}$ and so $R$ is centralized by $L_{\alpha'}=\langle Q_{\alpha'+1}, R_{\alpha'}, Z_{\alpha}\rangle$ and $R\le Z(L_{\alpha'})\le Z_{\alpha', \alpha'-1}$. But $R\le Z_{\beta}\le V_{\alpha}$ and since $b>2$, $V_{\alpha}$ is abelian so centralizes $R$. In particular, $R$ is centralized by $L_{\alpha'-1}=\langle V_{\alpha}, R_{\alpha'-1}, Q_{\alpha'}\rangle$. But then $R\normaleq \langle L_{\alpha'}, L_{\alpha'-1}\rangle$, a final contradiction. Hence, $b\leq 2$.
\end{proof}

\begin{proposition}\label{G23}
Suppose that $Z_{\beta}\ne \Omega(Z(S))$, $b=2$ and for $\lambda\in\{\alpha, \beta\}$, $Z_\lambda/\Omega(Z(L_\lambda))$ is a natural $\SL_2(q_{\lambda})$-module for $L_\lambda/R_\lambda\cong\SL_2(q_{\lambda})$. Then $p=3$ and $G$ is locally isomorphic to $H$ where $F^*(H)\cong\mathrm{G}_2(3^n)$.
\end{proposition}
\begin{proof}
Since $b>1$, by \cref{BetaCenter} (iii), we have that $q_{\alpha}=q_{\beta}$ and $V_{\alpha}\not\le Q_{\beta}$. But then $Q_{\alpha}=V_{\alpha}(Q_{\alpha}\cap Q_{\alpha'})$ and it follows that $O^p(L_{\alpha})$ centralizes $Q_{\alpha}/V_{\alpha}$. In particular, $V_{\alpha}$ contains all non-central chief factors for $L_{\alpha}$ within $Q_{\alpha}$, and consequently $C_{L_{\alpha}}(V_{\alpha})$ is a $p$-group. By \cref{BetaCenter} (i), there is $\alpha-1\in\Delta(\alpha)$ such that $(\alpha-1,\beta)$ is a critical pair with $[Z_{\alpha-1}, Z_{\beta}]\ne\{1\}$ and applying \cref{BetaCenter} (ii) again, $C_{L_{\alpha-1}}(V_{\alpha-1})$ is a $p$-group. By \cref{BetaCenter} (iv), unless $q_{\alpha}\in\{2,3\}$, we conclude that $\bar{L_{\alpha}}\cong \bar{L_{\beta}}\cong \SL_2(q_\alpha)$ and $G$ has a weak BN-pair of rank $2$. Comparing with \cite{Greenbook}, the result holds.

Hence, we assume that $q_{\alpha}=q_{\beta}\in\{2,3\}$ and for $\lambda\in\{\alpha,\beta\}$, $V_{\lambda}/Z_{\lambda}$ and $Z_{\lambda}$ are FF-modules for $\bar{L_{\lambda}}$. Moreover, for some $\delta\in\{\alpha,\beta\}$, we assume that $C_{L_{\delta}}(V_{\delta}/Z_{\delta})\ne R_{\delta}$ and $\bar{L_\delta}\not\cong \SL_2(p)$. By \cref{Badp2} (ii) and \cref{Badp3} (ii), $\bar{L_{\delta}}\cong (3\times 3):2$ or $(Q_8\times Q_8):3$ for $p=2$ or $3$ respectively. Since $O^p(L_{\delta})$ centralizes $Q_{\delta}/V_{\delta}$ we have that $C_{L_\delta}(V_\delta/Z_{\delta})$ normalizes $Q_{\alpha}\cap Q_{\beta}$. 

If $p=2$, by \cref{Badp2} (iii), we may choose $P_{\alpha}\le L_{\alpha}$ such that $\bar{P_{\alpha}}\cong\Sym(3)$, $\Omega(Z(S))\not\normaleq P_{\alpha}$ and $Q_{\alpha}\cap Q_{\beta}\not\normaleq P_{\alpha}$. If $\bar{L_{\alpha}}\cong\Sym(3)$ then $L_{\alpha}=P_{\alpha}$, and if $\bar{L_{\alpha}}\cong (3\times 3):2$, then as there are two choices for $P_{\alpha}$, both are $G_{\alpha,\beta}$-invariant and neither normalizes $Q_{\alpha}\cap Q_{\beta}$. For such a $P_{\alpha}$, set $H_{\alpha}=P_{\alpha}G_{\alpha,\beta}$. We make an analogous choice for $H_{\beta}\le G_{\beta}$ and observe that $P_{\lambda}=O^{2'}(H_{\lambda})$ for $\lambda\in\{\alpha,\beta\}$.

If $p=3$, by \cref{Badp3} (iii), we may choose $P_{\alpha}\le L_{\alpha}$ such that $\bar{P_{\alpha}}\cong\SL_2(3)$, $\Omega(Z(S))\not\normaleq P_{\alpha}$ and $Q_{\alpha}\cap Q_{\beta}\not\normaleq P_{\alpha}$. If $\bar{L_{\alpha}}\cong\SL_2(3)$ then $L_{\alpha}=P_{\alpha}$, and if $\bar{L_{\alpha}}\cong (Q_8\times Q_8):3$, then there are three choices for $P_{\alpha}$. Since all contain $S$, there is at least one choice such that $P_{\alpha}$ is $G_{\alpha,\beta}$-invariant and does not normalize $Q_{\alpha}\cap Q_{\beta}$. For this $P_{\alpha}$, set $H_{\alpha}=P_{\alpha}G_{\alpha,\beta}$ and choose $H_{\beta}$ in a similar fashion. Again, observe that $P_{\lambda}=O^{2'}(H_{\lambda})$ for $\lambda\in\{\alpha,\beta\}$.

Set $X:=\langle H_{\alpha}, H_{\beta}\rangle$ and suppose that there is $\{1\}\ne Q\le S$ with $Q\normaleq X$. Then $Q\le O_p(H_{\alpha})\cap O_p(H_{\beta})=Q_{\alpha}\cap Q_{\beta}$. Suppose $\Omega(Z(S))\not\le Q$. Then $V_{\beta}=\langle \langle \Omega(Z(S))^{H_{\alpha}}\rangle^{H_{\beta}}\rangle$ centralizes $Q$ and since $Q$ is normal in $H_{\alpha}$, $[O^p(P_{\alpha}), Q]\le [V_{\beta}, Q]^{H_{\alpha}}=\{1\}$. Considering the action of $V_\alpha=\langle \langle \Omega(Z(S))^{H_{\beta}}\rangle^{H_{\alpha}}\rangle$ on $Q$ yields $[O^p(P_{\beta}), Q]=\{1\}$. But $Q\normaleq S$ and so $Q\cap \Omega(Z(S))$ is non-trivial and centralized by $G=\langle H_{\alpha}, R_{\alpha}, H_{\beta}, R_{\beta}\rangle$, a contradiction. Hence, $\Omega(Z(S))\le Q$. But then $Q\ge V_{\beta}=\langle \langle\Omega(Z(S))^{H_{\alpha}}\rangle^{H_{\beta}}\rangle\not\le Q_{\alpha}$, a contradiction. 

Thus, any subgroup of $G_{\alpha,\beta}$ which is normal in $X$ is a $p'$-group. Such a subgroup would be contained in $H_{\lambda}$ and so would centralize $Q_{\lambda}$ for $\lambda\in\{\alpha,\beta\}$. Since $S\le H_{\lambda}\le G_{\lambda}$, we have that $H_{\lambda}$ is of characteristic $p$, $C_{H_{\lambda}}(Q_{\lambda})\le Q_{\lambda}$ and no non-trivial subgroup of $G_{\alpha,\beta}$ is normal in $X$. Moreover, $\bar{P_{\alpha}}\cong\bar{P_{\alpha}}\cong \SL_2(p)$ and $X$ has a weak BN-pair of rank $2$. For $\lambda\in\{\alpha,\beta\}$, since $Q_{\lambda}$ contains precisely two non-central chief factors for $P_{\lambda}$, and neither $P_{\alpha}$ nor $P_{\beta}$ normalizes $\Omega(Z(S))$, by \cite{Greenbook}, $X$ is locally isomorphic to $\mathrm{G}_2(3)$ and $S$ is isomorphic to a Sylow $3$-subgroup of $\mathrm{G}_2(3)$. Then $Q_{\alpha}$ and $Q_{\beta}$ are distinguished up to isomorphism. Noticing that \cite[Lemma 7.8]{parkersem} applies in this situation independent of any fusion system hypothesis, it follows that for $\lambda\in\{\alpha,\beta\}$, $\bar{G_{\lambda}}$ is isomorphic to a subgroup of $\GL_2(3)$, a contradiction to the assumption that $\bar{L_\delta}\not\cong \SL_2(p)$. Thus, we conclude that $G$ has a weak BN-pair of rank $2$ and the result follows upon comparison with \cite{Greenbook}.
\end{proof}

\begin{remark}
The graph automorphism of $\mathrm{G}_2(3)$ normalizes $S\in\syl_3(\mathrm{G}_2(3))$ and fuses $Q_{\alpha}$ and $Q_{\beta}$, and so \cref{MainHyp} only allows for groups locally isomorphic to $\mathrm{G}_2(3^n)$ decorated by field automorphisms.
\end{remark}

\begin{proposition}\label{BetaCenterII}
Suppose that $Z_{\beta}\ne \Omega(Z(S))$ and for $\lambda\in\{\alpha, \beta\}$, $Z_\lambda/\Omega(Z(L_\lambda))$ is a natural $\SL_2(q_\lambda)$-module for $L_\lambda/R_\lambda$. Then $G$ is locally isomorphic to $H$ where $(F^*(H),p)$ is one of $(\PSL_3(p^n), p)$, $(\PSp_4(2^n), 2)$ or $(\mathrm{G}_2(3^n), 3)$.
\end{proposition}
\begin{proof}
By \cref{BetaCenterI} and \cref{G23}, we may suppose that $b=1$. Then, $Z_{\alpha}\not\le Q_{\beta}$, $Z_{\beta}\not\le Q_\alpha$, $Q_{\alpha}=Z_{\alpha}(Q_{\alpha}\cap {Q_\beta})$ and $Q_{\beta}=Z_{\beta}(Q_{\alpha}\cap {Q_\beta})$. In particular, $\Phi(Q_\alpha)=\Phi(Q_{\alpha}\cap Q_\beta)=\Phi(Q_\beta)$ is trivial and so both $Q_\alpha$ and $Q_\beta$ are elementary abelian. For $\lambda\in\{\alpha, \beta\}$, by coprime action we have that $Q_\lambda=[Q_\lambda, R_\lambda]\times C_{Q_{\lambda}}(R_\lambda)$ is an $S$-invariant decomposition. But $\Omega(Z(S))\le Z_\lambda\le C_{Q_{\lambda}}(R_\lambda)$ and since $[Q_\alpha, R_\lambda]\normaleq S$, we must have that $[Q_{\alpha}, R_{\lambda}]=\{1\}$. It follows that $R_\lambda$ centralizes $Q_\lambda$ and, as $G_{\lambda}$ is of characteristic $p$, $Q_\lambda=R_\lambda$. Thus, $G$ has a weak BN-pair of rank $2$ and is determined by \cite{Greenbook}, hence the result.
\end{proof}

\begin{remark}
Similarly to the $\mathrm{G}_2(3^n)$ example, the graph automorphisms for $\PSL_3(p^n)$ and $\PSp_4(2^n)$ fuse $Q_{\alpha}$ and $Q_{\beta}$ and are not permitted by the hypothesis.
\end{remark}

%% file: Contents/6.2.SymmetricCaseii.tex
\subsection{$Z_{\beta}=\Omega(Z(S))$}\label{ZBEQZS}

Given \cref{BetaCenterIII}, we may assume in this subsection that $b$ is even and $Z_{\beta}=\Omega(Z(S))$. The general aim will be to demonstrate that $b=2$ and $\bar{L_{\alpha}}\cong\SL_2(q)$ for then, it will quickly follow that the amalgam in question is symplectic and we may apply the classification in \cite{parkerSymp}. We are able to show that, in all the cases considered, $b=2$. However, at the end of this section we uncover a configuration where $R_{\alpha}\ne Q_{\alpha}$.

\begin{lemma}\label{ineq}
Let $\alpha-1\in\Delta(\alpha)\setminus \{\beta\}$ with $Z_{\alpha-1}\ne Z_\beta$. Then $\Omega(Z(L_\alpha))=\{1\}$, $Z_{\alpha}=Z_\beta\times Z_{\alpha-1}$ is a natural $\SL_2(q_\alpha)$-module, $Q_{\beta}\in\syl_p(R_{\beta})$ and $[Z_\alpha, Z_{\alpha'}]=Z_{\alpha'-1}=Z_{\alpha}\cap Q_{\alpha'}=Z_\beta=[V_\beta, Q_\beta]$.
\end{lemma}
\begin{proof}
Since $L_\beta$ is transitive on $\Delta(\beta)$ and centralizes $Z_\beta=\Omega(Z(S))$, by \cref{crit pair} (iv), we have that $Z(L_\alpha)=\{1\}$. Then, by \cref{SL2}, $Z_{\alpha}$ is a natural $\SL_2(q_\alpha)$-module for $L_{\alpha}/R_{\alpha}\cong \SL_2(q_\alpha)$.

Now, $[Z_{\alpha}, S]=[Z_{\alpha}, Z_{\alpha'}Q_\alpha]=[Z_{\alpha}, Z_{\alpha'}]=\Omega(Z(S))=Z_\beta$. Thus, $[V_{\beta}, Q_{\beta}]=[\langle Z_{\alpha}^{G_{\beta}}\rangle, Q_{\beta}]=Z_{\beta}\le C_{V_{\beta}}(O^p(L_{\beta}))$ and so $Q_{\beta}\le R_{\beta}$. By \cref{BasicVB}, we have that $Q_{\beta}\in\syl_p(R_{\beta})$.

By considering $[Z_{\alpha'}, Z_{\alpha}Q_{\alpha'}]$ and again employing \cref{SL2}, we deduce that, for $T\in\syl_p(G_{\alpha', \alpha'-1})$, $[Z_{\alpha'}, Z_{\alpha}]=\Omega(Z(T))=Z_{\alpha'-1}$. Then $Z_{\beta}=Z_{\alpha'-1}\le Q_{\alpha'}$ and it follows immediately that $Z_{\beta}=Z_{\alpha}\cap Q_{\alpha'}$. By properties of natural $\SL_2(q_\alpha)$-modules, $Z_{\alpha}=Z_\beta\times Z_\beta^x=Z_\beta\times Z_{\beta\cdot x}$ for $x\in L_\alpha\setminus G_{\alpha,\beta}R_{\alpha}$. In particular, we may choose $\alpha-1\in\Delta(\alpha)$ conjugate to $\beta$ by an element of $L_\alpha\setminus G_{\alpha,\beta}R_{\alpha}$ so that $Z_{\alpha}=Z_\beta \times Z_{\alpha-1}$. 
\end{proof}

\begin{lemma}\label{EvenLemma}
Suppose that $b>2$. Then the following hold:
\begin{enumerate}
    \item $[V_{\alpha}^{(2)}, V_{\alpha}^{(2)}]=\Phi(V_{\alpha}^{(2)})\le Z_{\alpha}$;
    \item $[V_{\alpha}^{(2)}, V_{\alpha'-1}, V_{\alpha'-1}]\le Z_{\alpha}$;
    \item $V_{\beta}/Z_{\beta}$ is a faithful quadratic module for $L_{\beta}/R_{\beta}$; and
    \item there is $\alpha-1\in\Delta(\alpha)$ such that $Z_{\alpha-1}\ne Z_{\beta}$, $V_{\alpha-1}\le Q_{\alpha'-2}$, $V_{\alpha-1}\not\le Q_{\alpha'-1}$ and $[V_{\alpha'-1}\cap Q_{\alpha}\cap Q_{\alpha-1}, V_{\alpha-1}]=\{1\}$.
\end{enumerate}
\end{lemma}
\begin{proof}
Notice that \[[V_{\alpha}^{(2)}, V_{\alpha}^{(2)}]=[V_{\alpha}^{(2)}, V_{\beta}]^{G_{\alpha}}\le [Q_{\beta}, V_{\beta}]^{G_{\alpha}}=\langle Z_{\beta}^{G_{\alpha}}\rangle=Z_{\alpha}.\] Since $V_{\alpha}^{(2)}$ is generated by $V_{\lambda}$ for $\lambda\in\Delta(\alpha)$ and $V_{\lambda}$ is elementary abelian, it follows that $V_{\alpha}^{(2)}/[V_{\alpha}^{(2)}, V_{\alpha}^{(2)}]$ is elementary abelian and (i) holds. Moreover, we have that $[V_{\alpha}^{(2)}, V_{\alpha'-1}, V_{\alpha'-1}]\le [V_{\alpha}^{(2)}, V_{\alpha'-2}^{(2)}, V_{\alpha'-1}]\le [Q_{\alpha'-1}, V_{\alpha'-1}]=Z_{\alpha'-1}=Z_{\beta}$, (ii) holds and since $Z_{\alpha}\le [V_{\alpha}^{(2)}, Q_{\alpha}]$, $V_{\alpha}^{(2)}/[V_{\alpha}^{(2)}, Q_{\alpha}]$ is a quadratic module for $\bar{L_{\alpha}}$. Furthermore, $[V_{\alpha'-1}, V_{\alpha}^{(2)}\cap Q_{\alpha'-2}, V_{\alpha}^{(2)}\cap Q_{\alpha'-2}]\le V_{\alpha'-1}\cap Z_{\alpha}=Z_{\alpha'-1}$ and so $V_{\alpha'-1}/Z_{\alpha'-1}$ is a quadratic module for $\bar{L_{\alpha'-1}}$, and (iii) holds.

Set $U:=\langle V_{\lambda}\mid Z_{\lambda}=Z_{\alpha-1}, \lambda\in\Delta(\alpha)\rangle$ for a fixed subgroup $Z_{\alpha-1}\ne Z_{\beta}$. If $U\not\le Q_{\alpha'-2}$, then up to relabeling, there is some $\alpha-2\in\Delta^{(2)}(\alpha)$ with $(\alpha-2, \alpha'-2)$ a critical pair and $Z_{\alpha-1}\ne Z_{\beta}$. But then $Z_{\alpha}=Z_{\alpha-1}\times Z_{\beta}\le V_{\alpha'-1}$, a contradiction since $b\geq 4$. Suppose that $U\le Q_{\alpha'-1}$ so that $[Z_{\alpha'}, U]=[Z_{\alpha'}, Z_{\alpha}(U\cap Q_{\alpha'})]\le Z_{\alpha}\le U$. Let $r\in R_{\alpha}Q_{\alpha-1}$. Since $r$ is an automorphism of the graph, it follows that for $V_{\delta}\le U$, we have that $V_{\delta}^r=V_{\delta\cdot r}$. But $Z_{\delta\cdot r}=Z_{\delta}^r=Z_{\alpha-1}^r=Z_{\alpha-1}$ and so $V_{\delta}^r\le U$. Then $U\normaleq L_{\alpha}=\langle R_{\alpha}, Z_{\alpha'}, Q_{\alpha-1}\rangle$. Since $Z_{\alpha'}$ centralizes $U/Z_{\alpha}$, $[O^p(L_{\alpha}), V_{\alpha-1}]\le [O^p(L_{\alpha}), U]=Z_{\alpha}\le V_{\alpha-1}$. In particular, $V_{\alpha-1}\normaleq \langle G_{\alpha}, G_{\alpha-1}\rangle$, a contradiction. Thus we may assume that $U\not\le Q_{\alpha'-1}$ and we may choose $V_{\alpha-1}\not\le Q_{\alpha'-1}$ with $Z_{\alpha-1}\ne Z_{\beta}$. and (iv) holds.
\end{proof}

\begin{lemma}\label{EvenNat1}
Suppose that $b>2$ and $m_p(S/Q_{\alpha})>1$. Then either $L_\beta/R_\beta\cong\SL_2(q)\cong L_{\alpha}/R_\alpha$ and both $Z_{\alpha}$ and $V_\beta/C_{V_{\beta}}(O^p(L_{\beta}))$ are natural modules; or $q_{\beta}>q_{\alpha}$, $[O^p(R_{\alpha}), V_{\alpha}^{(2)}]\ne\{1\}$ and $p=2$.
\end{lemma}
\begin{proof}
Assume that $p$ is an odd prime. Using that $m_p(S/Q_{\alpha})\geq 2$ and $V_{\alpha}^{(2)}/Z_{\alpha}$ admits quadratic action by $V_{\alpha'-1}$, we deduce from \cref{SEQuad} and a standard argument involving the Schur multiplier of $\PSL_2(q)$, that $R_{\alpha}=C_{L_{\alpha}}(V_{\alpha}^{(2)})Q_{\alpha}$. Assume now that $p=2$ and $q_{\beta}\leq q_{\alpha}$. Then as $V_{\alpha}^{(2)}$ is elementary abelian, $V_{\alpha}^{(2)}\cap Q_{\alpha'-2}\cap Q_{\alpha'-1}$ has index at most $q_{\alpha}^2$ and is centralized, modulo $Z_{\alpha}$, by $Z_{\alpha'}$. Applying \cref{2FRecog}, we deduce in this case also that $R_{\alpha}=C_{L_{\alpha}}(V_{\alpha}^{(2)})Q_{\alpha}$.

More generally, assume that $R_{\alpha}=C_{L_{\alpha}}(V_{\alpha}^{(2)})Q_{\alpha}$. An application of \cref{SimExt} to $Z_{\alpha'-1}=Z_{\beta}$ yields immediately that $b>4$ so that $V_{\alpha}^{(2)}$ is elementary abelian. Furthermore, $V_{\alpha}^{(2)}\le Q_{\alpha'-2}$ for otherwise there is $\lambda\in\Delta(\alpha)$ with $(\lambda-1, \alpha'-2)$ a critical pair. But then by \cref{ineq} we deduce that $Z_{\lambda}=Z_{\alpha'-3}$ and since $Z_{\alpha}\ne Z_{\alpha'-2}$, we must have that $Z_{\lambda}=Z_{\beta}$ and \cref{SimExt} yields that $V_{\lambda}=V_{\beta}\le Q_{\alpha'-2}$, a contradiction. 

NOw, $V_{\alpha-1}\cap Q_{\alpha'-1}=Z_{\alpha}(V_{\alpha-1}\cap Q_{\alpha'})$ so that $L_{\alpha}=\langle Z_{\alpha'}, Q_{\alpha-1}, R_{\alpha}\rangle$ normalizes $V_{\alpha-1}\cap Q_{\alpha'-1}$. Since $L_{\alpha}$ acts transitively on the neighbors of $\alpha$, $V_{\alpha-1}\cap V_{\beta}\le V_{\alpha-1}\cap Q_{\alpha'-1}\le \bigcap\limits_{\lambda\in\Delta(\alpha)} V_\lambda$. By conjugacy, $\bigcap\limits_{\mu\in\Delta(\alpha'-2)} V_\mu$ has index $|V_{\alpha-1}Q_{\alpha'-1}/Q_{\alpha'-1}|$ in $V_{\alpha'-1}$ and is centralized by $V_{\alpha-1}$. By \cref{SEFF} and conjugacy, $V_\beta/C_{V_{\beta}}(O^p(L_{\beta}))$ is a natural module for $L_\beta/R_\beta\cong\SL_2(q_\beta)$ and as $Z_{\alpha}C_{V_{\beta}}(O^p(L_{\beta}))/C_{V_{\beta}}(O^p(L_{\beta}))\cong Z_{\alpha}/Z_{\beta}$ has order $q_\alpha$, and $Z_{\alpha}C_{V_{\beta}}(O^p(L_{\beta}))/C_{V_{\beta}}(O^p(L_{\beta}))$ is $G_{\alpha,\beta}$-invariant, we deduce that $q_\alpha=q_\beta$ and the proof is complete.
\end{proof}

\sloppy{In the following lemma and proposition, we retain the definition of $U$ from \cref{EvenLemma}. That is, $U:=\langle V_{\lambda}\mid Z_{\lambda}=Z_{\alpha-1}, \lambda\in\Delta(\alpha)\rangle$ for a fixed subgroup $Z_{\alpha-1}\ne Z_{\beta}$ where $\alpha-1\in\Delta(\alpha)$. Furthermore, we define $\mathcal{U}:=[U, Q_{\alpha}; i]Z_{\alpha}$ with $i$ chosen minimally so that $[U, Q_{\alpha}; i+1]\le Z_{\alpha}$.}

\begin{lemma}\label{EvenLemma1}
Suppose that $b>2$. If $V_{\beta}/C_{V_{\beta}}(O^p(L_{\beta}))$ is not a natural module for $L_{\beta}/R_{\beta}\cong \SL_2(q_\alpha)$ then the following hold:
\begin{enumerate}
\item $Z(Q_{\alpha})=Z_{\alpha}$;
\item $\mathcal{U}=[V_{\lambda}, Q_{\alpha}; i]\normaleq G_{\alpha}$ for all $\lambda\in\Delta(\alpha)$; and
\item $[\mathcal{U}, Q_{\alpha}]=Z_{\alpha}$.
\end{enumerate}
\end{lemma}
\begin{proof}
Assume that $Z(Q_{\alpha})\not\le Q_{\alpha'-1}$. In particular, $V_{\alpha'-1}\cap Q_{\alpha}$ is centralized by $Z(Q_{\alpha})$ and $Z(Q_{\alpha})$ acts quadratically on $V_{\alpha'-1}$ so that $Z(Q_{\alpha})Q_{\alpha'-1}/Q_{\alpha'-1}$ is elementary abelian. We may assume that $m_p(S/Q_{\alpha})>1$ else we have a contradiction to the initial assumptions by \cref{SEFF}. If $Z(Q_{\alpha})\cap Q_{\alpha'-1}$ has index strictly less than $q_\alpha$ in $Z(Q_{\alpha})$ then $O^p(L_{\alpha})$ centralizes $Z(Q_{\alpha})/Z_{\alpha}$ and we deduce that $S$ centralizes $Z(Q_{\alpha})/Z_{\alpha}$. In particular, $[V_{\alpha'-1}, Z(Q_{\alpha})]\le Z_{\alpha}\cap V_{\alpha'-1}=Z_{\beta}=Z_{\alpha'-1}$, a contradiction since $Z(Q_{\alpha})\not\le Q_{\alpha'-1}$. Hence, $q_\alpha\leq |Z(Q_{\alpha})Q_{\alpha'-1}/Q_{\alpha'-1}|$ and \cref{SEFF} provides a contradiction to the initial assumption. Therefore, $Z(Q_{\alpha})\le Q_{\alpha'-1}$. Then $O^p(L_{\alpha})$ centralizes $Z(Q_{\alpha})/Z_{\alpha}$ and the irreducibility of $L_{\alpha}$ on $Z_{\alpha}$ yields that $Z(Q_{\alpha})$ is elementary abelian. But then $Z(Q_\alpha)\cap Q_{\alpha'}$ is centralized by $S=Z_{\alpha'}Q_{\alpha}$ so that $Z(Q_\alpha)\cap Q_{\alpha'}\le \Omega(Z(S))=Z_{\beta}\le Z_{\alpha}$. Thus, $Z(Q_{\alpha})=Z_{\alpha}$.

Set $\mathcal{U}:=[U, Q_{\alpha}; i]Z_{\alpha}$ with $i$ chosen minimally so that $[U, Q_{\alpha}; i+1]\le Z_{\alpha}$. In particular, since $Z(Q_{\alpha})=Z_{\alpha}$, $[\mathcal{U}, Q_{\alpha}]$ is non-trivial. Let $W:=\langle \mathcal{U}^{G_{\alpha}}\rangle$. Then $[W, Q_{\alpha}]=[\mathcal{U}, Q_{\alpha}]^{G_{\alpha}}=Z_{\alpha}$. Moreover, if $W/Z_{\alpha}$ contains no non-central chief factors for $L_{\alpha}$, then $W=\mathcal{U}=[V_{\alpha-1}, Q_{\alpha}; i]$ is normalized by $G_{\alpha}=O^p(L_{\alpha})G_{\alpha,\alpha-1}$ and the result holds.

Hence, we may assume that $W/Z_{\alpha}$ contains a non-central chief factor for $L_{\alpha}$. Suppose that $|(W\cap Q_{\alpha'-2})Q_{\alpha'-1}/Q_{\alpha'-1}|\leq p$. By \cref{SEFF}, $|WQ_{\alpha'-2}/Q_{\alpha'-2}|\geq q_{\alpha}/p$. But now, $[V_{\alpha'-2}^{(2)}\cap Q_{\alpha}, W]\le Z_{\alpha}\cap V_{\alpha'-2}^{(2)}=Z_{\alpha'-1}$ so that either $q_\alpha=p$ and $W\le Q_{\alpha'-2}$, or $V_{\alpha'-2}^{(2)}/Z_{\alpha'-2}$ contains a unique non-central chief factor for $L_{\alpha'-2}$ by \cref{SpeMod2}. In the latter case, by conjugacy and applying \cref{CommCF}, we conclude that $\mathcal{U}=U$, $W=V_{\alpha}^{(2)}$ and $[V_{\alpha}^{(2)}, Q_{\alpha}]=Z_{\alpha}$. Since $q_\alpha\leq q_\beta$ by \cref{EvenNat1} and $[Q_{\alpha}, V_{\beta}]=Z_{\alpha}$, we deduce that $V_{\beta}/Z_{\beta}$ is dual to an FF-module for $L_{\beta}/R_{\beta}\cong \SL_2(q_\beta)$ and by \cref{SEFF}, we have a contradiction to the initial hypothesis. If $q_\alpha=p$, then $V_{\alpha'-1}\cap Q_{\alpha}$ has index $p$  in $V_{\alpha'-1}$ and is centralized, modulo $Z_{\alpha'-1}$, by $W$ and we have a contradiction by \cref{SEFF}. 

Thus, $|(W\cap Q_{\alpha'-2})Q_{\alpha'-1}/Q_{\alpha'-1}|\geq p^2$. As before, we have that $V_{\alpha'-1}\cap Q_{\alpha}$ has index $q_\alpha$ in $V_{\alpha'-1}$ and is centralized, modulo $Z_{\alpha'-1}$, by $W$. In particular, \cref{SEFF} yields $q_\alpha>|(W\cap Q_{\alpha'-2})Q_{\alpha'-1}/Q_{\alpha'-1}|\geq p^2$ and \cref{EvenNat1} gives that $q_\beta>q_\alpha$ and $p=2$. But now, a contradiction is provided by \cref{SpeMod2}.
\end{proof}

\begin{proposition}\label{EvenNat}
Suppose that $b>2$. Then $L_\beta/R_\beta\cong\SL_2(q)\cong L_{\alpha}/R_\alpha$ and both $Z_{\alpha}$ and $V_\beta/C_{V_{\beta}}(O^p(L_{\beta}))$ are natural modules.
\end{proposition}
\begin{proof}
Assume throughout that $V_{\beta}/C_{V_{\beta}}(O^p(L_{\beta}))$ is not a natural module for $L_{\beta}/R_{\beta}\cong \SL_2(q_\alpha)$. Then we may use the results proved in \cref{EvenNat1} and \cref{EvenLemma1}. Suppose first that $m_p(S/Q_{\beta})=1$. Then by \cref{EvenNat1}, $m_p(S/Q_{\alpha})=1$. Since $V_{\alpha'-1}\cap Q_{\alpha}\cap Q_{\alpha-1}$ has index $p^2$ in $V_{\alpha'-1}$ and is centralized by $V_{\alpha-1}$, $L_{\alpha'-1}/R_{\alpha'-1}$ and $V_{\alpha'-1}/C_{V_{\alpha'-1}}(O^p(L_{\alpha'-1}))$ are determined by \cref{Quad2F}. Since \sloppy{$Z_{\alpha'}C_{V_{\alpha'-1}}(O^p(L_{\alpha'-1}))/C_{V_{\alpha'-1}}(O^p(L_{\alpha'-1}))$} has order $p$ and is $G_{\alpha',\alpha'-1}$-invariant, and $V_{\alpha'-1}=\langle Z_{\alpha'}^{L_{\alpha'-1}}\rangle$, by \cref{pgen} we have that $L_{\alpha'-1}/R_{\alpha'-1}\cong\Sz(2), \Dih(10)$, $(3\times 3):2$ or $(3\times 3):4$. In particular, using coprime action, it follows that for $V:=V_{\alpha'-1}/Z_{\alpha'-1}$, $V=[V, O^2(L_{\alpha'-1})]\times C_V(O^2(L_{\alpha'-1}))$ where $[V, O^2(L_{\alpha'-1})]$ is irreducible and $|C_V(O^2(L_{\alpha'-1}))|\leq 2$.

Assume that $L_{\alpha'-1}/R_{\alpha'-1}\cong\Sz(2)$ or $(3\times 3):4$. Then, by \cref{SzMod} (iii) and \cref{334} (iii), $[V, Q_{\alpha'-2}; 3]\ne \{1\}=[V, Q_{\alpha'-2}; 4]$ and, by conjugacy, we infer that $[V_{\alpha-1}, Q_{\alpha}; 4]\le Z_{\alpha-1}$. Then, \cref{EvenLemma1} implies that $\mathcal{U}=[V_{\alpha-1}, Q_{\alpha}, Q_{\alpha}]$ and $Z_{\alpha}=[V_{\alpha-1}, Q_{\alpha}; 3]$. Moreover, we deduce that $[U, Q_{\alpha}]\not\le Q_{\alpha'-1}$, else $[U, Q_{\alpha}]=Z_{\alpha}([U, Q_{\alpha}]\cap Q_{\alpha'})$ is centralized, modulo $Z_{\alpha}$, by $Z_{\alpha'}$ from which we have that $[V_{\alpha-1}, Q_{\alpha}]\normaleq L_{\alpha}$. But then, by conjugacy, $[V_{\alpha'-1}, Q_{\alpha'-2}]=[V_{\alpha'-3}, Q_{\alpha'-2}]$ is centralized by $V_{\alpha-1}$, contradicting \cref{SzMod} (ii) and \cref{334} (ii). If $[V_{\alpha}^{(2)}, Q_{\alpha}]\not\le Q_{\alpha'-2}$, then as $\Phi(V_{\alpha}^{(2)})\le Z_{\alpha}\le Q_{\alpha'-1}$, $V_{\alpha}^{(2)}\cap Q_{\alpha'-2}=[U, Q_{\alpha}](V_{\alpha}^{(2)}\cap Q_{\alpha'-2}\cap Q_{\alpha'-1})$ so that $V_{\alpha}^{(2)}=[V_{\alpha}^{(2)}, Q_{\alpha}](V_{\alpha}^{(2)}\cap Q_{\alpha'})$ and $V_{\alpha}^{(2)}/[V_{\alpha}^{(2)}, Q_{\alpha}]$ is centralized by $O^p(L_{\alpha})$, a contradiction by \cref{CommCF}. Thus, as $\Phi(U)\le \Phi(V_{\alpha}^{(2)})\le Z_{\alpha}\le Q_{\alpha'-1}$, $U[V_{\alpha}^{(2)}, Q_{\alpha}]=[V_{\alpha}^{(2)}, Q_{\alpha}](U[V_{\alpha}^{(2)}, Q_{\alpha}]\cap Q_{\alpha'})$ and $U[V_{\alpha}^{(2)}, Q_{\alpha}]\normaleq L_{\alpha}$. In particular, $V_{\alpha}^{(2)}=V_{\alpha-1}[V_{\alpha}^{(2)}, Q_{\alpha}]$ from which it follows that $[Q_{\alpha-1}, V_{\alpha}^{(2)}]\le [V_{\alpha}^{(2)}, Q_{\alpha}]$ and $O^p(L_{\alpha})$ centralizes $V_{\alpha}^{(2)}/[V_{\alpha}^{(2)}, Q_{\alpha}]$, and a contradiction is again provided by \cref{CommCF}.

Assume that $L_{\alpha'-1}/R_{\alpha'-1}\cong\Dih(10)$ or $(3\times 3):2$. Then, applying \cref{SzMod} (ii) and \cref{Badp2} (v), and using that $P/Q_{\alpha-1}=\Omega(P/Q_{\alpha-1})$ where $P\in\syl_2(G_{\alpha,\alpha-1})$, $[V_{\alpha-1}, Q_{\alpha}, Q_{\alpha}]\le Z_{\alpha-1}$, $[V_{\alpha-1}, Q_{\alpha}]\not\le Z_{\alpha}$ and \cref{EvenLemma1} gives that $\mathcal{U}=[V_{\alpha-1}, Q_{\alpha}]\normaleq G_{\alpha}$. But then $Z_{\alpha}=[V_{\alpha-1}, Q_{\alpha}, Q_{\alpha}]\le Z_{\alpha-1}$, another contradiction.

Hence by \cref{EvenNat1} we have that $q_\beta>q_\alpha$ and $p=2$. Assume that $\bar{L_\beta}/O_{2'}(\bar{L_{\beta}})\cong \Sz(q_\beta)$ or $\PSU_3(q_\beta)$. Since $V_{\alpha'-1}$ is elementary abelian, $V_{\alpha'-1}\cap Q_{\alpha}\cap Q_{\alpha-1}$ has index at most $q_\alpha q_\beta$ and is centralized by $V_{\alpha-1}$. Then by \cref{SpeMod2}, $|V_{\alpha-1}Q_{\alpha'-1}/Q_{\alpha'-1}|=2$ and $q_\alpha^3\geq q_\beta$ and $q_\alpha^2\geq q_\beta$ respectively. But now, $[V_{\alpha-1}\cap Q_{\alpha'-1}, x]\le Z_{\alpha}$ for some $x\in (V_{\alpha'-1}\cap Q_{\alpha})\setminus Q_{\alpha-1}$ from which it follows that $x$ centralizes a subgroup of index at most $pq_\alpha\le q_\beta$ in $V_{\alpha-1}$. Applying \cref{SpeMod2}, we have a contradiction. Hence, $\bar{L_\beta}/O_{2'}(\bar{L_{\beta}})\cong \PSL_2(q_\beta)$.

Now, $[Q_{\alpha}\cap Q_{\beta}, \mathcal{U}]\le Z_{\beta}$ so that $Q_{\alpha}\cap Q_{\beta}\le C_{Q_{\alpha}}(\mathcal{U}/Z_{\beta})\normaleq R_{\alpha}G_{\alpha,\beta}$. Since $\bar{L_\beta}/O_{2'}(\bar{L_{\beta}})\cong \PSL_2(q_\beta)$, $C_{Q_{\alpha}}(\mathcal{U}/Z_{\beta})\le Q_{\beta}$ as $S/Q_{\beta}\cong Q_{\alpha}/Q_{\alpha}\cap Q_{\beta}$ is irreducible under the action of $G_{\alpha,\beta}$ and $Q_{\alpha}$ does not centralize $\mathcal{U}/Z_{\beta}$. More generally, for any $\lambda, \mu\in\Delta(\alpha)$, if $Z_{\lambda}=Z_{\mu}$ then $Q_{\lambda}\cap Q_{\alpha}=Q_{\mu}\cap Q_{\alpha}$. Additionally, we have that $Q_{\alpha}\cap Q_{\alpha-1}\cap Q_{\beta}=C_{Q_{\alpha}}(\mathcal{U})\normaleq G_{\alpha}$ so that $[C_{Q_{\alpha}}(\mathcal{U}), V_{\alpha}^{(2)}]=[C_{Q_{\alpha}}(\mathcal{U}), V_{\beta}]^{G_{\alpha}}=Z_{\alpha}$.

Suppose that $[V_{\alpha}^{(2)}, V_{\alpha'-3}]\ne \{1\}$. Then there is a critical pair $(\lambda+1, \mu)$ for some $\lambda\in\Delta(\alpha)$ and $\mu\in\Delta(\alpha'-3)$ and by an argument in \cref{EvenLemma}, we deduce that $Z_\lambda=Z_{\beta}=Z_{\alpha'-1}=Z_{\alpha'-3}$. But then, $U\le Q_{\alpha'-3}\cap Q_{\alpha'-2}=Q_{\alpha'-2}\cap Q_{\alpha'-1}$, a contradiction by \cref{EvenLemma}. Indeed, this also proves that $Z_{\alpha'-1}\ne Z_{\alpha'-3}$ and $V_{\alpha}^{(2)}\le Q_{\alpha'-2}$. By \cref{EvenLemma1}, we have that $[V_{\alpha'-1}, Q_{\alpha'-2};i]\normaleq G_{\alpha'-2}$ where $i$ is the integer defining $\mathcal{U}$. Moreover, since $Z_{\alpha'-1}\ne Z_{\alpha'-3}$, $C_{Q_{\alpha'-2}}([V_{\alpha'-1}, Q_{\alpha'-2};i])=Q_{\alpha'-1}\cap Q_{\alpha'-3}$. But $[V_{\alpha'-1}, Q_{\alpha'-2};i]\le V_{\alpha'-3}$ is centralized by $V_{\alpha}^{(2)}$, a contradiction since $V_{\alpha}^{(2)}\not\le Q_{\alpha'-1}$.
\end{proof}

Before continuing, observe that we may now assume that whenever $b>2$ and $q=p$, both $L_{\alpha}/R_{\alpha}$ and $L_{\beta}/R_{\beta}$ are isomorphic to $\SL_2(p)$. Throughout this section, under these conditions and given a module $V$ on which $\bar{L_{\gamma}}$ acts, for any $\gamma\in\Gamma$, we will often utilize coprime action. By this, we mean that when $p\geq 5$, taking $T_\gamma$ to be the preimage in $\bar{L_{\gamma}}$ of $Z(L_{\gamma}/R_{\gamma})$, we have that $V=[V, T]\times C_V(T)$. Indeed, if $V$ is an FF-module for $\bar{L_{\gamma}}$, then this leads to a splitting $V=[V, \bar{L_{\gamma}}]\times C_V(\bar{L_{\gamma}})$. If $p\in\{2,3\}$, since $\bar{L_{\gamma}}$ is solvable, we automatically have the conclusion $V=[V, O^p(\bar{L_{\gamma}})]\times C_V(O^p(\bar{L_{\gamma}}))$. Without explaining this each time it is used, we will generally just refer to ``coprime action'' and hope that it is clear in each instance where the conclusions we draw come from.

\begin{lemma}\label{EvenZ}
Suppose that $b>2$ and $q=p$. Then $Z_{\beta}=Z(Q_{\beta})$ and $Z_{\alpha}=Z(Q_{\alpha})$.
\end{lemma}
\begin{proof}
By minimality of $b$, and using that $b$ is even, we infer that $Z(Q_{\alpha})\le Q_{\lambda}$ for all $\lambda\in\Delta^{(b-2)}(\alpha)$. In particular, $Z(Q_{\alpha})\le Q_{\alpha'-2}$. If $Z(Q_{\alpha})\not\le Q_{\alpha'-1}$ then as  $[Z(Q_{\alpha}), V_{\alpha'-1}, V_{\alpha'-1}]\le [V_{\alpha'-1}, V_{\alpha'-1}]=\{1\}$ and $[Z(Q_{\alpha}), V_{\alpha'-1}]$ is centralized by $V_{\alpha'-1}Q_{\alpha}\in\syl_p(L_{\alpha})$ and has exponent $p$. Thus, $[Z(Q_{\alpha}), V_{\alpha'-1}]\le \Omega(Z(S))=Z_{\beta}=Z_{\alpha'-1}$, a contradiction for otherwise $O^p(L_{\alpha'-1})$ centralizes $V_{\alpha'-1}$. Thus, $Z(Q_{\alpha})\le Q_{\alpha'-1}$ so that $Z(Q_{\alpha})=Z_{\alpha}(Z(Q_{\alpha})\cap Q_{\alpha'})$, $Z_{\alpha'}$ centralizes $Z(Q_{\alpha})/Z_{\alpha}$ and $O^p(L_{\alpha})$ centralizes $Z(Q_{\alpha})/Z_{\alpha}$. Since $Z_{\beta}\le Z_{\alpha}$ an application of coprime action yields $Z(Q_{\alpha})=[Z(Q_{\alpha}), O^p(L_{\alpha})]=Z_{\alpha}$, as desired. As a consequence, using that $Q_{\alpha}$ is self-centralizing, $Z(S)$ has exponent $p$.

Let $\alpha-1\in\Delta(\alpha)$ such that $Z_{\alpha-1}\ne Z_{\beta}$, $V_{\alpha-1}\le Q_{\alpha'-2}$ and $V_{\alpha-1}\not\le Q_{\alpha'-1}$, as chosen in \cref{EvenNat}. By minimality of $b$, and using that $b$ is even, we have that $Z(Q_{\alpha'-1})\le Q_{\lambda}$ for all $\lambda\in\Delta^{(b-1)}(\alpha)$. In particular, $Z(Q_{\alpha'-1})\le Q_{\alpha}$. 

If $Z(Q_{\alpha'-1})\not\le Q_{\alpha-1}$ then $Z(Q_{\alpha'-1})Q_{\alpha-1}\in\syl_p(L_{\alpha-1})$. Again, using minimality of $b$, we infer that $Z(Q_{\alpha-1})\le Q_{\alpha'-2}$ so that $[Z(Q_{\alpha'-1}), Z(Q_{\alpha-1})]\le Z(Q_{\alpha'-1})\cap Z(Q_{\alpha-1})$. Thus, $[Z(Q_{\alpha'-1}), Z(Q_{\alpha-1})]$ is centralized by $Z(Q_{\alpha'-1})Q_{\alpha-1}\in\syl_p(L_{\alpha-1})$. Then, $[Z(Q_{\alpha'-1}), Z(Q_{\alpha-1})]\le Z_{\alpha-1}$ and as $Z_{\alpha-1}\not\le Z(Q_{\alpha'-1})$, $[Z(Q_{\alpha'-1}), Z(Q_{\alpha-1})]=\{1\}$ and $Z(Q_{\alpha-1})$ is centralized by $Z(Q_{\alpha'-1})Q_{\alpha-1}\in\syl_p(L_{\alpha-1})$. But then $Z(Q_{\alpha-1})=Z_{\alpha-1}$ and by conjugacy, $Z(Q_{\alpha'-1})=Z_{\alpha'-1}\le Z_{\alpha'-2}\le Q_{\alpha-1}$, a contradiction.

Thus, $Z(Q_{\alpha'-1})\le Q_{\alpha-1}$ and so, $[Z(Q_{\alpha'-1}), V_{\alpha-1}]\le Z_{\alpha-1}\cap Z(Q_{\alpha'-1})$. Since $Z_{\alpha-1}$ does not centralize $Z_{\alpha'}$, we deduce that $[Z(Q_{\alpha'-1}), V_{\alpha-1}]=\{1\}$. But then $Z(Q_{\alpha'-1})$ is centralized by $V_{\alpha-1}Q_{\alpha'-1}\in\syl_p(L_{\alpha'-1})$ and $Z(Q_{\alpha'-1})=Z_{\alpha'-1}$, as required.
\end{proof}

Combining \cref{EvenNat} and \cref{EvenZ}, we now satisfy \cref{CommonHyp}. Thus, whenever $b$ and the non-central chief factors in $V_{\lambda}^{(n)}$ satisfy the necessary requirements for $\lambda\in\{\alpha,\beta\}$ and various values of $n$, we may freely apply the results contained between \cref{VBGood} and \cref{GoodAction4}.

\begin{lemma}\label{GoodVA}
Suppose that $b>2$. Then $|V_{\beta}|=q^3$ and $[V_{\alpha}^{(2)}, Q_{\alpha}]=Z_{\alpha}$.
\end{lemma}
\begin{proof}
If $V_{\alpha}^{(2)}\le Q_{\alpha'-2}$, then $Z_{\alpha}(V_{\alpha}^{(2)}\cap Q_{\alpha'})$ has index $q$ in $V_{\alpha}^{(2)}$ so that $V_{\alpha}^{(2)}/Z_{\alpha}$ has a unique non-central chief factor. Then the result holds by \cref{VBGood}. Thus, we suppose that $V_{\alpha}^{(2)}\not\le Q_{\alpha'-2}$. Then there is $\alpha-2$ such that $(\alpha-2, \alpha'-2)$ is a critical pair and by \cref{ineq}, we have that $Z_{\alpha-1}=Z_{\alpha'-3}$. Since $b>2$ and $Z_{\beta}Z_{\alpha-1}\le Z_{\alpha}\cap Z_{\alpha'-2}$, it follows that $Z_{\beta}=Z_{\alpha-1}=Z_{\alpha'-3}=Z_{\alpha'-1}$. If $|V_{\beta}|\ne q^3$, then by \cref{GoodAction1} since $Z_{\alpha}(V_{\alpha}^{(2)}\cap Q_{\alpha'-2}\cap Q_{\alpha'-1})$ has index at most $q^2$ in $V_{\alpha}^{(2)}$, $O^p(R_{\alpha})$ centralizes $V_{\alpha}^{(2)}$. By \cref{SimExt}, $Z_{\alpha-2}\le V_{\alpha-1}=V_{\beta}\le Q_{\alpha'-2}$, a contradiction. 
\end{proof}

\begin{lemma}\label{bneq 4}
$b\ne 4$.
\end{lemma}
\begin{proof}
Since none of the conclusions of \hyperlink{MainGrpThm}{Theorem C} satisfy $b=4$, we may suppose that $G$ is a minimal counterexample with $b=4$. Suppose that $V_{\alpha}^{(2)}\le Q_{\alpha'-2}$. Then $V_{\alpha}^{(2)}\cap Q_{\alpha'-1}=Z_{\alpha}(V_{\alpha}^{(2)}\cap Q_{\alpha'})$ is an index $q$ subgroup of $V_{\alpha}^{(2)}$ which is centralized, modulo $Z_{\alpha}$, by $Z_{\alpha'}$. Thus, $V_{\alpha}^{(2)}/Z_{\alpha}$ is an FF-module for $\bar{L_{\alpha}}$. Then \cref{GoodAction1} implies that $O^p(R_{\alpha})$ centralizes $V_{\alpha}^{(2)}$ and since $Z_{\alpha'-1}=Z_{\beta}$, \cref{SimExt} implies that $Z_{\alpha}\le V_{\beta}=V_{\alpha'-1}\le Q_{\alpha'}$, a contradiction. We have a similar contradiction if $V_{\alpha}^{(2)}\cap Q_{\alpha'-2}\le Q_{\alpha'-1}$.

Thus, $V_{\alpha}^{(2)}\not\le Q_{\alpha'-2}$ and $V_{\alpha}^{(2)}\cap Q_{\alpha'-2}\not\le Q_{\alpha'-1}$. In particular, $V_{\alpha}^{(2)}$ is non-abelian and $Z_{\alpha}\le \Phi(Q_{\alpha})$. Suppose that $r\in L_{\alpha}$ is of order coprime to $p$ and centralizes $V_{\alpha}^{(2)}$. Then, by the three subgroup lemma, $r$ centralizes $Q_{\alpha}/C_{Q_{\alpha}}(V_{\alpha}^{(2)})$. Since $C_{Q_{\alpha}}(V_{\alpha}^{(2)})\le Q_{\alpha'-2}$ and $V_{\alpha}^{(2)}\cap Q_{\alpha'-2}\not\le Q_{\alpha'-1}$, we have that $Z_{\alpha'}$ centralizes $C_{Q_{\alpha}}(V_{\alpha}^{(2)})V_{\alpha}^{(2)}/V_{\alpha}^{(2)}$ so that $O^p(L_{\alpha})$ centralizes $C_{Q_{\alpha}}(V_{\alpha}^{(2)})V_{\alpha}^{(2)}/V_{\alpha}^{(2)}$. By coprime action, $r$ centralizes $Q_{\alpha}$, and so $r=1$. Thus, every $p'$-element of $L_{\alpha}$ acts faithfully on $V_{\alpha}^{(2)}/\Phi(V_{\alpha}^{(2)})$.

Now, $Z_{\alpha}(V_{\alpha}^{(2)}\cap \dots \cap Q_{\alpha'})$ has index at most $q^2$ in $V_{\alpha}^{(2)}$ so that $V_{\alpha}^{(2)}/Z_{\alpha}$ is a $2$F-module for $\bar{L_{\alpha}}$. Furthermore, \[[V_{\alpha}^{(2)}, V_{\alpha'-1}, V_{\alpha'-1}]\le [V_{\alpha}^{(2)}, V_{\alpha'-2}^{(2)}, V_{\alpha'-1}]\le [Q_{\alpha'-1}, V_{\alpha'-1}]=Z_{\alpha'-1}=Z_{\beta}\] and $V_{\alpha}^{(2)}/Z_{\alpha}$ is a faithful quadratic $2$F-module for $\bar{L_{\alpha}}$. Then $\bar{L_{\alpha}}$ is determined by \cref{SEFF}, \cref{Quad2F} and \cref{2FRecog} and since $\bar{L_{\alpha}}$ has a quotient isomorphic to $\SL_2(q)$, we have that $\bar{L_{\alpha}}\cong \SL_2(q), \SU_3(2)', (3\times 3):2$ or $(Q_8\times Q_8):3$. Notice that $V_{\beta}/Z_{\alpha}$ is of order $q$ and is not contained in $C_{V_{\alpha}^{(2)}/Z_{\alpha}}(O^p(L_{\alpha}))$. Setting $V:=V_{\alpha}^{(2)}/Z_{\alpha}$ there is a $G_{\alpha,\beta}$-invariant subgroup of $V/C_V(O^p(L_{\alpha}))$ of order $q$ which generates $V$ and by \cref{pgen}, we have that $\bar{L_{\alpha}}\cong (3\times 3):2$. Moreover, since $V_{\alpha}^{(2)}/Z_{\alpha}$ contains two non-central chief factors for $L_{\alpha}$, for $U_\alpha:=[V_{\alpha}^{(2)}, L_{\alpha}]$, we have that $Z_{\alpha'-2}=Z_{\alpha'-1}[U_\alpha\cap Q_{\alpha'-2}, V_{\alpha'-1}]\le U_\alpha$ so that $V_{\beta}\le U_\alpha$, $V_{\alpha}^{(2)}=U_\alpha$ and $|V_{\alpha}^{(2)}/Z_{\alpha}|=2^4$.

Let $P_{\alpha}\le L_{\alpha}$ with $S\le P_{\alpha}$, $P_{\alpha}/Q_{\alpha}\cong \Sym(3)$, $L_{\alpha}=P_{\alpha}R_{\alpha}$ and $O_3(\bar{P_{\alpha}})\normaleq \bar{L_{\alpha}}$. Then $P_{\alpha}$ is $G_{\alpha,\beta}$-invariant and upon showing that no non-trivial subgroup of $S$ is normalized by both $P_{\alpha}$ and $G_{\beta}$, then triple $(P_{\alpha}G_{\alpha,\beta}, G_{\beta}, G_{\alpha,\beta})$ satisfies \cref{MainHyp}. To this end, suppose that $Q$ is non-trivial subgroup of $S$ normalized by $P_{\alpha}$ and $G_{\beta}$. Then $Z_{\beta}\le Q$ so that $Z_{\beta}\le \Omega(Z(Q))$. Taking consecutive normal closure, we deduce that $V_{\beta}\le \Omega(Z(Q))$ and $\Omega(Z(Q))/Z_{\alpha}$ contains some of the non-central $L_{\alpha}$-chief factors contained in $V_{\alpha}^{(2)}/Z_{\alpha}$. Write $W$ for the preimage in $V_{\alpha}^{(2)}$ of some non-central chief factor contained in $\Omega(Z(Q))\cap V_{\alpha}^{(2)}/Z_{\alpha}$, noting that by the definition of $V_{\alpha}^{(2)}$, $W\cap V_{\beta}=Z_{\alpha}$. However, $WV_{\beta}\le \Omega(Z(Q))$ and $[W, V_{\beta}]=\{1\}$ so that $W\le Q_{\alpha'-2}$ and $[W, V_{\alpha'-2}]\le Z_{\alpha'-2}\cap W=Z_{\beta}=Z_{\alpha'-1}$ and $W=Z_\alpha(W\cap Q_{\alpha'})$. Then $W$ contains no non-central chief factor for $L_{\alpha}$, a contradiction. Thus, $Q=\{1\}$ and $(P_{\alpha}G_{\alpha,\beta}, G_{\beta}, G_{\alpha,\beta})$  satisfies \cref{MainHyp}. Assuming that $G$ is a minimal counterexample to \hyperlink{MainGrpThm}{Theorem C}, we conclude that $P_{\alpha}/Q_{\alpha}\cong \Sym(3)\cong \bar{L_{\beta}}$ and $(P_{\alpha}G_{\alpha,\beta}, G_{\beta}, G_{\alpha,\beta})$ is a weak BN-pair of rank $2$. By \cite{Greenbook}, $|S|\leq 2^7$ and since $|V_{\alpha}^{(2)}|=2^6$ and $Q_{\alpha}/V_{\alpha}^{(2)}$ contains a non-central chief factor for $L_{\alpha}$, we have a contradiction.
\end{proof}

\begin{lemma}\label{bge8}
Suppose that $b>2$. Then the following hold:
\begin{enumerate}
\item $V_{\alpha}^{(2)}\le Q_{\alpha'-2}$ but $V_{\alpha}^{(2)}\not\le Q_{\alpha'-1}$;
\item $[V_{\alpha}^{(2)}, Q_{\alpha}]=Z_{\alpha}$ and $|V_{\beta}|=q^3$;
\item $O^p(R_{\alpha})$ centralizes $V_{\alpha}^{(2)}$ and $V_{\alpha}^{(2)}/Z_{\alpha}$ is a faithful FF-module for $L_{\alpha}/R_{\alpha}\cong\SL_2(q)$;
\item $b\geq 8$; and
\item $Z_{\alpha'-2}\le V_{\alpha}^{(2)}\le Z(V_{\alpha}^{(4)})$.
\end{enumerate}
\end{lemma}
\begin{proof}
By \cref{bneq 4}, we have that $b>4$ so that $V_{\alpha}^{(2)}$ is abelian. Moreover, (ii) holds by \cref{GoodVA}. Suppose first that $V_{\alpha}^{(2)}\not\le Q_{\alpha'-2}$ so that there is a critical pair $(\alpha-2, \alpha'-2)$ such that $[Z_{\alpha-2}, Z_{\alpha'-2}]=Z_{\alpha-1}=Z_{\alpha'-3}$. Since $b>2$, $Z_{\alpha}\ne Z_{\alpha'-2}$ and $Z_{\alpha-1}=Z_{\beta}$. Now, $[V_{\alpha}^{(2)}\cap Q_{\alpha'-2}, V_{\alpha'-1}]\le Z_{\alpha'-2}\cap V_{\alpha}^{(2)}$. Since $V_{\alpha}^{(2)}$ is abelian and $V_{\alpha}^{(2)}\not\le Q_{\alpha'-2}$, $Z_{\alpha'-2}\not\le V_{\alpha}^{(2)}$. But $Z_{\alpha'-1}\le V_{\alpha}^{(2)}$ and so it follows that $[V_{\alpha}^{(2)}\cap Q_{\alpha'-2}, V_{\alpha'-1}]\le Z_{\alpha'-1}$ and $V_{\alpha}^{(2)}\cap Q_{\alpha'-2}\le Q_{\alpha'-1}$. Then $V_{\alpha}^{(2)}/Z_{\alpha}$ is an FF-module and by \cref{GoodAction1}, $O^p(R_{\alpha})$ centralizes $V_{\alpha}^{(2)}$. But then by \cref{SimExt}, $Z_{\alpha-2}\le V_{\alpha-1}=V_{\beta}\le Q_{\alpha'-2}$, a contradiction since $(\alpha-2, \alpha'-2)$ is a critical pair.

Thus, $V_{\alpha}^{(2)}\le Q_{\alpha'-2}$. If $V_{\alpha}^{(2)}\le Q_{\alpha'-1}$, then $V_{\alpha}^{(2)}=Z_{\alpha}(V_{\alpha}^{(2)}\cap Q_{\alpha'})$ and $O^p(L_{\alpha})$ would centralize $V_{\alpha}^{(2)}/Z_{\alpha}$, a contradiction, and so (i) holds. Now, it follows that $V_{\alpha}^{(2)}/Z_{\alpha}$ is an FF-module and by \cref{GoodAction1}, $O^p(R_{\alpha})$ centralizes $V_{\alpha}^{(2)}$ and (iii) holds.

Since $V_{\alpha}^{(2)}\not\le Q_{\alpha'-1}$, we infer that $Z_{\alpha'-2}=[V_{\alpha}^{(2)}, V_{\alpha'-1}]Z_{\alpha'-1}\le V_{\alpha}^{(2)}$. If $b\geq 8$, then $V_{\alpha}^{(2)}\le Z(V_{\alpha}^{(4)})$ and (v) holds, and so we may assume that $b=6$ for the remainder of the proof. Notice that if $Z_{\alpha'-1}=Z_{\alpha'-3}$, it follows from \cref{SimExt} that $Z_{\alpha'}\le V_{\alpha'-1}=V_{\alpha'-3}\le Q_{\alpha}$, a contradiction. Since $Z_{\beta}=Z_{\alpha'-1}\ne Z_{\alpha'-3}$ and $b=6$, we have that $Z_{\alpha'-2}=Z_{\alpha+2}$. Let $\alpha-1\in\Delta(\alpha)$ such that $V_{\alpha-1}\not\le Q_{\alpha'-1}$ and $Z_{\alpha-1}\ne Z_{\beta}$, chosen as in \cref{EvenNat}. We have that $V_{\alpha'-1}^{(3)}\le Q_{\alpha+2}$ since $V_{\alpha'-1}^{(3)}$ centralizes $Z_{\alpha+2}=Z_{\alpha'-2}\le V_{\alpha'-1}$. Then $V_{\alpha'-1}^{(3)}\cap Q_{\beta}=V_{\alpha'-1}(V_{\alpha'-1}^{(3)}\cap Q_{\alpha})$ and \[[V_{\alpha-1}, V_{\alpha'-1}^{(3)}\cap Q_{\alpha}]\le [V_{\alpha}^{(2)}, V_{\alpha'-1}^{(3)}\cap Q_{\alpha}]\le Z_{\alpha}\cap V_{\alpha'-1}^{(3)}=Z_\beta=Z_{\alpha'-1}.\] In particular, $V_{\alpha'-1}^{(3)}/V_{\alpha'-1}$ contains a unique non-central chief factor $L_{\alpha'-1}$ which, as a $\mathrm{GF}(p)\bar{L_{\alpha'-1}}$-module is isomorphic to a natural $\SL_2(q)$-module. Thus, we may apply \cref{GoodAction3} so that $O^p(R_{\alpha'-1})$ acts trivially on $V_{\alpha'-1}^{(3)}$. Since $Z_{\alpha+2}=Z_{\alpha'-2}$, it follows from \cref{SimExt} that $Z_{\alpha'}\le V_{\alpha'-2}^{(2)}=V_{\alpha+2}^{(2)}\le Q_{\alpha}$, an obvious contradiction. Thus, $b\geq 8$ and the lemma holds.
\end{proof}

\begin{lemma}\label{b=2}
$b=2$.
\end{lemma}
\begin{proof}
We may suppose that $b\geq 8$ by \cref{bge8}. Suppose first that $V_{\alpha}^{(4)}\not\le Q_{\alpha'-4}$. Since $Z_{\alpha'-3}\le Z_{\alpha'-2}\le Z(V_{\alpha}^{(4)})$ is centralized by $V_{\alpha}^{(4)}$, it follows that $Z_{\alpha'-3}=Z_{\alpha'-5}$ and by \cref{SimExt}, we have that $V_{\alpha'-3}=V_{\alpha'-5}$. Now, $[V_{\alpha}^{(4)}\cap Q_{\alpha'-4}, V_{\alpha'-3}]=[V_{\alpha}^{(4)}\cap Q_{\alpha'-4}, V_{\alpha'-5}]\le Z_{\alpha'-5}=Z_{\alpha'-3}$ and so $V_{\alpha}^{(4)}\cap Q_{\alpha'-4}\le Q_{\alpha'-3}$. Since $V_{\alpha}^{(4)}$ centralizes $Z_{\alpha'-2}$, we deduce that $V_{\alpha}^{(4)}\cap Q_{\alpha'-4}=V_{\alpha}^{(2)}(V_{\alpha}^{(4)}\cap Q_{\alpha'-4}\cap Q_{\alpha'-1})$ and so $V_{\alpha}^{(4)}/V_{\alpha}^{(2)}$ contains a unique non-central chief factor for $L_{\alpha}$. Now, by \cref{GoodAction2} and \cref{SimExt}, since $Z_{\alpha'-3}=Z_{\alpha'-5}$ we conclude that $Z_{\alpha'}\le V_{\alpha'-3}^{(3)}=V_{\alpha'-5}^{(3)}\le Q_{\alpha}$, a contradiction.

Therefore, we continue assuming that $V_{\alpha}^{(4)}\le Q_{\alpha'-4}$. Then $V_{\alpha}^{(4)}\cap Q_{\alpha'-3}$ centralizes $Z_{\alpha'-2}$ and we may assume that $V_{\alpha}^{(4)}\not\le Q_{\alpha'-3}$, else $V_{\alpha}^{(4)}=V_{\alpha}^{(2)}(V_{\alpha}^{(2)}\cap Q_{\alpha'-1})$ and $O^p(L_{\alpha})$ centralizes $V_{\alpha}^{(4)}/V_{\alpha}^{(2)}$. Since $|V_{\alpha'-3}|=q^3$, $V_{\alpha}^{(4)}\not\le Q_{\alpha'-3}$ and $V_{\alpha}^{(4)}$ centralizes $Z_{\alpha'-2}$, by \cref{BasicVB} $V_{\alpha'-3}\ne Z_{\alpha'-2}Z_{\alpha'-4}$ and so, $Z_{\alpha'-2}=Z_{\alpha'-4}$. If $O^p(R_{\beta})$ centralizes $V_{\beta}^{(3)}$ then applying \cref{SimExt} to $Z_{\alpha'-2}=Z_{\alpha'-4}$ yields $Z_{\alpha'}\le V_{\alpha'-2}^{(2)}=V_{\alpha'-4}^{(2)}\le Q_{\alpha}$, a contradiction. Thus, to obtain a final contradiction, by \cref{GoodAction3}, it suffices to show that $V_{\alpha'-1}^{(3)}/V_{\alpha'-1}$ contains a unique non-central chief factor for $L_{\alpha'-1}$ which, as a $\mathrm{GF}(p)\bar{L_{\alpha'-1}}$-module, is an FF-module.

By the symmetry in the hypothesis of $(\alpha, \alpha')$ and $(\alpha', \alpha)$, we may assume that $Z_{\alpha+2}=Z_{\alpha+4}$. Let $\alpha-1\in\Delta(\alpha)$ such that $V_{\alpha-1}\not\le Q_{\alpha'-1}$ and $Z_{\alpha-1}\ne Z_{\beta}$, as in \cref{EvenNat}. Then $V_{\alpha'-1}^{(3)}$ centralizes $Z_{\alpha+2}$ so that $V_{\alpha'-1}^{(3)}\le Q_{\alpha+2}$, $V_{\alpha'-1}^{(3)}\cap Q_{\beta}=V_{\alpha'-1}(V_{\alpha'-1}^{(3)}\cap Q_{\alpha})$ and \[[V_{\alpha-1}, V_{\alpha'-1}^{(3)}\cap Q_{\alpha}]\le [V_{\alpha}^{(2)}, V_{\alpha'-1}^{(3)}\cap Q_{\alpha}]\le Z_{\alpha}\cap V_{\alpha'-1}^{(3)}=Z_\beta=Z_{\alpha'-1}.\] In particular, either $O^p(L_{\alpha'-1})$ centralizes $V_{\alpha'-1}^{(3)}/V_{\alpha'-1}$ or $V_{\alpha'-1}^{(3)}/V_{\alpha'-1}$ contains a unique non-central chief factor for $L_{\alpha'-1}$, and the result holds.
\end{proof}

\begin{proposition}\label{5Symp}
Suppose $p\geq 5$. Then $R_{\alpha}=Q_{\alpha}$, $G$ is a symplectic amalgam and one of the following holds:
\begin{enumerate}
\item $G$ is locally isomorphic to $H$ where $F^*(H)\cong\mathrm{G}_2(p^n)$;
\item $G$ is locally isomorphic to $H$ where $F^*(H)\cong{}^3\mathrm{D}_4(p^n)$; 
\item $p=5$, $|S|=5^6$, $Q_{\beta}\cong 5^{1+4}_+$ and $\bar{L_{\beta}}\cong 2^{1+4}_-.5$;
\item $p=5$, $|S|=5^6$, $Q_{\beta}\cong 5^{1+4}_+$ and $\bar{L_{\beta}}\cong 2^{1+4}_-.\Alt(5)$;
\item $p=5$, $|S|=5^6$, $Q_{\beta}\cong 5^{1+4}_+$ and $\bar{L_{\beta}}\cong 2\cdot\Alt(6)$; or
\item $p=7$, $|S|=7^6$, $Q_{\beta}\cong 7^{1+4}_+$ and $\bar{L_{\beta}}\cong 2\cdot\Alt(7)$.
\end{enumerate}
\end{proposition}
\begin{proof}
By \cref{b=2}, we have that $b=2$. Note that $Q_{\alpha}\cap Q_\beta=Z_{\alpha}(Q_\alpha\cap Q_\beta\cap Q_{\alpha'})$. Since $Z_{\alpha'}\le Q_\beta$, it follows that $[Q_{\alpha}, Z_{\alpha'}, Z_{\alpha'}, Z_{\alpha'}]=\{1\}$. Then by \cref{CubicAction} applied to $Q_{\alpha}/\Phi(Q_{\alpha})$, we have that $\bar{L_{\alpha}}\cong\SL_2(q)$.

We now intend to show that the amalgam is symplectic. We immediately satisfy condition (i) in the definition of a symplectic amalgam. We have that $W:=\langle (Q_\alpha\cap Q_\beta)^{L_\alpha}\rangle\not\le Q_\beta$, for otherwise $W=Q_\alpha\cap Q_\beta\normaleq L_\alpha$, a contradiction by \cref{push}. Therefore, by \cref{p-closure} (iii), we have that $G_\beta=\langle W^{L_\beta} \rangle N_{G_\beta}(S)$, satisfying condition (ii). From our hypothesis, we automatically satisfy condition (iii). By \cref{BetaCenterIII}, we satisfy condition (iv). Since $b=2$ and $d(\alpha,\beta)=1$, we have that $Z_\alpha\le Q_\beta$. Moreover, by hypothesis and the symmetry between $\alpha$ and $\alpha'$ we have that $Z_\alpha\not\le Q_{\alpha'}=Q_\alpha^x$ for some $x\in G_\beta$. Hence, $G$ is a symplectic amalgam and the result holds by \cref{Symp}.
\end{proof}

Thus, we have reduced to the case where $b=2$ and $p\in\{2,3\}$. Since \cref{EvenNat} only applied to the cases where $b\geq 4$, we have no knowledge of the structure of $\bar{L_{\beta}}$ or $V_{\beta}$. As intimated earlier, we attempt to show that $R_{\alpha}=Q_{\alpha}$ and apply the results in \cite{parkerSymp}.

\begin{lemma}\label{b=2lemma}
Suppose that $b=2$ and $p\in\{2,3\}$. Then the following hold:
\begin{enumerate}
    \item $[V_{\beta}, Q_{\beta}]=Z_{\beta}\le Z_{\alpha}\le \Phi(Q_{\alpha})$ and $Q_{\alpha}/\Phi(Q_{\alpha})$ is a faithful quadratic module for $\bar{L_{\alpha}}$;
    \item if $q_\alpha>p=2$ and $R_\alpha\ne Q_\alpha$ then $q_\beta>q_\alpha^2$ except perhaps when $\bar{L_{\beta}}/O_{2'}(\bar{L_{\beta}})\cong \PSU_3(q_\beta)$ in which case $q_\beta>q_\alpha$;
    \item if $p=3$ and $R_{\alpha}\ne Q_{\alpha}$ then $q_\alpha=3$; and
    \item $\Omega(Z(Q_{\alpha}))=Z_{\alpha}$.
\end{enumerate}
\end{lemma}
\begin{proof}
That $[V_{\beta}, Q_{\beta}]=Z_{\beta}\le Z_{\alpha}$ is contained in \cref{ineq}. Then $[Q_{\alpha}, V_{\beta}, V_{\beta}]\le Z_{\beta}\le Z_{\alpha}$, and so to prove (i) it suffices to show that $Z_{\alpha}\le \Phi(Q_{\alpha})$. Indeed, since $Z_{\alpha}$ is irreducible under the action of $L_{\alpha}$, if $\Phi(Q_{\alpha})$ is non-trivial, then (i) holds. So assume that $Q_{\alpha}$ is elementary abelian. If $R_\alpha=Q_\alpha$ then $\mathcal{A}$ is a symplectic amalgam and (i) holds comparing with \cite{parkerSymp}, so we assume that $R_\alpha\ne Q_\alpha$. Then applying coprime action, we have that $Q_{\alpha}=[Q_{\alpha}, R_{\alpha}]\times C_{Q_{\alpha}}(R_{\alpha})$ is an $S$-invariant decomposition. But $Z_{\beta}\le Z_{\alpha}\le C_{Q_{\alpha}}(R_{\alpha})$ from which it follows that $Q_{\alpha}=C_{Q_{\alpha}}(R_{\alpha})$ and $R_{\alpha}=Q_{\alpha}$, a contradiction. Hence (i).

Now, if $p=3$ then since $Q_{\alpha}/\Phi(Q_{\alpha})$ is a quadratic module, applying \cref{SEQuad} we deduce that either $R_{\alpha}=Q_{\alpha}$ or $|S/Q_\alpha|=3$ and (iii) holds. Suppose that $p=2<q_\alpha$. If $m_2(S/Q_{\beta})=1$, then $\Phi(Q_{\alpha})(Q_{\alpha}\cap Q_{\beta})$ has index at most $4$ in $Q_{\alpha}$ and applying \cref{SEQuad} we deduce that $\bar{L_{\alpha}}\cong \PSL_2(4)$. Hence, if $R_\alpha\ne Q_\alpha$, then $m_2(S/Q_\beta)>1$. If $\bar{L_{\beta}}/O_{2'}(\bar{L_{\beta}})\cong \Sz(q_\beta)$ or $\PSL_2(q_\beta)$, then $\Phi(Q_{\alpha})(Q_{\alpha}\cap Q_{\beta})$ has index $q_\beta$ in $Q_\alpha$. Applying \cref{2FRecog}, if $R_\alpha\ne Q_\alpha$, then we have that $q_\beta>q_\alpha^2$. If $\bar{L_{\beta}}/O_{2'}(\bar{L_{\beta}})\cong \PSU_3(q_\beta)$ then $\Phi(Q_{\alpha})(Q_{\alpha}\cap Q_{\beta})$ has index $q_\beta^2$ in $Q_\alpha$ and applying \cref{2FRecog}, if $R_\alpha\ne Q_\alpha$ then $q_\beta>q_\alpha$. Hence, (ii) holds.

Assume that $\Omega(Z(Q_{\alpha}))\not \le Q_\beta$. Since $\Omega(Z(Q_{\alpha}))$ is $G_{\alpha,\beta}$-invariant and elementary abelian, we have that $[Q_{\beta}, \Omega(Z(Q_{\alpha})), \Omega(Z(Q_{\alpha}))]=\{1\}$ and applying \cref{NotQuad} when $\bar{L_{\beta}}/O_{3'}(\bar{L_{\beta}})\cong \Ree(q_\beta)$ we deduce in all cases that $|\Omega(Z(Q_{\alpha}))Q_{\beta}/Q_{\beta}|=q_\beta$. But now, by (ii) and (iii), $q_\beta\geq q_\alpha$ and since $Q_{\beta}\cap Q_{\alpha}$ is a subgroup of $Q_{\beta}$ of index $q_\alpha$ which is centralized by $\Omega(Z(Q_{\alpha}))$, applying \cref{SEFF} the only possibility is that $\bar{L_{\beta}}\cong \SL_2(p)$. But then, $\Phi(Q_{\alpha})(Q_{\alpha}\cap Q_{\beta})$ has index $p$ in $Q_{\alpha}$ and \cref{SEFF} implies that $\bar{L_{\alpha}}\cong\SL_2(p)$. Therefore, $\mathcal{A}$ is symplectic and \cite[Lemma 5.3]{parkerSymp} yields a contradiction.

Hence, $\Omega(Z(Q_{\alpha}))\le Q_{\beta}$ and $\Omega(Z(Q_{\alpha}))=Z_{\alpha}(\Omega(Z(Q_{\alpha}))\cap Q_{\alpha'})$. Moreover, \sloppy{$\Omega(Z(Q_{\alpha}))\cap Q_{\alpha'}$} is centralized by $S=Z_{\alpha'}Q_{\alpha}$ so that $\Omega(Z(Q_{\alpha}))\cap Q_{\alpha'}=\Omega(Z(S))=Z_{\beta}\le Z_{\alpha}$. Hence, (iv) holds. 
\end{proof}

For the remainder of this section, we set $L:=\langle V_{\beta}, V_{\beta}^x\rangle Q_{\alpha}$ with $x\in L_{\alpha}$ chosen such that $Z_{\beta}^x\ne Z_{\beta}$ and $x^2\in G_{\alpha,\beta}$. In particular, $LR_{\alpha}=L_{\alpha}$. We write $\alpha-1:=\beta^x$.

Note that if $p=2$ then we may choose $x\in T\setminus Q_\alpha$ for some appropriate $T\in\syl_2(L_{\alpha})\setminus \{S\}$. In particular, we can arrange that $x\in L$. If $p=3$ then applying \cref{b=2lemma}, we have that $L_{\alpha}/R_{\alpha}\cong \SL_2(3)$. Moreover, since $[Q_{\alpha}, V_{\beta}, V_{\beta}]\le Z_{\alpha}$, every non-central chief factor $U/V$ within $Q_{\alpha}$ for $L_{\alpha}$ has $L_{\alpha}/C_{L_{\alpha}}(U/V)\cong \SL_2(3)$. By \cref{Badp3}, we deduce that whenever $C_{L_{\alpha}}(U/V)\ne R_{\alpha}$, we have that $L_{\alpha}/C_{L_{\alpha}}(U/V)\cap R_{\alpha}\cong (Q_8\times Q_3):3$. Extending this argument yields a subgroup $U\le \bar{L_{\alpha}}$ with $U\cong Q_8$ and $O^p(\bar{L_{\alpha}})\cong U\times \bar{R_{\alpha}}$. Then for any $x$ of order $4$ in the preimage of $U$ in $L_{\alpha}$, $x^2\in G_{\alpha,\beta}$ and $Z_{\beta}^x\ne Z_{\beta}$, as required.

\begin{lemma}\label{VBStrucI}
Suppose that $b=2$ and $p\in\{2,3\}$. Then the following hold:
\begin{enumerate}
    \item $Q_{\beta}\cap Q_{\alpha}\cap Q_{\alpha-1}$ is elementary abelian;
    \item if $V_{\beta}\cap Q_{\alpha}\cap Q_{\alpha-1}=V_{\alpha-1}\cap Q_{\alpha}\cap Q_{\beta}=V_{\beta}\cap Q_{\alpha}\cap Q_{\alpha-1}$;
    \item $V_{\beta}\cap Q_{\alpha}\not\le Q_{\alpha-1}$ and $V_{\alpha-1}\cap Q_{\alpha}\not\le Q_{\beta}$; and
    \item if $V<V_{\beta}$ with $V\normaleq G_{\beta}$ then $V\cap Q_{\alpha}\cap Q_{\alpha-1}=Z_{\beta}$.
\end{enumerate}
\end{lemma}
\begin{proof}
Since $[Q_{\beta}, V_{\beta}]=Z_{\beta}\le Z_{\alpha}$, $(Q_{\beta}\cap Q_{\alpha}\cap Q_{\alpha-1})/Z_{\alpha}$ is centralized by $O^p(L)$. Since $S=V_{\beta}Q_{\alpha}$ normalizes $Q_{\beta}\cap Q_{\alpha}\cap Q_{\alpha-1}$, if $Q_{\beta}\cap Q_{\alpha}\cap Q_{\alpha-1}$ is not elementary abelian then $Z_{\beta}\cap \Phi(Q_{\beta}\cap Q_{\alpha}\cap Q_{\alpha-1})$ is non-trivial and the construction of $L$ yields that $Z_{\alpha}\le \Phi(Q_{\beta}\cap Q_{\alpha}\cap Q_{\alpha-1})$, a contradiction. Thus, $Q_{\beta}\cap Q_{\alpha}\cap Q_{\alpha-1}$ is elementary abelian, completing the proof of (i).

By the choice of $x$, we have that $V_{\beta}^x=V_{\alpha-1}$ and $V_{\alpha-1}^x=V_{\beta}$. Moreover, $Z_{\alpha}\le V_{\beta}\cap Q_{\alpha}\cap Q_{\alpha-1}$ and $x$ normalizes $V_{\beta}\cap Q_{\alpha}\cap Q_{\alpha-1}$ so that $V_{\beta}\cap Q_{\alpha}\cap Q_{\alpha-1}=(V_{\beta}\cap Q_{\alpha}\cap Q_{\alpha-1})^x=V_{\alpha-1}\cap Q_{\alpha}\cap Q_{\beta}$ and (ii) holds.

Suppose that $V_{\beta}\cap Q_{\alpha}\le Q_{\alpha-1}$. Then, by (ii), $V_{\beta}\cap Q_{\alpha}=V_{\alpha-1}\cap Q_{\alpha}\normaleq L$ and $V_{\beta}$ centralizes the $L$-invariant series $\{1\}\normaleq Z_{\alpha}\normaleq V_{\beta}\cap Q_{\alpha}\normaleq Q_{\alpha}$. Since $O_p(L)=Q_{\alpha}$ and $V_{\beta}\not\le Q_{\alpha}$, an application of coprime action yields a contradiction. The action of $L$ implies also that $V_{\alpha-1}\cap Q_{\alpha}\not\le Q_{\beta}$.

Let $V$ be a $p$-group normal in $G_{\beta}$ and contained in $V_{\beta}$. Since $G_{\alpha,\beta}$ acts irreducibly on $Z_{\beta}$, $Z_{\beta}\le V$. Moreover, $V\cap Q_{\alpha}\cap Q_{\alpha-1}$ is elementary abelian, and $(V\cap Q_{\alpha}\cap Q_{\alpha-1})Z_{\alpha}\normaleq L$. Then $[Q_{\alpha}, V\cap Q_{\alpha}\cap Q_{\alpha-1}]=[Q_{\alpha}, (V\cap Q_{\alpha}\cap Q_{\alpha-1})Z_{\alpha}]\normaleq L$. If $[Q_{\alpha}, V\cap Q_{\alpha}\cap Q_{\alpha-1}]\cap Z_{\beta}$ is non-trivial, then by the construction of $L$, $Z_{\alpha}\le [Q_{\alpha}, V\cap Q_{\alpha}\cap Q_{\alpha-1}]\le V$, a contradiction. Thus, $[Q_{\alpha}, V\cap Q_{\alpha}\cap Q_{\alpha-1}]=\{1\}$ and $Z_{\beta}\le V\cap Q_{\alpha}\cap Q_{\alpha-1}\le \Omega(Z(Q_{\alpha}))=Z_{\alpha}$. If $V\cap Z_{\alpha}>Z_{\beta}$, then since $G_{\alpha,\beta}$ acts irreducibly on $Z_{\alpha}/Z_{\beta}$ and $V\normaleq G_{\alpha,\beta}$, we have that $Z_{\alpha}\le V$ and by the definition of $V_{\beta}$, $V=V_{\beta}$. This completes the proof.
\end{proof}

We finally determine some of the structural properties of $V_{\beta}$ and its chief factors in a general setting.

\begin{lemma}\label{VBStrucII}
Suppose that $b=2$ and $p\in\{2,3\}$. Assume that $R_\alpha \ne Q_\alpha$. Then the following hold:
\begin{enumerate}
    \item $R_{\beta}=Q_{\beta}$ and $O^p(L_\beta)$ centralizes $Q_{\beta}/V_{\beta}$;
    \item if $V<V_{\beta}$ with $V\normaleq G_{\beta}$ then $V\le C_{V_{\beta}}(O^p(L_{\beta}))$; and 
    \item $Z(V_{\beta})=C_{V_{\beta}}(O^p(L_{\beta}))$ and $V_{\beta}/Z(V_{\beta})$ is $\bar{G_{\beta}}$-irreducible.
\end{enumerate}
\end{lemma}
\begin{proof}
By \cref{VBStrucI}, we have that $[Q_{\beta}\cap Q_{\alpha}, V_{\alpha-1}\cap Q_{\alpha}]\le V_{\beta}$ and since $Q_{\beta}=V_{\beta}(Q_{\beta}\cap Q_{\alpha})$ and $V_{\alpha-1}\cap Q_{\alpha}\not\le Q_{\beta}$, we infer that $O^p(L_{\beta})$ centralizes $Q_{\beta}/R_{\beta}$ and $R_{\beta}=Q_{\beta}$.

Let $V<V_{\beta}$ with $V\normaleq G_{\beta}$ and assume that $V$ contains a non-central chief factor for $L_{\beta}$. By \cref{VBStrucI}, we have that $V\cap Q_{\alpha}\cap Q_{\alpha-1}=Z_{\beta}$ has index at most $q_\alpha r_\beta$ in $V$, where $r_\beta=|(V_{\beta}\cap Q_{\alpha})Q_{\alpha-1}/Q_{\alpha-1}|$.

Suppose first that $m_p(S/Q_{\beta})=1$. By \cref{b=2lemma}, we have that $q_\alpha=p$ so that $\Phi(V)Z_\beta$ has index at most $p^3$ in $L_{\beta}$. In particular, unless $p=2$ and $S/Q_{\beta}$ is generalized quaternion, $\Phi(V)Z_\beta$ has index $p^2$ in $V$ and we deduce by \cref{p-closure} that $L_{\beta}/C_{L_{\beta}}(V/\Phi(V))\cong \SL_2(p)$ and $Q_{\beta}\in\syl_p(C_{L_{\beta}}(V/\Phi(V)))$. Even if $p=2$ and $S/Q_{\beta}$ is generalized quaternion, applying \cref{p-closure}, we have that $Q_{\beta}\in\syl_p(C_{L_{\beta}}(V/\Phi(V)))$ and $L_{\beta}/C_{L_{\beta}}(V/\Phi(V))$ is isomorphic to a subgroup of $\GL_3(2)$. In this latter case, $S/Q_{\beta}$ is isomorphic to a subgroup of $\Dih(8)$, a contradiction. Thus, $|S/Q_{\beta}|=p$, $Q_{\alpha}\cap Q_{\beta}$ has index $p$ in $Q_{\alpha}$ and $R_{\alpha}=Q_\alpha$, a contradiction to the initial assumption.

Suppose now that $m_p(S/Q_{\beta})>1$. Since $R_\alpha \ne Q_\alpha$, we may exploit the bounds in \cref{b=2lemma}. Applying \cref{p-closure}, we have that $\bar{C_{L_{\beta}}(V/\Phi(V))}\le O_{p'}(\bar{L_{\beta}})$ so that $L_{\beta}/C_{L_{\beta}}(V/\Phi(V))$ has a strongly $p$-embedded subgroup. If $\bar{L_{\beta}}/O_{2'}(\bar{L_{\beta}})\cong \Sz(q_\beta)$, then $\Phi(V)Z_\beta$ has index at most $q_\alpha q_\beta^2$ in $V$. Since $q_\alpha<q_\beta$ by \cref{b=2lemma}, we deduce that $|V/\Phi(V)|< q_\beta^3$, a contradiction by \cref{SpeMod2}. If $\bar{L_{\beta}}/O_{3'}(\bar{L_{\beta}})\cong \Ree(q_\beta)$ or $\bar{L_{\beta}}/O_{p'}(\bar{L_{\beta}})\cong \PSU_3(q_\beta)$, then $\Phi(V)Z_\beta$ has index at most $q_\alpha q_\beta^3$ in $V$. Since $q_\alpha<q_\beta$ by \cref{b=2lemma}, we deduce that $|V/\Phi(V)|< q_\beta^5$, and a contradiction is provided by \cref{SpeMod2} and \cref{SpeModOdd}. Finally, if $S/Q_{\beta}$ is elementary abelian, then $\Phi(V)Z_\beta$ has index at most $q_\alpha q_\beta$ in $V$. Using that $q_\alpha<q_\beta$ by \cref{b=2lemma}, we conclude that $|V/\Phi(V)|< q_\beta^2$, and a contradiction is provided by \cref{SpeMod2} and \cref{SpeModOdd}.

To complete the proof, we need only deduce (iii). By (ii), we have that $V_{\beta}=[V_{\beta}, O^p(L_{\beta})]\le O^p(L_{\beta})$ and $C_{V_{\beta}}(O^p(L_{\beta}))\le Z(V_{\beta})$. Moreover, since $V_{\beta}$ is non-abelian, (ii) also implies that $Z(V_{\beta})\le C_{V_{\beta}}(O^p(L_{\beta}))$ and there are no proper subgroups of $V_{\beta}$ which are normal in $G_{\beta}$ and properly contain $C_{V_{\beta}}(O^p(L_{\beta}))$. Then (iii) follows immediately from these observations. 
\end{proof}

We complete the $b=2$ case over the next two propositions. The first proposition deals with the case where $m_p(S/Q_{\beta})=1$.

\begin{proposition}\label{b2i}
Suppose that $p\in\{2,3\}$, $b=2$ and $m_p(S/Q_{\beta})=1$. Then $R_\alpha=Q_{\alpha}$, $|S|\leq 2^6$, $\mathcal{A}=\mathcal{A}(G_\alpha, G_\beta, G_{\alpha,\beta})$ is a symplectic amalgam and one of the following holds:
\begin{enumerate}
\item $G$ has a weak BN-pair of rank $2$ and $G$ is locally isomorphic to $H$ where $F^*(H)\cong\mathrm{G}_2(2)'$; or
\item $p=2$, $|S|=2^6$, $Q_{\beta}\cong 2^{1+4}_+$ and $\bar{L_{\beta}}\cong (3\times 3):2$.
\end{enumerate}
\end{proposition}
\begin{proof}
If $R_{\alpha}=Q_{\alpha}$, then $\bar{L_{\alpha}}\cong\SL_2(q_\alpha)$ and similarly to \cref{5Symp}, $\mathcal{A}$ is a symplectic amalgam and the result holds after comparing with the tables listed in \cite{parkerSymp} and an application of \cite{Greenbook} and \cite{Fan}. Hence, we assume throughout that $\bar{L_{\alpha}}\not\cong\SL_2(q)$ and $R_{\alpha}\ne Q_{\alpha}$, and so we may use the results in \cref{b=2lemma} to \cref{VBStrucII}.

If $S/Q_{\beta}$ is cyclic then $\Phi(Q_{\alpha})(Q_{\alpha}\cap Q_{\beta})$ is an index $p$ subgroup of $Q_{\alpha}$ and since $V_{\beta}\not\le Q_{\alpha}$ and $[V_{\beta}, Q_{\alpha}\cap Q_{\beta}]\le\Phi(Q_{\alpha})$, it follows that $Q_{\alpha}/\Phi(Q_{\alpha})$ contains a unique non-central chief factor for $L_{\alpha}$ which is isomorphic to an FF-module for $\bar{L_{\alpha}}\cong\SL_2(p)$, a contradiction. Hence, we may assume that $p=2$ and $S/Q_{\beta}$ is generalized quaternion. But then, $\Phi(Q_{\alpha})(Q_{\alpha}\cap Q_{\beta})$ has index at most $4$ in $Q_{\alpha}$ and we deduce that $L_{\alpha}/R_{\alpha}\cong \Sym(3)$ and $Q_{\alpha}/\Phi(Q_{\alpha})$ is a faithful, quadratic $2$F-module for $\bar{L_{\alpha}}$. Applying \cref{Quad2F} with the stipulation that $L_{\alpha}/R_{\alpha}\cong\Sym(3)$ and $R_{\alpha}\ne Q_{\alpha}$, we have that $\bar{L_{\alpha}}\cong (3\times 3):2$ or $\SU_3(2)'$.

Now, $[V_{\beta}, V_{\beta}]=Z_{\beta}\le Q_{\alpha-1}$ and so $(V_{\beta}\cap Q_{\alpha})Q_{\alpha-1}/Q_{\alpha-1}$ is abelian and since $m_p(S/Q_{\beta})=1$, $(V_{\beta}\cap Q_{\alpha})Q_{\alpha-1}/Q_{\alpha-1}$ is cyclic. By coprime action, $V_{\beta}/Z_{\beta}=[V_{\beta}/Z_{\beta}, O^2(L_{\beta})]\times C_{V_{\beta}/Z_{\beta}}(O^2(L_{\beta}))$. Since $[V_{\beta}/Z_{\beta}, O^2(L_{\beta})]$ contains a non-central chief factor for $L_{\beta}$ and is normal in $G_{\beta}$, we conclude by \cref{VBStrucI} that $V_{\beta}/Z_{\beta}=[V_{\beta}/Z_{\beta}, O^2(L_{\beta})]$ and $C_{V_{\beta}}(O^2(L_{\beta}))=Z_{\beta}$. Similarly, we deduce that $Z_{\beta}=Z(V_{\beta})=\Phi(V_{\beta})=[V_{\beta}, V_{\beta}]$ and $V_{\beta}$ is an extraspecial $2$-group. Since $m_2(S/Q_{\beta})=1$, we have that $|(V_{\beta}\cap Q_{\alpha})Q_{\alpha-1}/Q_{\alpha-1}|=2$ and $V_{\beta}\cap Q_{\alpha}\cap V_{\alpha-1}$ is an elementary abelian subgroup of index $4$ in $V_{\beta}$. Thus, $|V_{\beta}|\leq 2^5$. Since $R_{\beta}=Q_{\beta}$ by \cref{VBStrucII}, $\bar{L_{\beta}}$ acts faithfully on $V_{\beta}$ and has generalized quaternion Sylow $2$-subgroups. Comparing with \cite{Winter}, we have a contradiction.
\end{proof}

\begin{proposition}\label{b2ii}
Suppose that $p\in\{2,3\}$, $b=2$ and $m_p(S/Q_{\beta})>1$. Then one of the following holds:
\begin{enumerate}
\item $R_{\alpha}=Q_{\alpha}$, $\mathcal{A}$ is a weak BN-pair of rank $2$, and either $\mathcal{A}$ is locally isomorphic to $H$ where $(F^*(H),p)$ is $(\mathrm{G}_2(2^n), 2)$ or $({}^3\mathrm{D}_4(p^a), p)$, or $p=2$ and $\mathcal{A}$ is parabolic isomorphic to $\mathrm{J}_2$ or $\Aut(\mathrm{J}_2)$; or 
\item $p=2$, $|S|=2^9$, $\bar{L_{\beta}}\cong\Alt(5)$, $Q_{\beta}\cong 2^{1+6}_+$, $V_{\beta}=O^2(L_{\beta})$, $V_{\beta}/Z_{\beta}$ is a natural $\Omega_4^-(2)$-module for $\bar{L_{\beta}}$, $\bar{L_{\alpha}}\cong \SU_3(2)'$, $Q_{\alpha}$ is a special $2$-group of shape $2^{2+6}$ and $Q_{\alpha}/Z_{\alpha}$ is a natural $\SU_3(2)$-module.
\end{enumerate}
\end{proposition}
\begin{proof}
Suppose first that $R_{\alpha}=Q_{\alpha}$. Then, as in \cref{5Symp}, $\mathcal{A}$ is a symplectic amalgam and the result follows from \cref{Symp}. Indeed, the amalgams presented in \cite{parkerSymp} satisfying the above hypothesis are either weak BN-pairs of rank $2$ (and (i) holds by \cite{Greenbook}); or $\mathcal{A}_{42}$ when $p=2$. In the latter case, $\PSp_6(3)$ is listed as an example completion. But comparing with the list of maximal subgroups in \cite{atlas}, for $G\cong \PSp_6(3)$, $\bar{L_{\alpha}}\cong 2^{2+6}: \SU_3(2)'$ and from the perspective of this work, $R_{\alpha}\ne Q_{\alpha}$. Either way, we assume throughout this proof that $R_{\alpha}\ne Q_{\alpha}$ with the goal of showing that $G$ has ``the same'' structural properties as $\mathcal{A}_{42}$ in \cite{parkerSymp} in order to satisfy outcome (ii). Thus, we may apply the results in \cref{b=2lemma} and \cref{VBStrucII}. We set $r_\beta:=|(V_{\beta}\cap Q_{\alpha})Q_{\alpha-1}/Q_{\alpha-1}|$.

We aim to show that $Z(V_{\beta})\le Q_{\alpha-1}$ so that $[V_{\beta}, V_{\beta}]=\Phi(V_{\beta})=Z(V_{\beta})=Z_{\beta}$ is of order $q_\alpha$ and $V_{\beta}/Y$ is an extraspecial group, for $Y$ any maximal subgroup of $Z_{\beta}$. In the language of Beisiegel \cite{beisiegel}, $V_{\beta}$ is a semi-extraspecial group. Towards this goal, we suppose that $Z(V_{\beta})\not\le Q_{\alpha-1}$. Then the action of $L$ implies that $Z(V_{\alpha-1})\not\le Q_{\beta}$. Set $V:=V_{\beta}/Z(V_{\beta})$ throughout, and let $Z\le Z(V_{\alpha-1})$ such that $Z_{\alpha-1}\le Z$, $|Z/Z_{\alpha-1}|=p$ and $Z\not\le Q_{\beta}$. Since $L_{\alpha-1}$ centralizes $Z(V_{\alpha-1})/Z_{\alpha-1}$ we have that $Z\normaleq L_{\alpha-1}$. Then $Z$ centralizes an index $p$ subgroup of $V_{\beta}\cap Q_{\alpha}$, and so centralizes an index $pq_\alpha$ subgroup of $V_{\beta}$. 

Assume that $O_{p'}(\bar{L_{\beta}})\ne\{1\}$. Then by the irreducibility of $V$ and since $R_{\beta}=Q_{\beta}$, we have that $V=[V, O_{p'}(\bar{L_{\beta}})]$ and by \cref{VBStrucI}, we have that $V_{\beta}/Z_{\beta}=[V_{\beta}/Z_{\beta}, O_{p'}(\bar{L_{\beta}})]$. Since $Z(V_{\beta})/Z_{\beta}\le C_{V_{\beta}/Z_{\beta}}(O^p(L_{\beta}))\le C_{V_{\beta}/Z_{\beta}}(O_{p'}(\bar{L_{\beta}}))$, we infer by coprime action that $Z(V_{\beta})=Z_{\beta}$, as desired. Hence, we restrict our analysis to the case where $O_{p'}(\bar{L_{\beta}})=\{1\}$.

Suppose that $\bar{L_{\beta}}\cong \mathrm{M}_{11}$ or $\PSL_3(4)$ and $p=3$. Then $q_\alpha=3$ and $Z$ centralizes an index $9$ subgroup of $V_{\beta}$. If $\bar{L_{\beta}}\cong \mathrm{M}_{11}$ then there is $x\in L_{\beta}$ such that for $J:=\langle Z, Z^x, Q_{\beta}\rangle$, $\bar{J}\cong \PSL_2(11)$ and $J$ centralizes a subgroup of $V$ of index at most $3^4$. Since $11$ does not divide $|\GL_4(3)|$, $J$ centralizes $V$, a contradiction since $\bar{J}$ contains a non-trivial $3$-element. If $\bar{L_{\beta}}\cong\PSL_3(4)$, then there is $x\in L_{\beta}$ such that $L_{\beta}=\langle Z, Z^x, Q_{\beta}\rangle$ so that $|V|\leq 3^4$, a contradiction by \cref{SpeModOdd}.

If $p=3$ and $\bar{L_{\beta}}\cong \Ree(q_{\beta})$ then $Z$ centralizes an index $9$ subgroup of $V_{\beta}$ and \cref{SpeModOdd} gives a contradiction. If $p=2$ and $\bar{L_{\beta}}\cong \Sz(q_\beta)$ for $n\geq 3$ then $Z$ centralizes an index $2q_\alpha$ subgroup of $V_{\beta}$ and \cref{SpeMod2} gives a contradiction.

If $\bar{L_{\beta}}\cong \PSU_3(q_\beta)$ then $Z$ centralizes an index $pq_\alpha$ subgroup of $V_{\beta}$. Applying \cref{SpeMod2} and \cref{SpeModOdd}, we deduce that $\bar{L_{\beta}}\cong\SU_3(3)$. Then $\bar{L_{\beta}}$ is generated by only three conjugates of $\bar{Z}$, $|V|\leq p^6$ and $V$ is a natural $\SU_3(3)$-module for $\bar{L_\beta}$. But then, comparing with \cref{SUMod} and using that $V_{\beta}\cap Q_{\alpha}$ is a $G_{\alpha, \beta}$-invariant subgroup of index $3$, we have a contradiction.

Finally, if $\bar{L_{\beta}}\cong \PSL_2(q_\beta)$ then $Z$ centralizes an index $pq_\alpha$ subgroup of $V_{\beta}$. Then, since $q_\alpha<q_\beta^2$, \cref{SpeMod2} and \cref{SpeModOdd} yield that $q_\alpha=p$ and $|V|\leq p^6$. Since $V_{\beta}\cap Q_{\alpha}$ is a $G_{\alpha,\beta}$-subgroup of index $p$ in $V_{\beta}$, applying \cref{SL2ModRecog} we deduce that $V$ is a natural $\Omega_4^-(p)$-module for $\bar{L_{\beta}}\cong \PSL_2(p^2)$. Moreover, we have that $|Z(V_{\beta})Q_{\alpha-1}/Q_{\alpha-1}|=p$, $|Z(V_{\beta})/Z_{\beta}|=p$ and $|V_{\beta}/Z_{\beta}|=p^5$. Since $V_{\beta}=[V_{\beta}, L_{\beta}]$, applying \cref{A6Cohom}, we deduce that $p=3$ and $[V_{\beta}/Z_{\beta}, S, S]$ is $2$-dimensional as a $\mathrm{GF}(3)$-module. But then $Z_{\alpha}Z(V_{\beta})=[V_{\beta}, S, S]=[V_{\beta}, V_{\alpha-1}\cap Q_{\alpha}, V_{\alpha-1}\cap Q_{\alpha}]\le V_{\alpha-1}$, a contradiction since $Z(V_{\beta})\not\le Q_{\alpha-1}$.

Thus, $Z(V_{\beta})\le Q_{\alpha-1}$ and by a previous observation, $Z(V_{\beta})=Z_{\beta}=\Phi(V_{\beta})$ is of order $q_\alpha$ and $V_{\beta}$ is a semi-extraspecial group. Moreover, $V_{\beta}\cap Q_{\alpha}\cap Q_{\alpha-1}$ has index $q_\alpha r_{\beta}$ in $V_{\beta}$ and is elementary abelian. Then for $Z$ a maximal subgroup of $Z_{\beta}$, we set $|V_{\beta}/Z|=p^{2r+1}$ so that $|V_{\beta}/Z_{\beta}|=p^{2r}$. Then $|(V_{\beta}\cap Q_{\alpha}\cap Q_{\beta})/Z|=p^{2r+1}/q_{\alpha}r_{\beta}$ and since the maximal abelian subgroups of $V_{\beta}/Z$ have order $p^{r+1}$, we deduce that $p^{2r+1}/q_\alpha r_{\beta}\leq p^{r+1}$ and $p^r\leq q_\alpha r_{\beta}$.

Suppose that $\bar{L_{\beta}}/O_{p'}(\bar{L_{\beta}})\cong\mathrm{(P)SU}_3(q_\beta)$ so that $r_{\beta}\leq q_{\beta}$ when $p=2$ and $r_{\beta}\leq q_{\beta}^2$ when $p=3$. In particular, $2^r\leq q_{\alpha}q_{\beta}< q_{\beta}^3$ when $p=2$ and $3^r\leq 3q_{\beta}^2$ when $p=3$. Thus, $|V_{\beta}/Z_{\beta}|<q_{\beta}^6$ when $p=2$ and $|V_{\beta}/Z_{\beta}|\leq 3^2q_{\beta}^4$ when $p=3$. Then by \cref{SpeMod2} and \cref{SpeModOdd}, we conclude that $q_{\beta}=3=p$, $\bar{L_{\beta}}/O_{3'}(\bar{L_{\beta}})\cong\mathrm{(P)SU}_3(3)$ and $|V_{\beta}/Z_{\beta}|=3^6$. A calculation in $\Sp_6(3)$ promises that $\bar{L_{\beta}}\cong \PSU_3(3)$ and $V_{\beta}/Z_{\beta}$ is a natural module for $\bar{L_{\beta}}$. But then $V_{\beta}\cap Q_{\alpha}$ is a $G_{\alpha,\beta}$-invariant subgroup of index $3$, and we have a contradiction by \cref{SUMod} (iii).

Suppose that $\bar{L_{\beta}}/O_{3'}(\bar{L_{\beta}})\cong\Ree(q_\beta)$. Then $3^r\leq r_{\beta}3\leq 3q_{\beta}^2$ and so $|V_{\beta}/Z_{\beta}|\leq 3^2q_{\beta}^4< \mathrm{max}(q_{\beta}^6, 3^7)$. Then \cref{SpeModOdd} provides a contradiction. If $\bar{L_{\beta}}O_{2'}(\bar{L_{\beta}})\cong\Sz(q_\beta)$, then $r_\beta\le q_\beta$ and so $2^r\leq q_\alpha q_\beta$ and $|V_{\beta}/Z_{\beta}|\leq q_{\alpha}^2q_{\beta}^2<q_{\beta}^4$. Then \cref{SpeMod2} provides a contradiction.

Hence, we may suppose that $S/Q_{\beta}$ is elementary abelian of order $p^n$ and $n>1$. Then $|V_{\beta}/Z_{\beta}|\leq q_{\alpha}^2q_{\beta}^2\le q_{\beta}^3$. In particular, it follows from \cref{SpeMod2} and \cref{SpeModOdd} that $V_{\beta}/Z_{\beta}$ contains a unique non-central chief factor for $L_{\beta}$ and so $V_{\beta}/Z_{\beta}$ is an irreducible $\bar{L_{\beta}}$-module. If $q_\beta>p^2$ then \cref{q^3module} yields that $\bar{L_{\beta}}\cong \SL_2(q_\beta)$ or $\PSL_2(q_{\beta})$. Furthermore, since $q_\alpha^2<q_\beta$, we infer that $V_{\beta}/Z_{\beta}$ is a triality module and $q_\alpha^3=q_\beta$. Then for $K_\alpha$ a critical subgroup of $Q_{\alpha}$, we have that $[V_{\beta}, K_\alpha, K_\alpha, K_\alpha]=\{1\}$. By \cref{trialitydescription}, this yields $K_\alpha\le Q_{\beta}$ so that $K_\alpha=Z_{\alpha}(K_\alpha\cap Q_{\alpha'})$, $[O^p(L_{\alpha}), K_\alpha]=Z_{\alpha}$ and $R_{\alpha}=Q_{\alpha}$, a contradiction.

Hence, we may suppose that $S/Q_{\beta}$ is elementary abelian of order $p^2$ so that $q_{\alpha}=p$, $V_{\beta}$ is an extraspecial group and $|V_{\beta}/Z_{\beta}|\in\{p^4, p^6\}$. In particular, applying \cite{Winter}, if $p=3$ then $\bar{L_{\beta}}$ is isomorphic to a subgroup of $\Sp_6(3)$ and if $p=2$, then $\bar{L_{\beta}}$ is isomorphic to a subgroup of $\PSL_4(2)$ or $\PSU_4(2)$. We deduce in both cases that $\bar{L_{\beta}}\cong\SL_2(p^2)$ or $\PSL_2(p^2)$ and $V_{\beta}/Z_{\beta}$ is described by \cref{SL2ModRecog}. Since $V_{\beta}\cap Q_{\alpha}$ is a $G_{\alpha\beta}$-invariant subgroup of index $p$ containing $[V_{\beta}, S]$ \cref{SL2ModRecog} implies that $V_{\beta}/Z_{\beta}$ is a natural $\Omega_4^-(p)$-module and $|V_{\beta}|=p^5$. Now, as $L_{\beta}/C_{\beta}$ embeds in the automorphism group of $V_{\beta}$, we infer that $Q_{\beta}=V_{\beta}C_{\beta}$. Moreover, using \cite{Winter}, if $p=2$ then $\bar{L_{\beta}}\cong \Out(V_{\beta})\cong\Omega_4^-(2)$ and $V_{\beta}\cong Q_8\circ D_8\cong 2^{1+4}_-$; and if $p=3$ then $V_{\beta}$ has exponent $3$.

Suppose that $p=3$ and let $K\in \syl_2(L_{\beta})$. Since $\bar{L_{\beta}}\cong\PSL_2(9)$, $K\cong \Dih(8)$. Letting $1\ne i\in Z(K)$, we have that $|C_{V_{\beta}/Z_{\beta}}(i)|=9$ and by coprime action $V_{\beta}=C_{V_{\beta}}(i)[V_{\beta}, i]$. Since $[V_{\beta}, V_{\beta}]\le C_{V_{\beta}}(i)$ it follows from the three subgroup lemma that $[[V_{\beta}, i], C_{V_{\beta}}(i)]=\{1\}$ and since $|[V_{\beta}, i]|\leq 3^3$, it follows that $Z_{\beta}=C_{V_{\beta}}(i)\cap [V_{\beta}, i]$ and $C_{V_\beta}(i)\cong [V_{\beta}, i]\cong 3^{1+2}_+$. Since $i\le Z(K)$, $K$ normalizes $[V_{\beta}, i]$ and since $Z_{\beta}=Z(L_{\beta})$, $K$ acts trivially on $Z_{\beta}=Z([V_{\beta}, i])$ and by \cite{Winter}, $K$ embeds into $\Sp_2(3)\cong \SL_2(3)$. But $\SL_2(3)$ has quaternion Sylow $2$-subgroups, a contradiction.

Thus, we have shown that $p=2$. Now, $Z_{\alpha}\not\le C_{\beta}$ and so $Z_{\beta}=C_{\beta}\cap Q_{\alpha-1}$ has index at most $4$ in $C_{\beta}$ and $|C_{\beta}|\leq 8$. Since $Z(C_{\beta})$ is centralized by $L_{\beta}=O^2(L_{\beta})C_{\beta}$ and $Q_{\alpha}$ is self centralizing, $Z(C_{\beta})\le Z(Q_{\alpha})=Z_{\alpha}$. Thus, $Z(C_{\beta})=Z_{\beta}$ and as $|C_{\beta}|\leq 8$, either $C_{\beta}=Z_{\beta}$, or $C_{\beta}\cong Q_8$ or $\Dih(8)$. If $C_{\beta}=Z_{\beta}$ then we have that $Q_{\beta}=V_{\beta}\cong 2^{1+4}_-$, $|S|=2^7$ and $|Q_{\alpha}|=2^6$. Since $Z_{\alpha}\le \Phi(Q_{\alpha})$ and $R_{\alpha}\ne Q_{\alpha}$, we have that $Z_{\alpha}=\Phi(Q_{\alpha})$ and $Q_{\alpha}/Z_{\alpha}$ is a faithful quadratic $2$F-module for $\bar{L_{\alpha}}$. As $L_{\alpha}/R_{\alpha}\cong \Sym(3)$, it follows that $\bar{L_{\alpha}}\cong (3\times 3):2$. Now, for every subgroup $Z$ of $Z_{\alpha}$ of order $2$, is easy to check that $Q_{\alpha}/Z$ is an extraspecial group. In the language of Beisiegel \cite{beisiegel}, $Q_{\alpha}$ is an ultraspecial $2$-group of order $2^6$. Checking in MAGMA utilizing the Small Groups library, the automorphism groups of all such groups have $3$-part at most $9$. Since there is $r\in (L_{\beta}\cap G_{\alpha,\beta})$ a $3$-element centralizing $Z_{\alpha}$ by \cref{Omega4} (v), $r\in G_{\alpha}\setminus L_{\alpha}$ and a Sylow $3$-subgroup of $\bar{G_{\alpha}}$ has order at least $27$, and as $\bar{G_{\alpha}}$ acts faithfully on $Q_{\alpha}$, we have a contradiction.

Thus, $C_{\beta}$ is non-abelian of order $8$. Furthermore, $|S|=2^9$ and if $Q_{\alpha}/Z_{\alpha}$ is a natural $\SU_3(2)$-module for $\bar{L_{\alpha}}\cong \SU_3(2)'$, then since $C_{\beta}$ is $G_{\alpha,\beta}$-invariant, there is a $3$-element in $L_{\alpha}\cap G_{\alpha,\beta}$ which acts non-trivially on $C_{\beta}$ so that $C_{\beta}\cong Q_8$ and $Q_{\beta}=2^{1+6}_+$. Thus, to complete the proof, it suffices to show that $Q_{\alpha}/Z_{\alpha}$ is a natural $\SU_3(2)$-module. Now, $Q_{\alpha}\cap Q_{\beta}=Z_{\alpha}(Q_{\alpha}\cap Q_{\alpha'})$ has index $4$ in $Q_{\alpha}$ and, modulo $Z_{\alpha}$, is centralized by $Z_{\alpha'}$. It is clear that $Z_{\alpha'}$ acts quadratically on $Q_{\alpha}/Z_{\alpha}$ and, since $Z_{\alpha}\le \Phi(Q_{\alpha})$ and $R_{\alpha}\ne Q_{\alpha}$, $\bar{L_{\alpha}}$ is determined by \cref{Quad2F}. Since $L_{\alpha}/R_{\alpha}\cong \Sym(3)$, we need only rule out the case where $\bar{L_{\alpha}}\cong (3\times 3):2$.

Assume that $\bar{L_{\alpha}}\cong (3\times 3):2$ and $|C_{\beta}|=8$. Observe that $Q_{\alpha}=(Q_{\alpha}\cap Q_{\beta})(Q_{\alpha}\cap Q_{\alpha-1})=(V_{\beta}\cap Q_{\alpha})(V_{\alpha-1}\cap Q_{\beta})(Q_{\beta}\cap Q_{\alpha}\cap Q_{\alpha-1})$. Then, $V_{\beta}\cap Q_{\alpha}\cap Q_{\alpha-1}=V_{\alpha-1}\cap Q_{\alpha}\cap Q_{\beta}=Z_{\alpha}$, and it follows that $Z_{\alpha}=\Phi(Q_{\alpha})$. By coprime action, we have that $Q_{\alpha}/Z_{\alpha}=[Q_{\alpha}/Z_{\alpha}, O^2(L_{\alpha})]\times C_{Q_{\alpha}/Z_{\alpha}}(O^2(L_{\alpha}))$ where $|[Q_{\alpha}/Z_{\alpha}, O^2(L_{\alpha})]|=2^4$. Taking $Q_{\alpha}^*$ to be the preimage in $Q_{\alpha}$ of $[Q_{\alpha}/Z_{\alpha}, O^2(L_{\alpha})]$, form $S^*=V_{\beta}Q_{\alpha}^*$ and $L_{\lambda}^*=\langle (S^*)^{L_{\lambda}}$ for $\lambda\in\{\alpha,\beta\}$. It is clear that $S^*\in\syl_2(L_{\lambda}^*)$, $V_{\beta}=O_2(L_{\beta}^*)$ and $Q_{\alpha}^*=O_2(L_{\alpha}^*)$, and $L_{\lambda}^*/O_2(L_{\lambda}^*)\cong \bar{L_{\lambda}}$ for $\lambda\in\{\alpha,\beta\}$. Then for $K$ a Hall $2'$-subgroup of $G_{\alpha,\beta}$, we conclude that  $(L_{\alpha}^*K, L_{\beta}^*K, S^*K)$ satisfies \cref{MainHyp} and since $G$ is a minimal counterexample, comparing with \hyperlink{MainGrpThm}{Theorem C}, we have a contradiction.
\end{proof}

In summary, in this section we have proved the following:

\begin{theorem}\label{Mainbeven}
Suppose that $\mathcal{A}=\mathcal{A}(G_\alpha, G_\beta, G_{\alpha,\beta})$ is an amalgam satisfying \cref{MainHyp}. If $Z_{\alpha'}\not\le Q_{\alpha}$, then one of the following holds:
\begin{enumerate}
\item $\mathcal{A}$ is a weak BN-pair of rank $2$;
\item $\mathcal{A}$ is a symplectic amalgam; or
\item $p=2$, $|S|=2^9$, $\bar{L_{\beta}}\cong\PSL_2(4)$, $Q_{\beta}\cong 2^{1+6}_-$, $V_{\beta}=O^2(L_{\beta})$, $V_{\beta}/Z_{\beta}$ is a natural $\Omega_4^-(2)$-module for $\bar{L_{\beta}}$, $\bar{L_{\alpha}}\cong \SU_3(2)'$, $Q_{\alpha}$ is a special $2$-group of shape $2^{2+6}$ and $Q_{\alpha}/Z_{\alpha}$ is a natural $\SU_3(2)$-module.
\end{enumerate}
\end{theorem}

%% file: Contents/7.1.Non-symmetricCasei.tex
\section{$Z_{\alpha'}\le Q_\alpha$}\label{oddsec}

We now begin the second half of our analysis, where $Z_{\alpha'}\le Q_{\alpha}$ so that $[Z_{\alpha}, Z_{\alpha'}]=\{1\}$. Throughout, we set $S\in\syl_p(G_{\alpha,\beta})$ and $q_{\lambda}=|\Omega(Z(S/Q_{\lambda}))|$ for any $\lambda\in\Gamma$.

\begin{lemma}\label{bodd}
The following hold:
\begin{enumerate}
\item $Z_\beta=\Omega(Z(S))=\Omega(Z(L_\beta))$ and $b$ is odd; and
\item $Z(L_\alpha)=\{1\}$.
\end{enumerate}
\end{lemma}
\begin{proof}
Since $Z_{\alpha'}\le Q_{\alpha}$ we have that $\{1\}=[Z_{\alpha}, Z_{\alpha'}]$. Then, for $T\in\syl_p(G_{\alpha', \alpha'-1})$, as $Z_{\alpha}\not\le Q_{\alpha'}$, $Q_{\alpha'}<C_T(Z_{\alpha'})$ and by \cref{p-closure2} (ii), we get that $Z_{\alpha'}=\Omega(Z(T))=\Omega(Z(L_{\alpha'})$. By \cref{crit pair} (iii), $Z_{\alpha}\not\le \Omega(Z(L_{\alpha}))$ and so $\alpha$ and $\alpha'$ are not conjugate. Thus, $\alpha'$ is conjugate to $\beta$, $b$ is odd and $Z_{\beta}=\Omega(Z(S))=\Omega(Z(L_\beta)))$. Since $L_\beta$ acts transitively on $\Delta(\beta)$, by \cref{crit pair} (iv), we conclude that $Z(L_\alpha)=\{1\}$.
\end{proof}

\begin{lemma}\label{b>1}
Suppose that $b>1$. Then $V_\beta$ is abelian, $\{1\}\ne [V_\beta, V_{\alpha'}]\le V_{\alpha'}\cap V_\beta$ and $V_{\beta}$ acts quadratically on $V_{\alpha'}$.
\end{lemma}
\begin{proof}
Since $Z_{\alpha}\le V_\beta$ and $Z_\alpha\not\le Q_{\alpha'}$ it follows that that $V_\beta\not\le C_{L_{\alpha'}}(V_{\alpha'})$. By minimality of $b$, $V_\beta\le Q_{\alpha'-1}\le L_{\alpha'}$ and so $\{1\}\ne[V_\beta, V_{\alpha'}]\le V_{\alpha'}$. Again, by minimality of $b$, $V_{\alpha'}\le Q_{\alpha+2}\le L_\beta$ and so $[V_\beta, V_{\alpha'}]\le V_{\alpha'}\cap V_\beta$. Since $V_{\beta}$ is abelian, $[V_{\alpha'}, V_{\beta}, V_{\beta}]=\{1\}$, completing the proof.
\end{proof}

\begin{lemma}
Suppose that $b>1$ and $p$ is an odd prime. Then either
\begin{enumerate}
\item $L_{\beta}/R_{\beta}Q_{\beta}\cong \SL_2(p^n)$ where $p$ is any odd prime;
\item $L_{\beta}/R_{\beta}Q_{\beta}\cong\mathrm{(P)SU}_3(p^n)$ where $p$ is any odd prime; or
\item $L_{\beta}/C_{L_{\beta}}(U/V)\cong \SL_2(3)$, $2\cdot \Alt(5)$ or $2^{1+4}_-.\Alt(5)$ for every non-trivial irreducible composition factor $U/W$ of $V_\beta/C_{V_{\beta}}(O^p(L_{\beta}))$, $p=3$ and $|S/Q_{\beta}|=3$.
\end{enumerate}
\end{lemma}
\begin{proof}
Since $[V_{\alpha'}, V_{\beta}, V_{\beta}]=\{1\}$, this follows immediately applying \cref{SEQuad} since $\beta$ is conjugate to $\alpha'$.
\end{proof}

The following lemma is part of the proof of the qrc lemma, which we recreate for the purposes of this work.

\begin{lemma}\label{SL2implies}
Suppose that $b>1$ and $V_{\beta}/C_{V_{\beta}}(O^p(L_{\beta}))$ is an irreducible module for $O^p(L_{\beta})R_{\beta}/R_{\beta}$. Then $Z_{\alpha}$ is a natural module for $L_{\alpha}/R_{\alpha}\cong \SL_2(q_\alpha)$.
\end{lemma}
\begin{proof}
Let $Q:=Q_{\beta}\cap O^p(L_{\beta})$ so that, by \cref{push}, $Q\not\le Q_{\alpha}$. In particular, $Q$ acts non-trivially on $Z_{\alpha}$ and centralizes $C_{V_{\beta}}(O^p(L_{\beta}))$. Since $V_{\beta}/C_{V_{\beta}}(O^p(L_{\beta}))$ is irreducible, $C_{Z_{\alpha}}(Q)=Z_{\alpha}\cap C_{V_{\beta}}(O^p(L_{\beta}))$. Moreover, $[Z_{\alpha}, Q, Q]=\{1\}$.

Let $z\in Z_{\alpha}\setminus (Z_{\alpha}\cap C_{V_{\beta}}(O^p(L_{\beta})))$ and let $W_\beta:=\langle z^{O^p(L_{\beta})}\rangle$. Since $V_{\beta}/C_{V_{\beta}}(O^p(L_{\beta}))$ is irreducible, we have that $V_{\beta}=W_{\beta}C_{V_{\beta}}(O^p(L_{\beta}))$ so that $[Z_{\alpha}, Q]\le [V_{\beta}, Q]=[W_\beta, Q]$. Moreover, since $[z, Q]\le C_{V_{\beta}}(O^p(L_{\beta}))$, we have that $[z, Q]=[z, Q]^{O^p(L_{\beta})}=[W_\beta, Q]$ from which we conclude that $[Z_{\alpha}, Q]\le [W_{\beta}, Q]=[z, Q]\le [Z_{\alpha}, Q]$.

Now, there is a surjective homomorphism $\theta: Q\to [z, Q]$ and with kernel $C_Q(z)\geq C_Q(Z_{\alpha})$. In particular, $|[Z_{\alpha}, Q]|=|[z, Q]|=|Q/C_Q(z)|\leq |Q/C_Q(Z_\alpha)|$. Indeed, $Z_{\alpha}$ is dual to an FF-module for $L_{\alpha}/R_{\alpha}$. Hence, \cref{SEFF} yields that $Z_{\alpha}$ is a natural module for $L_{\alpha}/R_{\alpha}\cong \SL_2(q_\alpha)$, as desired.
\end{proof}

As in \cref{ZBEQZS}, throughout this section, we intend to control the action of $O^p(R_{\alpha})$ and $O^p(R_{\beta})$ using the methods in \cref{VBGood}-\cref{GoodAction4} in the expectation of applying \cref{SimExt} or \cite{Greenbook} to force contradictions. In the following lemmas, we demonstrate that we satisfy \cref{CommonHyp}, required for the application of these lemmas. Also, as in \cref{ZBEQZS}, whenever $L_{\alpha}/R_{\alpha}\cong L_{\beta}/R_{\beta}\cong\SL_2(p)$, we will often make a generic appeal to coprime action, utilizing that $L_{\lambda}$ is solvable when $p=2$ for $\lambda\in\{\alpha,\beta\}$, and that there is a central involution $t_{\lambda}\in L_{\lambda}/R_{\lambda}$ which acts fixed point freely on natural modules when $p$ is odd.

\begin{lemma}\label{ZQ=Z}
Suppose that $C_{V_\beta}(V_{\alpha'})=V_\beta \cap Q_{\alpha'}$, $V_{\alpha'}\le Q_{\beta}$, $q_\alpha=q_\beta=p$ and $Z_{\alpha}$ is a natural module for $L_{\alpha}/R_{\alpha}\cong \SL_2(p)$. Then $Z_{\alpha}=Z(Q_{\alpha})$ and $Z_{\beta}=Z(Q_{\beta})$.
\end{lemma}
\begin{proof}
Suppose that $V_{\alpha'}\le Q_{\beta}$. We aim to show that if the conclusion of the lemma fails to hold then $R=Z_{\beta}=Z_{\alpha'}$ for then, as $V_{\beta}\not\le Q_{\alpha'}$, $O^p(L_{\alpha'})$ centralizes $V_{\alpha'}$, a contradiction.

Suppose that $V_{\alpha'}\le Q_{\beta}$ and $Z_{\alpha}\ne Z(Q_{\alpha})$. By minimality of $b$, and using that $b$ is odd, we have that $Z_{\lambda}\le Q_{\alpha}$ and $Z(Q_{\alpha})\le Q_{\lambda}$ for all $\lambda\in\Delta^{(b-1)}(\alpha)$. In particular, $Z(Q_{\alpha})\le Q_{\alpha'-1}$ and $Z(Q_{\alpha})=Z_{\alpha}(Z(Q_{\alpha})\cap Q_{\alpha'})$. If $[V_{\alpha'}, Z(Q_{\alpha})\cap Q_{\alpha'}]=\{1\}$, it follows that $O^p(L_{\alpha})$ centralizes $Z(Q_{\alpha})/Z_{\alpha}$ and an application of coprime action, observing that $Z_{\beta}\le Z_{\alpha}=[Z(Q_{\alpha}), O^p(L_{\alpha})]$, gives a contradiction. If $[V_{\alpha'}, Z(Q_{\alpha})\cap Q_{\alpha'}]\ne\{1\}$, then $Z_{\alpha'}=[V_{\alpha'}, Z(Q_{\alpha})\cap Q_{\alpha'}]\le Z(Q_{\alpha})$ and so $Z_{\alpha'}$ is centralized by $V_{\alpha'}Q_{\alpha}\in\syl_p(L_{\alpha})$ from which it follows that $Z_{\alpha'}=Z_{\beta}$, a contradiction. Thus, $Z_{\alpha}=Z(Q_{\alpha})$. Since $Z(S)\le Z(Q_{\alpha})$ we conclude that $Z(S)=\Omega(Z(S))=Z_{\beta}$ is of exponent $p$.

Since $Z_{\lambda}\le Q_{\alpha'}$ for all $\lambda\in\Delta^{(b-2)}(\alpha')$, again using the minimality of $b$ and that $b$ is odd, we argue that $Z(Q_{\alpha'})\le Q_{\alpha+2}$. If $Z(Q_{\alpha'})\not\le Q_{\beta}$ then, as $Z(S)=Z_{\beta}$, $\{1\}\ne [Z(Q_{\alpha'}), Z(Q_{\beta})]\le Z(Q_{\alpha'})\cap Z(Q_{\beta})$, for otherwise $Z(Q_{\beta})$ is centralized by $Z(Q_{\alpha'})Q_{\beta}\in\syl_p(L_{\beta})$ and the result holds. Then, $[Z(Q_{\alpha'}), Z(Q_{\beta})]$ is centralized by $Z(Q_{\alpha'})Q_{\beta}\in\syl_p(L_{\beta})$ and since $Z(S)=Z_{\beta}$, $[Z(Q_{\alpha'}), Z(Q_{\beta})]=Z_{\beta}$. Moreover, since $[Z(Q_{\alpha'}), Z(Q_{\beta})]\ne\{1\}$, $Z(Q_{\beta})\not\le Q_{\alpha'}$, and by a similar reasoning, $[Z(Q_{\alpha'}), Z(Q_{\beta})]=Z_{\alpha'}$. But then $Z_{\beta}=Z_{\alpha'}$, a contradiction. Hence, $Z(Q_{\alpha'})\le Q_{\beta}$.

Observe that $Z(Q_{\alpha'})\not\le Q_{\alpha}$, else $Z(Q_{\alpha'})$ is centralized by $Z_{\alpha}Q_{\alpha'}\in\syl_p(L_{\alpha'})$ and $Z(Q_{\alpha'})=Z_{\alpha'}$, as desired. Then $Z_{\beta}=[Z(Q_{\alpha'}), Z_{\alpha}]\le \Omega(Z(Q_{\alpha'}))$ so that $Z_{\beta}$ is centralized by $Z_{\alpha}Q_{\alpha'}\in\syl_p(L_{\alpha'})$ and $Z_{\beta}=Z_{\alpha'}$, again a contradiction. Therefore, if $V_{\alpha'}\le Q_{\beta}$, we have shown that $Z(Q_{\beta})=Z_{\beta}$.
\end{proof}

\begin{lemma}
Suppose that $C_{V_\beta}(V_{\alpha'})=V_\beta \cap Q_{\alpha'}$, $V_{\alpha'}\not\le Q_{\beta}$, $q_\alpha=q_\beta=p$ and $Z_{\alpha}$ is a natural module for $L_{\alpha}/R_{\alpha}\cong \SL_2(p)$. Then $Z_{\alpha}=Z(Q_{\alpha})$ and $Z_{\beta}=Z(Q_{\beta})$.
\end{lemma}
\begin{proof}
Suppose that $V_{\alpha'}\not\le Q_{\beta}$ and $Z(Q_{\alpha'})\le Q_{\beta}$. In addition, assume first that $Z(Q_{\alpha'})\le Q_{\alpha}$ so that $Z(Q_{\alpha'})$ is centralized by $Z_{\alpha}Q_{\alpha'}\in\syl_p(L_{\alpha'})$. Set $Y^\beta:=\langle Z(Q_{\lambda}) \mid Z_{\lambda}=Z_{\alpha}, \lambda\in\Delta(\beta)\rangle$ and let $r\in R_{\beta}Q_{\alpha}$. Since $r$ is a graph automorphism, for $\lambda\in\Delta(\beta)$ such that $Z_{\lambda}=Z_{\alpha}$, $Z(Q_{\lambda})^r=Z(Q_{\lambda\cdot r})$. But now, $Z_{\lambda\cdot r}=Z_{\lambda}^r=Z_{\alpha}^r=Z_{\alpha}$ and so $Z(Q_{\lambda})^r\le Y^{\beta}$. Thus, $Y^\beta\normaleq R_{\beta}Q_{\alpha}$. Now, observe that by minimality of $b$, and using that $b$ is odd, $V_{\delta}\le Q_{\lambda}$ and $Z(Q_{\lambda})\le Q_{\delta}$ for all $\lambda\in\Delta(\beta)$ with $Z_{\lambda}=Z_{\alpha}$ and $\delta\in\Delta^{(b-1)}(\lambda)$ by \cref{BasicVB}. In particular, $Z(Q_{\alpha})\le Y^\beta\le Q_{\alpha'-1}$. Thus, $Z(Q_{\alpha})=Z_{\alpha}(Z(Q_{\alpha})\cap Q_{\alpha'})$ and $Y^\beta=Z_{\alpha}(Y^\beta\cap Q_{\alpha'})$.  

Since $Z(Q_{\alpha})\cap Q_{\alpha'}$ is a maximal subgroup of $Z(Q_{\alpha})$ not containing $Z_{\alpha}$, we must have that $Z_{\alpha}\not\le \Phi(Z(Q_{\alpha}))$. But then, by the irreducibility of $Z_{\alpha}$ under the action of $G_{\alpha}$, $Z_{\beta}\cap \Phi(Z(Q_{\alpha}))=\Omega(Z(S))\cap \Phi(Z(Q_{\alpha}))=\{1\}$ so that $\Phi(Z(Q_{\alpha}))=\{1\}$ and $Z(Q_{\alpha})=\Omega(Z(Q_{\alpha}))$ is elementary abelian.

Assume first that $[Y^\beta\cap Q_{\alpha'},V_{\alpha'}]=Z_{\alpha'}$ so that $Y^{\beta}\not\le V_{\beta}$ and there is some $\alpha'+1\in\Delta(\alpha')$ with $Y^\beta\cap Q_{\alpha'}\not\le Q_{\alpha'+1}$. Again, using the minimality of $b$ and that $b$ is odd, we deduce that $Z(Q_{\alpha'+1})\le Q_{\alpha+2}$. Write $Y_{\beta}=\langle Z(Q_{\alpha})^{G_{\beta}}\rangle$ so that $Y^\beta\le Y_\beta\normaleq G_{\beta}$ and, as $b>2$, $Y_\beta$ is abelian. Then $Z(Q_{\alpha'+1})$ normalizes $Y_\beta$, $[Z(Q_{\alpha'+1}), Y^\beta\cap Q_{\alpha'}, Y^\beta\cap Q_{\alpha'}]\le [Z(Q_{\alpha'+1}), Y_\beta, Y_\beta]=\{1\}$ and $Z(Q_{\alpha'+1})$ is quadratic module for $\bar{L_{\alpha'+1}}$. Moreover, by coprime action, $Z(Q_{\alpha'+1})=[Z(Q_{\alpha'+1}), R_{\alpha'+1}]\times C_{Z(Q_{\alpha'+1})}(R_{\alpha'+1})$ is invariant under $T\in\syl_p(G_{\alpha',\alpha'+1})$ and as $Z_{\alpha'}\le Z_{\alpha'+1}\le C_{Z(Q_{\alpha'+1})}(R_{\alpha'+1})$, we infer that $Z(Q_{\alpha'+1})=C_{Z(Q_{\alpha'+1})}(R_{\alpha'+1})$ and $Z(Q_{\alpha'+1})$ is a faithful module for $L_{\alpha'+1}/R_{\alpha'+1}\cong \SL_2(p)$. But then by \cref{DirectSum}, $Z(Q_{\alpha'+1})$ is a direct sum of natural $\SL_2(p)$-modules. Now, since $[Z(Q_{\alpha'+1}), Y^\beta\cap Q_{\alpha'}]$ is of exponent $p$ and centralized by $(Y^\beta\cap Q_{\alpha'})Q_{\alpha'+1}\in\syl_p(G_{\alpha',\alpha'+1})$, we have that $[Z(Q_{\alpha'+1}), Y^\beta\cap Q_{\alpha'}]=Z_{\alpha'}$ is of order $p$ from which it follows that $Z(Q_{\alpha'+1})$ contains a unique summand. Hence, $Z(Q_{\alpha'+1})=Z_{\alpha'+1}$ and by conjugacy, $Z_{\alpha}=Z(Q_{\alpha})$. But then $Y^\beta\le V_{\beta}$, and we have a contradiction.

Suppose now that $[Y^\beta\cap Q_{\alpha'}, V_{\alpha'}]=\{1\}$. Then $[V_{\alpha'}, Y^\beta]\le V_{\beta}$ and, as $Z_{\alpha}\ne Z_{\alpha+2}$, we conclude that $Y^\beta V_{\beta}\normaleq L_{\beta}=\langle V_{\alpha'}, R_{\beta}, Q_{\alpha}\rangle$. But $V_{\alpha'}$ centralizes $Y^\beta V_{\beta}/V_{\beta}$ so that $O^p(L_{\beta})$ centralizes $Y^\beta V_{\beta}/V_{\beta}$ and it follows that $Y^\beta V_{\beta}=Z(Q_{\alpha})V_{\beta}\normaleq L_{\beta}$. Then $[Z(Q_{\alpha}), Q_{\beta}]\normaleq L_{\beta}$ and since $Q_{\alpha}\cap Q_{\beta}$ centralizes $[Z(Q_{\alpha}), Q_{\beta}]$ and $Q_{\alpha}\cap Q_{\beta}\not\normaleq L_{\beta}$ by \cref{push}, we must have that $[Z(Q_{\alpha}), Q_{\beta}]\le Z(S)$ and $[Z(Q_{\alpha}), Q_{\beta}, L_{\beta}]=\{1\}$. Now, $[O^p(L_{\beta}), Z(Q_{\alpha}), Q_{\beta}]\le [V_{\beta}, Q_{\beta}]=Z_{\beta}$ and by the three subgroup lemma $[Q_{\beta}, O^p(L_{\beta}), Z(Q_{\alpha})]\le Z_{\beta}\le Z_{\alpha}$. Since $[Q_{\beta}, O^p(L_{\beta})]\not\le Q_{\alpha}$, it follows that $O^p(L_{\alpha})$ centralizes $Z(Q_{\alpha})/Z_{\alpha}$ and coprime action yields $Z(Q_{\alpha})=[Z(Q_{\alpha}), O^p(L_{\alpha})]\times C_{Z(Q_{\alpha})}(O^p(L_{\alpha}))$. But $Z_{\beta}\le Z_{\alpha}=[Z(Q_{\alpha}), O^p(L_{\alpha})]$ and $Z(Q_{\alpha})=Z_{\alpha}$. Since $Z(Q_{\alpha'})\le Z(T)$, for $T\in\syl_p(L_{\alpha'}\cap L_{\alpha'-1})$, we have that $Z(Q_{\alpha'})=Z_{\alpha'}$ and $Z(Q_{\alpha})=Z_{\alpha}$, as required.
\end{proof}

Thus, throughout this section, whenever we assume the necessary values of $b$, we are able to apply \cref{VBGood} through \cref{GoodAction4}. That the hypotheses of these lemmas are satisfied will often be left implicit in proofs.

We now provide some generic results in the case $C_{V_\beta}(V_{\alpha'})=V_\beta \cap Q_{\alpha'}$. These will also be useful for certain inductive arguments in the case $C_{V_\beta}(V_{\alpha'})<V_\beta \cap Q_{\alpha'}$ later.

\begin{lemma}\label{SL2VlQ}
Suppose that $b>1$, $C_{V_\beta}(V_{\alpha'})=V_\beta \cap Q_{\alpha'}$ and $V_{\alpha'}\le Q_\beta$. Then $Q_{\beta}\in\syl_p(R_{\beta})$, $L_{\alpha}/R_\alpha\cong\SL_2(p^n)\cong L_\beta/R_{\beta}$, both $Z_\alpha$ and $V_\beta/C_{V_{\beta}}(O^p(L_{\beta}))$ are natural $\SL_2(p^n)$-modules and $Z_{\alpha}Q_{\alpha'}\in\syl_p(L_{\alpha'})$. Moreover, $[Q_\beta, V_\beta]=Z_\beta=[V_{\alpha'}, V_\beta]\le V_{\alpha'}\cap V_\beta$.
\end{lemma}
\begin{proof}
Suppose that $C_{V_\beta}(V_{\alpha'})=V_\beta\cap Q_{\alpha'}$ and $V_{\alpha'}\le Q_\beta$. Note, that if $V_{\alpha'}\le Q_{\alpha}$, then $[Z_\alpha, V_{\alpha'}]=\{1\}$ and $Z_{\alpha}\le Q_{\alpha'}$, a contradiction. Additionally, $[Z_\alpha, V_{\alpha'}, V_{\alpha'}]\le [V_\beta, V_{\alpha'}, V_{\alpha'}]=\{1\}$ and it follows that both $Z_\alpha$ and $V_{\alpha'}$ admit quadratic action. Notice that $Z_{\alpha}\cap Q_{\alpha'}=C_{Z_{\alpha}}(V_{\alpha'})$ and that $V_{\alpha'}\cap Q_{\alpha}=C_{V_{\alpha'}}(Z_{\alpha})$, and set $r_{\alpha}=|Z_{\alpha}Q_{\alpha'}/Q_{\alpha'}|$ and $r_{\alpha'}=|V_{\alpha'}Q_{\alpha}/Q_{\alpha}|$. 

Assume first that $r_{\alpha'}<r_{\alpha}$. Then for $V:=V_{\alpha'}/[V_{\alpha'}, Q_{\alpha'}]$, $V$ contains a non-central chief factor for $L_{\alpha'}$ by \cref{CommCF} so that $Q_{\alpha'}\in\syl_p(C_{L_{\alpha'}}(V))$. Now, we have that $|V/C_V(Z_{\alpha})|\leq r_{\alpha'}<r_{\alpha}=|Z_{\alpha}/C_{Z_{\alpha}}(V)|$ and we obtain a contradiction by \cref{SEFF}.

Suppose now that $r_{\alpha}\leq r_{\alpha'}$. Then $|Z_{\alpha}/C_{Z_{\alpha}}(V_{\alpha'})|=r_{\alpha}\leq r_{\alpha'}=|V_{\alpha'}/C_{V_{\alpha'}}(Z_{\alpha})|$ and $Z_{\alpha}$ is an FF-module for $L_{\alpha}/R_{\alpha}\cong \SL_2(r_{\alpha})$, and applying \cref{SEFF} acknowledging that $Z(L_{\alpha})=\{1\}$, $Z_{\alpha}$ is a natural module. But then $r_{\alpha'}\leq r_{\alpha} \leq r_{\alpha'}$ so that $r_{\alpha}=r_{\alpha'}$, $[Z_{\alpha'-1}, Q_{\alpha'}]=[V_{\alpha'}, Q_{\alpha'}]=Z_{\alpha'}$ is of order $r_\alpha$ and $V_{\alpha'}/Z_{\alpha'}$ is an FF-module for $L_{\alpha'}/R_{\alpha'}\cong \SL_2(r_{\alpha'})$. In particular, $q_{\alpha'}=q_\beta=q_\alpha$. By \cref{BasicVB}, $Q_{\beta}\in\syl_p(R_{\beta})$ and as $\{1\}\ne[V_{\alpha'}, V_{\beta}]\le [Q_{\beta}, V_{\beta}]$, we conclude that $Z_{\beta}=[V_{\alpha'}, V_{\beta}]\le V_{\alpha'}\cap V_{\beta}$, and finally applying \cref{SEFF}, the result holds.
\end{proof}

\begin{lemma}\label{SL2VnQ}
Suppose that $b>1$, $C_{V_\beta}(V_{\alpha'})=V_\beta \cap Q_{\alpha'}$ and $V_{\alpha'}\not\le Q_\beta$. Then $Q_{\beta}\in\syl_p(R_{\beta})$, $L_{\alpha}/R_\alpha\cong\SL_2(p^n)\cong L_\beta/R_{\beta}$, both $Z_\alpha$ and $V_\beta/C_{V_{\beta}}(O^p(L_{\beta}))$ are natural $\SL_2(p^n)$-modules, $Z_{\alpha}Q_{\alpha'}\in\syl_p(L_{\alpha'})$ and $V_{\alpha'}Q_{\beta}\in\syl_p(L_{\beta})$.
\end{lemma}
\begin{proof}
Assume throughout that $C_{V_\beta}(V_{\alpha'})=V_\beta \cap Q_{\alpha'}$ and $V_{\alpha'}\not\le Q_\beta$. Suppose first that $|V_{\beta}/C_{V_\beta}(V_{\alpha'})|=|V_{\beta}Q_{\alpha'}/Q_{\alpha'}|=p$. Then by \cref{SEFF}, $V_{\beta}/C_{V_{\beta}}(O^p(L_{\beta}))$ is a natural $\SL_2(p)$-module for $L_{\beta}/R_{\beta}\cong \SL_2(p)$. Since $Q_{\alpha}\cap Q_{\beta}\not\normaleq L_{\beta}$ by \cref{push}, $Q_{\beta}\cap O^p(L_{\beta})\not\le Q_{\alpha}$ and $Z_{\alpha}\cap C_{V_\beta}(O^p(L_{\beta}))$ is centralized by $Q_{\beta}\cap O^p(L_{\beta})$. Now, $V_{\beta}\ne Z_{\alpha}C_{V_{\beta}}(O^p(L_{\beta}))$, for otherwise $Q_{\alpha}$ centralizes $V_{\beta}/C_{V_{\beta}}(O^p(L_{\beta}))$ and $O^p(L_{\beta})$ centralizes $V_{\beta}$, and so $Z_{\alpha}\cap C_{V_\beta}(O^p(L_{\beta}))$ has index $p$ in $Z_{\alpha}$. Thus, $Z_{\alpha}$ is an FF-module and by \cref{SEFF}, using that $Z(L_{\alpha})=\{1\}$, $Z_{\alpha}$ is a natural $\SL_2(p)$-module for $L_{\alpha}/R_{\alpha}\cong \SL_2(p)$. Then, $[Q_{\beta}, V_{\beta}]=[Q_{\beta}, Z_{\alpha}]^{G_{\beta}}=Z_{\beta}\le C_{V_{\beta}}(O^p(L_{\beta}))$ and by \cref{BasicVB}, $Q_{\beta}\in\syl_p(R_{\beta})$ and the result holds. 

Thus, $|V_{\beta}Q_{\alpha'}/Q_{\alpha'}|\geq p^2$ and as $V_{\beta}$ is elementary abelian, $m_p(S/Q_{\beta})\geq 2$. Since $V_{\beta}$ acts quadratically on $V_{\alpha'}$ we infer, using \cref{SEQuad} and \cref{GreenbookQuad} when $p$ is odd, that $\bar{L_{\beta}}/O_{p'}(\bar{L_{\beta}})$ is a rank $1$ group of Lie type but not a Ree group, $V_{\beta}Q_{\alpha'}/Q_{\alpha'}\le \Omega(Z(T/Q_{\alpha'}))$ where $T\in\syl_p(G_{\alpha', \alpha'-1})$, and $q_\beta\geq p^2$.

Assume first that $|V_{\alpha'}Q_{\beta}/Q_{\beta}|=p$. Then a subgroup of index at most $p r_\alpha$ is centralized by $Z_{\alpha}$, where $r_\alpha:=|(V_{\alpha'}\cap Q_{\beta})Q_{\alpha}/Q_{\alpha}|\ne\{1\}$. Applying \cref{SEFF} we have that if $|Z_{\alpha}Q_{\alpha'}/Q_{\alpha'}|\geq pr_\alpha$ then $|Z_{\alpha}Q_{\alpha'}/Q_{\alpha'}|= pr_\alpha$, $V_{\alpha'}/C_{V_{\alpha'}}(O^p(L_{\alpha'}))$ is natural module for $L_{\alpha'}/R_{\alpha'}\cong \SL_2(q_\alpha')$ and the results holds upon applying \cref{SL2implies}. Hence, we may assume that $|Z_{\alpha}Q_{\alpha'}/Q_{\alpha'}|\leq r_\alpha$. Moreover, since $Z_{\alpha}\cap Q_{\alpha'}$ is centralized by $V_{\alpha'}\cap Q_{\beta}$, applying \cref{SEFF} we deduce that $r_\alpha= |Z_{\alpha}Q_{\alpha'}/Q_{\alpha'}|\leq q_\beta$ and $Z_{\alpha}$ is a natural module for $L_{\alpha}/R_{\alpha}\cong \SL_2(r_\alpha)$. Since $|V_{\beta}Q_{\alpha'}/Q_{\alpha'}|\geq p^2$, we ascertain that $V_{\alpha'}$ contains a unique non-central chief factor for $L_{\alpha'}$ so that by \cref{CommCF}, $[V_{\alpha'}, Q_{\alpha'}]\le C_{V_{\alpha'}}(O^p(L_{\alpha'}))$ and $Q_{\beta}\in R_{\beta}$.

If $r_\alpha=p$ then $Z_\alpha$ is a natural module for $L_{\alpha}/R_{\alpha}\cong \SL_2(p)$. Moreover, $V_{\alpha'}/C_{V_{\alpha'}}(O^p(L_{\alpha'}))$ is a $2$F-module determined by \cref{Quad2F}. Since $|V_{\beta}Q_{\alpha'}/Q_{\alpha'}|\geq p^2$ and $V_{\beta}$ acts quadratically on $V_{\alpha'}$, the only possibility is that $V_{\alpha'}/C_{V_{\alpha'}}(O^p(L_{\alpha'}))$ is a natural $\SL_2(p^2)$-module, a contradiction by \cref{NatMod}.

Suppose now that $r_\alpha>p$. If $r_\alpha=p^2$ then a subgroup of index at most $p^3$ of $V_{\alpha'}$ is centralized by $Z_{\alpha}$ and by \cref{SEFF}, \cref{SpeMod2} and \cref{SpeModOdd}, we deduce that $\bar{L_{\alpha'}}/O_{p'}(\bar{L_{\alpha'}})\cong\PSL_2(q_\beta)$ where $q_\beta\in\{p^2, p^3\}$. But then, \cref{2FRecog} yields that $V_{\alpha'}/C_{V_{\alpha'}}(O^p(L_{\alpha'}))$ is a natural $\SL_2(q_\beta)$-module for $L_{\alpha'}/R_{\alpha'}\cong \SL_2(q_\beta)$. Since $Z_{\alpha}C_{V_{\beta}}(O^p(L_{\beta}))/C_{V_{\beta}}(O^p(L_{\beta}))$ is $G_{\alpha, \beta}$-invariant subgroup of order $r_\alpha=p^2$, we see by \cref{NatMod} that $q_\alpha=q_\beta=p^2$, as required. Hence, $r_\alpha\geq p^3$. Let $A\le Z_{\alpha}$ of index $p$ such that $Z_{\alpha}\cap Q_{\alpha'}<A$. Since $r_\alpha\geq p^3$, there is also $B\le Z_{\alpha}$ of index $p$ in $A$ strictly containing $Z_{\alpha}\cap Q_{\alpha'}$. Furthermore, since $Z_{\alpha}$ is a natural $\SL_2(r_\alpha)$-module, $C_{V_{\alpha'}\cap Q_{\beta}}(B)=V_{\alpha'}\cap Q_{\beta}\cap Q_{\alpha}$. Since $|V_{\alpha'}Q_{\beta}/Q_{\beta}|=p$, without loss of generality and writing $V:=V_{\alpha'}/Z_{\alpha'}$, either $C_V(A)=C_V(Z_{\alpha})$, or $C_V(A)=C_V(B)$ for every $Z_{\alpha}\cap Q_{\alpha'}<B\le A$ with $[A:B]=p$. We apply \cref{GLS2p'} so that $O_{p'}(\bar{L_{\alpha'}})$ is generated by centralizers of subgroups similar to $A$, or subgroups similar to $B$ for a fixed subgroup $A$. In particular, by the previous observations, $O_{p'}(\bar{L_{\alpha'}})$ normalizes $C_V(A)$. Then writing $H:=\langle (AQ_{\alpha'}/Q_{\alpha'})^{O_{p'}(\bar{L_{\alpha'}})}\rangle$, we have that  $[H, C_V(A)]=\{1\}$. Moreover, $V=[V, O_{p'}(H)]\times C_V(O_{p'}(H))$ is an $AQ_{\alpha'}/Q_{\alpha'}$-invariant decomposition and we conclude that $O_{p'}(H)$ centralizes $V_{\alpha'}$. Thus, $O_{p'}(L_{\alpha'}/R_{\alpha'})$ normalizes $AQ_{\alpha'}/Q_{\alpha'}$ and $L_{\alpha'}/R_{\alpha'}$ is a central extension of a Lie type group.

Since $V_{\alpha'}\cap Q_{\beta}\cap Q_{\alpha}$ has index at most $pq_\beta$ in $V_{\alpha'}$ and is centralized by $Z_{\alpha}$, we apply \cref{GreenbookQuad} to deduce that $L_{\alpha'}/R_{\alpha'}\cong \SL_2(q_{\alpha'})$. In particular, since $L_{\alpha'}=\langle Z_{\alpha}, Z_{\alpha}^x, R_{\alpha'}\rangle$ for some $x\in L_{\alpha'}$, we deduce that $C_{V_{\alpha'}}(O^p(L_{\alpha'}))$ has index at most $p^2 r_\alpha^2\leq p^2 q_\beta^2 \leq q_\beta^3$. Since $q_\beta/p\leq r_\alpha\geq p^2$, appealing to \cref{SL2ModRecog}, we deduce that $V_{\alpha'}/C_{V_{\alpha'}}(O^p(L_{\alpha'}))$ is a natural module for $L_{\alpha'}/R_{\alpha'}\cong \SL_2(q_{\alpha'})$. Then \cref{NatMod} yields that $q_\alpha'=q_\beta=q_\alpha$. 

Assume now that $|V_{\alpha'}Q_{\beta}/Q_{\beta}|\geq p^2$. Since $[V_{\beta}\cap Q_{\alpha'}, V_{\alpha'}]=\{1\}$, \cref{SpeMod2} and \cref{SpeModOdd} yield $\bar{L_{\beta}}/O_{p'}(\bar{L_{\beta}})\cong \PSL_2(q_\beta)$ and $V_{\beta}Q_{\alpha'}\in\syl_p(L_{\alpha'})$. An application of \cref{2FRecog} reveals that either $L_{\beta}/R_{\beta}\cong \SL_2(q_\beta)$ and $V_{\beta}/C_{V_{\beta}}(O^p(L_{\beta}))$ is a natural module; or $q_\beta\geq p^4$. 
Set $A\le V_{\alpha'}$ such that $V_{\alpha'}\cap Q_{\beta}< A$ and $|AQ_{\beta}/Q_{\beta}|p=|V_{\alpha'}Q_{\beta}/Q_{\beta}|$. Write $V:=V_{\beta}/Z_{\beta}$. Suppose that $C_{V}(A)>(V_{\beta}\cap Q_{\alpha'})/Z_{\beta}$ so that by \cref{SpeMod2} and \cref{SpeModOdd}, $|AQ_{\beta}/Q_{\beta}|=p$. Hence, $|V_{\alpha'}Q_{\beta}/Q_{\beta}|=p^2$ and $Z_{\alpha}$ centralizes an index $p^2r_\alpha$ subgroup of $V_{\alpha'}$. If $r_\alpha\geq p^3$, then \cref{2FRecog} gives $V_{\alpha'}/C_{V_{\alpha'}}(O^p(L_{\alpha'}))$ is a natural module for $L_{\alpha'}/R_{\alpha'}\cong \SL_2(q_\beta)$. If $r_\alpha=p^2$ then \cref{SpeMod2} and \cref{SpeModOdd} imply that $q_\beta=p^4$ and by \cref{Quad2F} that $C_V(A)$ has index $p^3$ in $V$. By conjugacy, there is $x\in Z_{\alpha}\setminus (Z_{\alpha}\cap Q_{\alpha'})$ such that $C_{V_{\alpha'}/Z_{\alpha'}}(x)$ has index $p^3$ in $V_{\alpha'}/Z_{\alpha'}$. Let $P_{\alpha'}=\langle Z_{\alpha}, Z_{\alpha}^x, R_{\alpha'}\rangle$ with $x\in L_{\alpha'}$ chosen such that $\bar{L_{\alpha'}}=\bar{P_{\alpha'}}O_{p'}(\bar{L_{\alpha'}})$. Then $|V_{\alpha'}/C_{V_{\alpha'}}(P_{\alpha'})|\leq q_\beta^2$ so that by \cref{q^3module}, we have that $P_{\alpha'}R_{\alpha'}/R_{\alpha'}\cong \SL_2(p^4)$ and $V_{\alpha'}/C_{V_{\alpha'}}(P_{\alpha'})$ is described by \cref{SL2ModRecog}. Since $V_{\beta}$ acts quadratically on $V_{\alpha'}$ and $C_{V_{\alpha'}/Z_{\alpha'}}(x)$ has index $p^3$ in $V_{\alpha'}/Z_{\alpha'}$, we have a contradiction. Finally, if $r_\alpha=p$ then for $x\in V_{\beta}\setminus (V_{\beta}\cap Q_{\alpha'})$, with $[x, A]\le Z_{\beta}$, then the commutation homomorphism $\phi:A/Z_{\alpha'}\to A/Z_{\alpha'}$ such that $aZ_{\alpha'}\phi=[x, a]Z_{\alpha'}$ has image contained in $Z_{\beta}Z_{\alpha'}/Z_{\alpha'}$ of order at most $p$ from which it follows that $x$ centralizes a subgroup of $V_{\alpha'}/Z_{\alpha'}$ of order at most $p^2$. Hence, $V_{\alpha'}/Z_{\alpha'}$ is determined by \cref{Quad2F} and since $q_\beta\geq p^4$, we have a contradiction.

Thus, $C_{V}(A)=(V_{\beta}\cap Q_{\alpha'})/Z_{\beta}$ for all $A\le V_{\alpha'}$ of index $p$ such that $V_{\beta}\cap Q_{\alpha}< A$. We again apply the coprime action argument from \cref{GLS2p'} in a similar manner as before so that $O_{p'}(L_{\beta}/R_{\beta})$ normalizes $V_{\alpha'}R_{\beta}/R_{\beta}$ and $L_{\beta}/R_{\beta}\cong \SL_2(q_\beta)$. Since $V_{\beta}$ acts quadratically on $V_{\alpha'}$, applying \cref{DirectSum}, we deduce that $V_{\beta}/C_{V_{\beta}}(O^p(L_{\beta}))$ is a natural module for $L_{\beta}/R_{\beta}$. Then \cref{NatMod} yields $q_\alpha=q_\beta$, as desired.
\end{proof}

We now prove the ``converse" to the above two statements.

\begin{lemma}\label{NotNatural}
If $b>1$ and $V_{\beta}/C_{V_\beta}(O^p(L_\beta))$ is a natural module for $L_{\beta}/R_{\beta}\cong\SL_2(q_\beta)$, then $C_{V_\beta}(V_{\alpha'})=V_\beta \cap Q_{\alpha'}$.
\end{lemma}
\begin{proof}
Applying \cref{SL2implies}, we have that $Z_{\alpha}$ is a natural module for $L_{\alpha}/R_{\alpha}\cong \SL_2(q_\alpha)$ and $q_\alpha=q_\beta$. Throughout, we assume that $C_{V_\beta}(V_{\alpha'})<V_\beta \cap Q_{\alpha'}$ so that $[V_{\beta}\cap Q_{\alpha'}, V_{\alpha'}]=Z_{\alpha'}$. We may suppose that $C_{V_{\beta}}(O^p(L_{\beta}))$ acts non-trivially on $V_{\alpha'}$ for otherwise $V_{\beta}\cap Q_{\alpha'}=C_{V_{\beta}}(O^p(L_{\beta}))Z_{\alpha+2}$ centralizes $V_{\alpha'}$. As in the proof of \cref{VBGood}, we have that $[V_{\alpha'}, C_{V_{\beta}}(O^p(L_{\beta}))]\le [L_{\beta}, C_{V_{\beta}}(O^p(L_{\beta}))]\le Z_{\beta}$.

Let $\alpha'+1\in\Delta(\alpha')$ with $[V_{\beta}\cap Q_{\alpha'}, Z_{\alpha'+1}]\ne \{1\}$. Indeed, we have that $[V_{\beta}\cap Q_{\alpha'}, Z_{\alpha'+1}]=Z_{\alpha'}$. We set $V^{\alpha'+1}:=\langle C_{V_{\alpha'}}(O^p(L_{\alpha'}))^{G_{\alpha'+1}}\rangle$ throughout. Note that if $|V_{\beta}|=q_\beta^3$ then $V_{\beta}\cap Q_{\alpha'}=Z_{\alpha+2}$ centralizes $V_{\alpha'}$, a contradiction. Hence, by \cref{VBGood}, both $V_{\alpha'+1}^{(2)}/V^{\alpha'+1}$ and $V^{\alpha'+1}/Z_{\alpha'+1}$ contains non-central chief factors for $L_{\alpha'+1}$.

Assume first that $|Z_{\alpha'+1}Q_{\beta}/Q_{\beta}|< q_{\beta}$. Then, $Z_{\beta}=[V_{\beta}\cap Q_{\alpha'}, Z_{\alpha'+1}\cap Q_{\beta}]=Z_{\alpha'}$. Indeed, since $V_{\beta}\not\le Q_{\alpha'}$, $|V_{\alpha'}Q_{\beta}/Q_{\beta}|=q_\beta$. But then, since $Z_{\alpha'-1}\le Q_{\beta}$, we must have that $C_{V_{\alpha'}}(O^p(L_{\alpha'}))\not\le Q_{\beta}$, and $[C_{V_{\alpha'}}(O^p(L_{\alpha'})), V_{\beta}]\le Z_{\alpha'}=Z_{\beta}$, a contradiction. Hence, $Z_{\alpha'+1}Q_{\beta}\in\syl_p(L_{\beta})$. 

Note that if $C_{V_{\beta}}(O^p(L_{\beta}))\not\le Q_{\alpha'}$, then $[V_{\alpha'}, C_{V_{\beta}}(O^p(L_{\beta}))]\le Z_{\beta}$ and as $|[V_{\alpha'}, C_{V_{\beta}}(O^p(L_{\beta}))]|\geq q_\beta$, $[V_{\alpha'}, C_{V_{\beta}}(O^p(L_{\beta}))]=Z_{\beta}\le V_{\alpha'}$ and $Z_{\alpha'}\cap Z_{\beta}=\{1\}$. If $C_{V_{\beta}}(O^p(L_{\beta}))\le Q_{\alpha'}$ then as $C_{V_{\beta}}(O^p(L_{\beta}))$ is non-trivial on $V_{\alpha'}$, $Z_{\alpha'}=[C_{V_{\beta}}(O^p(L_{\beta})), Z_{\alpha'+1}]=Z_{\beta}$.

Suppose that $Z_{\alpha+3}\ne Z_{\beta}$. If $b=3$, then $Z_{\alpha'}\cap Z_{\beta}=\{1\}$. Then $[V_{\alpha'}, C_{V_{\beta}}(O^p(L_{\beta}))]\le Z_{\beta}\le Z_{\alpha'-1}$. Moreover, $[Q_{\alpha'-1}Q_{\alpha'}, V_{\alpha'}]=[\langle C_{V_{\beta}}(O^p(L_{\beta}))^{G_{\alpha', \alpha'-1}}\rangle Q_{\alpha'}, V_{\alpha'}]\le Z_{\alpha'-1}$ so that $Z_{\alpha'-1}Z_{\alpha'-1}^g$ is of order $q_\beta^3$ and normalized by $L_{\alpha'}=\langle Q_{\alpha'-1}, Q_{\alpha'-1}^g, R_{\alpha'}\rangle$ for some appropriately chosen $g\in L_{\alpha'}$. But then, by definition, $|V_{\alpha'}|=|V_{\beta}|=q_\beta^3$ and $V_{\beta}\cap Q_{\alpha'}=Z_{\alpha'-1}$ centralizes $V_{\alpha'}$, a contradiction. 

Hence, $b>3$ so that $V_{\alpha'}^{(3)}$ centralizes $Z_{\beta}\le V_{\alpha'}$. Then $V_{\alpha'}^{(3)}$ centralizes $Z_{\alpha+2}=Z_{\alpha+3}Z_{\beta}$ and either $V_{\alpha'}^{(3)}=V_{\alpha'}(V_{\alpha'}^{(3)}\cap Q_{\beta})$, or $V_{\alpha'}^{(3)}\not\le Q_{\alpha+3}$ and $Z_{\alpha+4}C_{V_{\alpha+3}}(O^p(L_{\alpha+3}))=Z_{\alpha+2}C_{V_{\alpha+3}}(O^p(L_{\alpha+3}))$. The former case yields a contradiction since $V_{\beta}\not\le Q_{\alpha'}$ and $[V_{\beta}, V_{\alpha'}^{(3)}]\le V_{\alpha'}$. In the latter case, we still have that $[V_{\beta}, V_{\alpha'}^{(3)}\cap Q_{\alpha+3}]\le V_{\alpha'}$ so that $V_{\alpha'}^{(3)}/V_{\alpha'}$ contains a unique non-central chief factor for $L_{\alpha'}$, which as $\bar{L_{\alpha'}}$-module, is an FF-module. By \cref{VBGood} we have that $Z_{\alpha+2}=Z_{\alpha+4}$ and \cref{GoodAction3} with \cref{SimExt} implies that $Z_{\alpha}\le V_{\alpha+2}^{(2)}=V_{\alpha+4}^{(2)}\le Q_{\alpha'}$, a contradiction.

Therefore, for the remainder of this proof, we may assume that $Z_{\alpha+3}=Z_{\beta}$. Assume first that $Z_{\alpha'}=Z_{\beta}$. Then $V_{\alpha'+1}^{(2)}\cap Q_{\alpha+3}\cap Q_{\alpha+2}$ is centralized, modulo $Z_{\alpha'+1}$, by $V_{\beta}\cap Q_{\alpha'}$. If $b=3$ and $V^{\alpha'+1}\not\le Q_{\alpha'}$, then since $V^{\alpha'+1}$ is $G_{\alpha'+1, \alpha'}$ invariant, $V^{\alpha'+1}Q_{\alpha'}\in\syl_p(G_{\alpha'+1, \alpha'})$. But then, $[V^{\alpha'+1}, V_{\alpha'}]\le Z_{\alpha'+1}$ so that $Z_{\alpha'+1}Z_{\alpha'+1}^g$ is normalized by $L_{\alpha'}=\langle Q_{\alpha'+1}, Q_{\alpha'+1}^g, R_{\alpha'}\rangle$ for some appropriate $g\in L_{\alpha'}$, and $|V_{\alpha'}|=|Z_{\alpha'+1}Z_{\alpha'+1}^g|=q_\beta^3$, a contradiction. Hence, $V^{\alpha'+1}\le Q_{\alpha'}$. Then $V^{\alpha'+1}\cap Q_{\alpha'-1}$ is centralized, modulo $Z_{\alpha'+1}$, by $V_{\beta}\cap Q_{\alpha'}$. But then adapting \cref{GoodAction1}, $O^p(R_{\alpha'+1})$ centralizes $V^{\alpha'+1}$ and an adaptation of \cref{SimExt} to $Z_{\alpha'}=Z_{\beta}$ yields that $Z_{\alpha'-1}C_{V_{\alpha'}}(O^p(L_{\alpha'}))=Z_{\alpha'+1}C_{V_{\beta}}(O^p(L_{\beta}))=V_{\beta}\cap Q_{\alpha'}$, a contradiction. Hence, $b>3$.

Note that $V_{\alpha+3}\le Q_{\alpha'+1}$ so that $[V_{\alpha+3}, V^{\alpha'+1}]\le Z_{\alpha'+1}\cap V_{\alpha+3}=Z_{\alpha'}=Z_{\alpha+3}$. In particular, $V^{\alpha'+1}\le Q_{\alpha+3}$. Assume that $C_{Z_{\alpha+2}}(V^{\alpha'+1})>Z_{\alpha+3}$. Then $V^{\alpha'+1}\le Q_{\alpha+2}$ so that $V_{\beta}\cap Q_{\alpha'}$ centralizes $V^{\alpha'+1}/Z_{\alpha'+1}$, a contradiction. Since $b>3$ and $V_{\alpha'+1}^{(2)}$ is abelian, $Z_{\alpha+2}\cap V_{\alpha'+1}^{(2)}=Z_{\alpha+3}=Z_{\alpha'}$. Set $V^{\alpha+2}=\langle Z_{\alpha+2}C_{V_{\alpha+3}}(O^p(L_{\alpha+3}))^{G_{\alpha+2}}\rangle$. Then $[V_{\alpha'}, V^{\alpha+2}]\le V_{\alpha'}\cap Z_{\alpha+2}=Z_{\alpha'}$ and $V^{\alpha+2}\le Q_{\alpha'}$. Now, $V^{\alpha'+1}\cap Q_{\alpha+2}$ is centralized, modulo $Z_{\alpha'}$, by $V^{\alpha+2}$ and $V^{\alpha+2}\cap Q_{\alpha'+1}$ is centralized, modulo $Z_{\alpha+3}$, by $V^{\alpha'+1}$. it follows from \cref{SEFF}, using that $\alpha'+1, \alpha+2\in\alpha^G$, that $V^{\alpha}/Z_{\alpha}$ is an FF-module for $\bar{L_{\alpha}}$ so that $V^{\alpha'+1}Q_{\alpha+2}\in\syl_p(L_{\alpha+2})$ and $V^{\alpha+2}Q_{\alpha'+1}\in\syl_p(L_{\alpha'+1})$. But now $V_{\alpha'+1}^{(2)}\cap Q_{\alpha'}$ is centralized, modulo $V^{\alpha'+1}$, by $V^{\alpha+2}$ so that by \cref{SEFF}, $V_{\alpha}^{(2)}/V^{\alpha}$ is also an FF-module for $\bar{L_{\alpha}}$. By \cref{GoodAction1}, $O^p(R_{\alpha})$ centralizes $V_{\alpha}^{(2)}$ and \cref{SimExt} applied to $Z_{\beta}=Z_{\alpha+3}$ yields that $V_{\beta}=V_{\alpha+3}\le Q_{\alpha'}$, a contradiction.

Assume now that $Z_{\alpha'}\cap Z_{\beta}=\{1\}$ so that $b>3$. In particular, since $V^{\alpha'+1}\cap V_{\alpha'}=Z_{\alpha'+1}C_{V_{\alpha'}}(O^p(L_{\alpha'}))$, it follows that $Z_{\beta}\cap V^{\alpha'+1}=\{1\}$. Since $V_{\alpha+3}\le Q_{\alpha'+1}$, we have that $[V_{\alpha'+1}^{(2)}\cap Q_{\alpha+3}, V_{\alpha+3}]\le V^{\alpha'+1}\cap Z_{\alpha+3}=\{1\}$ and $V_{\alpha'+1}^{(2)}\cap Q_{\alpha+3}\le Q_{\alpha+2}$. But now, $V^{\alpha'+1}\cap Q_{\alpha+3}$ is centralized, modulo $Z_{\alpha'+1}$, by $V_{\beta}\cap Q_{\alpha'}$ and we have that $V^{\alpha'+1}Q_{\alpha+3}\in\syl_p(L_{\alpha+3})$. Moreover, $[V^{\alpha'+1}, V_{\alpha+3}]\le V_{\alpha+3}\cap Z_{\alpha'+1}=Z_{\alpha'}$. Then, $Z_{\alpha'}Z_{\alpha'}^gZ_{\alpha+3}$ is normalized by $L_{\alpha+3}=\langle V_{\alpha'+1}, V_{\alpha'+1}^g, R_{\alpha+3}\rangle$ for some appropriate $g\in L_{\alpha+3}$. Hence, $|[V_{\alpha+3}, O^p(L_{\alpha+3})]|=q_\beta^3$ and $V_{\alpha+3}=[V_{\alpha+3}, O^p(L_{\alpha})]Z_{\alpha+2}$ has order either $q_\beta^3$ or $q_\beta^4$. Either way, we have that $[V_{\alpha'}, V_{\beta}]Z_{\alpha+2}$ has index $q_\beta$ in $V_{\beta}$, from which it follows that $V_{\beta}\cap Q_{\alpha'}=[V_{\alpha'}, V_{\beta}]Z_{\alpha+2}$ centralizes $V_{\alpha'}$, a final contradiction.
\end{proof}

\begin{proposition}
Suppose that $b>1$, $C_{V_\beta}(V_{\alpha'})=V_\beta \cap Q_{\alpha'}$ and $V_{\alpha'}\not\le Q_\beta$. Then $C_{V_{\alpha'}}(V_{\beta})=V_{\alpha'}\cap Q_{\beta}$.
\end{proposition}
\begin{proof}
Since $V_{\alpha'}\not\le Q_{\beta}$, there is $\alpha'+1\in\Delta(\alpha')$ with $Z_{\alpha'+1}\not\le Q_{\beta}$ and $(\alpha'+1, \beta)$ a critical pair. Moreover, by \cref{SL2VnQ} and since $\alpha'\in \beta^G$, $V_{\alpha'}/C_{V_{\alpha'}}(O^p(L_{\alpha'}))$ is a natural module for $L_{\alpha'}/R_{\alpha'}\cong \SL_2(q_\alpha')$. Then \cref{NotNatural} applied to $(\alpha'+1, \alpha')$ in place of $(\alpha, \alpha')$ gives the result.
\end{proof}

\subsection{$C_{V_\beta}(V_{\alpha'})<V_{\beta}\cap Q_{\alpha'}$}

The hypothesis for this subsection is $b>1$ and $C_{V_\beta}(V_{\alpha'})<V_{\beta}\cap Q_{\alpha'}$. Notice by \cref{b>1} this this condition is equivalent to $[V_{\beta}\cap Q_{\alpha'}, V_{\alpha'}]\ne\{1\}$. By \cref{SL2VlQ}, \cref{SL2VnQ} and \cref{NotNatural}, throughout this section, whenever $(\alpha^*, {\alpha^*}')$ is a critical pair, we have that $[V_{\alpha^*+1}\cap Q_{{\alpha^*}'}, V_{{\alpha^*}'}]\ne\{1\}$ and $V_{{\alpha^*}'}/C_{{\alpha^*}'}(O^p(L_{{\alpha^*}'}))$ is never a natural $\SL_2(q_{{\alpha^*}'})$-module.

The aim of this subsection will be to recognize amalgams of type ${}^2\mathrm{F}_4(2^n)$ and ${}^2\mathrm{F}_4(2)'$ via the identifications provided in \cite{Greenbook}.

\begin{proposition}\label{VnotB1}
Suppose that $b>1$ and $C_{V_\beta}(V_{\alpha'})<V_\beta \cap Q_{\alpha'}$. Then $V_{\alpha'}\not\le Q_{\beta}$.
\end{proposition}
\begin{proof}
Suppose that $V_{\alpha'}\le Q_{\beta}$. Note that if $Z_{\alpha}$ is a natural $\SL_2(q_\alpha)$-module for $L_{\alpha}/R_{\alpha}$, then $[V_{\beta}, Q_{\beta}]=[Z_{\alpha}, Q_{\beta}]^{L_{\beta}}=Z_{\beta}$. Moreover, for $\lambda\in\Delta(\alpha')$ with $V_{\beta}\cap Q_{\alpha'}\not\le Q_{\lambda}$, $Z_{\alpha'}=[V_{\beta}\cap Q_{\alpha'}, Z_{\lambda}]\le [V_{\beta}, Q_{\beta}]=Z_{\beta}$ so that $Z_{\alpha'}=Z_{\beta}=[V_{\alpha'}, V_{\beta}]$, a contradiction since $V_{\beta}\not\le Q_{\alpha'}$. Hence, by \cref{SL2implies}, $V_{\beta}/C_{V_{\beta}}(O^p(L_{\beta}))$ contains at least two non-central chief factors.

Let $\alpha'+1\in\Delta(\alpha')$ with $Z_{\alpha'+1}\not\le Q_{\alpha}$. Since neither $Z_{\alpha'+1}$ nor $Z_{\alpha}$ are FF-modules, we have that \[|(Z_{\alpha}\cap Q_{\alpha'})Q_{\alpha'+1}/Q_{\alpha'+1}|<|Z_{\alpha'+1}Q_{\alpha}/Q_{\alpha}|<|Z_{\alpha}/(Z_{\alpha}\cap Q_{\alpha'}\cap Q_{\alpha'+1})|.\] In particular, $|Z_{\alpha}Q_{\alpha'}/Q_{\alpha'}|>p$ and as $Z_{\alpha}$ acts quadratically on $V_{\alpha'}$, by conjugacy and applying \cref{SEQuad}, $\bar{L_{\beta}}/O_{p'}(\bar{L_{\beta}})\cong \PSL_2(q_\beta), \PSU_3(q_\beta)$ or $\Sz(q_\beta)$. If $p=2$, then $Z_{\alpha}Q_{\alpha'}/Q_{\alpha'}\le \Omega(Z(T/Q_{\alpha'}))$ and if $p$ is odd, then applying \cref{SEQuad} and \cref{GreenbookQuad}, we again conclude that $Z_{\alpha}Q_{\alpha'}/Q_{\alpha'}\le \Omega(Z(T/Q_{\alpha'}))$ for $T\in\syl_p(G_{\alpha', \alpha'-1})$. Hence $p<|Z_{\alpha}Q_{\alpha'}/Q_{\alpha'}|\leq q_\beta$ and by \cref{SEFF} and \cref{NotNatural}, we have that $p<|Z_{\alpha}Q_{\alpha'}/Q_{\alpha'}|<|V_{\alpha'}Q_{\alpha}/Q_{\alpha}|$. Arguing as before, $p<|V_{\alpha'}Q_{\alpha}/Q_{\alpha}|\leq q_\alpha$ and $\bar{L_{\alpha}}/O_{p'}(\bar{L_{\alpha}})\cong \PSL_2(q_\alpha), \PSU_3(q_\alpha)$ or $\Sz(q_\alpha)$. 

Since $Z_{\alpha}$ centralizes a subgroup of $V_{\alpha'}$ of index at most $q_\alpha$ and $V_{\alpha'}$ contains at least two non-central chief factors for $L_{\alpha'}$, applying \cref{SpeMod2} and \cref{SpeModOdd}, we deduce that $q_\alpha \geq q_\beta^2$. By a similar reasoning, if $|(Z_{\alpha}\cap Q_{\alpha'})Q_{\alpha'+1}/Q_{\alpha'+1}|=p$, then \cref{SpeMod2} and \cref{SpeModOdd} yields $q_\beta p\geq q_\alpha$ so that $q_\alpha p^2\geq q_\beta^2 p^2\geq q_\alpha^2$ and $p^2< q_\beta^2\leq q_\alpha\leq p^2$, a contradiction.
Then $|(Z_{\alpha}\cap Q_{\alpha'})Q_{\alpha'+1}/Q_{\alpha'+1}|>p$ and applying \cref{SpeMod2} and \cref{SpeModOdd} we conclude that $|Z_{\alpha'+1}Q_{\alpha}/Q_{\alpha}|=q_\alpha$, $\bar{L_{\alpha}}/O_{p'}(\bar{L_{\alpha}})\cong \PSL_2(q_\alpha)$ and $S=Z_{\alpha'+1}Q_{\alpha}$. Then $Z_{\beta}$ has index at most $|Z_{\alpha}/(Z_{\alpha}\cap Q_{\alpha'}\cap Q_{\alpha'+1})|\leq q_\alpha q_\beta\leq q_\alpha^{\frac{3}{2}}$ in $Z_{\alpha}$. Applying \cref{2FRecog}, we deduce that $Z_{\alpha}$ is a natural module for $L_{\alpha}/R_{\alpha}\cong \SL_2(q_\alpha)$, a contradiction.
\end{proof}

For the remainder of this section, we fix $\alpha'+1\in\Delta(\alpha')$ such that $V_{\beta}\cap Q_{\alpha'}\not\le Q_{\alpha'+1}$. Note also that $[Z_{\alpha'+1}, V_{\beta}\cap Q_{\alpha'}, V_{\beta}\cap Q_{\alpha'}]=\{1\}$ so that both $Z_{\alpha'+1}$ and $V_{\alpha'}$ admit non-trivial quadratic action. Throughout, we set $R:=[V_{\beta}\cap Q_{\alpha'}, V_{\alpha'}]$.

The following lemma, along with its proof, appeared earlier as \cref{Quad2F} and \cref{pgen} where the necessary additional hypothesis there follow from \cref{MainHyp}. We recall it here as it will be applied liberally throughout this subsection.

\begin{lemma}\label{NewQuad2F}
For $\gamma\in\Gamma$, $G:=\bar{L_{\gamma}}$ and $S\in\syl_p(G)$, assume that $V$ is a faithful $\mathrm{GF}(p)G$-module with $C_V(O^p(G))=\{1\}$ and $V=\langle C_V(S)^G\rangle$. If there is a $p$-element $1\ne x\in G$ such that $[V, x,x]=\{1\}$ and $|V/C_V(x)|=p^2$ then, setting $L:=\langle x^G\rangle$, one of the following holds:
\begin{enumerate}
\item $p$ is odd, $G=L\cong\mathrm{(P)SU}_3(p)$ and $V$ is the natural module;
\item $p$ is arbitrary, $G\cong\SL_2(p^2)$ and $V$ is the natural module;
\item $p=2$, $G=L\cong \PSL_2(4)$ and $V$ is a natural $\Omega_4^-(2)$-module;
\item $p=3$, $G=L\cong 2\cdot\Alt(5)$ or $2^{1+4}_-.\Alt(5)$ and $V$ is the unique irreducible quadratic $2$F-module of dimension $4$;
\item $p$ is arbitrary, $G=L\cong\SL_2(p)$ and $V$ is the direct sum of two natural $\SL_2(p)$-modules;
\item $p=2$, $L\cong\SU_3(2)'$, $G$ is isomorphic to a subgroup of $\SU_3(2)$ which contains $\SU_3(2)'$ and $V$ is a natural $\SU_3(2)$-module viewed as an irreducible $\mathrm{GF}(2)G$-module by restriction;
\item $p=2$, $L\cong \Dih(10)$, $G\cong\Dih(10)$ or $\Sz(2)$ and $V$ is a natural $\Sz(2)$-module viewed as an irreducible $\mathrm{GF}(2)G$-module by restriction;
\item $p=3$, $G=L\cong (Q_8\times Q_8):3$ and $V=V_1\times V_2$ where $V_i$ is a natural $\SL_2(3)$-module for $G/C_G(V_i)\cong\SL_2(3)$;
\item $p=2$, $G=L\cong(3\times 3):2$ and $V=V_1\times V_2$ where $V_i$ is a natural $\SL_2(2)$-module for $G/C_G(V_i)\cong\Sym(3)$; or
\item $p=2$, $L\cong(3\times 3):2$, $G\cong (3\times 3):4$, $V$ is irreducible as a $\mathrm{GF}(2)G$-module and $V|_L=V_1\times V_2$ where $V_i$ is a natural $\SL_2(2)$-module for $L/C_L(V_i)\cong\Sym(3)$. 
\end{enumerate}
Moreover, if $V$ is generated by a $N_G(S)$-invariant subspace of order $p$ then $(G, V)$ satisfies outcome (iii), (vii) (ix) or (x).
\end{lemma}

\begin{lemma}\label{FirstReduction}
Suppose that $b>1$, $C_{V_\beta}(V_{\alpha'})<V_\beta \cap Q_{\alpha'}$ and $Z_{\alpha}$ is not a natural $\SL_2(q_\alpha)$-module. Then $Z_{\alpha'+1}\not\le Q_{\beta}$ and either:
\begin{enumerate}
\item $|Z_{\alpha'+1}Q_{\beta}/Q_{\beta}|>p$; or
\item $L_{\beta}/R_{\beta}\cong \SL_2(p), (3\times 3):2, (3\times 3):4$ or $(Q_8\times Q_8):3$, $Q_{\beta}\in\syl_p(R_{\beta})$, $[V_{\beta}, Q_{\beta}]\le C_{V_{\beta}}(O^p(L_{\beta}))$, $V_{\beta}/C_{V_{\beta}}(O^p(L_{\beta}))$ is described in \cref{NewQuad2F} and $q_\alpha=p$.
\end{enumerate}
\end{lemma}
\begin{proof}
Note that if $m_p(S/Q_{\beta})>1$ then as $V_{\alpha'}$ acts quadratically on $V_{\beta}$ and vice versa, we have that $\bar{L_{\beta}}/O_{p'}(\bar{L_{\beta}})\cong \PSL_2(q_\beta), \PSU_3(q_\beta)$ or $\Sz(q_\beta)$. As in \cref{VnotB1}, we conclude that $|V_{\alpha'}Q_{\beta}/Q_{\beta}|\leq q_\beta$. We assume throughout that $Z_{\alpha}$ is not a natural module for $L_{\alpha}/R_{\alpha}\cong \SL_2(q_\alpha)$ so that \cref{SL2implies} implies that $V_{\beta}/C_{V_{\beta}}(O^p(L_{\beta}))$ contains at least two non-central chief factors for $L_{\beta}$ whenever $q_\beta>3$.

Assume first that $V_{\alpha'}\cap Q_{\beta}\le Q_{\alpha}$. Applying \cref{SEFF} and \cref{NotNatural}, we deduce that $p<|V_{\alpha'}Q_{\beta}/Q_{\beta}|\leq q_\beta$ and $m_p(S/Q_{\beta})>1$. Applying \cref{SpeMod2} and \cref{SpeModOdd}, since $V_{\alpha'}$ contains at least two non-central chief factors for $L_{\alpha'}$, we deduce that $q_\beta\geq q_\beta^{\frac{4}{3}}$, a contradiction. Hence, $V_{\alpha'}\cap Q_{\beta}\not\le Q_{\alpha}$ and $q_\alpha\geq q_\beta^{\frac{1}{3}}$. 

Taking \cref{SpeMod2} and \cref{SpeModOdd} further whenever $|Z_{\alpha}Q_{\alpha'}/Q_{\alpha'}|>p$, we have that $q_\alpha\geq q_\beta$. Hence, with any condition on the order of $|Z_{\alpha}Q_{\alpha'}/Q_{\alpha'}|$, we have by \cref{SEFF} that there is $\lambda\in\Delta(\alpha')$ with $Z_{\alpha}\cap Q_{\alpha'}\not\le Q_{\alpha'+1}$. We may as well assume that $\lambda=\alpha'+1$. Define $r_\beta:=|Z_{\alpha'+1}Q_{\beta}/Q_{\beta}|$, $r_\alpha:=|(Z_{\alpha'+1}\cap Q_{\beta})Q_{\alpha}/Q_{\alpha}|$, $r_{\alpha'}=|Z_{\alpha}Q_{\alpha'}/Q_{\alpha'}|$ and $r_{\alpha'+1}:=|(Z_{\alpha}\cap Q_{\alpha'})Q_{\alpha'+1}/Q_{\alpha'+1}|$ so that $r_\lambda\le q_\lambda$ for $\lambda\in\{\alpha, \beta, \alpha', \alpha'+1\}$.

Assume that $r_\beta\leq p$. Then by \cref{SEFF}, since $Z_{\alpha}$ is not a natural $\SL_2(p)$-module, $r_\alpha\geq p$. Suppose that $r_\alpha>p$ so that $\bar{L_{\alpha}}/O_{p'}(\bar{L_{\alpha}})\cong \PSL_2(q_\alpha), \PSU_3(q_\alpha)$ or $\Sz(q_\alpha)$. Assume first that $r_{\alpha'+1}=p$. Applying \cref{SEFF}, we have that $p<r_{\alpha'}\leq q_\beta$ so that $q_\beta\leq q_\alpha$. Applying \cref{SpeMod2} and \cref{SpeModOdd}, we conclude that $\bar{L_{\alpha}}/O_{p'}(\bar{L_{\alpha}})\cong\PSL_2(q_\alpha)$ and $q_\alpha/p \leq r_{\alpha'}\leq q_\beta\leq q_\alpha$. If $q_\alpha=p^2$ then applying \cref{2FRecog} to $Z_{\alpha}$ yields that $Z_{\alpha}$ is a natural module for $L_{\alpha}/R_{\alpha}\cong \SL_2(q_\alpha)$. Hence, we either have that $q_\alpha=q_\beta>p^2$, or $p^2\leq q_\alpha/p=r_{\alpha'}=q_\beta$. Applying \cref{2FRecog} to non-central chief factors contained in $V_{\alpha'}$, we deduce that there are exactly two non-central chief factors for $L_{\alpha'}$ within $V_{\alpha'}$ and both, as $\bar{L_{\alpha'}}$-modules, are FF-modules. Hence, applying \cref{DirectSum} and \cref{2FRecog}, if $[V_{\beta}, Q_{\beta}]$ contains no non-central chief factors for $L_{\beta}$, then $V_{\beta}/C_{V_\beta}(O^p(L_{\beta}))$ is a direct sum of two natural modules for $L_{\beta}/R_{\beta}\cong \SL_2(q_\beta)$. But then, $Z_{\beta}=Z_{\alpha}\cap C_{V_{\beta}}(O^p(L_{\beta}))$ has index $q_\beta^2\leq q_\alpha^2$ in $Z_{\alpha}$ and \cref{2FRecog} implies that $q_\beta=q_\alpha$ and $L_{\alpha}/R_{\alpha}\cong \SL_2(q_\alpha)$ and $Z_{\alpha}$ is a direct sum of two natural modules, a contradiction since an index $pr_\alpha<q_\alpha^2$ subgroup of $Z_{\alpha'+1}$ is centralized by $Z_{\alpha}\cap Q_{\alpha'}\not\le Q_{\alpha'+1}$. Hence, $[V_{\beta}, Q_{\beta}]$ contains a non-central chief factor for $L_{\beta}$ and $Q_{\beta}$ is not quadratic on $Z_\alpha$. Since $(V_{\alpha'}\cap Q_{\beta})Q_{\alpha}$ acts quadratically on $Z_{\alpha}$ and $Z_\beta=\Omega(Z(S))$, we must have that $(V_{\alpha'}\cap Q_{\beta})Q_{\alpha}<S$ and $r_\alpha'=q_\beta=q_\alpha/p$. Applying \cref{2FRecog}, we deduce that $|Z_{\alpha}/C_{Z_{\alpha}}(V_{\alpha'}\cap Q_{\beta})|>q_\alpha^2/p^2$. But then we may form $L^{\alpha}=\langle Z_{\alpha'+1}\cap Q_{\beta}, (Z_{\alpha'+1}\cap Q_{\beta})^x, Q_{\alpha}\rangle$ with $x\in L_{\alpha}$ chosen such that $\bar{L^{\alpha}}/O_{p'}(\bar{L^{\alpha}})\cong \bar{L_{\alpha}}/O_{p'}(\bar{L_{\alpha}})$. Since $r_\alpha'p=q_\alpha$, we have that $|Z_{\alpha}/C_{Z_{\alpha}}(L^\alpha)|\leq q_\alpha^2$. Note that $V_{\alpha'}\cap Q_{\beta}\le L^\alpha$ and $[Z_{\alpha}, V_{\alpha'}\cap Q_{\beta}]\le C_{Z_{\alpha}}(V_{\alpha'}\cap Q_{\beta})$ from which it follows that $|[Z_{\alpha}, V_{\alpha'}\cap Q_{\beta}]C_{Z_{\alpha}}(L^\alpha)/C_{Z_{\alpha}}(L^\alpha)|<p^2$. Hence, $|Z_{\alpha}/C_{Z_{\alpha}}(L^\alpha)|$ is dual to an FF-module for $\bar{L^\alpha}$ and applying \cref{SEFF}, since $q_\alpha>p$, we have a contradiction.

Thus, if $r_\beta\leq p$ and $r_\alpha>p$, we may assume that $r_{\alpha'+1}>p$. Then applying \cref{SpeMod2} and \cref{SpeModOdd}, we infer have that $q_\alpha/p\leq r_\alpha$, $\bar{L_{\alpha}}/O_{p'}(\bar{L_{\alpha}})\cong \PSL_2(q_\alpha)$ and $Z_{\alpha}$ is irreducible under the action of $L_{\alpha}$. Since $r_\alpha r_\beta\leq r_\alpha p\leq q_\alpha p$ and $r_{\alpha'+1}\geq p^2$, \cref{2FRecog} applied to the action of $Z_{\alpha}\cap Q_{\alpha'}$ on $Z_{\alpha'+1}$ yields that $q_\alpha>p^2$. Hence, $r_\alpha>q_{\alpha}^{\frac{1}{2}}$ and \cref{2FRecog} yields $r_\alpha^2\leq r_{\alpha'+1}r_{\alpha'}$. If $r_{\alpha'+1}\leq q_{\alpha}^{\frac{1}{2}}$ then as $r_{\alpha'+1}\geq p^2$, it follows that $r_{\alpha'}=q_\alpha=p^4$, $r_\alpha=q_\alpha/p$ and $r_{\alpha'+1}=p^2$. But then, \cref{2FRecog} on the action of $Z_{\alpha'+1}\cap Q_{\beta}$ on $Z_{\alpha}$ implies that $Z_{\alpha}$ is direct sum of two natural $\SL_2(p^4)$-modules. Since $r_\alpha r_\beta\leq q_\alpha$, this is a contradiction. Hence, $r_{\alpha'+1}>q_{\alpha}^{\frac{1}{2}}$ so that $q_{\alpha}p<r_{\alpha'+1}^2$ and \cref{2FRecog} again gives, $r_{\alpha'+1}^2\leq r_\alpha r_\beta\lq q_\alpha p$, another contradiction.

Hence, we deduce that if $r_\beta\leq p$ then $r_\alpha=p$. Again applying \cref{SEFF}, $r_\beta=p$, $Z_{\alpha'+1}\not\le Q_{\beta}$ and $Z_{\alpha}\cap Q_{\alpha'}$ has a quadratic $2$F-action $Z_{\alpha'+1}$. Applying \cref{NewQuad2F}, we conclude that $q_\alpha=p$ or $L_{\alpha}/R_{\alpha}\cong \PSL_2(4)$ and $Z_{\alpha}$ is a natural $\Omega_4^-(2)$-module for $L_{\alpha}/R_{\alpha}\cong \PSL_2(4)$. Since $V_{\alpha'}\cap Q_{\beta}$ acts quadratically on $Z_{\alpha}$, in either case we deduce that $|(V_{\alpha'}\cap Q_{\beta})Q_{\alpha}/Q_{\alpha}|=p$. But then $Z_{\alpha}$ centralizes a subgroup of index at most $q_\beta p$ in $V_{\alpha'}$, and since $V_{\alpha'}$ contains at least two non-central chief factors for $L_{\alpha'}$, we have that $q_\beta p\geq q_\beta^{\frac{4}{3}}$, $q_\beta\leq p^3$ and $V_{\alpha'}$ contains exactly two non-central chief factors. If $q_\beta=p^3$, then it follows by \cref{SEFF} that both non-central chief factors are quadratic $2$F-modules, a contradiction by \cref{NewQuad2F}. If $q_\beta=p^2$, then an index $p$ subgroup of one of the non-central chief factors is centralized and by \cref{SEFF} we have a contradiction. Hence, $q_\beta=p$.

Suppose that $[V_{\beta}, Q_{\beta}]$ contains a non-central chief factor for $L_{\beta}$. Then $[Z_{\alpha}, Q_{\beta}]\not\le Z_{\beta}$ and $|S/Q_{\alpha}|>p$. Moreover, by \cref{CommCF}, $V_{\beta}/[V_{\beta}, Q_{\beta}]$ also contains a non-central chief factor and by \cref{SEFF}, both $V_{\beta}/[V_{\beta}, Q_{\beta}]C_{V_{\beta}}(O^p(L_{\beta}))$ and $[V_{\beta}, Q_{\beta}]C_{V_{\beta}}(O^p(L_{\beta}))/C_{V_{\beta}}(O^p(L_{\beta}))$ are natural $\SL_2(p)$-modules. In particular, by \cref{Badp2}, \cref{Badp3} and \cref{SEQuad} when $p\geq 5$, we have that $L_{\beta}/R_{\beta}\cong \SL_2(p), (3\times 3):2$ or $(Q_8\times Q_8):3$ so that $S=Q_{\alpha}Q_{\beta}$. Furthermore, $(Q_\beta\cap O^p(L_\beta))Q_{\alpha}$ centralizes an index $p$ subgroup of $[Z_\alpha, S]$. Appealing to \cref{NewQuad2F} for the structure of $Z_{\alpha}$, we deduce that $Z_{\alpha}$ is a natural $\SU_3(p)$-module for $L_{\alpha}/R_{\alpha}$ which is isomorphic to a subgroup $\SU_3(p)$ containing $\SU_3(p)'$. But $Z_{\alpha}[V_{\beta}, Q_{\beta}]C_{V_{\beta}}(O^p(L_{\beta}))/[V_{\beta}, Q_{\beta}]C_{V_{\beta}}(O^p(L_{\beta}))$ has order $p$ since $V_{\beta}/[V_{\beta}, Q_{\beta}]C_{V_{\beta}}(O^p(L_{\beta}))$ is a natural $\SL_2(p)$-module so that there is $G_{\alpha,\beta}$-invariant subgroup of $Z_{\alpha}$ of index $p$, a contradiction. Hence, $[V_{\beta}, Q_{\beta}]\le C_{V_{\beta}}(O^p(L_{\beta}))$, $Q_{\beta}\in\syl_p(R_{\beta})$ and $V_{\beta}/C_{V_{\beta}}(O^p(L_{\beta}))$ is described in \cref{NewQuad2F}. Since $V_{\beta}/C_{V_{\beta}}(O^p(L_{\beta}))$ contains two non-central chief factors for $O^p(L_{\beta})$, we have that $L_{\beta}/R_{\beta}\cong \SL_2(p), (3\times 3):2, (3\times 3):4$ or $(Q_8\times Q_8):3$. Applying \cref{SEFF}, we have that $Z_{\alpha}\cap C_{V_{\beta}}(O^p(L_{\beta}))$ is centralized by $Q_{\beta}\cap O^p(L_{\beta})$ and has index $p^2$ in $Z_{\alpha}$ so that $L_{\alpha}/R_{\alpha}\not\cong \PSL_2(4)$ and $q_\alpha=p$.
\end{proof}

\begin{lemma}\label{SecondReduction}
Suppose that $b>1$, $C_{V_\beta}(V_{\alpha'})<V_\beta \cap Q_{\alpha'}$ and $Z_{\alpha}$ is not a natural $\SL_2(q_\alpha)$-module. If $|Z_{\alpha'+1}Q_{\beta}/Q_{\beta}|>p$ then $Q_{\beta}\in\syl_p(R_{\beta})$, $[V_{\beta}, Q_{\beta}]\le C_{V_{\beta}}(O^p(L_{\beta}))$, $L_{\beta}/R_{\beta}\cong \SL_2(q_\alpha)\cong L_{\alpha}/R_{\alpha}$, $q_\alpha>p$ and both $V_{\beta}/C_{V_{\beta}}(O^p(L_{\beta}))$ and $Z_\alpha$ are a direct sum of two natural modules.
\end{lemma}
\begin{proof}
As in the proof of \cref{FirstReduction}, we set $r_\beta:=|Z_{\alpha'+1}Q_{\beta}/Q_{\beta}|$, $r_\alpha:=|(Z_{\alpha'+1}\cap Q_{\beta})Q_{\alpha}/Q_{\alpha}|$, $r_{\alpha'}=|Z_{\alpha}Q_{\alpha'}/Q_{\alpha'}|$ and $r_{\alpha'+1}:=|(Z_{\alpha}\cap Q_{\alpha'})Q_{\alpha'+1}/Q_{\alpha'+1}|$ so that $r_\lambda\le q_\lambda$ for $\lambda\in\{\alpha, \beta, \alpha', \alpha'+1\}$. Since $r_\beta>p$, we have that $m_p(S/Q_{\beta})>1$ and $\bar{L_{\beta}}/O_{p'}(\bar{L_{\beta}})\cong \PSL_2(q_\beta), \PSU_3(q_\beta)$ or $\Sz(q_\beta)$. Furthermore, as intimated in \cref{FirstReduction}, since $r_{\alpha'}>p$, we have that $q_\alpha \geq q_\beta$.

Suppose first that $L_{\beta}/R_{\beta}\cong \SL_2(q_\beta)$ and $V_{\beta}$ contains exactly two non-central chief factors, both of which, as $\bar{L_{\beta}}$-modules, are natural $\SL_2(q_\beta)$-modules. By \cref{CommCF}, we infer that $[V_{\beta}, Q_{\beta}, Q_{\beta}]$ contains no non-central chief factors for $L_{\beta}$ and that $[Z_{\alpha}, Q_{\beta}, Q_{\beta}]\le C_{V_{\beta}}(O^p(L_{\beta}))$. 

Assume that $[V_{\beta}, Q_{\beta}]$ contains no non-central chief factors for $L_{\beta}$. Then $Z_{\alpha}\cap C_{V_{\beta}}(O^p(L_{\beta}))$ has index $q_\beta^2$ in $Z_{\alpha}$ and is centralized by $Q_{\beta}\cap O^p(L_{\beta})$. Since $q_\alpha\geq q_\beta$, using \cref{2FRecog} and observing that $[Z_{\alpha}, Q_{\beta}, Q_{\beta}\cap O^p(L_{\beta})]\ne\{1\}$ when $Z_{\alpha}$ is a natural $\Sz(q_\alpha)$-module, we may assume that $Z_{\alpha}$ is a natural $\SU_3(q_\alpha)$-module and $q_\alpha=q_\beta$. But then, $[Z_{\alpha}\cap Q_{\alpha'}, Z_{\alpha'+1}\cap Q_{\beta}]$ has order $q_\alpha^2$ and is contained in $Z_{\beta}\cap Z_{\alpha'}$ from which it follows that $Z_{\beta}=Z_{\alpha'}$. But then $[V_{\beta}\cap Q_{\alpha'}, V_{\alpha'}]=Z_{\alpha'}=Z_{\beta}$, a contradiction by \cref{SEFF} and \cref{NotNatural}.

Thus, $[V_{\beta}, Q_{\beta}]$ contains a non-central chief factor for $L_{\beta}$. Note that $[Z_{\alpha}, V_{\alpha'}\cap Q_{\beta}]Z_{\beta}$ is normalized by $L_{\beta}=\langle V_{\alpha'}, Q_{\alpha}, R_{\beta}\rangle$ and we deduce that $[Z_{\alpha}, V_{\alpha'}\cap Q_{\beta}]Z_{\beta}\le C_{V_{\beta}}(O^p(L_{\beta}))$. Indeed, we ascertain that $[Z_{\alpha}, V_{\alpha'}\cap Q_{\beta}, Q_{\beta}\cap O^p(L_{\beta})]=\{1\}$ and if $\bar{L_{\alpha}}/O_{p'}(\bar{L_{\alpha}})\cong \PSL_2(q_\alpha)$, then taking the closure of this commutator under the action of $G_{\alpha,\beta}$ yields $[Z_{\alpha}, S, S]=\{1\}$ and $[Z_{\alpha}, Q_{\beta}]\le Z_{\beta}$, a contradiction. Hence, $\bar{L_{\alpha}}/O_{p'}(\bar{L_{\alpha}})\not\cong \PSL_2(q_\alpha)$.

Without loss of generality, we may assume that $r_{\alpha}\leq r_{\alpha'+1}$ and an index $q_\beta r_{\alpha}$ subgroup of $Z_{\alpha'+1}$ is centralized by $Z_{\alpha}\cap Q_{\alpha'}$. Applying \cref{SpeMod2} and \cref{SpeModOdd}, we infer that $r_{\alpha}=r_{\alpha'+1}=p$ and $q_\alpha=q_\beta=r_\beta=r_{\alpha'}$; or $q_\alpha=q_\beta=r_\alpha=r_\beta=r_{\alpha'}=r_{\alpha'+1}$. In the former case, we have that an index $pq_\beta<q_\alpha^2$ subgroup of $Z_{\alpha}$ is centralized by $X:=\langle Z_{\alpha'+1}, Q_{\alpha}, O^p(R_{\beta})\rangle$ and $|(X\cap Q_{\beta})Q_{\alpha}/Q_{\alpha}|>p$. Then applying \cref{SpeMod2} and \cref{SpeModOdd} yields a contradiction.

In the latter case, we have by \cref{2FRecog}, $Z_{\alpha}$ is a natural module for $L_{\alpha}/R_{\alpha}\cong \SU_3(q_\alpha)$ or $\Sz(q_\alpha)$. If $Z_{\alpha}$ is a natural $\SU_3(q_\alpha)$-module then $|Z_{\alpha}/[Z_{\alpha}, Q_{\beta}]|=q_{\alpha}^2$ and since $V_{\beta}/[V_{\beta}, Q_{\beta}]$ contains a unique non-central chief factor which is a natural $\SL_2(q_\beta)$-module, we have a contradiction. Hence, $Z_{\alpha}$ is a natural $\Sz(q_\alpha)$-module. Now, as $[Z_{\alpha}, Q_{\beta}]=Z_{\alpha}\cap Q_{\alpha'}$ and $[Z_{\alpha'+1}, Q_{\alpha'}]=Z_{\alpha'+1}\cap Q_{\beta}$, we have that $[Z_{\alpha}\cap Q_{\alpha'}, Z_{\alpha'+1}\cap Q_{\beta}]=Z_{\alpha'}=Z_{\beta}$. But then $[V_{\beta}\cap Q_{\alpha'}, V_{\alpha'}]=Z_{\alpha'}=Z_{\beta}$, a contradiction by \cref{SEFF} and \cref{NotNatural}. Hence, we assume for the rest of the proof that $V_{\beta}$ either contains more than two non-central chief factors for $L_{\beta}$, or that if $V_{\beta}$ contains only two non-central chief factors for $L_{\beta}$ then at least one of them, as  $\bar{L_{\beta}}$-module, is not an FF-module.

Note that if $r_\beta>q_\beta^{\frac{1}{2}}$ then as a subgroup of $V_{\beta}$ of index $q_\alpha q_\beta$ is centralized by $Z_{\alpha'+1}$ and $V_{\beta}$ contains at least two non-central chief factors, applying \cref{SpeMod2} and \cref{SpeModOdd}, a subgroup of a non-central $L_{\beta}$-chief factor of index at most $q_\alpha$ is centralized by $Z_{\alpha'+1}$. Applying \cref{2FRecog}, we infer that $r_\beta^2\leq q_\alpha$ with equality yielding $|(V_{\beta}\cap Q_{\alpha'})Q_{\alpha'+1}|=q_\alpha$. Likewise, if $r_\beta\leq q_\beta^{\frac{1}{2}}$ then as $q_\beta\leq q_\alpha$, $r_\beta^2\leq q_\alpha$. Hence, $p^4\leq r_\beta^2\leq q_\alpha$. In a similar manner, we also have that $r_{\alpha'}^2\leq q_{\alpha}$.

If $r_{\alpha'+1}>p$, then \cref{SpeMod2} and \cref{SpeModOdd} imply that $r_{\alpha}r_\beta\geq q_\alpha$ and $\bar{L_{\alpha'+1}}/O_{p'}(\bar{L_{\alpha'+1}})\cong \PSL_2(q_\alpha)$. Moreover, $r_\alpha>p$ and $r_{\alpha'}r_{\alpha'+1}\geq q_\alpha$. Instead assuming that $r_{\alpha}>p$ yields a similar result. In fact, by \cref{2FRecog}, using that $Z_{\alpha}$ is not a natural module, we have that $r_{\alpha'}=r_{\alpha'+1}=r_\beta=r_\alpha=q_\alpha^{\frac{1}{2}}$ or $r_\alpha>q_\alpha^{\frac{1}{2}}<r_{\alpha'+1}$. In the latter case, since $Z_{\alpha}$ is not a natural $\SL_2(q_\alpha)$-module, applying \cref{2FRecog}, \cref{SpeMod2} and \cref{SpeModOdd}, we deduce that $r_{\alpha}^2\leq r_{\alpha'}r_{\alpha'+1}$ and $r_{\alpha'+1}^2\leq q_\alpha r_\beta$. Then $r_{\alpha}^4\leq r_{\alpha'}^2r_{\alpha'+1}^2\leq r_{\alpha'}^2 r_\alpha r_\beta$ so that $r_{\alpha}^3\leq r_{\alpha'}^2 r_\beta$. But then, $q_\alpha^{\frac{3}{2}}<r_\alpha^3\leq r_{\alpha'}^2 r_\beta\leq q_\alpha^{\frac{3}{2}}$, a contradiction. Hence, $r_{\alpha'}=r_{\alpha'+1}=r_\beta=r_\alpha=q_\alpha^{\frac{1}{2}}$. Note that by \cref{SpeMod2} and \cref{SpeModOdd}, if $q_\alpha=q_\beta$, then since $V_{\alpha'}$ contains two non-central chief factors for $L_{\alpha'}$, $S=(V_{\alpha'}\cap Q_{\beta})Q_{\alpha}$ acts quadratically on $Z_\alpha$. If $q_\alpha>q_\beta$, then applying \cref{2FRecog} since $r_\beta>q_{\beta}^{\frac{1}{2}}$, again we have that $S=(V_{\alpha'}\cap Q_{\beta})Q_{\alpha}$ acts quadratically on $Z_\alpha$.

Assume that $r_\alpha=r_{\alpha'+1}=p$. Then an index $r_\beta p\leq q_\alpha^{\frac{1}{2}}p$ subgroup of $Z_{\alpha}$ is centralized by $Z_{\alpha'+1}\cap Q_{\beta}$, and we deduce by \cref{SpeMod2} and \cref{SpeModOdd} that $\bar{L_{\alpha}}/O_{p'}(\bar{L_{\alpha}})\cong \PSL_2(q_\alpha)$ and $q_{\alpha}\leq p^6$.
Note that if $r_\beta^2<q_\alpha$ then $r_\beta^2p\leq q_\alpha$ and since $q_{\alpha}^{2/3}\leq r_\beta p$ by \cref{SpeMod2} and \cref{SpeModOdd}, we have that $q_{\alpha}^{2/3}r_\beta\leq r_\beta^2p\leq q_\alpha$ so that $r_\beta\leq q_{\alpha}^{1/3}$ and $q_\alpha\leq p^3$. But $r_\beta^2\geq p^4>q_\alpha$, a contradiction. Hence, $r_\beta^2=q_\alpha$. If $q_\beta<q_\alpha$ then $r_\beta>q_\beta^{\frac{1}{2}}$ and by an earlier observation, $S=(V_{\alpha'}\cap Q_{\beta})Q_{\alpha}$ acts quadratically on $Z_\alpha$. If $q_\beta=q_\alpha$ then $r_{\beta}>p$ and $V_{\beta}$ contains two non-central chief factors for $L_{\beta}$, applying \cref{SpeMod2} and \cref{SpeModOdd}, $(V_{\beta}\cap Q_{\alpha'})Q_{\alpha'+1}\in\syl_p(L_{\alpha'+1})$ and conjugating to $\alpha$, $S$ acts quadratically on $Z_{\alpha}$.

In all cases, $[Z_{\alpha}, Q_{\beta}]\leq [Z_{\alpha}, S]\le Z_{\beta}$ and $\bar{L_{\alpha}}/O_{p'}(\bar{L_{\alpha}})\cong \PSL_2(q_\alpha)$. Furthermore, forming $L^{\alpha}$ from $Q_{\alpha}$ and $d$ conjugates of $ Z_{\alpha'+1}\cap Q_{\beta}$ where $d=2$ when $r_{\alpha}>p$ and $d=3$ when $r_\alpha=p$, we have that $\bar{L^{\alpha}}/O_{p'}(\bar{L^{\alpha}})\cong \bar{L_{\alpha}}/O_{p'}(\bar{L_{\alpha}})$, $S\in\syl_p(L^\alpha)$ and $|Z_{\alpha}/C_{Z_{\alpha}}(L^\alpha)|\leq q_\alpha^2$ when $d=2$ and $r_\beta^3p^3$ when $d=3$. Since $O^p(L^{\alpha})$ acts non-trivially on $Z_{\alpha}$, $C_{Z_{\alpha}}(L^\alpha)<C_{Z_{\alpha}}(L^\alpha)[Z_{\alpha}, Q_{\beta}]\le Z_{\beta}$ and applying \cref{SEFF}, we have that $|[Z_{\alpha}, Q_{\beta}]C_{Z_{\alpha}}(L^\alpha)/C_{Z_{\alpha}}(L^\alpha)|\geq q_\alpha$. But then $Z_{\beta}$ has index at most $q_\alpha$ when $d=2$ and $r_\beta^3p^3/q_\alpha<q_\alpha^2$ and by \cref{2FRecog}, we deduce that $Z_{\alpha}$ is a natural $\SL_2(q_\alpha)$-module for $L_{\alpha}/R_{\alpha}\cong \SL_2(q_\alpha)$, a contradiction.
\end{proof}

\begin{lemma}\label{IdenNat}
Suppose that $b>1$ and $C_{V_\beta}(V_{\alpha'})<V_\beta \cap Q_{\alpha'}$. Then $Z_{\alpha'+1}\not\le Q_{\beta}$, $Q_{\beta}\in\syl_p(R_{\beta})$, $[V_{\beta}, Q_{\beta}]\le C_{V_{\beta}}(O^p(L_{\beta}))$ and either:
\begin{enumerate}
\item $Z_{\alpha}$ is a natural module for $L_{\alpha}/R_{\alpha}\cong \SL_2(q_\alpha)$ and $q_\alpha=q_\beta$;
\item $L_{\beta}/R_{\beta}\cong \SL_2(q_\alpha)\cong L_{\alpha}/R_{\alpha}$, $q_\alpha>p$ and both $V_{\beta}/C_{V_{\beta}}(O^p(L_{\beta}))$ and $Z_\alpha$ are a direct sum of two natural modules; or
\item $L_{\beta}/R_{\beta}\cong \SL_2(p), (3\times 3):2, (3\times 3):4$ or $(Q_8\times Q_8):3$, $V_{\beta}/C_{V_{\beta}}(O^p(L_{\beta}))$ is described in \cref{NewQuad2F} and $q_\alpha=p$.
\end{enumerate}
\end{lemma}
\begin{proof}
By \cref{FirstReduction} and \cref{SecondReduction}, we may assume that $Z_{\alpha}$ is a natural $\SL_2(q_\alpha)$-module for $L_{\alpha}/R_{\alpha}$ and so it remains to show that $q_\alpha=q_\beta$. Assume that $Z_{\alpha'+1}\cap Q_{\beta}\not\le Z_{\alpha'}$. Then $Z_{\alpha'}\le [Z_{\alpha'+1}\cap Q_{\beta}, V_{\beta}\cap Q_{\alpha'}]\le [V_{\beta}, Q_{\beta}]=Z_{\beta}$. But then $[V_{\alpha'}\cap Q_{\beta}, V_{\beta}]\le Z_{\alpha'}$ and we force a contradiction by \cref{NotNatural}. More generally, we deduce that $Z_{\beta}\ne Z_{\alpha'}$ in all cases. 

Hence, $Z_{\alpha'+1}\cap Q_{\beta}=Z_{\alpha'}$ so that $|Z_{\alpha'+1}Q_{\beta}/Q_{\beta}|=q_{\alpha}\leq q_\beta$. Hence, to complete the proof we assume that $q_\alpha<q_\beta$. Since $Z_{\alpha}$ centralizes a subgroup of index at most $q_\beta q_\alpha$ in $V_{\alpha'}$, we deduce by \cref{SpeMod2} and \cref{SpeModOdd} that either $\bar{L_{\alpha'}}/O_{p'}(\bar{L_{\alpha'}})\cong \PSL_2(q_{\alpha'})$, or $q_\alpha=p$ and $\bar{L_{\alpha'}}/O_{p'}(\bar{L_{\alpha'}})\cong \PSU_3(p^2)$ or $\Sz(8)$. Note that $|V_{\alpha'}Q_{\beta}/Q_{\beta}|\geq |Z_{\alpha'+1}Q_{\beta}/Q_{\beta}|=q_\alpha$ and so either $|V_{\alpha'}Q_{\beta}/Q_{\beta}|\geq p^2$ or, since $q_\beta>q_\alpha$ and applying \cref{NewQuad2F}, $L_{\alpha'}/R_{\alpha'}\cong \PSL_2(4)$ and $V_{\alpha'}/C_{V_{\alpha'}}(O^p(L_{\alpha'}))$ is isomorphic to a natural $\Omega_4^-(2)$-module.

Suppose that $b\geq 7$. If $|V_{\alpha'}Q_{\beta}/Q_{\beta}|=p$, then $V_{\lambda}/C_{V_{\lambda}}(O^2(L_{\lambda}))$ is isomorphic to a natural $\Omega_4^-(2)$-module for all $\lambda\in\beta^G$, and $q_\beta=4$. Since $V_{\beta}^{(3)}$ is elementary abelian, $V_{\beta}^{(3)}$ acts quadratically on $V_{\alpha'-2}$ and $V_{\beta}^{(3)}\cap Q_{\alpha'-2}\cap Q_{\alpha'-1}$ acts quadratically on $V_{\alpha'}$ so that $V_{\beta}(V_{\beta}^{(3)}\cap \dots Q_{\alpha'})$ has index at most $4$ in $V_{\beta}^{(3)}$. Since $q_\beta=4$, we conclude that $V_{\beta}^{(3)}\not\le Q_{\alpha'-2}$ and $V_{\beta}\cap Q_{\alpha'-2}\not\le Q_{\alpha'-1}$. Since $Z_{\alpha'}\le V_{\beta}$ is centralized by $V_{\beta}^{(3)}$, we must have that $Z_{\alpha'}=Z_{\alpha'-2}$. Furthermore, there is $\beta-3\in\Delta^{(3)}(\beta)$ with $(\beta-3, \alpha'-2)$ a critical pair. By \cref{SL2VlQ} and \cref{SL2VnQ}, $(\beta-3, \alpha'-2)$ satisfies the same conditions as $(\alpha, \alpha')$, and applying the same reasoning to this critical pair, and iterating backwards further through critical pairs, we deduce that $Z_{\alpha'}=Z_{\alpha'-2}=\dots=Z_{\beta}$, a contradiction. Hence, if $b \geq 7$, then $|V_{\alpha'}Q_{\beta}/Q_{\beta}|\geq p^2$. By a similar argument, and using the symmetry in the critical pairs $(\alpha, \alpha')$ and $(\alpha'+1, \beta)$, we deduce that $|V_{\beta}Q_{\alpha'}/Q_{\alpha'}|\geq p^2$. Furthermore, if $\bar{L_{\beta}}/O_{p'}(\bar{L_{\beta}})\not\cong \PSL_2(q_{\alpha'})$, then $V_{\beta}(V_{\beta}^{(3)}\cap Q_{\alpha'})$ has index at most $q_\beta q_\alpha<q_\beta^2$ in $V_{\beta}^{(3)}$ and applying \cref{SpeMod2}, we have a contradiction. Hence, $\bar{L_{\beta}}/O_{p'}(\bar{L_{\beta}})\cong \PSL_2(q_{\alpha'})$. 

If $V_{\beta}^{(3)}\le Q_{\alpha'-2}$ then applying \cref{SpeMod2} and \cref{SpeModOdd} via the action of $V_{\alpha'}$ on $V_{\beta}^{(3)}/V_{\beta}$, using that $Z_{\alpha'}\le V_{\beta}$, we must have that $q_\alpha q_\beta/|V_{\beta}Q_{\alpha'}/Q_{\alpha'}|\geq q_\beta$. But then, $q_\alpha\leq |V_{\beta}Q_{\alpha'}/Q_{\alpha'}|\leq q_\alpha$ so that $q_\alpha\geq p^2$. Applying \cref{SpeMod2} and \cref{SpeModOdd} via the action of $Z_{\alpha'+1}$ on $V_{\beta}$, we have that $q_\alpha^2\geq q_\beta>q_\alpha$. Let $P_{\beta}$ be a $p$-minimal subgroup of $L_{\beta}$ containing $Q_{\alpha+2}$ such that $\bar{L_{\beta}}=\bar{P_{\beta}}O_{p'}(\bar{L_{\beta}})$. Set $V^P:=\langle Z_{\alpha}^{P_{\beta}}\rangle/Z_{\beta}$ so that $P_\beta$ non-trivially on $V^P$. Still, $Z_{\alpha'+1}$ is a quadratic $2F$-offender on the module $V^P/C_{V^p}(P_\beta)$ and applying \cite[Theorem 1]{ChermakSmall}, using that $|Z_{\alpha}C_{V^p}(P_\beta)/C_{V^p}(P_\beta)|=q_\alpha<q_\beta$, we have that $V^P/C_{V^p}(P_\beta)$ is a natural $\Omega_4^-(q_\alpha)$-module for $\bar{P_{\beta}}\cong \PSL_2(q_\beta)$ and $q_\beta=q_\alpha^2$. By conjugacy, and using that $V_{\beta}^{(3)}\cap Q_{\alpha'-1}$ acts quadratically on $V_{\alpha'}$, we have that $V_{\beta}(V_{\beta}^{(3)}\cap Q_{\alpha'-1})$ has index $q_\alpha$ in $V_{\beta}^{(3)}$ and is centralized, modulo $V_{\beta}$, by $V_{\alpha'}$, a contradiction by \cref{SpeMod2} and \cref{SpeModOdd}. Thus, $V_{\beta}^{(3)}\not\le Q_{\alpha'-2}$.

Note that $(\beta-3, \alpha'-2)$ is also a critical pair which satisfies the same conditions as the critical pair $(\alpha, \alpha')$. Indeed, iterating backward through critical pair, upon relabeling we may assume that $Z_{\alpha'}\ne Z_{\alpha'-2}$, for otherwise $Z_{\alpha'}=\dots=Z_{\beta}$, yielding a contradiction. Since $Z_{\alpha'}\le V_{\beta}$, $V_{\beta}^{(3)}$ centralizes $Z_{\alpha'-1}$.

Assume that $L_{\beta}/R_{\beta}\not\cong \SL_2(q_\beta)$ so that by \cref{SEQuad} we have that $p=2$. Note that, for $Z_{\alpha'}<A\le Z_{\alpha'+1}$, $[V_{\beta}\cap Q_{\alpha'}, A]=Z_{\alpha'}\ne Z_{\beta}$ so that for $V:=V_{\beta}/Z_{\beta}$ and $C_A$ the preimage in $V_{\beta}$ of $C_V(A)$, $C_A\cap Q_{\alpha'}\le Q_{\alpha'+1}$. If there is $Z_{\alpha'}<B\le Z_{\alpha'+1}$ with $C_V(A)=C_V(B)$ for all $Z_{\alpha'}<A$ with $[B:A]=p$, then applying coprime action, we have that $O_{p'}(\bar{L_{\beta}})$ normalizes $C_V(B)$. Then writing $H:=\langle (BQ_{\beta})/Q_{\beta})^{O_{p'}(\bar{L_{\beta}})}\rangle$, we have that  $[H, C_V(B)]=\{1\}$. Moreover, $V=[V, O_{p'}(H) ]\times C_{V}(O_{p'}(H)$ is a $BQ_{\beta}/Q_{\beta}$-invariant decomposition and we conclude that $O_{p'}(H)$ centralizes $V$. Thus, $O_{p'}(L_{\beta}/R_{\beta})$ normalizes $BQ_{\beta}/Q_{\beta}$ and $L_{\beta}/R_{\beta}$ is a central extension of $\PSL_2(q_\beta)$, a contradiction since $L_{\beta}/R_{\beta}\not\cong \SL_2(q_\beta)$ and $p=2$.

Hence, we deduce that either $q_\alpha=p$ or for every $Z_{\alpha'}<A<B\le Z_{\alpha'+1}$ with $|A/Z_{\alpha'}|=p=|B/A|$ we have that $C_V(A)$ has index at most $r_{\alpha'}p$ in $V$, where $r_{\alpha'}=|V_{\beta}Q_{\alpha'}/Q_{\alpha'}|$, and $C_V(B)$ has index at most $r_{\alpha'}p^2$ in $V$. In the former case, we select $\lambda\in\Delta(\alpha')$ with $Z_\lambda\not\le Q_{\beta}$ and $|Z_{\alpha'+1}Z_{\lambda}Q_{\beta}/Q_{\beta}|=p^2$. Then $Z_{\alpha'+1}Z_\lambda$ centralizes a subgroup of $V_{\beta}$ of index $r_{\alpha'}p^2$. In either case, applying \cref{SpeMod2}, we deduce that $r_{\alpha'}p^2\geq q_\beta$. Form $X:=\langle B, B^x, Q_{\beta}\rangle$ for such a $B$ and some appropriately chosen $x\in L_{\beta}$ with $\bar{X}O_{2'}(\bar{L_{\beta}})=\bar{L_{\beta}}$. Then $|V/C_V(X)|\leq r_{\alpha'}^2 2^4$ and applying \cref{q^3module}, we deduce that either $XR_{\beta}/R_{\beta}\cong \PSL_2(q_\beta)$ and the unique non-central $X$-chief factor within $V$ is determined by \cref{SL2ModRecog}, or $q_\beta=r_{\alpha'}\leq p^3$ and $q_\alpha\leq p^2$. In the first case, since $V_{\beta}^{(3)}$ acts quadratically on $V_{\alpha'-2}$, $V_{\beta}^{(3)}\cap Q_{\alpha'-2}$ acts quadratically on $V_{\alpha'}$, and $V_{\beta}^{(3)}$ contains a non-central chief factor for $L_{\beta}$, we deduce that the unique non-central $X$-chief factor within $V$ is a natural $\SL_2(q_{\beta})$-module. Then for $T\in\syl_2(X)$, we have that $[V, t]C_V(X)=[V, y]C_V(X)$ for all $x,y\in T\setminus Q_{\beta}$ and we conclude that $C_V(A)=C_V(B)$ where $A$ and $B$ are subgroups of $Z_{\alpha'+1}$ defined previously, a contradiction. Hence, $q_\beta=r_{\alpha'}\leq p^3$ and $q_\alpha\leq p^2$. If $q_\alpha=p^2$ then $q_\beta=p^3$ and $V_{\alpha'}=Z_{\alpha'+1}Z_\lambda(V_{\alpha'}\cap Q_{\beta})$ for some $\lambda\in\Delta(\alpha')$ with $Z_\lambda\cap Z_{\alpha'+1}=Z_{\alpha'}$. But then $|Z_{\alpha'+1}Z_{\lambda}|=p^6$ so that $|Z_{\alpha'+1}Z_{\lambda}\cap Q_{\alpha'}|=p^3$. Finally, $|Z_{\alpha'+1}(Z_{\alpha'+1}Z_{\lambda}\cap Q_{\alpha'})|=p^5$ from which it follows that $|Z_{\alpha'+1}\cap Q_{\alpha'}|=p^3$, a contradiction. Hence, $q_{\alpha}=p$ so that $V_{\alpha'}$ centralizes a subgroup of index at most $q_\beta^2$ in $V_{\beta}/Z_{\beta}$. Indeed, applying \cref{2FRecog} we may assume that $V_{\alpha'}$ centralizes subgroup of index exactly $q_\beta^2$ in $V_{\beta}/Z_{\beta}$. Let $P_{\beta}$ be a $2$-minimal subgroup of $L_{\beta}$ containing $Q_{\alpha+2}$ such that $\bar{L_{\beta}}=\bar{P_{\beta}}O_{2'}(\bar{L_{\beta}})$. Set $V^P:=\langle Z_{\alpha}^{P_{\beta}}\rangle/Z_{\beta}$ so that $P_\beta$ non-trivially on $V^P$. Still, $V_{\alpha'}$ is a quadratic $2F$-offender on the module $V^P/C_{V^p}(P_\beta)$ and applying \cite[Theorem 1]{ChermakSmall}, using that $|Z_{\alpha}C_{V^p}(P_\beta)/C_{V^p}(P_\beta)|=p<q_\beta$ and a Sylow $2$-subgroup of $P_\beta/Q_\beta$ acts quadratically on $V^P$, we have a contradiction.

Thus, $L_{\beta}/R_{\beta}\cong \SL_2(q_\beta)$. Assume that $r_{\alpha'}^2 q_\alpha^2\leq q_\beta^3$ so that $V_{\beta}/C_{V_{\beta}}(O^p(L_{\beta}))$ is determined by \cref{SL2ModRecog}. Since $V_{\beta}$ acts quadratically on $V_{\alpha'}$, by \cref{NotNatural}, it follows that $r_{\alpha'}\leq q_{\beta}^{\frac{1}{2}}$. But $q_\alpha\leq r_{\alpha'}$ from which it follows that $q_\alpha=r_{\alpha'}=q_{\beta}^{\frac{1}{2}}$ and $V_{\beta}/C_{V_{\beta}}(O^2(L_{\beta}))$ is a natural $\Omega_4(q_\alpha)$-module. But then, since $V_{\beta}^{(3)}$ acts quadratically on $V_{\alpha'-2}$ and $V_{\beta}^{(3)}\cap Q_{\alpha'-2}$ is quadratic on $V_{\alpha'}$, we have that $V_{\beta}(V_{\beta}^{(3)}\cap Q_{\alpha'})$ has index $q_\alpha$ in $V_{\beta}^{(3)}$ and is centralized, modulo $V_{\beta}$, by $V_{\alpha'}$ and \cref{SpeMod2} yields a contradiction. Hence, $r_{\alpha'}^2 q_{\alpha}^2>q_\beta^3$. In particular, $r_{\alpha'}>q_{\beta}^{\frac{1}{2}}$ and writing $P_\beta=\langle V_{\alpha'}, V_{\alpha'}^g, C_{L_{\beta}}(V_{\beta}^{(3)}/V_\beta), Q_{\beta}\rangle$, we have that $P_\beta$ has an $\SL_2(q_\beta)$ quotient and centralizes an subgroup of $V_{\beta}^{(3)}/[V_{\beta}^{(3)}, Q_{\beta}]V_{\beta}$ of index strictly less than $q_\beta^3$. Then \cref{q^3module} implies that $P_\beta/C_{P_\beta}(V_{\beta}^{(3)}Q_{\beta}\cong \SL_2(q_\beta)$ and \cref{SL2ModRecog} yields that $V_{\beta}^{(3)}/V_{\beta}$ contains a unique non-central chief factor which, as a $\bar{P_{\beta}}$-module, is a natural $\SL_2(q_\beta)$-module. It follows that all elements of order $p$ in $S/Q_{\beta}$ has the same centralizer on $V_{\beta}^{(3)}/[V_{\beta}^{(3)}, Q_{\beta}]V_{\beta}$ and applying coprime action and arguing as above, we get that $L_{\beta}/C_{L_{\beta}}(V_{\beta}^{(3)})Q_{\beta}\cong \SL_2(q_\beta)$. Since $V_{\beta}^{(3)}\not\le Q_{\alpha'-2}$ yet centralizes $Z_{\alpha'-1}$, we deduce that $Z_{\alpha'-1}=Z_{\alpha'-3}$ and and application of \cref{SimExt} implies that $V_{\alpha'}\le V_{\alpha'-1}^{(2)}=V_{\alpha'-3}^{(2)}\le Q_{\beta}$, a contradiction.

Suppose now that $b=5$. Since $Z_{\alpha'}\ne Z_{\beta}$, without loss of generality, we have that $Z_{\beta}\ne Z_{\alpha'-2}$. But $Z_{\beta}=[V_{\beta}, V_{\alpha'}\cap Q_{\beta}]\le [V_{\alpha'-2}^{(3)}, V_{\alpha'-2}^{(3)}]$ from which we deduce that $Z_{\alpha+2}\le [V_{\alpha'-2}^{(3)}, V_{\alpha'-2}^{(3)}]$ and ultimately $V_{\alpha'-2}\le [V_{\alpha'-2}^{(3)}, V_{\alpha'-2}^{(3)}]\le \Phi(V_{\alpha'-2}^{(3)})$. By conjugation, $V_{\beta}\le [V_{\beta}^{(3)}, V_{\beta}^{(3)}]$. If $[V_{\beta}^{(3)}, Q_{\beta}]/V_{\beta}$ does not contain a non-central chief factor for $L_{\beta}$, then $[V_{\beta}^{(3)}, Q_{\beta}]=[V_{\lambda}^{(2)}, Q_{\beta}]V_{\beta}$ for all $\lambda\in\Delta(\beta)$ and as $V_{\lambda}^{(2)}$ is elementary abelian, $[V_{\beta}^{(3)}, Q_{\beta}]\le Z(V_{\beta}^{(3)})$. But then, by the three subgroups lemma, $[V_{\beta}^{(3)}, V_{\beta}^{(3)}, Q_{\beta}]=\{1\}$, a contradiction since $[V_{\beta}, Q_{\beta}]\ne \{1\}$. Hence, by \cref{CommCF}, $V_{\beta}^{(3)}/V_{\beta}$ contains at least two non-central chief factors for $L_{\beta}$. Note that if $|V_{\beta}Q_{\alpha'}/Q_{\alpha'}|=p$, then $V_{\lambda}/C_{V_{\lambda}}(O^2(L_{\lambda}))$ is isomorphic to a natural $\Omega_4^-(2)$-module for all $\lambda\in\beta^G$. Since $V_{\alpha'}^{(3)}$ is quadratic on $V_{\alpha'-2}$, we get that $V_{\alpha'}(V_{\alpha'}^{(3)}\cap Q_{\beta})$ has index at most $8$ in $V_{\alpha'}^{(3)}$ and is centralized, modulo $V_{\alpha'}$, by $V_{\beta}$. Applying \cref{SEFF} since $q_\beta=4$ and $V_{\beta}^{(3)}/V_{\beta}$ contains at least two non-central chief factors for $L_{\beta}$, this is a contradiction. Hence, $|V_{\beta}Q_{\alpha'}/Q_{\alpha'}|\geq p^2$ and by a similar argument, $|V_{\alpha'}Q_{\beta}/Q_{\beta}|\geq p^2$. Using that $V_{\alpha'}^{(3)}\cap Q_{\alpha'-2}\le Q_{\alpha+2}$, we deduce that $V_{\beta}$ centralizes a subgroup fo $V_{\alpha'}^{(3)}/V_{\alpha'}$ of index strictly less than $q_\beta^2$ and applying \cref{SpeMod2} and \cref{SpeModOdd}, and using that $V_{\alpha'}^{(3)}$ contains at least two non-central chief factors, we have a contradiction. 

Suppose that $|V_{\alpha'}Q_{\beta}/Q_{\beta}|\geq p^2$ and $b=3$. If, in addition, $b=3$ and $\bar{L_{\beta}}/O_{p'}(\bar{L_{\beta}})\not\cong \PSL_2(q_\beta)$ then $q_\alpha=p$ and since both $Q_{\beta}/C_{\beta}$ and $V_{\beta}$ contain a non-central chief factor for $L_{\beta}$, applying \cref{SpeMod2} and \cref{SpeModOdd}, we have that $\bar{L_{\beta}}/O_{p'}(\bar{L_{\beta}})\cong \PSU_3(p^2)$, $O^p(L_{\beta})$ centralizes $C_{\beta}/V_{\beta}$ and both $Q_{\beta}/C_{\beta}$ and $V_{\beta}$ contain a unique non-central chief factor for $L_{\beta}$. If $p$ is odd, then as $V_{\beta}$ is a quadratic module, we deduce that $L_{\beta}/R_{\beta}\cong \SU_3(p^2)$ and $V_{\beta}/C_{V_{\beta}}(O^p(L_{\beta}))$ is a natural module. But then, since $q_\alpha=p$, $Z_{\alpha}C_{V_{\beta}}(O^p(L_{\beta}))$ is a $G_{\alpha,\beta}$-invariant subgroup of order $p$, a contradiction by \cref{SUMod}. Hence, $p=2$. In particular, all involutions in $V_{\alpha'}Q_{\beta}/Q_{\beta}$ are conjugate and so have the same size centralizer of $V_{\beta}/Z_{\beta}$. By \cref{SpeMod2}, we have that $(V_{\beta}\cap Q_{\alpha'}\cap Q_{\alpha'+1})/Z_{\beta}=C_{V_{\beta}}(Z_{\alpha'+1})/Z_{\beta}=C_{V_{\beta}/Z_{\beta}}(Z_{\alpha'+1})$ has index $8$ in $V_{\beta}/Z_{\beta}$. Indeed, for some other $\lambda\in\Delta(\alpha')$ such that $V_{\alpha'}=Z_{\alpha'+1}Z_{\lambda}(V_{\alpha'}\cap Q_{\beta})$, we have that $C_{V_{\beta}/Z_{\beta}}(V_{\alpha'})=(V_{\beta}\cap Q_{\alpha'}\cap Q_{\alpha'+1}\cap Q_{\lambda})/Z_{\beta}$ which has index at most $16$ in $V_{\beta}/Z_{\beta}$ and applying \cref{2FRecog}, we conclude that $V_{\beta}/C_{V_{\beta}}(O^2(L_{\beta}))$ is a natural $\SU_3(4)$-module for $L_{\beta}/R_{\beta}$. Then using the $G_{\alpha, \beta}$-invariance of $Z_{\alpha}$ along with \cref{SUMod} using that $q_\alpha=2$ gives a contradiction.

If $\bar{L_{\beta}}/O_{p'}(\bar{L_{\beta}})\cong \PSL_2(q_\beta)$ and $q_\alpha=p$, then since $|V_{\alpha'}Q_{\beta}/Q_{\beta}|\geq p^2$ and $V_{\alpha'}$ centralizes $(Q_{\beta}\cap Q_{\alpha'})/V_{\beta}$, applying \cref{SpeMod2} and \cref{SpeModOdd}, we must have that $|V_{\beta}Q_{\alpha'}/Q_{\alpha'}|=p$. As above, using \cref{NewQuad2F} since $Z_{\alpha'+1}$ centralizes an index $p^2$ subgroup of $V_{\beta}$, this implies that $L_{\beta}/R_{\beta}\cong \PSL_2(4)$ and $V_{\beta}/C_{V_{\beta}}(O^p(L_{\beta}))$ is isomorphic to a natural $\Omega_4^-(2)$-module. But $V_{\alpha'}Q_{\beta}\in\syl_2(L_{\beta})$ and $V_{\alpha'}$ acts quadratically on $V_{\beta}$, a contradiction.

Finally, if $\bar{L_{\beta}}/O_{p'}(\bar{L_{\beta}})\cong \PSL_2(q_\beta)$ and $q_\alpha>p$, then $Z_{\alpha'+1}$ centralizes a subgroup of $Q_{\beta}$ of index at most $q_\alpha^2 q_\beta$ and as $Q_{\beta}/C_{\beta}$ and $V_{\beta}$ contain non-central chief factors for $L_{\beta}$, we deduce that $q_\alpha^2\geq q_\beta$ and $O^p(L_{\beta})$ centralizes $C_{\beta}/V_{\beta}$. Furthermore, for any $r\in L_{\beta}$ of order coprime to $p$, if $[r, Q_{\beta}, V_{\beta}]=\{1\}$, then $[r, V_{\beta}]\le Z(Q_{\beta})$ by the three subgroups lemma, from which it follows that $[r, Q_{\beta}]=\{1\}$, $r=1$ and $C_{L_{\beta}}(Q_{\beta}/C_{\beta})\le Q_{\beta}$. Similarly, if $[r, V_{\beta}]=\{1\}$ then $[r, Q_{\beta}, V_{\beta}]=\{1\}$, $[r, Q_{\beta}]=\{1\}$ and $R_{\beta}=Q_{\beta}$. Assume that $Z_{\alpha'+1}$ centralizes a subgroup of index strictly less than $q_\alpha^2$ in $Q_{\beta}/C_{\beta}$ or $V_{\beta}$. Applying \cref{2FRecog}, we infer that $\bar{L_{\beta}}\cong \SL_2(q_\beta)$ and applying \cref{SpeMod2} and \cref{SpeModOdd}, we have that both $Q_{\beta}/C_{\beta}$ and $V_{\beta}$ contain a unique non-central chief factor for $\bar{L_{\beta}}$. Hence, $V_{\beta}/C_{V_{\beta}}(O^p(L_{\beta}))$ is irreducible for $\bar{L_{\beta}}$. Since $Z_{\alpha}C_{V_{\beta}}(O^p(L_{\beta}))/C_{V_{\beta}}(O^p(L_{\beta})$ is $G_{\alpha,\beta}$-invariant, centralized by $S$ and of order $q_\alpha$, by Smith's theorem \cite[Theorem 2.8.11]{GLS3}, it follows that $Z_{\alpha}C_{V_{\beta}}(O^p(L_{\beta}))/C_{V_{\beta}}(O^p(L_{\beta}))=C_{V_{\beta}/C_{V_{\beta}}(O^p(L_{\beta}))}(S)$ and $\GF(q_\alpha)$ is the largest field extension $K$ of $\GF(p)$ for which $V_{\beta}/C_{V_{\beta}}(O^p(L_{\beta}))$ may be written as a $K\bar{L_{\beta}}$-module. Indeed, we deduce that $q_\alpha^2=q_\beta$ and $|V_\beta/C_{V_{\beta}}(O^p(L_{\beta}))|<q_\beta^2 q_{\alpha}^2\leq q_\beta^3$. Since $Z_{\alpha'+1}$ centralizes a subgroup of index strictly less than $q_\alpha^2$ in $Q_{\beta}/C_{\beta}$, we have a contradiction by \cref{SpeMod2} and \cref{SpeModOdd}.

Hence, $C_{\beta}(Q_{\beta}\cap \dots \cap Q_{\alpha'+1})$ has index at least $q_\alpha^2$ in $Q_{\beta}$ and $V_{\beta}\cap \dots \cap Q_{\alpha'+1}$ has index at least $q_\alpha^2$ in $V_{\beta}$. Since $Z_{\alpha'+1}$ centralizes a subgroup of $Q_{\beta}$ of index at most $q_\alpha^2 q_\beta$ and $q_\alpha^2\geq q_\beta$, we deduce that $q_\alpha^2=q_\beta$. Indeed, since $[V_{\alpha'}, Q_{\beta}\cap Q_{\alpha'}]\le V_{\beta}$, we must have that $|V_{\alpha'}Q_{\beta}/Q_{\beta}|=|V_{\beta}Q_{\alpha'}/Q_{\alpha'}|=q_\alpha$ and $(V_{\beta}\cap Q_{\alpha'})Q_{\alpha'+1}\in\syl_p(L_{\alpha'+1})$. Assume that $q_\alpha=p^2$ and $C_{V_{\beta}/Z_{\beta}}(x)>(V_{\beta}\cap Q_{\alpha'+1})/Z_{\beta}$ for all $x\in Z_{\alpha'+1}\setminus Z_{\alpha'}$. Then for $C_x$ the preimage in $V_{\beta}$ of $C_{V_{\beta}/Z_{\beta}}(x)$, we must have that $[C_x\cap Q_{\alpha'}, Z_{\alpha'}\langle x\rangle]\le Z_{\beta}\cap Z_{\alpha'}=\{1\}$ so that $C_x\cap Q_{\alpha'}\le Q_{\alpha'+1}$. If $V_{\beta}=C_x(V_{\beta}\cap Q_{\alpha'})$ then using $q_\beta=p^4$ and applying \cref{NewQuad2F}, we have a contradiction. Thus, $|C_xQ_{\alpha'}/Q_{\alpha'}|=p$ and $x$ centralizes a subgroup of index $p^3$ in $V_{\beta}/Z_{\beta}$. Let $P_{\beta}$ be a $p$-minimal subgroup of $L_{\beta}$ containing $Q_{\alpha'-1}$ such that $\bar{L_{\beta}}=\bar{P_{\beta}}O_{p'}(\bar{L_{\beta}})$. Set $V^P:=\langle Z_{\alpha}^{P_{\beta}}\rangle/Z_{\beta}$ so that $P_\beta$ non-trivially on $V^P$. Still, $Z_{\alpha'+1}$ is a $2F$-offender on the module $V^P/C_{V^p}(P_\beta)$ and applying \cite[Theorem 1]{ChermakSmall} and using that $x$ centralizes a subgroup of index at most $p^3<q_\beta$ in $V^P$, we obtain a contradiction. 

If $q_\alpha>p^2$ then by \cref{SpeMod2} and \cref{SpeModOdd} we must have that $C_{V_{\beta}/Z_{\beta}}(A)=(V_{\beta}\cap Q_{\alpha'+1})/Z_{\beta}$ for all $A\le Z_{\alpha'+1}$ such that $Z_{\alpha'}\le A$ and $|Z_{\alpha'+1}/A|=p$. Hence, for $q_\alpha\geq p^2$ and $C:=V_{\beta}\cap Q_{\alpha'}\cap Q_{\alpha'+1}$, applying \cref{GLS2p'}, we deduce that $C$ is normalized by $\bar{Z_{\alpha'+1}}O_{p'}(\bar{L_{\beta}})$. Then writing $H:=\langle (\bar{Z_{\alpha'+1}})^{O_{p'}(\bar{L_{\beta}})}\rangle$, we have that  $[H, C]\le Z_{\beta}$. Moreover, $V_{\beta}/C_{V_{\beta}}(O^p(L_{\beta}))=[V_{\beta}/C_{V_{\beta}}(O^p(L_{\beta})), O_{p'}(H)]\times C_{V_{\beta}/C_{V_{\beta}}(O^p(L_{\beta}))}(O_{p'}(H))$ is a $\bar{Z_{\alpha'+1}}O_{p'}(\bar{L_{\beta}})$-invariant decomposition and we conclude that $O_{p'}(H)$ centralizes $V_{\beta}$. Thus, $O_{p'}(L_{\beta}/R_{\beta})$ normalizes $\bar{Z_{\alpha'+1}}$ and $\bar{L_{\beta}}$ is a central extension of $\PSL_2(q_\beta)$ with a faithful quadratic module. Thus, $\bar{L_{\beta}}$ and we deduce that $|V_{\beta}/C_{V_{\beta}}(O^p(L_{\beta}))|\leq q_\beta^3$. Applying \cref{SL2ModRecog}, we have that $V_\beta/C_{V_{\beta}}(O^p(L_{\beta}))$ is a natural $\Omega_4^-(q_\alpha)$-module for $\bar{L_{\beta}}\cong \PSL_2(q_\beta)$ and $|V_{\beta}Q_{\alpha'}/Q_{\alpha'}|=q_\alpha$. Since $V_{\beta}$ admits quadratic action, we deduce that $p=2$.

Thus, we have reduced to the case where $b=3$, $p=2$, $L_{\beta}/R_{\beta}\cong \PSL_2(q_\beta)$ and $V_{\beta}/C_{V_{\beta}}(O^2(L_{\beta})$ is a natural $\Omega_4^-(q_\alpha)$-module. If $[V_{\alpha'-1}^{(2)}, Q_{\alpha'-1}]/Z_{\alpha'-1}$ does not contain a non-central chief factor for $L_{\alpha'-1}$, then $[V_{\beta}, Q_{\alpha'-1}]=[V_{\alpha'}, Q_{\alpha'-1}]\le V_{\alpha'}\cap V_{\beta}$ is centralized by $V_{\beta}$, a contradiction to the structure of $V_{\alpha'}/C_{V_{\alpha'}}(O^p(L_{\alpha'-1}))$. Hence, by conjugation and \cref{CommCF}, both $V_{\alpha}^{(2)}/[V_{\alpha}^{(2)}, Q_{\alpha}]$ and $[V_{\alpha}^{(2)}, Q_{\alpha}]/Z_{\alpha}$ contain a non-central chief factor for $L_{\alpha}$. 

If $Z_{\alpha'}\le [V_{\alpha}^{(2)}, Q_{\alpha}]$ then $Z_{\alpha'-1}\le [V_{\alpha}^{(2)}, Q_{\alpha}]$ and since $L_{\beta}/R_{\beta}\cong \PSL_2(q_\beta)$ is $2$-transitive on the set $\{Z_\lambda \mid \lambda\in\Delta(\beta)\}$, we have that $V_{\beta}=Z_{\alpha}\langle Z_{\alpha'-1}^S\rangle\le [V_{\alpha}^{(2)}, Q_{\alpha}]$, a contradiction since $[V_{\alpha}^{(2)}, Q_{\alpha}]<V_{\alpha}^{(2)}$. Thus, since $[V_{\alpha}^{(2)}, Q_{\alpha}]\le \Phi(Q_{\alpha})\le Q_{\beta}$ and $[V_{\alpha'}, Q_{\alpha'-1}]C_{V_{\alpha'}}(O^p(L_{\alpha'}))=V_{\alpha'}\cap Q_{\beta}$, we have that $[V_{\alpha'}\cap Q_{\beta}, [V_{\alpha}^{(2)}, Q_{\alpha}]\cap Q_{\alpha'-1}]\le Z_{\alpha'-1}\cap [V_{\alpha}^{(2)}, Q_{\alpha}]=Z_{\beta}$ so that $[V_{\alpha}^{(2)}, Q_{\alpha}]/Z_{\alpha}$ contains a unique non-central chief factor which, as a $\bar{L_{\alpha}}$-module, is an FF-module. Hence, $[V_{\alpha}^{(2)}, Q_{\alpha}, Q_{\alpha}]=[V_{\beta}, Q_{\alpha}, Q_{\alpha}]Z_{\alpha}$. Indeed, either $[V_{\beta}, Q_{\alpha}, Q_{\alpha}]\le Z(Q_{\alpha})\cap V_{\beta}$ or $Z_{\alpha}\le [V_{\beta}, Q_{\alpha}, Q_{\alpha}]$. In the former case, we observe that $Z(Q_{\alpha})\cap C_{V_{\beta}}(O^p(L_{\beta}))$ is normal in $G_{\beta}$ and centralized by $Q_{\alpha}\cap Q_{\beta}$. By the irreducibility under the action of $G_{\alpha,\beta}$ of $Q_\beta/(Q_{\alpha}\cap Q_{\beta})\cong S/Q_{\alpha}$, by \cref{push} we deduce that $S$ centralizes $Z(Q_{\alpha})\cap C_{V_{\beta}}(O^p(L_{\beta}))$ so that $[V_{\beta}, Q_{\alpha}, Q_{\alpha}]=Z_{\beta}$, a contradiction to the structure of a natural $\Omega_4^-(2)$-module. Hence, $Z_{\alpha}\le [V_{\beta}, Q_{\alpha}, Q_{\alpha}]$. BY \cref{A6Cohom}, $V_{\beta}/Z_{\beta}=[V_{\beta}/Z_{\beta}, O^2(L_{\beta})]\times C_{V_{\beta}/Z_{\beta}}(O^2(L_{\beta}))$ so that $Z_{\alpha}/Z_{\beta}\le [V_{\beta}/Z_{\beta}, O^2(L_{\beta}), Q_{\alpha}, Q_{\alpha}]$ and by the definition of $V_{\beta}$, we have that $C_{V_{\beta}}(O^2(L_{\beta}))=Z_{\beta}$ and $V_{\beta}/Z_{\beta}$ is irreducible. 

Since $G_{\alpha,\beta}\cap L_{\beta}$ acts irreducibly on $[V_{\beta}, Q_{\alpha}]/[V_{\beta}, Q_{\alpha}, Q_{\alpha}]$, writing $C^\alpha$ to be the preimage in $V_{\alpha}^{(2)}$ of $C_{[V_{\alpha}^{(2)}, Q_{\alpha}]/Z_{\alpha}}(O^p(L_{\alpha}))$, it follows that either $[V_{\alpha}^{(2)}, Q_{\alpha}]=[V_{\beta}, Q_{\alpha}]C^\alpha$ or $[V_{\beta}, Q_{\alpha}]\le C^\alpha$. In the former case, we have that $Q_{\beta}$ centralizes $[V_{\alpha}^{(2)}, Q_{\alpha}]/C^\alpha$, an obvious contradiction; whereas in the latter case, conjugating from $\alpha$ to $\alpha+2$, we have that $[V_{\beta}, Q_{\alpha+2}]=[V_{\alpha'}, Q_{\alpha+2}]$ is centralized by $V_{\beta}$, a contradiction to the structure of $V_{\alpha'}/C_{V_{\alpha'}}(O^p(L_{\alpha'-1}))$ by \cref{Omega4}. This completes the proof.
\end{proof}

\begin{lemma}\label{ZaStruc}
Suppose that $C_{V_\beta}(V_{\alpha'})<V_\beta\cap Q_{\alpha'}$, $q_\alpha=p$ and $Z_{\alpha}$ is not a natural module for $L_{\alpha}/R_{\alpha}\cong\SL_2(q_\alpha)$. Then the following hold:
\begin{enumerate}
\item $S=Q_{\alpha}Q_{\beta}$;
\item $|S/Q_{\alpha}|=p$;
\item $L_{\alpha}/R_{\alpha}\in\{\SL_2(p), \SU_3(2)', \Dih(10), (3\times 3):2, (Q_8\times Q_8):3, 2\cdot\Alt(5), 2^{1+4}_-.\Alt(5)\}$;
\item $R=[Q_{\alpha'}, V_{\alpha'}]\le Z_{\alpha'}$;
\item $Q_{\beta}\in\syl_p(R_{\beta})$;
\item $|Z_{\alpha}/Z_{\beta}|=p^2$; and 
\item unless $L_{\alpha}/R_{\alpha}\cong\SU_3(2)'$ and $R<Z_{\alpha'}$, we have that $R=Z_{\alpha'}$ and $Z_{\alpha}=Z_{\beta}\times Z_{\alpha-1}$ for some $\alpha-1\in\Delta(\alpha-1)$.
\end{enumerate}
\end{lemma}
\begin{proof}
Since this result holds in all the relevant cases in \hyperlink{ThmC}{Theorem C}, we may assume that $G$ is a minimal counterexample to the lemma. We assume throughout that $Z_{\alpha}$ is not a natural module for $L_{\alpha}/R_{\alpha}\cong\SL_2(q_\alpha)$ and so $q_\alpha=p$ and $Z_{\alpha}$ is a quadratic $2$F-module determined by \cref{NewQuad2F}. By \cref{IdenNat}, we have that $L_{\beta}/R_{\beta}\cong \SL_2(p), (3\times 3):2, (3\times 3):4$ or $(Q_8\times Q_8):3$ and $V_{\beta}/C_{V_{\beta}}(O^p(L_{\beta}))$ is the associated $2F$-module described in \cref{NewQuad2F}.

Suppose first that $S\ne Q_{\alpha}Q_{\beta}$ so that $L_{\beta}/R_{\beta}\cong (3\times 3):4$. Moreover, since $Z_{\alpha}$ is also a quadratic $2$F-module determined by \cref{NewQuad2F}, $L_{\alpha}/R_{\alpha}\cong (3\times 3):4, \Sz(2), \SU_3(2)'.2$ or $\SU_3(2)$. For $\mu\in\{\alpha, \beta\}$, let $O_\mu$ be the preimage in $L_{\mu}$ of $O_{2'}(\bar{L_{\mu}})$ and $L_{\mu}^*:=O_\mu Q_{\alpha}Q_{\beta}$. Then $L_{\mu}^*\normaleq L_{\mu}$ and $L_{\mu}^*$ has index $2$ in $L_{\mu}$. Set $K$ to be a Hall $2'$-subgroup of $G_{\alpha,\beta}$ and set $G_{\mu}^*:=L_{\mu}^*K$. Then $G_{\mu}^*$ has index $2$ in $G_{\mu}$, and is normal in $G_{\mu}$. Moreover, for $X=\langle G_{\alpha}^*, G_{\beta}^*\rangle$, $X$ is normalized by $G_{\alpha,\beta}$ and $G=\langle X, G_{\alpha,\beta}\rangle$. Thus, any subgroup of $S$ which is normal in $X$ is also normal in $G$ and so is trivial. Hence, any subgroup of $G_{\alpha,\beta}\cap X$ which is normal in $X$ is a $2'$-group and we can arrange that it is contained in $K\le G_{\mu}^*$, a contradiction since $G_{\mu}^*$ is of characteristic $2$. Thus, the amalgam $(G_{\alpha}^*, G_{\beta}^*, KQ_{\alpha}Q_{\beta})$ satisfies \cref{MainHyp}. Since $G_{\alpha}^*$ and $G_{\beta}^*$ are solvable, by minimality, $(G_{\alpha}^*, G_{\beta}^*, KQ_{\alpha}Q_{\beta})$ is a weak BN-pair; or $X$ is a symplectic amalgam with $|S|=2^6$. In all cases, for some $\mu\in\{\alpha, \beta\}$, we infer that $\bar{L_{\mu}^*}\cong\Sym(3)$. But then, it follows that $\bar{L_{\mu}}\cong \Sym(3)\times R$, where $R$ is a $2$-group, a contradiction since $m_p(S/Q_{\mu})=1$. Hence, $S=Q_{\alpha}Q_{\beta}$ and (i) is proved.

Suppose that $|S/Q_{\alpha}|>p$. Since $S=Q_{\alpha}Q_{\beta}$ and $Q_{\beta}\cap O^p(L_{\beta})$ centralizes $[Z_{\alpha}, Q_{\beta}]$, $L_{\alpha}/R_{\alpha}\not\cong\Sz(2)$ or $(3\times 3):4$. Set $Q_\beta^*:=\langle (Q_{\alpha}\cap Q_{\beta})^{G_{\beta}}\rangle$. Then $Q_{\beta}^*$ centralizes $[Z_{\alpha}, Q_{\beta}]$. If $S=Q_{\beta}^*Q_{\alpha}$ then $S$ centralizes $[Z_{\alpha}, Q_{\beta}]$. However, since $|S/Q_{\alpha}|>p$, comparing with the list in \cref{NewQuad2F}, we have a contradiction. So $Q_{\beta}^*< Q_{\beta}$ and $Q_{\beta}^*Q_{\alpha}<S$. Then, $Q_{\beta}^*Q_{\alpha}$ is a proper $G_{\alpha,\beta}$-invariant subgroup of $S/Q_{\alpha}$, from which it follows that $L_{\alpha}/R_{\alpha}\cong\mathrm{(P)SU}_3(p)$ or $\SU_3(2)'.2$. Since $Q_{\beta}^*$ centralizes $[Z_{\alpha}, Q_{\beta}]$, $|Q_{\beta}^*Q_{\alpha}/Q_{\alpha}|=p$. 

If $p=2$, then as $m_2(S/Q_{\beta})=1$, $L_{\beta}$ is solvable. Set $L_{\beta}^*:=C_{L_{\beta}}([V_{\beta}, Q_{\beta}])$. Then, $Q_{\beta}^*\le O_2(L_{\beta}^*)$ and since $O_2(L_{\beta}^*)$ is $G_{\alpha,\beta}$-invariant and centralizes $[Z_{\alpha}, Q_{\beta}]$, $|O_2(L_{\beta}^*)Q_{\alpha}/Q_{\alpha}|=2$ and $Q_{\beta}^*=O_2(L_{\beta}^*)$. Moreover, $S^*:=Q_{\alpha}Q_{\beta}^*=S\cap L_{\beta}^*\in\syl_2(L_{\beta}^*)$. Setting $L_{\alpha}^*:=\langle (S^*)^{G_{\alpha}} \rangle$, we have that $L_{\alpha}^*\normaleq G_{\alpha}$ and $S^*\in\syl_2(L_{\alpha}^*)$. For $\mu\in\{\alpha, \beta\}$, set $G_{\mu}^*:=L_{\mu}^*K$, where $K$ is a Hall $2'$-subgroup of $G_{\alpha,\beta}$. Then the amalgam $X:=(G_{\alpha}^*, G_{\beta}^*, S^*K)$ satisfies \cref{MainHyp} and by the minimality of $G$, we have a contradiction.

Thus, $p$ is odd and $L_{\alpha}/R_{\alpha}\cong\mathrm{(P)SU}_3(p)$ and $L_{\beta}/R_{\beta}\cong \SL_2(p)$ or $(Q_8\times Q_8):3$. Suppose first that $b>3$. Then $V_{\alpha}^{(2)}$ is elementary abelian so that $V_{\alpha}^{(2)}\cap Q_{\alpha'}$ acts quadratically on $Z_{\alpha'+1}\cap Q_{\beta}$. In particular, $V_{\alpha}^{(2)}\cap Q_{\alpha'-2}\cap Q_{\alpha'-1}=Z_{\alpha}(V_{\alpha}^{(2)}\cap Q_{\alpha'+1})$, $V_{\alpha}^{(2)}\not\le Q_{\alpha'-2}$ and $V_{\alpha}^{(2)}\cap Q_{\alpha'-2}\not\le Q_{\alpha'-1}$ so that $V_{\alpha}^{(2)}/Z_{\alpha}$ contains a unique non-central chief factor for $L_{\alpha}$ which is also a natural $\SU_3(p)$-module. Then an argument on the Schur multiplier of $\SU_3(p)$ yields that $O^p(R_{\alpha})$ centralizes $V_{\alpha}^{(2)}$. Indeed, $Z_{\alpha'}\ne Z_{\alpha'-2}$ else by \cref{SimExt} $V_{\alpha'}=V_{\alpha'-2}\le Q_{\beta}$. But now, $L_{\alpha'-1}=\langle V_{\alpha}^{(2)}\cap Q_{\alpha'-2}, Q_{\alpha'}, R_{\alpha'-1}\rangle$ centralizes $Z_{\alpha'}=[V_{\beta}\cap Q_{\alpha'}, Z_{\alpha'+1}]\le V_{\beta}$, so that $Z_{\alpha'}\le Z(L_{\alpha'+1})=\{1\}$, a contradiction. 

Suppose now that $b=3$. Note that if $Z_{\beta}=Z(R_{\alpha'-1}Q_{\beta})$ and so either $Z_{\alpha'}=Z_{\beta}$ and $R_{\alpha'-1}Q_{\beta}=R_{\alpha'-1}Q_{\alpha}$, or $Z_{\alpha'}\cap Z_{\beta}$ is centralized by $L_{\alpha'-1}=\langle Q_{\alpha'}, Q_{\beta}, R_{\alpha'-1}\rangle$. In the former case, we have that $[V_{\beta}\cap Q_{\alpha'}, Z_{\alpha'+1}]=Z_{\alpha'}=Z_{\beta}$, a contradiction by \cref{NotNatural}. In the latter case, $Z_{\alpha'}\cap Z_{\beta}=\{1\}$ so that $[Z_{\alpha'+1}\cap Q_{\beta}, Z_{\alpha}\cap Q_{\alpha'}]\le Z_{\alpha'}\cap Z_{\beta}=\{1\}$ and an index $p$ subgroup of $Z_{\alpha}$ is centralized by $Z_{\alpha'+1}\cap Q_{\beta}$. Hence, $Z_{\alpha}$ is an FF-module, a contradiction.

Thus, $|S/Q_{\alpha}|=p$ and, as $Q_{\alpha}\cap Q_{\beta}\not\normaleq L_{\beta}$ by \cref{push}, $Q_{\beta}=(Q_{\alpha}\cap Q_{\beta})(Q_{\gamma}\cap Q_{\beta})$ for some $\gamma\in\Delta(\beta)$. Thus, $L_{\beta}=\langle Q_{\gamma}\mid \gamma\in\Delta(\beta)\rangle$ centralizes $[V_{\beta}, Q_{\beta}]$ and $[V_{\beta}, Q_{\beta}]\le Z_{\beta}$. The remaining properties follow from \cref{NewQuad2F} and may be checked in MAGMA.
\end{proof}

\begin{lemma}\label{VA2NotNat}
Suppose that $C_{V_\beta}(V_{\alpha'})<V_\beta\cap Q_{\alpha'}$ and $V_{\alpha}^{(2)}/Z_{\alpha}$ contains a unique non-central chief factor $U/V$ for $L_{\alpha}$. Then $U/V$ is not an FF-module for $\bar{L_{\alpha}}$. 
\end{lemma}
\begin{proof}
Suppose that $U/V$ is an FF-module for $\bar{L_{\alpha}}$. BY \cref{IdenNat}, $q_\alpha=q_\beta$. By \cref{CommCF}, $V_{\alpha}^{(2)}/[V_{\alpha}^{(2)}, Q_{\alpha}]$ contains a non-central chief factor for $L_{\alpha}$. Set $C$ to be the preimage in $V_{\alpha}^{(2)}$ of $C_{V_{\alpha}^{(2)}/Z_{\alpha}}(O^p(L_{\alpha}))$. Then $[V_{\alpha}^{(2)}, Q_{\alpha}]\le C$ and since $U/V$ is an FF-module and $|S/Q_{\alpha}|=p$, by \cref{SEFF}, $V_{\alpha}^{(2)}/C$ is isomorphic to a natural $\SL_2(q_\alpha)$-module. In particular, as $V_{\beta}\not\le C$, $V_{\beta}C/C$ is of order $q_\alpha$ for otherwise $Q_{\beta}$ centralizes $V_{\alpha}^{(2)}/C$. But now, $V_{\beta}\cap C$ has index $q_\alpha$ in $V_{\beta}$ and is normalized by $L_{\alpha}$. By conjugacy, an index $q_\alpha$ subgroup of $V_{\beta}$ is normalized by $L_{\alpha+2}$, and by transitivity, this subgroup is contained in $V_{\alpha+3}$ so that $V_{\beta}\cap V_{\alpha+3}$ is of index $q_\alpha$ in $V_{\beta}$. But then, as $V_{\beta}\not\le Q_{\alpha'}$, $V_{\beta}\cap Q_{\alpha'}=V_{\beta}\cap V_{\alpha+3}=V_{\beta}\cap C_{\alpha'}$ and $[V_{\beta}\cap Q_{\alpha'}, V_{\alpha'}]=\{1\}$, contradicting the initial assumption.
\end{proof}

\begin{lemma}\label{IsNatMod}
Suppose that $C_{V_\beta}(V_{\alpha'})<V_\beta\cap Q_{\alpha'}$. Then $Z_{\alpha}$ is a natural $\SL_2(q_\alpha)$-module.
\end{lemma}
\begin{proof}
Suppose that $b>3$ and $Z_{\alpha}$ is not a natural $\SL_2(q_\alpha)$-module. By \cref{IdenNat}, we have that $q_\alpha=q_\beta$. If $q_\alpha=p$ then $Z_{\alpha}$ is as described in \cref{ZaStruc}, whereas if $q_\alpha>p$ then $Z_{\alpha}$ is a direct sum of two natural $\SL_2(q_\alpha)$-modules. Since $V_{\alpha}^{(2)}$ is elementary abelian and $Z_{\alpha}\cap Q_{\alpha'}\not\le Q_{\alpha'+1}$, $V_{\alpha}^{(2)}\cap Q_{\alpha'-2}\cap Q_{\alpha'-1}=Z_{\alpha}(V_{\alpha}^{(2)}\cap Q_{\alpha'+1})$ is centralized, modulo $Z_{\alpha}$, by $Z_{\alpha'+1}\cap Q_{\beta}\not\le Q_{\alpha}$. By \cref{VA2NotNat}, we conclude that $V_{\alpha}^{(2)}\not\le Q_{\alpha'-2}$ and $V_{\alpha}^{(2)}\cap Q_{\alpha'-2}\not\le Q_{\alpha'-1}$.

Suppose first that $L_{\alpha}/R_{\alpha}\cong\SU_3(2)'$ so that $q_\alpha=q_\beta=p$. Then $R=[V_{\beta}\cap Q_{\alpha'}, V_{\alpha'}]=[V_{\alpha'}, Q_{\alpha'}]$ is of order $4$ and strictly contained in $Z_{\alpha'}$. Moreover, since $b>3$, $R$ is centralized by $X_{\alpha'-1}:=\langle V_{\alpha}^{(2)}\cap Q_{\alpha'-2}, R_{\alpha'-1}, Q_{\alpha'}\rangle$ and so either $Q_{\alpha'}Q_{\alpha'-1}$ is conjugate to $Q_{\alpha'}Q_{\alpha'-2}$ by an element of $R_{\alpha'-1}$; or $X_{\alpha'-1}/C_{X_{\alpha'-1}}(Z_{\alpha'-1})\cong\Sym(3)$. In the latter case, it follows that $R$ is invariant under the action of a subgroup of index $3$ in $L_{\alpha'-1}$, a contradiction to structure of $Z_{\alpha'-1}$. In the former case, it follows that $[V_{\alpha'}, Q_{\alpha'}]=[V_{\alpha'-2}, Q_{\alpha'-2}]$ and since $V_{\alpha}^{(2)}\not\le Q_{\alpha'-2}$, we may iterate backwards through critical pairs $(\alpha-2k, \alpha'-2k)$ for $k\geq 0$ so that $R=[V_{\alpha'}, Q_{\alpha'}]=[V_{\beta}, Q_{\beta}]\le Z_{\beta}$ and so an index $p$ subgroup of $V_{\beta}/Z_{\beta}$ is centralized by $V_{\alpha'}$. We have a contradiction by \cref{NotNatural}. 

Now, for all values of $q_\alpha$, we have that $Z_{\alpha'}=H\le V_{\beta}$. Then $V_{\alpha}^{(2)}\cap Q_{\alpha'-2}$ is not contained in $Q_{\alpha'-1}$ and centralizes $Z_{\alpha'}Z_{\alpha'-2}$. It follows that $Z_{\alpha'}=Z_{\alpha'-2}$. Moreover, since $V_{\alpha}^{(2)}\not\le Q_{\alpha'-2}$ there is some $\alpha-2$, with $(\alpha-2, \alpha'-2)$ a critical pair. By \cref{NotNatural}, we may assume that $(\alpha-2, \alpha'-2)$ satisfies the same hypothesis as $(\alpha, \alpha')$. Iterating through critical pairs, we conclude that $Z_{\alpha'}=\dots=Z_{\beta}$. But then $R=[V_{\beta}\cap Q_{\alpha'}, V_{\alpha'}]=Z_{\beta}$ and $V_{\beta}/C_{V_{\beta}}(O^p(L_{\beta}))$ is a natural $\SL_2(q_\alpha)$-module for $L_{\beta}/R_{\beta}$, a contradiction by \cref{NotNatural}. Hence if $Z_{\alpha}$ is not a natural $\SL_2(q_\alpha)$-module then $b=3$.

Suppose that $L_{\alpha}/R_{\alpha}\cong\SU_3(2)'$, $b=3$ and $Z_{\alpha}$ is the restriction of a natural $\SU_3(2)$-module. Since $Q_{\alpha}$ is non-abelian, by the irreducibility of $Z_{\alpha}$, $Z_{\alpha}\le \langle (Z_{\beta}\cap \Phi(Q_{\alpha}))^{G_{\alpha}}\rangle\le \Phi(Q_{\alpha})$. If $|S/Q_{\beta}|=2$, then $Q_{\alpha}\cap Q_{\beta}\cap Q_{\alpha'-1}=Z_{\alpha}(Q_{\alpha}\cap \dots \cap Q_{\alpha'+1})$ and since $Q_{\alpha}/\Phi(Q_{\alpha})$ is not an FF-module, $\bar{L_{\alpha}}\cong\SU_3(2)'$ and $Q_{\alpha}/\Phi(Q_{\alpha})$ contains a unique non-central chief factor, $U/V$ say. Moreover, $U/V$ is isomorphic to $Z_{\alpha}$ and $U\not\le Q_{\beta}$. But $U\cap Q_{\beta}$ is $G_{\alpha,\beta}$-invariant subgroup of index $2$ in $U$, a contradiction.

Applying \cref{NewQuad2F}, we see that $L_{\beta}/R_{\beta}\cong (3\times 3):4$. Now $V_{\beta}(Q_{\beta}\cap Q_{\alpha'-1}\cap Q_{\alpha'})$ has index at most $4$ in $Q_{\beta}$ and since $|S/Q_{\beta}|\ne 2$, no non-central chief factor is an FF-module for $\bar{L_{\beta}}$ and so $Q_{\beta}/V_{\beta}$ contains a unique non-central chief factor for $\bar{L_{\beta}}$, and this chief factor lies in $Q_{\beta}/C_{\beta}$. Then, an application of the three subgroup lemma implies that $R_{\beta}=Q_{\beta}$. Thus, $\bar{L_{\beta}}\cong (3\times 3):4$. However, from the structure of $Z_{\alpha}$, we conclude that $Z_{\beta}=Z_{\alpha}\cap C_{V_{\beta}}(O^p(L_{\beta}))$ has index $4$ in $Z_{\alpha}$ so that a subgroup of order $4$ of $V_{\beta}/C_{V_{\beta}}(O^2(L_{\beta}))$ is centralized by $S=Q_{\alpha}Q_{\beta}$, contradicting the structure of the $2$F-module associated to $(3\times 3):4$. Hence, $L_{\alpha}/R_{\alpha}\not\cong\SU_3(2)'$.

By \cref{ZaStruc} when $q_\alpha=p$, we may now assume that $Z_{\alpha}=Z_{\beta}\times Z_{\alpha-1}$ for some $\alpha-1\in\Delta(\alpha)$. Then $[V_{\beta}\cap Q_{\alpha'}, V_{\alpha'}\cap Q_{\beta}]\le Z_{\alpha'}\cap Z_{\beta}$. But $Z_{\alpha'}Z_{\beta}\le Z_{\alpha'-1}$ and by \cref{ZaStruc}, either $Z_{\alpha'}=Z_{\beta}$, or $Z_{\alpha'}\cap Z_{\beta}=\{1\}$. If $Z_{\alpha'}=Z_{\beta}$, then $R=[V_{\beta}\cap Q_{\alpha'}, V_{\alpha'}]=Z_{\beta}$ and so $V_{\beta}/C_{V_{\beta}}(O^p(L_{\beta}))$ is a natural $\SL_2(q_\alpha)$-module, a contradiction by \cref{NotNatural}. Hence, $[Z_{\alpha}\cap Q_{\alpha'}, Z_{\alpha'+1}\cap Q_{\beta}]\le [V_{\beta}\cap Q_{\alpha'}, V_{\alpha'}\cap Q_{\beta}]\le Z_{\alpha'}\cap Z_{\beta}=\{1\}$ and by \cref{SEFF}, $Z_{\alpha}$ is an FF-module, a final contradiction.
\end{proof}

\begin{lemma}\label{ZaNat}
Suppose that $C_{V_\beta}(V_{\alpha'})<V_\beta\cap Q_{\alpha'}$. Then $q_\alpha=q_\beta$ and for $V:=V_{\beta}/C_{V_{\beta}}(O^p(L_{\beta}))$ either:
\begin{enumerate}
\item $V$ is a natural $\Sz(q_\beta)$-module for $L_{\beta}/R_{\beta}\cong\Sz(q_\beta)$;
\item $V$ is a natural $\Sz(2)$-module for $L_{\beta}/R_{\beta}\cong\Dih(10)$; or
\item $V$ is a $2$F-module for $L_{\beta}/R_{\beta}\cong (3\times 3):2$ or $(3\times 3):4$.
\end{enumerate}
\end{lemma}
\begin{proof}
Suppose first that $q_\beta=p$. By \cref{NotNatural}, $V$ is not an FF-module and so, as $V$ is a quadratic $2$F-module, the structure of $V$ and $L_{\beta}/R_{\beta}$ follows from \cref{NewQuad2F}. Since $Z_{\alpha}C_{V_{\beta}}(O^p(L_{\beta}))/C_{V_{\beta}}(O^p(L_{\beta}))$ is of order $p$, $G_{\alpha,\beta}$-invariant and $V_{\beta}=\langle Z_{\alpha}^{L_{\beta}}\rangle$, by \cref{NewQuad2F}, we conclude that $L_{\beta}/R_{\beta}\cong\Sz(2), \Dih(10)$, $(3\times 3):2$ or $(3\times 3):4$.

Hence, we assume that $q_\alpha=q_\beta>p$ so that $\bar{L_{\beta}}/O_{p'}(\bar{L_{\beta}})\cong \PSL_2(q_\beta), \PSU_3(q_\beta)$ or $\Sz(q_\beta)$. Then $V_{\beta}\cap Q_{\beta}\cap Q_{\alpha}$ has index at most $q_\beta^2$ and is centralized by $Z_{\alpha}$. Applying \cref{2FRecog} and \cref{NotNatural}, we deduce that $L_{\alpha'}/R_{\alpha'}\cong \SL_2(q_\beta), \mathrm{(P)SU}_3(q_\beta)$ or $\Sz(q_\beta)$ and $V_{\beta}/C_{V_{\beta}}(O^2(L_{\beta}))$ is either a direct sum of two natural $\SL_2(q_\beta)$-modules, or the associated natural module in the latter two cases. Moreover, $Z_{\alpha}C_{V_{\beta}}(O^2(L_{\beta}))/C_{V_{\beta}}(O^2(L_{\beta}))\cong Z_{\alpha}/Z_{\beta}$ has order $q_\alpha=q_\beta$ in $V_{\beta}/C_{V_{\beta}}(O^2(L_{\beta}))$ and is $G_{\alpha,\beta}$-invariant. Then \cref{NatGen} and \cref{SUMod} imply that $L_{\beta}/R_{\beta}\cong \Sz(q_\beta)$ and $V_{\beta}/C_{V_{\beta}}(O^2(L_{\beta}))$ is a natural module.
\end{proof}

\begin{lemma}\label{nb=3}
Suppose that $C_{V_\beta}(V_{\alpha'})<V_\beta\cap Q_{\alpha'}$. Then $b=3$.
\end{lemma}
\begin{proof}
Suppose that $b>3$. Since $V_{\beta}$ is centralized by $V_{\beta}^{(3)}$ we deduce that $[V_{\alpha'}, V_{\beta}, V_{\beta}^{(3)}]=\{1\}$, $V_{\beta}^{(3)}\cap Q_{\alpha'-2}\cap Q_{\alpha'-1}=V_{\beta}(V_{\beta}^{(3)}\cap\dots\cap Q_{\alpha'})$ and $[V_{\beta}^{(3)}\cap\dots\cap Q_{\alpha'}, V_{\alpha'}]=[V_{\alpha'}, Q_{\alpha'}]=Z_{\alpha'}=R\le V_{\beta}$.

Assume that $|S/Q_{\beta}|\ne 2$. Since $V_{\beta}^{(3)}\cap Q_{\alpha'-2}\cap Q_{\alpha'-1}$ has index at most $q_\beta^2$ in $V_{\beta}^{(3)}$, applying \cref{2FRecog} when $q_\beta>2$ and \cref{SEFF} when $|S/Q_{\beta}|\ne p$, we conclude that $V_{\beta}^{(3)}\not\le Q_{\alpha'-2}$ and $(V_{\beta}^{(3)}\cap Q_{\alpha'-2})Q_{\alpha'-1}\in\syl_p(L_{\alpha'-1})$. Then $Z_{\alpha'}$ is centralized by $\langle Q_{\alpha'}, R_{\alpha'-1}, (V_{\beta}^{(3)}\cap Q_{\alpha'-2})\rangle$ and we conclude that $R_{\alpha'-1}Q_{\alpha'-2}=R_{\alpha'-1}Q_{\alpha'}$ and $Z_{\alpha'}=Z_{\alpha'-2}$. Since $V_{\beta}^{(3)}\not\le Q_{\alpha'-2}$ there is a critical pair $(\beta-3, \alpha'-2)$ satisfying the same hypothesis as $(\alpha, \alpha')$ by \cref{SL2VlQ} and \cref{SL2VnQ}, and iterating back through critical pairs, we conclude that $Z_{\alpha'}=\dots=Z_{\beta}$, $[V_{\beta}\cap Q_{\alpha'}, V_{\alpha'}]=R=Z_{\alpha'}=Z_{\beta}$ and $V_{\beta}/C_{V_{\beta}}(O^2(L_{\beta}))$ is a natural $\SL_2(q_\beta)$-module, a contradiction by \cref{NotNatural}.

Hence, $|S/Q_{\beta}|=2$ and we may assume by the above argument that either $V_{\beta}^{(3)}\le Q_{\alpha'-2}$ and $Z_{\alpha'}=Z_{\alpha'-2}$, or $V_{\beta}^{(3)}\cap Q_{\alpha'-2}\le Q_{\alpha'-1}$. Repeating the proof of \cref{VA2NotNat}, we deduce that if $V_{\alpha}^{(2)}/Z_{\alpha}$ contains a unique non-central chief factor for $L_{\alpha}$, then it is not an FF-module for $\bar{L_{\alpha}}$, and so we have that either $V_{\alpha}^{(2)}\le Q_{\alpha'-2}$, $V_{\alpha}^{(2)}\not\le Q_{\alpha'-1}$ and $Z_{\alpha'}=Z_{\alpha'-2}$, or $V_{\alpha}^{(2)}\not\le Q_{\alpha'-2}$ and $V_{\alpha}^{(2)}\cap Q_{\alpha'-2}\le Q_{\alpha'-1}$. 

In the former case, there is $\alpha-1\in\Delta(\alpha)$ with $V_{\alpha-1}\le Q_{\alpha'-2}$ and $V_{\alpha-1}\not\le Q_{\alpha'-1}$. Hence, an index $2$ subgroup of $V_{\alpha-1}$ is centralized by $Z_{\alpha'-1}$ and \cref{SEFF} yields that $Z_{\alpha'-1}\le Q_{\alpha-1}$. But then $Z_{\alpha'-2}=[Z_{\alpha'-1}, V_{\alpha-1}]=Z_{\alpha-1}$ and since $b>3$ we have that $Z_{\alpha'}=Z_{\alpha'-2}=Z_{\alpha-1}=Z_{\beta}$. But then, as above, \cref{NotNatural} gives a contradiction. Hence, $V_{\alpha}^{(2)}\not\le Q_{\alpha'-2}$ and so there is a critical pair $(\alpha-2, \alpha'-2)$. Applying \cref{SL2VlQ} and \cref{SL2VnQ}, $(\alpha-2, \alpha'-2)$ satisfies the same properties as $(\alpha, \alpha')$ and we can iterate far back enough through critical pairs so that we find a value $m$ such that for all $k\geq m$, $(\alpha-2k, \alpha'-2k)$ is a critical pair and $(\alpha'-3-2k, \alpha'-2k+b-3)$ is a critical pair. In particular, $Z_{\alpha'-1-2k}\ne Z_{\alpha'-3-2k}$ so that $L_{\alpha'-2-2k}=\langle V_{\alpha-2k}^{(2)}, Q_{\alpha'-1-2k}, R_{\alpha'-2-2k}\rangle$. Since $Z_{\alpha'-2k}Z_{\alpha'-2-2k}$ is normalized by $L_{\alpha'-2-2k}$, we deduce that $Z_{\alpha'-2k}=Z_{\alpha'-2-2k}$ for all $k\geq m$ and iterating even further back through the critical pairs, we have that $Z_{\alpha'-2k}=Z_{\beta-2k}$ and again, \cref{NotNatural} provides a contradiction.
\end{proof}

\begin{lemma}\label{nb=3ii}
Suppose that $C_{V_\beta}(V_{\alpha'})<V_\beta\cap Q_{\alpha'}$ and $b=3$. Then $q_\alpha=q_\beta$, $L_{\alpha}/R_{\alpha}\cong\SL_2(q_\alpha)$, $Z_{\alpha}$ is natural $\SL_2(q_\alpha)$-module, $O^2(L_{\beta})$ centralizes $C_{\beta}/V_{\beta}$ and one of the following holds:
\begin{enumerate}
\item $\bar{L_{\beta}}\cong\Sz(q_\beta)$ and $V_{\beta}/Z_{\beta}$ is a natural module $\Sz(q_\beta)$-module; or
\item $\bar{L_{\beta}}\cong (3\times 3):4$ and $V_{\beta}/Z_{\beta}$ is an irreducible $2$F-module.
\end{enumerate}
\end{lemma}
In particular, $L_{\beta}$ is $2$-minimal in either case.
\begin{proof}
Suppose that $|S/Q_{\beta}|=2$ so that $L_{\beta}/R_{\beta}\cong\Dih(10)$ or $(3\times 3):2$. Then $C_{\beta}\le Q_{\alpha'-1}$ and $C_{\beta}=V_{\beta}(C_{\beta}\cap Q_{\alpha'})$. Since $[V_{\alpha'}, Q_{\alpha'}]=Z_{\alpha'}\le V_{\beta}$, we deduce that $O^2(L_{\beta})$ centralizes $C_{\beta}/V_{\beta}$. Then for $r\in R_{\beta}$ of odd order, if $[r, Q_{\beta}, V_{\beta}]=\{1\}$ then $[r, V_{\beta}, Q_{\beta}]=\{1\}$ by the three subgroup lemma, and so $r$ centralizes $Q_{\beta}$. But now, $Q_{\beta}\cap Q_{\alpha'-1}=V_{\beta}(Q_{\beta}\cap Q_{\alpha'})$, and so $Q_{\beta}/V_{\beta}$ contains a unique non-central chief factor for $L_{\beta}$, which is a faithful FF-module for $\bar{L_{\beta}}$, and $\bar{L_{\beta}}\cong \Sym(3)$ by \cref{SEFF}, a contradiction.

Thus, $|S/Q_{\beta}|=4$ or $q_\beta>2$ and by \cref{SEFF}, no non-central chief factor within $Q_{\beta}$ is an FF-module for $\bar{L_{\beta}}$. Since $C_{\beta}\le Q_{\alpha'-1}$, $V_{\beta}(C_{\beta}\cap Q_{\alpha'})$ has index at most $q_\beta$ in $C_{\beta}$ and since $[Q_{\alpha'}, V_{\alpha'}]=Z_{\alpha'}\le V_{\beta}$, $V_{\alpha'}$ centralizes $C_{\beta}/V_{\beta}$ so that $O^2(L_{\beta})$ centralizes $C_{\beta}/V_{\beta}$. Now, applying the three subgroup lemma, any $2'$-element of $R_{\beta}$ centralizes $Q_{\beta}/C_{\beta}$ and $V_{\beta}$ so centralizes $Q_{\beta}$, and we deduce that $R_{\beta}=Q_{\beta}$. By \cref{ZaNat}, $\bar{L_{\beta}}\cong \Sz(q_\beta)$ or $(3\times 3):4$ and $V_{\beta}/C_{V_{\beta}}(O^2(L_{\beta}))$ is a natural module or an irreducible $2F$-module described in \cref{NewQuad2F}. 

Assume that $q_\beta=2$ so that $L_{\beta}$ is solvable. Applying coprime action, we have that $V_{\beta}/Z_{\beta}=[V_{\beta}/Z_{\beta}, O^2(L_{\beta})]\times C_{V_{\beta}/Z_{\beta}}(O^2(L_{\beta}))$ where $[V_{\beta}/Z_{\beta}, O^2(L_{\beta})]$ is irreducible of dimension $4$. Letting $V^\beta$ be the preimage in $V_{\beta}$ of $[V_{\beta}/Z_{\beta}, O^2(L_{\beta})]$, we must have that $[V^\beta\cap Q_{\alpha'}, V_{\alpha'}]=Z_{\alpha'}\le V^\beta$ so that $Z_{\alpha'-1}=Z_{\alpha'}\times Z_{\beta}\le V^\beta$. But then, by definition, $V^\beta=V_{\beta}$ and $V_{\beta}/Z_{\beta}$ is irreducible of dimension $4$. Since $Q_{\beta}/C_{\beta}$ is dual to $V_{\beta}/Z_{\beta}$, $Q_{\beta}/C_{\beta}$ is also irreducible of dimension $4$.

Assume that $q_\beta>2$. Note that $V_{\beta}(Q_{\beta}\cap Q_{\alpha'-1}\cap Q_{\alpha'})$ has index $q_\beta^2$ in $Q_{\beta}$ and centralized, modulo $V_{\beta}$, by $V_{\alpha'}$. Since $L_{\beta}=\langle V_{\alpha'}, V_{\alpha'}^g, Q_{\beta}\rangle$ for some appropriately chosen $g\in L_{\beta}$, we deduce that $|Q_{\beta}/\Phi(Q_{\beta})V_{\beta}|\leq q_\beta^4$. Furthermore, since $\Phi(Q_{\beta})V_{\beta}\le Q_{\alpha'-1}\cap Q_{\beta}\not\normaleq L_{\beta}$ by \cref{push}, $Q_{\beta}/\Phi(Q_{\beta})V_{\beta}$ is irreducible. But now, $\Phi(Q_{\beta})V_{\beta}\le Q_{\beta}\bigcap\limits_{\lambda\in\Delta(\beta)} Q_\lambda\le C_\beta$ so that $|Q_{\beta}/C_{\beta}|=q_\beta^4$ and $Q_{\beta}=(Q_{\beta}\cap O^2(L_{\beta}))C_{\beta}$ and $L_{\beta}$ centralizes $C_{V_{\beta}}(O^2(L_{\beta}))$. Hence, $V_{\beta}/Z_{\beta}$ is a natural module and the result holds.
\end{proof}

\begin{proposition}\label{F42}
Suppose that $C_{V_\beta}(V_{\alpha'})<V_\beta \cap Q_{\alpha'}$ and $b>1$. If $q_\beta>2$ then $G$ is locally isomorphic to ${}^2\mathrm{F}_4(q_\beta)$.
\end{proposition}
\begin{proof}
We have that $b=3$, $L_{\alpha}/R_{\alpha}\cong\SL_2(q_\alpha)$, $Z_{\alpha}$ is natural $\SL_2(q_\alpha)$-module, $\bar{L_{\beta}}\cong\Sz(q_\beta)$ and $V_{\beta}/Z_{\beta}$ is a natural module $\Sz(q_\beta)$-module. Assume that $Z_{\alpha'-1}\cap [V_{\beta}, Q_{\alpha}]>Z_{\beta}$. Then, by the $2$-transitivity of $\bar{L_{\beta}}$ on the neighbours of $\beta$, $Z_{\alpha}\cap [V_\beta, Q_{\alpha'-1}]>Z_{\beta}$. Since $[V_{\beta}, Q_{\alpha'-1}]\le Q_{\alpha'}$, this is a contradiction. Hence, $V_{\beta}=Z_{\alpha'}[V_{\beta}, Q_{\alpha}]$ so that $[V_{\alpha}^{(2)}, Q_{\alpha}]\cap Z_{\alpha'}=\{1\}$.

For the remainder of this proof, set $V:=[V_{\alpha}^{(2)}, Q_{\alpha}]$ and $C:=C_{Q_{\alpha}}(V)$. Then $C$ centralizes $[V_{\beta}, Q_{\alpha}]$ so that $C\le Q_{\beta}$. Moreover, $C\cap Q_{\alpha'-1}$ centralizes $V_{\beta}\le Z_{\alpha'}V$ so that $C\cap Q_{\alpha'-1}=C\cap C_{\beta}$. From the structure of $Q_{\beta}/C_{\beta}$ as an $\bar{L_{\beta}}$-module, we infer that $C_{\beta}=V_{\beta}(C_{\beta}\cap Q_{\alpha'})$. Hence, $C\cap Q_{\alpha'-1}=Z_{\alpha}(C\cap Q_{\alpha'})$.

Now, $V\ge [V_{\alpha}^{(2)}, Q_{\alpha}\cap Q_{\beta}]\not\le C_{\beta}$. In particular, $V\not\le Q_{\alpha'-1}$, else $V\le Q_{\alpha}\cap Q_{\beta}\cap Q_{\alpha'-1}$ and as $\bar{L_{\beta}}$ acts $2$-transitively on neighbours of $\beta$, this would imply that $V\le Q_{\beta}\cap \bigcap\limits_{\lambda\in\Delta(\beta)} Q_\lambda=C_{\beta}$, a contradiction. Hence, $Z_{\alpha'}\cap C=Z_{\alpha'}\cap V=\{1\}$ and applying \cref{SEFF}, we have that $C/Z_{\alpha}$ contains a unique non-central chief factor which is an FF-module for $\bar{L_{\alpha}}$ so that $O^2(R_{\alpha})$ centralizes $C$. Assume that $|VC_{\beta}/C_{\beta}|> q_\beta$. Then $[V, Q_{\alpha}]\not\le C_{\beta}$ and as $[V_{\beta}, Q_{\alpha}, Q_{\alpha}]\le C_{\beta}$, $[V, Q_{\alpha}]/Z_{\alpha}$ also contains a non-central chief factor for $\bar{L_{\alpha}}$. Then both $V/[V, Q_{\alpha}]$ and $[V, Q_{\alpha}]/Z_{\alpha}$ are FF-modules for $\bar{L_{\alpha}}$ and $O^2(R_{\alpha})$ centralizes $V$. Assume that $|VC_{\beta}/C_{\beta}|=q_\beta$ so that $V\cap Q_{\alpha'-1}=V\cap C_\beta$ and we deduce that $V\cap Q_{\alpha'-1}=Z_{\alpha}(V\cap Q_{\alpha'})$ and $V/Z_{\alpha}$ contains at most one non-central chief factor and, if it exists, it is an FF-module for $\bar{L_{\alpha}}$ so that $O^2(R_{\alpha})$ centralizes $V$. Hence, $O^2(R_\alpha)$ centralizes $CV$.

Finally, by the three subgroups lemma $[O^2(R_{\alpha}), Q_\alpha, V]=\{1\}$ so that $[O^2(R_\alpha), Q_\alpha]\le C$ and coprime action gives $[O^2(R_\alpha), Q_\alpha]=\{1\}$ and $R_\alpha=Q_\alpha$. Then $G$ has a weak BN-pair of rank $2$ and comparing with \cite{Greenbook}, the result holds.
\end{proof}

\begin{proposition}
Suppose that $C_{V_\beta}(V_{\alpha'})<V_\beta \cap Q_{\alpha'}$ and $b>1$. Then $G$ is locally isomorphic to ${}^2\mathrm{F}_4(2^n)$ or ${}^2\mathrm{F}_4(2)'$.
\end{proposition}
\begin{proof}
By \cref{F42}, \cref{nb=3} and \cref{nb=3ii}, we have that $b=3$, $L_{\alpha}/R_{\alpha}\cong\Sym(3)$, $Z_{\alpha}$ is natural $\SL_2(2)$-module and either $\bar{L_{\beta}}\cong\Sz(2)$ or $\bar{L_{\beta}}\cong (3\times 3):4$. Suppose first that $L_{\alpha}$ is also $2$-minimal group. Then the amalgam is determined in \cite{Hayashi}, $G$ has a weak BN-pair of rank $2$ and the result follows by \cite{Greenbook} and \cite{Fan}. Hence, to complete the proof, we assume that $L_{\alpha}$ is not $2$-minimal and derive a contradiction. We may choose $P_{\alpha}< L_{\alpha}$ such that $P_{\alpha}$ is $2$-minimal. Better, by McBride's lemma (\cref{McBride}), we may choose $P_{\alpha}$ such that $P_{\alpha}\not\le R_{\alpha}$ and $L_{\alpha}=P_{\alpha}R_{\alpha}$. Moreover, we may assume that $G$ is a minimal counterexample to \hyperlink{ThmC}{Theorem C}. Form $X:=\langle P_{\alpha}, L_{\beta}(G_{\alpha,\beta}\cap P_{\alpha})\rangle$ and let $K$ be the largest subgroup of $S$ which is normal in $X$.

If $K=\{1\}$, then it follows that any non-trivial normal subgroup of $X$ which is contained in $G_{\alpha,\beta}\cap P_{\alpha}$ is a $2'$-group, a contradiction for then $Q_{\lambda}$ is not self centralizing in $G_{\lambda}$, where $\lambda\in\{\alpha,\beta\}$. Thus, no non-trivial normal subgroup of $G_{\alpha,\beta}\cap P_{\alpha}$ is normal in $X$ and the triple $(P_{\alpha}, L_{\beta}(G_{\alpha,\beta}\cap P_{\alpha}), G_{\alpha,\beta}\cap P_{\alpha})$ satisfies \cref{MainHyp}. Then, by minimality and comparing with the list of amalgams in \hyperlink{ThmC}{Theorem C}, it follows that $X$ is locally isomorphic to ${}^2\mathrm{F}_4(2)$ or ${}^2\mathrm{F}_4(2)'$. In particular, $P_{\alpha}/Q_{\alpha}\cong\Sym(3)$, $G_{\beta}/Q_{\beta}\cong\Sz(2)$ and $S$ is isomorphic to a Sylow $2$-subgroup of ${}^2\mathrm{F}_4(2)$ or ${}^2\mathrm{F}_4(2)'$. But then $2^2\leq |Q_{\alpha}/\Phi(Q_{\alpha})|\leq 2^3$ and so, $\bar{L_{\alpha}}$ is isomorphic to a subgroup of $\GL_3(2)$ which has a strongly $2$-embedded subgroup. An elementary calculation, which may be performed in MAGMA, yields $\bar{L_{\alpha}}\cong\bar{P_{\alpha}}\cong\Sym(3)$ and $L_{\alpha}$ is $2$-minimal, a contradiction.

Thus, $K\ne\{1\}$ and since $P_{\alpha}$ does not centralize $Z_{\beta}$ and $K\normaleq S$, we deduce that $Z_{\alpha}\le K$ and so $V_{\beta}\le K$. Moreover, since $K\le Q_{\alpha}\cap Q_{\beta}$ and $K\normaleq L_{\beta}$, $K\le C_{\beta}$. If $\Phi(K)\ne\{1\}$ then $Z_{\beta}\le \Phi(K)$ and arguing as above, $V_{\beta}\le \Phi(K)$. But then $O^2(L_{\beta})$ centralizes $K/\Phi(K)$, a contradiction. Thus, $K$ is elementary abelian and since $C_S(K)\le C_{\beta}$, $C_S(K)=C_{Q_{\alpha}}(K)=C_{Q_{\beta}}(K)\normaleq X$ and $C_S(K)=K$. 

Suppose that there is $r\in P_{\alpha}$ such that $[r, Q_{\alpha}]\le K$. If $r$ centralizes $C_{K}(Q_{\alpha})$, then by the A$\times$B-lemma, $r$ centralizes $K$. But then $r$ centralizes $Q_{\alpha}$, and so $r$ is trivial. Now, since $Q_{\alpha}$ is self centralizing in $S$, $C_{K}(Q_{\alpha})\le Z(Q_{\alpha})$. But $V_{\alpha'}\cap Q_{\alpha}$ is of index $4$ in $V_{\alpha'}$, contains $Z_{\alpha'-1}$ and is centralized by $Z(Q_{\alpha})$ from which it follows that $Z(Q_{\alpha})=Z_{\alpha}(Z(Q_{\alpha})\cap Q_{\alpha'})$. Since $Z_{\alpha'}\not\le Z(Q_{\alpha})$, otherwise $Z_{\alpha'-1}=Z_{\alpha'}\times Z_{\beta}$ would be normalized by $L_{\beta}=\langle Q_{\beta}, Q_{\alpha}, Q_{\alpha'-1}\rangle$, it follows that $V_{\alpha'}\cap Q_{\beta}$ centralizes $Z(Q_{\alpha})/Z_{\alpha}$ and so $O^2(L_{\alpha})$ centralizes $Z(Q_{\alpha})/Z_{\alpha}$. Since $Z_{\beta}\le Z_{\alpha}=[Z(Q_{\alpha}), O^2(L_{\alpha})]$, it follows from coprime action that $Z(Q_{\alpha})=Z_{\alpha}$. Hence, for $r$ of odd order such that $[r, Q_{\alpha}]\le K$, we have that $r\not\le R_{\alpha}$ and it follows that $r$ is of order $3$ and $\langle r\rangle K/K=O_{2'}(P_{\alpha}/K)$. Then, by coprime action and as $r$ acts non-trivially on $Z_{\alpha}$, we have that $K=[K, r]$. But now, as $K$ is elementary abelian and contains $V_{\beta}$, it follows that $K\cap Q_{\alpha'}\cap Q_{\alpha'+1}$ is has index $p^2$ in $K$ and is centralized by $Z_{\alpha'+1}\cap Q_{\beta}\not\le Q_{\alpha}$. In particular, $K$ contains at most two non-central chief factors for $P_{\alpha}$ and $K$ is acted upon quadratically $V_{\alpha'}\cap Q_{\beta}$. Note that $K/[K, Q_{\alpha}]$ is not centralized by $r$, and neither is $[K, Q_{\alpha}]$. But then $[K, Q_{\alpha}]\le Z(Q_{\alpha})=Z_{\alpha}$ and $K/[K, Q_{\alpha}]$ is an FF-module, absurd for then the action of $r$ implies that $2^5=|V_{\beta}|<|K|=2^4$. Thus, $P_{\alpha}/K$ is of characteristic $2$.

Suppose that there is $s\in L_{\beta}(P_{\alpha}\cap G_{\alpha,\beta})$ such that $[s, Q_{\beta}]\le K$. Since $L_{\beta}/Q_{\beta}\cong\Sz(2)$ it follows that $L_{\beta}(P_{\alpha}\cap G_{\alpha,\beta})/Q_{\beta}=L_{\beta}/Q_{\beta}\times O_{2'}((P_{\alpha}\cap G_{\alpha,\beta})/Q_{\beta})$. Since $K\le C_{\beta}$ and $Q_{\beta}/C_{\beta}$ is an irreducible module for $\bar{L_{\beta}}$, $s\not\le L_{\beta}$. Hence, $s$ centralizes $S/Q_{\beta}$ and so centralizes $S/K$. Then $s\in P_{\alpha}$ and centralizes $Q_{\alpha}/K$, and by the previous paragraph, $s=1$. Thus, $L_{\beta}(P_{\alpha}\cap G_{\alpha,\beta})/K$ is of characteristic $2$. Moreover, no subgroup of $S$ properly containing $K$ is normal in $X$ and since $P_{\alpha}/K$ is of characteristic $2$, it follows that no non-trivial subgroup of $(G_{\alpha,\beta}\cap P_{\alpha})/K$ is normal in $X/K$. Then the triple $(P_{\alpha}/K, (L_{\beta}(G_{\alpha,\beta}\cap P_{\alpha}))/K, (G_{\alpha,\beta}\cap P_{\alpha})/K)$ satisfies \cref{MainHyp}. By minimality and since $L_{\beta}/Q_{\beta}\cong\Sz(2)$, $X/K$ is locally isomorphic to ${}^2\mathrm{F}_4(2)$ or ${}^2\mathrm{F}_4(2)'$. But there is only one non-central chief factor in $Q_{\beta}/K$ for $L_{\beta}$, and we have a contradiction.
\end{proof}

%% file: Contents/7.2.Non-symmetricCaseii.tex
\subsection{$C_{V_{\beta}}(V_{\alpha'})=V_{\beta}\cap Q_{\alpha'}$}

We continue with the analysis of the case $[Z_{\alpha}, Z_{\alpha'}]=\{1\}$, this time with the additional assumptions that $b>1$ and $[V_{\beta}\cap Q_{\alpha'}, V_{\alpha'}]=\{1\}$. Recall from \cref{SL2VlQ} and \cref{SL2VnQ} that this hypothesis implies that $q:=q_\alpha=q_\beta$, $L_{\alpha}/R_{\alpha}\cong L_{\beta}/R_{\beta}\cong\SL_2(q)$ and $Z_{\alpha}$ and $V_{\beta}/C_{V_{\beta}}(O^p(L_{\beta}))$ are natural $\SL_2(q)$-modules.

Throughout this subsection, we fix the notation $V^\lambda:=\langle (C_{V_{\mu}}(O^p(L_{\mu})))^{G_{\lambda}}\rangle$ whenever $\lambda\in\alpha^G$, $\mu\in\Delta(\lambda)$ and $|V_{\beta}|\ne q^3$, and we remark that when $|V_{\beta}|\ne q^3$ and $b>3$, for $\gamma\in\beta^G$ and some fixed $\delta\in\Delta(\gamma)$, the subgroup $\langle V^{\mu}\mid Z_{\mu}=Z_{\delta}, \mu\in\Delta(\gamma)\rangle$ is normal in $R_{\gamma}Q_{\delta}$ by essentially the same argument as \cref{UWNormal}. Throughout, we set $R:=[V_{\alpha'}, V_{\beta}]$ so that $R\le Z_{\alpha+2}C_{V_{\beta}}(O^p(L_{\beta}))\cap Z_{\alpha'-1}C_{V_{\alpha'}}(O^p(L_{\alpha'}))\le V_{\beta}\cap V_{\alpha'}$ and, in particular, if $|V_{\beta}|=q^3$, then $R\le Z_{\alpha+2}\cap Z_{\alpha'-1}$. By \cref{NotNatural}, we may assume in this section that every critical pair $(\alpha, \alpha')$ satisfies the condition $C_{V_{\beta}}(V_{\alpha'})=V_{\beta}\cap Q_{\alpha'}$. We reiterate that whenever we assume the necessary values of $b$, we are able to apply \cref{VBGood} through \cref{GoodAction4}. That the hypotheses of these lemmas are satisfied will often be left implicit in proofs. 

The first goal in the analysis of the case $C_{V_\beta}(V_{\alpha'})=V_\beta \cap Q_{\alpha'}$ will be to show that $b\leq 5$. Then the methods for $b=5$ differ slightly from the techniques employed for larger values of $b$ and so, for the most part, we treat the case when $b=5$ independently from the the other cases. The case when $b=3$ is different again and so this case is also treated separately.

The following lemma is also valid whenever $b=3$ but, as mentioned above, since the techniques we apply when $b=3$ are somewhat disparate from the rest of this subsection, we only prove it here whenever $b>3$.

\begin{lemma}\label{VA1}
Suppose that $C_{V_\beta}(V_{\alpha'})=V_\beta \cap Q_{\alpha'}$ and $b>3$. If $V_{\alpha}^{(2)}\le Q_{\alpha'-2}$ and $V_{\alpha'}\le Q_{\beta}$ then $R=Z_{\beta}\le Z_{\alpha'-1}$, $|V_{\beta}|=q^3$, $V_{\alpha}^{(2)}/Z_{\alpha}$ is an FF-module for $\bar{L_{\alpha}}$ and one of the following holds:
\begin{enumerate}
\item $V_{\alpha}^{(2)}\le Q_{\alpha'-1}$ and $[V_{\alpha}^{(2)}\cap Q_{\alpha'}, V_{\alpha'}]=Z_{\alpha'}\le V_{\alpha}^{(2)}$; or
\item $V_{\alpha}^{(2)}\not\le Q_{\alpha'-1}$, $[V_{\alpha}^{(2)}\cap Q_{\alpha'}, V_{\alpha'}]=\{1\}$, $R=Z_{\beta}=Z_{\alpha'-2}$ and $Z_{\alpha'-1}=Z_{\alpha'}\times Z_{\beta}$.
\end{enumerate}
\end{lemma}
\begin{proof}
Suppose first that $V_{\alpha}^{(2)}\le Q_{\alpha'-1}$. Then $V_{\alpha}^{(2)}=Z_{\alpha}(V_{\alpha}^{(2)}\cap Q_{\alpha'})$ and since $V_{\alpha}^{(2)}/Z_{\alpha}$ contains a non-central chief factor for $L_{\alpha}$, $[V_{\alpha}^{(2)}\cap Q_{\alpha'}, V_{\alpha'}]=Z_{\alpha'}\not\le Z_{\alpha}$. Then, for $\alpha'+1\in\Delta(\alpha')$ with $Z_{\alpha'+1}\not\le Q_{\alpha}$ it follows that $[Z_{\alpha'+1}, V_{\alpha}^{(2)}\cap Q_{\alpha'}\cap Q_{\alpha'+1}]=\{1\}$ and $V_{\alpha}^{(2)}/Z_{\alpha}$ contains a unique non-central chief factor which is an FF-module for $\bar{L_{\alpha}}$. Then by \cref{VBGood}, $|V_{\beta}|=q^3$, $[V_{\alpha}^{(2)}, Q_{\alpha}]=Z_{\alpha}$ and $Z_{\beta}=R\le Z_{\alpha'-1}\cap Z_{\alpha+2}$.

Suppose now that $V_{\alpha}^{(2)}\not\le Q_{\alpha'-1}$ and $|V_{\beta}|>q^3$. Then by \cref{VBGood}, both $V^\alpha/Z_{\alpha}$ and $V_{\alpha}^{(2)}/V^\alpha$ contain a non-central chief factor for $L_{\alpha}$. If $V^\alpha\not\le Q_{\alpha'-1}$, then $V^\alpha(V_{\alpha}^{(2)}\cap Q_{\alpha'})$ has index strictly less than $q$ in $V_{\alpha}^{(2)}$ and so, we have that $[V_{\alpha'}, V_{\alpha}^{(2)
}\cap Q_{\alpha'}]=Z_{\alpha'}\le V_{\alpha}^{(2)}$ and $Z_{\alpha'}\not\le V^\alpha$. In particular, $[V^{\alpha}\cap Q_{\alpha'}, V_{\alpha'}]=\{1\}$ and since $V^{\alpha}/Z_{\alpha}$ contains a non-central chief factor, we deduce that $V^{\alpha}Q_{\alpha'-1}\in\syl_p(L_{\alpha'-1})$. Since $b>3$, $V_{\alpha}^{(2)}$ is elementary abelian and $V_{\alpha}^{(2)}\not\le Q_{\alpha'-1}$, $Z_{\alpha'}=C_{Z_{\alpha'-1}}(V_{\alpha}^{(2)})=Z_{\alpha'-2}=[V^\alpha, Z_{\alpha'-1}]\le V^\alpha$, a contradiction. 

Thus, if $V_{\alpha}^{(2)}\not\le Q_{\alpha'-1}$ and $|V_{\beta}|>q^3$ then $V^{\alpha}\le Q_{\alpha'-1}$ and since $V^\alpha/Z_{\alpha}$ contains a non-central chief factor, it follows that $[V^\alpha\cap Q_{\alpha'}, V_{\alpha'}]=Z_{\alpha'}\le V^\alpha$, $(V^{\alpha}\cap Q_{\alpha'})Q_{\alpha'+1}\in\syl_p(L_{\alpha'+1})$ for any $\alpha'+1\in\Delta(\alpha')$ with $Z_{\alpha'+1}\not\le Q_{\alpha}$ and $V^\alpha/Z_{\alpha}$ is an FF-module for $\bar{L_{\alpha}}$. Since $V_{\alpha}^{(2)}\not\le Q_{\alpha'-1}$ and $V_{\alpha}^{(2)}$ is abelian, $Z_{\alpha'}=C_{Z_{\alpha'-1}}(V_{\alpha}^{(2)})=Z_{\alpha'-2}$, $(V_{\alpha}^{(2)}\cap Q_{\alpha'-1})/V^\alpha$ is centralized by $V_{\alpha'}$ and $V_{\alpha}^{(2)}/V^{\alpha}$ is also an FF-module for $\bar{L_{\alpha}}$. Then, applying \cref{GoodAction1} and \cref{SimExt} to $Z_{\alpha'}=Z_{\alpha'-2}$, we conclude that $V_{\alpha'}=V_{\alpha'-2}\le Q_{\alpha}$, a contradiction.

Thus, we assume now that $|V_{\beta}|=q^3$ and $Z_{\beta}=R\le Z_{\alpha'-1}\cap Z_{\alpha+2}$. Indeed, $Z_{\beta}=C_{Z_{\alpha'-1}}(V_{\alpha}^{(2)})=Z_{\alpha'-2}$ intersects $Z_{\alpha'}$ trivially, and $Z_{\alpha'-1}=Z_{\alpha'}\times Z_{\beta}$. If $[V_{\alpha}^{(2)}\cap Q_{\alpha'}, V_{\alpha'}]=Z_{\alpha'}\le V_{\alpha}^{(2)}$, then $Z_{\alpha'-1}=Z_{\beta}\times Z_{\alpha'}$ is centralized by $V_{\alpha}^{(2)}$ and $V_{\alpha}^{(2)}\le Q_{\alpha'-1}$, a contradiction. Thus, $[V_{\alpha}^{(2)}\cap Q_{\alpha'}, V_{\alpha'}]=\{1\}$, $(V_{\alpha}^{(2)}\cap Q_{\alpha'-1})/Z_{\alpha}$ is centralized by $V_{\alpha'}$ and $V_{\alpha}^{(2)}/Z_{\alpha}$ is an FF-module for $\bar{L_{\alpha}}$.
\end{proof}

\begin{lemma}\label{VB1}
Suppose that $C_{V_\beta}(V_{\alpha'})=V_\beta \cap Q_{\alpha'}$ and $b>5$. If $V_{\alpha'}\not\le Q_{\beta}$ and $V_{\alpha}^{(2)}\le Q_{\alpha'-2}$, then $|V_{\beta}|=q^3$.
\end{lemma}
\begin{proof}
Suppose throughout that $|V_{\beta}|\ne q^3$ so that both $V^\alpha/Z_{\alpha}$ and $V_{\alpha}^{(2)}/V^\alpha$ contain a non-central chief factor for $L_{\alpha}$. Choose $\alpha'+1\in\Delta(\alpha')$ with $Z_{\alpha'+1}\not\le Q_{\beta}$. In particular, $(\alpha'+1, \beta)$ is a critical pair and we may assume that $C_{V_{\alpha'}}(V_{\beta})=V_{\alpha'}\cap Q_{\beta}$. Set $U^\beta:=\langle V^\lambda\mid \lambda\in\Delta(\beta), Z_{\lambda}=Z_{\alpha}\rangle$ so that $R_{\beta}Q_{\alpha}$ normalizes $U^\beta$ by \cref{UWNormal}. Setting $U^{\alpha'}:=\langle V^{\mu}\mid \mu\in\Delta(\alpha'), Z_{\mu}=Z_{\alpha'+1}\rangle$, it follows similarly that $U^{\alpha'}\normaleq R_{\alpha'}Q_{\alpha'+1}$. Throughout, for $\mu\in\beta^G$, we set $U_{\mu}:=\langle (V^{\mu+1})^{L_{\mu}}\rangle$ where $\mu+1\in\Delta(\mu)$. In particular, $U^\beta\le U_\beta\normaleq L_{\beta}$. Note that for $\lambda\in\Delta(\beta)$ whenever $Z_{\alpha}=Z_{\lambda}$, we have that $Q_{\alpha}R_{\beta}=Q_{\lambda}R_{\beta}$. Furthermore, $Q_{\alpha}O^p(R_{\beta})=C_{L_{\beta}}(Z_{\alpha})=C_{L_{\beta}}(Z_{\lambda})=Q_{\lambda}O^p(R_{\beta})$, $Q_{\alpha}\in\syl_p(Q_\lambda O^p(R_{\beta}))$, $Q_\alpha$ is conjugate to $Q_\lambda$ by an element of $O^p(R_\beta)$ and since $O^p(R_{\beta})$ centralizes $Q_{\beta}/C_{\beta}$, $Q_{\alpha}\cap Q_{\beta}=Q_{\lambda}\cap Q_{\beta}$. In particular, $[Q_{\alpha}\cap Q_{\beta}, U^\beta]=Z_{\alpha}$.

We will use throughout that if $R\le Z_{\alpha'-1}$, then $Z_{\alpha'-1}Z_{\alpha'-1}^g$ is normalized by $L_{\alpha'}=\langle V_{\beta}, (V_{\beta})^g, R_{\alpha'}\rangle$ for some suitable $g\in L_{\alpha'}$. Then, from the definition of $V_{\alpha'}$, we conclude that $V_{\alpha'}=Z_{\alpha'-1}Z_{\alpha'-1}^g$ is of order $q^3$, as required. A similar conclusion follows if $R\le Z_{\alpha+2}$.

Suppose first that $U^{\beta}\not\le Q_{\alpha'-2}$ and so there is some $\lambda\in\Delta(\beta)$ with $V^{\lambda}\not\le Q_{\alpha'-2}$ and $Z_{\lambda}=Z_{\alpha}$. In particular, since $V_{\alpha'-2}\le Q_{\lambda}$ and $Z_{\alpha}\cap Q_{\alpha'}=Z_{\beta}$, we deduce that $[V_{\alpha'-2}, V^{\lambda}]=Z_{\beta}\le V_{\alpha'-2}$ and $Z_{\alpha'-2}\cap Z_{\beta}=\{1\}$. If, in addition, $U_{\alpha'-2}\not\le Q_{\beta}$, then there is $\delta\in\Delta(\alpha'-2)$ with $V^\delta\not\le Q_{\beta}$ and $[V^\delta, V_{\beta}]\le Z_{\delta}$. In particular, it follows that $R\le [V^\delta, V_{\beta}]\le Z_{\delta}$ and since $R\not\le Z_{\alpha'-1}$, otherwise $|V_{\alpha'}|=q^3$, it follows that $Z_{\alpha'-2}<R Z_{\alpha'-2}\le Z_{\delta}$ and $RZ_{\alpha'-2}$ centralizes $V^{\lambda}$. But $V^{\lambda}\not\le Q_{\alpha'-2}$ and since $V^\lambda$ centralizes $Z_{\alpha'-3}C_{V_{\alpha'-2}}(O^p(L_{\alpha'-2}))$ we have that $Z_{\delta}=Z_{\alpha'-3}$ by \cref{VBGood}. But now, $RZ_{\alpha'-2}\le V_{\alpha'}$ and it follows that $Z_{\alpha'-3}\le V_{\alpha'-2}\cap V_{\alpha'}$ and again by \cref{VBGood}, we conclude that $R\le Z_{\alpha'-3}=Z_{\alpha'-1}$, a contradiction.

If $U_{\alpha'-2}\le Q_{\beta}$, then for any $\delta\in\Delta(\alpha'-2)$, $[V^\delta, V_{\beta}]\le Z_{\delta}$ by \cref{VBGood}. If $[V^\delta, V_{\beta}]\ne \{1\}$ for some $\delta$ then $Z_{\beta}=[V^\delta, V_{\beta}]\le Z_{\delta}$, $[V^\lambda, V_{\alpha'-2}]= Z_{\beta}\le Z_{\delta}$ and $|V_{\alpha'-2}|=q^3$, a contradiction. Thus, $[U_{\alpha'-2}, V_{\beta}]=\{1\}$ and $U_{\alpha'-2}\le Q_{\lambda}$ so that $[U_{\alpha'-2}, V^{\lambda}]=Z_{\lambda}\cap U_{\alpha'-2}=Z_{\alpha}\cap U_{\alpha'-2}\le Z_{\beta}\le V_{\alpha'-2}$ by \cref{VBGood}, and $V^\lambda$ centralizes $U_{\alpha'-2}/V_{\alpha'-2}$. But then $O^p(L_{\alpha'-2})$ centralizes $U_{\alpha'-2}/V_{\alpha'-2}$, a contradiction by \cref{VBGood}, for then $V^{\alpha'-1}V_{\alpha'-2}\normaleq L_{\alpha'-2}$. Thus, $U^\beta\le Q_{\alpha'-2}$. Notice that the hypothesis $V_{\alpha}^{(2)}\le Q_{\alpha'-2}$ is not involved in the above arguments and so we may repeat the above arguments to conclude also that $U^{\alpha'}\le Q_{\alpha+3}$.

Assume now that $U^{\beta}\le Q_{\alpha'-2}$ but $U^\beta\not\le Q_{\alpha'-1}$. Then, as $Z_{\alpha'-1}\le Q_{\alpha}$, it follows by \cref{VBGood} that $Z_{\alpha'-2}=[U^\beta, Z_{\alpha'-1}]\le Z_{\alpha}$ and $Z_{\alpha'-2}=Z_{\beta}=Z_{\alpha}\cap Q_{\alpha'}$. Then $[V^{\alpha'-1}, V_{\beta}]\le Z_{\alpha'-1}\cap V_{\beta}$ and since $V_{\beta}U^\beta\le V_{\beta}^{(3)}$ is abelian, it follows that $[V^{\alpha'-1}, V_{\beta}]\le Z_{\alpha'-2}=Z_{\beta}$ and $V^{\alpha'-1}\le Q_{\beta}$. If $V^{\alpha'-1}\le Q_{\lambda}$ for some $\lambda\in\Delta(\beta)$ with $Z_{\lambda}=Z_{\alpha}$ and $V^\lambda\not\le Q_{\alpha'-1}$, then $[V^{\alpha'-1}, V^\lambda]\le Z_{\lambda}\cap Q_{\alpha'}=Z_{\beta}=Z_{\alpha'-2}\le Z_{\alpha'-1}$, a contradiction since $V^\lambda\not\le Q_{\alpha'-1}$. Therefore $V^{\alpha'-1}\not\le Q_{\lambda}$ and as \[[V^{\lambda}\cap Q_{\alpha'-1}, V^{\alpha'-1}]\le Z_{\alpha'-1}\cap V^{\lambda}=C_{Z_{\alpha'-1}}(U^\beta)=Z_{\alpha'-2}=Z_{\beta}\le Z_{\alpha}=Z_{\lambda},\] $V^{\lambda}/Z_{\lambda}$ is an FF-module for $\bar{L_{\lambda}}$ and $V^\lambda Q_{\alpha'-1}\in\syl_p(L_{\alpha'-1})$. Moreover, $V_{\lambda}^{(2)}\cap Q_{\alpha'-2}=V^\lambda(V_{\lambda}^{(2)}\cap Q_{\alpha'-1})$, $V_{\lambda}^{(2)}/V^{\lambda}$ is also an FF-module for $\bar{L_{\lambda}}$ and $V_{\lambda}^{(2)}Q_{\alpha'-2}\in \syl_p(L_{\alpha'-2})$. Then \cref{GoodAction1} implies that $O^p(R_{\lambda})$ centralizes $V_{\lambda}^{(2)}$.

By \cref{SimExt}, $Z_{\alpha+3}\ne Z_{\beta}=Z_{\alpha'-2}$ and so $V_{\alpha'}^{(3)}\cap Q_{\alpha+3}$ centralizes $Z_{\alpha+2}=Z_{\alpha+3}\times Z_{\alpha'-2}$ and $V_{\alpha'}^{(3)}\cap Q_{\alpha+3}=V_{\alpha'}(V_{\alpha'}^{(3)}\cap Q_{\beta})$. Since $Z_{\beta}\le V_{\alpha'}$, have that $V_{\alpha'}^{(3)}Q_{\alpha+3}\in\syl_p(L_{\alpha+3})$, $V_{\alpha'}^{(3)}/V_{\alpha'}$ contains a unique non-central chief factor for $L_{\alpha'}$ which, as a $\bar{L_{\alpha'}}$-module, is itself an FF-module and, by \cref{GoodAction3} and conjugacy, $O^p(R_{\beta})$ centralizes $V_{\beta}^{(3)}$. By \cref{SimExt}, $Z_{\alpha}=Z_{\lambda}$ implies that $V^\alpha=V^{\lambda}=U^\beta$ and $V_{\alpha}^{(2)}=V_{\lambda}^{(2)}$. But $V_{\alpha}^{(2)}\le Q_{\alpha'-2}$, and this is a contradiction.

Thus, we may assume for the remainder of this proof that $U^\beta\le Q_{\alpha'-1}$. If $[U^\beta\cap Q_{\alpha'}, V_{\alpha'}]\le V_{\beta}U^\beta$, then $V_{\alpha'}$ normalizes $V_{\beta}U^\beta$ and so $U_\beta=V_{\beta}U^\beta\normaleq L_{\beta}=\langle V_{\alpha'}, R_{\beta}, Q_{\alpha}\rangle$. Since $U_\beta /V_{\beta}$ contains a non-central chief factor for $L_{\beta}$, we have that $Z_{\alpha'}\le U_\beta$, $Z_{\alpha'}\not\le V_{\beta}$ and $U^\beta\cap Q_{\alpha'}\not\le Q_{\alpha'+1}$. 

Assume first that $Z_{\alpha'}=Z_{\alpha'-2}$ and $q=p$. Then, $(U^\beta\cap Q_{\alpha'})Q_{\alpha'+1}\in\syl_p(L_{\alpha'+1})$. Now, $U^{\alpha'}\le Q_{\alpha+3}$ and if $U^{\alpha'}\not\le Q_{\alpha+2}$, then $Z_{\alpha+3}=[U^{\alpha'}, Z_{\alpha+2}]\le Z_{\alpha'+1}$ from which we deduce that $Z_{\alpha'}=Z_{\alpha+3}\le V_{\beta}$. But then, $Z_{\alpha'+1}$ centralizes $U^\beta V_\beta/V_\beta$ and we deduce that $V^\alpha V_\beta\normaleq L_{\beta}$, a contradiction by \cref{VBGood}. Hence, $U^{\alpha'}\le Q_{\alpha+2}$ so that $V^{\alpha'+1}=Z_{\alpha'+1}(V^{\alpha'+1}\cap Q_{\beta})$. If $V^{\alpha'+1}\cap Q_{\beta}\le Q_{\alpha}$, then $[V^{\alpha'+1}\cap Q_{\beta}, U^\beta\cap Q_{\alpha'}]\le Z_{\alpha}\cap V^{\alpha'+1}$ and since $V^{\alpha+1}/Z_{\alpha'+1}$ contains a non-central chief factor we must have that $Z_{\beta}\le V^{\alpha'+1}$ and $V^{\alpha'+1}/Z_{\alpha'+1}$ is (dual to) an FF-module. But then $U^{\alpha'}V_{\alpha'}\normaleq L_{\alpha'}=\langle V_{\beta}, Q_{\alpha'+1}, R_{\alpha'}\rangle$. Since $Q_{\alpha'}\cap Q_{\alpha'+1}\le C_{Q_{\alpha'}}(U^{\alpha'}V_{\alpha'}/V_{\alpha'})$, we infer by \cref{push} that $[Q_{\alpha'}, U_{\alpha'}]=[Q_{\alpha'}, U^{\alpha'}V_{\alpha'}]\le V_{\alpha'}$. But then $[V^{\alpha'+1}\cap Q_{\beta}, U^\beta\cap Q_{\alpha'}]\le Z_{\alpha}\cap V_{\alpha'}$ so that $Z_{\beta}\le V_{\alpha'}$ and $V_{\beta}$ centralizes $U^{\alpha'}V_{\alpha'}/V_{\alpha'}$, a contradiction by \cref{VBGood}.

Hence, $V^{\alpha'+1}\cap Q_{\beta}\not\le Q_{\alpha}$. In particular, $Z_{\beta}=[Z_{\alpha}, V^{\alpha'+1}\cap Q_{\beta}]\le U_{\alpha'}$. Since $V^{\alpha}/Z_{\alpha}$ contains a non-central chief factor and $[V^{\alpha}\cap Q_{\alpha'+1}, V^{\alpha'+1}\cap Q_{\beta}]\le Z_{\alpha'}\cap V^{\alpha'+1}=\{1\}$, we must have that $V^{\alpha}\cap Q_{\alpha'}\not\le Q_{\alpha'+1}$. Furthermore, since $Z_{\alpha'}=Z_{\alpha'-2}\not\le V^{\alpha}$, we have that $[V_{\alpha}^{(2)}, Z_{\alpha'-1}]=\{1\}$ and $V_{\alpha}^{(2)}=V^{\alpha'+1}(V_{\alpha}^{(2)}\cap Q_{\alpha'+1})$. Hence, $Z_{\alpha'}\le V_{\alpha}^{(2)}$, $V_{\alpha}^{(2)}/V^\alpha$ is an FF-module for $\bar{L_{\alpha}}$ and by \cref{GoodAction1}, $O^p(R_{\alpha})$ centralizes $V_{\alpha}^{(2)}$. Then, \cref{SimExt} applied to $Z_{\alpha'}=Z_{\alpha'-2}$ yields $V_{\alpha'}=V_{\alpha'-2}\le Q_{\beta}$, a contradiction.

Assume now that $Z_{\alpha'}=Z_{\alpha'-2}$ and $q>p$. Since $Z_{\alpha'}\not\le V_\beta$ and $[U_\beta, C_{\beta}]\le V_{\beta}$, we have that $U_\beta\cap Q_{\alpha'-2}\le Q_{\alpha'-1}$ and as $Z_{\alpha'+1}$ acts quadratically on $U_\beta$, we deduce by \cref{2FRecog} that $O^p(R_{\beta})$ centralizes $U_\beta$. Since $[V_{\beta}^{(3)}\cap Q_{\alpha'}, V_{\alpha'}]\le Z_{\alpha'}\le U_\beta$, we also deduce by \cref{2FRecog} that $O^p(R_{\beta})$ centralizes $V_{\beta}^{(3)}$. In particular, $U^\beta=V^\alpha$ and by \cref{SimExt} and \cref{VBGood}, $Z_{\alpha+2}C_{V_{\alpha+3}}(O^p(L_{\alpha+3}))\ne Z_{\alpha+4}C_{V_{\alpha+3}}(O^p(L_{\alpha+3}))$.

Note that $[V_{\alpha}^{(2)}, Z_{\alpha'-1}]\le Z_{\alpha'}\cap V^\alpha=\{1\}$ so that $V_{\alpha}^{(2)}=V_{\beta}(V_{\alpha}^{(2)}\cap Q_{\alpha'})$. Since $Z_{\alpha'}\le U_\beta$, we deduce that $V_{\beta}^{(3)}=V_{\alpha}^{(2)}U_\beta\normaleq \langle V_{\alpha'}, Q_{\alpha}, R_{\beta}\rangle$. In particular, by conjugation, $V_{\alpha'-2}^{(3)}=U_{\alpha'-2}V_{\alpha'-3}^{(2)}$ so that $R=[V_{\beta}, V_{\alpha'}]\le [V_{\beta}, U_{\alpha'-2}]$. Hence, there is $\delta\in\Delta(\alpha'-2)$ such that $Z_\delta\ne Z_{\alpha'-3}$ and $[V_\beta, V^\delta]\le V_{\beta}\cap Z_{\delta}\not\le Z_{\alpha'-2}$.

Since $U_\beta$ centralizes $Z_{\alpha'-3}(V_{\beta}\cap Z_\delta)$ we infer that $V_{\beta}^{(3)}\le Q_{\alpha'-2}$ and since $U_\beta\le Q_{\alpha'-1}$, we have that $V_{\beta}^{(3)}=V_{\alpha}^{(2)} U_\beta=V_{\beta}(V_{\beta}^{(3)}\cap Q_{\alpha'})$. Then $U_\beta/V_{\beta}$ contains a unique non-central chief factor for $L_{\beta}$, $Z_{\alpha'}\le U_\beta$ and $Z_{\alpha'}\cap V_{\beta}=\{1\}$. Since $U_\beta/V_{\beta}$ contains a unique non-central chief factor for $L_{\beta}$, writing $C$ for the preimage in $U_\beta$ of $C_{U_\beta/V_{\beta}}(O^p(L_{\beta}))$, we have that $U_\beta=CV^\alpha V^{\alpha+2}$, $Z_{\alpha'}\cap C=\{1\}$ and $V^\alpha \cap V^{\alpha+2}\le C$.

Note that $[C\cap Q_{\alpha'}, Z_{\alpha'+1}]=\{1\}$ and $[V^{\alpha+2}\cap Q_{\alpha'}, Z_{\alpha'+1}]\le Z_{\alpha+2}\cap Z_{\alpha'}=\{1\}$ from which we deduce that $U_\beta=V_{\beta}V^\alpha (U_\beta\cap Q_{\alpha'+1})$ and $(V^{\alpha}\cap Q_{\alpha'})Q_{\alpha'+1}\in\syl_p(L_{\alpha'+1})$. Then $V_{\alpha}^{(2)}=V^{\alpha}(V_{\alpha}^{(2)}\cap V^{\alpha'+1})$ and we also infer by a similar reasoning that $S=(V^{\alpha'+1}\cap Q_{\beta})Q_{\alpha}$. Since both $V_{\alpha}^{(2)}/V^{\alpha}$ and $V^{\alpha}/Z_{\alpha}$ contain a non-central chief factor for $L_{\alpha}$, we deduce that both $V_{\alpha}^{(2)}/V^{\alpha}$ and $V^{\alpha}/Z_{\alpha}$ are FF-modules for $\bar{L_{\alpha}}$ and applying \cref{GoodAction1}, $O^p(R_{\alpha})$ centralizes $V_{\alpha}^{(2)}$. Then \cref{SimExt} applied to $Z_{\alpha'}=Z_{\alpha'-2}$ provides a contradiction as in the $q=p$ case.

Hence, we have shown that if $|V_{\beta}|>q^3$ then $Z_{\alpha'}\ne Z_{\alpha'-2}$. Note that if $R\le V_{\alpha'-2}$, then as $R\not\le Z_{\alpha'-1}$, $RZ_{\alpha'-1}\le V_{\alpha'}\cap V_{\alpha'-2}\le Z_{\alpha'-1}C_{V_{\alpha'}}(O^p(L_{\alpha'}))$. Then $[R, Q_{\alpha'}]\le Z_{\alpha'}$ and $[R, Q_{\alpha'-2}]\le Z_{\alpha'-2}$ and so $RZ_{\alpha'-1}\normaleq L:=\langle Q_{\alpha'}, Q_{\alpha'-2}\rangle$. But now, $[R, Q_{\alpha'-1}\cap Q_{\alpha'}]=[R, V_{\beta}(Q_{\alpha'-1}\cap Q_{\alpha'})]=[RZ_{\alpha'-1}, Q_{\alpha'-1}]\normaleq L$ and since $R\le V_{\alpha'}$, $[RZ_{\alpha'-1}, Q_{\alpha'-1}]\le Z_{\alpha'}$ and $R\le \Omega(Z(Q_{\alpha'-1}))\cap V_{\alpha'}$. Assume $\Omega(Z(Q_{\alpha'-1}))\cap C_{V_{\alpha'}}(O^p(L_{\alpha'}))>Z_{\alpha'}$. Then $Q_{\alpha'-1}\cap Q_{\alpha'}\le C_{Q_{\alpha'}}(\Omega(Z(Q_{\alpha'-1}))\cap C_{V_{\alpha'}}(O^p(L_{\alpha'})))\normaleq G_{\alpha'}$ and \cref{push} yields a contradiction. Hence, $\Omega(Z(Q_{\alpha'-1}))\cap V_{\alpha'}=Z_{\alpha'-1}$, a contradiction since $R\not\le Z_{\alpha'-1}$. Thus, $R\not\le V_{\alpha'-2}$.

Now, if $q>p$ then $V_{\beta}^{(3)}\cap Q_{\alpha'-2}$ centralizes $Z_{\alpha'-1}$ so that $V_{\beta}(V_{\beta}^{(3)}\cap Q_{\alpha'+1})$ has index at most $q^2$ in $V_{\beta}^{(3)}$ and is centralized, modulo $V_{\beta}$, by $Z_{\alpha'+1}$. Furthermore, $Z_{\alpha'+1}$ acts quadratically on $V_{\beta}^{(3)}$ and by \cref{2FRecog} we deduce that $O^p(R_{\beta})$ centralizes $V_{\beta}^{(3)}$. Then by \cref{SimExt}, $Z_{\alpha'-1}\ne Z_{\alpha'-3}$ for otherwise $V_{\alpha'}\le V_{\alpha'-1}^{(2)}=V_{\alpha'-3}^{(2)}\le Q_{\alpha'}$. Hence, by \cref{VBGood}, $Z_{\alpha'-1}C_{V_{\alpha'-2}}(O^p(L_{\alpha'-2}))\ne Z_{\alpha'-3}C_{V_{\alpha'-2}}(O^p(L_{\alpha'-2}))$ and since $V_{\beta}^{(3)}$ centralizes $Z_{\alpha'-1}$, $V_{\beta}^{(3)}=V_{\beta}(V_{\beta}^{(3)}\cap Q_{\alpha'})$ and $V_{\beta}^{(3)}/V_{\beta}$ contains a unique non-central chief factor for $L_{\beta}$ which, as a $\bar{L_{\beta}}$-module, is an FF-module. Since $V^\alpha V_{\beta}\not\normaleq L_{\beta}$, we deduce that $U_\beta/V_{\beta}$ contains the non-central chief factor and $V_{\beta}^{(3)}=U_\beta V_{\alpha}^{(2)}$. But then, by conjugacy, $V_{\alpha'}\le V_{\alpha'-2}^{(3)}=V_{\alpha'-3}^{(2)}U_{\alpha'-2}$ and since $V_{\beta}$ centralizes $V_{\alpha'-3}^{(2)}$, $R=[V_{\beta}, V_{\alpha'}]\le [V_{\beta}, U_{\alpha'-2}]\le V_{\alpha'-2}$, a contradiction.

If $q=p$, then the above observations yield that $V_{\beta}^{(3)}/U_\beta$ contains a non-central chief factor for $L_{\beta}$. We deduce that $U_\beta\le Q_{\alpha'-2}$ and as $Z_{\alpha'-1}=Z_{\alpha'}\times Z_{\alpha'-2}$ is centralized by $U_\beta$, $U_\beta\le Q_{\alpha'-1}$. Then, as $V^{\alpha}V_{\beta}\not\normaleq L_{\beta}$ by \cref{VBGood}, $U_{\beta}/V_{\beta}$ contains a unique non-central chief factor and $(U_\beta\cap Q_{\alpha'}) Q_{\alpha'+1}\in\syl_p(L_{\alpha'+1})$. Moreover, by a similar argument, $V_{\beta}^{(3)}\cap Q_{\alpha'-2}\le Q_{\alpha'-1}$, $V_{\beta}^{(3)}Q_{\alpha'-2}\in\syl_p(L_{\alpha'-2})$ and as $V_{\beta}^{(3)}$ centralizes $Z_{\alpha'-1}$, $Z_{\alpha'-1}=Z_{\alpha'-3}$ by \cref{VBGood}. Hence, to force a contradiction via \cref{SimExt} as before, we need only show that $O^p(R_{\beta})$ centralizes $V_{\beta}^{(3)}$. Note that both $V_{\beta}^{(3)}/U_{\beta}$ and $U_\beta/V_{\beta}$ contain exactly one non-central chief factor for $L_{\beta}$ and in both cases, the non-central chief factor is an FF-module for $\bar{L_{\beta}}$.

Set $R_1:=C_{L_{\beta}}(U_{\beta}/V_{\beta})$ and $R_2:=C_{L_{\beta}}(V_{\beta}^{(3)}/U_{\beta})$. Since the non-central chief factor within $V_{\beta}^{(3)}/U_{\beta}$ is an FF-module, it follows that either $R_2Q_{\beta}=R_{\beta}$; or $L_{\beta}=\langle R_2, R_{\beta}, S\rangle$ and $q\in\{2,3\}$ by \cref{Badp2} (ii), (iii) and \cref{Badp3} (ii), (iii). In the former case, since $V_{\alpha}^{(2)}\le Q_{\alpha'-1}$, $V_{\beta}^{(3)}=V_{\alpha}^{(2)}U_\beta\normaleq L_{\beta}=\langle V_{\alpha'}, R_{\beta}, Q_{\alpha}\rangle$ and $V_{\beta}^{(3)}\le Q_{\alpha'-2}$, a contradiction. In the latter case, $V_{\alpha}^{(2)}U_{\beta}\normaleq R_2S$ and if $[C_{\beta}, V_{\alpha}^{(2)}U_{\beta}]\le V_{\beta}$, then $[C_{\beta}, V_{\alpha}^{(2)}U_{\beta}]$ is centralized by $O^p(R_{\beta})$ and so $[C_{\beta}, V_{\alpha}^{(2)}U_{\beta}]\normaleq L_{\beta}=\langle R_2, R_{\beta}, S\rangle$. Thus, $[C_{\beta}, V_{\beta}^{(3)}]=[C_{\beta}, V_{\alpha}^{(2)}U_{\beta}]\le V_{\beta}$ and by conjugacy, $R\le [V_{\alpha'-2}^{(3)}, V_{\beta}]\le [V_{\alpha'-2}^{(3)}, C_{\alpha'-2}]\le V_{\alpha'-2}$, a contradiction. Thus, $[C_{\beta}, V_{\alpha}^{(2)}]\le V^{\alpha}$ but $[C_{\beta}, V_{\alpha}^{(2)}]\not\le V_{\beta}$. If $R_1Q_{\beta}=R_2Q_{\beta}$ then, assuming that $G$ is a minimal counterexample to \hyperlink{MainGrpThm}{Theorem C}, we may apply \cref{SubAmal} with $\lambda=\beta$. Since $b>5$, $R_1Q_{\beta}$ normalizes $V_{\alpha}^{(2)}$ and $\lambda=\beta$, conclusion (d) holds. Then, $V_{\alpha}^{(4)}\le V:=\langle Z_{\beta}^X\rangle$ and the images of $Q_{\alpha}/C_{Q_{\alpha}}(V_{\alpha}^{(2)})$ and $C_{Q_{\alpha}}(V_{\alpha}^{(2)})/C_{Q_{\alpha}}(V_{\alpha}^{(4)})$ resp. $Q_{\beta}/C_{\beta}$ and $C_{\beta}/C_{Q_{\beta}}(V_{\beta}^{(3)})$ contain a non-central chief factor for $\wt L_{\alpha}$ resp. $\wt L_{\beta}$, and we have a contradiction.

Thus, we may assume that $R_1Q_{\beta}\ne R_2Q_{\beta}\ne R_{\beta}$ and again by \cref{Badp2} (iii) and \cref{Badp3} (iii), we deduce that $L_{\beta}=\langle R_1, R_2, S\rangle$. Then $V_{\alpha}^{(2)}U_\beta\normaleq R_2S$ so that $V^{\alpha}V_{\beta}\ge [C_{\beta}, V_{\alpha}^{(2)}U_\beta]V_{\beta}\normaleq R_2S$. Furthermore, as $O^p(R_1)$ centralizes $U_{\beta}/V_{\beta}$, $[C_{\beta}, V_{\alpha}^{(2)}U_\beta]V_{\beta}\normaleq R_1S$ so that $[C_{\beta}, V_{\alpha}^{(2)}U_\beta]V_{\beta}\normaleq L_{\beta}$. Since $V^{\alpha}V_{\beta}\not\normaleq L_{\beta}$, we may assume that $[C_{\beta}, V_{\alpha}^{(2)}]V_{\beta}< V^{\alpha}V_{\beta}$. Note that $[V^{\alpha'+1}\cap Q_{\alpha+3}, Z_{\alpha+2}]\le Z_{\alpha+3}\cap Z_{\alpha'}=\{1\}$ since $Z_{\alpha'}\not\le V_{\beta}$. Then, if $V^{\alpha'+1}\cap Q_{\beta}\le Q_{\alpha}$, we have that $[V^{\alpha'+1}\cap Q_{\beta}, U^\beta\cap Q_{\alpha'}]\le Z_{\beta}\cap V^{\alpha'+1}$. By \cref{SEFF}, by conjugacy, either $V^\alpha/Z_{\alpha}$ is an FF-module, or $V^{\alpha'+1}\not\le Q_{\alpha+3}$ and $Z_{\beta}\le V^{\alpha'+1}$. In the latter case, $V_{\alpha'+1}^{(2)}=V^{\alpha'+1}(V_{\alpha'+1}^{(2)}\cap Q_{\alpha})$ and as $Z_{\beta}\le V^{\alpha'+1}$, we have a contradiction since $V_{\alpha'+1}^{(2)}/V^{\alpha'+1}$ contains a non-central chief factor. Now, $V^{\alpha}/Z_{\alpha}$ is an FF-module generated by $C_{V_{\beta}}(O^p(L_{\beta}))/Z_{\alpha}$ of order $p$ so that by \cref{SEFF}, $p^2\leq |V^{\alpha}/Z_{\alpha}|\leq p^3$ and $p^4\leq |V^{\alpha}|\leq p^5$. Hence, $p^5\leq |V^{\alpha}V_{\beta}|\leq p^6$, accordingly. But now, as $[C_{\beta}, V_{\alpha}^{(2)}U_\beta]V_{\beta}>V_{\beta}$, $|[C_{\beta}, V_{\alpha}^{(2)}U_\beta]V_{\beta}|\leq p^5$ and as $[C_{\beta}, V_{\alpha}^{(2)}]V_{\beta}< V^{\alpha}V_{\beta}$, we get that $|V^{\alpha}|=p^5$, $|V^{\alpha}V_{\beta}|=p^6$ and $[Q_{\beta}, V^\alpha]\not\le Z_{\alpha}C_{V_{\beta}}(O^p(L_{\beta}))$.

Writing $C^\alpha$ for the preimage in $V^\alpha$ of $C_{V^{\alpha}/Z_{\alpha}}(O^p(L_{\alpha}))$, we have that $|C^\alpha|=p^3$, $C^\alpha\cap V_{\beta}=Z_{\alpha}$,  $|Q_{\alpha}/C_{Q_{\alpha}}(C^\alpha)|\leq p^2$ and a calculation using the three subgroup lemma yields $[R_{\alpha}, Q_{\alpha}]\le C_{Q_{\alpha}}(C^\alpha)$. Since $Z(Q_{\alpha})=Z_{\alpha}$, calculating in $\GL_3(p)$, we infer that $Q_{\alpha}/C_{Q_{\alpha}}(C^\alpha)$ is a non-central chief factor of order $p^2$ for $L_{\alpha}$. Hence, $Q_{\alpha}/C_{Q_{\alpha}}(C^\alpha)$ is a natural $\SL_2(p)$ module for $L_{\alpha}/R_{\alpha}$.

Now, by \cref{CommCF}, $U_{\beta}/([U_{\beta}, Q_{\beta}]V_{\beta})$ contains the unique non-central chief factor within $U_{\beta}/V_{\beta}$ and so $O^p(L_{\beta})$ centralizes $[U_\beta,Q_{\beta}]V_{\beta}/V_{\beta}$. Thus, $[V^{\alpha}, Q_{\beta}]V_{\beta}\normaleq L_{\beta}$ from which it follows that $Z_{\alpha}\ge[V^{\alpha}, Q_{\beta}, Q_{\beta}]\normaleq L_{\beta}$ and $[V^\alpha, Q_{\beta}, Q_{\beta}]=Z_{\beta}$. But $C^\alpha\le Z_{\alpha}C_{V_{\beta}}(O^p(L_{\beta}))[V^\alpha, Q_{\beta}]$ so that $[Q_{\beta}, C^\alpha]=Z_{\beta}$. In particular, $C_{Q_{\alpha}}(C^\alpha)\le Q_{\beta}$ for otherwise $Z_{\beta}=[C^\alpha, Q_{\alpha}\cap Q_{\beta}]=[C^\alpha, Q_{\alpha}]\normaleq L_{\alpha}$, a contradiction. 

If $V^{\alpha'-1}\not\le Q_{\beta}$, then $RZ_{\beta}\le [V^{\alpha'-1}, V_{\beta}]Z_{\beta}\le Z_{\alpha'-1}Z_{\beta}$. Then, as $R\not\le Z_{\alpha'-1}$, we get that $Z_{\beta}\le RZ_{\alpha'-1}\le V_{\alpha'}$. If $V^{\alpha'-1}\le Q_{\beta}$ but $V_{\beta}\not\le C_{\beta}$, we deduce that $Z_{\beta}=[V^{\alpha'-1}, V_{\beta}]\le Z_{\alpha'-1}$. In either case, since $O^p(R_{\alpha})$ centralizes $V_{\alpha}^{(2)}$, by \cref{SimExt}, $Z_{\beta}\ne Z_{\alpha+3}$ and so $V_{\alpha'}^{(3)}$ centralizes $Z_{\alpha+2}=Z_{\beta}Z_{\alpha+3}$. But then $V_{\alpha'}^{(3)}\cap Q_{\alpha+3}=V_{\alpha'}(V_{\alpha'}^{(3)}\cap Q_{\beta})$ and since $Z_{\beta}\le Z_{\alpha'-1}\le V_{\alpha'}$, $V_{\alpha'}^{(3)}/V_{\alpha'}$ contains a unique non-central chief factor, a contradiction. Thus, $[V_{\beta}, V^{\alpha'-1}]=\{1\}$ and $V_{\beta}\le C_{Q_{\alpha'-1}}(C^{\alpha'-1})\le Q_{\alpha'}$, a final contradiction.
\end{proof}

\begin{lemma}\label{VB2}
Suppose that $C_{V_\beta}(V_{\alpha'})=V_\beta \cap Q_{\alpha'}$ and $b>5$. If $V_{\alpha'}\not\le Q_{\beta}$ and $V_{\alpha}^{(2)}\le Q_{\alpha'-2}$, then $Z_{\alpha'-1}\le V_{\beta}^{(3)}\le Q_{\alpha'-1}$, $Z_{\alpha'}\not\le V_{\alpha}^{(2)}$, $V_{\beta}^{(3)}/V_{\beta}$ contains a unique non-central chief factor for $\bar{L_{\beta}}$ which, as a $\bar{L_{\beta}}$-module, is an FF-module and $O^p(R_{\beta})$ centralizes $V_{\beta}^{(3)}$.
\end{lemma}
\begin{proof}
By \cref{VB1}, $|V_{\beta}|=q^3$ so that $R=[V_{\beta}, V_{\alpha'}]\le Z_{\alpha'-1}\cap Z_{\alpha+2}$. Suppose first that $V_{\alpha}^{(2)}\not\le Q_{\alpha'-1}$. Then $Z_{\alpha'-2}=[V_{\alpha}^{(2)}, V_{\alpha'-2}]\le Z_{\alpha}\cap Q_{\alpha'-2}$, so that $Z_{\beta}=Z_{\alpha'-2}$. But $Z_{\beta}\cap  R=\{1\}$ and $Z_{\alpha'-1}=R\times Z_{\beta}\le V_{\beta}$, a contradiction since $V_{\alpha}^{(2)}$ is abelian. Thus, we may assume throughout that $V_{\alpha}^{(2)}\le Q_{\alpha'-1}$. 

Suppose that $V_{\beta}^{(3)}\cap Q_{\alpha'-2}\le Q_{\alpha'-1}$. If $V_{\beta}^{(3)}\le Q_{\alpha'-2}$, then $V_{\beta}^{(3)}=V_{\beta}(V_{\beta}^{(3)}\cap Q_{\alpha'})$. Since $O^p(L_{\beta})$ does not centralize $V_{\beta}^{(3)}/V_{\beta}$, $Z_{\alpha'}=[V_{\beta}^{(3)}\cap Q_{\alpha'}, V_{\alpha'}]\le V_{\beta}^{(3)}$. Even still, $V_{\beta}^{(3)}/V_{\beta}$ contains a unique non-central chief factor for $\bar{L_{\beta}}$ which is an FF-module and by \cref{GoodAction3}, $O^p(R_{\beta})$ centralizes $V_{\beta}^{(3)}$. If $Z_{\alpha'}\le V_{\alpha}^{(2)}$ or $[V_{\alpha}^{(2)}\cap Q_{\alpha'}, V_{\alpha'}]=\{1\}$, then $V_{\alpha}^{(2)}\normaleq L_{\beta}=\langle V_{\alpha'}, Q_{\alpha}, R_{\beta}\rangle$, a contradiction. The lemma follows in this case so we may assume that $V_{\beta}^{(3)}\not\le Q_{\alpha'-2}$ and $Z_{\alpha'}=[V_{\beta}^{(3)}\cap Q_{\alpha'}, V_{\alpha'}]\le V_{\beta}^{(3)}$.

Continuing under the assumption that $V_{\beta}^{(3)}\not\le Q_{\alpha'-2}$ and $V_{\beta}^{(3)}\cap Q_{\alpha'-2}\le Q_{\alpha'-1}$, since $Z_{\alpha'-1}=Z_{\alpha'}R\le V_{\beta}^{(3)}$ and $b>5$, we deduce that $Z_{\alpha'-1}=Z_{\alpha'-3}$, otherwise $V_{\beta}^{(3)}$ centralizes $V_{\alpha'-2}$. By \cref{SimExt}, $O^p(R_{\beta})$ does not centralize $V_{\beta}^{(3)}$ and so by \cref{GoodAction3}, either $V_{\beta}^{(3)}/V_{\beta}$ contains more than one non-central chief factor, or a non-central chief factor within $V_{\beta}^{(3)}/V_{\beta}$ is not an FF-module. Hence, we infer that $Z_{\alpha'-1}=[V_{\beta}^{(3)}\cap Q_{\alpha'-2}, V_{\alpha'}]\not\le V_{\beta}$. Moreover, since $b>5$, $[V_{\beta}^{(3)}, Z_{\alpha'+1}, Z_{\alpha'+1}]\le [V_{\beta}^{(3)}, V_{\alpha'-2}^{(3)}, V_{\alpha'-2}^{(3)}]=\{1\}$ and $V_{\beta}^{(3)}$ admits quadratic action. In particular, if $p\geq 5$ then the Hall--Higman theorem implies that $O^p(R_{\beta})$ centralizes $V_{\beta}^{(3)}$ and so $p=2$ or $3$. Moreover, we may apply \cref{SEQuad} when $p=3$ is odd and \cref{2FRecog} when $p=2$ to deduce that $q=p$.

Notice that $Z_{\alpha'-1}=Z_{\alpha'-3}\le V_{\beta}^{(3)}\le Z(V_{\beta}^{b-4})$. Suppose that $b>7$ and let $n\leq \frac{b-5}{2}$ be chosen minimally such that $V_{\beta}^{(2n+1)}\le Q_{\alpha'-2n}$. Since $V_{\beta}^{(3)}\not\le Q_{\alpha'-2}$, if such an $n$ exists then $n\geq 2$. Notice $V_{\beta}^{(5)}$ centralizes $Z_{\alpha'-3}\le V_{\beta}^{(3)}$ so that either $Z_{\alpha'-3}= Z_{\alpha'-5}\le V_{\beta}^{(3)}$ or $V_{\beta}^{(5)}\le Q_{\alpha'-4}$ and $n=2$. Extending through larger subgroups, it is clear that for a minimally chosen $n$, $Z_{\alpha'-1}=Z_{\alpha'-3}=\dots=Z_{\alpha'-2n+1}\le V_{\beta}^{(3)}$ is centralized by $V_{\beta}^{(2n+1)}$ so that $V_{\beta}^{(2n+1)}\le Q_{\alpha'-2n+1}$. Then $V_{\beta}^{(2n+1)}=V_{\beta}^{(2(n-1)+1)}(V_{\beta}^{(2n+1)}\cap Q_{\alpha'-2n+2})$. Moreover, $Z_{\alpha'-1}=\dots=Z_{\alpha'-2n+1}$, $V_{\beta}^{(2n+1)}\cap Q_{\alpha'-2a}\le Q_{\alpha'-2a+1}$ and $V_{\beta}^{(2n+1)}\cap Q_{\alpha'-2a}=V_{\beta}^{(2(a-2)+1)}(V_{\beta}^{(2n+1)}\cap Q_{\alpha'-2a+2})$ from which it follows that $V_{\beta}^{(2n+1)}=V_{\beta}^{(2(n-1)+1)}(V_{\beta}^{(2n+1)}\cap Q_{\alpha'})$ so that $O^p(L_{\beta})$ centralizes $V_{\beta}^{(2n+1)}/V_{\beta}^{(2(n-1)+1)}$, a contradiction. Thus, no such $n$ exists for $n\leq \frac{b-5}{2}$ and it follows that $V_{\beta}^{(b-4)}\not\le Q_{\alpha'-b+5}=Q_{\alpha+5}$ and $Z_{\alpha'-1}=\dots=Z_{\alpha+6}=Z_{\alpha+4}$. If $b=7$, then $Z_{\alpha'-1}=Z_{\alpha'-3}=Z_{\alpha+4}$ by definition. Since $Z_{\alpha'-1}\not\le V_{\beta}$, to obtain a contradiction, we need only show that $Z_{\alpha+2}=Z_{\alpha+4}$. 

If $Z_{\beta}$ is centralized by $V_{\alpha'}^{(3)}$, then $V_{\alpha'}^{(3)}$ centralizes $Z_{\alpha+2}=R\times Z_{\beta}$ and if $Z_{\alpha+2}\ne Z_{\alpha+4}$, then $V_{\alpha'}^{(3)}$ centralizes $V_{\alpha+3}$ and $V_{\alpha'}^{(3)}=V_{\alpha'}(V_{\alpha'}^{(3)}\cap Q_{\beta})$ so that $V_{\alpha'}^{(3)}/V_{\alpha'}$ contains a unique non-central chief factor which is an FF-module and by \cref{GoodAction3}, $O^p(R_{\alpha'})$ centralizes $V_{\alpha'}^{(3)}$. By conjugacy, $O^p(R_{\beta})$ centralizes $V_{\beta}^{(3)}$, a contradiction. Thus, $V_{\alpha'}^{(3)}$ does not centralize $Z_{\beta}$. Since $V_{\alpha'}^{(3)}$ centralizes $Z_{\alpha+3}\times R\le Z_{\alpha+2}$, we deduce that $R=Z_{\alpha+3}$. Furthermore, since $b>5$ and $V_{\alpha'}^{(3)}$ is abelian, $V_{\alpha'}^{(3)}\cap Q_{\alpha+3}\cap Q_{\alpha+2}\cap Q_{\beta}\le C_{\beta}$.

Now, $V_{\beta}\le C_{\alpha'-2}$ and since $[Q_{\lambda}, V_{\lambda}^{(2)}]=Z_{\lambda}$ for all $\lambda\in\Delta(\alpha'-2)$, we have that $R\le [V_{\beta}, V_{\alpha'-2}^{(3)}]\le Z_{\alpha+2}\cap V_{\alpha'-2}$. If $Z_{\alpha+2}\le V_{\alpha'-2}$, then $Z_{\alpha+2}=Z_{\alpha'-3}=Z_{\alpha'-1}\le V_{\beta}$, a contradiction and so $[V_{\beta}, V_{\alpha'-2}^{(3)}]=R$ and $[V_{\beta}, V_{\alpha'-2}^{(3)}\cap Q_{\beta}]=R\cap Z_{\beta}=\{1\}$. Then $V_{\alpha'-2}^{(3)}\cap Q_{\beta}\le C_{\beta}$ so that $[V_{\beta}^{(3)}, V_{\alpha'-2}^{(3)}\cap Q_{\beta}]\le V_{\beta}\cap V_{\alpha'-2}^{(3)}$. Since $b>5$, $V_{\beta}\not\le V_{\alpha'-2}^{(3)}$ and since $R\le V_{\alpha'-2}$, $Z_{\alpha+2}\le V_{\alpha'-2}^{(3)}$ and $Z_{\beta}\le V_{\alpha'-2}^{(3)}$ but $Z_{\beta}\not\le V_{\alpha'-2}$. If $b>7$, $V_{\alpha'}^{(3)}$ centralizes $Z_{\beta}$, a contradiction by the above. 

Thus, we assume that $b=7$, $V_{\beta}^{(3)}\not\le Q_{\alpha'-2}$, $V_{\beta}^{(3)}\cap Q_{\alpha'-2}\le Q_{\alpha'-1}$, $Z_{\alpha'-1}=Z_{\alpha'-3}\ne Z_{\alpha+2}$ and $[Z_{\beta}, V_{\alpha'}^{(3)}]\ne\{1\}$. Set $W^\beta=\langle V_{\delta}^{(2)}\mid Z_{\delta}=Z_{\alpha}, \delta\in\Delta(\beta)\rangle$ so that $[C_{\beta}, W^{\beta}]=[C_{\beta}, V_{\alpha}^{(2)}]\le Z_{\alpha}$. Then $[W^{\beta}, V_{\alpha'-2}]\le Z_{\alpha'-3}\cap Z_{\alpha}$ and by \cref{UWNormal}, $W^\beta\normaleq R_{\beta}Q_{\alpha}$. If $Z_{\beta}\le Z_{\alpha'-3}=Z_{\alpha'-1}$, then $Z_{\alpha'-1}=Z_{\beta}\times R=Z_{\alpha+2}\le V_{\beta}$, a contradiction. Thus, $W^{\beta}=V_{\beta}(W^{\beta}\cap Q_{\alpha'})$. If $[W^{\beta}\cap Q_{\alpha'}, V_{\alpha'}]\le W^{\beta}$, then $V_{\alpha}^{(2)}\le W^{\beta}\normaleq L_{\beta}=\langle V_{\alpha'}, Q_{\alpha}, R_{\beta}\rangle$ and $V_{\beta}^{(3)}=W^{\beta}\le Q_{\alpha'-2}$, a contradiction. Thus, $W^{\beta}\cap Q_{\alpha'}\not\le Q_{\alpha'+1}$ for some $\alpha'+1\in\Delta(\alpha')$ and since $Z_{\alpha'+1}Z_{\alpha'-1}=V_{\alpha'}\not\le Q_{\beta}$, $(\alpha'+1, \beta)$ is a critical pair.

Since $V_{\alpha+3}\le Q_{\alpha'+1}$, $[V_{\alpha'+1}^{(2)}\cap Q_{\alpha+3}, V_{\alpha+3}]\le Z_{\alpha'+1}\cap Z_{\alpha+3}=Z_{\alpha'+1}\cap R=\{1\}$ and $V_{\alpha'+1}^{(2)}\cap Q_{\alpha+3}=Z_{\alpha'+1}(V_{\alpha'+1}^{(2)}\cap C_{\beta})$. Furthermore, $[V_{\alpha'+1}^{(2)}\cap C_{\beta}, W^{\beta}\cap Q_{\alpha'}]\le V_{\alpha'+1}^{(2)}\cap Z_{\alpha}$ and since $Z_{\beta}\not\le V_{\alpha'}^{(3)}$, we have that $W^{\beta}\cap Q_{\alpha'}$ centralizes $(V_{\alpha'+1}^{(2)}\cap Q_{\alpha+3})/Z_{\alpha'+1}$. Thus, $V_{\alpha'+1}^{(2)}\not\le Q_{\alpha+3}$ and $V_{\alpha'+1}^{(2)}/Z_{\alpha'+1}$ is an FF-module. By \cref{GoodAction1}, $O^p(R_{\alpha})$ centralizes $V_{\alpha}^{(2)}$ and since $V_{\beta}^{(3)}$ does not centralize $V_{\alpha'-2}$, it follows from \cref{SimExt} that $Z_{\alpha'-2}\ne Z_{\alpha'-4}=R$. 

Suppose that $[V_{\beta}^{(3)}, Q_{\beta}]V_{\beta}/V_{\beta}$ contains a non-central chief factor for $L_{\beta}$. In particular, $[Q_{\beta}, V_{\alpha}^{(2)}]\not\le V_{\beta}$, and since $V_{\alpha}^{(2)}/Z_{\alpha}$ is an FF-module, $|V_{\alpha}^{(2)}|=p^5$. The non-central chief factor within $[V_{\beta}^{(3)}, Q_{\beta}]V_{\beta}/V_{\beta}$, $U/V$ say, is an FF-module for $\bar{L_{\beta}}$ and $L_{\beta}/C_{L_{\beta}}(U/V)\cong\SL_2(p)$. Set $R_1:=C_{L_{\beta}}(U/V)$ and $R_2:=C_{L_{\beta}}(V_{\beta}^{(3)}/([V_{\beta}^{(3)}, Q_{\beta}]V_{\beta}))$, noticing that also $L_{\beta}/R_2\cong \SL_2(p)$. If $R_1 \ne R_{\beta}$, and employing \cref{Badp3} (iii) when $p=3$, we conclude that $L_{\beta}=\langle R_1, R_{\beta}, S\rangle$. Similarly, if $R_2\ne R_{\beta}$ then $L_{\beta}=\langle R_2, R_{\beta}, S\rangle$. 

Suppose that $R_1\ne R_{\beta}$. Then $[V_{\alpha}^{(2)}, Q_{\beta}]V_{\beta}\normaleq R_1$ and $[V_{\alpha}^{(2)}, Q_{\beta}, Q_{\beta}]\le V_{\beta}$ so that $[V_{\alpha}^{(2)}, Q_{\beta}, Q_{\beta}]\normaleq L_{\beta}=\langle R_1, R_{\beta}, S\rangle$. Since $[V_{\alpha}^{(2)}, Q_{\beta}, Q_{\beta}]\le Z_{\alpha}$, we have that $[V_{\alpha}^{(2)}, Q_{\beta}, Q_{\beta}]=Z_{\beta}$. Setting $C^\alpha$ to be the preimage in $V_{\alpha}^{(2)}$ of $C_{V_{\alpha}^{(2)}/Z_{\alpha}}(O^p(L_{\alpha}))$, we have that $C^\alpha\le V_{\beta}[V_{\alpha}^{(2)}, Q_{\beta}]$ and so $[C^\alpha, Q_{\beta}]=Z_{\beta}$. As in \cref{VB1} (where $C^\alpha$ is defined slightly differently), we see that $|Q_{\alpha}/C_{Q_{\alpha}}(C^\alpha)|=p^2$ and $C_{Q_{\alpha}}(C^\alpha)\le Q_{\beta}$. Now, $V_{\beta}\le Q_{\alpha'-2}$ and so $[V_{\beta}, C^{\alpha'-1}]\le Z_{\alpha'-2}\cap Z_{\alpha+2}=\{1\}$, for otherwise $Z_{\alpha+2}=Z_{\alpha'-1}$. But then, $V_{\beta}\le C_{Q_{\alpha'-1}}(C^{\alpha'-1})\le Q_{\alpha'}$, a contradiction. Thus, $R_1=R_{\beta}$. 

Suppose that $R_2\ne R_{\beta}$. Then $V_{\alpha}^{(2)}[V_{\beta}^{(3)}, Q_{\beta}]\normaleq R_2$ and so $[V_{\alpha}^{(2)}, Q_{\beta}][V_{\beta}^{(3)}, Q_{\beta}, Q_{\beta}]\normaleq L_{\beta}=\langle R_1, R_2, S\rangle$. Since $[V_{\alpha}^{(2)}, Q_{\beta}, Q_{\beta}]\le Z_{\alpha}$, we have that $[V_{\beta}^{(3)}, Q_{\beta}, Q_{\beta}]\le V_{\beta}$ and so $[V_{\alpha}^{(2)}, Q_{\beta}]V_{\beta}\normaleq L_{\beta}$. But then $[V_{\beta}^{(3)}, Q_{\beta}]V_{\beta}=[V_{\alpha}^{(2)}, Q_{\beta}]V_{\beta}$ is centralized by $Q_{\alpha}$, modulo $V_{\beta}$, and so $([V_{\beta}^{(3)}, Q_{\beta}]V_{\beta})/V_{\beta}$ does not contain a non-central chief factor for $L_{\beta}$. Thus, $R_2=R_{\beta}$. But now $O^p(R_{\beta})$ centralizes $V_{\beta}^{(3)}$ and \cref{SimExt} applied to $Z_{\alpha'-1}=Z_{\alpha'-3}$ gives $V_{\alpha'}\le V_{\alpha'-1}^{(2)}=V_{\alpha'-3}^{(2)}\le Q_{\beta}$, a contradiction.

Therefore, we may assume that $([V_{\beta}^{(3)}, Q_{\beta}]V_{\beta})/V_{\beta}$ does not contain a non-central chief factor for $L_{\beta}$ and $[V_{\alpha}^{(2)}, Q_{\beta}]V_{\beta}\normaleq L_{\beta}$. As before, since $[V_{\alpha}^{(2)}, Q_{\beta}, Q_{\beta}]\le Z_{\alpha}$, we have that $[V_{\alpha}^{(2)}, Q_{\beta}, Q_{\beta}]=Z_{\beta}$ and either $|V_{\alpha}^{(2)}|=p^4$; or $[C^\alpha, Q_{\beta}]=Z_{\beta}$ for $C^\alpha$ as defined above. In the latter case, we again see that $V_{\beta}\le C_{Q_{\alpha'-1}}(C^{\alpha'-1})\le Q_{\alpha'}$, a contradiction. Thus, $|V_{\alpha}^{(2)}|=p^4$, $[V_{\alpha}^{(2)}, Q_{\beta}]\le V_{\beta}$ and $[V_{\beta}^{(3)}, Q_{\beta}]=V_{\beta}$. Since $O^p(R_{\beta})$ does not centralize $V_{\beta}^{(3)}$, by \cref{GoodAction3}, $V_{\beta}^{(3)}/V_{\beta}$ is a quadratic $2F$-module for $\bar{L_{\beta}}$. Moreover, since $V_{\alpha}^{(2)}$ generates $V_{\beta}^{(3)}$, is $G_{\alpha,\beta}$-invariant and has order $p$ modulo $V_{\beta}$, comparing with \cref{pgen} and using that $|S/Q_{\beta}|=p$, it follows that $p=2$ and $L_{\beta}/C_{L_{\beta}}(V_{\beta}^{(3)}/V_{\beta})\cong \Dih(10)$ or $(3\times 3):2$.

Now, $C_{L_{\beta}}(V_{\beta}^{(3)}/V_{\beta})$ normalizes $V_{\alpha}^{(2)}$ so that $[V_{\alpha}^{(2)}, C_{\beta}]\le Z_{\alpha}$ is also normalized by $C_{L_{\beta}}(V_{\beta}^{(3)}/V_{\beta})$. Since $R_{\beta}$ normalizes $Z_{\alpha}$, if $L_{\beta}=\langle S, R_{\beta}, C_{L_{\beta}}(V_{\beta}^{(3)}/V_{\beta})\rangle$ then $[V_{\alpha}^{(2)}, C_{\beta}]=Z_{\beta}$ and $[V_{\beta}^{(3)}, C_{\beta}]=Z_{\beta}$. But then $R=[V_{\alpha'}, V_{\beta}]\le [V_{\alpha'-2}^{(3)}, V_{\beta}]=Z_{\alpha'-2}$, a contradiction. Thus  $L_{\beta}/C_{L_{\beta}}(V_{\beta}^{(3)}/V_{\beta})\cong (3\times 3):2$ and $C_{L_{\beta}}(V_{\beta}^{(3)}/V_{\beta})\le R_{\beta}$. Then $V_{\alpha'-1}^{(2)}<\langle (V_{\alpha'-3}^{(2)})^{R_\beta S}\rangle=:W$ and $|W/V_{\beta}|=4$. But now, $[W, V_{\beta}^{(3)}]\le [W, Q_{\alpha'-3}]\le Z_{\alpha'-3}$ and $[V_{\alpha'-2}^{(3)}\cap Q_{\beta}, V_{\beta}]\le Z_{\beta}\cap V_{\alpha'-2}=\{1\}$ and $[V_{\alpha'-2}^{(3)}\cap Q_{\beta}, V_{\beta}^{(3)}]\le V_{\beta}\cap V_{\alpha'-2}^{(3)}=Z_{\alpha+2}\le V_{\alpha'-3}^{(2)}$. Therefore, $[V_{\beta}^{(3)}, V_{\alpha'-2}^{(3)}]\le V_{\alpha'-3}^{(2)}$, a contradiction since $V_{\beta}^{(3)}/V_{\beta}$ is not dual to an FF-module.

Therefore, $V_{\beta}^{(3)}\cap Q_{\alpha'-2}\not\le Q_{\alpha'-1}$. Since $R\le Z_{\alpha'-1}$ and $R\ne Z_{\alpha'}$, it follows that $V_{\beta}^{(3)}$ does not centralize $Z_{\alpha'}$. Hence, as $b>5$ and $V_{\beta}^{(3)}$ is abelian, we conclude that $[V_{\beta}^{(3)}\cap\dots\cap Q_{\alpha'}, V_{\alpha'}]=\{1\}$. In particular, $[V_{\alpha}^{(2)}\cap Q_{\alpha'}, V_{\alpha'}]=\{1\}$ and so $[V_{\alpha'}, V_{\alpha}^{(2)}]=R\le V_{\alpha}^{(2)}$. Additionally, since $V_{\beta}^{(3)}$ centralizes $Z_{\alpha'-2}$, we have that $R=Z_{\alpha'-2}$ intersects $Z_{\beta}$ trivially.

We set $W^\beta=\langle V_{\delta}^{(2)}\mid Z_{\lambda}=Z_{\alpha}, \lambda\in\Delta(\beta)\rangle$ noting that $W^\beta\normaleq R_{\beta}Q_{\alpha}$ by \cref{UWNormal}. For such a $\lambda\in\Delta(\beta)$, $(\lambda, \alpha')$ is a critical pair. Suppose that $V_{\lambda}^{(2)}\not\le Q_{\alpha'-2}$. Then $\{1\}\ne [V_{\lambda-1}, V_{\alpha'-2}]\le Z_{\lambda}\cap Z_{\alpha'-3}=Z_{\alpha}\cap Z_{\alpha'-3}$ so that $Z_{\beta}\le Z_{\alpha'-3}$ and so $Z_{\alpha'-3}=Z_{\alpha'-2}\times Z_{\beta}=Z_{\alpha+2}$. Now, there is $\alpha'+1\in\Delta(\alpha')$ such that $(\alpha'+1, \beta)$ is a critical pair. As in the above steps, if $V_{\alpha'+1}^{(2)}\not\le Q_{\alpha+3}$ then $Z_{\alpha'}=[V_{\alpha'+2}, V_{\alpha+3}]\le V_{\alpha+3}$, a contradiction as $V_{\alpha+3}$ is centralized by $V_{\beta}^{(3)}$. Thus, $V_{\alpha'+1}^{(2)}\le Q_{\alpha+2}$ and since $V_{\alpha'}^{(3)}\cap Q_{\alpha+3}\le Q_{\alpha+2}$, applying the previous results in this proof, $O^p(R_{\alpha'})$ centralizes $V_{\alpha'}^{(3)}$. But then $V_{\alpha}^{(2)}\normaleq L_{\beta}=\langle V_{\alpha'}, Q_{\alpha}, R_{\beta}\rangle$, a contradiction.

Thus, $W^{\beta}\le Q_{\alpha'-2}$, and so every $\lambda\in\Delta(\beta)$ with $Z_{\lambda}=Z_{\alpha}$ provides a critical pair $(\lambda, \alpha')$ satisfying the same hypothesis as $(\alpha, \alpha')$. By an observation in the first paragraph of this proof, this yields that $W^\beta\le Q_{\alpha'-1}$ and $W^{\beta}=V_{\beta}(W^{\beta}\cap Q_{\alpha'})$. Then $V_{\alpha'}$ centralizes $W^\beta/V_{\beta}$ so that $W^{\beta}\normaleq L_{\beta}=\langle V_{\alpha'}, R_{\beta}, Q_{\alpha}\rangle$. Since $V_{\alpha'}$ centralizes $W^{\beta}/V_{\beta}$, it follows that $V_{\alpha}^{(2)}\normaleq L_{\beta}$, a final contradiction.
\end{proof}

\begin{lemma}\label{GoodCent}
Suppose that $C_{V_\beta}(V_{\alpha'})=V_\beta \cap Q_{\alpha'}$ and $b>5$. If $V_{\alpha}^{(2)}\not\le Q_{\alpha'-2}$ and $|V_{\beta}|\ne q^3$, then we may assume that $[V_{\alpha}^{(2)}, Z_{\alpha'-1}]\ne \{1\}$.
\end{lemma}
\begin{proof}
Suppose that $|V_{\beta}|\ne q^3$. By \cref{VA1} and \cref{VB1}, we may assume that for any critical pair $(\alpha^*, {\alpha^*}')$, $V_{\alpha^*}^{(2)}\not\le Q_{{\alpha^*}'-2}$. In particular, there is an infinite path $(\alpha', \alpha'-1, \alpha'-2,\dots, \beta, \alpha, \alpha-1, \alpha-2,\dots)$ such that $(\alpha-2k, \alpha'-2k)$ is a critical pair for all $k\geq 0$. For $2k>b$, we have that $Z_{\alpha'-2k-1}\ne Z_{\alpha'-2k-3}$ and so we can arrange that for our chosen critical pair $(\alpha, \alpha')$ we have that $Z_{\alpha'-1}\ne Z_{\alpha'-3}$. If $[V_{\alpha}^{(2)}, Z_{\alpha'-1}]=\{1\}$, then $V_{\alpha}^{(2)}$ centralizes $Z_{\alpha'-1}Z_{\alpha'-3}$ and by \cref{VBGood}, we have a contradiction.
\end{proof}

Notice that by \cref{VA1} and \cref{VB1}, whenever $|V_{\beta}|\ne q^3$ we have that $V_{\lambda}^{(2)}\not\le Q_{\lambda+b-2}$ for any critical pair $(\lambda, \lambda+b)$ with $\lambda\in\alpha^G$. Moreover, as demonstrated in \cref{GoodCent}, we may iterate backwards through critical pairs far enough that the conclusion of \cref{GoodCent} holds for all critical pairs beyond a certain point. The net result of this that whenever $|V_{\beta}|\ne q^3$, we may assume that we have a critical pair $(\alpha, \alpha')$ with $V_{\alpha}^{(2)}\not\le Q_{\alpha'-2}$ and $[V_{\alpha}^{(2)}, Z_{\alpha'-1}]\ne \{1\}$, and for all $k\geq 0$ we also have that $(\alpha-2k, \alpha'-2k)$ is a critical pair with $V_{\alpha-2k}^{(2)}\not\le Q_{\alpha'-2-2k}$ and $[V_{\alpha-2k}^{(2)}, Z_{\alpha'-1-2k}]\ne \{1\}$. We will use this fact in the following two lemmas.

\begin{lemma}\label{b=7i}
Suppose that $C_{V_\beta}(V_{\alpha'})=V_\beta \cap Q_{\alpha'}$ and $b=7$. If $V_{\alpha}^{(2)}\not\le Q_{\alpha'-2}$, then $|V_{\beta}|=q^3$.
\end{lemma}
\begin{proof}
Suppose that $b=7$. By \cref{VA1} and \cref{VB1}, we may consider a critical pair $(\alpha, \alpha')$ iterated backwards so that $(\alpha+2, \alpha'+2)$ is also a critical pair. Suppose first that $V^\alpha\not\le Q_{\alpha'-2}$. Then $[V^\alpha, V_{\alpha'-2}]\le Z_{\alpha}$ and so $[V^\alpha, V_{\alpha'-2}]=Z_{\beta}\le V_{\alpha'-2}$. Since $Z_{\alpha+2}\not\le Q_{\alpha'+2}$ and $b>5$, we have that $Z_{\beta}=Z_{\alpha+3}$ intersects $Z_{\alpha'-2}$ trivially. But now, $Z_{\alpha+3}Z_{\alpha+3}^gZ_{\alpha'-2}=Z_{\alpha'-3}Z_{\alpha'-3}^g$ is normalized by $L_{\alpha'-2}=\langle V^\alpha, (V^\alpha)^g, R_{\alpha'-2}\rangle$ for some appropriately chosen $g\in L_{\alpha'-2}$, so that $V_{\alpha'-2}=Z_{\alpha'-3}Z_{\alpha'-3}^g$ is of order $q^3$, a contradiction. Thus, we may assume that $V^\alpha\le Q_{\alpha'-2}$. 

If $V^\alpha\not\le Q_{\alpha'-1}$, then $Z_{\alpha'-2}=[V^\alpha, V_{\alpha'-2}]\le Z_{\alpha}$ and $Z_{\alpha'-2}=Z_{\beta}$. Moreover, for some $\alpha-2\in\Delta^{(2)}(\alpha)$ with $(\alpha-2, \alpha'-2)$ a critical pair, $V_{\alpha-2}^{(2)}$ centralizes $Z_{\alpha'-2}$ and since $V_{\alpha-2}^{(2)}$ does not centralize $Z_{\alpha'-3}$, we deduce that $Z_{\alpha'-2}=Z_{\alpha+3}=Z_{\beta}$. Now, $[V^{\alpha'-1}, V_{\beta}]\le Z_{\alpha'-1}$ and since $V^\alpha$ does not centralize $Z_{\alpha'-1}$, $[V^{\alpha'-1}, V_{\beta}]\le Z_{\alpha'-2}=Z_{\beta}$ and $V^{\alpha'-1}\le Q_{\beta}$. If $V^{\alpha'-1}\le Q_{\alpha}$, then $[V^{\alpha'-1}, V^\alpha]\le Z_{\alpha}$ so that $[V^{\alpha'-1}, V^\alpha]=Z_{\beta}=Z_{\alpha'-2}$ and $V^\alpha$ centralizes $V^{\alpha'-1}/Z_{\alpha'-1}$, a contradiction since $V^{\alpha'-1}/Z_{\alpha'-1}$ contains a non-central chief factor for $L_{\alpha'-1}$. Thus, $S=V^{\alpha'-1}Q_{\alpha}$ and $V_{\alpha'-1}^{(2)}\cap Q_{\beta}=V^{\alpha'-1}(V_{\alpha'-1}^{(2)}\cap Q_{\alpha})$. Since $Z_{\alpha}\not\le V_{\alpha'-1}^{(2)}$, $[V_{\alpha'-1}^{(2)}\cap Q_{\alpha}, V^\alpha]=Z_{\beta}=Z_{\alpha'-2}$ and it follows that $V_{\alpha'-1}^{(2)}/V^{\alpha'-1}$ is an FF-module for $\bar{L_{\alpha'-1}}$. Similarly, $[V^{\alpha'-1}\cap Q_{\alpha}, V^\alpha]=Z_{\alpha'-2}$ and $V^{\alpha'-1}/Z_{\alpha'-1}$ is an FF-module for $\bar{L_{\alpha'-1}}$. Then \cref{GoodAction1} and \cref{SimExt} applied to $Z_{\beta}=Z_{\alpha+3}$ implies that $V_{\beta}=V_{\alpha+3}\le Q_{\alpha'}$, a contradiction.

Thus, $V^\alpha=Z_{\alpha}(V^\alpha\cap Q_{\alpha'})$. Suppose that $V_{\alpha'}\le Q_{\beta}$ and again let $(\alpha-2, \alpha'-2)$ be a critical pair. Since $V^\alpha/Z_{\alpha}$ contains a non-central chief factor, $Z_{\alpha'}\le V^\alpha$ and $Z_{\alpha'}\not\le Z_{\alpha}$. Then $Z_{\alpha'}=Z_{\alpha'-2}$, otherwise $[V_{\alpha}^{(2)}, Z_{\alpha'-1}]=\{1\}$. But now, since $b>5$, $V_{\alpha-2}^{(2)}$ centralizes $Z_{\alpha'-2}\le V^\alpha$ and since $[V_{\alpha-2}^{(2)}, Z_{\alpha'-3}]\ne\{1\}$, it follows that $Z_{\alpha'-2}=Z_{\alpha'-4}=Z_{\alpha+3}$. Since $R=[V_{\alpha'}, V_{\beta}]=Z_{\beta}\le V_{\alpha'}$, as $Z_{\alpha+2}\not\le Q_{\alpha'+2}$, we must have that $Z_{\alpha+3}=Z_{\beta}$ . But then $R=Z_{\beta}=Z_{\alpha'}$, a contradiction.

Finally, we have that $V^\alpha\le Q_{\alpha'-1}$ and $V_{\alpha'}\not\le Q_{\beta}$. Set $U^\beta=\langle V^\delta \mid Z_{\delta}=Z_{\alpha}, \delta\in\Delta(\beta)\rangle$. Then $(\delta, \alpha')$ is a critical pair for all such $\delta\in\Delta(\beta)$ and so $V^\delta\le Q_{\alpha'-1}$ for all such $\delta$. By \cref{UWNormal}, $R_{\beta}Q_{\alpha}$ normalizes $U^\beta$. Now, $U^\beta V_{\beta}=V_{\beta}(U^\beta\cap Q_{\alpha'})$ and either $Z_{\alpha'}\le V_{\beta}^{(3)}$; or $V_{\alpha'}$ centralizes $U^\beta V_{\beta}/V_{\beta}$. In the former case, since $V_{\beta}^{(3)}$ does not centralize $Z_{\alpha'-1}$, $Z_{\alpha'}=Z_{\alpha'-2}$. Iterating backwards through critical pairs, this eventually implies that $Z_{\alpha'}=Z_{\beta}$ and again, $V_{\alpha'}$ centralizes $U^\beta V_{\beta}/V_{\beta}$. Thus, in all cases, $U^\beta V_{\beta}\normaleq L_{\beta}=\langle V_{\alpha'}, R_{\beta}, Q_{\alpha}\rangle$ and since $V_{\alpha'}$ centralizes $U^\beta V_{\beta}/V_{\beta}$, $O^p(L_{\beta})$ centralizes $U^\beta V_{\beta}/V_{\beta}$. Then $V^\alpha V_{\beta}\normaleq L_{\beta}$, a contradiction by \cref{VBGood}. 
\end{proof}

\begin{lemma}\label{VBp3}
Suppose that $C_{V_\beta}(V_{\alpha'})=V_\beta \cap Q_{\alpha'}$ and $b>5$. If $V_{\alpha}^{(2)}\not\le Q_{\alpha'-2}$, then $|V_{\beta}|=q^3$.
\end{lemma}
\begin{proof}
By \cref{b=7i}, we may assume that $b>7$. In the following, the aim will be to prove that $Z_{\alpha'-2}=Z_{\alpha'-4}$ for then extending far enough backwards along the critical path, by \cref{GoodCent}, we can manufacture a situation in which $(\alpha, \alpha')$ is a critical pair, $Z_{\alpha'-1-2k}\ne Z_{\alpha'-3-2k}$ for all $k\geq 0$ and $Z_{\alpha'}=Z_{\alpha'-2}=\dots=Z_{\alpha+3}=Z_{\beta}$. Throughout we consider a critical pair $(\alpha, \alpha')$ iterated backwards far enough so that $(\alpha+2, \alpha'+2)$ is also a critical pair.

Suppose first that $V_{\beta}^{(3)}\cap Q_{\alpha'-2}\not\le Q_{\alpha'-1}$. Then $Z_{\alpha'-2}=[V_{\beta}^{(3)}\cap Q_{\alpha'-2},  Z_{\alpha'-1}]\le V_{\beta}^{(3)}$ is centralized by $V_{\alpha-2}^{(2)}$ since $b>7$. Since $V_{\alpha-2}^{(2)}$ does not centralizes $Z_{\alpha'-3}$, we have that $Z_{\alpha'-2}=Z_{\alpha'-4}$, as desired. Thus, $V_{\beta}^{(3)}\cap Q_{\alpha'-2}=V_{\beta}(V_{\beta}^{(3)}\cap Q_{\alpha'})$. If $Z_{\alpha'}=[V_{\beta}^{(3)}\cap Q_{\alpha'}, V_{\alpha'}]\le V_{\beta}^{(3)}$ then, as $V_{\beta}^{(3)}$ does not centralize $Z_{\alpha'-1}$, we deduce that $Z_{\alpha'}=Z_{\alpha'-2}\le V_{\beta}^{(3)}$. Similarly to the above, using $b>7$, we have that $Z_{\alpha'-2}=Z_{\alpha'-4}$, as desired. Thus, $[V_{\beta}^{(3)}\cap Q_{\alpha'}, V_{\alpha'}]=\{1\}$.

Suppose that $V_{\alpha'}\le Q_{\beta}$. Then, by the above, $V_{\alpha}^{(2)}\cap Q_{\alpha'-2}=Z_{\alpha}(V_{\alpha}^{(2)}\cap Q_{\alpha'})$ and $[V_{\alpha}^{(2)}\cap Q_{\alpha'}, V_{\alpha'}]=\{1\}$, a contradiction since both $V_{\alpha}^{(2)}/V^\alpha$ and $V^\alpha/Z_{\alpha}$ contain a non-central chief factor. Thus, $V_{\alpha'}\not\le Q_{\beta}$ and $V_{\beta}^{(3)}/V_{\beta}$ contains a unique non-central chief factor which is an FF-module for $\bar{L_{\beta}}$. By \cref{GoodAction3}, $O^p(R_{\beta})$ centralizes $V_{\beta}^{(3)}$. If $V^\alpha\le Q_{\alpha'-2}$, then $V^\alpha V_{\beta}=V_{\beta}(V^\alpha V_{\beta}\cap Q_{\alpha'})$ and it follows that $V^\alpha V_{\beta}\normaleq L_{\beta}=\langle V_{\alpha'}, Q_{\alpha}, R_{\beta}\rangle$, a contradiction by \cref{VBGood}. Therefore, $V^\alpha\not\le Q_{\alpha'-2}$ and since $V_{\alpha'-2}\le Q_{\alpha}$, we have that $[V^\alpha, V_{\alpha'-2}]=Z_{\beta}\ne Z_{\alpha'-2}$. 

Suppose that $b=9$ and consider the critical pair $(\alpha-2, \alpha'-2)$. Then, as $V_{\alpha'-4}\le Q_{\alpha-2}$, we have that $[V^{\alpha-2}, V_{\alpha'-4}]\le Z_{\alpha-1}$. Suppose that $Z_{\alpha-1}=[V^{\alpha-2}, V_{\alpha'-4}]\le V_{\alpha'-4}$. Since $Z_{\alpha}, Z_{\alpha+2}\not\le V_{\alpha'-4}$, we must have that $Z_{\alpha-1}=Z_{\beta}=Z_{\alpha+3}=Z_{\alpha'-6}$. But then, $[V^{\alpha-2}, V_{\alpha'-4}]=Z_{\alpha'-6}$ and $Z_{\alpha'-5}Z_{\alpha'-5}^g\normaleq L_{\alpha'-4}=\langle V^{\alpha-2}, (V^{\alpha-2})^g, R_{\alpha'-4}\rangle$ for some appropriately chosen $g\in L_{\alpha'-4}$. Then $V_{\alpha'-4}=Z_{\alpha'-5}Z_{\alpha'-5}^g$ is of order $q^3$, a contradiction. Thus, $[V^{\alpha-2}, V_{\alpha'-4}]=\{1\}$ so that $V^{\alpha-2}V_{\alpha-1}=V_{\alpha-1}(V^{\alpha-2}V_{\alpha-1}\cap Q_{\alpha'-2})$ and since $V^{\alpha-1} V_{\alpha-1}\not\normaleq L_{\alpha-1}$, it follows that $Z_{\alpha'-2}\le V_{\alpha-1}^{(3)}$. Then $V_{\alpha-2}^{(2)}$ centralizes $Z_{\alpha'-2}$ and so $Z_{\alpha'-2}=Z_{\alpha'-4}$, as desired.

Thus, we may assume that $b>9$. Since $V_{\alpha'}\not\le Q_{\beta}$, there is $\lambda\in\Delta(\alpha')$ such that $(\lambda, \beta)$ is a critical pair with $V_{\beta}\not\le Q_{\alpha'}$ and $V_{\lambda}^{(2)}\not\le Q_{\alpha+3}$. In particular, since $b>5$, $V_{\lambda}^{(2)}$ centralizes $Z_{\beta}\le V_{\alpha'-2}$ and $Z_{\beta}=Z_{\alpha+3}$. Now, $(\alpha-2. \alpha'-2)$ is also a critical pair from which we deduce in a similar fashion that $Z_{\beta}=Z_{\alpha-1}$. By this logic, we could have chosen the original critical pair $(\alpha, \alpha')$ such that $Z_{\alpha'}=\dots=Z_{\beta}$, as desired.

In all cases we have reduced to the case where $Z_{\alpha'-2}=Z_{\alpha'-4}$. By a previous observation we may now assume that $(\alpha, \alpha')$ is a critical pair such that $Z_{\alpha'}=Z_{\alpha'-2}=\dots=Z_{\beta}=Z_{\alpha-1}=\dots$ and $Z_{\alpha'-1-2k}\ne Z_{\alpha'-3-2k}$ for any $k\geq 0$. Now, $[V_{\alpha'-2}, V^\alpha]\le [Q_{\alpha}, V^\alpha]\le Z_{\alpha}$ so that $[V_{\alpha'-2}, V^\alpha]=Z_{\beta}=Z_{\alpha'-2}$ and $V^\alpha\le Q_{\alpha'-2}$. Moreover, $V_{\alpha'}\not\le Q_{\beta}$, otherwise $R=Z_{\beta}=Z_{\alpha'}$ and $O^p(L_{\alpha'})$ centralizes $V_{\alpha'}$. 

Suppose that $V^\alpha\not\le Q_{\alpha'-1}$. Then $V_{\beta}\le Q_{\alpha'-1}$ and so $[V_{\beta}, V^{\alpha'-1}]\le Z_{\alpha'-1}$ and since $V^{\alpha}\not\le Q_{\alpha'-1}$, $[V^{\alpha'-1}, V_{\beta}]=Z_{\alpha'-2}=Z_{\beta}$ and $V^{\alpha'-1}\le Q_{\beta}$. Moreover, $V^{\alpha'-1}\not\le Q_{\alpha}$, else $[V^{\alpha}, V^{\alpha'-1}]=Z_{\beta}=Z_{\alpha'-2}\le Z_{\alpha'-1}$ and $V^{\alpha}\le Q_{\alpha'-1}$. Thus, $[V_{\alpha}^{(2)}\cap Q_{\alpha'-2}, V^{\alpha'-1}]=[V^{\alpha}(V_{\alpha}^{(2)}\cap Q_{\alpha'-1}), V^{\alpha'-1}]\le V^{\alpha} Z_{\alpha'-2}=V^{\alpha}$. It follows that both $V_{\alpha}^{(2)}/V^{\alpha}$ and $V^{\alpha}/Z_{\alpha}$ are FF-modules for $\bar{L_{\alpha}}$ and by \cref{GoodAction1} and \cref{SimExt}, we have that $Z_{\beta}=Z_{\alpha-3}$ implies that $V_{\beta}=V_{\alpha+3}\le Q_{\alpha'}$, a contradiction.

Thus, $V^{\alpha}V_{\beta}=V_{\beta}(V^{\alpha}V_{\beta}\cap Q_{\alpha'})$. As in the $b=7$ case, again set $U^\beta=\langle V^\delta \mid Z_{\lambda}=Z_{\alpha}, \lambda\in\Delta(\beta)\rangle\normaleq R_{\beta}Q_{\alpha}$ so that $(\lambda, \alpha')$ is a critical pair for all such $\lambda$ and, by the above, $V^\lambda\le Q_{\alpha'-1}$. Then, $U^\beta V_{\beta}\normaleq L_{\beta}=\langle V_{\alpha'}, R_{\beta}, Q_{\alpha}\rangle$ and since $V_{\alpha'}$ centralizes $U^\beta V_{\beta}/ V_{\beta}$, $O^p(L_{\beta})$ centralizes $U^{\beta}V_{\beta}/V_{\beta}$ and $V_{\beta} V^\alpha\normaleq L_{\beta}$. A contradiction is provided by \cref{VBGood}.
\end{proof}

As a consequence of \cref{VBp3}, we may assume that whenever $b>5$, we have that $|V_{\beta}|=q^3$.

\begin{lemma}\label{Vna-2}
Suppose that $C_{V_\beta}(V_{\alpha'})=V_\beta \cap Q_{\alpha'}$ and $b>5$. If $V_{\alpha}^{(2)}\not\le Q_{\alpha'-2}$ then either:
\begin{enumerate}
\item $R=Z_{\alpha'-2}\le Z_{\alpha+2}\cap Z_{\alpha'-1}$; or
\item $Z_{\alpha'-1}=Z_{\alpha'-3}=R\times Z_{\beta}$ and $V_{\alpha'}\le Q_{\beta}$.
\end{enumerate}
\end{lemma}
\begin{proof}
By \cref{VBp3}, we have that $|V_{\beta}|=q^3$, so that $R=[V_{\alpha'}, V_{\beta}]\le Z_{\alpha'-1}\cap Z_{\alpha+2}$. Moreover, $R$ is centralized by $V_{\alpha}^{(2)}$ and as $V_{\alpha}^{(2)}\not\le Q_{\alpha'-2}$, $R\le Z_{\alpha'-3}$. If $R\ne Z_{\alpha'-2}$, then it follows that $Z_{\alpha'-1}=Z_{\alpha'-3}$.

If $Z_{\alpha'-2}\ne R$ and $V_{\alpha'}\not\le Q_{\beta}$, then as $R\cap Z_{\beta}=\{1\}$ and since $[V_{\alpha'-2}, V_{\alpha}^{(2)}]\le Z_{\alpha}\cap Q_{\alpha'}=Z_{\beta}$, we must have that $Z_{\beta}=[V_{\alpha'-2}, V_{\alpha}^{(2)}]\le Z_{\alpha'-3}=Z_{\alpha'-1}$ and $Z_{\alpha'-1}=R\times Z_{\beta}\le V_{\beta}$. Thus, $V_{\beta}^{(3)}\cap Q_{\alpha'-2}\le Q_{\alpha'-1}$, $V_{\beta}^{(3)}\cap Q_{\alpha'-2}=V_{\beta}(V_{\beta}^{(3)}\cap Q_{\alpha'})$ and since $Z_{\alpha'}\le Z_{\alpha'-1}\le V_{\beta}$, $V_{\beta}^{(3)}/V_{\beta}$ contains a unique non-central chief factor for $L_{\beta}$ which is an FF-module. Then, by \cref{GoodAction3} and \cref{SimExt}, $Z_{\alpha'-1}=Z_{\alpha'-3}$ implies that $V_{\alpha'}\le V_{\alpha'-1}^{(2)}=V_{\alpha'-3}^{(2)}\le Q_{\alpha}$, a contradiction. Hence, if $Z_{\alpha'-2}\ne R$ then $V_{\alpha'}\le Q_{\beta}$.
\end{proof}

\begin{lemma}\label{GoodCritPair}
Suppose that $C_{V_\beta}(V_{\alpha'})=V_\beta \cap Q_{\alpha'}$ and $b>5$. Then there exists a critical pair $(\alpha^*, {\alpha^*}')$ such that $V_{\alpha^*}^{(2)}\le Q_{{\alpha^*}'-2}$.
\end{lemma}
\begin{proof}
Assume otherwise. Since $V_{\alpha}^{(2)}\not\le Q_{\alpha'-2}$, there is another critical pair $(\alpha-2, \alpha'-2)$ and we may assume recursively, that there is a path $(\alpha',\alpha'-1, \dots, \alpha, \alpha-1, \alpha-2, \alpha-3,\dots)$ such that $(\alpha-2k, \alpha'-2k)$  is a critical pair satisfying $V_{\alpha-2k}^{(2)}\not\le Q_{\alpha'-2k-2}$ for all $k\geq 0$. Set $R_k:=[V_{\alpha-2k+1}, V_{\alpha'-2k}]$ for each critical pair $(\alpha-2k, \alpha'-2k)$. In particular, $R=R_0$. 

Choose $k\geq (b-1)/2$ and suppose that $Z_{\alpha'-2k-1}=Z_{\alpha'-2k-3}$. Then as $k\geq (b-1)/2$, $2k+3\geq b+2$ and so, by assumption, $(\alpha'-2k-3, \alpha'-2k-3+b)$ is a critical pair, a contradiction. Thus, for $k \geq (b-1)/2$, we may assume that for every critical pair $(\alpha-2k, \alpha'-2k)$, we have that $R_k=Z_{\alpha'-2k-2}\le Z_{\alpha-2k+2}$. Now, if $R_k\ne Z_{\alpha-2k+3}$, then $Z_{\alpha-2k+2}\cap Q_{\alpha'-2k+2}>Z_{\alpha-2k+3}$ a contradiction as $k\geq 1$ and $(\alpha-2k+2, \alpha'-2k+2)$ is a critical pair. Thus, we may assume that $Z_{\alpha'-2k-2}=Z_{\alpha-2k+3}$ for sufficiently large $k$. Then, $R_k=R_{k+1}$ for otherwise $R_k R_{k+1}\le Z_{\alpha-2k+2}\cap Q_{\alpha'-2k+2}>Z_{\alpha-2k+3}$ since $b>5$. In particular, $Z_{\beta-2k}=Z_{\alpha-1-2k}$ and $(\alpha-(b-1)-2k, \beta-2k)$ is a critical pair with $R_{\frac{b-1}{2}-k}=Z_{\beta-2k-2}=Z_{\alpha-1-2k}=Z_{\beta-2k}$. But then $O^p(L_{\beta-2k})$ centralizes $V_{\beta-2k}/Z_{\beta-2k}$, a contradiction.
\end{proof}

We aim to show that $b\leq 5$, and by \cref{GoodCritPair}, we can fix some pair $(\alpha, \alpha')$ with $V_{\alpha}^{(2)}\le Q_{\alpha'-2}$. We start with the case where $V_{\alpha'}\le Q_{\beta}$. 

\begin{lemma}\label{VA41}
Suppose that $C_{V_\beta}(V_{\alpha'})=V_\beta \cap Q_{\alpha'}$ and $b>5$. Assume that $V_{\alpha'}\le Q_{\beta}$ and $V_{\alpha}^{(2)}\le Q_{\alpha'-2}$. Then $V_{\alpha}^{(2)}\le Q_{\alpha'-1}$.
\end{lemma}
\begin{proof}
Suppose for a contradiction that $V_{\alpha}^{(2)}\not\le Q_{\alpha'-1}$. Then, as $R\le Z_{\alpha'-1}$, we conclude that $R=Z_{\alpha'-2}$. Let $\alpha-1\in\Delta(\alpha)$ such that $V_{\alpha-1}\not\le Q_{\alpha'-1}$. If $Z_{\alpha'-1}\le Q_{\alpha-1}$, then $Z_{\alpha'-2}=[V_{\alpha-1}, Z_{\alpha'-1}]=Z_{\alpha-1}$ from which it follows that $Z_{\alpha-1}=Z_{\beta}$. Then, recalling that $O^p(R_{\alpha})$ centralizes $V_{\alpha}^{(2)}$ by \cref{VA1}, by \cref{SimExt} we have that $V_{\alpha-1}=V_{\beta}\le Q_{\alpha'-1}$, a contradiction. Thus, $(\alpha'-1, \alpha-1)$ is a critical pair and $Z_{\alpha-1}\ne Z_{\beta}$. 

Note that if $Z_{\alpha'-2}=Z_{\alpha'-4}$, then $Z_{\alpha'-1}\le V_{\alpha'-2}=V_{\alpha'-4}$ is centralized by $V_{\alpha}^{(2)}$, a contradiction. Thus, $Z_{\alpha'-3}$ is centralized by $V_{\alpha}^{(4)}$ and either $Z_{\alpha'-3}=Z_{\alpha'-5}$ or $V_{\alpha'-4}=Z_{\alpha'-3}Z_{\alpha'-5}$ and $V_{\alpha}^{(4)}\le Q_{\alpha'-3}$. Assume that $Z_{\alpha'-3}=Z_{\alpha'-5}$. Notice that $Z_{\alpha'-1}\le V_{\alpha'-3}^{(2)}$ and $[V_{\alpha}^{(2)}, V_{\alpha'-5}^{(2)}]=\{1\}$, and so by \cref{SimExt} and \cref{GoodAction3}, there is not a unique non-central chief factor within $V_{\alpha-1}^{(3)}/V_{\alpha-1}$ which is an FF-module. Suppose that $V_{\alpha-1}^{(3)}\le Q_{\alpha'-4}$. Then $V_{\alpha-1}^{(3)}\le Q_{\alpha'-3}$ and $V_{\alpha-1}^{(3)}\cap Q_{\alpha'-2}=V_{\alpha-1}(V_{\alpha-1}^{(3)}\cap Q_{\alpha'-1})$ has index $q$ in $V_{\alpha-1}^{(3)}$ and is centralized, modulo $V_{\alpha-1}$, by $Z_{\alpha'-1}$, yielding a contradiction. Thus, there is $\alpha-4\in\Delta^{(3)}(\alpha-1)$ such that $(\alpha-4, \alpha'-4)$ is a critical pair. Then $\{1\}\ne [V_{\alpha'-4}, V_{\alpha-3}]\le Z_{\alpha-2}\cap Z_{\alpha'-5}$. If $[V_{\alpha'-4}, V_{\alpha-3}]\ne Z_{\alpha-1}$ then, as $b>5$, $Z_{\alpha-1}<Z_{\alpha-1}[V_{\alpha'-4}, V_{\alpha-3}]\le Z_{\alpha-2}\cap Q_{\alpha'-1}$, a contradiction. Thus, again as $b>5$, $Z_{\alpha}=[V_{\alpha'-4}, V_{\alpha-3}]\times Z_{\beta}\le Q_{\alpha'}$, another contradiction.

Hence, $Z_{\alpha'-3}\ne Z_{\alpha'-5}$ and $V_{\alpha}^{(4)}\le Q_{\alpha'-3}$. It follows that $Z_{\alpha'-2}\le [V_{\alpha}^{(4)}, V_{\alpha'-2}]\le Z_{\alpha'-3}$. If $Z_{\alpha'-2}=[V_{\alpha}^{(4)}, V_{\alpha'-2}]$, then $V_{\alpha}^{(4)}=V_{\alpha}^{(2)}(V_{\alpha}^{(4)}\cap Q_{\alpha'})$ and since $Z_{\alpha'}\not\le V_{\alpha}^{(4)}$, otherwise $V_{\alpha}^{(2)}$ centralizes $Z_{\alpha'-1}=Z_{\alpha'}\times R$, it follows that $V_{\alpha'}$ centralizes $V_{\alpha}^{(4)}/V_{\alpha}^{(2)}$, a contradiction. Thus, $[V_{\alpha}^{(4)}, V_{\alpha'-2}]=Z_{\alpha'-3}$. Since $V_{\alpha}^{(4)}\cap Q_{\alpha'-2}=V_{\alpha}^{(2)}(V_{\alpha}^{(4)}\cap Q_{\alpha'})$, we have that $V_{\alpha}^{(4)}/V_{\alpha}^{(2)}$ contains a unique non-central chief factor and by \cref{GoodAction2}, $O^p(R_{\alpha})$ centralizes $V_{\alpha}^{(4)}$. Furthermore, since $V_{\alpha-1}^{(3)}\not\le Q_{\alpha'-2}$, otherwise $Z_{\alpha'-1}$ centralizes $V_{\alpha-1}^{(3)}/V_{\alpha-1}$, we may suppose that $Z_{\alpha'-3}=[V_{\alpha-1}^{(3)}, V_{\alpha'-2}]$.
 
Suppose first that $b>9$. Then, $V_{\alpha}^{(6)}$ centralizes $Z_{\alpha'-3}\le V_{\alpha-1}^{(3)}$ and so centralizes $Z_{\alpha'-4}Z_{\alpha'-6}$. If $Z_{\alpha'-4}=Z_{\alpha'-6}$, then by \cref{SimExt} we have that $Z_{\alpha'-1}\le V_{\alpha'-4}^{(3)}=V_{\alpha'-6}^{(3)}$ is centralized by $V_{\alpha}^{(2)}$, a contradiction. Thus, $V_{\alpha}^{(6)}$ centralizes $Z_{\alpha'-5}$ and so either $Z_{\alpha'-5}=Z_{\alpha'-7}$ or $V_{\alpha}^{(6)}$ centralizes $Z_{\alpha'-3}Z_{\alpha'-5}Z_{\alpha'-7}=V_{\alpha'-6}V_{\alpha'-4}$. In the latter case, $V_{\alpha}^{(6)}=V_{\alpha}^{(4)}(V_{\alpha}^{(6)}\cap Q_{\alpha'-2})$ and since $Z_{\alpha'}\not\le V_{\alpha}^{(6)}$, we conclude that $O^p(L_{\alpha})$ centralizes $V_{\alpha}^{(6)}/V_{\alpha}^{(4)}$, a contradiction. Thus, $Z_{\alpha'-5}=Z_{\alpha'-7}$ and as $Z_{\alpha'-1}\le V_{\alpha'-5}^{(4)}$ and $V_{\alpha}^{(2)}$ centralizes $V_{\alpha'-7}^{(4)}$, by \cref{SimExt}, \cref{GoodAction3} and \cref{GoodAction4}, we need only show that both $V_{\beta}^{(5)}/V_{\beta}^{(3)}$ and $V_{\beta}^{(3)}/V_{\beta}$ contain a unique non-central chief factor which is an FF-module for $\bar{L_{\beta}}$. We may prove it for any $\lambda\in\beta^G$ and, following the steps in an earlier part of this proof, we infer that $V_{\beta}^{(3)}/V_{\beta}$ satisfies the required condition. By the steps above, $V_{\alpha-1}^{(3)}\not\le Q_{\alpha'-2}$. Then, as $V_{\alpha'-4}=Z_{\alpha'-3}Z_{\alpha'-7}$ is centralized by $V_{\alpha-1}^{(5)}$, $V_{\alpha-1}^{(5)}\cap Q_{\alpha'-6}=V_{\alpha-1}^{(3)}(V_{\alpha-1}^{(3)}\cap Q_{\alpha'-2})$ and since $V_{\alpha'-2}\not\le Q_{\alpha-1}$ and $Z_{\alpha'-2}\le V_{\alpha-1}^{(3)}$, $V_{\alpha-1}^{(5)}/V_{\alpha-1}^{(3)}$ contains a unique non-central chief factor and satisfies the required conditions. This provides the contradiction.

Suppose that $b=7$. Then $C_{Q_{\alpha}}(V_{\alpha}^{(4)})\le Q_{\alpha+4}=Q_{\alpha'-3}$. Thus, $V_{\alpha}^{(4)}C_{Q_{\alpha}}(V_{\alpha}^{(4)})=V_{\alpha}^{(4)}(V_{\alpha}^{(4)}C_{Q_{\alpha}}(V_{\alpha}^{(4)})\cap Q_{\alpha'})$ and since $Z_{\alpha'}\not\le C_{Q_{\alpha}}(V_{\alpha}^{(2)})\ge C_{Q_{\alpha}}(V_{\alpha}^{(4)})$, $O^p(L_{\alpha})$ centralizes $V_{\alpha}^{(4)}C_{Q_{\alpha}}(V_{\alpha}^{(4)})/V_{\alpha}^{(4)}$. Then for $r\in O^p(R_{\alpha})$ of order coprime to $p$, $[r, Q_{\alpha}, V_{\alpha}^{(4)}]=\{1\}$ by the three subgroup lemma and so $[Q_{\alpha}, r]=[Q_{\alpha}, r,r,r]\le [C_{Q_{\alpha}}(V_{\alpha}^{(4)}), r, r]\le [V_{\alpha}^{(4)}, r]=\{1\}$ so that $R_{\alpha}=Q_{\alpha}$ and $\bar{L_{\alpha}}\cong\SL_2(q)$. We may assume that $V_{\alpha-1}^{(3)}\le Q_{\alpha'-4}$, $V_{\alpha-1}^{(3)}\not\le Q_{\alpha'-2}$ and $O^p(R_{\alpha-1})$ centralizes $V_{\alpha-1}^{(3)}$. Moreover, $Z_{\alpha'-3}=[V_{\alpha'-2}, V_{\alpha-1}^{(3)}]\le V_{\alpha-1}^{(3)}$ and so $Z_{\alpha'-3}$ is centralized by $C_{Q_{\alpha-1}}(V_{\alpha-1}^{(3)})$. Since $Z_{\alpha'-3}\ne Z_{\alpha+2}$, otherwise by \cref{SimExt}, $Z_{\alpha}\le V_{\alpha+2}^{(2)}=V_{\alpha'-3}^{(2)}\le Q_{\alpha'}$, we have that $C_{Q_{\alpha-1}}(V_{\alpha-1}^{(3)})$ centralizes $V_{\alpha+3}$. It follows that $C_{Q_{\alpha-1}}(V_{\alpha-1}^{(3)})=V_{\alpha-1}^{(3)}(C_{Q_{\alpha-1}}(V_{\alpha-1}^{(3)})\cap Q_{\alpha'-2})$ and so $O^p(L_{\alpha-1})$ centralizes $C_{Q_{\alpha-1}}(V_{\alpha-1}^{(3)})/V_{\alpha-1}^{(3)}$. Now, letting $r\in O^p(R_{\alpha-1})$ of order coprime to $p$, $[r, Q_{\alpha-1}, V_{\alpha-1}^{(3)}]=\{1\}$ by the three subgroup lemma and $[Q_{\alpha-1}, r]=[Q_{\alpha-1}, r,r,r]=[C_{Q_{\alpha-1}}(V_{\alpha-1}^{(3)}), r,r]=[V_{\alpha-1}^{(3)}, r]=\{1\}$ so that $R_{\alpha-1}=Q_{\alpha-1}$ and $\bar{L_{\alpha-1}}\cong\SL_2(q)$. Thus, $G$ has a weak BN-pair of rank $2$ and by \cite{Greenbook}, no examples exist.

Suppose that $b=9$. Then $C_{Q_{\alpha}}(V_{\alpha}^{(4)})\le Q_{\alpha+4}=Q_{\alpha'-5}$. Moreover, $Z_{\alpha'-5}\ne Z_{\alpha'-3}\le V_{\alpha}^{(4)}$ so that $C_{Q_{\alpha}}(V_{\alpha}^{(4)})\le Q_{\alpha'-3}$ and $C_{Q_{\alpha}}(V_{\alpha}^{(4)})=V_{\alpha}^{(4)}(C_{Q_{\alpha}}(V_{\alpha}^{(4)})\cap Q_{\alpha'-2})$ and it follows that $O^p(L_{\alpha})$ centralizes $C_{Q_{\alpha}}(V_{\alpha}^{(4)})/V_{\alpha}^{(4)}$. As in the $b=7$ case, we get that $\bar{L_{\alpha}}\cong\SL_2(q)$. Since $Z_{\alpha'-3}\le V_{\alpha-1}^{(3)}$, $Z_{\alpha'-4}$ is centralized by $C_{Q_{\alpha-1}}(V_{\alpha-1}^{(3)})$ and $Z_{\alpha'-6}=Z_{\alpha+3}$ is centralized by $C_{Q_{\alpha-1}}(V_{\alpha-1}^{(3)})$ from which it follows that $C_{Q_{\alpha-1}}(V_{\alpha-1}^{(3)})$ centralizes $Z_{\alpha+4}=Z_{\alpha'-5}$. Continuing as above, we see that $C_{Q_{\alpha-1}}(V_{\alpha-1}^{(3)})=V_{\alpha-1}^{(3)}(C_{Q_{\alpha-1}}(V_{\alpha-1}^{(3)})\cap Q_{\alpha'-2})$ and $O^p(L_{\alpha-1})$ centralizes $C_{Q_{\alpha-1}}(V_{\alpha-1}^{(3)})/V_{\alpha-1}^{(3)}$ and an application of the three subgroup lemma and coprime action yields that $\bar{L_{\alpha-1}}\cong\SL_2(q)$ and $G$ has a weak BN-pair of rank $2$. By \cite{Greenbook}, no examples exist and the proof is complete.
\end{proof}

\begin{lemma}\label{VA4NQ}
Suppose that $C_{V_\beta}(V_{\alpha'})=V_\beta \cap Q_{\alpha'}$ and $b>5$. If $V_{\alpha'}\le Q_{\beta}$ then $V_{\alpha}^{(4)}\not\le Q_{\alpha'-4}$.
\end{lemma}
\begin{proof}
By \cref{VA41}, we may suppose that $V_{\alpha}^{(2)}\le Q_{\alpha'-1}$. Note that by \cref{VA1}, $Z_{\alpha'-1}=Z_{\alpha'}\times Z_{\beta}\le V_{\alpha}^{(2)}\le Z(V_{\alpha}^{(4)})$. Suppose that $V_{\alpha}^{(4)}\le Q_{\alpha'-4}$ throughout. If $Z_{\alpha'-1}\ne Z_{\alpha'-3}$, then $V_{\alpha}^{(4)}\cap Q_{\alpha'-3}=V_{\alpha}^{(2)}(V_{\alpha}^{(4)}\cap Q_{\alpha'})$ and since $Z_{\alpha'}\le V_{\alpha}^{(2)}$, $V_{\alpha}^{(4)}$ does not centralize $Z_{\alpha'-3}$. But $Z_{\alpha'-2}\le Z_{\alpha'-1}$ so that $Z_{\alpha'-2}Z_{\alpha'-4}$ is centralized by $V_{\alpha}^{(4)}$ and $Z_{\alpha'-2}=Z_{\alpha'-4}$. Now, both $V_{\alpha}^{(4)}/V_{\alpha}^{(2)}$ and $V_{\alpha}^{(2)}/Z_{\alpha}$ contain unique non-central chief factors and by \cref{GoodAction1} and \cref{GoodAction2}, we deduce that $O^p(R_{\alpha})$ centralizes $V_{\alpha}^{(4)}$. Therefore, applying \cref{SimExt} to $Z_{\alpha'-2}=Z_{\alpha'-4}$, we conclude that $V_{\alpha'}\le V_{\alpha'-2}^{(3)}=V_{\alpha'-4}^{(3)}$ is centralized by $Z_{\alpha}$, a contradiction.

Thus, $Z_{\alpha'-1}=Z_{\alpha'-3}$ and $V_{\alpha}^{(4)}\not\le Q_{\alpha'-2}$. In particular, it follows again by \cref{GoodAction2} that $O^p(R_{\alpha})$ centralizes $V_{\alpha}^{(4)}$ and so, similarly to the above, $Z_{\alpha'-2}\ne Z_{\alpha'-4}$. Moreover, by \cref{SimExt}, since $V_{\alpha'}\le V_{\alpha'-1}^{(2)}$ and $V_{\alpha'-3}^{(2)}\le Q_{\alpha}$, $O^p(R_{\beta})$ does not centralize $V_{\beta}^{(3)}$. In particular, $Z_{\alpha'-1}\cap Z_{\alpha+2}=R=Z_\beta$ for otherwise $V_{\alpha'}^{(3)}\cap Q_{\alpha+3}\le Q_{\alpha+2}$, $[V_{\alpha'-3}^{(3)}\cap Q_{\alpha+3}, V_{\beta}]\le Z_{\alpha+2}=Z_{\alpha'-1}\le V_{\alpha'}$ and $V_{\alpha'}^{(3)}/V_{\alpha'}$ contains a unique non-central chief factor which is an FF-module, and we would have a contradiction by \cref{GoodAction3}.

Suppose first that $b=7$. Then $Z_{\beta}Z_{\alpha+3}\le Z_{\alpha+2}\cap Z_{\alpha'-3}$ and so either $Z_{\beta}=Z_{\alpha+3}$ or $Z_{\alpha'-1}=Z_{\alpha'-3}=Z_{\alpha+2}$. The latter case yields an immediate contradiction, while in the former case, \cref{SimExt} implies that $V_{\beta}=V_{\alpha+3}\le Q_{\alpha'}$, another contradiction. Thus, we may assume $b>7$ throughout.

Assume that for $\alpha-4\in\Delta^{(4)}(\alpha)$, whenever $Z_{\alpha-4}\not\le Q_{\alpha'-2}$ we conclude that $Z_{\beta}=Z_{\alpha-1}$. Choose $\delta\in\Delta(\alpha)$ such that $Z_{\delta}\ne Z_{\beta}$ so that $V_{\delta}^{(3)}\le Q_{\alpha'-2}$. Moreover, $V_{\delta}^{(3)}$ centralizes $Z_{\alpha'-1}\le V_{\alpha}^{(2)}$ and $[V_{\delta}^{(3)}, V_{\alpha'}]=[V_{\alpha}^{(2)}, V_{\alpha'}][V_{\delta}^{(3)}\cap Q_{\alpha'}, V_{\alpha'}]\le V_{\alpha}^{(2)}$. Thus, $V_{\delta}^{(3)}\normaleq L_{\alpha}=\langle V_{\alpha'}, R_{\alpha}, Q_{\delta}\rangle$, a contradiction. Thus, we may assume that there exists $\alpha-4\in\Delta^{(4)}(\alpha)$ with $Z_{\alpha-4}\not\le Q_{\alpha'-2}$ and $Z_{\beta}\ne Z_{\alpha-1}$.

Suppose that $V_{\alpha'-2}\not\le Q_{\alpha-1}$. Since $V_{\alpha}^{(2)}\le Q_{\alpha'-2}$, it follows that $Z_{\alpha'-2}=[V_{\alpha}^{(2)}, V_{\alpha'-2}]=Z_{\beta}$. Moreover, there is $\lambda\in\Delta(\alpha'-2)$ such that $(\lambda, \alpha-1)$ is a critical pair with $V_{\alpha-1}\le Q_{\alpha'-2}$. If $V_{\lambda}^{(2)}\le Q_{\beta}$, then by \cref{VA41} $V_{\lambda}^{(2)}\le Q_{\alpha}$ and $Z_{\alpha}\le V_{\lambda}^{(2)}$, a contradiction since $b>5$. Thus, $V_{\lambda}^{(2)}\not\le Q_{\beta}$ and $(\lambda+2, \beta)$ is also a critical pair. Moreover, $\{1\}\ne [V_{\beta}, V_{\lambda+1}]\le Z_{\alpha+2}\cap Z_{\lambda}$. Since $Z_{\lambda}\not\le Q_{\alpha-1}$ and $Z_{\alpha'-2}\le V_{\alpha}^{(2)}$, it follows that $[V_{\beta}, V_{\lambda+1}]=Z_{\alpha'-2}=Z_{\beta}$. But then $V_{\lambda+1}\le Q_{\beta}$, a contradiction. Thus, $V_{\alpha'-2}\le Q_{\alpha-1}$ and $[V_{\alpha'-2}, V_{\alpha-1}]=\{1\}$, otherwise $Z_{\alpha-1}=[V_{\alpha'-2}, V_{\alpha-1}]=Z_{\alpha'-2}$ and since $Z_{\alpha}\not\le V_{\alpha'-2}$, $Z_{\alpha-1}=Z_{\beta}$, a contradiction. Therefore, $V_{\alpha'-2}\le Q_{\alpha-2}$. 

Suppose that $[V_{\alpha'-2}, V_{\alpha-3}]=Z_{\beta}$ so that $Z_{\alpha'-2}\cap Z_{\beta}=\{1\}$. As $Z_{\beta}\le Z_{\alpha-2}$ and $Z_{\beta}\cap Z_{\alpha-1}=\{1\}$, $Z_{\alpha}=Z_{\alpha-2}$. Immediately, we have that $[V_{\alpha}^{(2)}, V_{\alpha'-2}]\le Z_{\alpha'-2}\cap Z_{\alpha}=\{1\}$ so that $V_{\alpha}^{(2)}\le C_{\alpha'-2}$. 

Choose $\lambda\in\Delta(\alpha'-2)$ such that $Z_\lambda\ne Z_{\alpha'-1}$ and set $W^{\alpha'-2}:=\langle V_{\delta}^{(2)}\mid Z_{\delta}=Z_{\lambda}, \delta\in\Delta(\alpha'-2)\rangle$. Then, for $\delta\in\Delta(\alpha'-2)$ with $Z_{\delta}=Z_{\lambda}$, since $V_{\alpha}^{(2)}\le C_{\alpha'-2}$, we have that $[V_\beta, V_{\delta}^{(2)}]\le Z_{\delta}\cap Z_{\alpha+2}$. Since $Z_{\alpha+2}\le Z(V_{\alpha}^{(4)})$, $Z_{\delta}\cap Z_{\alpha+2}\le Z_{\alpha'-2}$, otherwise $V_{\alpha}^{(4)}$ centralizes $V_{\alpha'-2}=Z_{\delta}Z_{\alpha'-1}$. But now $[V_\beta, V_{\delta}^{(2)}]$ is either trivial or equal to $Z_{\beta}$ and since $Z_{\alpha'-2}\ne Z_{\beta}$, we have that $[V_\beta, V_{\delta}^{(2)}]=\{1\}$. Now, $[V_{\alpha}^{(2)}, V_{\lambda}^{(2)}]\le Z_{\lambda}\cap Z_{\alpha}$ and for a similar reason as before, $[V_{\alpha}^{(2)}, V_{\lambda}^{(2)}]=\{1\}$. It follows that $W^{\alpha'-2}\le Q_{\alpha-2}$ and $Z_{\beta}\le [W^{\alpha'-2}, V_{\alpha-3}]\le Z_{\alpha-2}=Z_{\alpha}$. Since $Z_{\alpha}\not\le V_{\alpha'-2}^{(3)}$, we have that $[W^{\alpha'-2}, V_{\alpha-3}]=Z_{\beta}\le V_{\alpha'-2}$ and $V_{\alpha-3}$ centralizes $W^{\alpha'-2}/V_{\alpha'-2}$. But now, by \cref{UWNormal}, $W^{\alpha'-2}\normaleq L_{\alpha'-2}=\langle V_{\alpha-3}, R_{\alpha'-2}, Q_{\lambda}\rangle$. Since $V_{\alpha-3}$ centralizes $W^{\alpha'-2}/V_{\alpha'-2}$, it follows that $V_{\lambda}^{(2)}\normaleq L_{\alpha'-2}$, a contradiction.

Suppose now that $Z_{\beta}\ne [V_{\alpha'-2}, V_{\alpha-3}]\le Z_{\alpha-2}\cap Z_{\alpha'-3}$. Then $Z_{\alpha}\ne Z_{\alpha-2}$, else $Z_{\beta}<Z_{\beta}[V_{\alpha'-2}, V_{\alpha-3}]\le Z_{\alpha}\cap Z_{\alpha'-3}$, an obvious contradiction. Still, $Z_{\alpha}<Z_{\alpha}[V_{\alpha'-2}, V_{\alpha-3}]\le V_{\alpha-1}\cap Z_{\alpha}Z_{\alpha'-1}$.As $V_{\beta}\le C_{\alpha'-2}$, it follows that $Z_{\beta}\le [V_{\beta}, V_{\alpha'-2}^{(3)}]\le Z_{\alpha+2}\cap V_{\alpha'-2}$. Since $Z_{\alpha+2}\cap Z_{\alpha'-1}=Z_{\beta}$, $Z_{\alpha+2}\cap V_{\alpha'-2}=Z_{\beta}$, otherwise $Z_{\alpha'-1}<Z_{\alpha'-1}(Z_{\alpha+2}\cap Z_{\alpha+2}\le V_{\alpha}^{(2)}$ would be centralized by $V_{\alpha}^{(4)}$. Thus, $[V_{\beta}, V_{\alpha'-2}^{(3)}]=Z_{\beta}$ and $V_{\alpha'-2}^{(3)}\le Q_{\beta}$.

Then $V_{\alpha'-2}^{(3)}\cap Q_{\alpha}$ centralizes $Z_{\alpha}[V_{\alpha'-2}, V_{\alpha-3}]\le V_{\alpha-1}$ and so $V_{\alpha'-2}^{(3)}\cap Q_{\alpha}\le Q_{\alpha-2}$. Then $[V_{\alpha'-2}, V_{\alpha-3}]\le [V_{\alpha'-2}^{(3)}\cap Q_{\alpha}, V_{\alpha-3}]\le Z_{\alpha-2}$.  If $[V_{\alpha'-2}^{(3)}\cap Q_{\alpha}, V_{\alpha-3}]=[V_{\alpha'-2}, V_{\alpha-3}]$, then $V_{\alpha'-2}^{(3)}/V_{\alpha'-2}$ contains a unique non-central chief factor which is an FF-module. By \cref{GoodAction3}, $O^p(R_{\alpha'-2})$ centralizes $V_{\alpha'-2}^{(3)}$ and \cref{SimExt} applied to $Z_{\alpha'-1}=Z_{\alpha'-3}$ implies that $V_{\alpha'}\le V_{\alpha'-1}^{(2)}=V_{\alpha'-3}^{(2)}\le Q_{\alpha}$, a contradiction. Thus, $Z_{\alpha-1}\le Z_{\alpha-2}\le V_{\alpha'-2}^{(3)}$ and since $b>5$, we have that $Z_{\beta}=Z_{\alpha-1}$, a final contradiction by the choice of $\alpha-4$.
\end{proof}

By \cref{VA4NQ}, whenever $b>5$ and $V_{\alpha'}\le Q_{\beta}$, we may assume that there is a critical pair $(\alpha-4, \alpha'-4)$. In the following lemma, we let $(\alpha-4, \alpha'-4)$ be such a pair and and investigate the action of $V_{\alpha'-4}$ on $V_{\alpha-3}$ and vice versa.

\begin{lemma}\label{TricAction}
Suppose that $C_{V_\beta}(V_{\alpha'})=V_\beta \cap Q_{\alpha'}$ and $b>5$. If $V_{\alpha'}\le Q_{\beta}$ then $b>7$, $Z_{\alpha}\ne Z_{\alpha-2}$, $O^p(R_{\beta})$ centralizes $V_{\beta}^{(3)}$ and setting $R^\dagger:=[V_{\alpha'-4}, V_{\alpha-3}]$, either:
\begin{enumerate}
\item $R^\dagger=Z_{\alpha-1}=Z_{\beta}$; or
\item $R^\dagger\ne Z_{\alpha-1}$.
\end{enumerate}
\end{lemma}
\begin{proof}
By \cref{VA41}, $V_{\alpha}^{(2)}\le Q_{\alpha'-1}$, $Z_{\alpha'-1}=Z_{\alpha'}\times Z_{\beta}\le V_{\alpha}^{(2)}\le Z(V_{\alpha}^{(4)})$, $V_{\alpha}^{(4)}\not\le Q_{\alpha'-4}$ and there is a critical pair $(\alpha-4, \alpha'-4)$. Set $R^\dagger:=[V_{\alpha'-4}, V_{\alpha-3}]\le Z_{\alpha'-5}\cap Z_{\alpha-2}$. By assumption $R^\dagger\cap Z_{\alpha'-4}=\{1\}$.

Suppose first that $R^\dagger=Z_{\alpha-1}\le Z_{\alpha'-5}$. Then, as $b>5$, $Z_{\alpha-1}=Z_{\beta}$ so that by \cref{SimExt}, $V_{\alpha-1}=V_{\beta}$. Then $[V_{\alpha'-4}^{(3)}, V_{\alpha-1}]=[V_{\alpha'-4}^{(3)}, V_{\beta}]=\{1\}$ and so $V_{\alpha'-4}^{(3)}\le Q_{\alpha-2}$. Moreover, $V_{\alpha'-4}\not\le Q_{\alpha-3}$, else $Z_{\alpha-3}=R^\dagger=Z_{\alpha-1}$ and by \cref{SimExt}, $V_{\alpha-3}=V_{\alpha-1}\le Q_{\alpha'-4}$, a contradiction as $(\alpha-4, \alpha'-4)$ is a critical pair. Then $V_{\alpha'-4}(V_{\alpha'-4}^{(3)}\cap Q_{\alpha-3}\cap Q_{\alpha-4})$ is an index $q$ subgroup of $V_{\alpha'-4}^{(3)}$ which is centralized, modulo $V_{\alpha'-4}$, by $Z_{\alpha-4}$ and so, $V_{\alpha'-4}^{(3)}/V_{\alpha'-4}$ contains a unique non-central chief factor and by \cref{GoodAction3}, and conjugacy, $O^p(R_{\beta})$ centralizes $V_{\beta}^{(3)}$. In particular, applying \cref{SimExt}, if $Z_{\alpha}=Z_{\alpha-2}$ then $Z_{\alpha-4}\le V_{\alpha-2}^{(2)}=V_{\alpha}^{(2)}\le Q_{\alpha'-4}$, a contradiction. Hence, subject to proving $b>7$, (i) holds.

Assume now that $R^\dagger\ne Z_{\alpha-1}$ so that $Z_{\alpha-1}R^\dagger\le Z_{\alpha-2}$ properly contains $Z_{\alpha-1}$ and is centralized by $V_{\alpha'-4}^{(3)}$. If $Z_{\alpha}=Z_{\alpha-2}$ then $Z_{\beta}R^\dagger \le Z_{\alpha}\cap Q_{\alpha'}$ and we deduce that $R^\dagger=Z_{\beta}\le V_{\alpha'-4}$. Now, $V_{\alpha'-4}^{(3)}$ centralizes $Z_{\alpha}=Z_{\alpha-2}$ so that $V_{\alpha'-4}^{(3)}\cap Q_{\alpha-1}\le Q_{\alpha-2}$ and $[V_{\alpha'-4}^{(3)}\cap Q_{\alpha-1}, V_{\alpha-3}]\le Z_{\alpha-2}\cap V_{\alpha'-4}^{(3)}$. Indeed, unless $b=7$, $Z_{\alpha-2}\cap V_{\alpha'-4}^{(3)}\le Z_{\alpha}\cap Q_{\alpha'}=Z_{\beta}\le V_{\alpha'-4}$ and $V_{\alpha'-4}^{(3)}/V_{\alpha'-4}$ contains a unique non-central chief factor. By \cref{GoodAction3}, $O^p(R_{\beta})$ centralizes $V_{\beta}^{(3)}$ and \cref{SimExt} applied to $Z_{\alpha}=Z_{\alpha-2}$ gives a contradiction as above.

If $R^\dagger\ne Z_{\alpha-1}$ and $Z_{\alpha}\ne Z_{\alpha-2}$ then it follows that $V_{\alpha'-4}^{(3)}$ centralizes $V_{\alpha-1}$ and $V_{\alpha'-4}^{(3)}\cap Q_{\alpha-3}\cap Q_{\alpha-4}$ is an index $q^2$ subgroup of $V_{\alpha'-4}^{(3)}$ centralized by $Z_{\alpha-4}$. Hence, $V_{\alpha'-4}^{(3)}$ contains only two non-central chief factors for $L_{\alpha'-4}$, one in $V_{\alpha'-4}$ and one in $V_{\alpha'-4}^{(3)}/V_{\alpha'-4}$. Moreover, both non-central chief factors are FF-modules for $\bar{L_{\alpha'-4}}$ and by \cref{GoodAction3}, and conjugacy, we have that $O^p(R_{\beta})$ centralizes $V_{\beta}^{(3)}$ and again, subject to proving $b>7$, (ii) holds.

Thus, to complete the proof it remains to prove that $b>7$, so assume that $b=7$ for the remainder of the proof. Suppose first that $R=Z_{\beta}=Z_{\alpha'-2}$. Since $Z_{\beta}\ne Z_{\alpha+3}=Z_{\alpha'-4}$, for otherwise by \cref{SimExt}, $V_{\beta}=V_{\alpha+3}\le Q_{\alpha'}$, we may assume that $Z_{\alpha+2}=Z_{\beta}\times Z_{\alpha+3}=Z_{\alpha'-2}\times Z_{\alpha'-4}=Z_{\alpha'-3}$. If $O^p(R_{\beta})$ centralizes $V_{\beta}^{(3)}$ then \cref{SimExt} applied to $Z_{\alpha+2}=Z_{\alpha'-3}$ implies that $Z_{\alpha}\le V_{\alpha+2}^{(2)}=V_{\alpha'-3}^{(2)}\le Q_{\alpha'}$, a contradiction.

But now, $V_{\alpha'}^{(3)}\cap Q_{\alpha+3}$ centralizes $Z_{\alpha+2}=Z_{\alpha'-3}$ and $[V_{\alpha'}^{(3)}\cap Q_{\alpha+3}\cap Q_{\beta}, V_{\beta}]\le Z_{\beta}=Z_{\alpha'-2}\le V_{\alpha'}$. Moreover, $[V_{\alpha'}^{(3)}, V_{\beta}, V_{\beta}]\le [V_{\alpha'}^{(3)}, V_{\alpha+3}^{(3)}, V_{\alpha+3}^{(3)}]=\{1\}$ and applying \cref{SEQuad} and \cref{2FRecog}, we deduce that $O^p(R_{\beta})$ centralizes $V_{\beta}^{(3)}$ unless perhaps, $q\in\{2,3\}$.

Now, $[V_{\alpha'}^{(3)}\cap Q_{\alpha+3}, V_{\beta}]\le Z_{\alpha+2}\cap V_{\alpha'}^{(3)}$. In particular, we deduce that $Z_{\alpha'-3}\ne Z_{\alpha'-1}$ for otherwise $V_{\alpha'}^{(3)}/V_{\alpha'}$ contains a unique non-central chief factor for $L_{\alpha'}$ and by \cref{GoodAction3}, $O^p(R_{\alpha'})$ centralizes $V_{\alpha'}^{(3)}$. But then, recalling from \cref{VA1} that $Z_{\alpha'-1}\le V_{\alpha}^{(2)}$, we have that $V_{\alpha'-2}=Z_{\alpha'-1}Z_{\alpha'-3}=Z_{\alpha'-1}Z_{\alpha+2}\le V_{\alpha}^{(2)}$. Since $V_{\alpha'-2}\le Q_{\alpha'}$, $Z_{\alpha}\not\le V_{\alpha'-2}$ and so $Z_{\alpha}V_{\alpha'-2}$ is a subgroup of $V_{\alpha}^{(2)}$ of order $p^4$. Now, $V_{\alpha}^{(2)}/Z_{\alpha}$ is a FF-module for $\bar{L_{\alpha}}$ and $V_{\beta}/Z_{\alpha}$ has order $p$ and generates $V_{\alpha}^{(2)}/Z_{\alpha}$, we infer that $p^4\leq |V_{\alpha}^{(2)}|\leq p^5$. If $|V_{\alpha}^{(2)}|=p^4$, then $[V_{\alpha}^{(2)}, V_{\alpha'}]=[V_{\alpha'-2}Z_{\alpha}, V_{\alpha'}]=Z_{\beta}$, a contradiction by \cref{VA1}. Thus, $|V_{\alpha}^{(2)}|=p^5$ and the preimage of $C_{V_{\alpha}^{(2)}/Z_{\alpha}}(O^p(L_{\alpha}))$ in $V_{\alpha}^{(2)}$, which we write as $C^{\alpha}$, has order $p^3$. By the action of $Q_{\beta}$ on $V_{\alpha}^{(2)}$, we must have that $C^{\alpha}V_{\beta}\le [V_{\alpha}^{(2)}, Q_{\beta}]V_{\beta}$. Moreover, since $Z_{\alpha}=Z(Q_{\alpha})$, we must have that $[Q_{\alpha}, C^{\alpha}]=Z_{\alpha}$

If $[V_{\beta}^{(3)}, Q_{\beta}]V_{\beta}/V_{\beta}$ is centralized by $O^p(L_{\beta})$ then we have that $C^{\alpha}V_{\beta}\normaleq L_{\beta}$. But then $Z_{\beta}\le [C^{\alpha}V_{\beta}, Q_{\beta}]\le Z_{\alpha}$ so that $[C^{\alpha}V_{\beta}, Q_{\beta}]=Z_{\beta}$. Then, we deduce that $C_{Q_{\alpha}}(C^{\alpha})\le Q_{\beta}$ for otherwise $Z_{\alpha}=[Q_{\alpha}, C^{\alpha}]=[Q_{\alpha}\cap Q_{\beta}, C^{\alpha}]\le Z_{\beta}$, a contradiction. But now, as $C^{\alpha'-1}V_{\alpha'-2}\normaleq L_{\alpha'-2}$, $V_{\beta}$ centralizes $C^{\alpha'-1}\le C^{\alpha'-3}V_{\alpha'-2}$ so that $V_{\beta}\le C_{Q_{\alpha'-1}}(C^{\alpha'-1})\le Q_{\alpha'}$, a contradiction. 

Thus, $[V_{\beta}^{(3)}, Q_{\beta}]V_{\beta}/V_{\beta}$ contains a non-central chief factor for $L_{\beta}$. Moreover, since $V_{\alpha'}^{(3)}\cap Q_{\alpha+3}\le Q_{\alpha+2}$, an index $p^2$ subgroup of $V_{\alpha'}^{(3)}/ V_{\alpha'}$ is centralized by $Z_{\alpha}$ and we conclude that $V_{\beta}^{(3)}/V_{\beta}$ contains two non-central chief factors for $L_{\beta}$, one in $V_{\beta}^{(3)}/[V_{\beta}^{(3)}, Q_{\beta}]V_{\beta}$ by \cref{CommCF} and one in $[V_{\beta}^{(3)}, Q_{\beta}]V_{\beta}/V_{\beta}$, and both are FF-modules for $\bar{L_{\beta}}$. Notice that $[V_{\alpha}^{(2)}, Q_{\beta}, Q_{\beta}]\le Z_{\alpha}$ so that $[V_{\beta}^{(3)}, Q_{\beta}, Q_{\beta}]\le V_{\beta}$ and write $R_1:=C_{L_{\beta}}([V_{\beta}^{(3)}, Q_{\beta}]V_{\beta}/V_{\beta})$ and $R_2:=C_{L_{\beta}}(V_{\beta}^{(3)}/[V_{\beta}^{(3)}, Q_{\beta}]V_{\beta})$ so that $L_{\beta}/R_1\cong L_{\beta}/R_2\cong L_{\beta}/R_{\beta}\cong \SL_2(p)$. Indeed, either $p\in\{2,3\}$ and $L_{\beta}=\langle R_1, R_2, S\rangle$ by \cref{Badp3} (ii) or $R_1=R_2$. In the former case, we have that $V_{\alpha}^{(2)}[V_{\beta}^{(3)}, Q_{\beta}]V_{\beta}\normaleq R_2S$ so that $[V_{\alpha}^{(2)}[V_{\beta}^{(3)}, Q_{\beta}]V_{\beta}, Q_{\beta}]V_{\beta}=[V_{\alpha}^{(2)}, Q_{\beta}]V_{\beta}\normaleq R_2S$. But $[V_{\alpha}^{(2)}, Q_{\beta}]V_{\beta}\normaleq R_1S$ so that $[V_{\beta}^{(3)}, Q_{\beta}]V_{\beta}=[V_{\alpha}^{(2)}, Q_{\beta}]V_{\beta}\normaleq L_{\beta}$, impossible as then $[V_{\beta}^{(3)}, Q_{\beta}]V_{\beta}/V_{\beta}$ is centralized by $Q_{\alpha}$, and so centralized by $O^p(L_{\beta})$. Thus, $R_1=R_2$ and as $O^p(R_{\beta})$ does not centralize $V_{\beta}^{(3)}$ and $R_{\beta}$ normalizes $Q_{\alpha}\cap Q_{\beta}$, we satisfy the hypothesis of \cref{SubAmal} with $\lambda=\beta$. Since $b=7$, outcome of \cref{SubAmal} holds and we have that $V_{\alpha}^{(4)}\le \langle Z_{\beta}^X\rangle\le Z(O_p(X))$. In particular, $V_{\alpha}^{(4)}$ is abelian, and by conjugacy $V_{\alpha'}, Z_{\alpha}\le V_{\alpha'-3}^{(4)}$, impossible since $[Z_{\alpha}, V_{\alpha'}]\ne \{1\}$.

Thus, we have that $Z_{\alpha'-2}\ne Z_{\beta}$ and $Z_{\alpha'-2}<Z_{\alpha'-2}Z_{\beta}\le Z_{\alpha'-1}$. If $Z_{\alpha'-2}\not\le Z_{\alpha+2}$, then $Z_{\alpha+2}<Z_{\alpha+2}Z_{\alpha'-2}\le V_{\alpha+3}=V_{\alpha'-4}$ and $Z_{\alpha+2}Z_{\alpha'-2}\le V_{\alpha}^{(2)}$ is centralized by $V_{\alpha}^{(4)}$, a contradiction by \cref{VA4NQ}. But then, since $Z_{\alpha'-2}Z_{\alpha'-4}\le Z_{\alpha+2}\cap Z_{\alpha'-3}$, we either get that $Z_{\alpha+2}=Z_{\alpha'-3}$ or $Z_{\alpha'-2}=Z_{\alpha'-4}$. In the former case, since $Z_{\beta}Z_{\alpha'-2}\le Z_{\alpha'-3}\cap Z_{\alpha'-1}$, we conclude that $Z_{\alpha'-1}=Z_{\alpha'-3}=Z_{\alpha+2}$. Then, $[V_{\alpha'}^{(3)}\cap Q_{\alpha+3}, V_{\beta}]\le Z_{\alpha+2}\le V_{\alpha'}$ and by \cref{GoodAction3}, $O^p(R_{\alpha'})$ centralizes $V_{\alpha'}^{(3)}$. Applying \cref{SimExt} gives $V_{\beta}\le V_{\alpha+2}^{(2)}\le V_{\alpha'-1}^{(2)}\le Q_{\alpha'}$, a contradiction. Hence, $Z_{\alpha+2}\cap Z_{\alpha'-3}=Z_{\alpha'-2}$ and $Z_{\alpha'-2}=Z_{\alpha'-4}$. In particular, $Z_{\alpha+2}=Z_{\beta}\times Z_{\alpha+3}=Z_{\beta}\times Z_{\alpha'-2}=Z_{\alpha'-1}$ and $[V_{\alpha'}^{(3)}\cap Q_{\alpha+3}, V_{\beta}]\le Z_{\alpha+2}\le V_{\alpha'}$ and by \cref{GoodAction3}, $O^p(R_{\beta})$ centralizes $V_{\beta}^{(3)}$. Indeed, $Z_{\alpha}\ne Z_{\alpha-2}$ and $Z_{\alpha'-1}\ne Z_{\alpha'-3}$, else by \cref{SimExt}, $Z_{\alpha-4}\le V_{\alpha-2}^{(2)}=V_{\alpha}^{(2)}\le Q_{\alpha'-4}$ and $V_{\alpha'}\le V_{\alpha'-1}^{(2)}=V_{\alpha'-3}^{(2)}\le Q_{\alpha}$ respectively. 

We will show that whenever $(\alpha-4, \alpha'-4)$ is a critical pair, we have that $Z_{\beta}=Z_{\alpha-1}$. Choose $\alpha-4$ such that $Z_{\alpha-4}\not\le Q_{\alpha'-4}$. By the above, since $Z_{\alpha}\ne Z_{\alpha-2}$, assuming $Z_{\beta}\ne Z_{\alpha-1}$, we deduce that (ii) holds and $R^\dagger:=[V_{\alpha-3}, V_{\alpha'-4}]\not\le Z_{\alpha}$. But $R^\dagger\le Z_{\alpha+2}\le V_{\beta}$ so that $Z_{\alpha}R^\dagger\le V_{\beta}\cap V_{\alpha-1}$ and we deduce that $V_{\beta}=Z_{\alpha}Z_{\alpha-2}=V_{\alpha-1}$. Then, if $Z_{\beta}\ne Z_{\alpha-1}$, $V_{\beta}\normaleq L_{\alpha}=\langle Q_{\beta}, Q_{\alpha-1}, R_{\alpha}\rangle$, a contradiction. Therefore, we have shown that whenever $Z_{\alpha-4}\not\le Q_{\alpha'-4}$, $Z_{\beta}=Z_{\alpha-1}$.

Choose $\delta\in\Delta(\alpha)$ such that $Z_{\delta}\ne Z_{\beta}$ so that $V_{\delta}^{(3)}\le Q_{\alpha'-4}$. Suppose that $V_{\delta}^{(3)}\not\le Q_{\alpha'-3}$. There is $\delta-2\in\Delta^{(2)}(\delta)$ such that $Z_{\alpha'-4}=[V_{\delta-2}, Z_{\alpha'-3}]\le Z_{\delta-1}$ and since $Z_{\alpha'-2}=Z_{\alpha'-4}=Z_{\alpha+3}$, $Z_{\alpha'-2}\le V_{\beta}\cap V_{\delta}$. If $Z_{\alpha'-2}\le Z_{\alpha}$, then $Z_{\alpha}=Z_{\beta}\times Z_{\alpha'-2}=Z_{\alpha'-1}$, a clear contradiction. Thus, $V_{\beta}=Z_{\alpha'-2}Z_{\alpha}=V_{\delta}$. But $Z_{\beta}\ne Z_{\delta}$ so that $V_{\beta}\normaleq L_{\alpha}=\langle Q_{\beta}, Q_{\delta}, R_{\alpha}\rangle$, a contradiction. 

Hence, $V_{\delta}^{(3)}\le Q_{\alpha'-3}$ and since $Z_{\alpha'-3}\ne Z_{\alpha'-1}=Z_{\alpha+2}$, $V_{\delta}^{(3)}$ centralizes $V_{\alpha'-2}$ and $V_{\delta}^{(3)}\le Q_{\alpha'-1}$. Setting $W^{\alpha}:=\langle V_{\lambda}^{(3)}\mid Z_{\lambda}=Z_{\delta}, \lambda\in\Delta(\alpha)\rangle$, we have that $W^{\alpha}=V_{\alpha}^{(2)}(W^{\alpha}\cap Q_{\alpha'})$ and as $Z_{\alpha'}\le V_{\alpha}^{(2)}$, $V_{\alpha'}$ centralizes $W^{\alpha}/V_{\alpha}^{(2)}$. Moreover, since $R_{\alpha}Q_{\delta}$ normalizes $W^{\alpha}$ by \cref{UWNormal},  $W^{\alpha}\normaleq L_{\alpha}=\langle V_{\alpha'}, Q_{\delta}, R_{\alpha}\rangle$. Since $V_{\alpha'}$ centralizes $W^{\alpha}/V_{\alpha}^{(2)}$, $O^p(L_{\alpha})$ centralizes $W^{\alpha}/V_{\alpha}^{(2)}$ and $V_{\delta}^{(3)}\normaleq L_{\alpha}$, a final contradiction. Hence, $b>7$, completing the proof.
\end{proof}

\begin{lemma}\label{VnotB}
Suppose that $C_{V_\beta}(V_{\alpha'})=V_\beta \cap Q_{\alpha'}$ and $b>5$. Then $V_{\alpha'}\not\le Q_{\beta}$.
\end{lemma}
\begin{proof}
Since $V_{\alpha'}\le Q_{\beta}$, by \cref{TricAction}, we may assume that $b>7$ throughout. Recall from \cref{VA1} that $Z_{\alpha'-1}\le V_{\alpha}^{(2)}\le Z(V_{\alpha}^{(4)})$. Notice that by \cref{TricAction}, we have that $O^p(R_{\beta})$ centralizes $V_{\beta}^{(3)}$ and by \cref{SimExt}, if $Z_{\alpha'-1}=Z_{\alpha'-3}$ then $V_{\alpha'}\le V_{\alpha'-1}^{(2)}=V_{\alpha'-3}^{(2)}\le Q_{\alpha}$, a contradiction. Hence, we may assume that $Z_{\alpha'-1}\ne Z_{\alpha'-3}$ throughout the remainder of the proof. We fix $\alpha-4\in\Delta^{(4)}(\alpha)$ with $(\alpha-4, \alpha'-4)$ a critical pair.

Suppose first that $Z_{\alpha'-2}\ne Z_{\alpha'-4}$ so that $Z_{\alpha'-3}=Z_{\alpha'-2}\times Z_{\alpha'-4}$ is centralized by $V_{\alpha}^{(4)}$. Then, $V_{\alpha'-2}=Z_{\alpha'-1}Z_{\alpha'-3}$ is centralized by $V_{\alpha}^{(4)}$ so $V_{\alpha}^{(4)}\cap Q_{\alpha'-4}=V_{\alpha}^{(2)}(V_{\alpha}^{(4)}\cap Q_{\alpha'})$ and since $Z_{\alpha'}\le V_{\alpha}^{(2)}$, it follows from \cref{GoodAction2} that $O^p(R_{\alpha})$ centralizes $V_{\alpha}^{(4)}$. In particular, we deduce that $Z_{\beta}\ne Z_{\alpha-1}$, otherwise by \cref{SimExt} we have that $V_{\alpha-3}\le V_{\alpha-1}^{(3)}=V_{\beta}^{(3)}\le Q_{\alpha'-4}$, a contradiction. Furthermore, as $V_{\alpha}^{(4)}\not\le Q_{\alpha'-4}$ we have that $Z_{\alpha'-3}=Z_{\alpha'-5}$. 

By \cref{TricAction}, $Z_{\alpha}\ne Z_{\alpha-2}$, $O^p(R_{\beta})$ centralizes $V_{\beta}^{(3)}$ and as $Z_{\alpha-1}\ne Z_{\beta}$, and again setting $R^\dagger:=[V_{\alpha'-4}, V_{\alpha-3}]$, we have that $Z_{\alpha-1}<R^\dagger Z_{\alpha-1}\le Z_{\alpha-2}$ and $R^\dagger Z_{\alpha-1}$ is centralized by $V_{\alpha'-4}^{(3)}$. Thus, $V_{\alpha'-4}^{(3)}\le Q_{\alpha-2}$. 
Notice that, as $b>7$, if $Z_{\alpha-2}\le V_{\alpha'-4}^{(3)}$ then $Z_{\alpha-1}\le V_{\alpha'-4}^{(3)}\le Q_{\alpha'}$ and we conclude that $Z_{\alpha-1}=Z_{\beta}$, a contradiction. Thus, $Z_{\alpha-2}\not\le V_{\alpha'-4}^{(3)}$.

If $V_{\alpha'-4}\not\le Q_{\alpha-3}$ then $R^\dagger\ne Z_{\alpha-3}$ and $V_{\alpha'-4}^{(3)}=V_{\alpha'-4}(V_{\alpha'-4}^{(3)}\cap Q_{\alpha-3})$. Then $Z_{\alpha-3}=[V_{\alpha-3}, (V_{\alpha'-4}^{(3)}\cap Q_{\alpha-3})]$ for otherwise, $O^p(L_{\alpha'-4})$ centralizes $V_{\alpha'-4}^{(3)}/V_{\alpha'-4}$. But then $Z_{\alpha-2}=R^\dagger\times Z_{\alpha-3}\le V_{\alpha'-4}^{(3)}$, a contradiction. Thus, $V_{\alpha'-4}\le Q_{\alpha-3}$,  $R^\dagger=Z_{\alpha-3}$ and $Z_{\alpha-3}\le [V_{\alpha'-4}^{(3)}, V_{\alpha-3}]\le Z_{\alpha-2}\cap V_{\alpha'-4}^{(3)}=Z_{\alpha-3}$ so that $[V_{\alpha'-4}^{(3)}, V_{\alpha-3}]=Z_{\alpha-3}$ and $V_{\alpha'-4}^{(3)}=V_{\alpha'-4}(V_{\alpha'-4}^{(3)}\cap Q_{\alpha-4})$. But then $O^p(L_{\alpha'-4})$ centralizes $V_{\alpha'-4}^{(3)}/V_{\alpha'-4}$, another contradiction.

Therefore, $Z_{\alpha'-2}=Z_{\alpha'-4}$ and by \cref{SimExt}, $V_{\alpha'-2}=V_{\alpha'-4}$ so that $V_{\alpha}^{(4)}\cap Q_{\alpha'-4}\cap Q_{\alpha'-3}\le Q_{\alpha'-2}$. Since $Z_{\alpha'-1}$ is centralized by $V_{\alpha}^{(4)}$, $V_{\alpha}^{(4)}\cap Q_{\alpha'-4}\cap Q_{\alpha'-3}=V_{\alpha}^{(2)}(V_{\alpha}^{(4)}\cap Q_{\alpha'})$. If $V_{\alpha}^{(4)}/V_{\alpha}^{(2)}$ contains a unique non-central chief factor which is an FF-module for $\bar{L_{\alpha}}$, then by \cref{SimExt}, $V_{\alpha'}\le V_{\alpha'-2}^{(3)}=V_{\alpha'-4}^{(3)}\le Q_{\alpha}$, a contradiction. Thus, $V_{\alpha}^{(4)}\not\le Q_{\alpha'-4}$ and $V_{\alpha}^{(4)}\cap Q_{\alpha'-4}\not\le Q_{\alpha'-3}$.

Since $b>7$, $Z_{\alpha'-4}=Z_{\alpha'-2}\le Z_{\alpha'-1}\le V_{\alpha}^{(2)}\le Z(V_{\alpha}^{(6)})$. If $Z_{\alpha'-4}=Z_{\alpha'-6}$, then by \cref{SimExt}, $V_{\alpha'-4}=V_{\alpha'-6}$ is centralized by $V_{\alpha}^{(4)}$, a contradiction. Thus, $Z_{\alpha'-5}Z_{\alpha'-7}$ is centralized by $V_{\alpha}^{(6)}$. If $Z_{\alpha'-5}\ne Z_{\alpha'-7}$ then $V_{\alpha}^{(6)}\le Q_{\alpha'-5}$ and $V_{\alpha}^{(6)}=V_{\alpha}^{(4)}(V_{\alpha}^{(6)}\cap Q_{\alpha'})$. But then $O^p(L_{\alpha})$ centralizes $V_{\alpha}^{(6)}/V_{\alpha}^{(4)}$, and we have a contradiction. Thus, $Z_{\alpha'-5}=Z_{\alpha'-7}$. But now, as $O^p(R_{\beta})$ centralizes $V_{\beta}^{(3)}$ by \cref{TricAction}, by \cref{SimExt} we have that $V_{\alpha'-4}\le V_{\alpha'-5}^{(2)}=V_{\alpha'-7}^{(2)}$ is centralized by $V_{\alpha}^{(4)}$, a final contradiction. 
\end{proof}

\begin{lemma}\label{bnot7}
Suppose that $C_{V_\beta}(V_{\alpha'})=V_\beta \cap Q_{\alpha'}$. Then $b\leq 7$.
\end{lemma}
\begin{proof}
By \cref{VnotB}, we have that $V_{\alpha'}\not\le Q_{\beta}$. Applying \cref{VB2}, we have that $Z_{\alpha'-1}\le V_{\beta}^{(3)}\le Q_{\alpha'-1}$ and $O^p(R_{\beta})$ centralizes $V_{\beta}^{(3)}$. In particular, if $Z_{\alpha'-1}=Z_{\alpha'-3}$, then $V_{\alpha'}\le V_{\alpha'-1}^{(2)}=V_{\alpha'-3}^{(2)}$ is centralized by $Z_{\alpha}$, a contradiction. Hence, $V_{\alpha'-2}=Z_{\alpha'-1}Z_{\alpha'-3}$. Suppose throughout that $b>7$.

Suppose first that $V_{\beta}^{(5)}\le Q_{\alpha'-4}$. Then, $V_{\beta}^{(5)}\cap Q_{\alpha'-3}$ centralizes $V_{\alpha'-2}$ and so $V_{\beta}^{(5)}\cap Q_{\alpha'-3}=V_{\beta}^{(3)}(V_{\beta}^{(5)}\cap Q_{\alpha'})$. Since $Z_{\alpha'}\le V_{\beta}^{(3)}$, $V_{\beta}^{(5)}\not\le Q_{\alpha'-3}$. Moreover by \cref{GoodAction4}, we have that $O^p(R_{\beta})$ centralizes $V_{\beta}^{(5)}$ and so $V_{\alpha}^{(4)}\not\le Q_{\alpha'-3}$, else $V_{\alpha}^{(4)}\normaleq L_{\beta}=\langle V_{\alpha'}, Q_{\alpha}, R_{\beta}\rangle$. Thus, there is $\alpha-4\in\Delta^{(4)}(\alpha)$ such that $Z_{\alpha'-4}=[Z_{\alpha-4}, Z_{\alpha'-3}]$ and since $Z_{\alpha'-2}\le Z_{\alpha'-1}\le V_{\beta}^{(3)}$, we deduce that $Z_{\alpha'-2}=Z_{\alpha'-4}$. 

Suppose that $Z_{\alpha'-3}\not\le Q_{\alpha-3}$. Then $(\alpha'-3, \alpha-3)$ is a critical pair with $V_{\alpha-3}\le Q_{\alpha'-4}$. By \cref{VnotB}, $V_{\alpha'-3}^{(2)}\not\le Q_{\alpha-1}$ and either $Z_{\alpha}=Z_{\alpha-2}$ or $Z_{\alpha-1}=[V_{\alpha'-4}, V_{\alpha-3}]=Z_{\alpha'-4}$. In the former case it follows from \cref{SimExt} that $Z_{\alpha-4}\le V_{\alpha-2}^{(2)}=V_{\alpha}^{(2)}\le Q_{\alpha'-3}$, a contradiction. In the latter case, we have that $Z_{\beta}=Z_{\alpha-1}=Z_{\alpha'-4}=Z_{\alpha'-2}$. Then $R\cap Z_{\alpha'-2}=\{1\}$, so that $Z_{\alpha'}\le Z_{\alpha'-1}=R\times Z_{\alpha'-2}\le V_{\beta}$ and $V_{\alpha'}$ centralizes $V_{\beta}^{(3)}/V_{\beta}$, a contradiction.

Thus, $Z_{\alpha'-3}\le Q_{\alpha-3}$ and $Z_{\alpha'-4}=Z_{\alpha-3}$. If $Z_{\alpha-3}\le Z_{\alpha}$, then $Z_{\alpha-3}=Z_{\beta}=Z_{\alpha'-4}=Z_{\alpha'-2}$. But then $R\cap Z_{\alpha'-2}=\{1\}$ and $Z_{\alpha'-1}=R\times Z_{\beta}$ so that $Z_{\alpha'-1}\le V_{\beta}$ and $V_{\alpha'}$ centralizes $V_{\beta}^{(3)}/V_{\beta}$, a contradiction. Thus, $V_{\alpha-1}=Z_{\alpha}Z_{\alpha-3}$ is centralized by $V_{\alpha'-3}^{(2)}$ so that $V_{\alpha'-3}^{(2)}\le Q_{\alpha-2}$. Then, $Z_{\alpha-3}\le [V_{\alpha'-3}^{(2)}, V_{\alpha-3}]\le Z_{\alpha-2}$ and since $V_{\alpha-3}$ does not centralize $V_{\alpha'-3}^{(2)}/Z_{\alpha'-3}$, we may assume that $Z_{\alpha-2}\le V_{\alpha'-3}^{(2)}$. Still, $[V_{\alpha'-3}^{(2)}\cap Q_{\alpha-3}, V_{\alpha-3}]\le Z_{\alpha'-3}$ and it follows from \cref{GoodAction1} then $O^p(R_{\alpha})$ centralizes $V_{\alpha}^{(2)}$. Since $Z_{\alpha'-2}=Z_{\alpha'-4}$, \cref{SimExt} implies that $V_{\alpha'-2}=V_{\alpha'-4}$. Moreover, since $V_{\alpha'-4}$ is not centralized by $V_{\beta}^{(5)}$, but $Z_{\alpha'-1}Z_{\alpha'-5}\le V_{\alpha'-4}$ is centralized, it follows that $Z_{\alpha'-1}=Z_{\alpha'-5}$.

Now, if $Z_{\alpha'-4}=Z_{\alpha'-6}$ then \cref{SimExt} implies that $Z_{\alpha'-3}\le V_{\alpha'-4}=V_{\alpha'-6}$ is centralized by $V_{\alpha}^{(4)}$, a contradiction. Thus $Z_{\alpha'-5}=Z_{\alpha'-4}\times Z_{\alpha'-6}$ is centralized by $V_{\alpha-4}^{(2)}$ since $Z_{\alpha'-4}=Z_{\alpha-3}$. Moreover, $Z_{\alpha'-5}\ne Z_{\alpha'-7}$, otherwise \cref{SimExt} implies that $Z_{\alpha'-3}\le V_{\alpha'-5}^{(2)}=V_{\alpha'-7}^{(2)}$ is centralized by $V_{\alpha}^{(4)}$. Hence, $V_{\alpha-4}^{(2)}$ centralizes $V_{\alpha'-6}$ and $V_{\alpha-4}^{(2)}\le Q_{\alpha'-5}$. If $V_{\alpha-4}^{(2)}\le Q_{\alpha'-4}$, then $V_{\alpha-4}^{(2)}=Z_{\alpha-4}(V_{\alpha-4}^{(2)}\cap Q_{\alpha'-3})$ is centralized, modulo $Z_{\alpha-4}$, by $Z_{\alpha'-3}$ so that $O^p(L_{\alpha-4})$ centralizes $V_{\alpha-4}^{(2)}/Z_{\alpha-4}$, a contradiction. Then $V_{\alpha-4}^{(2)}\not\le Q_{\alpha'-4}$ and $[V_{\alpha-4}^{(2)}, V_{\alpha'-4}]\not\le Z_{\alpha'-4}$. Since $Z_{\alpha'-4}=Z_{\alpha-3}\le V_{\alpha-4}^{(2)}$, we assume that $Z_{\alpha'-5}\le V_{\alpha-4}^{(2)}$. 

Now, $V_{\alpha-4}^{(4)}$ centralizes $Z_{\alpha'-6}\le Z_{\alpha'-5}$ and either $Z_{\alpha'-6}=Z_{\alpha'-8}$; or $V_{\alpha-4}^{(4)}$ centralizes $Z_{\alpha'-5}Z_{\alpha'-7}$. In the latter case, we may assume that $Z_{\alpha'-5}\ne Z_{\alpha'-7}$ for the same reason as above, and so either $V_{\alpha-4}^{(4)}\le Q_{\alpha'-5}$ and $O^p(L_{\alpha-4})$ centralizes $V_{\alpha-4}^{(4)}/V_{\alpha-4}^{(2)}$, a contradiction; or $Z_{\alpha'-7}=Z_{\alpha'-9}$, $O^p(R_{\beta})$ centralizes $V_{\beta}^{(5)}$ and $Z_{\alpha'-3}\le V_{\alpha'-7}^{(4)}=V_{\alpha'-9}^{(4)}$ is centralized by $V_{\alpha}^{(4)}$, another contradiction. Thus, $Z_{\alpha'-6}=Z_{\alpha'-8}$ so that $V_{\alpha'-6}=V_{\alpha'-8}$. Suppose that $V_{\alpha-4}^{(4)}\le Q_{\alpha'-8}$. Then $[V_{\alpha-4}^{(4)}\cap Q_{\alpha'-7}, V_{\alpha'-6}]=[V_{\alpha-4}^{(4)}\cap Q_{\alpha'-7}, V_{\alpha'-8}]=Z_{\alpha'-8}=Z_{\alpha'-6}$ and $V_{\alpha-4}^{(4)}\cap Q_{\alpha'-7}\le Q_{\alpha'-6}$. But $V_{\alpha-4}^{(4)}\cap Q_{\alpha'-7}$ centralizes $Z_{\alpha'-5}$ so that $V_{\alpha-4}^{(4)}\cap Q_{\alpha'-7}=V_{\alpha-4}^{(2)}(V_{\alpha-4}^{(4)}\cap Q_{\alpha'-4})$ and by \cref{GoodAction2}, $O^p(R_{\alpha-4})$ centralizes $V_{\alpha-4}^{(4)}$. But now, \cref{SimExt} applied to $Z_{\alpha'-2}=Z_{\alpha'-4}$ implies that $V_{\alpha'}\le V_{\alpha'-2}^{(3)}=V_{\alpha'-4}^{(3)}\le Q_{\alpha}$, a contradiction.

Thus, we have shown that there is a critical pair $(\alpha-8, \alpha'-8)$, $Z_{\alpha'-2}=Z_{\alpha'-4}$, $Z_{\alpha'-6}=Z_{\alpha'-8}$ and $V_{\alpha'-6}=V_{\alpha'-8}$. Since $Z_{\alpha'-5}Z_{\alpha'-9}\le V_{\alpha'-8}$ is centralized by $V_{\alpha-4}^{(4)}$, we get that $Z_{\alpha'-1}=Z_{\alpha'-5}=Z_{\alpha'-9}$. We claim that the pair $(\alpha-8, \alpha'-8)$ satisfies the same initial hypothesis as $(\alpha, \alpha')$. By \cref{VnotB}, $V_{\alpha-8}^{(2)}\not\le Q_{\alpha'-10}$. But $Z_{\alpha'-9}=Z_{\alpha'-5}\le V_{\alpha-4}^{(2)}$ is centralized by $V_{\alpha-8}^{(2)}$ since $b>7$, so that $Z_{\alpha'-9}=Z_{\alpha'-11}$. Then applying \cref{SimExt} gives $V_{\alpha'-8}\le V_{\alpha'-9}^{(2)}=V_{\alpha'-11}^{(2)}$ is centralized by $V_{\alpha-4}^{(4)}$, a contradiction. 

Suppose now that $V_{\beta}^{(5)}\not\le Q_{\alpha'-4}$. Since $Z_{\alpha'-2}\le Z_{\alpha'-1}$ is centralized by $V_{\beta}^{(5)}$, it follows that either $Z_{\alpha'-2}=Z_{\alpha'-4}$; or $Z_{\alpha'-3}=Z_{\alpha'-5}$. In the latter case, we have that $V_{\beta}^{(5)}\cap Q_{\alpha'-4}$ centralizes $V_{\alpha'-2}$ so that $V_{\beta}^{(5)}\cap Q_{\alpha'-4}=V_{\beta}^{(3)}(V_{\beta}^{(5)}\cap \dots \cap Q_{\alpha'})$ and \cref{GoodAction4} implies that $O^p(R_{\beta})$ centralizes $V_{\beta}^{(5)}$. But then \cref{SimExt} applied to $Z_{\alpha'-3}=Z_{\alpha'-5}$ gives $V_{\alpha'}\le V_{\alpha'-3}^{(4)}=V_{\alpha'-5}^{(4)}\le Q_{\alpha}$, a contradiction. Thus, $Z_{\alpha'-2}=Z_{\alpha'-4}$. If $O^p(R_{\alpha})$ centralizes $V_{\alpha}^{(2)}$, then using \cref{SimExt} and $Z_{\alpha'-2}=Z_{\alpha'-4}$, we have that $Z_{\alpha'-1}Z_{\alpha'-5}\le V_{\alpha'-4}$ is centralized by $V_{\beta}^{(5)}$ and we conclude that $Z_{\alpha'-1}=Z_{\alpha'-5}$.

We have demonstrated, regardless of the hypothesis on $V_{\beta}^{(5)}$, that $Z_{\alpha'-2-4k}=Z_{\alpha'-4-4k}$ for $k\geq 0$, and there are suitable critical pairs to iterate upon (either $(\beta-4k-1, \alpha'-4k)$ or $(\alpha-8k, \alpha'-8k)$). Suppose that $b=9$. Applying the above, we infer that $Z_{\alpha'-2}=Z_{\alpha'-4}$ and $Z_{\alpha'-6}=Z_{\alpha+3}=Z_{\beta}$. Since $V_{\alpha'}\not\le Q_{\beta}$, there is a critical pair $(\alpha'+1, \beta)$ with $V_{\beta}\not\le Q_{\alpha'}$. Moreover, $V_{\alpha'+1}^{(2)}\le Q_{\alpha+3}$, else by \cref{Vna-2}, $R=Z_{\alpha+3}=Z_{\beta}$, a clear contradiction. Thus, $(\alpha'+1, \beta)$ satisfies the same hypothesis as $(\alpha, \alpha')$. But then $Z_{\alpha'-6}=Z_{\alpha+3}=Z_{\alpha+5}=Z_{\alpha'-4}$ so that $Z_{\alpha'-2}=\dots=Z_{\beta}$. But then $R\ne Z_{\alpha'-2}$, $Z_{\alpha'-1}=Z_{\alpha'-2}\times R=Z_{\alpha+2}$ and $[V_{\alpha'}, V_{\beta}^{(3)}]=Z_{\alpha'-1}\le V_{\beta}$, a contradiction for then $O^p(L_{\beta})$ centralizes $V_{\beta}^{(3)}/V_{\beta}$. In fact, this applies whenever $b=4k+1$ for $k\geq 2$ but we will only require this when $b=9$. 

Suppose that $V_{\beta}^{(5)}\not\le Q_{\alpha'-4}$ and $b>7$. Then $b\geq 11$ and $V_{\beta}^{(7)}$ centralizes $Z_{\alpha'-4}\le Z_{\alpha'-1}\le V_{\beta}^{(3)}$ and so, unless $Z_{\alpha'-4}=Z_{\alpha'-6}$, $[V_{\beta}^{(7)}, Z_{\alpha'-5}]=\{1\}$. Notice that if $Z_{\alpha'-5}=Z_{\alpha'-7}$, then \cref{SimExt} implies that $V_{\alpha'-4}\le V_{\alpha'-5}^{(2)}=V_{\alpha'-7}^{(2)}$ is centralized by $V_{\beta}^{(5)}$, a contradiction. Thus, $V_{\beta}^{(7)}$ centralizes $V_{\alpha'-6}=Z_{\alpha'-5}Z_{\alpha'-7}$. But then $V_{\beta}^{(7)}=V_{\beta}^{(5)}(V_{\beta}^{(7)}\cap Q_{\alpha'-4})$ and $V_{\beta}^{(5)}\cap Q_{\alpha'-4}\le Q_{\alpha'-3}$, otherwise $V_{\beta}^{(7)}=V_{\beta}^{(5)}(V_{\beta}^{(7)}\cap Q_{\alpha'})$ so that $O^p(L_{\beta})$ centralizes $V_{\beta}^{(7)}/V_{\beta}^{(5)}$, another contradiction. Then, $V_{\beta}^{(5)}\cap Q_{\alpha'-4}=V_{\beta}^{(3)}(V_{\beta}^{(5)}\cap \dots \cap Q_{\alpha'})$ and \cref{GoodAction4} implies that $O^p(R_{\beta})$ centralizes $V_{\beta}^{(5)}$. In particular, $V_{\alpha}^{(4)}\not\le Q_{\alpha'-4}$ for otherwise $V_{\alpha}^{(4)}\normaleq L_{\beta}=\langle V_{\alpha'}, Q_{\alpha}, R_{\beta}\rangle$, a contradiction.

We have shown that, if $b>7$ and $O^p(R_{\alpha})$ centralizes $V_{\alpha}^{(2)}$ then $Z_{\alpha'-1}=Z_{\alpha'-1-4k}$ for all $k\geq 0$. Moreover, we can arrange that $\alpha$ lies along the infinite path $(\alpha',\alpha'-1,\dots, \alpha'-5, \dots)$; or for some critical pair $(\alpha^*, {\alpha^*}')$ we have that $Z_{{\alpha^*}'-2}=Z_{{\alpha^*}'-4}=Z_{{\alpha^*}'-6}$ and $V_{\beta^*}^{(5)}\not\le Q_{{\alpha^*}'-4}$. In this latter case, \cref{SimExt} implies that $V_{{\alpha^*}'-4}=V_{{\alpha^*}'-6}$ and $V_{\beta^*}^{(5)}$ centralizes $V_{{\alpha^*}'-4}$, a clear contradiction. Now, since $Z_{\alpha}\ne Z_{\alpha'-1}$, $Z_{\alpha'-1}=Z_{\alpha+2}=Z_{\alpha-2}$. But then $[V_{\beta}^{(3)}, V_{\alpha'}]=Z_{\alpha'-1}\le V_{\beta}$ and $O^p(L_{\beta})$ centralizes $V_{\beta}^{(3)}/V_{\beta}$, a contradiction. In particular, if we ever arrive at a critical pair $(\alpha^*, {\alpha^*}')$ such that $V_{\alpha^*}^{(4)}\le Q_{{\alpha^*}'-4}$, then $O^p(R_{\alpha})$ centralizes $V_{\alpha}^{(2)}$ and we have a contradiction. Thus, whenever $b>7$, we may assume that for every critical pair $(\alpha^*, {\alpha^*}')$, we have that $V_{{\alpha^*}'}\not\le Q_{\beta^*}$, $V_{\beta^*}^{(3)}\le Q_{{\alpha^*}'-1}$, $V_{\beta^*}^{(5)}\not\le Q_{{\alpha^*}'-4}$ and $Z_{{\alpha^*}'-2}=Z_{{\alpha^*}'-4}$. Also, whenever $Z_{{\alpha^*}'-4}\ne Z_{{\alpha^*}'-6}$, $V_{\alpha^*}^{(4)}\not\le Q_{{\alpha^*}'-4}$ and $V_{\alpha^*-4}^{(2)}\le Q_{{\alpha^*}'-6}$

Suppose that $Z_{\alpha'-4}\ne Z_{\alpha'-6}$ so that there is a critical pair $(\alpha-4, \alpha'-4)$ and $O^p(R_{\beta})$ centralizes $V_{\beta}^{(5)}$. We may also assume that $O^p(R_{\alpha})$ does not centralize $V_{\alpha}^{(2)}$. Since $V_{\alpha'}\not \le Q_{\beta}$, there is ${\alpha'+1}\in\Delta(\alpha')$ such that $({\alpha'+1}, \beta)$ is a critical pair. Suppose that $V_{{\alpha'+1}}^{(2)}$ centralizes $Z_{\beta}$. Since $Z_{\alpha+2}=Z_{\beta}\times R\ne Z_{\alpha+4}$, we have that $V_{{\alpha'+1}}^{(2)}$ centralizes $V_{\alpha+3}$ and $V_{{\alpha'+1}}^{(2)}=Z_{{\alpha'+1}}(V_{{\alpha'+1}}^{(2)}\cap Q_{\beta})$. In particular, $V_{{\alpha'+1}}^{(2)}\cap Q_{\beta}\not\le Q_{\alpha}$, otherwise $V_{{\alpha'+1}}^{(2)}$ is normalized by $L_{\alpha'}=\langle V_{\beta}, Q_{{\alpha'+1}}, R_{\alpha'}\rangle$. But now, $[V_{\alpha}^{(2)}\cap Q_{\alpha'}\cap Q_{{\alpha'+1}}, V_{{\alpha'+1}}^{(2)}\cap Q_{\beta}]\le Z_{{\alpha'+1}}\cap V_{\alpha}^{(2)}$ and since $Z_{\alpha'}\not\le V_{\alpha}^{(2)}$ by \cref{VB2}, we infer that $[V_{\alpha}^{(2)}\cap Q_{\alpha'}\cap Q_{{\alpha'+1}}, V_{{\alpha'+1}}^{(2)}\cap Q_{\beta}]=\{1\}$ and $V_{\alpha}^{(2)}/Z_{\alpha}$ is an FF-module for $\bar{L_{\alpha}}$, a contradiction by \cref{GoodAction1}. Thus, it suffices to prove that $Z_{\beta}$ is centralized by $V_{\alpha'}^{(3)}$. Since $V_{\alpha}^{(4)}\not\le Q_{\alpha'-4}$, $\{1\}\ne [V_{\alpha-3}, V_{\alpha'-4}]\le Z_{\alpha-2}\cap V_{\alpha'-4}$. If $[V_{\alpha-3}, V_{\alpha'-4}]=Z_{\alpha-1}$, then $Z_{\alpha-1}=Z_{\beta}\le V_{\alpha'-4}$, for otherwise $Z_{\alpha}\le Q_{\alpha'}$. Since $b>7$, this leads to a contradiction. Thus, $Z_{\alpha-1}<[V_{\alpha-3}, V_{\alpha'-4}]Z_{\alpha-1}\le Z_{\alpha-2}\ne Z_{\alpha}$ and $V_{\alpha'-4}^{(3)}$ centralizes $V_{\alpha-1}=Z_{\alpha-2}Z_{\alpha}$. Thus, since $V_{\alpha'-4}^{(3)}/V_{\alpha'-4}$ contains a non-central chief factor, $[V_{\alpha-3}, V_{\alpha'-4}]<[V_{\alpha-3}, V_{\alpha'-4}^{(3)}]\le Z_{\alpha-2}$ so that $Z_{\alpha-2}\le V_{\alpha'-4}^{(3)}$. In particular, $Z_{\alpha-1}\le V_{\alpha'-4}^{(3)}$ and since $b>7$, we have that $Z_{\alpha-1}=Z_{\beta}\le V_{\alpha'-4}^{(3)}$. Since $b>9$, $V_{\alpha'}^{(3)}$ centralizes $V_{\alpha'-4}^{(3)}$ so that $V_{{\alpha'+1}}^{(2)}$ centralizes $Z_{\beta}$, as required.

Thus, we have shown that whenever $b>7$ and $V_{\alpha'}\not\le Q_{\beta}$, $Z_{\alpha'-2}=Z_{\alpha'-4}=Z_{\alpha'-6}$ and there is a critical pair $(\beta-5, \alpha'-4)$. Then, as  $[V_{\beta-4}, V_{\alpha'-4}]\ne Z_{\alpha'-6}$ and $Z_{\alpha'-5}\ne Z_{\alpha'-7}$, $V_{\beta-5}^{(2)}\le Q_{\alpha'-6}$ and by \cref{VnotB}, we have that $V_{\alpha'-4}\not\le Q_{\beta-4}$. In particular, $(\beta-5, \alpha'-4)$ satisfies the same hypothesis as $(\alpha, \alpha')$ and applying the same methodology as above, we infer that $Z_{\alpha'-6}=Z_{\alpha'-8}=Z_{\alpha'-10}$. Applying this iteratively, we deduce that $Z_{\alpha'-2}=\dots=Z_{\beta}$. In particular, $Z_{\alpha'-2}=Z_{\beta}\ne R\le Z_{\alpha'-1}\cap Z_{\alpha+2}$ so that $Z_{\alpha'-1}=Z_{\alpha+2}$. But then $[V_{\beta}^{(3)}, V_{\alpha'}]=Z_{\alpha'-1}=Z_{\alpha+2}\le V_{\beta}$ and $O^p(L_{\beta})$ centralizes $V_{\beta}^{(3)}/V_{\beta}$, a final contradiction.
\end{proof}

\begin{lemma}
Suppose that $C_{V_\beta}(V_{\alpha'})=V_\beta \cap Q_{\alpha'}$. Then $b\ne 7$. 
\end{lemma}
\begin{proof}
By \cref{VnotB} and \cref{bnot7}, we have that $V_{\alpha'}\not\le Q_{\beta}$ and $b=7$. Since $V_{\alpha'-2}=Z_{\alpha'-1}Z_{\alpha'-3}\le V_{\beta}^{(3)}$ and $V_{\beta}^{(3)}$ is abelian, we have that $C_{Q_{\beta}}(V_{\beta}^{(3)})=V_{\beta}^{(3)}(C_{Q_{\beta}}(V_{\beta}^{(3)})\cap Q_{\alpha'})$ and since $Z_{\alpha'}\le V_{\beta}^{(3)}$, $O^p(L_{\beta})$ centralizes $C_{Q_{\beta}}(V_{\beta}^{(3)})/V_{\beta}^{(3)}$. In particular, $O^p(R_{\beta})$ centralizes $C_{Q_{\beta}}(V_{\beta}^{(3)})$. But now, by the three subgroup lemma, for $r\in O^p(R_{\beta})$ of order coprime to $p$, $[r, Q_{\beta}, V_{\beta}^{(3)}]=\{1\}$ and $r$ centralizes $Q_{\beta}$. Thus, $R_{\beta}=Q_{\beta}$ and $\bar{L_{\beta}}\cong\SL_2(q)$.

Let $\alpha'+1\in\Delta(\alpha')$ such that $Z_{\alpha'+1}\not\le Q_{\beta}$. Then, $V_{\alpha}^{(2)}\cap Q_{\alpha'}\not\le Q_{\alpha'+1}$, for otherwise $V_{\alpha'}$ normalizes $V_{\alpha}^{(2)}$, a contradiction for then $L_{\beta}=\langle V_{\alpha'}, Q_{\alpha}, Q_{\beta}\rangle$ normalizes $V_{\alpha}^{(2)}$. Notice that $[V_{\alpha'+1}^{(2)}, V_{\alpha+3}]\le Z_{\alpha+4}\cap Z_{\alpha'+1}$. Since $(\alpha'+1, \beta)$ is a critical pair, we have that $Z_{\alpha+4}\cap Z_{\alpha'+1}=Z_{\alpha'-3}\cap Z_{\alpha'+1}\le Z_{\alpha'}$. But if $Z_{\alpha'}\le Z_{\alpha'-3}$, since $Z_{\alpha'-1}\ne Z_{\alpha'-3}$, we deduce that $Z_{\alpha'}=Z_{\alpha'-2}\ne R$. Then $R\ne Z_{\alpha+3}$ for otherwise $Z_{\alpha'-1}=Z_{\alpha'-2}R=Z_{\alpha'-2}Z_{\alpha'-4}=Z_{\alpha'-3}$. Thus, $V_{\alpha'+1}^{(2)}$ centralizes $Z_{\alpha+3}R$, $Z_{\alpha+3}<Z_{\alpha+3}R$ and since $Z_{\alpha+2}\ne Z_{\alpha+4}$, we deduce that $[V_{\alpha'+1}^{(2)}, V_{\alpha+3}]=\{1\}$. Thus, whether $Z_{\alpha'}\le Z_{\alpha+4}$ or not, $V_{\alpha'+1}^{(2)}\le Q_{\alpha+2}$ and $V_{\alpha'+1}^{(2)}=Z_{\alpha'+1}(V_{\alpha'+1}^{(2)}\cap Q_{\beta})$ and since $V_{\alpha'+1}^{(2)}\not\normaleq L_{\alpha'}=\langle V_{\beta}, Q_{\alpha'+1}, Q_{\alpha'}\rangle$, we may assume that $V_{\alpha'+1}^{(2)}\cap Q_{\beta}\not\le Q_{\alpha}$ and $Z_{\beta}\not\le V_{\alpha'+1}^{(2)}$. But now, $[V_{\alpha'+1}^{(2)}\cap Q_{\beta}, V_{\alpha}^{(2)}\cap Q_{\alpha'}\cap Q_{\alpha'+1}]\le Z_{\alpha'+1}\cap V_{\alpha}^{(2)}$. If $Z_{\alpha'}\cap V_{\alpha}^{(2)}\ne\{1\}$, then since $Z_{\alpha'}V_{\alpha}^{(2)}\normaleq L_{\beta}=\langle V_{\alpha'}, Q_{\alpha}\rangle$, we deduce that $V_{\alpha}^{(2)}$ has index strictly less than $q$ in $V_{\beta}^{(3)}$ and is centralized, modulo $V_{\beta}$, by $Q_{\alpha}$, a contradiction by \cref{SEFF} since $V_{\beta}^{(3)}/V_{\beta}$ contains a non-central chief factor for $L_{\beta}$. Hence, $[V_{\alpha'+1}^{(2)}\cap Q_{\beta}, V_{\alpha}^{(2)}\cap Q_{\alpha'}\cap Q_{\alpha'+1}]=\{1\}$ and $V_{\alpha}^{(2)}/Z_{\alpha}$ is an FF-module for $\bar{L_{\alpha}}$. Then by \cref{GoodAction1}, $O^p(R_{\alpha})$ centralizes $V_{\alpha}^{(2)}$. 

It follows from the arguments above, that if $Z_{\alpha+3}=R\ne Z_{\alpha'-2}$, then $Z_{\alpha'-1}=Z_{\alpha'-3}$ and we have a contradiction. Similarly, $Z_{\alpha'-2}=R\ne Z_{\alpha+3}$ yields $Z_{\alpha+2}=Z_{\alpha+4}$, another contradiction. Suppose that $Z_{\alpha+3}\ne R\ne Z_{\alpha'-2}$. In particular, $R\not\le Z_{\alpha'-3}$. But now, $Z_{\alpha'-3}<RZ_{\alpha'-3}\le V_{\alpha'-2}\cap V_{\alpha'-4}$ and $V_{\alpha'-2}=V_{\alpha'-4}$. If $Z_{\alpha'-2}\ne Z_{\alpha'-4}$ then $L_{\alpha'-3}=\langle R_{\alpha'-3}, Q_{\alpha'-2}, Q_{\alpha'-4}\rangle$ normalizes $V_{\alpha'-2}$, a contradiction. Thus, $Z_{\alpha'-2}=Z_{\alpha'-4}=Z_{\alpha+3}$ so that $Z_{\alpha'-1}=RZ_{\alpha'-2}=RZ_{\alpha+3}=Z_{\alpha+2}\le V_{\beta}$ from which it follows that $V_{\alpha'}$ centralizes $V_{\beta}^{(3)}/V_{\beta}$, a contradiction. Thus, $R=Z_{\alpha'-2}=Z_{\alpha'-4}=Z_{\alpha+3}$ and by \cref{SimExt}, we conclude that $V_{\alpha'-2}=V_{\alpha'-4}$.

We may assume that $V_{\alpha}^{(4)}$ does not centralize $Z_{\alpha'-3}$, for otherwise $V_{\alpha}^{(4)}$ centralizes $V_{\alpha'-2}=V_{\alpha'-4}=Z_{\alpha'-3}Z_{\alpha+2}$, $V_{\alpha}^{(4)}=V_{\beta}^{(3)}(V_{\alpha}^{(4)}\cap Q_{\alpha'})$ and $V_{\alpha}^{(4)}\normaleq L_{\beta}=\langle V_{\alpha'}, Q_{\alpha}\rangle$. Choose $\alpha-4\in\Delta^{(4)}(\alpha)$ such that $[Z_{\alpha-4}, Z_{\alpha'-3}]\ne\{1\}$. If $Z_{\alpha'-3}\le Q_{\alpha-3}$, then $Z_{\alpha-3}=[Z_{\alpha-4}, Z_{\alpha'-3}]=Z_{\alpha+3}$ so that $Z_{\beta}\cap Z_{\alpha-3}=\{1\}$. Now, $Z_{\alpha-3}\le V_{\alpha-1}\cap V_{\beta}$ and
since $Z_{\alpha-3}\cap Z_{\alpha}\le Q_{\alpha'}\cap Z_{\alpha}=Z_{\beta}$, we deduce that $V_{\alpha-1}=Z_{\alpha}Z_{\alpha-3}=V_{\beta}$. Since $O^p(R_{\alpha})$ centralizes $V_{\alpha}^{(2)}$, $Z_{\alpha-1}=Z_{\beta}$, for otherwise $V_{\beta}\normaleq L_{\alpha}=\langle R_{\alpha}, Q_{\alpha-1}, Q_{\beta}\rangle$. 

If $Z_{\alpha'-3}\not\le Q_{\alpha-3}$ then $(\alpha'-3, \alpha-3)$ is a critical pair. By \cref{VnotB}, we may assume that either $(\alpha'-3, \alpha-3)$ satisfies the same hypothesis as $(\alpha, \alpha')$, in which case $Z_{\alpha-1}=Z_{\beta}$; or $V_{\alpha'-3}^{(2)}\not\le Q_{\alpha-1}$ and by \cref{Vna-2}, either $[V_{\alpha'-4}, V_{\alpha-3}]=Z_{\alpha-1}\le Z_{\alpha+2}$, and again $Z_{\alpha-1}=Z_{\beta}$, or $Z_{\alpha-2}=Z_{\alpha}$, and by \cref{SimExt}, we have a contradiction. 

Thus, we have shown that whenever there is $Z_{\alpha-4}$ such that $Z_{\alpha-4}$ does not centralizes $Z_{\alpha'-3}$, we have $Z_{\alpha-1}=Z_{\beta}$. Choose $\lambda\in\Delta(\alpha)$ such that $Z_{\lambda}\ne Z_{\beta}$ so that $V_{\lambda}^{(3)}$ centralizes $Z_{\alpha'-3}$. Then $V_{\lambda}^{(3)}$ centralizes $V_{\alpha'-4}=V_{\alpha'-2}$ so that $V_{\lambda}^{(3)}=V_{\beta}(V_{\lambda}^{(3)}\cap Q_{\alpha'})$. Then, $V_{\lambda}^{(3)}V_{\beta}^{(3)}\normaleq L_{\beta}=\langle Q_{\alpha}, V_{\alpha'}\rangle$. In particular, $[C_{\beta}, V_{\lambda}^{(3)}V_{\beta}^{(3)}]$ is a normal subgroup $L_{\beta}$ contained in $[C_{\beta}, V_{\beta}^{(3)}][Q_{\alpha}, V_{\lambda}^{(3)}]$.  Noticing that $[V_{\alpha'+1}^{(2)}\cap Q_{\beta}, V_{\alpha}^{(2)}]=[V_{\alpha'+1}\cap Q_{\beta}, V_{\beta}(V_{\alpha}^{(2)}\cap Q_{\alpha'})]=Z_{\beta}R=Z_{\alpha+2}$, we have that $[S, V_{\alpha}^{(2)}]\le V_{\beta}$ and $|V_{\alpha}^{(2)}|=q^4$. But then $[Q_{\beta}, V_{\beta}^{(3)}]=V_{\beta}$ and since $[V_{\alpha'}, V_{\beta}^{(3)}]=Z_{\alpha'-1}\le V_{\alpha+3}\le V_{\alpha+2}^{(2)}$, we must have that $|V_{\beta}^{(3)}|=q^5$ and $[Q_{\alpha}, V_{\beta}^{(3)}]=V_{\alpha}^{(2)}$. Thus, $V_{\beta}\not\ge[C_{\beta}, V_{\lambda}^{(3)}V_{\beta}^{(3)}]\le V_{\alpha}^{(2)}$ and it follows that $V_{\alpha}^{(2)}=V_{\beta}[C_{\beta}, V_{\lambda}^{(3)}V_{\beta}^{(3)}]\normaleq L_{\beta}$, a contradiction.
\end{proof}

Combining all the results in this subsection thus far, we have the following result.

\begin{proposition}
Suppose that $C_{V_\beta}(V_{\alpha'})=V_\beta \cap Q_{\alpha'}$. Then $b\leq 5$.
\end{proposition}

In conjunction with the results in earlier sections, we have now proved that \cref{MainHyp} implies that $b\leq 5$. In the next lemmas and proposition, we show this bound is tight by witnessing an example with $b=5$. In \cite{Greenbook} and \cite{F3}, this configuration is shown to be parabolic isomorphic to $\mathrm{F}_3$. In our case, it is demonstrated in \cite{ExoSpo} that this leads to an exotic fusion system. The presence of this fusion system may go some way to explaining why it is so difficult to uniquely determine $\mathrm{F}_3$ from a purely $3$-local perspective.

\begin{lemma}\label{b=5i}
Suppose that $C_{V_\beta}(V_{\alpha'})=V_\beta \cap Q_{\alpha'}$ and $b=5$. Then $V_{\alpha'}\not \le Q_{\beta}$.
\end{lemma}
\begin{proof}
Assume that $V_{\alpha'}\le Q_{\beta}$. If $V_{\alpha}^{(2)}\le Q_{\alpha'-2}$, then it follows from \cref{VA1} that $|V_{\beta}|=q^3$ and $O^p(R_{\alpha})$ centralizes $V_{\alpha}^{(2)}$. Then, $Z_{\beta}=R\le Z_{\alpha'-1}$ and since $V_{\beta}\ne V_{\alpha'-2}$, by \cref{SimExt}, we may assume that $Z_{\beta}\ne Z_{\alpha'-2}$ so that $Z_{\alpha'-1}=Z_{\alpha+2}$. But now, $V_{\alpha'}^{(3)}\cap Q_{\alpha'-2}\le Q_{\alpha+2}$ so that $[V_{\alpha'}^{(3)}\cap Q_{\alpha'-2}, V_{\beta}]\le Z_{\alpha+2}=Z_{\alpha'-1}\le V_{\alpha'}$. By \cref{GoodAction3}, $O^p(R_{\alpha'})$ centralizes $V_{\alpha'}^{(3)}$ and \cref{SimExt} applied to $Z_{\alpha'-1}=Z_{\alpha+2}$ implies that $V_{\beta}\le V_{\alpha+2}^{(2)}=V_{\alpha'-1}^{(2)}\le Q_{\alpha'}$, a contradiction.

Suppose now that $V_{\alpha'}\le Q_{\beta}$, $|V_{\beta}|=q^3$ and $V_{\alpha}^{(2)}\not\le Q_{\alpha'-2}$. If $Z_{\beta}=R\ne Z_{\alpha'-2}$ then, as above, $Z_{\alpha'-1}=Z_{\alpha+2}$ and $[V_{\alpha'}^{(3)}\cap Q_{\alpha'-2}, V_{\beta}]\le Z_{\alpha+2}=Z_{\alpha'-1}\le V_{\alpha'}$. Then $O^p(R_{\alpha'})$ centralizes $V_{\alpha'}^{(3)}$ and \cref{SimExt} provides a contradiction. Thus, $Z_{\beta}=Z_{\alpha'-2}\ne Z_{\alpha'}$. But now, $[V_{\alpha'-2}, V_{\alpha}^{(2)}]\le Z_{\alpha}\cap Z_{\alpha+2}=Z_{\beta}=Z_{\alpha'-2}$ and $V_{\alpha}^{(2)}\le Q_{\alpha'-2}$, a contradiction.

Thus, if $V_{\alpha'}\le Q_{\beta}$ then $|V_{\beta}|\ne q^3$. Notice that if $Z_{\alpha'-2}=Z_{\beta}$, then $Z_{\beta}Z_{\beta}^gZ_{\alpha'}=Z_{\alpha'-1}Z_{\alpha'-1}^g$ is of order $q^3$ and normalized by $L_{\alpha'}=\langle V_{\beta}, V_{\beta}^g, R_{\alpha'}\rangle$, for some appropriately chosen $g\in L_{\alpha'}$, a contradiction. Now, if $Z_{\alpha'-2}\le V^\alpha$, then $V_{\beta}=Z_{\alpha+2}Z_{\alpha}C_{V_{\beta}}(O^p(L_{\beta}))\le V^\alpha$. But then $V^\alpha=V_{\alpha}^{(2)}$ and we have a contradiction. Since $[Q_{\alpha}, V_{\alpha}^{(2)}]\le V^\alpha$ and $V_{\alpha'-2}\le Q_{\alpha}$, it follows that $V_{\alpha}^{(2)}\cap Q_{\alpha'-2}$ centralizes $V_{\alpha'-2}$ and $V_{\alpha}^{(2)}\cap Q_{\alpha'-2}\le Q_{\alpha'-1}$. Since both $V_{\alpha}^{(2)}/V^\alpha$ and $V^\alpha/Z_{\alpha}$ have non-central chief factors, $[V_{\alpha}^{(2)}\cap Q_{\alpha'-2}\cap Q_{\alpha'}, V_{\alpha'}]=Z_{\alpha'}\le V_{\alpha}^{(2)}$ and both $V_{\alpha}^{(2)}/V^\alpha$ and $V^\alpha/Z_{\alpha}$ are FF-modules for $\bar{L_{\alpha}}$. Then by \cref{GoodAction1}, we have that $O^p(R_{\alpha})$ centralizes $V_{\alpha}^{(2)}$ and by \cref{SimExt}, $Z_{\alpha'}\ne Z_{\alpha'-2}$ and $Z_{\alpha'-1}\le V_{\alpha}^{(2)}$. Since $V_{\alpha}^{(2)}\not\le Q_{\alpha'-2}$, and $V_{\alpha}^{(2)}$ centralizes $Z_{\alpha'-1}Z_{\alpha+2}$. By \cref{VBGood}, we deduce that $Z_{\alpha'-1}=Z_{\alpha+2}$. But now $[V_{\beta}, V_{\alpha'}]=Z_{\beta}\le Z_{\alpha'-1}$ and $Z_{\alpha'-1}Z_{\alpha'-1}^g$ is of order $q^3$ and normalized by $L_{\alpha'}=\langle V_{\beta}, V_{\beta}^g, R_{\alpha'}\rangle$, a final contradiction. 
\end{proof}

\begin{lemma}\label{b=5ii}
Suppose that $C_{V_\beta}(V_{\alpha'})=V_\beta \cap Q_{\alpha'}$ and $b=5$. Then $|V_{\beta}|=q^3$, $R=Z_{\alpha'-2}\ne Z_{\beta}\ne Z_{\alpha'}\ne R$ and $Z_{\alpha'-1}\ne Z_{\alpha+2}$.
\end{lemma}
\begin{proof}
By \cref{b=5i}, we have that $V_{\alpha'}\not\le Q_{\beta}$ for all critical pairs $(\alpha,\alpha')$, so that $Z_{\alpha'}\ne R\ne Z_{\beta}$. Fix ${\alpha'+1}\in\Delta(\alpha')$ such that $Z_{{\alpha'+1}}\not\le Q_{\beta}$. In particular, $({\alpha'+1}, \beta)$ is a critical pair satisfying the same hypothesis as $(\alpha, \alpha')$. Assume first that $|V_{\beta}|=q^3$. Then $R\le Z_{\alpha'-1}\cap Z_{\alpha+2}$. If $Z_{\alpha'-1}=Z_{\alpha+2}$ then $V_{\beta}^{(3)}\cap Q_{\alpha'-2}\le Q_{\alpha'-1}$,  $[V_{\beta}^{(3)}\cap Q_{\alpha'-2}, V_{\alpha'}]\le Z_{\alpha'-1}\le V_{\beta}$ and it follows that $V_{\beta}^{(3)}/V_{\beta}$ contains a unique non-central chief factor for $L_{\beta}$ which, as a $\bar{L_{\beta}}$-module, is an FF-module. Then, \cref{GoodAction3} and \cref{SimExt} applied to $Z_{\alpha'-1}=Z_{\alpha+2}$ gives $V_{\alpha'}\le V_{\alpha'-1}^{(2)}=V_{\alpha+2}^{(2)}\le Q_{\beta}$, a contradiction. Now, $R\le Z_{\alpha'-1}\cap Z_{\alpha+2}$ and so $R= Z_{\alpha'-2}$, otherwise $Z_{\alpha'-1}=Z_{\alpha+2}$. Hence, we may assume that $|V_{\beta}|\ne q^3$ for the remainder of the proof.

Suppose first that $Z_{\alpha'}=Z_{\alpha'-2}$. Then $[V^\alpha, V_{\alpha'-2}]\le Z_{\alpha}\cap V_{\alpha'-2}\le Z_{\beta}\le Z_{\alpha+2}$. In particular, if $V^{\alpha}\not\le Q_{\alpha'-2}$ then $Z_{\beta}\ne Z_{\alpha'-2}$ and $Z_{\beta}Z_{\beta}^gZ_{\alpha'-2}=Z_{\alpha+2}Z_{\alpha+2}^g$ is of order $q^3$ and normalized by $L_{\alpha'-2}=\langle V^\alpha, (V^\alpha)^g, R_{\alpha'-2}\rangle$ for some appropriately chosen $g\in L_{\alpha'-2}$, a contradiction. Thus, $V^\alpha\le Q_{\alpha'-2}$. Then, either $V^\alpha=Z_{\alpha}(V^\alpha \cap Q_{\alpha'})$, or $V^{\alpha}\not\le Q_{\alpha'-1}$ and $Z_{\beta}=[V^\alpha, Z_{\alpha'-1}=Z_{\alpha'-2}$. Assume that $V^\lambda\le Q_{\alpha'-1}$ for all $\lambda\in\Delta(\beta)$ with $Z_{\lambda}=Z_{\alpha}$. Forming $U^\beta:=\langle V^\lambda\mid \lambda\in\Delta(\beta), Z_{\lambda}=Z_{\alpha}\rangle$ so that $R_{\beta}Q_{\alpha}$ normalizes $U^\beta$ by \cref{UWNormal}, we have that $U^\beta=Z_{\alpha}(U^\beta\cap Q_{\alpha'})$ and as $Z_{\alpha'}=Z_{\alpha'-2}\le V_{\beta}$, $[V_{\alpha'}, U^\beta V_{\beta}]\le V_{\beta}$ so that $U^\beta V_{\beta}\normaleq L_{\beta}=\langle V_{\alpha'}, Q_{\alpha}, R_{\beta}\rangle$ and $O^p(L_{\beta})$ centralizes $U^\beta V_{\beta}/V_{\beta}$. But then, $V^\alpha V_{\beta}\normaleq L_{\beta}$, a contradiction by \cref{VBGood}. Thus, we deduce that $Z_{\alpha'}=Z_{\alpha'-2}=Z_{\beta}$ and $V^\lambda\not\le Q_{\alpha'-1}$ for some $\lambda\in\Delta(\beta)$ with $Z_{\lambda}=Z_{\alpha}$. Then, using that $[V^{\alpha'-1}, V_{\beta}]\le Z_{\alpha'-1}$ and $V_{\beta}$ is centralized by $V^\lambda$, it follows that $[V^{\alpha'-1}, V_{\beta}]\le Z_{\alpha'-2}=Z_{\beta}$ so that $V^{\alpha'-1}\le Q_{\beta}$. Then $[V^{\alpha'-1}\cap Q_{\lambda}, V^\lambda]\le Z_\lambda=Z_{\alpha}$ and since $Z_{\alpha}\not\le V_{\alpha'-1}^{(2)}$, we have that $[V^{\alpha'-1}\cap Q_{\lambda}, V^\lambda]\le Z_{\beta}\le Z_{\alpha'-1}$. Since $V^{\alpha'-1}/Z_{\alpha'-1}$ contains a non-central chief factor, $S=V^{\alpha'-1}Q_{\alpha}$ and $V^{\alpha'-1}/Z_{\alpha'-1}$ is an FF-module for $\bar{L_{\alpha'-1}}$. Then $V_{\alpha'-1}^{(2)}\cap Q_{\beta}=V^{\alpha'-1}(V_{\alpha'-1}^{(2)}\cap Q_{\alpha})$ and since $Z_{\alpha}\not\le V_{\alpha'-1}^{(2)}$, $V_{\alpha'-1}^{(2)}/V^{\alpha'-1}$ is also an FF-module for $\bar{L_{\alpha'-1}}$. Then \cref{GoodAction1} and \cref{SimExt} applied to $Z_{\beta}=Z_{\alpha'-2}=Z_{\alpha'}$ gives $V_{\alpha'}=V_{\beta}$, a contradiction. Hence, we have shown that $Z_{\alpha'}\ne Z_{\alpha'-2}$.

Since $(\alpha'+1, \beta)$ is also critical pair, we deduce that $Z_{\beta}\ne Z_{\alpha'-2}\ne Z_{\alpha'}$ arguing as above. As in \cref{b=5i}, this implies that $Z_{\alpha'-2}\not\le V^\alpha$ so that $V_{\alpha}^{(2)}\cap Q_{\alpha'-2}\le Q_{\alpha'-1}$. Moreover, $[V^\alpha, V_{\alpha'-2}]\le Z_{\alpha}\cap V_{\alpha'-2}=Z_{\beta}$ and if $V^\alpha\not\le Q_{\alpha'-2}$, then $Z_{\beta}Z_{\beta}^gZ_{\alpha'-2}=Z_{\alpha+2}Z_{\alpha+2}^g$ is of order $q^3$ and normalized by $L_{\alpha'-2}=\langle V^\alpha, (V^\alpha)^g, R_{\alpha'-2}\rangle$ for some appropriately chosen $g\in L_{\alpha'-2}$, a contradiction. Thus, $V^\alpha=Z_{\alpha}(V^\alpha\cap Q_{\alpha'})$. Set $U_{\beta}:=\langle (V^\alpha)^{G_{\beta}}\rangle$. Then $[U_{\beta}, V_{\alpha'-2}]\le [U_\beta, C_{\beta}]\cap V_{\alpha'-2}\le V_{\alpha'-2}\cap V_{\beta}$. Assume that $U_\beta\not\le Q_{\alpha'-2}$ and so there is some $\beta-3\in\Delta^{(3)}(\beta)$ with $(\beta-3, \alpha'-2)$ a critical pair. Indeed, we must have that $V_{\beta}\cap V_{\alpha'-2}\le Z_{\alpha+2}C_{V_{\beta}}(O^p(L_{\beta}))\cap Z_{\alpha+2}C_{V_{\alpha'-2}}(O^p(L_{\alpha'-2}))$. But then $Z_{\beta}\ge [Q_{\alpha+2}, V_{\alpha'-2}\cap V_{\beta}]\le Z_{\alpha'-2}$ and we deduce that $V_{\alpha'-2}\cap V_{\beta}\le \Omega(Z(Q_{\alpha+2}))\cap V_{\beta}=Z_{\alpha+2}$. But then, $Z_{\alpha+2}Z_{\alpha+2}^g$ is of order $q^3$ and normalized by $L_{\alpha'-2}=\langle U_\beta, U_\beta^g, R_{\alpha'-2}\rangle$, for some appropriate $g\in L_{\alpha'-2}$, another contradiction. Thus, $U_\beta\le Q_{\alpha'-2}$ and $[U_\beta, V_{\alpha'-2}]\le Z_{\alpha'-2}$. 

Suppose that $U_\beta\not\le Q_{\alpha'-1}$ so that $V^{\mu}\not\le Q_{\alpha'-1}$ for some $\mu\in\Delta(\beta)$. By a similar argument to the above, we also have that $[V^{\alpha'-1}, V_{\beta}]=\{1\}$ and $V^{\alpha'-1}\le Q_{\mu}$.  Then $\{1\}\ne [V^\mu, V^{\alpha'-1}]\le Z_{\mu}\cap V^{\alpha'-1}$. Notice that $Z_{\beta}\not\le V^{\alpha'-1}$ for otherwise $V_{\alpha'-2}=Z_{\alpha'-1}Z_{\alpha+2}C_{V_{\alpha'}}(O^p(L_{\alpha'-2}))\le V^{\alpha'-1}$. Thus, $Z_{\mu}=[V^{\alpha'-1}, V^{\mu}]\times Z_{\beta}$ centralizes $V_{\alpha'}$ and since $R\ne \{1\}$, it follows that $Z_{\mu}=Z_{\alpha+2}$. Since $Z_{\alpha'-2}\le V^{\alpha'-1}$ and $Z_{\beta}\not \le V^{\alpha'-1}$, we have that $[V^\mu, V^{\alpha'-1}]=Z_{\alpha'-2}\le Z_{\alpha'-1}$, a contradiction since $V^{\mu}\not\le Q_{\alpha'-1}$. Thus, $U_\beta\le Q_{\alpha'-1}$. Since $V^{\alpha}V_{\beta}\not\normaleq L_{\beta}$, $U_\beta/V_{\beta}$ contains a non-central chief factor for $L_{\beta}$ and we conclude that $[U_{\beta}\cap Q_{\alpha'}, V_{\alpha'}]=Z_{\alpha'}\le U_{\beta}$ and $Z_{\alpha'}\cap V_{\beta}=\{1\}$. But now, $Z_{\alpha'-2}\ne Z_{\alpha'}$ and $V_{\alpha'-2}=Z_{\alpha'-1}Z_{\alpha+2}C_{V_{\alpha'-2}}(O^p(L_{\alpha'-2}))\le U_\beta$. 

Suppose that $[V_{\alpha}^{(2)}, Z_{\alpha'-1}]\ne \{1\}$. Then there is $\alpha-1\in\Delta(\alpha)$ such that $[V_{\alpha-1}, Z_{\alpha'-1}]\ne\{1\}$. If $Z_{\alpha'-1}\le Q_{\alpha-1}$, then $Z_{\alpha-1}=[Z_{\alpha'-1}, V_{\alpha-1}]\le [V_{\alpha'-2}, V_{\alpha}^{(2)}]$. Since $Z_{\alpha}\not\le V_{\alpha'-2}$, it follows that $Z_{\alpha-1}=Z_{\beta}$. Then $V_{\alpha-1}\not\le Q_{\alpha'-2}$ since $Z_{\alpha'-2}\ne Z_{\beta}$. But then $[V_{\alpha'-2}, V_{\alpha}^{(2)}]\le Z_{\beta}Z_{\alpha'-2}$ and if $V_{\alpha}^{(2)}\not\le Q_{\alpha'-2}$, then $Z_{\alpha+2}Z_{\alpha+2}^g$ is of order $q^3$ and normalized by $L_{\alpha'-2}=\langle V_{\alpha}^{(2)}, (V_{\alpha}^{(2)})^g, R_{\alpha'-2}\rangle$ for some appropriately chosen $g\in L_{\alpha'-2}$, a contradiction. Thus, $Z_{\alpha'-1}\not\le Q_{\alpha-1}$ and $(\alpha'-1, \alpha-1)$ is a critical pair. 

Since $Z_{\alpha'-2}\ne Z_{\beta}$, $(\alpha'-1, \alpha-1)$ satisfies the same hypothesis as $(\alpha, \alpha')$ and so we see that $V_{\beta}\le U_{\alpha'-2}$. But then $R=[V_{\beta}, V_{\alpha'}]\le [U_{\alpha'-2}, C_{\alpha'-2}]\le V_{\alpha'-2}$ and $R\le V_{\beta}\cap V_{\alpha'-2}\le Z_{\alpha+2}C_{V_{\beta}}(O^p(L_{\beta}))\cap Z_{\alpha+2}C_{V_{\alpha'-2}}(O^p(L_{\alpha'-2}))$. Similarly to before, this implies that $|V_{\beta}|=q^3$, and we have a contradiction. 

Thus, $[V_{\alpha}^{(2)}, Z_{\alpha'-1}]=\{1\}$ and since $Z_{\alpha'-1}\ne Z_{\alpha'-3}$, it follows that $V_{\alpha}^{(2)}$ centralizes $V_{\alpha'-2}$ and $V_{\alpha}^{(2)}\le Q_{\alpha'-1}$. In particular, this holds for any $\lambda\in\Delta(\beta)$ with $Z_{\lambda}=Z_{\alpha}$. Forming $W^\beta:=\langle V_{\lambda}^{(2)}\mid Z_{\lambda}=Z_{\alpha}, \lambda\in\Delta(\beta)\rangle$, we have that $W^\beta U_{\beta}/U_{\beta}$ is centralized by $V_{\alpha'}$, and by \cref{UWNormal}, normalized by $R_{\beta}Q_{\alpha}$. But then $W^\beta U_\beta\normaleq L_{\beta}=\langle V_{\alpha'}, R_{\beta}, Q_{\alpha}\rangle$ and since $V_{\alpha'}$ centralizes $W^{\beta}U_{\beta}/U_{\beta}$ we deduce that $V_{\beta}^{(3)}=V_\alpha^{(2)} U_{\beta}\normaleq L_{\beta}$. Now, $R=[V_{\beta}, V_{\alpha'}]\le [V_{\beta}, V_{\alpha'-1}^{(2)}U_{\alpha'-2}]=[V_{\beta}, V_{\alpha+2}^{(2)}U_{\alpha'-2}]=[V_{\beta}, U_{\alpha'-2}]$ and since $V_{\beta}\le C_{\alpha'-2}$, it follows that $R\le V_{\beta}\cap V_{\alpha'-2}$, which again implies that $|V_{\beta}|=q^3$, a final contradiction.
\end{proof}

For the remainder of the analysis when $b=5$, we set $W^\beta:=\langle V_{\lambda}^{(2)}\mid\lambda\in\Delta(\beta)\,,Z_{\lambda}=Z_{\alpha}\rangle\normaleq R_{\beta}Q_{\alpha}$.
 
\begin{lemma}\label{NOCF2}
Suppose that $C_{V_\beta}(V_{\alpha'})=V_\beta \cap Q_{\alpha'}$ and $b=5$. Then $O^p(L_{\beta})$ centralizes $[V_{\beta}^{(3)}, Q_{\beta}]V_{\beta}/V_{\beta}$.
\end{lemma}
\begin{proof}
Suppose that $[V_{\beta}^{(3)}, Q_{\beta}]V_{\beta}/V_{\beta}$ contains a non-central chief factor for $L_{\beta}$. In addition, assume that $Z_{\alpha'}\not\le [V_{\beta}^{(3)}, Q_{\beta}]V_{\beta}$. Notice that $[W^\beta, Q_{\beta}]= [W^\beta, (Q_{\alpha}\cap Q_{\beta})][W^\beta, (Q_{\beta}\cap Q_{\alpha+2})]\le Z_{\alpha}[Q_{\alpha+2}, Q_{\alpha+2}]\le Q_{\alpha'-2}$. Now, $[W^\beta\cap Q_{\alpha'-2}, Z_{\alpha'-1}]\le Z_{\alpha'-2}\cap [W^\beta, Z_{\alpha'-1}]\le Z_{\alpha'-2}\cap Z_{\alpha}=\{1\}$ and so $[W^\beta, Q_{\beta}, V_{\alpha'}]\le Z_{\alpha'-1}\cap [V_{\beta}^{(3)}, Q_{\beta}]\le Z_{\alpha'-2}\le V_{\beta}$. In particular, it follows that $[W^\beta, Q_{\beta}]V_{\beta}\normaleq L_{\beta}=\langle V_{\alpha'}, Q_{\alpha}, R_{\beta}\rangle$ and $[W^\beta, Q_{\beta}]V_{\beta}=[V_{\beta}^{(3)}, Q_{\beta}]V_{\beta}$. But then, $[[V_{\beta}^{(3)}, Q_{\beta}]V_{\beta}, V_{\alpha'}]\le V_{\beta}$, a contradiction since $[V_{\beta}^{(3)}, Q_{\beta}]V_{\beta}/V_{\beta}$ contains a non-central chief factor. 

Thus, $Z_{\alpha'}\le [V_{\beta}^{(3)}, Q_{\beta}]V_{\beta}$. But then $V_{\alpha'-2}=Z_{\alpha'}Z_{\alpha+2}\le [V_{\beta}^{(3)}, Q_{\beta}]V_{\beta}$. If $W^\beta\le Q_{\alpha'-2}$ then $[W^\beta, Z_{\alpha'-1}]\le Z_{\alpha'-2}\cap Z_{\alpha}=\{1\}$ so that $W^\beta [V_{\beta}^{(3)}, Q_\beta]$ is normalized by $L_{\beta}=\langle V_{\alpha'}, Q_{\alpha}, R_{\beta}\rangle$ so that $V_{\beta}^{(3)}=W^\beta [V_{\beta}^{(3)}, Q_\beta]$. But then $V_{\alpha'}$ centralizes $V_{\beta}^{(3)}/[V_{\beta}^{(3)}, Q_\beta]$, a contradiction by \cref{CommCF}. Hence, there is $\lambda\in\Delta(\beta)$ with $(\lambda, \alpha')$ critical and $V_{\lambda}^{(2)}\not\le Q_{\alpha'-2}$. Thus, $(\lambda-2, \alpha'-2)$ is a critical pair so that by \cref{b=5ii}, $V_{\alpha'-2}\not\le Q_{\lambda-1}$. Now, for all $\mu\in\Delta(\beta)$ with $Z_\mu\ne Z_{\lambda}$, $[V_{\mu}^{(2)}, Q_{\beta}]V_{\beta}=[V_{\mu}^{(2)}, Q_\mu\cap Q_\beta][V_{\mu}^{(2)}, Q_{\lambda}\cap Q_{\beta}]\le Z_\mu [Q_{\lambda}, Q_{\lambda}]\le Q_{\lambda-1}$. If $Z_\mu=Z_\lambda$ then $V_{\mu}^{(2)}\le Q_{\lambda-1}$ for otherwise there is $\mu-2\in\Delta^{(2)}(\mu)$ with $(\mu-2, \lambda-1)$ a critical pair, and as $Z_\mu=Z_\lambda$ this contradicts \cref{b=5ii}. But then, $V_{\alpha'-2}\le [V_{\beta}^{(3)}, Q_{\beta}]V_{\beta}\le Q_{\lambda-1}$, a final contradiction.
\end{proof}

\begin{lemma}\label{b=5iv}
Suppose that $C_{V_\beta}(V_{\alpha'})=V_\beta \cap Q_{\alpha'}$ and $b=5$. Then $[V_{\beta}^{(3)}, Q_{\beta}]V_{\beta}\le Z(V_{\beta}^{(3)})$, $V_{\alpha'}$ acts quadratically on $V_{\beta}^{(3)}/V_{\beta}$, $p\in\{2,3\}$ and for $V:=V_{\alpha}^{(2)}/Z_{\alpha}$ either:
\begin{enumerate}
\item $S$ acts quadratically on $V$;
\item $V=[V, R_{\alpha}]$; or 
\item $V=C_V(R_{\alpha})$.
\end{enumerate}
Moreover, $\bar{L_{\beta}}\cong\SL_2(q)$.
\end{lemma}
\begin{proof}
Since $O^p(L_{\beta})$ centralizes $[V_{\beta}^{(3)}, Q_{\beta}]V_{\beta}/V_{\beta}$, $[V_{\lambda}^{(2)}, Q_{\beta}]V_{\beta}\normaleq L_{\beta}$ for any $\lambda\in\Delta(\beta)$. It follows that $[V_{\beta}^{(3)}, Q_{\beta}]V_{\beta}=[V_{\lambda}^{(2)}, Q_{\beta}]V_{\beta}$ for any $\lambda\in\Delta(\beta)$. But $V_{\lambda}^{(2)}$ is elementary abelian so that  $[V_{\beta}^{(3)}, Q_{\beta}]V_{\beta}\le Z(V_{\beta}^{(3)})$. Moreover, $[V_{\beta}^{(3)}, Q_{\beta}, V_{\beta}^{(3)}]=\{1\}$ and it follows from the three subgroup lemma that $[V_{\beta}^{(3)}, V_{\beta}^{(3)}, Q_{\beta}]=\{1\}$ so that $[V_{\beta}^{(3)}, V_{\beta}^{(3)}]\le Z(Q_{\beta})\cap V_{\beta}=Z_{\beta}$. Since $V_{\beta}V_{\alpha'}\le V_{\alpha'-2}^{(3)}$, it follows by conjugacy that $V_{\beta}^{(3)}$ is non-abelian and so $[V_{\beta}^{(3)}, V_{\beta}^{(3)}]=Z_{\beta}$. Then $[V_{\beta}^{(3)}, V_{\alpha'}, V_{\alpha'}]\le [V_{\beta}^{(3)}, V_{\alpha'-2}^{(3)}, V_{\alpha'-2}^{(3)}]\le Z_{\alpha'-2}\le V_{\beta}$, and $V_{\beta}^{(3)}/V_{\beta}$ admits quadratic action.

Now, $C_{Q_{\beta}}(V_{\beta}^{(3)})$ centralizes $V_{\alpha'-2}\le V_{\beta}^{(3)}$ and since $Z_{\alpha'}\not\le Z(V_{\beta}^{(3)})$, else $O^p(L_{\beta})$ centralizes $V_{\beta}^{(3)}/Z(V_{\beta}^{(3)})$ and $V_{\beta}^{(3)}$ is abelian, we have that $V_{\alpha'}$ centralizes $C_{Q_{\beta}}(V_{\beta}^{(3)})/V_{\beta}$ and so $O^p(L_{\beta})$ centralizes $C_{Q_{\beta}}(V_{\beta}^{(3)})V_{\beta}^{(3)}/V_{\beta}^{(3)}$. An application of the three subgroups lemma and coprime action yields that $O^p(C_{L_{\beta}}(V_{\beta}^{(3)}))$ centralizes $Q_{\beta}$ so that $C_{L_{\beta}}(V_{\beta}^{(3)})$ is a $p$-group. Moreover, $C_{\beta}=V_{\beta}^{(3)}(C_{\beta}\cap Q_{\alpha'-2})$ and $[C_{\beta}\cap Q_{\alpha'-2}, V_{\alpha'-2}^{(3)}]\le [V_{\alpha'-2}^{(3)}, Q_{\alpha'-2}]\le V_{\alpha+2}^{(2)}\le V_{\beta}^{(3)}$ so that $O^p(L_{\beta})$ centralizes $C_{\beta}/V_{\beta}^{(3)}$. But then $O^p(R_{\beta})$ centralizes $Q_{\beta}/V_{\beta}^{(3)}$. Indeed, $V_{\beta}^{(3)}/Z_{\beta}=[V_{\beta}^{(3)}/Z_{\beta}, O^p(R_{\beta})]\times C_{V_{\beta}^{(3)}/Z_{\beta}}(O^p(R_{\beta}))$. Therefore, $[O^p(R_{\beta}), V_{\beta}^{(3)}, Q_{\beta}]\le Z_{\beta}$ by the three subgroup lemma, and \sloppy{$[V_{\beta}^{(3)}, O^p(R_{\beta}), C_{V_{\beta}^{(3)}}(O^p(R_{\beta}))]=\{1\}$.}

If $[V_{\beta}^{(3)}, O^p(R_{\beta})]\not\le Q_{\alpha'-2}$, then $[V_{\alpha'-2}^{(3)}\cap Q_{\beta}, [V_{\beta}, O^p(R_{\beta})]]\le Z_{\beta}\le V_{\alpha'-2}$ and we deduce that $V_{\alpha'-2}^{(3)}/V_{\alpha'-2}$ contains a unique non-central chief factor which is an FF-module for $\bar{L_{\alpha'-2}}$. Then \cref{GoodAction3} implies that $O^p(R_{\beta})$ centralizes $V_{\beta}^{(3)}$ and an application of the three subgroup lemma yields that $O^p(R_{\beta})$ centralizes $Q_{\beta}$ and $\bar{L_{\beta}}\cong \SL_2(q)$. If $[V_{\beta}, O^p(R_{\beta})]\le Q_{\alpha'-2}$ then $Z_{\alpha'-2}\le V_{\beta}\le C_{V_{\beta}^{(3)}}(O^p(R_{\beta}))$ so that $[V_{\beta}, O^p(R_{\beta})]\le C_{\alpha'-2}$. If $Z_{\alpha'}\le [V_{\beta}, O^p(R_{\beta})]$, then $C_{V_{\beta}^{(3)}}(O^p(R_{\beta}))$ centralizes $V_{\alpha'-2}=Z_{\alpha+2}Z_{\alpha'}$ and $V_{\beta}^{(3)}\le C_{\alpha'-2}$, a contradiction. Thus, $[V_{\beta}, O^p(R_{\beta}), V_{\alpha'-2}^{(3)}]\le V_{\alpha'-2}\cap [V_{\beta}, O^p(R_{\beta})]=Z_{\beta}$ so that $O^p(L_{\beta})$ centralizes $[V_{\beta}, O^p(R_{\beta})]$. Hence, $O^p(R_{\beta})$ centralizes $V_{\beta}^{(3)}$ and the three subgroup lemma yields that $R_{\beta}=Q_{\beta}$ and $\bar{L_{\beta}}\cong \SL_2(q)$. 

Writing $Q:=Q_{\beta}\cap O^p(L_{\beta})$, we have that $[V_{\alpha}^{(2)}, Q, Q]\le [V_{\beta}^{(3)}, Q, Q]\le V_{\beta}$. By coprime action, and setting $V:=V_{\alpha}^{(2)}/Z_{\alpha}$, we have that $V=[V, R_{\alpha}]\times C_V(R_{\alpha})$ and either $V_{\beta}/Z_{\alpha}\le [V, R_{\alpha}]$ or $Q$ acts quadratically on $[V, R_{\alpha}]$. Similarly, either $V_{\beta}/Z_{\alpha}\le C_V(R_{\alpha})$ or $Q$ acts quadratically on $C_V(R_{\alpha})$. Since both $[V, R_{\alpha}]$ and $C_V(R_{\alpha})$ are normalized by $G_{\alpha}$, and $V_{\beta}/Z_{\alpha}$ generates $V$, we have shown that either $Q$ acts quadratically on $V$, $V=[V, R_{\alpha}]$ or $V=C_V(R_{\alpha})$. 

In all cases, $Q$ acts cubically on $V$ and so if $p\geq 5$ the Hall-Higman theorem yields that $O^p(R_{\alpha})$ centralizes $V_{\alpha}^{(2)}$. Since $Q$ centralizes $C_{\beta}/V_{\beta}^{(3)}$, $[C_{Q_{\alpha}}(V_{\alpha}^{(2)}), Q, Q]\le [C_{\beta}, Q, Q]\le [V_{\beta}^{(3)}, Q]\le V_{\alpha}^{(2)}$ and a standard argument implies that $O^p(R_{\alpha})$ centralizes $C_{Q_{\alpha}}(V_{\alpha}^{(2)})$ and a final application of the three subgroup lemma yields that $O^p(R_{\alpha})$ centralizes $Q_{\alpha}$, $G$ has a weak BN-pair of rank $2$ and \cite{Greenbook} provides a contradiction. Hence, $p\in\{2,3\}$.
\end{proof}

\begin{proposition}
Suppose that $C_{V_\beta}(V_{\alpha'})=V_\beta \cap Q_{\alpha'}$ and $b=5$. Then $p=3$ and $G$ is parabolic isomorphic to $\mathrm{F}_3$.
\end{proposition}
\begin{proof}
We assume throughout that $R_{\alpha}>Q_{\alpha}$, for otherwise $G$ has a weak BN-pair of rank $2$ and the proposition holds by \cite{Greenbook}. Note that $[V_{\beta}^{(3)}\cap Q_{\alpha'-2}, V_{\alpha'-2}]\le Z_{\alpha'-2}\cap [V_{\beta}^{(3)}, V_{\beta}^{(3)}]=Z_{\alpha'-2}\cap Z_{\beta}=\{1\}$ so that $V_{\beta}^{(3)}\cap Q_{\alpha'-2}\le C_{\alpha'-2}$. Then $V_{\beta}^{(3)}\cap Q_{\alpha'-2}=V_{\beta}(V_{\beta}^{(3)}\cap Q_{\alpha'})$ so that a subgroup of index at most $q^2$ in $V_{\beta}^{(3)}$ is centralized, modulo $V_{\beta}$, by $V_{\alpha'}$. In particular, $V_{\beta}^{(3)}/[V_{\beta}^{(3)}, Q_{\beta}]V_{\beta}$ is a quadratic $2$F-module for $\bar{L_{\beta}}\cong \SL_2(q)$. Write $A_\beta$ to be the preimage in $V_{\beta}^{(3)}$ of $C_{V_{\beta}^{(3)}/[V_{\beta}^{(3)}, Q_{\beta}]V_{\beta}}(O^p(L_{\beta}))$ so that $V_{\beta}^{(3)}/A_\beta$ is a direct sum of natural $\SL_2(q)$-modules by \cref{DirectSum}. Indeed, it follows that $[V_{\alpha}^{(2)}, Q]=A_\beta$ from which we deduce that $C_{V_{\beta}^{(3)}}(V_{\alpha}^{(2)})=V_{\alpha}^{(2)}$. Then, $[C_{Q_{\alpha}}(V_{\alpha}^{(2)}), Q]\le [C_{\beta}, Q]\cap C_{Q_{\alpha}}(V_{\alpha}^{(2)})\le C_{V_{\beta}^{(3)}}(V_{\alpha}^{(2)})\le V_{\alpha}^{(2)}$. By the three subgroups lemma, $p'$-order elements of $L_{\alpha}$ acts faithfully on $V_{\alpha}^{(2)}$, and by assumption we have that $V\ne C_V(R_\alpha)$. By \cref{b=5iv}, either $V$ is a quadratic module, or $V=[V, R_{\alpha}]$.

Assume that $V_{\beta}^{(3)}/A_\beta$ is a direct sum of two natural $\SL_2(q)$-modules. Write $V_1=\langle V_{\alpha'-2}^{L_{\beta}}\rangle A_{\beta}$ so that by \cref{NatGen}, $|V_1/A_\beta|=q^2$. Set $V_2 \normaleq L_{\beta}$ such that $A_\beta<V_2<V_{\beta}^{(3)}$ and $V_2\cap V_1=A_\beta$. If $V_1\not\le Q_{\alpha'-2}$, then $V_{\alpha'-2}^{(3)}$ centralizes a subgroup of $V_{\beta}^{(3)}/V_1$ of index strictly less than $q$, a contradiction. Hence, $V_1\le Q_{\alpha'-2}$ so that $V_1\le C_{\alpha'-2}$ and as $A_\beta\le Z(V_{\beta}^{(3)})$, $V_1$ is elementary abelian. Note that there is $V_\lambda\not\le V_1$ for some $\lambda\in\Delta(\alpha+2)$ and we may as well assume that $V_2=\langle V_\lambda^{L_{\lambda}}\rangle A_\beta$ and $V_2$ is also elementary abelian. Furthermore, $[V_2\cap Q_{\alpha'-2}, V_{\alpha'-2}^{(3)}]\le V_{\beta}$ and $V_{\beta}^{(3)}=V_2(V_{\beta}^{(3)}\cap Q_{\alpha'-2})$. Since $[V_{\alpha'-1}^{(2)}, V_2, V_2]=\{1\}$ and $[V_2, Q_{\alpha'-2}]Q_{\alpha'-1}\in\syl_p(L_{\alpha'-1})$, it follows that a Sylow $p$-subgroup of $L_{\alpha'-1}$ acts quadratically on $V_{\alpha'-1}^{(2)}/Z_{\alpha'-1}$ and, by conjugation, $[V_{\alpha}^{(2)}, Q_\beta, Q_\beta]\le Z_{\alpha}$ so that $V_{\alpha}^{(2)}/Z_\alpha$ is a quadratic $2$F-module for $\bar{L_{\alpha}}$. Applying \cref{2FRecog} whenever $q>p$ and \cref{Quad2F} otherwise, we deduce that since $\bar{L_{\alpha}}\not\cong \SL_2(q)$, $q\in\{2,3\}$ and by \cref{pgen}, $L_{\alpha}/C_{L_{\alpha}}(V_{\alpha}^{(2)}/Z_{\alpha})\cong \SL_2(p)$, $\Dih(10)$ or $(3\times 3):2$. Applying \cref{Badp3} when $q=3$ yields that $\bar{L_{\alpha}}\cong (Q_8\times Q_8):3$. But then, there is $t\in L_{\alpha}\cap G_{\alpha,\beta}$ an involution with $[t, L_{\alpha}]\le Q_{\alpha}$ and, since $\bar{L_{\beta}}\cong\SL_2(3)$, we can choose $t$ such that $[t, L_{\beta}]\le Q_{\beta}$, a contradiction by \cref{BasicAmal} (v). Hence, $q=2$. 

If $L_{\alpha}\ne R_{\alpha}C_{L_{\alpha}}(V_{\alpha}^{(2)}/Z_{\alpha})S$ then $C_{L_{\alpha}}(V_{\alpha}^{(2)}/Z_{\alpha})\le R_{\alpha}$ and $\bar{L_{\alpha}}\cong (3\times 3):2$. Since $\bar{L_{\beta}}\cong \Sym(3)$ in this case, we may apply \cite[Theorem B]{ChermakTri} and since $(3\times 3):2$ is not a homomorphic image of $\Sym(4)$, we force a contradiction. Hence, whenever $V_{\beta}^{(3)}/A_\beta$ is a direct sum of two natural $\SL_2(q)$-modules for $\bar{L_{\beta}}$, we have that $L_{\alpha}=R_{\alpha}C_{L_{\alpha}}(V_{\alpha}^{(2)}/Z_{\alpha})S$.

Suppose that $V_{\beta}^{(3)}/A_\beta$ is a natural $\SL_2(q)$-module and $V_{\alpha}^{(2)}\cap A_\beta$ has index $q$ in $V_{\alpha}^{(2)}$. Assume further that $V$ admits quadratic action by $S$ so that $V_{\alpha}^{(2)}\cap A_\beta=[V_{\alpha}^{(2)}, Q_{\beta}]$ has index $q$ in $V_{\alpha}^{(2)}$ and is centralized, modulo $Z_{\alpha}$, by $S$. Then by \cref{SEFF}, we deduce that $L_{\alpha}/C_{L_{\alpha}}(V)\cong \SL_2(q)$. Using \cref{Badp2}, \cref{Badp3} and an argument utilizing the Schur multiplier of $\PSL_2(q)$ when $q=3^n>3$, we conclude that $q=p$ and $\bar{L_{\alpha}}\cong (3\times 3):2$ or $(Q_8\times Q_8):3$ when $p=2$ or $3$. Indeed, $L_{\alpha}=R_{\alpha}C_{L_{\alpha}}(V)S$ and $R_{\alpha}S/Q_{\alpha}\cong \SL_2(p)$ in each case.

Whether or not $V_{\beta}^{(3)}/A_\beta$ is a sum of one or two natural modules, form $X:=\langle G_{\beta}, R_{\alpha}S\rangle$ and let $K$ be the largest subgroup of $S$ which is normalized by $X$. Then $Z_{\beta}=\langle Z_{\beta}^X\rangle\le Z(K)$ and by construction, $Z_{\alpha}\not\le K$, otherwise $V_{\beta}^{(3)}\le \langle Z_{\alpha}^X\rangle\le Z(K)$ is abelian, a contradiction. Even still, $[K, Z_{\alpha}]=\{1\}$ and taking normal closures under $X$, we deduce that $K\le C_{Q_{\beta}}(V_{\beta}^{(3)})$. But $O^p(L_{\beta})$ centralizes $C_{Q_{\beta}}(V_{\beta}^{(3)})/V_{\beta}$ and so $L_{\beta}/K$ is of characteristic $p$. Assume that $R_{\alpha}S/K$ is not of characteristic $p$ so that $O^p(R_{\alpha}S)$ acts non-trivially on $K$. Since $Z_{\alpha}\not\le K$, $K$ is not self-centralizing and we may assume that $C_S(K)\le Q_{\alpha}$ and $C_S(K)\not\le Q_{\beta}$. If $C_S(K)^x\cap Q_{\beta}\not\le Q_{\alpha}$ for some $x\in L_{\beta}$, then $[C_S(K)^x\cap Q_{\beta}, K]=\{1\}$ so that $[O^p(R_{\alpha}S), K]\le [\langle (C_S(K)^x\cap Q_{\beta})^{R_{\alpha}S}\rangle, K]=\{1\}$, a contradiction. Thus, $\langle (C_S(K)\cap Q_{\beta})^{L_{\beta}}\rangle\le Q_{\alpha}$ and so $[O^2(L_{\beta}), Q_{\beta}]\le [\langle C_S(K)^{L_{\beta}}\rangle, Q_{\beta}]\le \langle (C_S(K)\cap Q_{\beta})^{L_{\beta}}\rangle\le Q_{\alpha}$ and $Q_{\alpha}\cap Q_{\beta}\normaleq L_{\beta}$, a contradiction by \cref{push}. Thus, the triple $(G_{\beta}/K, R_{\alpha}G_{\alpha,\beta}/K, G_{\alpha,\beta}/K)$ satisfies \cref{MainHyp} and assuming that $G$ is a minimal counterexample to \hyperlink{MainGrpThm}{Theorem C}, and using that $Q_{\beta}$ contains three non-central chief factors for $\bar{L_{\beta}}$, we conclude that $(G_{\beta}/K, R_{\alpha}G_{\alpha,\beta}/K, G_{\alpha,\beta}/K)$ is parabolic isomorphic to $\mathrm{F}_3$. But now, $Z_{\alpha}/K=\Omega(Z(S/K))$ so that $V_{\alpha}^{(2)}/K$ is a natural $\SL_2(3)$-module for $\bar{P_{\alpha}}$ and $V_{\beta}^{(3)}/V_{\beta}K$ is not a quadratic module for $\bar{L_{\beta}}$, a contradiction.

Hence, $S$ does not act quadratically on $V$, $V_{\beta}^{(3)}/A_\beta$ is a natural $\SL_2(q)$-module for $\bar{L_{\beta}}$ and we deduce that $q>2$. Let $Z_{\alpha}<U< V_{\alpha}^{(2)}$ with $U\normaleq L_{\alpha}$. Since $U<V_{\alpha}^{(2)}$, $U\cap V_{\beta}=Z_{\alpha}$. Thus, if $U/Z_{\alpha}$ contains a non-central chief factor for $L_{\alpha}$, by \cref{SEFF} we deduce that $V_{\alpha}^{(2)}=U(V_{\alpha}^{(2)}\cap A_\beta)$. But then $[V_{\alpha}^{(2)}, Q, Q]=[U, Q, Q]\le V_{\beta}\cap U=Z_{\alpha}$, a contradiction since $V$ is not quadratic. Hence, $U/Z_{\alpha}\le C_V(O^p(L_{\alpha}))\le C_V(R_{\alpha})=\{1\}$ and $V$ is an irreducible faithful module for $L_{\alpha}/C_{L_{\alpha}}(V)$.

For $v\in V_{\alpha}^{(2)}\setminus A_\beta$, we have that $[v, Q_{\beta}]V_{\beta}\normaleq O^p(L_{\beta})$. But for $R$ a Hall $p'$-subgroup of $L_{\beta}\cap G_{\alpha,\beta}$, we have that $V_{\alpha}^{(2)}=\langle v(V_{\alpha}^{(2)}\cap A_\beta)^{R}\rangle$ so that $[V_{\alpha}^{(2)}, Q_{\beta}]=[v(V_{\alpha}^{(2)}\cap A_\beta), Q_{\beta}]=[v, Q_{\beta}]V_{\beta}$ and $V$ is a nearly quadratic module for $\bar{L_{\alpha}}$. An appeal to \cref{nearquad} yields that $L_{\alpha}/C_{L_{\alpha}}(V)\cong \mathrm{(P)SL}_2(q)$. Since $L_{\alpha}/R_{\alpha}\cong \SL_2(q)$, applying \cref{GoodAction1} and \cref{Badp3} (iii) when $q=3$ and an argument utilizing the Schur multiplier of $\PSL_2(q)$ when $q>3$, we deduce that $q=3$, $\bar{L_{\alpha}}\cong (2^2\times Q_8):3$ where $\bar{C_{L_{\alpha}}(V)}\cong Q_8$ and $\bar{R_{\alpha}}\cong 2^2$, and $V$ is an irreducible $2$F-module for $L_{\alpha}/C_{L_{\alpha}}(V)\cong \PSL_2(3)$.

Let $K$ be the largest subgroup of $S$ normalized by both $L_{\beta}$ and $R_{\alpha}S$. Suppose that $K\ne\{1\}$. Then $Z_{\alpha}\not\le K$, for otherwise, $Z_{\alpha}\le Z(K)$ and taking respective normal closures yields $V_{\beta}^{(3)}\le Z(K)$ is abelian, a contradiction. Since $Z_{\alpha}$ centralizes $K$, we infer that $[K, V_{\beta}^{(3)}]=\{1\}$ and $K\le C_{Q_{\beta}}(V_{\beta}^{(3)})$. Since coprime elements of $O^3(L_{\beta})$ act faithfully $Q_{\beta}/C_{\beta}$, we conclude that $L_{\beta}/K$ is of characteristic $3$. Assume that $R_{\alpha}S/K$ is not of characteristic $3$ so that $O^3(L)$ acts non-trivially on $K$. Since $Z_{\alpha}\not\le K$, $K$ is not self-centralizing and we may assume that $C_S(K)\le Q_{\alpha}$ and $C_S(K)\not\le Q_{\beta}$. If $C_S(K)^x\cap Q_{\beta}\not\le Q_{\alpha}$ for some $x\in L_{\beta}$, then $[C_S(K)^x\cap Q_{\beta}, K]=\{1\}$ so that $[O^3(L), K]\le [\langle (C_S(K)^x\cap Q_{\beta})^{R_{\alpha}S}\rangle, K]=\{1\}$, a contradiction. Thus, $\langle (C_S(K)\cap Q_{\beta})^{L_{\beta}}\rangle\le Q_{\alpha}$ and so $[O^3(L_{\beta}), Q_{\beta}]\le [\langle C_S(K)^{L_{\beta}}\rangle, Q_{\beta}]\le \langle (C_S(K)\cap Q_{\beta})^{L_{\beta}}\rangle\le Q_{\alpha}$ and $Q_{\alpha}\cap Q_{\beta}\normaleq L_{\beta}$, a contradiction by \cref{push}. Thus, the triple $(L_{\beta}/K, R_{\alpha}G_{\alpha\beta}/K, G_{\alpha,\beta}/K)$ satisfies \cref{MainHyp} and assuming that $G$ is a minimal counterexample to \hyperlink{MainGrpThm}{Theorem C}, and using that $Q_{\beta}$ contains three non-central chief factors for $\bar{L_{\beta}}$ and $O^{3'}(R_{\alpha}S/Q_{\alpha})\cong \PSL_2(3)$, we have a final contradiction.
\end{proof}

Now, we may assume that $b=3$. Unfortunately, most of the techniques introduced earlier in this section are not applicable in this setting and so the methodology for this case is different from the rest of this subsection. The aim throughout will be to show that $R_{\beta}=Q_{\beta}$ and $R_{\alpha}=Q_{\alpha}$ for then an appeal to \cite{Greenbook} yields $p=2$ and $G$ is parabolic isomorphic to $\mathrm{M}_{12}$ or $\Aut(\mathrm{M_{12}})$.

\begin{lemma}\label{b=3i}
Suppose that $C_{V_\beta}(V_{\alpha'})=V_\beta \cap Q_{\alpha'}$ and $b=3$. Then $R_{\beta}=Q_{\beta}$, $\bar{L_{\beta}}\cong\SL_2(q)$ and $O^p(L_{\beta})$ centralizes $C_{\beta}/V_{\beta}$. 
\end{lemma}
\begin{proof}
Suppose first that $V_{\alpha'}\le Q_{\beta}$. Then $R:=[V_{\beta}, V_{\alpha'}]=Z_{\beta}\le Z_{\alpha+2}=Z_{\alpha'-1}$ so that $L_{\beta}=\langle V_{\beta}, V_{\beta}^g, R_{\alpha'}\rangle$ normalizes $Z_{\alpha+2}Z_{\alpha+2}^g$ for some suitably chosen $g\in L_{\alpha'}$. Thus, $|V_{\alpha'}|=q^3=|V_{\beta}|$. Moreover, since $V_{\beta}\not\le Q_{\alpha'}$ and $[C_{\alpha'}, V_{\beta}]\le Z_{\alpha+2}\le V_{\alpha'}$, we deduce that $O^p(L_{\alpha'})$ centralizes $C_{\alpha'}/V_{\alpha'}$. Then for any $r\in R_{\alpha'}$ of order coprime to $p$, we have that $[r, Q_{\alpha'}]\le C_{\alpha'}$ by the three subgroup lemma so that $[r, Q_{\alpha'}]=\{1\}$ by coprime action. Hence, $R_{\alpha'}=Q_{\alpha'}$, $\bar{L_{\alpha'}}\cong \SL_2(q)$ and in this case the result holds by conjugacy.

If $V_{\alpha'}\not\le Q_{\beta}$, then $C_{\alpha'}=V_{\alpha'}(C_{\alpha'}\cap Q_{\beta})$ so that $[V_{\beta}, C_{\alpha'}]\le RZ_{\beta}\le V_{\alpha'}$ and again, $O^p(L_{\alpha'})$ centralizes $C_{\alpha'}/V_{\alpha'}$. As above, this implies that $R_{\alpha'}=Q_{\alpha'}$, $\bar{L_{\alpha'}}\cong \SL_2(q)$ and the result holds by conjugacy.
\end{proof}

\begin{lemma}\label{b=3iia}
Suppose that $C_{V_\beta}(V_{\alpha'})=V_\beta \cap Q_{\alpha'}$ and $b=3$. Then $|V_{\beta}|=q^3$, $Q_{\beta}/C_{\beta}$ is a natural module for $\bar{L_{\beta}}\cong \SL_2(q)$, and there exists a critical pair $(\alpha, \alpha')$ with $Z_{\alpha'}\cap Z_{\beta}=\{1\}$.
\end{lemma}
\begin{proof}
As described in \cref{b=3i}, the proof when $V_{\alpha'}\le Q_{\beta}$ is straightforward. Suppose that $V_{\alpha'}\not\le Q_{\beta}$ and $Z_{\alpha'}\ne Z_{\beta}$. Since $Z_{\alpha'-1}$ is a natural module for $L_{\alpha'-1}/R_{\alpha'-1}\cong \SL_2(q)$, we have that $Z_{\beta}\cap Z_{\alpha'}=\{1\}$. Moreover, since $V_{\alpha'}\not\le Q_{\beta}$, we deduce that $Q_{\alpha'-1}=V_{\alpha'}V_{\beta}(Q_{\alpha'-1}\cap Q_{\alpha'}\cap Q_{\beta})$. Then $[V_{\alpha'}\cap V_{\beta}, Q_{\alpha'-1}]\le Z_{\alpha'}\cap Z_{\beta}=\{1\}$ so that $V_{\alpha'}\cap V_{\beta}\le \Omega(Z(Q_{\alpha'-1}))$. Note that by \cref{push}, $Q_{\beta}\cap O^p(L_{\beta})\not\le Q_{\alpha'-1}$ so that $(Q_{\beta}\cap O^p(L_{\beta}))Q_{\alpha'-1}\in\syl_p(G_{\beta, \alpha'-1})$ and $\Omega(Z(Q_{\alpha'-1}))\cap C_{V_{\beta}}(O^p(L_{\beta}))=Z_{\beta}$. Hence, $[V_{\alpha'}, V_{\beta}]\le V_{\alpha'}\cap V_{\beta}=Z_{\alpha'-1}$ and we conclude that $|V_{\beta}|=q^3$.

Thus, we may assume that for every critical pair $(\alpha, \alpha')$, we have that $V_{\alpha'}\not\le Q_{\beta}$ and $Z_{\alpha'}=Z_{\beta}$. Since $Z_{\beta}\not\normaleq L_{\alpha+2}$, there is $\lambda\in\Delta(\alpha+2)$ such that $Z_{\lambda}\cap Z_{\beta}=\{1\}$. Moreover, by  assumption, $V_{\lambda}\le Q_{\beta}$ and $V_{\beta}\le Q_{\lambda}$ so that $[V_{\lambda}, V_{\beta}]\le Z_{\beta}\cap Z_{\lambda}=\{1\}$. Then, $[C_{\beta}\cap Q_{\lambda}, V_{\lambda}] \le [C_{\beta}, C_{\beta}]\cap Z_{\lambda}$. If $Z_{\lambda}\cap \Phi(C_{\beta})\ne\{1\}$, then $Z_{\alpha+2}=Z_{\lambda}\times Z_{\beta}\le \Phi(C_{\beta})$ and $V_{\beta}\le \Phi(C_{\beta})$. But $O^p(L_{\beta})$ centralizes $C_{\beta}/V_{\beta}$, a contradiction by coprime action. Therefore, $C_{\beta}\cap Q_{\lambda}=C_{\beta}\cap C_{\lambda}$ is of index at most $q$ in $C_{\beta}$ and $C_{\lambda}$. By the same reasoning, $C_{\alpha'}\cap Q_{\lambda}=C_{\alpha'}\cap C_{\lambda}$ and since $V_{\alpha'}\le C_{\lambda}$ and $V_{\alpha'}\not\le C_{\beta}$, $C_{\beta}\not\le Q_{\lambda}$, $C_{\beta}\cap C_{\lambda}$ is proper in $C_{\beta}$, and $C_{\lambda}=V_{\alpha'}V_{\beta}(C_{\lambda}\cap C_{\alpha'}\cap C_{\beta})$.

Since $V_{\beta}V_{\lambda}\le C_{\beta}\cap C_{\lambda}$, we have that $C_{\beta}\cap C_{\lambda}\normaleq \langle Q_{\alpha+2}, O^p(L_{\lambda}), O^p(L_{\beta})\rangle=\langle L_{\beta}, L_{\lambda}\rangle$. If $C_{\lambda}\cap C_{\beta}$ is not elementary abelian then $Z_{\alpha+2}=Z_{\beta}\times Z_{\lambda}\le \Phi(C_{\beta}\cap C_{\lambda})\le \Phi(C_{\beta})$ and $V_{\beta}\le \Phi(C_{\beta})$, a contradiction for then $O^p(L_{\beta})$ centralizes $C_{\beta}/V_{\beta}$. Hence, $C_{\beta}\cap C_{\lambda}$ is elementary abelian. Then $\Omega(Z(C_{\lambda}))=C_{\lambda}\cap C_{\beta}\cap C_{\alpha'}$ and $C_{\lambda}=V_{\beta}V_{\alpha'}\Omega(Z(C_{\lambda}))$. But then $Z_{\lambda}\le [C_{\lambda}, C_{\lambda}]=[V_{\beta}, V_{\alpha'}]=R$. By the structure of $Z_{\alpha'-1}$, there is $\mu\in\Delta(\alpha'-1)$ such that $Z_{\mu}\cap Z_{\beta}=\{1\}$ and $Z_{\alpha'-1}=Z_\mu \times Z_\lambda$. Repeating the above argument, we deduce that $Z_{\alpha'-1}\le R$. But since $R=[C_{\lambda}, C_{\lambda}]\normaleq L_{\lambda}$, we deduce that $V_\lambda\le R$, a clear contradiction. Thus, there is a pair $(\alpha, \alpha')$ satisfying the required properties and we deduce that $|V_{\beta}|=q^3$ in all cases. In particular, $C_{\beta}=Q_{\alpha}\cap Q_{\alpha+2}$ has index $q^2$ in $Q_{\beta}$ and $Q_{\beta}/C_{\beta}$ is a natural $\SL_2(q)$-module for $\bar{L_{\beta}}$.
\end{proof}

For the remainder of this section, we arrange that $(\alpha, \alpha')$ is a critical pair with $Z_{\alpha'}\ne Z_{\beta}$.

\begin{lemma}\label{b=3iib}
Suppose that $C_{V_\beta}(V_{\alpha'})=V_\beta \cap Q_{\alpha'}$ and $b=3$. Then $V_{\alpha}^{(2)}\not\le Q_{\beta}$, $V_{\alpha}^{(2)}\cap Q_{\beta}\not\le C_{\beta}$, $\Phi(V_{\alpha}^{(2)})=Z_{\alpha}$ and $C_{L_{\alpha}}(V_{\alpha}^{(2)}/Z_{\alpha})=Q_{\alpha}$
\end{lemma}
\begin{proof}
Note that if $V_{\alpha}^{(2)}\le Q_{\beta}$, then $V_{\lambda}^{(2)}\le Q_{\beta}$ for all $\lambda\in\Delta(\beta)$. Since $\alpha'$ is conjugate to $\beta$, we have that $V_{\mu}^{(2)}\le Q_{\alpha'}$ for all $\mu\in\Delta(\alpha')$. But then, $V_{\beta}\le V_{\alpha'-1}^{(2)}\le Q_{\alpha'}$, a contradiction. Now, as $Q_{\alpha}\cap Q_{\beta}\not\normaleq L_{\beta}$ by \cref{push}, we have that $O^p(L_{\beta})$ does not centralize $Q_{\beta}/C_{\beta}$. In particular, $V_{\alpha}^{(2)}\cap Q_{\beta}\ge [V_{\alpha}^{(2)}, Q_{\beta}]\not\le C_{\beta}$.

Since $V_{\alpha}^{(2)}$ is non-abelian, otherwise by conjugacy $V_{\alpha'}\le V_{\alpha+2}^{(2)}$ centralizes $V_{\beta}$, we have that $[V_{\alpha}^{(2)}, V_{\alpha}^{(2)}]=Z_{\alpha}$. But now, $V_{\alpha}^{(2)}$ is generated by $V_{\delta}$ for $\delta\in\Delta(\alpha)$ so that $V_{\alpha}^{(2)}/Z_{\alpha}$ is elementary abelian and $\Phi(V_{\alpha}^{(2)})=Z_{\alpha}$.

Let $r\in C_{L_{\alpha}}(V_{\alpha}^{(2)}/Z_\alpha)$ have order coprime to $p$. Then by coprime action, $r$ centralizes $V_{\alpha}^{(2)}$. Applying the three subgroup lemma, we have that $r$ centralizes $Q_{\alpha}/C_{Q_{\alpha}}(V_{\alpha}^{(2)})$. But $C_{Q_{\alpha}}(V_{\alpha}^{(2)})\le C_{\beta}$ so that for $Q:=Q_{\beta}\cap O^p(L_{\beta})$, $Q\not\le Q_{\alpha}$ and $[Q, C_{Q_{\alpha}}(V_{\alpha}^{(2)})]\le C_{Q_{\alpha}}(V_{\alpha}^{(2)})\cap V_{\beta}=Z_{\alpha}$ and we deduce that $r$ centralizes $Q_{\alpha}/Z_{\alpha}$ and $r=1$. Hence, $C_{L_{\alpha}}(V_{\alpha}^{(2)}/Z_{\alpha})=Q_{\alpha}$
\end{proof}

\begin{proposition}\label{b=3min}
Suppose that $C_{V_\beta}(V_{\alpha'})=V_\beta \cap Q_{\alpha'}$ and $b=3$. Then either $[V_{\alpha}^{(2)}, O^p(L_{\alpha})]=V_{\alpha}^{(2)}=Q_{\alpha}$, $Z(Q_{\alpha})=Z_{\alpha}$, $C_{V_{\alpha}^{(2)}/Z_{\alpha}}(O^p(L_{\alpha}))=\{1\}$, $L_{\beta}=\langle V_{\alpha}^{(2)}, V_{\alpha+2}^{(2)}\rangle$, $C_{\beta}=V_{\alpha}^{(2)}\cap V_{\alpha+2}^{(2)}$ and $\Phi(C_{\beta})\le Z_{\beta}$; or $G$ is parabolic isomorphic to $\Aut(\mathrm{M}_{12})$.
\end{proposition}
\begin{proof}
Note that the configurations in \hyperlink{MainGrpThm}{Theorem C} with $b=3$ (that is, the amalgams which are parabolic isomorphic to $\mathrm{M}_{12}$ and $\Aut(\mathrm{M}_{12})$) satisfy the conclusion of the lemma. Hence, assuming that $G$ does not satisfy the conclusion of the lemma, we may also suppose that $G$ is a minimally chosen counterexample to \hyperlink{MainGrpThm}{Theorem C}.

Since $Z(Q_{\alpha})$ centralizes $V_{\beta}$, $Z(Q_{\alpha})\le C_{\beta}$ and $Q_{\beta}\cap O^p(L_{\beta})$ centralizes $Z(Q_{\alpha})/Z(Q_{\alpha})\cap V_{\beta}$ where $Z(Q_{\alpha})\cap V_{\beta}=Z_{\alpha}$. Hence, $L_{\alpha}$ centralizes $Z(Q_{\alpha})/Z_\alpha$. Indeed, $Z(Q_{\alpha})$ is elementary abelian, for otherwise $L_{\alpha}$ would centralize $Z(Q_{\alpha})/\Phi(Z(Q_{\alpha}))$. Now, $Z(Q_{\alpha})=Z_\alpha(Z(Q_{\alpha})\cap Q_{\alpha'})$ and so we need only demonstrate that $Z(Q_{\alpha})\cap Q_{\alpha'}=Z_{\beta}$ to show that $Z(Q_{\alpha})=Z_{\alpha}$

For this, we observe that $Z(Q_{\alpha})\cap Q_{\alpha'}$ has exponent $p$ and if $Z(Q_{\alpha})\cap Q_{\alpha'}\le C_{\alpha'}$, then $Z(Q_{\alpha})\cap Q_{\alpha'}$ is centralized by $\langle V_{\alpha'}, Q_{\alpha}\rangle$. If $V_{\alpha'}\le Q_{\beta}$, then $S=V_{\alpha'}Q_{\alpha}$ centralizes $Z(Q_{\alpha})\cap Q_{\alpha'}=\Omega(Z(S))=Z_{\beta}$ whereas if $V_{\alpha'}\not\le Q_{\beta}$, then $L_{\beta}=\langle V_{\alpha'}, Q_{\alpha}\rangle$ centralizes $Z(Q_{\alpha})\cap Q_{\alpha'}=\Omega(Z(L_{\beta}))=Z_{\beta}$. Thus, we may assume that $Z(Q_{\alpha})\cap Q_{\alpha'}\not\le C_{\alpha'}$. But $Z(Q_{\alpha})V_{\beta}\normaleq L_{\beta}$ so that $Z(Q_{\alpha})V_{\beta}=Z(Q_{\alpha+2})V_{\beta}$ and $Z(Q_{\alpha})\cap Q_{\alpha'}\le Z(Q_{\alpha+2})V_{\beta}\cap Q_{\alpha'}=Z(Q_{\alpha+2})(V_{\beta}\cap Q_{\alpha'})\le C_{\alpha'}$, a contradiction. Therefore, $Z(Q_{\alpha})=Z_{\alpha}$.

Since $Q_{\alpha}=V_{\alpha}^{(2)}C_{\beta}$, $Q_{\beta}\cap O^p(L_{\beta})$ centralizes $Q_{\alpha}/V_{\alpha}^{(2)}$ and $[Q_{\alpha}, O^p(L_{\alpha})]\le [V_{\alpha}^{(2)}, O^p(L_{\alpha})]$. Furthermore, if $V_{\beta}\cap [V_{\alpha}^{(2)}, O^p(L_{\alpha})]>Z_{\alpha}$, then by the $G_{\alpha,\beta}$-irreducibility of $V_{\beta}/Z_{\alpha}$, $V_{\beta}\le [V_{\alpha}^{(2)}, O^p(L_{\alpha})]=V_{\alpha}^{(2)}$. Otherwise, we have that $C_{V_{\alpha}^{(2)}/Z_{\alpha}}(O^p(L_{\alpha}))$ is non-trivial. Write $C^\alpha$ for the preimage in $V_{\alpha}^{(2)}$ of this group. Note that, by definition, $V_{\alpha}^{(2)}=[V_{\alpha}^{(2)}, O^p(L_{\alpha})]V_{\beta}$ so that $[V_{\alpha}^{(2)}, Q_{\beta}]=[V_{\alpha}^{(2)}, O^p(L_{\alpha}), Q_{\beta}]$. In particular, since $C^\alpha$ is $G_{\alpha,\beta}$-invariant, if $C^\alpha\not\le C_{\beta}$ then $V_{\alpha}^{(2)}\cap Q_{\beta}=(C^\alpha\cap Q_{\beta})(V_{\alpha}^{(2)}\cap C_{\beta}$ and $[V_{\alpha}^{(2)}\cap Q_{\beta}, Q_{\beta}]\le C^\alpha([V_{\alpha}^{(2)}, O^p(L_{\alpha})]\cap V_{\beta})\le C^\alpha$ and by \cref{SEFF}, $R_{\alpha}=Q_{\alpha}$. Then \cite{Greenbook} yields a contradiction. Hence, $C^\alpha\le C_{\beta}$ so that $\{1\}=[V_\beta, C^\alpha]^{G_{\alpha}}=[V_{\alpha}^{(2)}, C^\alpha]$. Then, by the three subgroups lemma and using that $[Q_{\alpha}, O^p(L_{\alpha})]\le V_{\alpha}^{(2)}$, $[Q_{\alpha}, C^\alpha, O^p(L_{\alpha})]=\{1\}$ and since $[Z_{\alpha}, O^p(L_{\alpha})]\ne \{1\}$, $C^\alpha\le Z(Q_{\alpha})=Z_{\alpha}$, a contradiction. Thus, $V_{\alpha}^{(2)}=[V_{\alpha}^{(2)}, O^p(L_{\alpha})]$.

Set $L_{\beta}^*:=\langle V_{\alpha}^{(2)}, V_{\alpha+2}^{(2)}\rangle$ so that $L_{\beta}=L_{\beta}^*C_{\beta}$, $L_{\beta}^*\normaleq L_{\beta}$ and $O^p(L_{\beta})\le L_{\beta}^*$. Since $V_{\alpha}^{(2)}\cap Q_{\beta}\cap Q_{\alpha+2}=V_{\alpha}^{(2)}\cap C_{\beta}\normaleq G_{\beta}$, we have $V_{\alpha}^{(2)}\cap Q_{\beta}\cap Q_{\alpha+2}\le V_{\alpha+2}^{(2)}$ so that $V_{\alpha}^{(2)}\cap Q_{\beta}\cap Q_{\alpha+2}=V_{\alpha}^{(2)}\cap V_{\alpha+2}^{(2)}=V_{\alpha+2}^{(2)}\cap Q_{\beta}\cap Q_{\alpha}$. Indeed, it follows that $C_{\beta}\cap L_{\beta}^*=V_{\alpha}^{(2)}\cap V_{\alpha+2}^{(2)}$, $\Phi(C_{\beta}\cap L_{\beta}^*)\le Z_{\alpha}\cap Z_{\alpha+2}=Z_{\beta}$ and $V_{\alpha}^{(2)}$ has index $q$ in a Sylow $p$-subgroup of $L_{\beta}^*$. We note that $L_{\beta}^*<L_{\beta}$ for otherwise the conclusion of the lemma holds, and we have a contradiction since $G$ was an assumed minimal counterexample.

Since $Q_{\alpha}=V_{\alpha}^{(2)}(Q_{\alpha}\cap C_{\beta})$ and $[Q_{\beta}\cap O^p(L_{\beta}), C_{\beta}]\le V_{\beta}\le V_{\alpha}^{(2)}$ we deduce that $O^p(L_{\alpha})$ centralizes $Q_{\alpha}/V_{\alpha}^{(2)}$. Moreover, $S/V_{\alpha}^{(2)}=Q_{\alpha}/V_{\alpha}^{(2)}\rtimes (V_{\alpha+2}^{(2)}\cap Q_{\beta})V_{\alpha}^{(2)}/V_{\alpha}^{(2)}$ and by Gaschutz' theorem \cite[(3.3.2)]{kurz}, $L_{\alpha}/V_{\alpha}^{(2)}$ splits over $Q_{\alpha}/V_{\alpha}^{(2)}$. Thus, writing $S^*=S\cap L_{\beta}^*=V_{\alpha}^{(2)}(V_{\alpha+2}^{(2)}\cap Q_{\beta})$, we have that $S^*=V_{\alpha}^{(2)}(Q_{\beta}\cap O^p(L_{\beta}))\normaleq G_{\alpha,\beta}$ and $S^*\in\syl_p(\langle (S^*)^{G_{\alpha}}\rangle$.

Set $L_{\alpha}^*:=\langle (S^*)^{G_{\alpha}}\rangle$ so that $V_{\alpha}^{(2)}=O_p(L_{\alpha}^*)$. Since $S^*\normaleq G_{\alpha,\beta}$ and $L_{\lambda}^*=O^p(L_{\lambda})S^*$, $L_\lambda^*$ is normalized by $G_{\alpha,\beta}$ and $N_{L_{\mu}^*}(S^*)$ for $\{\lambda, \mu\}=\{\alpha, \beta\}$. Indeed, since no non-trivial subgroup of $G_{\alpha,\beta}$ is normal in $G$, no non-trivial subgroup $G_{\alpha,\beta}^*:=N_{L_{\alpha}^*}(S^*)N_{L_{\beta}^*}(S^*)$ is normal in $\langle L_{\alpha}^*, L_{\beta}^*\rangle$. Writing $G_{\lambda}^*=L_{\lambda}^*G_{\alpha,\beta}^*$, and using that $L_{\lambda}/Q_{\lambda}\cong L_{\lambda}^*/O_p(L_{\lambda}^*)$, for $\lambda\in\{\alpha,\beta\}$ the tuple $(G_{\alpha}^*, G_{\beta}^*, G_{\alpha,\beta}^*)$ satisfies \cref{MainHyp}. Since $L_{\beta}^*<L_{\beta}$ and $G$ was an assumed minimal counterexample, we conclude that $(G_{\alpha}^*, G_{\beta}^*, G_{\alpha,\beta}^*)$ is one of the configurations listed in \hyperlink{MainGrpThm}{Theorem C}. In particular, $p=2$, $G_{\lambda}/Q_\lambda\cong G_{\lambda}^*/O_2(G_{\lambda}^*)\cong \Sym(3)$ for $\lambda\in\{\alpha,\beta\}$ and $(G_{\alpha}^*, G_{\beta}^*, G_{\alpha,\beta}^*)$ is parabolic isomorphic to one of $\mathrm{M}_{12}$ or $\Aut(\mathrm{M}_{12})$. But then, $G$ has a weak BN-pair of rank $2$ so itself is captured in \hyperlink{MainGrpThm}{Theorem C} and satisfies the conclusion of the lemma. This contradiction completes the proof.
\end{proof}

\begin{lemma}\label{b=3minii}
Suppose that $C_{V_\beta}(V_{\alpha'})=V_\beta \cap Q_{\alpha'}$, $b=3$ and $G$ is not parabolic isomorphic to $\mathrm{M}_{12}$ or $\Aut(\mathrm{M}_{12})$. Then $[Q_{\beta}, C_{\beta}]=V_{\beta}$.
\end{lemma}
\begin{proof}
Note that $Q_{\beta}=C_{\beta}[Q_{\beta}, O^p(L_{\beta})]$ so that $[Q_{\beta}, C_{\beta}]\le [C_{\beta}, C_{\beta}][O^p(L_{\beta}), C_{\beta}]=V_{\beta}$. Assume that $C_{Q_{\beta}}(C_\beta/Z_{\alpha})>V_{\alpha}^{(2)}\cap Q_{\beta}$ and set $V:=V_{\alpha}^{(2)}/Z_{\alpha}$. Then as $C_{Q_{\beta}}(C_\beta/Z_{\alpha})$ is invariant under $G_{\alpha.\beta}$, we have that $[Q_{\beta}, C_{\beta}]\le Z_{\alpha}$ and $[Q_{\beta}, C_{\beta}]=Z_{\beta}$. Then for $v\in V_{\alpha}^{(2)}\setminus Q_{\beta}$, we have that $[v, Q_{\beta}]C_{\beta}=[V_{\alpha}^{(2)}, Q_{\beta}]C_\beta$ has index $q$ in $V_{\alpha}^{(2)}$. In particular, if $Q_{\beta}$ acts quadratically on $V$, then an index $q$ subgroup of $V$ is centralized by $Q_\beta$ and a combination of \cref{SEFF} and \cite{Greenbook} yields a contradiction.

If $V$ is not irreducible for $\bar{L_{\alpha}}$, then for some $Z_{\alpha}<U<V_{\alpha}^{(2)}$ with $U\normaleq L_{\alpha}$, since $V_{\alpha}^{(2)}=[V_{\alpha}^{(2)}, O^p(L_{\alpha})]$ and $C_V(O^p(L_{\alpha}))=\{1\}$, $U$ is non-trivial. Note that if $U\not\le Q_{\beta}$, then $V_{\alpha}^{(2)}\cap Q_{\beta}=[U, Q_{\beta}]C_{\beta}$ and so a $UC_{\beta}$ has index strictly less than $q$ in $V_{\alpha}^{(2)}$, a contradiction by \cref{SEFF}. But then $U\le Q_{\beta}$ and by \cref{SEFF}, $|UC_{\beta}/C_{\beta}|=q$ and $V_{\alpha}^{(2)}\cap Q_{\beta}=[U, Q_{\beta}]C_{\beta}$. Then by \cref{SEFF}, both $V_{\alpha}^{(2)}$ and $U/Z_{\alpha}$ are natural modules for $\bar{L_{\alpha}}$. Since $Q_{\beta}$ is not quadratic on $V$, $q>2$ and using a standard argument involving the Schur multiplier of $\PSL_2(q)$, we conclude by \cref{Badp3} that $\bar{L_{\alpha}}\cong (Q_\times Q_8):3$. But then, there is $t\in L_{\alpha}\cap G_{\alpha,\beta}$ an involution with $[t, L_{\alpha}]\le Q_{\alpha}$ and, since $\bar{L_{\beta}}\cong\SL_2(3)$, we can choose $t$ such that $[t, L_{\beta}]\le Q_{\beta}$, a contradiction by \cref{BasicAmal} (v). Hence, if $Q_{\beta}$ does not act quadratically on $V$ then $V$ is a faithful irreducible nearly quadratic module for $\bar{L_{\alpha}}$. Comparing with the list in \cref{nearquad}, using that $L_{\alpha}/R_{\alpha}\cong \SL_2(q)$, we conclude that $R_\alpha=Q_\alpha$ and \cite{Greenbook} yields a contradiction. Hence, $[Q_{\beta}, C_{\beta}]=V_{\beta}$.
\end{proof}

\begin{lemma}\label{b=3q=p}
Suppose that $C_{V_\beta}(V_{\alpha'})=V_\beta \cap Q_{\alpha'}$, $b=3$ and $G$ is not parabolic isomorphic to $\mathrm{M}_{12}$ or $\Aut(\mathrm{M}_{12})$. Then $q=p$.
\end{lemma}
\begin{proof}
Let $x\in Q_{\beta}\setminus (Q_{\beta}\cap V_{\alpha}^{(2)})$. Write $C$ for the preimage in $C_{\beta}$ of $C_{C_{\beta}/Z_{\alpha}}(x)$ so that, as $V_{\beta}\le C$, $C\normaleq L_{\beta}$. Then, $C$ is normalized by $L_{\beta}\cap G_{\alpha,\beta}$ so that $[C, Q_\beta]=[C, V_{\alpha}^{(2)}\cap Q_{\beta}][C, \langle x^{L_{\beta}\cap G_{\alpha,\beta}}\rangle]\le Z_{\alpha}$ so that $[C, Q_{\beta}]=Z_{\beta}$. In particular, $C_{C_{\beta}/Z_{\alpha}}(x)<C_\beta/Z_\alpha$ and $C/Z_{\beta}=Z(Q_\beta/Z_{\beta})$. Set $V:=V_{\alpha}^{(2)}/Z_{\alpha}$ and let $C_x$ be the preimage in $V_{\alpha}^{(2)}$ of $C_{V}(x)$ so that $C_x\cap C_{\beta}=C$. By the action of $S$ on $Q_{\beta}/C_{\beta}$, we have that $C_x\le Q_{\beta}$ so that $|C_x/C|\leq q$. Note that if $q=p$, then $C_{V}(x)=C_{V}(Q_{\beta})$.

Suppose that $q>p$ and there is $A\le S/Q_{\alpha}$ with $|A|\geq p^2$ and $C_{V}(A)=C_{V}(B)$ for all $B\le A$ with $[A:B]=p$. Using \cref{GLS2p'}, we have that $O_{p'}(\bar{L_{\alpha}})=\langle C_{O_{p'}(\bar{L_{\alpha}})}(B) \mid [A:B]=p\rangle$ so that $C_{V}(A)$ is normalized by $AO_{p'}(\bar{L_{\alpha}})$. Set $H_\alpha=\langle A^{AO_{p'}(\bar{L_{\alpha}})}\rangle$ so that $[H_\alpha, C_{V}(A)]=\{1\}$. Then, by coprime action, we get that $V=[V, O_{p'}(H_{\alpha})]\times C_{V}(O_{p'}(H_{\alpha}))$ is an $A$-invariant decomposition and as $C_{V}(A)\le C_{V}(O_{p'}(H_{\alpha}))$, it follows that $O_{p'}(H_\alpha)$ centralizes $V$ so that $O_{p'}(H_{\alpha})=\{1\}$ and $H_\alpha=A\normaleq A O_{p'}(\bar{L_{\alpha}})$. Then, $[\langle A^{\bar{L_{\alpha}}}\rangle, O_{p'}(\bar{L_{\alpha}})]=[A, O_{p'}(\bar{L_{\alpha}})]=\{1\}$ and we have that $O_{p'}(\bar{L_{\alpha}})\le Z(\bar{L_{\alpha}})$. Then \cite{Greenbook} gives a contradiction since $q>p$.

Hence, if $q>p$, then there is a maximal subgroup $A$ of $Q_\beta$ with $C_V(A)>C_V(Q_\beta)\ge C/Z_{\alpha}$. But then, for $C_A$ the preimage in $V_{\alpha}^{(2)}$ of $C_V(A)$, we have that $[C_A, V_{\alpha}^{(2)}, A]\le [Z_{\alpha}, A]=Z_{\beta}$, $[A, C_A, V_{\alpha}^{(2)}]=\{1\}$ and by the three subgroups lemma, $[V_{\alpha}^{(2)}, A, C_A]\le Z_{\beta}$. Since $([V_{\alpha}^{(2)}, A]\cap C_{\beta})V_{\beta}\normaleq L_{\beta}$, we have that $Z_{\beta}= [([V_{\alpha}^{(2)}, A]\cap C_{\beta})V_{\beta}, C_AC_\beta]^{L_{\beta}}=[([V_{\alpha}^{(2)}, A]\cap C_{\beta})V_{\beta}, \langle C_A^{L_{\beta}}\rangle C_\beta]=[([V_{\alpha}^{(2)}, A]\cap C_{\beta})V_{\beta}, Q_{\beta}]$. In particular, $([V_{\alpha}^{(2)}, A]\cap C_{\beta})V_{\beta}\le C$. 

It follows that $[V_{\alpha}^{(2)}, Q_{\beta}]\cap C_{\beta}\le C$ so that $[V, Q_{\beta}, Q_{\beta}, Q_{\beta}]=\{1\}$ and by \cref{CubicAction}, if $p\geq 5$ then $R_{\alpha}=Q_{\alpha}$. In this scenario, \cite{Greenbook} provides a contradiction. Hence, $p\leq 3$. Moreover, $[V_{\alpha}^{(2)}, Q_{\beta}]C=[V_{\alpha}^{(2)}, y]C$ for all $y\in Q_{\beta}\setminus (V_{\alpha}^{(2)}\cap Q_{\beta})$. Indeed, $[V, y, y]=[V, Q_{\beta}, y]$. Hence if $y$ is quadratic on $V$, then $[V, Q_{\beta}, Q_{\beta}]=\{1\}$ and $C_V(Q_{\beta})=C_V(A)$ for all maximal subgroups of $A$ of $Q_{\beta}$ which contain $V_{\alpha}^{(2)}\cap Q_{\beta}$, a contradiction. Since involutions in $S/Q_{\alpha}$ always act quadratically on $V$ when $p=2$, we conclude that $p=3$.

Now, $C_V(x)>C_V(B)>\dots>C_V(A)>C_V(Q_{\beta})\geq C/Z_{\alpha}$ for a maximal chain of subgroups $\langle x\rangle<B<\dots <A<Q_\beta$ and we deduce that $C_xC$ has index at most $3$ in $[V_{\alpha}^{(2)}, Q_{\beta}]C$. Then, the commutation map $\theta:[V_{\alpha}^{(2)}, Q_{\beta}]C/Z_\alpha\to [V_{\alpha}^{(2)}, Q_{\beta}]C/Z_\alpha$ such that $[v, q]cZ_\alpha\theta=[[v, q]cZ_\alpha, x]$ has image $[[V_{\alpha}^{(2)}, Q_{\beta}], x]Z_\alpha/Z_{\alpha}$ and kernel $C_x/Z_{\alpha}$. Hence, $|[V, Q_{\beta}, x]|=3$. Note that $[V, Q_{\beta}, x]$ is normalized by $S=V_{\alpha}^{(2)}Q_{\beta}$ and the three subgroups lemma yields that $[V, Q_{\beta}, x]=[V, x, Q_{\beta}]$. But using $[V_{\alpha}^{(2)}, Q_{\beta}]C=[V_{\alpha}^{(2)}, x]C$, we have that $[V, x, Q_{\beta}]=[V, Q_{\beta}, Q_{\beta}]$ has order $p$ and $[V, Q_{\beta}, Q_{\beta}]=[V, y, y]$ for all $y\in Q_{\beta}\setminus (V_{\alpha}^{(2)}\cap Q_{\beta})$. By coprime action, $O_{p'}(\bar{L_{\alpha}})=\langle C_{O_{p'}(\bar{L_{\alpha}})}(yQ_\alpha) \mid y\in Q_{\beta}\setminus (V_{\alpha}^{(2)}\cap Q_{\beta})\rangle$ so that $O_{p'}(\bar{L_{\alpha}})\bar{G_{\alpha,\beta}}$ normalizes $[V, Q_{\beta}, Q_{\beta}]$.

Form $H=\langle \bar{S}^{O_{3'}(\bar{L_{\alpha}})\bar{G_{\alpha,\beta}}}\rangle$ so that $O_{p'}(H)\le \bar{R_{\alpha}}$ since $q>3$. By coprime action, $V=[V, O_{3'}(H)]\times C_V(O_{3'}(H))$ and we write $V_1, V_2$ for the preimage in $V_{\alpha}^{(2)}$ of $[V, O_{3'}(H)]$, $C_V(O_{3'}(H))$ respectively. Note that $[V, Q_{\beta}, Q_{\beta}]\le C_V(O_{3'}(H))$ so that $[V_1, Q_\beta, Q_\beta]\le Z_{\alpha}$. Assume first that $V_1\not\le Q_{\beta}$. Since $V_1$ is $G_{\alpha,\beta}$-invariant, $V_{\alpha}^{(2)}=V_1C_{\beta}$ so that $[V_{\alpha}^{(2)}, Q_{\beta}, Q_{\beta}]=[V_1, Q_{\beta}, Q_{\beta}]Z_{\beta}\le Z_{\alpha}$, a contradiction. Suppose now that $V_1\le Q_{\beta}$ but $V_1\not\le C_{\beta}$. Since $V_{\alpha}^{(2)}\not\le Q_{\beta}$, $V_2\not\le Q_{\beta}$ and $V_{\alpha}^{(2)}=V_2C_{\beta}$. Then $V_1\le [V_2, Q_{\beta}]C_{\beta}$ so that $[V_1, Q_{\beta}]\le V_1\cap [V_2, Q_{\beta}, Q_{\beta}]V_{\beta}\le V_{\beta}$. But then $[V_{\alpha}^{(2)}\cap Q_{\beta}, Q_{\beta}]=[V_1, Q_{\beta}][C_{\beta}, Q_{\beta}]=V_{\beta}$ so that $\Phi(Q_{\beta})=V_{\beta}$. Since $p=3$, $Q_{\beta}/V_{\beta}$ splits under the action of the central involution of $\bar{L_{\beta}}$ and writing $Q:=Q_{\beta}\cap O^3(L_{\beta})$, $Q\cap C_{\beta}=V_{\beta}$ and $[V, Q, Q]\le V_{\beta}/Z_{\alpha}$. Thus, $[V, Q, Q]=\{1\}$, a contradiction since $Q\not\le V_{\alpha}^{(2)}$. Finally, we have that $V_1\le C_{\beta}$ so that $[V_1, V_{\beta}]=\{1\}$. But then $[V_{\alpha}^{(2)}, O_{p'}(H), V_{\beta}]\le [V_1, V_{\beta}]=\{1\}$, $[V_{\alpha}^{(2)}, V_{\beta}, O_{3'}(H)]=[Z_{\alpha}, O_{3'}(H)]=\{1\}$ and the three subgroup lemma yields that $[V_{\beta}, O_{3'}(H)]\le Z(V_{\alpha}^{(2)}=Z_{\alpha}$. But then $V_{\alpha}^{(2)}=V_2C_\beta$ so that $[V_{\alpha}^{(2)}, Q_{\beta}]\le V_2$ and $O_{3'}(H)$ centralizes $V$. Hence, $O_{3'}(H)=\{1\}$, $\bar{S}\normaleq H$, $[\bar{S}, O_{3'}(H)]=\{1\}$ and $\bar{L_{\alpha}}$ is central extension of $\PSL_2(q)$ by a $3'$-group. Since $q>3$, we have that $R_{\alpha}=Q_{\alpha}$ and \cite{Greenbook} provides a contradiction.
\end{proof}

\begin{proposition}
Suppose that $C_{V_\beta}(V_{\alpha'})=V_\beta \cap Q_{\alpha'}$ and $b=3$. Then $p=2$ and $G$ is parabolic isomorphic to $\mathrm{M}_{12}$ or $\Aut(\mathrm{M}_{12})$.
\end{proposition}
\begin{proof}
By \cref{b=3q=p}, we have that $q=p$. For $Q:=[Q_{\beta}, O^p(L_{\beta})]$, we have that $|Q/[Q,Q]V_{\beta}|=p^2$ by coprime action using the central involution of $\bar{L_{\beta}}$ when $p\geq 5$, and the $p$-solvability of $L_{\beta}$ otherwise. Then, since $Q_{\beta}=QC_{\beta}$ and $[Q, C_{\beta}]\le V_{\beta}$, we have that $\Phi(Q_{\beta})=\Phi(Q)V_{\beta}$ and $\Phi(Q)V_{\beta}=C_{\beta}\cap Q$. Note that by \cite[Lemma 2.73]{parkerSymp}, we have that $[Q,Q]V_{\beta}/V_{\beta}$ has order at most $p$ and $Q/V_{\beta}$ is either elementary abelian or an extraspecial group. 

Assume that $Q/V_{\beta}$ is non-abelian. Then $(V_{\alpha}^{(2)}\cap Q)/V_{\beta}$ is elementary abelian of index $p$ in $Q/V_{\beta}$ and it follows that $Q/V_\beta$ is dihedral when $p=2$; and is extraspecial of exponent $p$ when $p$ is odd. Since $\bar{L_{\beta}}$ acts faithfully on $Q/V_{\beta}$, we conclude that $p$ is odd. Furthermore, since $Q/Z_{\beta}$ has class at most $3$, applying the Hall-Higman theorem to the action of $Q$ of $V$, we deduce that either $p\in\{3,5\}$, or $R_{\alpha}=Q_{\alpha}$. In the latter case, \cite{Greenbook} yields a contradiction.

As in the proof of \cref{b=3q=p}, we form $C_x$ to be the preimage in $V_{\alpha}^{(2)}$ of $C_{V_{\alpha}^{(2)}/Z_{\alpha}}(Q_\beta)$ and we deduce that $C_x\le V_{\alpha}^{(2)}\cap Q_{\beta}$. Furthermore, $[C_xC_{\beta}, Q_{\beta}]=V_{\beta}$ and as $Q_{\beta}/V_{\beta}$ is non-abelian, $[V_{\alpha}^{(2)}\cap Q_{\beta}, Q_{\beta}]>V_{\beta}$ and $C_x\le C_{\beta}$. Then by the proof of \cref{b=3q=p}, $C_x=C<C_{\beta}$ where $C/Z_{\beta}=Z(Q_{\beta}/Z_{\beta})$. Indeed, $C_x$ has index at least $p^3$ in $V_{\alpha}^{(2)}$ Moreover, the commutation homomorphism $\theta:(V_{\alpha}^{(2)}\cap Q_{\beta})/Z_\alpha\to (V_{\alpha}^{(2)}\cap Q_{\beta})/Z_\alpha$ such that $vZ_\alpha\theta=[vZ_\alpha, x]$ has image contained in $\Phi(Q_\beta)V_{\beta}/Z_{\alpha}$ which has order $p^2$ and kernel $C_x$ from which we deduce that $C_x=C$ has index $p$ in $C_{\beta}$.

Since $p$ is odd, and $V_{\alpha}^{(2)}$ has class $2$ and is generated by elements of order $p$, \cite[Lemma 10.14]{GLS2} yields that $V_{\alpha}^{(2)}=\Omega(V_{\alpha}^{(2)})$ is of exponent $p$. In particular, $(V_{\lambda}^{(2)}\cap Q_{\beta})/Z_{\beta}$ is of exponent $p$ for all $\lambda\in\Delta(\beta)$ and we deduce that $\Omega(Q_{\beta}/Z_{\beta})=Q_{\beta}/Z_{\beta}$. Since $|Q/Z_{\alpha}|\in\{3^4, 5^4\}$ and $|Q/Z_{\beta}|\in\{3^5, 5^5\}$, we will make liberal use of the Small Groups library in MAGMA to derive a contradiction when $\Phi(Q_{\beta})>V_\beta$.

Assume first that $Z(Q/Z_{\alpha})=V_{\beta}/Z_{\alpha}$. Then $\Phi(Q)/Z_{\beta}\not\le Z(Q_{\beta}/Z_{\beta})$ so that $C_\beta=\Phi(Q)C$. Write $D=C_{C}(O^p(L_{\beta}))$. Then $C=V_{\beta}D$ and $[Q_{\beta}, D]=[QC, D]=D'$. Note that $D\le Q_{\alpha'}$ and $[D, V_{\alpha'}]\le Z_{\alpha'}\cap D=\{1\}$ so that $D\le C_{\beta}\cap C_{\alpha'}$. In particular, $\Phi(D)\le \Phi(C_{\beta})\cap \Phi(C_{\alpha'})=Z_{\beta}\cap Z_{\alpha'}=\{1\}$ and $D$ is centralized by $V_{\alpha'}Q_{\beta}$. Since $D$ is of exponent $p$, $D=Z_{\beta}$, $C=V_{\beta}$, $|C_{\beta}/V_{\beta}|=p$ and $|Q|=|Q_{\beta}|=p^6$. If $p=5$, then $\bar{L_{\alpha}}$ is a subgroup of $\SL_4(5)$ with a strongly $5$-embedded subgroup and some quotient by a $5'$-group isomorphic to $\SL_2(5)$. Computing in MAGMA, we have that $R_{\alpha}=Q_{\alpha}$, and a contradiction follows from \cite{Greenbook}. 

Therefore, if $Z(Q/Z_{\alpha})=V_{\beta}/Z_{\alpha}$ then $p=3$. Now, $|\Phi(Q_\beta)Z_{\alpha}/Z_{\alpha}|=9$ and we deduce that $V_{\beta}/Z_{\alpha}<\Phi(Q_\beta)Z_{\alpha}/Z_{\alpha}=\Phi(Q_\beta/Z_{\alpha})$ and $Q_\beta/Z_{\alpha}$ has maximal class. Moreover, $(V_{\alpha}^{(2)}\cap Q_\beta)/Z_{\alpha}$ is elementary abelian of index $p$ in $Q_\beta/Z_{\alpha}$ so that $Q_\beta/Z_{\alpha}$ is isomorphic to $3 \wr 3=SmallGroup(3^4, 7)$. Furthermore, we infer that $V_{\beta}/Z_{\beta}=Z(Q_\beta/Z_{\beta})$ has order $9$, $\Phi(Q_\beta/Z_{\beta})$ has order $27$, $\Omega(Q_\beta/Z_{\beta})=Q_\beta/Z_{\beta}$ and $Q_\beta/Z_\lambda\cong Q_\beta/Z_{\alpha}$ for any $\lambda\in\Delta(\beta)$. Calculating in MAGMA, the only possibility is $Q_\beta/Z_{\beta}\cong SmallGroup(3^5, 3)$. Note that $L_{\beta}/C_{Q_{\beta}}(Q/Z_{\beta})$ embeds a subgroup of $\Aut(Q/Z_{\beta})$. Even better, $L_{\beta}/C_{Q_{\beta}}(Q/Z_{\beta})$ embeds as a subgroup of the normal closure of a Sylow $3$-subgroup in $\Aut(Q/Z_{\beta})$. But in this case, $\Aut(Q/Z_{\beta})$ has normal Sylow $3$-subgroup, a contradiction.

Assume now that $Z(Q/Z_{\alpha})=V_{\beta}\Phi(Q)/Z_{\alpha}$ has index $p^2$ in $Q/Z_{\alpha}$. In particular, $Q$ has class $2$ and so acts cubically on $V_{\alpha}^{(2)}/Z_{\alpha}$. By \cref{CubicAction}, either $p=3$ or $R_{\alpha}=Q_{\alpha}$. In the latter case, \cite{Greenbook} gives a contradiction. Note that $Z_{\alpha}\ge [\Phi(Q)V_{\beta}, Q]\normaleq L_{\beta}$ so that $\Phi(Q)V_{\beta}/Z_{\beta}=Z(Q/Z_{\beta})$. If $V_{\beta}\le \Phi(Q)$ then $\Phi(Q)/Z_{\beta}=\Phi(Q/Z_{\beta})=Z(Q/Z_{\beta})$ has order $27$, and $\Omega(Q/Z_{\beta})=Q/Z_{\beta}$. No such group exists. 
Hence, $\Phi(Q)\cap V_{\beta}=Z_{\beta}$ so that $\Phi(Q/Z_{\beta})=3$, $Z(Q/Z_{\beta})=27$ and $Q/Z_{\beta}\cong SmallGroup(3^5, 62)$. Then $Q/Z_{\alpha}\cong SmallGroup(81,12)$ has four distinct elementary abelian subgroups of order $3^3$. Since one of these groups is $(V_{\alpha}^{(2)}\cap Q)/Z_{\alpha}$ which is normalized by $L_{\beta}\cap G_{\alpha,\beta}$, at least one other is also normalized by $L_{\beta}\cap G_{\alpha,\beta}$. But $Q/\Phi(Q)V_{\beta}\cong Q_{\beta}/C_{\beta}$ is a natural $\SL_2(3)$-module for $\bar{L_{\beta}}$ and, as such, has a unique subspace of order $3$ normalized by $\bar{L_{\beta}\cap G_{\alpha,\beta}}$, a contradiction. 

Thus, $\Phi(Q_{\beta})=V_{\beta}$ and setting $V:=V_{\alpha}^{(2)}/Z_{\alpha}$, $Q_{\beta}$ acts cubically on $V$ and \cref{CubicAction} yields that $p\in\{2,3\}$ or $R_{\alpha}=Q_{\alpha}$. In the latter case, \cite{Greenbook} provides a contradiction. We again recall the notations $C$ and $C_x$ from \cref{b=3q=p}. We may assume that $R_{\alpha}\ne Q_{\alpha}$ otherwise the result holds by \cite{Greenbook}. In particular, $V$ is not an FF-module by \cref{SEFF}. Then, forming the commutation homomorphism $\theta:(V_{\alpha}^{(2)}\cap Q_{\beta})/Z_\alpha\to (V_{\alpha}^{(2)}\cap Q_{\beta})/Z_\alpha$ such that $vZ_\alpha\theta=[vZ_\alpha, x]$ has image $V_{\beta}/Z_{\alpha}$ and kernel $C_x/Z_{\alpha}$. Thus, $C_x$ has index $p$ in $V_{\alpha}^{(2)}\cap Q_{\beta}$ and index $p^2$ in $V_{\alpha}^{(2)}$. Note that $C_x\ne C_{\beta}$ for otherwise $[Q_{\beta}, C_{\beta}]=[\langle x\rangle(V_{\alpha}^{(2)}\cap Q_{\beta}), C_{\beta}]\le Z_{\alpha}$, a contradiction by \cref{b=3minii}. Indeed, $V_{\alpha}^{(2)}\cap Q_{\beta}=C_xC_{\beta}$.

Since $V=[V, O^p(L_{\alpha})]$, if $V$ is not irreducible for $\bar{L_{\alpha}}$ then by \cref{SEFF}, $V$ contains two non-central chief factors for $\bar{L_{\alpha}}$ and both are natural $\SL_2(p)$-modules. By \cref{Badp2} and \cref{Badp3}, we conclude that $\bar{L_{\alpha}}\cong (3\times 3):2$ or $(Q_8\times Q_8):3$ and $|V|=p^4$. If $x$ acts quadratically on $V$ then both $V$ and $\bar{L_{\alpha}}$ are described by \cref{Quad2F}. Since $R_{\alpha}\ne Q_{\alpha}$ we either have that $\bar{L_{\alpha}}\cong (3\times 3):2$ or $\SU_3(2)'$ and $p=2$, or $p=3$ and $\bar{L_{\alpha}}\cong (Q_8\times Q_8):3$. 

If $\bar{L_{\alpha}}\cong \SU_3(2)'$ or $(Q_8\times Q_8):3$ then there is $t\in L_{\alpha}\cap G_{\alpha,\beta}$ non-trivial with $[t, L_{\alpha}]\le Q_{\alpha}$ and, since $\bar{L_{\beta}}\cong\SL_2(p)$, we can choose $t$ such that $[t, L_{\beta}]\le Q_{\beta}$, a contradiction by \cref{BasicAmal} (v). Hence, $p=2$, $\bar{L_{\alpha}}\cong (3\times 3):2$ and as $V=[V, O^2(L_{\alpha})]$, $|V|=2^4$ and $|S|=2^7$. By \cref{Badp3}, there is $P_{\alpha}\le L_{\alpha}$ such that $P_{\alpha}/Q_{\alpha}\cong \Sym(3)$, $L_{\alpha}=P_{\alpha}R_{\alpha}$ and we may choose $P_{\alpha}$ such that neither $V_{\beta}$ nor $C_{\beta}$ are normal in $P_{\alpha}$. It follows that no subgroup of $S$ is normal in both $P_{\alpha}$ and $L_{\beta}$ so that $(P_{\alpha}, L_{\beta}, S)$ satisfies \cref{MainHyp}. Since we could have chosen $G$ minimally as a counterexample, and as $|S|=2^7$, we deduce that $(P_{\alpha}, L_{\beta}, S)$ is parabolic isomorphic to $\Aut(\mathrm{M}_{12})$. But then one can calculate, e.g. using MAGMA, that $|\Aut(Q_{\alpha})|_3=3$, a contradiction.

Hence, we have that $V$ is irreducible, $x$ is not quadratic on $V$ and $p=3$. But then $[V_{\alpha}^{(2)}, x]\not\le C_x$ and we deduce that $V_{\alpha}^{(2)}\cap Q_{\beta}=[V_{\alpha}^{(2)}, x]C_x$. Then, $V$ is a nearly quadratic module for $\bar{L_{\alpha}}$ and an appeal to \cref{nearquad} yields a contradiction. This completes the proof.
\end{proof}

%% file: Contents/7.3.b=1.tex
\subsection{$b=1$}

From this point on, restating \cref{bodd}, we may assume the following:

\begin{itemize}
\item $b=1$ so that $Z_\alpha\not\le Q_\beta$;
\item $\Omega(Z(S))=Z_\beta=\Omega(Z(L_\beta))$; and
\item $Z(L_\alpha)=\{1\}$.
\end{itemize}

\begin{proposition}\label{betadeduct}
Suppose that $p$ is odd. Then $\bar{L_{\beta}}\cong\SL_2(p^n)$ or $\mathrm{(P)SU}_3(p^n)$, or $|S/Q_{\beta}|=3$.
\end{proposition}
\begin{proof}
Since $Q_{\beta}/\Phi(Q_{\beta})$ is a faithful $\bar{L_{\beta}}$ module and $[Q_{\beta}, Z_{\alpha}, Z_{\alpha}]=\{1\}$, the result follows immediately from \cref{SEQuad}.
\end{proof}

\begin{proposition}
Suppose that $p\geq 5$. Then $G$ has a weak BN-pair of rank $2$ and is locally isomorphic to $H$ where $F^*(H)=\PSp_4(p^n)$, $\PSU_4(p^n)$ or $\PSU_5(p^n)$.
\end{proposition}
\begin{proof}
Let $K_{\beta}$ be a critical subgroup of $Q_{\beta}$. By \cref{CriticalSubgroup}, $O^p(L_{\beta})$ acts faithfully on $K_\beta/\Phi(K_{\beta})$. Assume that $K_{\beta}\le Q_{\alpha}$. Since $\bar{L_{\beta}}\cong \SL_2(p^n)$ or $\mathrm{(P)SU}_3(p^n)$, we have that $[K_{\beta}, O^p(L_{\beta})]\le [K_{\beta}, \langle Z_{\alpha}^{L_{\beta}}\rangle]=\{1\}$, a contradiction. Hence, $K_\beta\not\le Q_\alpha$, $[Q_{\alpha}, K_{\beta}, K_{\beta}, K_{\beta}]=\{1\}$ and $K_{\beta}$ acts cubically on $Q_{\alpha}$. 

Since $Q_{\alpha}/\Phi(Q_{\alpha})$ is a faithful $\bar{L_{\alpha}}$-module which admits cubic action, we may apply \cref{CubicAction} so that $\bar{L_{\alpha}}\cong\mathrm{(P)SL}_2(p^r)$ or $\mathrm{(P)SU}_3(p^r)$, or $p=5$ and $\bar{L_{\alpha}}\cong 3\cdot\Alt(6)$ or $3\cdot\Alt(7)$ and for $W$ some composition factor of $Q_{\alpha}/\Phi(Q_{\alpha})$, $|W|\geq 5^6$. If $\bar{L_{\alpha}}\cong\mathrm{(P)SL}_2(p^r)$ or $\mathrm{(P)SU}_3(p^r)$ then $G$ has a weak BN-pair of rank $2$ and by \cite{Greenbook}, $G$ is locally isomorphic to $H$ where $F^*(H)=\PSp_4(p^{n+1})$, $\PSU_4(p^n)$ or $\PSU_5(p^n)$ for $n\geq 1$. Thus it remains to check that $\bar{L_{\alpha}}\not\cong 3\cdot\Alt(6)$ or $3\cdot\Alt(7)$ and so assume that $p=5$ and $|S/Q_{\alpha}|=5$. Since $Q_{\beta}$ is not centralized by $Z_{\alpha}$, $\bar{L_{\beta}}\cong\SL_2(5)$ and $Q_{\beta}$ contains exactly one non-central chief factor for $L_{\beta}$, which is isomorphic to a natural $\SL_2(5)$-module. Since $Z(L_{\alpha})=\{1\}$, $Z_{\alpha}$ contains a non-central chief factor for $L_{\alpha}$ and admits cubic action, $Z_{\alpha}$ is also a faithful $\bar{L_{\alpha}}$-module and $|Z_{\alpha}|\geq 5^6$, so that $R_{\alpha}=Q_{\alpha}$.

Suppose that $Z_{\alpha}\cap Q_{\beta}\le Q_{\lambda}$ for all $\lambda\in\Delta(\beta)$. Since $L_{\beta}=\langle Z_{\lambda}, Q_{\beta} \mid \lambda\in\Delta(\beta)\rangle$, it follows that $Z_{\alpha}\cap Q_{\beta}$ is centralized by $O^p(L_{\beta})$. Since $Q_{\alpha}\cap Q_{\beta}\not\normaleq L_{\beta}$, $O^p(L_{\beta})\cap Q_{\beta}\not\le Q_{\alpha}$ and so $[Z_{\alpha}, Q_{\beta}, Q_{\beta}\cap O^p(L_{\beta})]=\{1\}$ and $Z_{\alpha}$ is a quadratic module, a contradiction to \cref{SEQuad}. Thus, $Z_{\alpha}\cap Q_{\beta}\not\le Q_{\alpha+2}$ for some $\alpha+2\in\Delta(\beta)$ and $Z_{\alpha}\cap Q_{\beta}\cap Q_{\alpha+2}$ has index at most $25$ in $Z_{\alpha}$. If $Z_{\alpha+2}\cap Q_{\beta}\le Q_{\alpha}$ then $[Z_{\alpha+2}, Z_{\alpha}, Z_{\alpha}]=\{1\}$ and so, $Z_{\alpha}\cap Q_{\beta}$ acts quadratically on $Z_{\alpha+2}$ and since $\alpha+2$ is conjugate to $\alpha$, we have a contradiction. Thus, $Z_{\alpha+2}\cap Q_{\beta}\not\le Q_{\alpha}$. But now, $\bar{L_{\alpha}}$ is generated by two conjugates of $(Z_{\alpha+2}\cap Q_{\beta})Q_{\alpha}/Q_{\alpha}$, and as an index $25$ subgroup of $Z_{\alpha}$ is centralized by $Z_{\alpha+2}\cap Q_{\beta}$ and $Z(L_{\alpha})=\{1\}$, we have that $|Z_{\alpha}|\leq 5^4$, a contradiction.
\end{proof}

Given the above proposition, we suppose that $p\in \{2,3\}$ for the remainder of this subsection. We introduce some notation specific to the case where $b=1$.

\begin{itemize}
\item $F_\beta$ is a normal subgroup of $G_\beta$ which satisfies $[F_\beta, O^p(L_{\beta})]\ne \{1\}$ and is minimal by inclusion with respect to adhering to these conditions.
\item $W_\beta:=\langle(Z_{\alpha}\cap Q_{\beta})^{G_{\beta}}\rangle$.
\item $D_{\beta}:=C_{Q_{\beta}}(O^p(L_{\beta}))$.
\end{itemize}

\begin{lemma}\label{NCCFZA}
The following hold:
\begin{enumerate}
\item $F_{\beta}\not\le Q_{\alpha}$;
\item $F_{\beta}=[F_{\beta}, O^p(L_{\beta})]\le O^p(L_{\beta})$; and
\item for any $p$-subgroup $U\normaleq L_{\alpha}$ with $U\not\le Q_{\beta}$, $[F_{\beta}, Q_{\beta}]\le U$.
\end{enumerate}
\end{lemma}
\begin{proof}
We have that $[F_{\beta}, O^p(L_{\beta})]\le O^p(L_{\beta})$ and by coprime action $[F_{\beta}, O^p(L_{\beta}), O^p(L_{\beta})]=[F_{\beta}, O^p(L_{\beta})]$. By minimality of $F_{\beta}$, $F_{\beta}=[F_{\beta}, O^p(L_{\beta})]$. If $F_{\beta}\le Q_{\alpha}$, then $[F_{\beta}, S]$ is strictly contained in $F_{\beta}$ and normalized by $G_{\beta}=\langle Z_{\alpha}^{L_{\beta}}\rangle G_{\alpha,\beta}$ and, by minimality, $[F_{\beta}, S]\le D_{\beta}$. But then $[F_{\beta}, L_{\beta}]\le D_{\beta}$, a contradiction.

Let $H_{\beta}:=\langle (U\cap F_{\beta})^{G_{\beta}}\rangle\normaleq G_{\beta}$ for $U\normaleq L_{\alpha}$ with $U\not\le Q_{\beta}$. By minimality of $F_{\beta}$, either $H_{\beta}=F_{\beta}$ or $H_{\beta}\le D_{\beta}$. Suppose the latter. Then $[F_{\beta}, U]\le F_{\beta}\cap U\le H_{\beta}\le D_{\beta}$ so that $[F_{\beta}, \langle U^{G_{\beta}}\rangle]\le D_{\beta}$. Now, $F_{\beta}=[F_{\beta}, O^p(L_{\beta})]\le [F_{\beta}, \langle U^{G_{\beta}}\rangle G_{\alpha,\beta}]\le D_{\beta}[F_{\beta}, G_{\alpha,\beta}]$. Then, by minimality of $F_{\beta}$, $F_{\beta}/F_{\beta}\cap D_{\beta}$ is an irreducible $\bar{G_{\alpha,\beta}}$-module so that $[S, F_{\beta}]\le D_{\beta}$. As above, this implies that $[F_{\beta}, L_{\beta}]\le D_{\beta}$, a contradiction. Thus, $H_{\beta}=F_{\beta}$. Now, $[U\cap F_{\beta}, Q_{\beta}]\le [F_{\beta}, Q_{\beta}]\le D_{\beta}$ and so $[U\cap F_{\beta}, Q_{\beta}]\normaleq G_{\beta}$. But then $[U\cap F_{\beta}, Q_{\beta}]=[U\cap F_{\beta}, Q_{\beta}]^{G_{\beta}}=[\langle (U\cap F_{\beta})^{G_{\beta}}\rangle, Q_{\beta}]=[F_{\beta}, Q_{\beta}]$ and $U\ge [U\cap F_{\beta}, Q_{\beta}]=[F_{\beta}, Q_{\beta}]$, completing the proof.
\end{proof}

\begin{lemma}\label{mp1} 
Suppose that $m_p(S/Q_{\alpha})=1$ and $p\in\{2,3\}$. Then $p=3$, $\bar{L_{\beta}}\cong\SL_2(3)$, $Z_{\alpha}$ is an irreducible 2$F$-module for $\bar{L_{\alpha}}$ and $Q_{\alpha}$ is elementary abelian.
\end{lemma}
\begin{proof}
Assume that $m_p(S/Q_{\alpha})=1$. Since $W_{\beta}$ is generated by elements of order $p$ and $m_p(S/Q_{\alpha})=1$, $|W_\beta Q_\alpha/Q_\alpha|=p$ and $Z_\alpha$ centralizes an index $p$ subgroup of $W_\beta$. Since $[Z_{\alpha}, Q_{\beta}]\le W_{\beta}$, $W_{\beta}$ contains all non-central chief factors for $L_{\beta}$ in $Q_{\beta}$ and so, $W_{\beta}/C_{W_{\beta}}(O^p(L_{\beta}))$ is the unique non-central chief factor for $L_{\beta}$ inside $Q_{\beta}$. Moreover, $W_{\beta}/C_{W_{\beta}}(O^p(L_{\beta}))$ is a natural $\SL_2(p)$-module for $\bar{L_{\beta}}\cong\SL_2(p)$ and $L_{\beta}=\langle Q_{\alpha}, Q_{\beta}, Z_{\alpha+2}\rangle$ for some $\alpha+2\in\Delta(\beta)$. Then $Z_{\alpha}\cap Q_{\beta}\le (Z_{\alpha}\cap W_{\beta})(Z_{\alpha+2}\cap W_{\beta})\normaleq L_{\beta}$ and so $W_{\beta}=(Z_{\alpha}\cap W_{\beta})(Z_{\alpha+2}\cap W_{\beta})$.

Suppose first that $W_{\beta}$ is abelian. Then, as $Z_{\alpha}\cap Q_{\beta}\le W_{\beta}$, an index $p$ subgroup of $Z_{\alpha}$ is centralized by $W_{\beta}$ and $Z_{\alpha}$ is a natural $\SL_2(p)$-module. But then $Z_{\alpha}\cap Q_{\beta}=Z_{\beta}$ and $W_{\beta}=Z_{\beta}$, a contradiction.

Since $W_{\beta}$ is non-abelian and $W_{\beta}\cap Q_{\alpha}\cap Q_{\alpha+2}$ has index $p^2$ in $W_{\beta}$, $W_{\beta}\cap Q_{\alpha}\cap Q_{\alpha+2}=\Omega(Z(W_\beta))=W_{\beta}\cap D_{\beta}$. Notice that every element of $W_{\beta}$ lies in $(Z_{\lambda}\cap W_{\beta})\Omega(Z(W_{\beta}))$ for some $\lambda\in\Delta(\beta)$, and that $(Z_{\lambda}\cap W_{\beta})\Omega(Z(W_{\beta}))$ is of exponent $p$, from which it follows that $W_{\beta}$ is of exponent $p$. In particular, since $W_{\beta}$ is not elementary abelian, $p\ne 2$. Therefore, $\Omega(Z(W_{\beta}))$ has index $9$ in $Z_{\alpha}$, $Z_{\alpha}$ is $2$F-module and since $[Z_{\alpha}, W_{\beta}]\not\le \Omega(Z(W_{\beta}))$ and $S/Q_{\alpha}$ has a unique element of order $3$, $Z_{\alpha}$ does not admit quadratic action by any element $x\in S\setminus Q_{\alpha}$.

Now, by minimality of $F_{\beta}$, $\Phi(F_{\beta})\le Q_{\alpha}$ so that $F_{\beta}(Q_{\alpha}\cap Q_{\beta})=W_{\beta}(Q_{\alpha}\cap Q_{\beta})$ since $S/Q_{\alpha}$ has a unique subgroup of order $p$. Then $[F_{\beta}, Z_{\alpha}]=[W_{\beta}, Z_{\alpha}]$. Moreover, $F_{\beta}=[F_{\beta}, O^p(L_{\beta})]\le [F_{\beta}, Z_{\alpha}]^{L_{\beta}}\le W_{\beta}$ and since $F_{\beta}$ contains a non-central chief factor, $W_{\beta}=F_{\beta}Z(W_{\beta})$. Then, since $[F_{\beta}, Q_{\alpha}]=[F_{\beta}, Z_{\alpha}(Q_{\alpha}\cap Q_{\beta})]\le Z_{\alpha}$ by \cref{NCCFZA}, it follows that $O^3(L_{\alpha})$ centralizes $Q_{\alpha}/Z_{\alpha}$. In particular, every $p'$-element of $L_{\alpha}$ acts non-trivially on $Z_{\alpha}$.

Let $U<Z_{\alpha}$ be a non-trivial subgroup of $Z_{\alpha}$ which is normal in $L_{\alpha}$. If $C_S(U)\not\le Q_{\alpha}$, then $O^3(L_{\alpha})$ centralizes $U$ and as $U\normaleq S$, $U\cap Z_{\beta}\ne\{1\}$ and $Z(L_{\alpha})\ne\{1\}$, a contradiction. If $U\not\le Q_{\beta}$, then $Z_{\alpha}=U(Z_{\alpha}\cap Q_{\beta})$ and by \cref{NCCFZA}, it follows that $[F_{\beta}, Z_{\alpha}]\le U$ so that $[O^3(L_{\alpha}), Z_{\alpha}]\le U$ and $C_{Z_{\alpha}}(O^3(L_{\alpha}))\ne\{1\}$ by \cref{SplitMod}. But then $Z(L_{\alpha})\ge Z_{\beta}\cap C_{Z_{\alpha}}(O^3(L_{\alpha}))\ne\{1\}$, a contradiction. Thus, $U\le Q_{\beta}$ and as $Z_{\alpha}$ is $2$F, we may assume that both $Z_{\alpha}/U$ and $U$ are FF-modules for $\bar{L_{\alpha}}$ and by \cref{Badp3} (ii), either $\bar{L_{\alpha}}\cong\SL_2(3)$ or $(Q_8\times Q_8):3$. If $\bar{L_{\alpha}}\cong\SL_2(3)$, then $G$ has a weak BN-pair of rank $2$ and by \cite{Greenbook}, we have a contradiction. If $\bar{L_{\alpha}}\cong (Q_8\times Q_8):3$, since $|\Out(\bar{L_{\beta}})|=2$ and a Hall $3'$-subgroup of $L_{\alpha}\cap G_{\alpha,\beta}$ is isomorphic to an elementary abelian group of order $4$, it follows that that there is an involution $t\in G_{\alpha,\beta}$ such that $[L_{\alpha}, t]\le Q_{\alpha}$ and $[L_{\beta}, t]\le Q_{\beta}$, a contradiction by \cref{BasicAmal} (v).

Thus, we have that $Z_{\alpha}$ is an irreducible $2$F-module. Since $Z_{\alpha}$ is irreducible and $Z_{\alpha}\not\le \Phi(Q_{\alpha})$, $Z_{\alpha}\cap \Phi(Q_{\alpha})=Z_{\beta}\cap \Phi(Q_{\alpha})=\{1\}$ so that $\Phi(Q_{\alpha})=\{1\}$ and $Q_{\alpha}$ is elementary abelian.
\end{proof}

\begin{proposition}\label{b=1a}
Suppose that $m_p(S/Q_{\alpha})=1$ and $p\in\{2,3\}$. Then $p=3$, $Z_{\alpha}=Q_{\alpha}$ is an irreducible $\mathrm{GF}(3)\bar{L_{\alpha}}$-module and one of the following holds:
\begin{enumerate}
\item $G$ has a weak BN-pair of rank $2$ and $G$ is locally isomorphic to $H$ where $F^*(H)\cong\PSp_4(3)$;
\item $|S|=3^5$, $\bar{L_{\alpha}}\cong \Alt(5)$, $Z_{\alpha}$ is the restriction of the permutation module, $\bar{L_{\beta}}\cong\SL_2(3)$ and $Q_{\beta}\cong 3\times 3^{1+2}_+$;
\item $|S|=3^5$, $\bar{L_{\alpha}}\cong O^{3'}(2\wr \Sym(4))$, $Z_{\alpha}$ is a reflection module, $\bar{L_{\beta}}\cong\SL_2(3)$ and $Q_{\beta}\cong 3\times 3^{1+2}_+$; or
\item $|S|=3^6$, $\bar{L_{\alpha}}\cong O^{3'}(2\wr \Sym(5))$, $Z_{\alpha}$ is a reflection module, $\bar{L_{\beta}}\cong\SL_2(3)$ and $Q_{\beta}\cong 3\times 3\times 3^{1+2}_+$.
\end{enumerate}
\end{proposition}
\begin{proof}
By \cref{mp1}, $Z_{\alpha}$ is the unique non-central chief factor for $L_{\alpha}$ in $Q_{\alpha}$ and $Q_{\alpha}$ is elementary abelian. Moreover, $W_{\beta}/\Omega(Z(W_{\beta}))$ is the unique non-central chief factor for $L_{\beta}$ inside $Q_{\beta}$, and is a natural $\SL_2(3)$-module for $\bar{L_{\beta}}\cong\SL_2(3)$. 

Suppose first that $|Z_{\alpha}|=3^3$. Then $\bar{L_{\alpha}}$ is isomorphic to a subgroup $X$ of $\GL_3(3)$ which has a strongly $3$-embedded subgroup. One can check that the only groups which satisfy $X=O^{3'}(X)$ are $\PSL_2(3)$, $\SL_2(3)$ and $13:3$. In the first two cases, $G$ has a weak BN-pair of rank $2$ and comparing with \cite{Greenbook}, we have that $\bar{L_{\alpha}}\cong\PSL_2(3)$ and $G$ to locally isomorphic to $H$, where $F^*(H)\cong \PSp_4(3)$. Suppose that $\bar{L_{\alpha}}\cong 13:3$ and let $t_{\beta}\in L_{\beta}\cap G_{\alpha,\beta}$ be an involution. Then $t_\beta\in G_{\alpha}$ and writing $\bar{t_{\beta}}:=t_{\beta}Q_{\alpha}/Q_{\alpha}$, $\bar{t_{\beta}}$ acts on $\bar{L_{\alpha}}$ and inverts $\bar{S}=Q_{\beta}Q_{\alpha}/Q_{\alpha}$, a contradiction since any involutary automorphism of $13:3$ centralizes a Sylow $3$-subgroup.

Thus, we may assume that $|Z_{\alpha}|>3^3$. Again, let $t_{\beta}\le G_{\alpha,\beta}\cap L_{\beta}$ be an involution. Then, using coprime action, $[t_{\beta}, Q_{\alpha}]\le W_{\beta}$ and $[t_{\beta}, C_{W_{\beta}}(O^3(L_{\beta}))]=\{1\}$. In particular, it follows that $t_{\beta}$ centralizes an index $3$ subgroup of $Q_{\alpha}$. Let $L^*:=\langle t_{\beta}^{G_{\alpha}}\rangle$ and $\bar{L^*}=L^*Q_{\alpha}/Q_{\alpha}\le \bar{G_{\alpha}}$. Since $\bar{L^*}\normaleq \bar{G_{\alpha}}$, we have that $[\bar{L^*}, \bar{L_{\alpha}}]\le \bar{L^*}$. Note that $t_{\beta}$ inverts $W_{\beta}Q_{\alpha}/Q_{\alpha}\cong W_{\beta}/W_{\beta}\cap Q_{\alpha}$ and so $W_{\beta}Q_{\alpha}/Q_{\alpha}=[W_{\beta}Q_{\alpha}/Q_{\alpha}, t_{\beta}]\le [\bar{L_{\alpha}}, \bar{L^*}]\le \bar{L^*}$. If $\bar{G_{\alpha}}$ is not $3$-solvable, then $\bar{L_{\alpha}}/O_{3'}(\bar{L_{\alpha}})$ is a non-abelian finite simple group and since $\bar{L^*}\normaleq \bar{G_{\alpha}}$, we have that $\bar{L_{\alpha}}\le \bar{L^*}$.

If $\bar{G_{\alpha}}$ is $3$-solvable, let $O_{\alpha}$ be the preimage of $O_{3'}(\bar{L_{\alpha}})$ in $L_{\alpha}$. By coprime action, we have that $Q_{\alpha}=[Q_{\alpha}, O_{\alpha}]\times C_{Q_{\alpha}}(O_{\alpha})$ is an $S$-invariant decomposition. Since $Z_{\alpha}$ is irreducible, we infer that $[Q_{\alpha}, O_{\alpha}]=[Z_{\alpha}, O_{\alpha}]=Z_{\alpha}$ and as $Z_{\beta}\le Z_{\alpha}$, it follows that $C_{Q_{\alpha}}(O_{\alpha})=\{1\}$ and $Q_{\alpha}=Z_{\alpha}$. If $|S/Q_{\alpha}|>3$, then $W_{\beta}\le \Phi(Q_{\beta})(Z_{\alpha}\cap Q_{\beta})$ and it follows from the Dedekind modular law that $W_{\beta}=\Phi(Q_{\beta})(Z_{\alpha}\cap Q_{\beta})\cap W_{\beta}=(Z_{\alpha}\cap Q_{\beta})(\Phi(Q_{\beta})\cap W_{\beta})$. Since $W_{\beta}$ contains all non-central chief factors for $L_{\beta}$ inside $Q_{\beta}$, $\Phi(Q_{\beta})\cap W_{\beta}\le Z(W_{\beta})$ so that $W_{\beta}=(Z_{\alpha}\cap Q_{\beta})Z(W_{\beta})$, a contradiction. Thus, $|S/Q_{\alpha}|=3$ and, again, $\bar{L_{\alpha}}\le \bar{L^*}$.

Since $S/Q_{\alpha}$ does not act quadratically on $Z_{\alpha}$, $\bar{L^*}$ is not generated by transvections and as $|Z_{\alpha}|\geq 3^4$, we may apply the main result of \cite{ZalesskiiReflection}. Using that $S/Q_{\alpha}$ is cyclic, we have that $\bar{L^*}$ is isomorphic to the reduction modulo $3$ of a finite irreducible reflection group of degree $n$ in characteristic $0$, and $3^4\leq |Z_{\alpha}|\leq 3^5$. 

Suppose that there is $t_\alpha\in L^*\cap G_{\alpha,\beta}$ an element of order $4$ with $t_\alpha^2Q_{\alpha}\in Z(\bar{L^*})$. Then $t_\alpha\in G_{\beta}$ and $t_\alpha$ acts on $\bar{L_{\beta}}$. We may assume that $t_\alpha^2$ acts non-trivially on $\bar{L_{\beta}}$ for otherwise $t_\alpha^2Q_{\alpha}$ is centralized by $\bar{L_{\alpha}}$ and $t_\alpha^2Q_{\beta}$ is centralized by $\bar{L_{\beta}}$, a contradiction by \cref{BasicAmal} (v). But $t_\alpha$ normalizes $S/Q_{\beta}$ and so either $t_\alpha$ inverts $S/Q_{\beta}$ or centralizes $S/Q_{\beta}$. In either case, $t_\alpha^2$ centralizes $S/Q_{\beta}$ and by \cref{SLGen} (viii), $t_\alpha^2$ acts trivially on $\bar{L_{\beta}}$, a contradiction.

Upon comparing the groups listed in \cite{ZalesskiiReflection} and the orders of $\GL_4(3)$ and $\GL_5(3)$ we are left with the groups $G(1,1,5)$, $G(2,1,4)$, $G(2,2,4)$, $G(2,1,5)$ and $G(2,2,5)$ (in the Todd-Shepherd enumeration convention) as candidates for $\bar{L^*}$. In particular, $|S/Q_{\alpha}|=3$. If $\bar{L^*}\cong G(1,1,5)\cong\Sym(5)$, then $\bar{L_{\alpha}}\cong\Alt(5)$. Then $G$ is determined in \cite{StellmacherL25} and outcome (ii) follows in this case. Thus, $O_2(\bar{L^*})\ne \{1\}$ and writing $O_{\alpha}$ for the preimage of $O_{2}(\bar{L^*})$ in $G_{\alpha}$, we have by coprime action that $Q_{\alpha}=[Q_{\alpha}, O_{\alpha}]\times C_{Q_{\alpha}}(O_{\alpha})$ and since $Z_{\alpha}$ is irreducible and is the unique non-central chief factor within $Q_{\alpha}$, $Q_{\alpha}=[Q_{\alpha}, O_{\alpha}]=Z_{\alpha}$. In particular, $W_{\beta}=Q_{\beta}$, $|Q_{\beta}|\leq 3^5$ and $Q_{\beta}/Z_{\beta}$ is a natural $\SL_2(3)$-module for $\bar{L_{\beta}}$.

Now, $G(2,1,4)\cong 2\wr \Sym(4)$ and $G(2,2,4)$ is isomorphic to an index $2$ subgroup of $G(2,1,4)$. Therefore, if $|Z_{\alpha}|=3^4$, $\bar{L_{\alpha}}\cong O^{3'}(2\wr \Sym(4))$ and the possible actions of $\bar{L_{\alpha}}$ are determined up to conjugacy in $\GL_4(3)$. Indeed, it follows in this case that $S$ is isomorphic to a Sylow $3$-subgroup of $\Alt(12)$. Furthermore, $Q_{\beta}$ has exponent $3$, is of order $3^4$ and $Z(Q_{\beta})=Z_{\beta}$ is elementary abelian of order $9$. Indeed, $Q_{\beta}\cong 3^{1+2}_+\times 3$.

Finally, $G(2,1,5)\cong 2\wr \Sym(5)$ and $G(2,2,5)$ is isomorphic to an index $2$ subgroup of $G(2,1,5)$. Therefore, if $|Z_{\alpha}|=3^5$, $\bar{L_{\alpha}}\cong O^{3'}(2\wr \Sym(5))$ and the possible actions of $\bar{L_{\alpha}}$ are determined up to conjugacy in $\GL_5(3)$. Indeed, it follows in this case that $S$ is isomorphic to a Sylow $3$-subgroup of $\Alt(15)$. Furthermore, $Q_{\beta}$ has exponent $3$, is of order $3^5$ and $Z(Q_{\beta})=Z_{\beta}$ is elementary abelian of order $27$. Indeed, $Q_{\beta}\cong 3^{1+2}_+\times 3\times 3$.
\end{proof}

\begin{remark}
Note that the modules occurring in the above proposition are some of the nearly quadratic modules described in \cref{nearquad}. Indeed, one could prove the above proposition using this characterization along with some knowledge of the quasisimple groups which have a Sylow $3$-subgroup of order $3$.  
\end{remark}

\begin{lemma}\label{alpha+2}
Assume that $m_p(S/Q_{\alpha})\geq 2$. Then there is $\alpha+2\in\Delta(\beta)$ such that $Z_{\alpha}\cap Q_{\beta}\not\le Q_{\alpha+2}$ and $Z_{\alpha+2}\cap Q_{\beta}\not\le Q_{\alpha}$.
\end{lemma}
\begin{proof}
Suppose that $Z_{\alpha}\cap Q_{\beta}\le Q_{\lambda}$ for all $\lambda\in \Delta(\beta)$. Then $Z_{\alpha}\cap Q_{\beta}$ is centralized by $\langle Z_{\alpha}^{G_{\beta}}\rangle$ and so normalized by $G_{\beta}=\langle Z_{\alpha}^{G_{\beta}}\rangle G_{\alpha,\beta}$. But then, $Z_{\alpha}$ centralizes $Q_{\beta}/Z_{\alpha}\cap Q_{\beta}$ and $Z_{\alpha}\cap Q_{\beta}$, yielding a contradiction. Choose $\alpha+2\in \Delta(\beta)$ such that $Z_{\alpha}\cap Q_{\beta}\not\le Q_{\alpha+2}$. 

Assume first that $q_\beta=p$ so that $Z_{\alpha+2}\cap Q_{\beta}$ has index $p$ in $Z_{\alpha+2}$. If $Z_{\alpha+2}\cap Q_{\beta}\le Q_{\alpha}$ then $Z_{\alpha+2}\cap Q_{\beta}$ is centralized by $Z_{\alpha}\cap Q_{\beta}$ and \cref{SEFF} implies that $L_{\alpha+2}/R_{\alpha+2}\cong \SL_2(p)$. By conjugacy, and since $Z(L_{\alpha})=\{1\}$, we have that $Z_{\alpha}$ is a natural $\SL_2(p)$-module for $L_{\alpha}/R_{\alpha}$ so that $Z_{\alpha}\cap Q_{\beta}=Z_{\beta}\le Q_{\alpha+2}$, a contradiction. Thus $q_\beta\geq p^2$ and by \cref{betadeduct}, we have that $\bar{L_{\beta}}/O_{p'}(\bar{L_{\beta}})$ is isomorphic to a rank $1$ group of Lie type. If $p=2$ then since $L_{\beta}=O^{2'}(L_{\beta})$, there is $T\in\syl_2(L_{\beta})$ and $t\in T$ with $(Z_{\alpha}\cap Q_{\beta})^t\not\le Q_\lambda$ for some $\lambda\in\Delta(\beta)$ and $t^2\in Q_{\beta}$. If $p$ is odd, then $\bar{L_{\beta}}\cong \SL_2(q_\beta)$ or $\SU_3(q_\beta)$ and again there is $t\in L_{\beta}\setminus (L_{\beta}\cap G_{\alpha,\beta})$ with $t^2\in G_{\alpha, \beta}$ and $(Z_{\alpha}\cap Q_{\beta})^t\not\le Q_\lambda$ for some $\lambda\in\Delta(\beta)$. We may as well set $\alpha+2=\alpha^t$. Then $(Z_{\alpha+2}\cap Q_{\beta})\le Q_{\alpha}$ yields $(Z_{\alpha}\cap Q_{\beta})=(Z_{\alpha+2}\cap Q_{\beta})^t\le Q_{\alpha}^t=Q_{\alpha+2}$, a contradiction. Hence, the result.
\end{proof}
 
We write $r_{\alpha}:=|(Z_{\alpha}\cap Q_{\beta})Q_{\alpha+2}/Q_{\alpha+2}|$ and $r_{\alpha+2}:=|(Z_{\alpha+2}\cap Q_{\beta})Q_{\alpha}/Q_{\alpha}|$. Since both $(\alpha, \beta)$ and $(\alpha+2, \beta)$ are critical pairs, without loss of generality, we assume throughout the remainder of this section that $r_{\alpha+2} \geq r_{\alpha}$.

\begin{lemma}\label{alphaaction}
Assume that $m_p(S/Q_{\alpha})\geq 2$ and write $\wt L_{\alpha}:=L_{\alpha}/R_{\alpha}$. Then either
\begin{enumerate}
    \item $r_{\alpha+2}=r_\alpha=p$;
    \item $O_{p'}(\wt L_{\alpha})$ is central in $\wt L_{\alpha}$; or
    \item there is a subgroup $C$ of $\wt L_{\alpha}$ of order $p^2$ such that $C_{Z_{\alpha}}(C)$ has index at most $q_\beta p^2$ in $Z_{\alpha}$.
\end{enumerate}
\end{lemma}
\begin{proof}
Assume throughout that $r_{\alpha+2}>p$. Let $A=(Z_{\alpha+2}\cap Q_{\beta})Q_{\alpha}/Q_{\alpha}$ so that $A$ is elementary abelian of order $r_{\alpha+2} \geq r_\alpha$ and $C_{Z_{\alpha}}(A)$ has at most index $q_\beta r_\alpha$ in $Z_{\alpha}$. Since $r_{\alpha+2}>p$, using \cref{GLS2p'}, we have that $O_{p'}(\wt{L_{\alpha}})=\langle C_{O_{p'}(\wt{L_{\alpha}})}(B) \mid [A:B]=p\rangle$. Assume that $C_{Z_{\alpha}}(A)=C_{Z_{\alpha}}(B)$ for all $B\le A$ with $|A/B|=p$. Then $C_{Z_{\alpha}}(A)$ is normalized by $\wt AO_{p'}(\wt{L_{\alpha}})$. Set $H_\alpha=\langle \wt A^{\wt AO_{p'}(\wt{L_{\alpha}})}\rangle$ so that $[H_\alpha, C_{Z_{\alpha}}(\wt A)]=\{1\}$. Then, by coprime action, we get that $Z_{\alpha}=[Z_{\alpha}, O_{p'}(H_{\alpha})]\times C_{Z_{\alpha}}(O_{p'}(H_{\alpha}))$ is an $A$-invariant decomposition and as $C_{Z_{\alpha}}(A)\le C_{Z_{\alpha}}(O_{p'}(H_{\alpha}))$, it follows that $O_{p'}(H_\alpha)$ centralizes $Z_{\alpha}$ so that $O_{p'}(H_{\alpha})=\{1\}$ and $H_\alpha=\wt A\normaleq \wt A O_{p'}(\wt{L_{\alpha}})$. Then, $[\langle A^{\wt{L_{\alpha}}}\rangle, O_{p'}(\wt{L_{\alpha}})]=[A, O_{p'}(\wt{L_{\alpha}})]=\{1\}$ and we have that $O_{p'}(\wt{L_{\alpha}})\le Z(\bar{L_{\alpha}})$. 

So assume now that $C_{Z_{\alpha}}(A)<C_{Z_{\alpha}}(B)$ for some $B\le A$ with $|A/B|=p$. Applying the same line of reasoning as above to $B$ in place of $A$, and continuing recursively, we descend to a group $C\le A$ of order $p^2$ such that $C_{Z_{\alpha}}(C)$ has index at most $q_\beta r_\alpha p^2/r_{\alpha+2} \leq q_\beta p^2$ in $Z_{\alpha}$, as required.
\end{proof}

\begin{lemma}\label{S/QAAbelian}
Suppose that $m_p(S/Q_{\alpha})\geq 2$. Then $S/Q_{\alpha}$ is elementary abelian and $S=Q_{\alpha}Q_{\beta}$.
\end{lemma}
\begin{proof}
Suppose first that $\bar{L_{\alpha}}/O_{2'}(\bar{L_{\alpha}})\cong \Sz(2^n)$. Assume that $q_\beta>q_\alpha$ so that $F_{\beta}\cap Q_{\alpha}$ has index strictly less than $q_\beta$ in $F_{\beta}$ and is centralized by $Z_{\alpha}$. By \cref{SEFF}, we have a contradiction. Thus, if $\bar{L_{\alpha}}/O_{2'}(\bar{L_{\alpha}})\cong \Sz(2^n)$ then we have that $q_\beta\leq q_\alpha$.

Applying \cref{SpeMod2}, we see that in $q_\alpha^{\frac{4}{3}}\leq 2q_\beta \leq 2q_\alpha$ in outcome (i) of \cref{alphaaction} and $q_\alpha^2\leq 2^2 q_\beta \leq  2^2q_\alpha$ in outcome (iii) of \cref{alphaaction}. Hence, we have that $r_{\alpha+2}=r_\alpha=2$ and $q_\alpha=q_\beta=8$ so that, by \cref{SEFF}, $Q_{\beta}/\Phi(Q_{\beta})$ contains a unique non-central chief factor, which is a natural module for $\bar{L_{\beta}}\cong \SL_2(8)$. Hence, $Q_{\beta}=(Q_{\beta}\cap Q_{\alpha})(Q_{\beta}\cap Q_{\alpha+2})\Phi(Q_{\beta})$ so that $Q_{\beta}=(Q_{\beta}\cap Q_{\alpha})(Q_{\beta}\cap Q_{\alpha+2})$. In particular, $(Q_{\alpha}\cap Q_{\beta})Q_{\alpha+2}\in\syl_2(L_{\alpha+2})$. Then $\Phi(Q_{\alpha})Z_{\alpha}(Q_{\alpha}\cap Q_{\beta}\cap Q_{\alpha+2})$ has index $q_\alpha$ in $Q_\alpha$ and $\Phi(Q_{\alpha})(\Phi(Q_{\alpha})Z_{\alpha}\cap Q_{\beta}\cap Q_{\alpha+2})$ has index $q_\alpha$ in $\Phi(Q_{\alpha})Z_{\alpha}$. Applying \cref{SpeMod2}, we have that $O^p(L_\alpha)$ centralizes $Q_\alpha/\Phi(Q_{\alpha})Z_{\alpha}$ and $\Phi(Q_{\alpha})Z_{\alpha}/\Phi(Q_{\alpha})$ so that $O^p(L_\alpha)$ centralizes $Q_{\alpha}/\Phi(Q_{\alpha})$, a contradiction.

Thus, if $\bar{L_{\alpha}}/O_{2'}(\bar{L_{\alpha}})\cong \Sz(2^n)$ then $q_\beta\leq q_\alpha$ and $L_{\alpha}/R_{\alpha}\cong \Sz(q_\alpha)$. Since $Z_{\alpha}\cap Q_{\beta}$ is $G_{\alpha,\beta}$-invariant of order $q_\beta$, we have that $q_\alpha=q_\beta$ by \cite[Smith Theorem 2.8.11]{GLS3}. Then applying \cref{SEFF} we conclude that $Q_{\beta}/\Phi(Q_{\beta})$ contains a unique non-central chief factor, which is a natural module for $\bar{L_{\beta}}\cong \SL_2(q_\beta)$. Now, $Q_{\alpha}=Z_{\alpha}(Q_{\alpha}\cap Q_{\beta})$ and by \cref{NCCFZA}, we deduce that $O^2(L_{\alpha})$ centralizes $Q_{\alpha}/Z_{\alpha}$ so that $R_{\alpha}=Q_{\alpha}$. But then $(G_\alpha, G_\beta, G_{\alpha, \beta})$ is a weak BN-pair of rank $2$, a contradiction by \cite{Greenbook}.

Suppose now that $\bar{L_{\alpha}}/O_{3'}(\bar{L_{\alpha}})\cong \Ree(q_\alpha)$. Assume first that $S\ne Q_{\alpha}Q_{\beta}$ so that $\bar{L_{\beta}}\cong \SU_3(q_{\beta})$ and $Q_{\alpha}Q_{\beta}$ has index $q_\beta^2$ in $S$. Thus, either $q_\alpha=q_\beta$ or $q_\alpha=q_\beta^2$. In the former case, we have that $Q_{\beta}\cap Q_{\alpha}$ has index $q_\beta$ in $Q_{\beta}$ and is centralized by $Z_{\alpha}$. Applying \cref{SEFF} gives a contradiction. In the latter case, we get that $q_\alpha$ is a square, another contradiction since $q_\alpha$ is an odd power of $3$. Hence, $S=Q_{\alpha}Q_{\beta}$.

Now, by \cref{NCCFZA} we have that $[Q_{\beta}, F_{\beta}]\le Z_{\alpha}$ and by the structure of $S/Q_{\alpha}$, we deduce that $F_{\beta}Q_{\alpha}/Q_{\alpha}=Z(S/Q_{\alpha})$ has order $q_\alpha$. Then $F_{\beta}\cap Q_{\alpha}$ is an index $q_\alpha$ subgroup of $F_{\beta}$ which is centralized by $Z_{\alpha}$ and we deduce that $q_\beta\leq q_\alpha$. Applying \cref{SpeModOdd} and \cref{alphaaction}, we see that $\mathrm{max}(q_\alpha^2, 9)\leq 3^2q_\beta \leq 3^2q_\alpha$ and $q_\alpha=q_\beta=3$. However, we now form $H=\langle Z_{\alpha+2}\cap Q_{\beta}, (Z_{\alpha+2}\cap Q_{\beta}), Q_\alpha\rangle$ with $\bar{H}O_{3'}(\bar{L_{\alpha}})/O_{3'}(\bar{L_{\alpha}})\cong \PSL_2(8)$ so that $|Z_{\alpha}/C_{Z_{\alpha}}(H)|\leq 3^6$, a contradiction by \cref{SpeModOdd}. Thus, $L_{\alpha}/R_{\alpha}\cong\Ree(q_{\alpha})$.

Notice that if $\bar{L_{\beta}}\cong \SL_2(q_\beta)$ then applying \cref{NCCFZA} we have that $O^3(L_{\alpha})$ centralizes $Q_{\alpha}/Z_{\alpha}$ and $R_{\alpha}=Q_{\alpha}$. Then $(G_{\alpha}, G_{\beta}, G_{\alpha,\beta})$ is a weak BN-pair of rank $2$ and \cite{Greenbook} provides a contradiction. Thus, $\bar{L_{\beta}}\cong\SU_3(q_{\beta})$. Then $\Phi(Q_{\beta})(Q_{\beta}\cap Q_{\alpha})$ has index $q_\alpha$ in $Q_{\beta}$ and is centralized, modulo $\Phi(Q_{\beta})$, by $Z_{\alpha}$. Applying \cref{SU3Gen} and using that the minimum $\mathrm{GF}(3)$-representation of $\SU_3(3^n)$ is $6n$, we deduce that $q_\alpha\geq q_\beta^2$. But then $\Phi(Q_{\alpha})(Q_{\beta}\cap Q_{\alpha})$ is an index $q_\alpha$ subgroup of $Q_{\alpha}$ which is centralized, modulo $\Phi(Q_{\alpha})$, by $F_{\beta}$ by \cref{NCCFZA}. \cref{SEFF} yields a contradiction in this case.

Finally, suppose that $\bar{L_{\alpha}}/O_{p'}(\bar{L_{\alpha}})\cong \PSU_3(q_\alpha)$. Then $F_{\beta}\cap Q_{\alpha}$ has index $q_{\alpha}$ in $F_{\beta}$ and is centralized by $Z_{\alpha}$ so that by \cref{SEFF}, we deduce that $q_\beta \leq q_\alpha$. Applying \cref{SpeMod2} and \cref{SpeModOdd}, we see that $q_\alpha^{\frac{3}{2}}\leq pq_\beta \leq pq_\alpha$ in outcome (i) of \cref{alphaaction} and $q_\alpha^2\leq p^2 q_\beta \leq  p^2q_\alpha$ in outcome (iii) of \cref{alphaaction}, from which we deduce that $q_\alpha=q_\beta\leq p^2$. Indeed, $Q_{\beta}/\Phi(Q_{\beta})$ is faithful quadratic $2$F-module for $\bar{L_{\beta}}$ and applying \cref{2FRecog}, we have that $\bar{L_{\beta}}\cong \SL_2(q_\beta)$. Then, applying \cref{NCCFZA} we have that $O^3(L_{\alpha})$ centralizes $Q_{\alpha}/Z_{\alpha}$ and $R_{\alpha}=Q_{\alpha}$. Since $O_{p'}(\bar{L_{\alpha}})\ne\{1\}$, using coprime action and that $Z_{\beta}\le Z_{\alpha}$, we deduce that $Q_{\alpha}=Z_{\alpha}$. Moreover, $D_{\beta}\cap F_{\beta}\le Z_{\alpha}$ and since $F_{\beta}/F_{\beta}\cap D_{\beta}$ is natural $\SL_2(q_\beta)$-module, $F_{\beta}=(F_{\beta}\cap Z_{\alpha})(F_{\beta}\cap Z_{\alpha+2})$. Applying \cref{Quad2F} when $q_\alpha=p$ and \cref{2FRecog} when $q_\alpha=p^2$, we deduce that $O_{p'}(\bar{L_{\alpha}}\le Z(\bar{L_{\alpha}})$ and we are in outcome (ii) of \cref{alphaaction}.

Finally, more generally when we are in outcome (ii) of \cref{alphaaction} we have that $L_{\alpha}/R_{\alpha}$ is isomorphic to $\SU_3(q_\alpha)$ or $\PSU_3(q_\alpha)$. If $q_\beta<q_\alpha$ and $\bar{L_{\beta}}/O_{p'}(\bar{L_{\beta}})\not\cong\PSL_2(q_{\beta})$ then $\Phi(Q_{\alpha})(Q_{\alpha}\cap Q_{\beta})$ is an index at most $q_\beta^2$ subgroup of $Q_{\alpha}$ which centralized, modulo $\Phi(Q_{\alpha})$ by $F_{\beta}$ Then \cref{SpeMod2} and \cref{SpeModOdd} provides a contradiction. If $q_{\beta}=q_{\alpha}$ then $F_{\beta}$ is an FF-module for $\bar{L_{\beta}}$ and we deduce that $\bar{L_{\beta}}/O_{2'}(\bar{L_{\beta}})\cong\PSL_2(q_{\beta})$ when $p=2$, and $\bar{L_{\beta}}\cong \SL_2(q_\beta)$ when $p=3$. Thus, in either case, we have that $\bar{L_{\beta}}/O_{2'}(\bar{L_{\beta}})\cong\PSL_2(q_{\beta})$ when $p=2$, and $\bar{L_{\beta}}\cong \SL_2(q_\beta)$ when $p=3$. In particular, applying \cref{NCCFZA}, we see that $O^p(L_{\alpha})$ centralizes $Q_{\alpha}/Z_{\alpha}$, $R_{\alpha}=Q_{\alpha}$ and $\bar{L_{\alpha}}\cong \PSU_3(q_\alpha)$ or $\SU_3(q_\alpha)$. 

If $p=3$, then $(G_{\alpha}, G_{\beta}, G_{\alpha,\beta})$ is a weak BN-pair of rank $2$ and a comparison with \cite{Greenbook} yields the result. If $p=2$ and $q_\alpha=q_\beta$ then $Z_{\alpha}$ acts quadratically on $Q_{\beta}$ and $\Phi(Q_{\beta})(Q_{\alpha}\cap Q_{\beta})$ is an index $q_\beta^2$ subgroup of $Q_{\beta}$ which is centralized by $Z_{\alpha}$, applying \cref{2FRecog} we have that $\bar{L_{\beta}}\cong \PSL_2(q_{\beta})$. If $p=2$ and $q_\beta<q_\alpha$, then since $r_{\alpha+2}>2$, using \cref{SU3Gen} we have that a index $(q_\beta r_\alpha)^3<q_\alpha^6$ subgroup of $Z_{\alpha}$ is centralized by $O^2(L_{\alpha})$ and \ref{SpeMod2} yields a contradiction. Hence, if $p=2$ then $(G_{\alpha}, G_{\beta}, G_{\alpha,\beta})$ is a weak BN-pair of rank $2$ and a comparison with \cite{Greenbook} yields the result.
\end{proof}

\begin{proposition}\label{b=1b}
Suppose that $m_p(S/Q_{\alpha})\geq 2$, $m_p(S/Q_{\beta})\geq 2$ and $p\in\{2,3\}$. Then one of the following holds
\begin{enumerate}
\item $G$ has a weak BN-pair of rank $2$ and $G$ is locally isomorphic to $H$ where $F^*(H)\cong \PSU_4(p^{n+1}), \PSU_5(2^{n+1}), \PSU_5(3^n)$ or $\PSp_4(3^{n+1})$ for $n\geq 1$; or
\item $p=3$, $|S|=3^7$, $\bar{L_{\alpha}}\cong \mathrm{M}_{11}$, $Z_{\alpha}=Q_{\alpha}$ is the ``code'' module for $\bar{L_{\alpha}}$, $\bar{L_{\beta}}\cong \SL_2(9)$ and $Q_{\beta}\cong 3^{1+4}_+$.
\end{enumerate}
Moreover, if $G$ is obtained from a fusion system $\fs$ satisfying \cref{HypFus} then one of the following holds:
\begin{enumerate}
\item $p=2$ and $\fs=\fs_S(G)$ where $F^*(G)\cong\PSU_4(2^n)$ or $\PSU_5(2^n)$ and $n\geq 2$; or
\item $p=3$ and $\fs=\fs_S(G)$ where $F^*(G)\cong\PSp_4(3^{n+1})$, $\PSU_4(3^{n+1})$, $\PSU_5(3^n)$ for $n\geq 1$; or
\item $p=3$ and $\fs=\fs_S(G)$ where $G\cong\mathrm{Co}_3$.
\end{enumerate}
\end{proposition}
\begin{proof}
We have that $S/Q_{\alpha}$ is elementary abelian by \cref{S/QAAbelian}. Assume first that $q_\alpha>q_\beta$. Since $m_p(S/Q_{\beta})>1$, by the groups listed in \cref{SE2} we deduce that either $q_\alpha>q_\beta\geq p^2$, or $\bar{L_{\beta}}\cong\SU_3(3)$. In the latter case, we have that $\Phi(Q_{\alpha})(Q_{\alpha}\cap Q_{\beta})$ is a subgroup of $Q_{\alpha}$ of index $9$ which, by \cref{NCCFZA}, is centralized, modulo $\Phi(Q_{\alpha})$, by $F_{\beta}$. Since $m_p(S/Q_{\alpha})\geq p^2$, by \cref{SEFF} we have that $\bar{L_{\alpha}}\cong \SL_2(9)$ and $(G_{\alpha}, G_{\beta}, G_{\alpha,\beta})$ is a weak BN-pair of rank $2$. Applying \cite{Greenbook}, $G$ is locally isomorphic to $H$ where $F^*(H)\cong\PSU_5(3)$.

Assume that $q_\alpha>q_{\beta}\geq p^2$ and $\bar{L_{\beta}}/O_{2'}(\bar{L_{\beta}})\cong \Sz(q_\beta)$. Then $\Phi(Q_{\alpha})(Q_{\alpha}\cap Q_{\beta})$ has index $q_\beta$ in $Q_{\alpha}$ and is centralized, modulo $\Phi(Q_{\alpha})$, by $F_{\beta}$. Then \cref{SEFF} yields a contradiction. 

Assume that $q_\alpha>q_{\beta}\geq p^2$ and $\bar{L_{\beta}}/O_{p'}(\bar{L_{\beta}})\cong \PSU_3(q_\beta)$. Then applying \cref{SEFF} to $Q_{\alpha}/\Phi(Q_{\alpha})$ in a similar manner as before, we conclude that $q_\beta^2\geq q_\alpha$. Then \cref{SpeMod2} and \cref{SpeModOdd} yield $q_\alpha=q_\beta^2$ so that $Q_{\alpha}/\Phi(Q_{\alpha})$ contains a unique non-central chief factor for $L_{\alpha}$ which, as an $\bar{L_{\alpha}}$-module, is isomorphic to a natural module for $\bar{L_{\alpha}}\cong \SL_2(q_\alpha)$. Moreover, $\Phi(Q_{\beta})(Q_{\beta}\cap Q_{\alpha})$ has index $q_\beta^2$ in $Q_{\beta}$ and is centralized, modulo $\Phi(Q_{\beta})$, by $Z_{\alpha}$ and so applying \cref{2FRecog}, we have that $\bar{L_{\beta}}\cong \SU_3(q_\beta)$. Therefore, $(G_{\alpha}, G_{\beta}, G_{\alpha,\beta})$ is a weak BN-pair of rank $2$ and comparing with \cite{Greenbook}, $G$ is locally isomorphic to $H$ where $F^*(H)\cong\PSU_5(p^n)$ where $p\in\{2,3\}$ and $p^n>2$.

Note that if $q_\alpha\leq q_\beta$ then \cref{SEFF} implies that $\bar{L_{\beta}}\cong \SL_2(q_\beta)$ and $q_\alpha=q_\beta$. Thus, to complete the proof we assume generally that $q_\beta\leq q_\alpha$ and $\bar{L_{\beta}}/O_{p'}(\bar{L_{\beta}})\cong \PSL_2(q_\beta)$. In particular, since $Q_{\alpha}=Z_{\alpha}(Q_{\alpha}\cap Q_{\beta})$, we have by \cref{NCCFZA} that $O^p(L_{\alpha})$ centralizes $Q_{\alpha}/Z_{\alpha}$ and $R_{\alpha}=Q_{\alpha}$.

Now, if $\bar{L_{\beta}}\not\cong \SL_2(q_{\beta})$ then comparing with \cref{SEQuad} and \cref{2FRecog}, we infer that $p=2$ and $q_\beta^2<q_\alpha$. Since $q_\beta>p$, we have that $q_\alpha>p^4$ and $\bar{L_{\alpha}}/O_{p'}(\bar{L_{\alpha}})\cong \PSL_2(q_\alpha)$. Applying \cref{SpeMod2}, we see that $q_\alpha^{\frac{2}{3}}\leq pq_\beta$ in outcome (i) of \cref{alphaaction} and $q_\alpha\leq p^2 q_\beta$ in outcome (iii) of \cref{alphaaction}, a contradiction. In outcome (ii) of \cref{alphaaction} we have that $\bar{L_{\alpha}}$ is isomorphic to a central extension of $\PSL_2(q_\alpha)$ and we deduce that $|Z_{\alpha}|\leq (q_{\beta} q_{\alpha})^2<q_\alpha^3$. Applying \cref{SL2ModRecog}, and using that $Z_{\alpha}\cap Q_{\beta}$ is a $G_{\alpha, \beta}$-invariant subgroup of $Z_{\alpha}$ of index $q_\beta< q_\alpha^{\frac{1}{2}}$, we deduce that $q_\alpha=q_\beta^3$ and $Z_{\alpha}$ is a triality module for $\bar{L_{\alpha}}$. However, since $F_{\beta}\le O^2(L_{\beta})$ and $F_{\beta}'\le D_{\beta}$, we have that $[Z_{\alpha}, F_{\beta}, F_{\beta}, F_{\beta}]\le [F_{\beta}', F_{\beta}]=\{1\}$, a contradiction to the module structure of $Z_{\alpha}$.  

Thus $\bar{L_{\beta}}\cong \SL_2(q_\beta)$ and applying \cref{DirectSum}, we have that $Q_\beta/D_\beta$ is a direct product of $n$ natural $\SL_2(q_\beta)$-modules and $q_\alpha=q_\beta^n$. In particular, $Z_{\alpha}\cap D_\beta$ has index $q_\beta q_\alpha$ in $Z_\alpha$ and is centralized by $S$. In particular, we have by \cref{SEFF} that either $Z_{\alpha}$ contains a unique non-central chief factor for $L_{\alpha}$, or $Z_{\alpha}$ contains two non-central chief factors and $q_\alpha=q_\beta$. In the latter case, by \cref{SEFF}, both chief factors are natural $\SL_2(q_\alpha)$ modules. Writing $R_1$ and $R_2$ for their centralizers in $L_{\alpha}$, we have that $L_{\alpha}/R_1\cong L_{\alpha}/R_2\cong \SL_2(q_\alpha)$ and $R_1\cap R_2=Q_{\alpha}$. It follows that $\bar{L_{\alpha}}\cong \SL_2(q_\alpha)$ in this case and by \cite{Greenbook}, no configurations exist with this structure. Hence, $Z_{\alpha}$ contains a unique non-central chief factor for $L_{\alpha}$ and this chief factor is also the unique non-central chief factor for $L_{\alpha}$ within $Q_{\alpha}$. 

Since $Z(L_{\alpha})=\{1\}$, \cref{SplitMod} yields that $Z_{\alpha}$ is irreducible and since $O^p(L_{\alpha})$ centralizes $Q_{\alpha}/Z_{\alpha}$, $\Phi(Q_{\alpha})=\{1\}$ and $Q_{\alpha}$ is elementary abelian. Since $D_{\beta}\le Q_{\alpha}$, $D_{\beta}$ is centralized by $S=F_{\beta}Q_{\alpha}$ so that $D_{\beta}=Z_{\beta}$. Now, since $Q_{\beta}$ is a direct sum of natural modules, for any $z\in Z_{\alpha}\setminus [Z_{\alpha}, Q_{\beta}]Z_{\beta}$, we have that $[z, Q_{\beta}]Z_{\beta}=[Z_{\alpha}, Q_{\beta}]Z_{\beta}$. Then, $Z_{\alpha}$ is a faithful simple nearly quadratic module for $\bar{L_{\alpha}}$. Applying \cref{nearquad}, we have that $F^*(\bar{L_{\alpha}})=Z(\bar{L_{\alpha}})K$ where $K$ is a component of $\bar{L_{\alpha}}$. If $\bar{L_{\alpha}}$ is not quasisimple then as $\bar{L_{\alpha}}/O_{3'}(\bar{L_{\alpha}})$ is a simple group, the only possibility is that $p=3$, $K\cong \Sz(2^r)$ for some $r$, and $\bar{L_{\alpha}}\le \Aut(K)$. Since $\Out(K)$ is solvable, this gives a contradiction.

Thus $\bar{L_{\alpha}}$ is quasisimple and by \cref{SE2}, $\bar{L_{\alpha}}\cong\SL_2(p^n), \PSL_2(p^n)$ for $n\geq 2$ and $p\in\{2,3\}$; or $\bar{L_{\alpha}}\cong \mathrm{M}_{11}$ or a coprime central extension of $\PSL_3(4)$ and $p=3$. If $\bar{L_{\alpha}}/Z(\bar{L_{\alpha}})\cong \PSL_2(q_\alpha)$ then $G$ has a weak BN-pair of rank $2$ and $G$ is determined up to local isomorphism in \cite{Greenbook}. Comparing with the amalgams determined there, we have that $G$ is locally isomorphic to $H$ where $F^*(H)\cong \PSU_4(p^n)$ or $\PSp_4(3^n)$ for $n\geq 2$. Hence, we may assume that $p=3$ and $q_\beta=q_\alpha=9$.

Suppose that $\bar{L_{\alpha}}\cong\mathrm{M}_{11}$. Then the amalgam is described in \cite{pap} and we have (v) as a conclusion in this case. If $G$ is obtained from a fusion system $\fs$ satisfying \cref{HypFus}, then since $S\in\syl_3(O^3(G_{\alpha}))$, it follows that $O^3(\fs)=\fs$ and we may apply the results in \cite{Comp1}. Indeed, $\fs$ is isomorphic to the $3$-fusion system of $\mathrm{Co_{3}}$.

We may assume that $\bar{L_{\alpha}}$ is isomorphic to a central extension of $\PSL_3(4)$. Since $\bar{L_{\alpha}}$ is generated by two conjugate Sylow $3$-subgroups and $Z_{\beta}$ has index $3^4$ in $Z_{\alpha}$, we deduce that $|Z_{\alpha}|\leq 3^8$. In fact, the possibilities for irreducible modules for central extensions of $\PSL_3(4)$ imply that $|Z_{\alpha}|\in\{3^6, 3^8\}$. If $|Z_{\alpha}|=3^8$, then $\bar{L_{\alpha}}\cong 4\cdot \PSL_3(4)$ and we infer that $|Z_{\beta}|=3^4$. Comparing with the structure of the $8$-dimensional module for $\bar{L_{\alpha}}$, this yields a contradiction. Thus, $|Z_{\alpha}|=3^6$. One can check that for each of the modules in question that $|Z(\bar{L_{\alpha}})|=2$ and so by coprime action, $Q_{\alpha}=[Q_{\alpha}, Z(\bar{L_{\alpha}})]\times C_{Q_{\alpha}}(Z(\bar{L_{\alpha}}))$ and as $Z_{\alpha}=[Q_{\alpha}, Z(\bar{L_{\alpha}})]$, we deduce that $Q_\alpha=Z_\alpha$. For $t$ an involution in the preimage in $L_{\alpha}$ of $Z(\bar{L_{\alpha}})$, we have that $S=Z_{\alpha}\rtimes C_S(t)$. Moreover, $S$ is of order $3^8$ and is isomorphic to a Sylow $3$-subgroup of $\mathrm{Suz}$ or $\PSp_4(9)$. In the former case, $Z_{\beta}$ is of order $3$ a contradiction since $Z_{\beta}$ has index $3^4$ in $Z_{\alpha}$.

Thus $S$ is isomorphic to a Sylow $3$-subgroup of $\PSp_4(9)$, and we calculate in MAGMA that $\bar{G_{\alpha}}$ embeds as a subgroup of $2\cdot\PSL_3(4).2^2$ and for any element $x\in \bar{G_{\alpha}}$ of order $8$, $[x^4, Z_{\alpha}]\not\le [S, Z_{\alpha}]=Z_{\alpha}\cap Q_{\beta}$. Let $t_{\beta}\in L_{\beta}\cap G_{\alpha,\beta}$ be an element of order $8$, so that $t_{\beta}^4Q_{\beta}\le Z(\bar{L_{\beta}})$. But then $[t_{\beta}^4, Z_{\alpha}]\le Z_{\alpha}\cap Q_{\beta}$ and since $t_{\beta}\le G_{\alpha}$, we have a contradiction.
\end{proof}

\begin{proposition}\label{b=1c}
Suppose that $m_p(S/Q_{\alpha})\geq 2$, $m_p(S/Q_{\beta})=1$ and $p\in\{2,3\}$. Then one of the following holds:
\begin{enumerate}
\item $G$ has a weak BN-pair of rank $2$ and $G$ is locally isomorphic to $H$ where $F^*(H)\cong \PSU_4(p), \PSU_5(2)$;
\item $p=3$, $|S|=3^6$, $\bar{L_{\alpha}}\cong\PSL_2(9)$, $Z_{\alpha}=Q_{\alpha}$ is a natural $\Omega_4^-(3)$-module, $\bar{L_{\beta}}\cong (Q_8\times Q_8):3$ and $Q_{\beta}\cong 3^{1+4}_+$;
\item $p=3$, $|S|=3^6$, $\bar{L_{\alpha}}\cong\PSL_2(9)$, $Z_{\alpha}=Q_{\alpha}$ is a natural $\Omega_4^-(3)$-module, $\bar{L_{\beta}}\cong 2\cdot\Alt(5)$ and $Q_{\beta}\cong 3^{1+4}_+$;
\item $p=3$, $|S|=3^6$, $\bar{L_{\alpha}}\cong\PSL_2(9)$, $Z_{\alpha}=Q_{\alpha}$ is a natural $\Omega_4^-(3)$-module, $\bar{L_{\beta}}\cong 2^{1+4}_-.\Alt(5)$ and $Q_{\beta}\cong 3^{1+4}_+$;
\item $p=3$, $|S|=3^7$, $\bar{L_{\alpha}}\cong \mathrm{M}_{11}$ and $Z_{\alpha}=Q_{\alpha}$ is the ``cocode'' module for $\bar{L_{\alpha}}$, $\bar{L_{\beta}}\cong \SL_2(3)$ and $Q_{\beta}\cong 3^{1+1+4}\cong T\in\syl_3(\SL_3(9))$; or
\item $p=3$, $|S|=3^7$, $\bar{L_{\alpha}}\cong \mathrm{M}_{11}$ and $Z_{\alpha}=Q_{\alpha}$ is the ``cocode'' module for $\bar{L_{\alpha}}$, $\bar{L_{\beta}}\cong \SL_2(5)$ and $Q_{\beta}\cong 3^{1+1+4}\cong T\in\syl_3(\SL_3(9))$.
\end{enumerate}
Moreover, if $G$ is obtained from a fusion system $\fs$ satisfying \cref{HypFus} then one of the following holds:
\begin{enumerate}
\item $p=2$ and $\fs=\fs_S(G)$ where $G\cong\PSU_4(2)$, $\Aut(\PSU_4(2))$, $\PSU_5(2)$ or $\Aut(\PSU_5(2))$;
\item $p=3$ and $\fs=\fs_S(G)$ where $F^*(G)\cong \PSU_4(3)$; or
\item  $p=3$ and $\fs=\fs_S(G)$ where $G\cong\mathrm{McL}$, $\Aut(\mathrm{McL})$, $\mathrm{Co}_2$, $\mathrm{Ly}$, $\mathrm{Suz}$, $\Aut(\mathrm{Suz})$,  $\PSU_6(2)$ or $\PSU_6(2).2$.
\end{enumerate}
\end{proposition}
\begin{proof}
Suppose that $m_p(S/Q_{\alpha})\geq 2$ and $m_p(S/Q_{\beta})=1$. Observe first that if $|S/Q_{\beta}|\ne p$, then using that $m_p(S/Q_{\alpha})>1$, we must have that $S/Q_{\beta}$ generalized quaternion and by \cref{NCCFZA} that $Q_{\alpha}/\Phi(Q_{\alpha})$ has an index $4$ subgroup centralized by $F_{\beta}$. Then \cref{SEFF} yields that $\bar{L_{\alpha}}\cong \PSL_2(4)$. But then $Q_{\beta}/\Phi(Q_{\beta})$ is a quadratic $2F$-module and by \cref{Quad2F}, we deduce that $\bar{L_{\beta}}\cong \SU_3(2)$. Thus, $(G_{\alpha}, G_{\beta}, G_{\alpha, \beta})$ is a weak BN-pair and is locally isomorphic to $H$ where $F^*(H)\cong \PSU_5(2)$. 

Thus, we continue with the added assumption that $|S/Q_{\beta}|=p$. By \cref{NCCFZA} we have that $O^p(L_{\alpha})$ centralizes $Q_{\alpha}/Z_{\alpha}$ so that $R_{\alpha}=Q_{\alpha}$. Furthermore, there is a unique non-central chief factor within $Z_{\alpha}$ for $L_{\alpha}$ and as $Z(L_{\alpha})=\{1\}$ and applying \cref{SplitMod}, we conclude that $Z_{\alpha}$ is an irreducible faithful module for $\bar{L_{\alpha}}$. Since $O^p(L_{\alpha})$ centralizes $Q_{\alpha}/Z_{\alpha}$, we conclude that $\Phi(Q_{\alpha})\cap Z_{\alpha}=\Phi(Q_{\alpha})=\{1\}$ and $Q_{\alpha}$ is elementary abelian. Note that if $O_{p'}(\bar{L_{\alpha}})\ne\{1\}$, then $Z_{\beta}\le Z_{\alpha}=[Q_{\alpha}, O_{p'}(\bar{L_{\alpha}})]$ by the irreducibility of $Z_{\alpha}$ and by coprime action, we infer that $Q_{\alpha}=Z_{\alpha}$ is elementary abelian.

We aim to show that $O_{p'}(\bar{L_{\alpha}})\le Z(\bar{L_{\alpha}})$. Assume that there does not exist $x\in S/Q_{\alpha}$ of order $p$ such that $|Z_{\alpha}/C_{Z_{\alpha}}(x)|=p^2$. Applying \cref{alphaaction}, we conclude that there is $C\le S/Q_{\alpha}$ of order $p^2$ such that $|Z_{\alpha}/C_{Z_{\alpha}}(C)|=|Z_{\alpha}/C_{Z_{\alpha}}(x)|=p^3$ for all $1\ne x \in C$. Following the methodology of \cref{alphaaction} using \cref{GLS2p'} we deduce that $O_{p'}(\bar{L_{\alpha}})\le Z(\bar{L_{\alpha}})$ as desired. 

Hence, we may assume that there is $x\in S/Q_{\alpha}$ of order $p$ such that $|Z_{\alpha}/C_{Z_{\alpha}}(x)|=p^2$. Suppose first that $p=3$. Note that if $x$ acts quadratically on $Z_{\alpha}$ then $O_{3'}(\bar{L_{\alpha}})\le Z(\bar{L_{\alpha}})$ by \cref{SEQuad}. Thus, $[Z_{\alpha}, x]C_{Z_{\alpha}}(x)$ has index $3$ in $Z_{\alpha}$ and by \cref{nearquad} we have that $\bar{L_{\alpha}}$ is quasisimple, as desired. Suppose now that $p=2$ so that $[Z_{\alpha}, x, x]=\{1\}$. Since $\bar{L_{\alpha}}$ is not quasisimple by assumption, we deduce that $F^*(\bar{L_{\alpha}})=F(\bar{L_{\alpha}})$. Letting $P_\alpha$ be a $2$-minimal subgroup of $L_{\alpha}$ such that $\bar{L_{\alpha}}=\bar{P_{\alpha}}O_{2'}(\bar{L_{\alpha}})$, applying \cite[Theorem 1]{ChermakSmall}, we conclude that for $U$ a non-central $P_{\alpha}$-chief factor of $Z_{\alpha}$ that $U$ is either a natural $\SL_2(4)$-module, or a natural $\Omega_4^-(2)$-module for $\bar{P_{\alpha}}\cong \PSL_2(4)$. In particular, $Z_{\alpha}\cap Q_{\beta}$ is an elementary abelian subgroup of $Q_{\beta}$ of index $4$ and we calculate by the action of $G_{\beta}$ that $Z(Q_{\beta})=Z_\beta$ has index $2^4$ in $Q_{\beta}$. Moreover, $Q_{\beta}/D_{\beta}$ is a quadratic faithful $2$F-module for $\bar{L_{\beta}}$ so that by \cref{Quad2F}, $\bar{L_{\beta}}\cong \Sym(3), (3\times 3):2$ or $\Dih(10)$ and $Q_{\beta}=(Z_{\alpha}\cap Q_{\beta})(Z_{\alpha+2}\cap Q_{\beta})$. For $W$ the preimage in $Z_{\alpha}$ of $U$, if $U$ is a natural $\SL_2(4)$-module then $C_{Z_{\alpha}}(x)=C_{Z_{\alpha}}(y)=C_{Z_{\alpha}}(S)=Z_{\beta}$ for distinct $x,y\in \bar{S}$. But then $[Z_{\alpha}, Q_{\beta}]\le [Z_{\alpha}, S]\le Z_{\beta}$, a contradiction since $Z_{\alpha}\not\le Q_{\beta}$. If $U$ is a natural $\Omega_4^-(2)$-module for $\bar{P_{\alpha}}$ then repeatedly applying \cref{A6Cohom}, we deduce that $Z_{\alpha}=[Z_{\alpha}, P_{\alpha}]\times C_{Z_{\alpha}}(P_{\alpha})$. In particular, $|Z_{\alpha}/Z_{\beta}|=2^3$ and applying \cite[Proposition 4.6]{ChermakSmall}, we conclude that $[F(\bar{L_{\alpha}}), \bar{L_{\alpha}}]\cong 3\times 3$ and where each factor is normalized by $\bar{L_{\alpha}}$. But then $[F(\bar{L_{\alpha}}), \bar{L_{\alpha}}, \bar{L_{\alpha}}]=\{1\}$ and by coprime action, $[F(\bar{L_{\alpha}}), \bar{S}]=\{1\}$, a contradiction since $F(\bar{L_{\alpha}})$ is self-centralizing. Hence, $O_{p'}(\bar{L_{\alpha}})\le Z(\bar{L_{\alpha}})$ and $\bar{L_{\alpha}}$ is quasisimple.

Suppose that $\bar{L_{\alpha}}\cong\SL_2(p^n)$ or $\PSL_2(p^n)$ for any $n>1$. Unless $r_\alpha=p$ we have that $|Z_{\alpha}|\leq r_\alpha^2p^2\leq p^{2n+2}$ and if $r_\alpha=p$, we have that $|Z_{\alpha}|\leq p^6$. Since the minimal degree of a $\mathrm{GF}(p)$-representation of $\bar{L_{\alpha}}$ is $2n$ and $n\geq 2$, $r_{\alpha}\geq p^{n-1}$ and it follows that there is at most one non-trivial irreducible composition factor within $Z_{\alpha}$. Since $C_{Z_{\alpha}}(O^p(L_{\alpha}))=Z(L_{\alpha})=\{1\}$, by \cref{SplitMod}, $Z_{\alpha}=[Z_{\alpha}, L_{\alpha}]$ is irreducible. Setting $K$ to be Hall $p'$-subgroup of $L_{\alpha}\cap G_{\alpha,\beta}$, it follows from Smith's theorem (\cite[Theorem 2.8.11]{GLS3}) that $Z_{\beta}=C_{Z_{\alpha}}(S)$ and $Z_{\alpha}/[Z_{\alpha}, S]$ are irreducible and $1$-dimensional as $\bar{F}K$-modules, where $\bar{F}$ is an algebraically closed field of characteristic $p$. But $[Z_{\alpha}, S]=[Z_{\alpha}, Q_{\beta}]\le Z_{\alpha}\cap Q_{\beta}$ and since $Z_{\alpha}\cap Q_{\beta}$ has index $p$ in $Z_{\alpha}$, $[Z_{\alpha}, S]=Z_{\alpha}\cap Q_{\beta}$ and $|Z_{\beta}|=|Z_{\alpha}/[Z_{\alpha}, S]|=p$. If $n>2$, then $|Z_{\alpha}|\leq p^{2n+2}< p^{3n}$ and \cite[Lemma 2.6]{ChermakJ} implies that $Z_{\alpha}$ is a triality module for $\bar{L_{\alpha}}\cong\SL_2(p^3)$ and $|Z_{\alpha}|=p^8$. Since $|Z_{\alpha}|\leq r_\alpha^2p^2$, we have that $r_\alpha=p^3$ and $S=(Z_{\alpha+2}\cap Q_{\beta})Q_\alpha$ centralizes $Z_{\alpha}\cap Q_{\beta}\cap Q_{\alpha+2}$. But then $Z_{\beta}=Z_{\alpha}\cap Q_{\beta}\cap Q_{\alpha+2}$ is index $p^4$ in $Z_{\alpha}$. Since $|Z_{\beta}|=p$, $p^5=|Z_{\alpha}|=p^8$, a contradiction.

Thus, we may assume that $|S/Q_{\alpha}|=p^2$ for the remainder of the proof. Then $F_{\beta}/F_{\beta}\cap D_{\beta}$ is a quadratic $2F$-module and so, by \cref{Quad2F}, both $\bar{L_{\beta}}$ and $F_{\beta}/D_{\beta}\cap F_{\beta}$ are determined. We have that $\bar{L_{\alpha}}$ is isomorphic to $\PSL_2(p^2), \SL_2(p^2), \mathrm{M}_{11}$ or a central extension of $\PSL_3(4)$. Since both $\mathrm{M}_{11}$ and central extensions of $\PSL_3(4)$ are generated by two conjugate Sylow $3$-subgroups (or by three $3$-elements), we see that $|Z_{\alpha}|\leq 3^6$ and by the above, in all cases we have that $|Z_{\alpha}|\leq p^6$. Checking against the degrees of the minimal $\mathrm{GF}(p)$-representations of the candidates for $\bar{L_{\alpha}}$, we see that $Z_{\alpha}$ contains a unique irreducible composition factor and since $C_{Z_{\alpha}}(O^p(L_{\alpha}))=Z(L_{\alpha})=\{1\}$, it follows from \cref{SplitMod}, that $Z_{\alpha}=[Z_{\alpha}, L_{\alpha}]$ is irreducible.
 
Now, checking against the list of groups provided in \cref{Quad2F}, either $\bar{L_{\beta}}$ is $p$-solvable or has a non-trivial center, and for $T_\beta$ the preimage in $L_{\beta}$ of $O_{p'}(\bar{L_{\beta}})$, we have by coprime action $Q_{\beta}/\Phi(Q_{\beta})=[Q_{\beta}/\Phi(Q_{\beta}), T_\beta]\times C_{Q_{\beta}/\Phi(Q_{\beta})}(T_\beta)$ where $[Q_{\beta}/\Phi(Q_{\beta}), T_\beta]$ contains all non-central chief factors in $Q_{\beta}/\Phi(Q_{\beta})$ and $C_{Q_{\beta}/\Phi(Q_{\beta})}(T_\beta)=C_{Q_{\beta}/\Phi(Q_{\beta})}(O^p(L_\beta))$. In particular, $F_{\beta}\Phi(Q_{\beta})/\Phi(Q_{\beta})=[Q_{\beta}/\Phi(Q_{\beta}), T_\beta]$. Since $\Phi(Q_{\beta})\le Q_{\alpha}$, $[\Phi(Q_{\beta}), Z_{\alpha}]=\{1\}$ and it follows that $\Phi(Q_{\beta})\le D_{\beta}$ so that $Q_{\beta}=F_{\beta}D_{\beta}$. Since $D_{\beta}\le Q_{\alpha}$ is elementary abelian and $F_{\beta}\le O^p(L_{\beta})$, $S=F_{\beta}Q_{\alpha}$ centralizes $D_{\beta}$ so that $D_{\beta}=Z_{\beta}$. 

Suppose that $\bar{L_{\alpha}}$ is isomorphic to a central extension of $\PSL_3(4)$. Since $|Z_{\alpha}|\leq 3^6$, we have that $|Z_{\alpha}|=3^6$, $|Z(\bar{L_{\alpha}})|=2$ and since $O^3(L_{\alpha})$ centralizes $Q_{\alpha}/Z_{\alpha}$, an easy coprime action argument yields that $Q_{\alpha}=Z_{\alpha}$ is elementary abelian. Hence, $|S|=3^8$ and comparing with the modules in \cref{Quad2F}, $|Q_{\beta}/Z_{\beta}|=3^4$ so that $|Z_{\beta}|=3^3$. Checking the irreducible $\mathrm{GF}(3)$-modules associated to $\bar{L_{\alpha}}$, we have that $|Z_{\beta}|\leq 3^2$, a contradiction.

Suppose that $\bar{L_{\alpha}}\cong\mathrm{M}_{11}$. Then $p=3$, $|Z_{\alpha}|=3^5$ and $\bar{L_{\beta}}\cong 2\cdot\Alt(5)$, $2^{1+4}_-.\Alt(5)$, $\SL_2(3)$ or $(Q_8\times Q_8):3$ by \cref{Quad2F}. In the first three cases, the structure of $L_{\alpha}$ and $L_{\beta}$ is determined in \cite{pap} and outcomes (vi) and (vii) follow in these cases. Suppose that $\bar{L_{\beta}}\cong (Q_8\times Q_8):3$ with $|Q_{\beta}/Z_{\beta}|=p^4$ and let $K_\beta$ be a Hall $2'$-subgroup of $G_{\alpha,\beta}\cap L_{\beta}$.  Then $K_\beta \le G_{\alpha}$ and so $K_{\beta}$ acts on $L_{\alpha}/Q_{\alpha}$. Since $\mathrm{M}_{11}$ has no outer automorphisms, if $K_{\beta}\not\le L_{\alpha}$, then there is an involution $t\in K_\beta$ such that $[t, L_{\alpha}]\le Q_{\alpha}$ and $[t, L_{\beta}]\le Q_{\beta}$, a contradiction by \cref{BasicAmal}. Thus, $K_\beta \le L_{\alpha}$ and we may assume that $L_{\alpha}=G_{\alpha}$. Since $[K_\beta, Z_{\alpha}]\le Z_{\alpha}\cap Q_{\beta}$ and $K_\beta$ centralizes $Z_{\beta}$ it follows that $|C_{Z_{\alpha}}(K_\beta)|=3^3$, and one can check that this provides a contradiction. 

Finally, suppose that $\bar{L_{\alpha}}\cong \PSL_2(p^2)$ or $\SL_2(p^2)$. Then, again by Smith's theorem, $|Z_{\beta}|=p$ so that $F_{\beta}=Q_{\beta}$. By the minimality of $F_{\beta}$, it follows that $Z(Q_{\beta})=\Phi(Q_{\beta})=Z_{\beta}$ is of order $p$ and $Q_{\beta}$ is extraspecial. Since $Q_{\beta}\cap Q_{\alpha}$ is an elementary abelian subgroup of index $p^2$ in $Q_{\beta}$, we have that $|Q_{\beta}|=p^5$. In particular, $|S|=p^6$ and $Z_{\alpha}=Q_{\alpha}$ is of order $p^4$. 

If $p=2$, then $\bar{L_{\beta}}\cong \Dih(10), \Sym(3)$ or $(3\times 3):2$. In the first two cases, $G$ has a weak BN-pair and so comparing with \cite{Greenbook}, we have that $L_{\beta}\cong\Sym(3)$ and $G$ is locally isomorphic to $H$ where $F^*(H)\cong\PSU_4(2)$. Since $Q_{\beta}$ is extraspecial, comparing with \cite{Winter}, $\bar{L_{\beta}}$ is isomorphic to a subgroup of $O_4^+(2)$ if $Q_{\beta}\cong 2^{1+4}_+$;  or $O_4^-(2)$ if $Q_{\beta}\cong 2^{1+4}_-$. Note that $9$ does not divide $|O_4^-(2)|$ and so we may assume that $Q_{\beta}\cong 2^{1+4}_+$. Let $K$ be a Sylow $3$-subgroup of $L_{\alpha}\cap G_{\alpha,\beta}$. Then $K$ acts non-trivially on $Q_{\beta}$ and so $K$ also embeds into $O_4^+(2)$ while normalizing $\bar{L_{\beta}}\cong (3\times 3):2$. But for $H\le O_4^+(2)$ with $H\cong (3\times 3):2$ we have that $|N_{O_4^+(2)}(H)/H|=2$, a contradiction. 

Thus, we may assume that $p=3$ and $L_{\beta}\cong \SL_2(3), (Q_8\times Q_8):3, 2\cdot\Alt(5)$ or $2^{1+4}_-.\Alt(5)$. Since $|Z_{\alpha}|=3^4$ and is not quadratic, we have that $\bar{L_{\alpha}}\cong \PSL_2(9)$ and $Z_{\alpha}$ is a natural $\Omega_4^-(3)$-module. If $\bar{L_{\beta}}\cong \SL_2(3)$ then $G$ has a weak BN-pair and comparing with \cite{Greenbook}, we have that $G$ is locally isomorphic to $H$ where $F^*(H)\cong\PSU_4(3)$. If $\bar{L_{\beta}}\cong 2\cdot\Alt(5)$ or $2^{1+4}_-.\Alt(5)$ then the structure of $L_{\alpha}$ and $L_{\beta}$ is determined in \cite{pap} and we obtain conclusions (iii) and (iv). If $G$ is obtained from a fusion system $\fs$ satisfying \cref{HypFus}, then applying the results in \cite{Comp1}, $\fs$ is isomorphic to the $3$-fusion system of $\mathrm{McL}$, $\Aut(\mathrm{McL})$ or $\mathrm{Co}_2$. Finally, suppose that $\bar{L_{\beta}}\cong (Q_8\times Q_8):3$. Since $Q_{\beta}$ is extraspecial of order $3^5$ and $\bar{L_{\beta}}$ embeds in the automorphism group of $Q_{\beta}$, it follows from \cite{Winter} that $Q_{\beta}\cong 3^{1+4}_+$ and we have (ii) as a conclusion. If $G$ is obtained from a fusion system $\fs$ satisfying \cref{HypFus}, then applying the results in \cite{Comp1}, $\fs$ is isomorphic to the $3$-fusion system of $\PSU_6(2)$ or $\PSU_6(2).2$.
\end{proof}

We conclude this final section by summarizing what has been shown:

\begin{theorem}\label{Mainbodd}
Suppose that $\mathcal{A}=\mathcal{A}(G_\alpha, G_\beta, G_{\alpha,\beta})$ is an amalgam satisfying \cref{MainHyp}. If $Z_{\alpha'}\le Q_{\alpha}$, then one of the following holds:
\begin{enumerate}
\item $G$ has a weak BN-pair of rank $2$; or
\item $p=3$, $b=1$, $|S|\leq 3^7$ and the shapes of $L_{\alpha}$ and $L_{\beta}$ are known.
\end{enumerate}
\end{theorem}